%% file: Z.Main.tex
\documentclass[a4paper, 10 pt]{memoir}
\usepackage{theorem} 
\usepackage{amssymb,amsmath, amscd, bbm} 
\usepackage{latexsym} 
 \usepackage[dvips]{graphicx} 
\usepackage{color}
\usepackage{psfrag} 
\usepackage[all, cmtip]{xy} 
\usepackage{xspace}
\usepackage[applemac]{inputenc}
\usepackage[T1]{fontenc}
\usepackage{lmodern}
\usepackage{textcomp}
\usepackage{url}
\input{Z.layout.tex}
\input{Z.definitions.tex}
\graphicspath{{figures(colour)/}}
\usepackage{makeidx}
\makeindex

\begin{document}
\input{chapt00_Title}
\frontmatter
%
\chapterstyle{front}
\pagestyle{frontP}
\renewcommand{\theequation}{\arabic{equation}}
\include{chapt_Preface}
\pagestyle{frontToC}
\def\contentsname{Table of Contents}
\tableofcontents

%
%
\mainmatter
%
\chapterstyle{main}
\pagestyle{main}
\setsecnumdepth{subsubsection}
\renewcommand{\theequation}{\thesection.\arabic{equation}}

\markright{Introduction}
\include{chapt0_Intro}

\include{chaptA_Construction}

\include{chaptB_Generating-sets}

\include{chaptC_Bounded-homeos}

\include{chaptD_Presentations}

\include{chaptE_Isos-and-autos}
\appendix
\setcounter{section}{0}
\include{chaptN_Notes}

\backmatter
\chapterstyle{back}
\pagestyle{backB}
\markboth{Bibliography}{Bibliography}
\addcontentsline{toc}{chapter}{Bibliography}
\bibliographystyle{amsalpha}%
\bibliography{Z.Main}
\markright{Bibliography}
\clearpage
%
%
%

%
\pagestyle{backSI}
\renewcommand{\indexname}{Subject Index}
\printindex

\end{document}

%% file: chapt00_Title.tex
\thispagestyle{empty}
{
\vspace*{1cm}
\usefont{T1}{cmss}{m}{n} \fontsize{14pt}{0pt} \selectfont 
\begin{center}
Robert Bieri and Ralph Strebel
\end{center}
\vspace*{1cm}
\usefont{T1}{cmss}{bx}{n}  \fontsize{17.0pt}{0pt} \selectfont
\begin{center}
On Groups of PL-homeomorphisms \\[3mm]
of the Real Line
\end{center}
\vspace{1 cm}
\usefont{T1}{cmss}{m}{sl} \fontsize{12pt}{0pt} \selectfont 
\begin{center}
with Preface and Notes by the second author
\end{center}

\vspace*{2cm}

\renewcommand{\rmdefault}{\sfdefault}
\begin{abstract}
\noindent
Richard J. Thompson invented his group $F$ in the 60s;
it is a group full of surprises:
it has a finite presentation with 2 generators and 2 relators,
and a derived group that is simple;
it admits a peculiar infinite presentation
and has a local definition 
which implies that $F$ is dense in the topological group of all orientation preserving homeomorphisms of the unit interval.

In this monograph groups $G(I;A,P)$ are studied
which generalize the local definition of Thompson's group $F$ 
in the following manner:
the group $G(I;A,P)$ consists of all orientation preserving PL-auto-homeomorphisms of the real line with support in the interval $I$, 
slopes in the multiplicative subgroup $P$ of the positive reals
and breaks in a \emph{finite} subset of the additive $\Z[P]$-submodule 
$A$ of $\R_{\add}$.
If $I$ is the unit interval, 
$A$ the subring $\Z[1/2]$ and $P$ the cyclic group generated by the integer 2,
one recovers Thompson's group $F$. 

A first aim of the monograph is to investigate
in which form familiar properties of $F$ continue to hold for these groups.
Here is a sample: the group $F$ is known to be simple by abelian;
we shall see that the group $G(I;A,P)$ has the same form 
if $I$ is a compact interval with endpoints in $A$
and that it is, in general, simple by soluble of derived length at most 3.
Main aims of the monograph are the determination of isomorphisms among the groups $G(I;A,P)$ and the study of their automorphism groups. 
Complete answers are obtained 
if the group $P$ is not cyclic or if the interval $I$ is the full line. 
In the case of automorphisms, 
these answers are the only known results about $\Aut G(I;A,P)$,
save for findings due to M. G. Brin and to Brin-Guzm{\'a}n 
that deal with special cases not covered by the monograph.
\end{abstract}
}
\newpage
%

%% file: chapt_Preface.tex
%
\chapter{Preface}
\label{chap:Preface}
\markboth{Preface}{Preface}
\addtocontents{toc}{\setcounter{tocdepth}{0}}
The main part of this monograph is a corrected and augmented version
of the Bieri-Strebel memoir \cite{BiSt85}.
The manuscript of that memoir was completed in the autumn of 1985; 
it was then typed and the typescript was distributed to a few mathematicians.
At the beginning of 2014,
I started to revise it, 
my original plan being to stay close to the original text.
Little by little,
I had to abandon this idea for reasons explained in Section 
\ref{sec:Differences-summary}
of the chapter called \emph{Notes}.
That chapter contains two further sections:
Section \ref{ssec:Changes-details} details the differences 
between the memoir and this monograph,
and Section \ref{sec:Related-articles} surveys results 
obtained after 1985 
and related to the topics of the memoir. 

At the end of 2014,
a corrected version of the memoir 
was published in the repository \emph{arXiv} as \cite{BiSt14}.
For this version the text has been reworked once more,
the chapter \emph{Notes} and the bibliography have been updated, 
a Subject Index has been added,
and a number of misprints have been emended.

\section*{Origin of the memoir}
\label{sec:Origin-BiSt85}

In May 1984, Robert Bieri and I attended a meeting at the \emph{Mathematisches Forschungsinstitut Oberwolfach} (Germany).
There Ross Geoghegan introduced us to some recent discoveries of Matthew Brin and Craig Squier 
concerning groups of piecewise linear homeomorphisms of the real line; 
many of these findings were to be published later on in \cite{BrSq85}.
\index{Geoghegan, R.}%
\index{Brin, M. G.}%
\index{Squier, C. C.}%

These discoveries are a fascinating mixture of general results 
and explicit computations with specific groups.
To describe them I need a bit of notation.
Let $\PL_o(\R)$ denote the group of all orientation-preserving piecewise linear homeomorphisms 
of the real line
with only finitely many breaks.
\label{notation:PLo(R)}%
\index{Group PLo(R)@Group $\PL_o(\R)$!definition|textbf}%
It is easy to see that this group is torsion-free and  locally indicable.
Brin and Squier establish that it satisfies no laws
and that it contains no non-abelian free subgroups.
\index{Group PLo(R)@Group $\PL_o(\R)$!properties}%

The group $\PL_o(\R)$ has intriguing subgroups:
first a finitely presented subgroup, 
discovered by Richard J. Thompson in about 1965, 
\index{Thompson, R. J.}%
later rediscovered independently by Freyd and Heller (from a homotopy viewpoint) 
\index{Freyd, P.}%
\index{Heller, A.}%
and by Dydak and Minc (from a shape theory viewpoint)
(see \cite{FrHe93} and \cite{Dyd77}).
\index{Dydak, J.}%
\index{Minc, P.}%
The group has a very peculiar infinite presentation,
namely
\begin{equation}
\label{eqP:Infinite-presentation-F}
\index{Thompson's group F@Thompson's group $F$!infinite presentation}%
F = \langle x_0, x_1, x_2, \ldots \mid \act{x_i}x_j = x_{j+1} \text{ for all indices } i < j \rangle.
\end{equation}
(I use left action, as employed in the memoir \cite{BiSt85}.)

According to Thompson's manuscript \cite{Tho74},
\index{Thompson, R. J.}%
the derived group of $F$ is a minimal normal subgroup of $F$;
it is actually simple,
a fact discovered around 1969 by Freyd and Heller 
\index{Freyd, P.}%
\index{Heller, A.}%
(see part (T3) of the Main Theorem in \cite{FrHe93}), 
but known to Thompson, too.

Brin and Squier studied a second subgroup $G$ of $\PL_o(\R)$.
\index{Brin, M. G.}%
\index{Squier, C. C.}%
It consists of all PL-homomorphisms of $\R$ 
with supports contained in the interval $[0, \infty[$\,,
having only finitely many singularities, all at dyadic rationals, 
and with slopes that are powers of 2.
The group $G$ has the infinite presentation
\begin{equation}
\label{eq:Presentation-G}
\langle x_0, x_1, x_2, \ldots \mid\act{x_i} x_j =   x_{2j-i}\text{ for all indices } i < j \rangle,
\end{equation}
but it has also the finite presentation
\begin{equation}
\label{eq:Finite-presentation-G}
\langle x_0, x_1,  \mid \act{x_1 x_0}x_1 =   \act{x_0^2} x_1\text{ and  } 
\act{x_1 x_0^2}x_1 =   \act{x_0^4} x_1\rangle.
\end{equation}

The subgroups of $\PL_o(\R)$  exhibit further surprises.
In \cite{Tho74}, 
R. J. Thompson  \index{Thompson, R. J.}%
represents the group $F$ by PL-homeomorphisms 
with supports in the unit interval $[0,1]$;
as discovered by him and, 
independently, by P. Freyd and A. Heller, 
it can also be realized by PL-homeomorphisms
with supports in the half line $[0,\infty[$ or with supports in the real line.
\footnote{see the \emph{First Canonical Representation} and the \emph{Second Canonical Representation}, defined on pages 100 and 102 in \cite{FrHe93}.}
\index{Freyd, P.}%
\index{Heller, A.}%

\section*{Layout of the memoir}
\label{sec:Layout-of-BiSt85}
The groups $F$ and $G$ have explicit finite presentations 
with generating sets made up of concretely given PL-homeomorphisms.
In addition,
they have definitions that might be called \emph{local}:
the group $F$ consists of all PL-homeomorphisms 
with supports in the unit interval $I = [0,1]$,
slopes in the multiplicative group $P$ generated by 2
and breaks in the $\Z[P]$-submodule $\Z[1/2]$ of the additive group of $\R$,
while $G$ is made up of all PL-homeomorphisms with supports in the half line
$I = [0,\infty[$,
slopes in $P = \gp(2)$ and breaks in $A = \Z[1/2]$.

These local definitions of $F$ and of $G$ form the starting point of the Bieri-Strebel memoir \cite{BiSt85}.
Similarly to what M. G. Brin and C. C. Squier propose on page 490 of \cite{BrSq85},
\index{Brin, M. G.}%
\index{Squier, C. C.}%
the authors consider subgroups of $\PL_o(\R)$, 
depending on three parameters $I$, $A$ and $P$,
where  $I \subseteq \R$ is a closed interval, 
\label{notation:I}%
$P$ is a subgroup of the multiplicative group of the positive reals 
$\R^\times_{> 0}$
\label{notation:P}%
and  $A$ is a $\Z[P]$-submodule 
\label{notation:Z[P]}%
of the additive group $\R_{\add}$.
\label{notation:A}%
To every such triple they attach the subset
\begin{equation*}
\label{eq:Definition-G(I;A, P)}
\index{Group G(I;A,P)@Group $G(I;A,P)$!definition|textbf}%
\index{Group G(R;A,P)@Group $G(\R;A,P)$!definition|textbf}%
\index{Group G([0,infty[;A,P)@Group $G([0, \infty[\;;A,P)$!definition|textbf}%
\index{Group G([a,c];A,P)@Group $G([a,c];A,P)$!definition|textbf}%
G (I; A, P) = \{ g \in \PL_o(\R) \mid  g \text{ has support in $I$,  slopes in  $P$, breaks in }A\}.
\end{equation*}
This subset is closed under composition of functions 
and under passage to the inverse
\footnote{If $I = \R$ this holds only if one requires, in addition,  
that the PL-homeomorphisms maps $A$ onto itself. 
I shall impose this condition from now on.}
and $G (I; A, P) $ equipped with this composition is a group which,
by abuse of notation, will again be denoted by $G (I; A, P)$.
In order to avoid trivialities, the authors require, in addition, 
that $I$ have positive length, that $P \neq \{1\}$ and that $A \neq \{0\}$.
These non-triviality assumptions imply 
that $A$ is a dense subgroup of $\R$.

\subsection*{Themes}
\label{ssec:Themes}
\index{Main themes of the monograph}%
The group $F$ has many striking properties;
the following five are chosen by R. Bieri and R. Strebel as themes of their memoir: 
\begin{enumerate}[(i)]
\item $F$ is dense in the space of all orientation preserving homeomorphisms 
of the compact interval $[0,1]$,
\footnote{with respect to the $L_\infty$-norm}
\item its derived group is simple,
\item  $F$ is finitely generated, 
\item  $F$ admits a finite presentation, and
\item $F$  is isomorphic to subgroups of $\PL_o(\R)$ 
with supports in $[0, \infty[$\;. 
\end{enumerate}
Analogues of some of these properties hold in great generality,
analogues of others only in rather restricted settings.
In the sequel, 
I briefly comment on the findings obtained in \cite{BiSt85}
and reproduced in the main part of this monograph.
\subsection*{Construction of PL-homeomorphisms}
\label{ssec:Construction-PL-homeos}
I begin with property (i).
The fundamental question here is whether, 
given positive numbers $b$, $b'$  in $A$, 
there exists a PL-homeomorphism $g$ with slopes in $P$ and breaks in $A$ such that $g([0,b]) = [0, b']$.
The answer involves a submodule of $A$ 
that is familiar from the Homology Theory of Groups,
\index{Homology Theory of Groups!submodule IPA@submodule $IP \cdot A$}%
namely
\begin{equation}
\label{eqP:definition-IPA}
IP \cdot A = \left\{ a \in A \mid \exists p_j \in P  \text{ and } \exists a_j \in A
\text{ with  } a = \sum\nolimits_{j} (p_j-1) \cdot a_j\right\}.
\end{equation}
In view of the non-triviality hypothesis, this submodule is dense in $\R$.
\index{Submodule IPA@Submodule $IP \cdot A$!definition|textbf}%
\index{Submodule IPA@Submodule $IP \cdot A$!density property}%

The answer now reads like this (see Theorem \ref{TheoremA}):
\begin{thmP}
\label{thm:Constructing-homeos}
\index{Submodule IPA@Submodule $IP \cdot A$!significance}%
\index{Construction of PL-homeomorphisms!basic result}%
\index{Theorem \ref{TheoremA}!statement}
A PL-homeomorphism with the stated properties exists if,
and only if, $b'-b \in IP \cdot A$.
\end{thmP}

It follows, first of all, that property (i) extends to all groups $G(I;P,A)$ with $I$ a compact interval
(see Corollary \ref{CorollaryA5}).
Other consequences hold for groups with  intervals $I$ distinct from $\R$,
in particular the following ones (see Corollary \ref{CorollaryA4}):
every orbit $\Omega$ of $G = G(I;A,P)$ 
that is contained in the intersection of $A$ 
and the interior $\Int(I)$ of $I$  \label{notation:Int(I)}
has the form $(a + IP \cdot A) \cap \Int(I)$
and it is dense in $\Int(I)$. 
In addition,
the group $G$ acts transitively on the collection of all ordered lists 
$a_1 < a_2 < \cdots < a_\ell$ with elements $a_i$ in $\Omega$.
(Here $\ell$ is a positive integer.)
This addendum will be a key to the answer to item (ii).
\index{Multiple transitivity}%
\index{Group G(I;A,P)@Group $G(I;A,P)$!multiple transitivity}%
\index{Orbits of G(I;A,P)@Orbits of $G(I;A,P)$}%

\subsection*{Simplicity of the derived group}
\label{ssec:Simplicity}
In the preceding section, 
a property of the groups $G(I;A,P)$ has been stated
that holds for all choices of the parameters $A$, $P$
and which can be formulated in terms of the group $P$, 
the module $A$,  
and the closely related submodule $IP \cdot A$.
In this section, 
I explain how property (ii) can be rephrased in such a way 
that it holds for all groups of the form $G(I;A,P)$.

The derived group of Thompson's group $F$ coincides with the subgroup $B$ 
made up of the elements that are the identity near 0 and near 1.
For every choice of $I$, $A$ and $P$,
the group $G(I;P,A)$ contains an analogously defined subgroup,
namely:
\begin{equation}
\label{eq:Definition-B(I;A,P)}
B(I;P,A) = \{g \in G (I; A, P) \mid g \text{ is the identity near the endpoints of } I \}.
\end{equation}
If $I$ is the line or a half line, 
the stated requirement for $g$ is to be interpreted suitably; 
see section \ref{ssec:2.3} for a precise formulation.

The group $G(I;A,P)$ can be viewed as a permutation group of the linearly ordered set $\R$
and the same is true for its subgroups.
The facts enunciated in the paragraph following Theorem 
\ref{thm:Constructing-homeos} imply 
that the normal subgroup $B(I;P,A)$ acts doubly transitively on its orbits $\Omega$ in $A \cap \Int(I)$;
more precisely, 
it acts transitively on the subset of all ordered pairs $\{a_1, a_2\} \in \Omega^2$ 
with $a_1 < a_2$.
\index{Doubly transitive action}%
Since each orbit is dense in $I$, 
the action of $B(I;P,A)$ on an orbit is faithful;
a theorem of G. Higman's  thus allows one to infer
that the \emph{derived group of $B = B(I;A,P)$ is non-abelian and simple}.
(For more information, see section \ref{ssec:10.2}).
\index{Higman, G.}%
\index{Simplicity result for [B,B]@Simplicity result for $[B,B]$}%
\index{Subgroup B(I;A,P)@Subgroup $B(I;A,P)$!simplicity of [B,B]@simplicity of $[B,B]$}%

The quotient group $G/B$ can be expressed
 in terms of the parameters $A$ and $P$:
if $I$ is a compact interval with endpoints in $A$
it is abelian and isomorphic to $P \times P$;
if $I$ is the real line $\R$ or the half line $[0, \infty[$ it is metabelian, 
but not abelian
(a fuller description is given by Corollary \ref{CorollaryA3}).
It follows, in particular,  that $G(I;A,P) $ is an extension of a simple group 
by the soluble group $G/B'$.

In order to arrive at a more detailed description of the group $G/B'$
 one has to determine the abelianization $B_{\ab}$ of $B$ and the action of $G/B$ on $B_{\ab}$.
These topics  are investigated in Section \ref{sec:12}.
\index{Subgroup B(I;A,P)@Subgroup $B(I;A,P)$!abelianization}%
A first finding is 
that the isomorphism type of $B$ does not depend on the interval  $I$ 
(see Proposition \ref{PropositionC1}).
\index{Subgroup B(I;A,P)@Subgroup $B(I;A,P)$!independence on I@independence on $I$}%
In the analysis of $B_{\ab}$ one can therefore assume that $I = [0,\infty[$ 
whence $G = G(I;A,P)$ has an explicitly known presentation; 
for details  see section \ref{ssec:14.1}.

The results obtained in Section \ref{sec:12} 
can be summarized as follows:
by Proposition \ref{PropositionC4} 
the abelian group $B_{\ab}$ is the middle term of a right exact sequence
\begin{equation}
\label{eq:Describing-B/B'}
\index{Subgroup B(I;A,P)@Subgroup $B(I;A,P)$!abelianization}
\index{Subgroup K(A,P)@Subgroup $K(A,P)$!significance}%
\bar{L} \to B_{\ab} \xrightarrow{\bar{\mu}} K(A,P) \to 0.
\end{equation}
The group $K(A,P)$ is studied in section \ref{ssec:12.3New}.
\index{Subgroup K(A,P)@Subgroup $K(A,P)$!properties}%
By Proposition \ref{prp:LemmaC5},
it is \emph{trivial} if, and only if, $A = IP \cdot A$; 
it is  \emph{finitely generated} if, and only if,  
either $A = IP \cdot A$, or $A/(IP \cdot A )$ is finite 
and $P$ is a finitely generated group.
Moreover, if $K(A,P)$ is finitely generated 
it is free abelian of rank $\rk P \cdot (\card(A/(IP \cdot A)) -1)$.
\label{notation:rk}%

The group $\bar{L}$ can  be described in terms of homology groups;
in some cases, 
this description allows one to deduce that  $\bar{L} $ is reduced to 0.
Section \ref{sssec:12.3c} has some details on this matter.
%
%
\section*{Finiteness properties}
\label{sec:Finiteness-properties}
%
I come now to a third topic treated in the memoir,
finiteness properties of the groups $G(I;A,P)$.
Actually only the properties of \emph{finite generation},
of \emph{finite generation of the abelianized group}  
and of \emph{finite presentation} are studied.
%
\subsection*{Finite Generation}
\label{ssec:Finite-generation-group-of-G}
I begin with a remark.
One of the peculiarities of the Bieri-Strebel memoir is the fact
that the influence of the interval $I$ 
on the groups $G(I;A,P)$ and on their subgroups 
is studied systematically.
For some questions, this influence is small:
one knows, for instance,
that the subgroup $B(I;A,P)$ is isomorphic to $B(\R;A,P)$ 
for every interval $I$ 
(with non-empty interior; see Proposition \ref{PropositionC1}).
In addition,
the isomorphism type of the groups $G([0, b];A, P)$ 
does not depend on $b \in A_{>0}$
whenever $P$ is cyclic (Theorem \ref{TheoremE07}).
\index{Theorem \ref{TheoremE07}!consequences}%
For other questions, however, the answers depend strongly on the type of interval.
This is the case for the determination of isomorphisms and automorphism groups,
a topic discussed later on in this Preface,
but also for the problems of finding out when a group $G(I;A,P)$ is finitely generated
or  when it admits a finite presentation.

The question of finite generation is the subject matter of Chapter \ref{chap:B};
here the interval $I$ plays a crucial rôle.
Suppose first that $I$ is either the line or the half line $[0, \infty[$\,.
Then the ordered pair $(A,P)$ 
for which the group $G(I;A,P)$ is finitely generated
can be characterized in a simple manner.
Indeed,
by combining Theorems \ref{TheoremB2} and \ref{TheoremB4} 
one arrives at
\begin{thmP}
\label{thm:Finite-generation}
\index{Group G(R;A,P)@Group $G(\R;A,P)$!finite generation}%
\index{Group G([0,infty[;A,P)@Group $G([0, \infty[\;;A,P)$!finite generation}%
\index{Finiteness properties of!G(R;A,P)@$G(\R;A,P)$}%
\index{Finiteness properties of!G([0,infty[;A,P)@$G([0, \infty[\;;A,P)$}%
\index{Submodule IPA@Submodule $IP \cdot A$!significance}%
\index{Homology Theory of Groups!H0(P,A)@$H_0(P,A)$}%
Let $I$ be one of the intervals $\R$ or $[0, \infty[\,$. 
Then $G(I;A,P)$ is finitely generated if, and only if,
\begin{enumerate}[(i)]
\item $P$ is a finitely generated group,
\item $A$ is a finitely generated $\Z[P]$-module, and 
\item the quotient module $ A/(IP \cdot A)$ is finite or  $I = \R$.
\end{enumerate}
\end{thmP}

The proof depends on a result of Brin and Squier, 
\index{Brin, M. G.}%
\index{Squier, C. C.}%
according to which the groups $G(\R; A, P)$ and $G([0, \infty[\,; A, P)$
admit convenient (infinite) presentations 
whose generators are PL-homeomorphisms 
with at most one break (see \cite[Corollary 2.8]{BrSq85}).

No similar result is known for compact intervals $I = [0,b]$ with $b \in A_{>0}$. 
\label{page:Difficulty-of-finding-generators}
Conditions (i) through (iii) listed in Theorem \ref{thm:Finite-generation} 
are still necessary for finite generation (see Proposition \ref{PropositionB1}),
but they may no longer be sufficient.
To date, the finite generation property seems to have been established 
only for the following special choices of $A$ and $P$:
\index{Group G([a,c];A,P)@Group $G([a,c];A,P)$!finite generation}%
\index{Submodule IPA@Submodule $IP \cdot A$!significance}%
\begin{enumerate}
\index{Group G([a,c];A,P)@Group $G([a,c];A,P)$!finite generation|(}%
\index{Finiteness properties of!G([a,c];A,P)@$G([a,c];A,P)$|(}%
\item[(e)] $P$ is generated by positive integers $p_1$, \ldots,  $p_k$
and $A = \Z[1/(p_1 \cdots p_k)]$ 
(K. S. Brown (unpublished), 
\cf{}\cite[Thm.\,2.5]{Ste92},  or Corollary \ref{CorollaryB10});
\index{Brown, K. S.}%
\index{Stein, M.}%
\item[(f)] $P = \gp(u)$ and $A = \Z[u, u^{-1}]$ 
with $u$ an algebraic integer of the form
$u = (\sqrt{n^2 + 4}-n)/2 $ and $n$ a positive integer
(S. Cleary \cite{Cle95} and \cite[Section 4]{Cle00}).
\index{Group G([a,c];A,P)@Group $G([a,c];A,P)$!finite generation|)}%
\index{Finiteness properties of!G([a,c];A,P)@$G([a,c];A,P)$|)}%
\index{Cleary, S.}%
\end{enumerate}
In addition, a general fact is known:
whether or not a group $G([0,b];A, P)$  is finitely generated does not depend on $b \in A_{>0}$
(see Lemma \ref{LemmaB6}).
(Groups with distinct parameters $b_1$, $b_2$ may not be isomorphic; 
see Corollary  \ref{CorollaryE12}).
%
\subsection*{Finite generation of the abelianized group}
\label{ssec:Finitely-generated-abelianization}
In the memoir,
the abelianization of the groups $G = G(I;A,P)$ 
is only studied for intervals that are bounded from below or from above,
but not on both sides, in particular for $I = [0, \infty[\,$.
The following facts induced the authors to concentrate on this special case:
the isomorphism type of $B = B(I;A,P)$ does not depend on $I$ 
(Proposition \ref{PropositionC1}),
the study of $B_{\ab}$ is relevant for the question 
whether $B$ is simple,
the extension $B \triangleleft G \epi G/B$ 
and methods from the \emph{Homology Theory of Groups} allow 
one to relate the abelianizations of $B$, $G$ and $G/B$, 
and last, but not least, $G([0,\infty[\;;A,P)$ 
has a very convenient, infinite presentation.

In section \ref{ssec:12.1}, 
this presentation is used to obtain a description 
of the abelianization of $G = G([0, \infty[\;; A, P)$ in terms of $P$ 
and the abelian group $A/(IP \cdot A)$
(see Proposition \ref{PropositionC2} and its proof).
The cited proposition allows one then to obtain
a criterion for the finite generation of $G_{\ab}$:
\begin{crlP}[Corollary \ref{crl:PropositionC2}]
\label{crlP:PropositionC2}
\index{Submodule IPA@Submodule $IP \cdot A$!significance}%
\index{Homology Theory of Groups!H0(P,A)@$H_0(P,A)$}%
If $I$ is a half line,
the abelianization of $G = G(I;A,P)$ is \emph{finitely generated} if, and only if,
$P$ is finitely generated and and $IP \cdot A$ has finite index in $A$.
If $G_{\ab}$ is finitely generated, it is free abelian of rank 
\[
 \rk P \cdot \left(1 + \card A/(IP \cdot A)\right) 
\]
if the end point of $I$ lies in $A$, 
and otherwise of rank $ \rk P \cdot \left( \card A/(IP \cdot A) \right) $.
\end{crlP}
%
\subsection*{Finite presentation}
\label{ssec:Finite-presentation-of-group-G}
%
Chapter \ref{chap:D} of the Bieri-Strebel memoir addresses the question 
whether a  group of the form $G(I;A,P)$ admits a finite presentation.
The obtained results are far less general than Theorem \ref{thm:Finite-generation};
but, again,  they are more satisfactory for the line or a half line than for compact intervals.

A first general insight is provided by the next proposition;
it is an immediate consequence of  Propositions \ref{PropositionD2} 
and \ref{PropositionD5}.
\begin{prpP}
\label{prp:D2-D5-preface}
\index{Group G(R;A,P)@Group $G(\R;A,P)$!finite presentation}%
\index{Group G([0,infty[;A,P)@Group $G([0, \infty[\;;A,P)$!finite presentation}%
Suppose $I = \R$ or $I = [0,\infty[\,$.
If $G(I;A,P)$ admits a finite presentation 
so does the metabelian group $\Aff(A,P) \iso A \rtimes P$.
\end{prpP}
As is well-known,
a finitely generated metabelian group of the form $A \rtimes P$ 
need not admit a finite presentation
(see, \eg{}\cite{BiSt80} or \cite{Str84}).
\index{Bieri, R.}%
\index{Strebel, R.}%
The above proposition thus implies 
that many groups of the form $G(\R;A,P)$ or $G([0,\infty[\,;A,P)$ are finitely generated, 
but infinitely related, for instance the group $G(\R; \Z[1/6], \gp(3/2))$.

No analogue of Proposition \ref{prp:D2-D5-preface}
is known for the interval $I = [0, b]$;
indeed, it seems that every group of the form $G([0,b];A,P)$ 
that has been shown to be finitely generated 
is actually finitely presented and of type $\FP_\infty$.
\index{Finiteness properties of!G([a,c];A,P)@$G([a,c];A,P)$}
\index{Group G([a,c];A,P)@Group $G([a,c];A,P)$!finite presentation}
\index{Group G([a,c];A,P)@Group $G([a,c];A,P)$!type FPinfty@type $\FP_\infty$}
\smallskip

Now to some examples of groups of the form $G(I;A,P)$ 
that have been \emph{proved to admit a finite presentation}.
Their list is quite short; 
in displaying it,
I denote by $\gp(\XX)$ the group generated by the set $\XX$. \label{notation:gp(XX)}%
\index{Group G(R;A,P)@Group $G(\R;A,P)$!finite presentation}%
\index{Group G([0,infty[;A,P)@Group $G([0, \infty[\;;A,P)$!finite presentation}%
\index{Group G([a,c];A,P)@Group $G([a,c];A,P)$!finite presentation}%
\index{Finiteness properties of!G(R;A,P)@$G(\R;A,P)$}%
\index{Finiteness properties of!G([0,infty[;A,P)@$G([0, \infty[\;;A,P)$}%
\index{Finiteness properties of!G([a,c];A,P)@$G([a,c];A,P)$}%
\begin{enumerate}[(a)]
\item $I =\R$, $P$ freely generated by a finite set of integers $p_j \geq 2$ 
and $A = \Z[P]$ 
(Proposition \ref{PropositionD4} and example 1 in section \ref{sssec:13.3cNew};
the cited proposition yields actually some further examples);
\item $I = [0, \infty[$, $P = \gp(p)$  with $p \geq 2$ an integer 
and $A = \Z[1/p]$ (\cite[\S{}2]{BrSq85});
\index{Brin, M. G.}%
\index{Squier, C. C.}%
\item $I = [0, \infty[$, $P = \gp(p_1, \ldots, p_k)$ with each $p_j$ a positive integer and  $A = \Z[P]$ (Theorem \ref{TheoremD6});
 \item $I = [0, b]$ with $b \in A$,   and $P = \gp(p)$ with $p > 1$ an integer, 
 and  $A = \Z[1/p]$ 
 (\cite[Theorem 4.17]{Bro87a} or Proposition \ref{PropositionD10});
 \index{Brown, K. S.}%
\item  $I = [0, b]$ with $b \in A$, $P = \gp(p_1, \ldots, p_k)$ with each $p_j$ a positive integer 
and $A = \Z[P]$ (K. S. Brown (unpublished); 
\cf{}\cite[Theorem 2.5]{Ste92});
\index{Stein, M.}%
\item $I = [0, b]$ with $b \in A = \Z[u, u^{-1}]$, $P = \gp(u)$ 
where $u$ is the algebraic integer $ (\sqrt{n^2 + 4}-n)/2 $ 
with $n$ a positive integer 
(see \cite{Cle95} and \cite[Section 4]{Cle00}).
\index{Cleary, S.}%
\end{enumerate}
Notice that every module $A$ figuring in the above list 
is cyclic and is actually the module underlying the subring $\Z[P]$,
and that the group $P$ is always generated by rational numbers, except in case (f).
\section*{Isomorphisms and automorphisms}
\label{sec:Isos-and-autos}
Isomorphisms and automorphisms are topics
for which very satisfactory results are known;
Chapter \ref{chap:E} focusses on them.
All isomorphisms are induced 
by conjugation by homeomorphisms.

These homeomorphisms are obtained by two methods.
In the first one,
the homeomorphism $\varphi \colon \Int(I) \iso \Int(\bar{I})$ 
is the result of an explicit construction.
In the simplest case, $\varphi$  is (the restriction of) an \emph{affine map} 
$t \mapsto s\cdot t + b$.
Translations allow one to see that, for every $a \in A$, 
the group $G([0, \infty[\,; A, P)$ is isomorphic to $G([a, \infty[\,; A, P)$  
and $G([0, b]; A, P)$ is isomorphic to $G([a, a + b]; A, P)$.
The homothety $\alpha \colon t \mapsto s \cdot t$ with $s > 0$ implies
that $G([0, \infty[\,; A, P)$ is isomorphic to $G([0, \infty[\,; s \cdot A, P)$ 
and that $G([0, b]; A, P)$ is isomorphic to $G([0, s \cdot b]; s \cdot A, P)$.
Farther reaching are the effects of \emph{PL-homeomorphisms 
with infinitely many breaks} 
that do not accumulate in the interior of $I$.
With the help of such homeomorphisms one can show 
that the isomorphism type of $B(I;A,P)$ does not depend on the interval $I$ 
(see Proposition \ref{PropositionC1}),
and  that the isomorphism type of  $G([0, b]; A, P)$ with $b \in A$ and $P$ cyclic 
does not depend on $b$
(see Theorem \ref{TheoremE07}). 
\index{Theorem \ref{TheoremE07}!consequences}%
A further application will be discussed on page  
\pageref{sssec:Isomorphisms-between-groups-containing-B} below.

The second route leading to homeomorphisms $\varphi$ 
makes use of the proof of the Main Theorem in S. H. McCleary's paper \cite{McC78b} \index{McCleary, S. H.}%
and of transitivity properties of the subgroup $B(I;A,P)$ (established in Section \ref{sec:5}). 
This route is detailed in sections \ref{ssec:16.1} through \ref{ssec:16.3} 
and culminates in
\begin{thmP}[Theorem \ref{TheoremE04}]
\label{thm:TheoremE04-intro}
\index{Representation Theorem for isomorphisms!main result}%
\index{Theorem \ref{TheoremE04}!statement}
Assume $G$ is a subgroup of $G(I;A,P)$  containing the derived group $B'$ of $B = B(I;A,P)$
and that $\bar{G}$ is a subgroup of $G(\bar{I}; \bar{A}, \bar{P})$ 
with the analogous property. 
If $\alpha \colon G \iso \bar{G}$  is an isomorphism of groups
there exists a unique homeomorphism $\varphi \colon \Int(I) \iso \Int(\bar{I}) $
that induces $\alpha$ by conjugation;
more precisely, 
the equation
\begin{equation}
\label{eq:Inducing-iso-groups}
\alpha(g) {\restriction{\Int(\bar{I})  }} = \varphi \circ 
\left( 
g {\restriction{\Int (I)} } 
\right) 
\circ \varphi^{-1}
\end{equation}
holds for every $g \in G$. 
In addition, $\varphi$ maps $A \cap \Int(I)$ onto $\bar{A} \cap  \Int(\bar{I})$.
\end{thmP}
\begin{remarkP}
\label{remark:Improvements-of-TheoremE04}
Theorem \ref{thm:TheoremE04-intro} is deduced from Theorem \ref{TheoremE3}, 
a result in the spirit of the Main Theorem 4 of  \cite{McC78b};
\index{McCleary, S. H.}%
Theorem \ref{TheoremE3} deals with certain orientation preserving permutation groups of an open interval $J \subseteq \R$.
As pointed out on page 15 of \cite{McRu05}, 
the proof of Theorem \ref{TheoremE3} can be adjusted 
so as to yield a sharpening of the Main Theorem 4 of  \cite{McC78b}.
\index{McCleary, S. H.}%
\index{Rubin, M.}%
In the more recent paper \cite{McRu05}, 
S. McCleary and M. Rubin generalize this sharpened result further; 
see Theorem 4.1 in  \cite{McRu05}.
\end{remarkP}

Theorem \ref{thm:TheoremE04-intro} has many consequences; 
some of them will be mentioned in this Preface.
A first one deals with the subgroup $B(I;A,P)$.
This group is normal in $G(I;A,P)$. 
Theorem \ref{TheoremE04} now implies 
that $B$ is mapped onto itself 
by every automorphism of a subgroup $G$ of $G(I;A,P)$
that contains $B$; 
so $B$ is a characteristic subgroup of every subgroup $G \supset B$ 
(for more details, see Corollary \ref{CorollaryE5}). 
\index{Subgroup B(I;A,P)@Subgroup $B(I;A,P)$!properties}%

One of the main problems studied in Chapter \ref{chap:E}  is 
whether or not the homeomorphisms $\varphi \colon \Int(I) \iso \Int(\bar{I})$ 
are piecewise linear,
possibly with infinitely many breaks.
The answers are simpler and more uniform when $P$ is not cyclic 
and hence a dense subgroup of $\R^\times_{>0}$.
I begin therefore by reporting on this case.
%
\subsection*{Case 1: $P$ is not cyclic}
\label{sssec:P-not-cyclic}
The basic result here is
\begin{thmP}[Theorem \ref{TheoremE10}]
\label{thm:TheoremE10}
\index{Group G(I;A,P)@Group $G(I;A,P)$!isomorphisms}%
\index{Theorem \ref{TheoremE10}!statement}
Suppose  $G$ is a subgroup of $G(I;A,P)$ 
that contains the derived group of $B(I;A,P)$
and $\bar{G}$ is a subgroup of $G(\bar{I}; \bar{A}, \bar{P})$ 
with the analogous property. 
Assume $G$ and $\bar{G}$ are isomorphic 
and let $\varphi \colon \Int(I) \iso \Int(\bar{I})$ be a homeomorphism 
inducing an isomorphism $\alpha \colon G \iso \bar{G}$.
If $P$ is \emph{not cyclic} the following statements hold:
\begin{enumerate}[(i)] 
\item $ \bar{P} = P$;
\item there exists a non-zero real $s$ such that  $\bar{A} = s \cdot A$
and $\varphi$ is a PL-homeo\-mor\-phism 
with slopes in the coset $s \cdot P$,
breaks in $A$
and its set of breaks is a discrete subset of $\Int(I)$.
\end{enumerate}
\end{thmP}

In the preceding theorem, 
$G$  is allowed to be $B(I;A,P)$ and similarly for $\bar{G}$.
Then the PL-homeomorphism $\varphi$ can well have infinitely many breaks.
If $G$ is all of $G(I;A,P)$ sharper conclusions can be drawn:
\begin{supplementP}[Supplement \ref{SupplementE11New}]
\label{supplement:SupplementE11New}
\index{Group G(R;A,P)@Group $G(\R;A,P)$!isomorphisms}%
\index{Group G([0,infty[;A,P)@Group $G([0, \infty[\;;A,P)$!isomorphisms}%
\index{Group G([a,c];A,P)@Group $G([a,c];A,P)$!isomorphisms}%
\index{Automorphisms of G(R;A,P)@Automorphisms of $G(\R;A,P)$!properties}%
\index{Automorphisms of G([0,infty[;A,P)@Automorphisms of $G([0, \infty[\;;A,P)$!properties}%
\index{Automorphisms of G([a,c];A,P)@Automorphisms of $G([a,c];A,P)$!properties}%
Assume that $P$ is not cyclic and that $\alpha$ is an isomorphism of $G(I;A, P)$ 
onto a subgroup $\bar{G}$ of $G(\bar{I}; \bar{A}, \bar{P})$
that contains the derived group of $B(\bar{I}; \bar{A}, \bar{P})$.
Let $\varphi \colon \Int(I) \iso \Int(\bar{I})$ 
be the unique homeomorphism inducing $\alpha$.
Assume, furthermore,
that $I$ is the line $\R$ or a half line $[a, \infty[$ with $a \in A$ 
or a compact interval with endpoints in $A$,
and that the intervals $I$and $\bar{I}$ are similar.

Then $\varphi$ is a finitary PL-homeomorphism and $\bar{G} = G(\bar{I}; \bar{A}, \bar{P})$.
\end{supplementP}

This supplement allows one to work out the automorphism group of 
$G =G(I;A,P)$.
The result involves a homomorphism $\eta$
which is obtained as follows:
by part (ii) in Theorem \ref{thm:TheoremE10},
every automorphism $\alpha$ of $G$ gives rise to a well-defined coset 
$ s \cdot P$ of $\R^\times_{>0}/P$.
Since $s \cdot A = A$, multiplication by  $s$ is an automorphism of $A$. 
These automorphisms of $A$ form the group
\begin{equation}
\label{eqP:Definition-Aut(A)}
\index{Group Aut(A)@Group $\Aut(A)$!definition|textbf}%
\Aut(A) = \{ s \in \R^\times \mid s \cdot A = A\}.
\end{equation}
The assignment $\alpha \mapsto \varphi \mapsto s \cdot P$ 
then gives rise to the homomorphism of groups
\begin{equation}
\label{eq:Definition-eta}
\index{Homomorphism!07-eta@$\eta$}%
\eta \colon \Aut G  \to \Aut(A)/P.
\end{equation}
The automorphism group $\Aut G $ of the group $G = G(I;A,P)$ 
can now be described like this:
\begin{crlP}[Corollary \ref{CorollaryE13}]
\label{crl:CorollaryE13-Preface}
\index{Automorphism group of G(R;A,P)@Automorphism group of $G(\R;A,P)$!description}%
\index{Automorphism group of G([0,infty[;A,P)@Automorphism group of $G([0, \infty[\;;A,P)$!description}%
\index{Automorphism group of G([a,c];A,P)@Automorphism group of $G([a,c];A,P)$!description}%
\index{Homomorphism!07-eta@$\eta$}%
\index{Group Aut(A)@Group $\Aut(A)$!significance}%
\index{Subgroup Auto(A)@Subgroup $\Aut_o(A)$!significance}%
\index{Subgroup Qb@Subgroup $Q_b$!definition|textbf}%
\index{Subgroup Qb@Subgroup $Q_b$!significance}%
Set $G = G(I;A,P)$.
The kernel of $\eta$ is the subgroup $\Inn G $ of inner automorphisms of $G$;
its image depends on $I$:  
for $I = \R$ it is $\Aut(A)/P$,
for $I$ the half line $[0, \infty[$ it is  $\Aut_o(A)/P$ 
where $\Aut_o(A) = \Aut(A) \cap \R^\times_{>0}$;
if, finally, $I = [0, b]$ with $b \in A$, 
the image of $\eta$ depends on $b$ and is 
\begin{equation}
\label{eq:Definition-Qb}
Q_b/P \quad \text{with} \quad Q_b =  \left\{ s \in \Aut(A)\mid (|s| -1) \cdot b \in IP \cdot A \right\}.
\end{equation}
\end{crlP}

The answer given by Corollary \ref{crl:CorollaryE13-Preface} is very satisfactory.
To see this, 
assume first that $A \neq \{0\}$ is a subgroup of $\R_{\add}$ 
for which the group 
\begin{equation}
\label{eqP:Definition-Auto(A)}
\Aut_o(A) = \{s \in \R^\times_{>0} \mid s \cdot A = A \}
\end{equation} 
is not cyclic.
If $I = \R$,
the outer automorphism group of $G = G(I; A, \Aut_o(A))$ is of order 2 
and $\Aut G$ consists of the inner automorphisms of $G$
 and of the compositions of inner automorphisms 
 with the the automorphism induced by the reflection $t \mapsto -t$.
 The situation is similar if $I = [0, b]$ with $b \in A$: 
 then $\Inn G $ has again index 2 in $\Aut G$ 
 and the automorphism by conjugation by the reflection $t \mapsto b - t$ 
 represents the coset $\Aut G  \smallsetminus \Inn G $.
 If, finally, $I$ is a half line with endpoint in $A$, 
 then $\Aut G = \Inn G $ 
 and the group $G(I;A; \Aut_o(A))$ is complete.
 
 The assumptions that $P$ is all of $\Aut_o(A)$ 
 and that it is not cyclic hold, in particular, 
for $\PL_o(\R) = G(\R; \R, \R^\times_{>0})$, a group considered in Corollary 31 of \cite{McC78b}.
Corollary \ref{crl:CorollaryE13-Preface}
 thus puts McCleary's result into a larger context.
 \index{McCleary, S. H.}%

There is a second aspect of Corollary \ref{crl:CorollaryE13-Preface} that deserves mention,
namely the fact that, given parameters $I$, $A$ and $P$ as in the corollary,
$\Aut G(I;A,P) $ has a subgroup $\Aut_o G(I;A,P)$ of index at most 2,
which is a subgroup of the group $G(I;A, \Aut_o(A))$.
Here $\Aut_o G(I;A,P)$ denotes the subgroup of \emph{increasing} automorphisms,
the automorphisms induced by orientation preserving auto-homeo\-morphisms.
\label{notation:Auto-G(I;A,P)}%
\index{Group Auto G(I;A,P)@Group $\Aut_o G(I;A,P)$!definition|textbf}%
Some consequences of this fact are spelled out by Corollary
\ref{crl:Explicit-description-of-Aut(G)}.

Supplement \ref{SupplementE11New} allows one also to determine 
when two intervals $[0,b_1]$ and $[0,b_2]$ lead to isomorphic groups.
The answer depends once more on the automorphism group of $A$, 
more precisely on its index 2 subgroup  $\Aut_o(A)$.
\begin{prpP}[Corollary \ref{CorollaryE12}]
\label{preface:CorollaryE12}
\index{Subgroup Auto(A)@Subgroup $\Aut_o(A)$!significance}%
\index{Group G([a,c];A,P)@Group $G([a,c];A,P)$!isomorphisms}%
\index{Group G([a,c];A,P)@Group $G([a,c];A,P)$!dependence on [a,c]@dependence on $[a,c]$}%
Assume $P$ is not cyclic and $b_1$, $b_2$ are positive elements of $A$.
Then the groups $G([0,b_1];A,P)$ and $G([0,b_2];A,P)$ are isomorphic if, and only if,
$b_1 + IP \cdot A$ and $b_2 + IP \cdot A$ 
lie in the same orbit of $\Aut_o(A)/P$.
\end{prpP}
%
\subsection*{Case 2: $P$  is cyclic and $I$ is unbounded}
\label{ssec:P-cyclic-I-unbounded}
If $P$ is cyclic,
the homeomorphisms 
$
\varphi \colon \Int (I) \iso  \Int (\bar{I})
$
inducing isomorphisms $\alpha \colon G \iso \bar{G}$
are more diverse than those arising for non-cyclic groups $P$.
First of all, $\varphi$ can be an infinitary PL-homeomorphism 
even in the case of the groups   $G = G([0,b];A,P)$ and $\bar{G} = G([0,\bar{b}]; \bar{A}, \bar{P})$ 
(see Theorem \ref{TheoremE07}).
\index{Theorem \ref{TheoremE07}!consequences}%
Secondly, there exist embeddings
\[
\index{Group G([a,c];A,P)@Group $G([a,c];A,P)$!isomorphisms}%
\index{Group G([0,infty[;A,P)@Group $G([0, \infty[\;;A,P)$!isomorphisms}%
G([0,b];A,P) \mono G([0,\infty[\;; A, P) \mono G(\R;A,P),
\]
induced by infinitary PL-homeomorphisms, 
that map the corresponding subgroups of bounded homeomorphisms isomorphically 
onto each other  (see section \ref{ssec:18.2}).

In spite of these facts,
isomorphisms  $\alpha \colon G(I;A,P) \iso G(\bar{I}; \bar{A}, \bar{P})$ 
where $I$ is the line or a half line are still fairly manageable.
Indeed,
most of the conclusions of Theorem \ref{thm:TheoremE10} 
and Supplement \ref{supplement:SupplementE11New} 
continue to be valid:
\begin{thmP}
\label{thm:TheoremE14}
\index{Group G(R;A,P)@Group $G(\R;A,P)$!isomorphisms}%
\index{Group G([0,infty[;A,P)@Group $G([0, \infty[\;;A,P)$!isomorphisms}%
Set $G = G(I;A,P)$, $\bar{G} = G(\bar{I};\bar{A},\bar{P})$
and assume there exists an isomorphism $\alpha \colon G \iso \bar{G}$.
If $I$ and $\bar{I}$ are both the line $\R$ or both the half line $[0, \infty[$
the following assertions hold:
\begin{enumerate}[(i)]
\item $\bar{P} = P$;
\item there exist a non-zero real $s$ and a unique PL-homeomorphism $\varphi \colon \Int{I} \iso \Int(\bar{I})$ 
so that $\bar{A} = s \cdot A$, that $\alpha$ is induced by conjugation by $\varphi$ and $\varphi$ has slopes in the coset $s \cdot P$. 
\item If $I = \R$ then $\varphi$ has only finitely many singularities; 
if $I = [0,\infty]$ its singularities may be infinite in number,
but they can only accumulate in 0.
\end{enumerate}
\end{thmP}

Theorem \ref{thm:TheoremE14} 
allows one to work out the automorphism group of $G(I;A,P)$  
with $I$ a line or half line.
If $I = \R$, the automorphism group is as described by Corollary \ref{crl:CorollaryE13-Preface}: 
it is an extension of $G(\R;A,P)$ by the abelian group $\Aut(A)/P$ 
(see section \ref{ssec:19.2a}).
\index{Automorphism group of G(R;A,P)@Automorphism group of $G(\R;A,P)$!description}%
If $I$ is a half line,
the automorphism group of $G =  G(I;A,P)$ is more complicated on account of the fact 
that the PL-homeomorphisms inducing automorphisms of $G$ may have infinitely many singularities 
that accumulate in 0. 
\index{Automorphism group of G([0,infty[;A,P)@Automorphism group of $G([0, \infty[\;;A,P)$!description}%
All the same, 
the outer automorphism group of $G$ admits of a concrete description; 
details can be found in section \ref{sssec:19.2b}
(see, in particular, Proposition \ref{crl:PropositionE19}).
%
\subsection*{Case 3:  $P$ is cyclic and  $I$ is compact}
\label{ssec:Cyclic-P-bounded-interval}
We are left with the case 
where $P$ is cyclic and $I$ is a compact interval, 
say  $I = [0,b]$ with  $b \in A$.
The methods developed in the memoir \cite{BiSt85} are then not sufficiently powerful
to determine the automorphism group of $G([0,b]; A, P)$;
in particular, 
they give no answer to the question 
whether every auto-homeomorphism $\varphi \colon ]0,b[\, \iso  \,]0, b[$ 
inducing an automorphism of $G([0,b]; A, P)$ is necessarily piecewise linear. 
They permit one, however, to solve some easier problems.
\subsubsection*{Dependence on $b$.}
\label{sssec:P-cyclic-dependence-in-b}
If $P$ is cyclic all groups in the family 
\[
\{G([0,b];A,P) \mid b \in A\, \cap \; ]0, \infty[ \, \}
\]
are isomorphic to each other (see Theorem \ref{TheoremE07}; 
\index{Group G([a,c];A,P)@Group $G([a,c];A,P)$!isomorphisms}%
\index{Theorem \ref{TheoremE07}!consequences}%
the proof relies on the construction of infinitary PL-homeomorphisms).
This finding is in contrast with the situation holding for groups with non-cyclic $P$ 
(see Proposition \ref{preface:CorollaryE12}).

Replacing, if need be, the given submodule $A$ by a multiple $s \cdot A$ with $s \in \R_{>0}$,
one can thus reduce to the case where $1 \in A$; 
in principle,
it suffices therefore to consider the unit interval $[0,1]$.
%
\subsubsection*{Isomorphisms between subgroups containing $B$.}
\label{sssec:Isomorphisms-between-groups-containing-B}
%
A surprising consequence of Supplement \ref{supplement:SupplementE11New} is this:
an injective endomorphism $\mu$ of the group $G(I; A, P)$ is bijective 
whenever its image contains the derived group of $B = B(I; A, P)$ 
and $P$ is \emph{not} cyclic.
The analogous statement for cyclic $P$ is false.

Indeed,
for every integer $m \geq 1$ 
there exists an infinitary PL-homeomorphism 
$\varphi_m \colon ]0, b[\, \iso \,]0, b[$
which induces a  monomorphism 
$\mu_m \colon G = G([0,b];A,P) \mono G$  whose image has index $m$; 
the image consists of all PL-homeomorphisms in $G$
whose \emph{leftmost slope} is a power of $p^m$;
here $p$ denotes the generator of $P$ with $p > 1$.
\label{Endomomorphism-mu-m}%
\index{Endomorphism mum@Endomorphism $\mu_m$!properties}%
Similarly,
there exists for every $n \geq 1$ an infinitary PL-homeomorphism $\psi_n \colon ]0, b]\, \iso \,]0, b[$
which induces a  monomorphism 
$\nu_n \colon G = G([0,b];A,P) \mono G$ with image of index $n$ 
that consists of all PL-homeomorphisms in $G$ 
whose \emph{rightmost slope} is a power of $p^n$;
see section 18.5.
\label{Endomorphism nun}
\index{Endomorphism nun@Endomorphism $\nu_n$!properties}%

The compositions $\nu_n \circ \mu_m \colon G \mono G$ have images 
with index $m \cdot n$ that contain the bounded subgroup $B = B([0,b];A,P)$.
They permit one to work out
when two subgroups of $G([0,b];A,P)$ with finite indices and containing $B([0,b];A,P)$ 
are isomorphic; see Theorem \ref{thm:Isomorphism-types-class-II}.
It follows, in particular, 
that a subgroup of finite index and containing $B$ is isomorphic to  $G([0,b];A,P)$  if, and only if,
it is the image of one of the monomorphisms 
$\nu_n \circ \mu_m \colon G \mono G$. 
\begin{remarkP}
\label{remark:Classification-subgroups}
\index{Endomorphism mum@Endomorphism $\mu_m$!applications}%
\index{Endomorphism nun@Endomorphism $\nu_n$!applications}%
\index{Thompson's group F@Thompson's group $F$!subgroups of finite index}%
\index{Group G([a,c];A,P)@Group $G([a,c];A,P)$!classification of subgroups with finite index}%
If  $I = [0,1]$, $A = \Z[1/2]$ and $P = \gp(2)$
the group $G(I;A,P)$ is (isomorphic to) Thompson's group $F$ 
and each of its  subgroups with finite index contains $B$,
for  $B = F'$ is an infinite, minimal normal subgroup.
The results reported in the previous section provide therefore a classification of the subgroups of $F$ with finite index.
\footnote{see section \ref{sssec:Notes-ChapterE-Bleak-Wassink} 
for a comment on this classification}
\end{remarkP}
%
\subsubsection*{Automorphism group of $G([0,b]; A, P)$.}
\label{sssec:Automorphism-group-of-G}
%
I come, finally, to the identification of the automorphisms of $G = G([0,b]; A,P)$.
A basic question here is 
whether an auto-homeomorphism $\varphi$ of the open interval $]0,b[$ is piecewise linear
if it induces an automorphism of $G$. 
As a preliminary step towards an answer,
Bieri and Strebel introduce the subgroup
\begin{equation}
\label{eqP:Aut-PL}
\index{Group Autfr G([0,1];A,P)@Group $\Autfr G([0,1];A,P)$!definition|textbf}%
\Aut_{\PL} G = \{\alpha \in \Aut G \mid \alpha \text{ is induced by a PL-homeomorphism } \varphi \}
\end{equation}
and describe the subgroup $\Aut_{\PL} G/\Inn(G)$ of the outer automorphism group of $G$
(see Proposition \ref{crl:PropositionE20New-part-II}).
This subgroup contains a copy of the group $G$ and is thus far from being abelian,
in contrast to the outer automorphism group of a group with non-cyclic group $P$
considered in Corollary \ref{crl:CorollaryE13-Preface}.
%
\section*{Acknowledgments}
\label{sec:Concluding-remarks}
I started to revise the memoir at the beginning of 2014,
but I have been incited to do so much earlier.
Robert Bieri encouraged me for years to update 
and publish at least part of the memoir,
foremost Chapter E on isomorphisms and automorphisms,
\index{Bieri, R.}%
and Vlad Sergiescu,
during a conference in Les Diablerets in 2013,
urged me anew to bring out the memoir.
\index{Sergiescu, V.}%

It is a pleasure to acknowledge the help 
I have received from many sides.
I am deeply indebted to Matthew Brin.
\index{Brin, M. G.}%
He furthered the revision in various ways,
by his interest in the project,
by sharing with me his thoughts about the history of the subject, 
and by answering my numerous queries.
Other colleagues assisted me as well.
A discussion with Markus Brodmann led to improvements in Illustration
\ref{illustration:4.3}; 
\index{Brodmann, M.}%
Sean Cleary advanced my comprehension of his papers 
\cite{Cle95} and \cite{Cle00};
\index{Cleary, S.}%
Ross Geoghegan provided me with details about the early history of Thompson's group $F$; 
\index{Geoghegan, R.}%
Jakub Gismatullin pointed out a gap in the proof of Proposition
\ref{prp:Higmans-Theorem1} published in \cite{BiSt14},
\index{Gismatullin, J.}%
and John Groves suggested the elementary proof of Lemma 
\ref{lemma:Quotients-A/IPA-and-IPA/IP2A}. 
\index{Groves, J. R. J.}%
Robert Bieri, finally, has been an understanding correspondent of my many mails, 
\index{Bieri, R.}%
at the beginning of 2014 
when I sorted out the plan of the revision and, later on, 
when a claim of \cite{BiSt85} seemed irreparably flawed.

A number of papers, often quite long ones,  
contain results 
that are based on, or related to, findings of the memoir \cite{BiSt85}.
Many of these connections are described  in the \emph{Notes} 
at the end of this monograph. 
In putting them together, 
I was fortunate to receive advice:
I am grateful to Melanie Stein for answering my queries about her paper \cite{Ste92}, \index{Stein, M.}%
to Matt Brin for supplying me with details about the articles \cite{Bri96} 
and \cite{BrGu98},
\index{Brin, M. G.}%
and to Swjat Gal for his explanation of the context of the paper \cite{GaGi16},
written jointly with Jakub Gismatullin.
\index{Gal, \'{S}. R.}%
\index{Gismatullin, J.}%

The original manuscript was typed by Frau Aquilino in 1985.
In the spring of 2014, then,
Marge Pratt, 
secretary at the Department of Mathematical Sciences in Binghamton,
coded the typescript into \LaTeX.
I thank Marge for her excellent job, 
in particular for the rendition of the many commutative diagrams,
and the agreeable collaboration.
The figures were created with \emph{Mathematica}.

To all these and others who have assisted me I express my gratitude.
\bigskip

\hfill Ralph Strebel

%% file: chapt0_Intro.tex
\thispagestyle{plain}
%
\chapter*{Introduction}
\label{chap:Introduction}
\markboth{Introduction}{}
\addtocontents{toc}{\setcounter{tocdepth}{1}}
\addcontentsline{toc}{chapter}{Introduction}
\renewcommand{\thesection}{\arabic{section}}
%
\section{Background}
\label{sec:1}
This paper has its roots in investigations carried out in various areas of mathematics: 
work of R. J. Thompson on logic and his discovery of two infinite simple groups with finite presentations; \index{Thompson, R. J.}%
next research of J. Dydak, P. Freyd, R. Geoghegan, H. M. Hastings and A. Heller 
on unpointed homotopy idempotents
\footnote{see \cite[Section 9.2]{Geo08} for a summary of this research.}
\index{Dydak, J.}%
\index{Freyd, P.}%
\nocite{FrHe79}%
\index{Geoghegan, R.}%
\index{Hastings, H. M.}%
\nocite{HaHe82}%
\index{Heller, A.}%
and the subsequent study of K. S. Brown and R. Geoghegan of infinite dimensional, torsion-free groups with type $\FP_\infty$.
\index{Brown, K. S.}%
Thirdly the result of G. Higman \index{Higman, G.}%
that a group of order preserving, bounded permutations 
of a totally ordered set $\Omega$
has a derived group which is simple and not reduced to the identity,
whenever the group acts 2-fold transitively on $\Omega$,
and a similar result of D. B. A. Epstein 
\index{Epstein, D. B. A.}
about the group of homeomorphisms with bounded support of a connected manifold.
Fourthly, the result, due to S. H. McCleary, 
\index{McCleary, S. H.}%
that under suitable hypotheses each group isomorphism 
$\alpha \colon G \iso \bar{G}$ of two ordered permutation groups 
is induced by conjugation by a homeomorphism $\varphi \colon \Omega \iso \bar{\Omega}$ 
of the underlying totally ordered sets, and similar theorems, 
due to J. V. Whittaker in the case of manifolds and to C. Holland in
the case of lattice-ordered permutation groups. 
\index{Holland, C.}%
\index{Whittaker, J. V.}%
Finally, 
the discovery, made recently by M. G. Brin and C. C. Squier, 
that the group of (finitary) piecewise linear homeomorphisms of the real line has no non-abelian free subgroups
and the introduction of generalized Thompson groups by these authors.
\index{Brin, M. G.}%
\index{Squier, C. C.}%
%
\subsection{Richard Thompson's discovery}
\label{ssec:1.1}
 In 1965, 
 R. J. Thompson discovered  \index{Thompson, R. J.}%
 that the 2-generator 2-relator group
 \begin{equation}
 \label{eq:Finite-presentation-F}
 \index{Thompson's group F@Thompson's group $F$!finite presentation}%
\left\langle 
 x_0, x_1 \mid \act{x_1^{-1} x_0^{-1}} x_1 = \act{x_0^{-2}}x_1,  
 \act{x_1^{-1}x_0^{-2}} x_1 = \act{x_0^{-3}} x_1 
 \right\rangle
 \end{equation}
has some surprising properties. 
Firstly,
 it has an intriguing infinite presentation
\begin{equation}
 \label{eq:Infinite-presentation-F}
  \index{Thompson's group F@Thompson's group $F$!definition|textbf}%
  \index{Thompson's group F@Thompson's group $F$!infinite presentation}%
 \left\langle 
 x_0, x_1, x_2, \ldots  \mid x_i^{-1}x_j x_i= x_{j+1} \text{ for all } i < j \right\rangle
 \end{equation}
 (see \cite{CFP96}, Example 1.1 and Theorem 3.4).
\footnote{We shall stick to the usual convention of analysts for composing functions 
and accordingly use left actions throughout; in particular, $\act{a} b$ is short for $aba^{-1}$.\label{notation:f-circ-g}%
The authors of \cite{CFP96} follow the convention of analysts for composing functions,
but use right action for conjugation.}

Next, 
Thompson observed that each element of \eqref{eq:Infinite-presentation-F}
can be written in a special form, \index{Thompson, R. J.}%
namely
\begin{equation}
\label{eq:1.3}
x_{i_1}x_{i_2} \cdots x_{i_k} \cdot x_{j_\ell}^{-1} \cdots x_{j_2}^{-1} x_{j_1}^{-1}
\end{equation}
with $i_1 \leq i_2 \leq \cdots \leq i_k$ and $j_1 \leq j_2 \leq \cdots \leq j_\ell$ and $i_k \neq j_\ell$.
 \index{Thompson's group F@Thompson's group $F$!properties}

The group, given by the presentation \eqref{eq:Finite-presentation-F}, 
can be realized by PL-homeo\-mor\-phisms, 
\ie{}by piecewise affine homeomorphisms, 
of the unit interval. 
Indeed, sending $x_0$, $x_1$ to the affine interpolations $f_0$, $f_1$ of the assignments depicted below,
gives such a realization (see \cite[Thm.\,3.4]{CFP96}).

\begin{equation}
\label{eq:Generating-elements-f0-f1}
\index{Thompson's group F@Thompson's group $F$!rectangle diagrams}%
\begin{minipage}[c]{8cm}
\psfrag{1}{\hspace*{-1.7mm}  \small  $0$}
\psfrag{2}{  \hspace*{-3.0mm} \small  $\tfrac{1}{2}$}
\psfrag{3}{\hspace*{-1.5mm}  \small  $\tfrac{3}{4}$}
\psfrag{4}{\hspace*{-1.2mm} \small   $1$}
\psfrag{11}{\hspace*{-1.2mm}  \small   $0$}
\psfrag{12}{  \hspace*{-2.0mm} \small   $\tfrac{1}{4}$}
\psfrag{13}{\hspace*{-0.9mm}  \small  $\tfrac{1}{2}$}
\psfrag{14}{\hspace*{-0.8mm}  \small   $1$}
\psfrag{21}{\hspace*{-1mm}  \small   $\tfrac{1}{2}$}
\psfrag{22}{\hspace*{-1.0mm}  \small   $1$}
\psfrag{23}{\hspace*{-1mm}  \small   2}
\psfrag{31}{\hspace*{-0.5mm} \small $0$}
\psfrag{32}{  \hspace*{-1.0mm}\small $\tfrac{1}{2}$}
\psfrag{33}{\hspace*{-1.3mm} \small $\tfrac{3}{4}$}
\psfrag{34}{\hspace*{0.3mm}\small $\tfrac{7}{8}$}
\psfrag{35}{\hspace*{0.3mm}\small  $1$}
\psfrag{41}{\hspace*{-0.2mm} \small  $0$}
\psfrag{42}{  \hspace*{-1.8mm}   \small   $\tfrac{1}{2}$}
\psfrag{43}{\hspace*{-0.6mm} \small   $\tfrac{5}{8}$}
\psfrag{44}{\hspace*{-0.5mm} \small   $\tfrac{3}{4}$}
\psfrag{45}{\hspace*{-0.3mm} \small   $1$}

\psfrag{51}{\hspace*{-1mm} \small   $1$}
\psfrag{52}{\hspace*{-1.0mm} \small $\tfrac{1}{2}$ }
\psfrag{53}{\hspace*{-0.5mm} \small   $1$}
\psfrag{54}{\hspace*{-0.7mm} \small   $2$}

\psfrag{la1}{\hspace{-9mm}$x_0 \longmapsto f_0 =$}
\psfrag{la2}{\hspace{-9mm}$x_1 \longmapsto f_1 =$}
\includegraphics[width = 7cm]{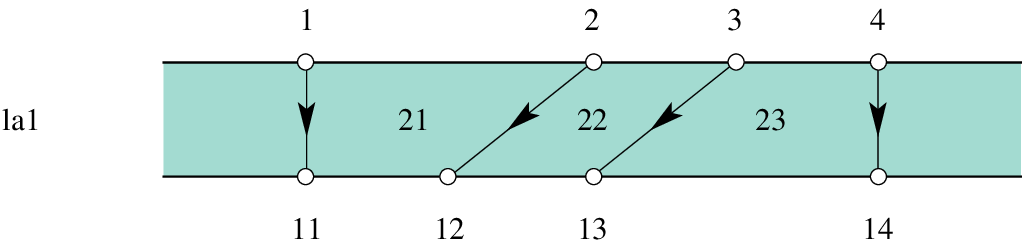}
\\[2mm]
\includegraphics[width= 7cm]{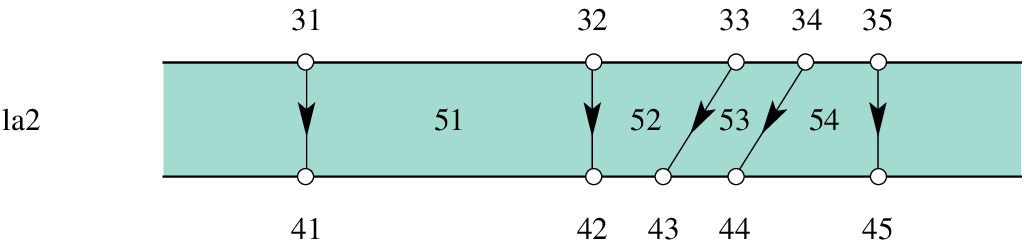} 
\end{minipage}
\end{equation}
Thompson concluded \index{Thompson, R. J.}%
that \eqref{eq:Finite-presentation-F} 
defines a torsion-free, non-abelian group $F$
and that the realization is faithful. 
He proceeded to deduce from the semi-normal form \eqref{eq:1.3} 
that every proper homomorphic image of $F$ is abelian or, to put it differently, 
that the derived group $F'$ of $F$ is the unique minimal normal subgroup $\neq \{\id\}$ of $F$. 
He used this fact to establish
that the group given by the presentation $T = \langle x_0, x_1, c_1 \mid \RR \rangle$  is simple
if $\RR$ is the following set of relations
\[
\index{Thompson's group T@Thompson's group $T$!presentation}%
 \index{Thompson's group T@Thompson's group $T$!definition|textbf} %
 \index{Thompson, R. J.}%
\act{x_1^{-1} x_0^{-1}} x_1 = \act{x_0^{-2}}x_1,  
 \act{x_1^{-1}x_0^{-2}} x_1 = \act{x_0^{-3}} x_1, 
c_1 = x_1c_2,  
c_2x_2 =x_1c_3,
c_1x_0 = c_2^2,
c_1^3 = 1.
\]
Here $x_2$, $c_2$ and $c_3$ are short for 
$\act{x_0^{-1}}x_1$, $x_0^{-1}c_1x_1$ and $x_0^{-2} c_1 x_1^2$, respectively
(see Theorem 5.8 and Corollary 5.9 in \cite{CFP96} ).
To ascertain that $T$ is  a supergroup of $G$, 
Thompson gave a realization of $T$ by PL-homeomorphisms of the circle $\R/\Z$, 
associating to $x_0$, $x_1$ the homeomorphisms obtained from $f_0$, $f_1$ 
by identifying the endpoints of $[0,1]$, 
and to $c_1$ the PL-homeomorphism of $\R/\Z$ induced 
by the infinitary PL-homeomorphism depicted below (see \cite[Ex.\,5.1]{CFP96})
\begin{equation*}
\index{Thompson's group T@Thompson's group $T$!rectangle diagrams}%
\psfrag{61}{\hspace*{-1.2mm}  \small }
\psfrag{62}{  \hspace*{-2.5mm}  \small  }
\psfrag{63}{\hspace*{-3.0mm}   \small   $-\tfrac{1}{2}$}
\psfrag{64}{\hspace*{-3.5mm}  \small   $-\tfrac{1}{4}$}
\psfrag{65}{\hspace*{-0.7mm}   \small  0}
\psfrag{66}{  \hspace*{-2.3mm}  \small  $\tfrac{1}{2}$}
\psfrag{67}{\hspace*{-1.0mm}   \small  $\tfrac{3}{4}$}
\psfrag{68}{\hspace*{-0.7mm}  \small   1}
\psfrag{69}{\hspace*{-0.8mm}  \small   $\tfrac{3}{2}$}
\psfrag{70}{\hspace*{-0.6mm}  \small   $\tfrac{7}{4}$}
\psfrag{71}{\hspace*{-0.7mm}  \small  2}
\psfrag{81}{\hspace*{-3.2mm}  \small $-\tfrac{1}{2}$}
\psfrag{82}{\hspace*{-3.1mm}  \small   $-\tfrac{1}{4}$}
\psfrag{83}{\hspace*{-0.4mm}   \small  0}
\psfrag{84}{  \hspace*{-2.0mm}  \small  $\tfrac{1}{2}$}
\psfrag{85}{\hspace*{-0.8mm}   \small  $\tfrac{3}{4}$}
\psfrag{86}{\hspace*{-0.5mm}  \small   1}
\psfrag{87}{\hspace*{-0.8mm}  \small   $\tfrac{3}{2}$}
\psfrag{88}{\hspace*{-0.6mm}  \small   $\tfrac{7}{4}$}
\psfrag{89}{\hspace*{-0.8mm}  \small  2}
\psfrag{52}{\hspace*{0.8mm} \small  $\tfrac{1}{2}$}
\psfrag{53}{\hspace*{-1mm} \small  $2$}
\psfrag{54}{\hspace*{-0.5mm} \small  $1$}
\psfrag{55}{\hspace*{0.8mm} \small  $\tfrac{1}{2}$}
\psfrag{56}{\hspace*{-1mm} \small  $2$}
\psfrag{57}{\hspace*{-0.5mm} \small  $1$}
\psfrag{58}{\hspace*{0.8mm} \small  $\tfrac{1}{2}$}
\psfrag{59}{\hspace*{-1mm} \small  $2$}
\psfrag{62}{\hspace*{-1mm} \small}
\includegraphics[width= 11cm]{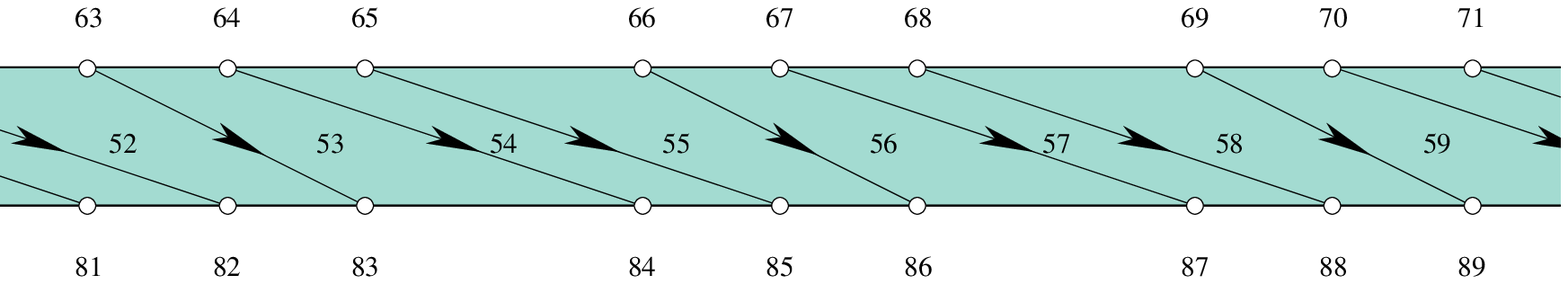}
\end{equation*}
%
\subsubsection{Remarks on the history of Thompson's discovery}
\label{sssec:Remarks-history-discovery}
\index{Thompson, R. J.}%
%
The history of Thompson's discovery is recounted on pages 475 and 476 of \cite{McTh73}.
In translating the definitions from \cite{McTh73} and \cite{Tho74} 
to the present set-up, 
it must be borne in mind that in both papers 
right composition of homeomorphisms is used. 
The group $F$ is called $P$ in both \cite{McTh73} and \cite{Tho74} 
with generators denoted $D$, $E$ in \cite{McTh73} and $D$, $R$ in \cite{Tho74}.
The group $T$ is denoted by $C$ in section  \emph{Part Two} of \cite{Tho74}; 
it is not mentioned in  \cite{McTh73}, 
where a larger group $C'$ is described in its stead. 
The group $C'$ occurs  in Part Four of \cite{Tho74} under the name $V$,
and is Thompson's second finitely presented, simple group;
it is isomorphic to the automorphism group called $G_{2,1}$ 
in \cite[Section 8]{Hig74}. \index{Higman, G.}%
The group $C'$ can be realized by homeomorphisms of the Cantor set.
%
\subsubsection{Rediscovery by homotopy theorists}
\label{sssec:1.2}
The presentation \eqref{eq:Infinite-presentation-F}
has been rediscovered by homotopy theorists, 
notably by J. Dydak, P. Freyd, A. Heller, and H. Hastings, 
 \index{Freyd, P.}%
 \index{Heller, A.}%
 \index{Dydak, J.}%
\index{Hastings, H. M.}%
when studying \emph{unpointed homotopy idempotents}. 
Their interest in the group $F$ was prompted by the fact 
that it admits an (actually injective) endomorphism $\psi$, 
taking $x_i$ to $x_{i+1}$ 
and satisfying $\psi^2(g) = x_0^{-1} \psi(g) x_0 $ for $g \in G$,
and that the triple $(F, \psi, x_0)$ is universal among all triples $(H, \varphi,y)$,
where $\varphi$ is an endomorphism of $H$ 
satisfying $\varphi^2(h)  = \varphi(h)^y$ for $h \in H$
(see, \eg{}\cite[p.\,82 ff]{DySe78},
\cite[pp. 26--30]{FrHe79} or \cite[9.2]{Geo08}). 
In the course of their investigations, 
the topologists rediscovered the semi-normal form \eqref{eq:1.3}
and found a free abelian subgroup of infinite rank; 
this latter fact implies that $F$ is a torsion-free group of infinite cohomological dimension. 
Later, 
K. S. Brown and R. Geoghegan, building on results about unpointed homotopy idempotents, 
were able to show that $\Z$ admits a $\Z{F}$-free resolution $\mathbf{F}_* \epi \Z$ 
with $\mathbf{F}_0 = \Z{G}$  and $\mathbf{F}_1$, $\mathbf{F}_2$, \ldots free $\Z{F}$-modules of rank 2 
(\cite[Section 6]{BrGe84}).
\index{Brown, K. S.}%
\index{Geoghegan, R.}%
This revealed that $F$ is a torsion-free, infinite dimensional group 
of type  $\FP_\infty$, 
the first known example of this kind. 
Moreover, in \cite{BrGe85} the same authors verified
that the cohomology groups $H^j (F, \Z{F})$ are trivial in all dimensions. 
They noticed also
that $F'$ is not merely a minimal normal subgroup of $G$, 
but actually simple.
%
\subsubsection{Local description of the groups of Thompson}
\label{ssec:1.3}
\index{Thompson, R. J.}
A further result of Thompson's lay dormant for quite a while: 
\index{Thompson, R. J.}%
stated in terms of the realization \eqref{eq:Generating-elements-f0-f1} it asserts 
that the image of $F$ in the group of PL-homeomorphisms is as large as could be hoped for: 
it is the group of all PL-homeomorphisms $f$ of $\R$ which satisfy the following conditions:
\begin{align}
&\text{the support of $f$ is contained in $]0,1[$}
\label{eq:Meaning-I},\\
&\text{the slopes of $f$ are powers of 2, and}
\label{eq:Meaning-P} \\
&\text{the singularities of $f $ lie the subring } \Z[1/2] \text{ of } \Q.
\label{eq:Meaning-A}
\end{align}
(This is what the proof given on pp.\;476--477 in \cite{McTh73} 
amounts to in the present setup;
see also the proof of Theorem 1.4 in \cite{Tho80}.) 

The stated fact is crucial 
for the further work of K. S. Brown and R. Geoghegan 
on groups of type $\FP_\infty$. 
\index{Brown, K. S.}%
\index{Geoghegan, R.}%
It is related to the discovery of W. P. Thurston 
\index{Thurston, W. P.}%
that the group $T$ maps under the given realization
isomorphically onto the group of all PL-homeomorphisms of the circle $\R/\Z$ 
satisfying \eqref{eq:Meaning-P}  and \eqref{eq:Meaning-A}. 

%
\subsection{Starting point of our article}
\label{ssec:1.4}
%
The discovery that
the group $F$ can be described 
by properties \eqref{eq:Meaning-I}, \eqref{eq:Meaning-P} and \eqref{eq:Meaning-A} 
is the starting point of the present memoir. 
It suggests to view $F$, and similarly defined groups, 
as ordered permutation groups 
and to bring results about such groups to bear on the case at hand. 
As a sample, take the afore-mentioned fact that $F'$ is simple. 
A theorem of G. Higman's \cite{Hig54a} states 
\index{Higman, G.}%
\index{Higman's simplicity result}%
that every group $B$
which acts 2-fold transitively on an infinite, totally ordered set $\Omega$ 
by order preserving permutations
has a derived group $[B,B]$ that is non-abelian and simple, 
provided it consists of permutations with bounded support. 
With the help of G. Higman's theorem the fact that $F'$ is simple 
can be explained as follows. 
Let $B$ denote the subgroup consisting of those homeomorphisms in $F$ 
which have slope 1 near the endpoints 0 and 1. 
Then the derived group $F'$ equals $B$ and all amounts to prove 
that $B/B'$ is trivial and that $B$ has an orbit 
which is dense in the open interval $]0,1[$ 
and on which $B$ acts 2-fold transitively.

There is a second topic 
which the theory of ordered permutations groups has much to say about, 
namely the existence of isomorphisms and automorphisms. 
The key idea is this: 
if $H$ is a group of homeomorphisms of $\R$ 
and if $\varphi \colon \R \iso \R$ is a homeomorphism, 
conjugation by $\varphi$ induces an isomorphism $\varphi_* \colon H \iso \act{\varphi}H$.
Similarly, 
if $N$ denotes the normalizer of $H$ in the group of all homeomorphisms of $\R$,
conjugation induces a homomorphism  $\vartheta \colon N \to \Aut H $.
Under suitable hypotheses converse statements hold: 
every isomorphism $\alpha \colon  H \iso H_1$ is induced by a unique homeomorphism $\varphi$ 
and $\vartheta$ is an isomorphism of groups. 
Results of this type 
are due to S. H. McCleary for ordered permutation groups \cite{McC78b}; 
\index{McCleary, S. H.}%
previously analogous results have been proved for homeomorphism groups of manifolds by J. V. Whittaker \cite{Whi63} 
\index{Whittaker, J. V.}%
and for lattice-ordered permutation groups by C. Holland \cite{Hol65c}. 
\index{Holland, C.}%
Using ideas of \cite{McC78b} we shall find, in particular, 
that the group of outer automorphisms $\Aut F /\Inn F $ of Thompson's group $F$, 
defined by \eqref{eq:Finite-presentation-F}, 
contains the direct square $T \times T$ of two copies of his simple group $T$.
%
\section{Outline of our investigation}
\label{sec:2}
%
In \cite{BrSq85} various classes of subgroups of the group $\PL_ o(\R)$
of order preserving finitary PL-homeomorphisms of the real line $\R$ 
are considered.
\index{Brin, M. G.}%
\index{Squier, C. C.}%
We shall deal with a similarly defined collection of groups.
%
\subsection{Definition of the group $G(I;A,P)$}
\label{ssec:2.1}
%
The subgroups of $\PL_o(\R)$ that we shall study will be singled out 
by restricting the support, the slopes and the singularities of their elements. 
Specifically, 
let $I$ be a closed interval of the real line, 
$P$ a subgroup of the multiplicative group $\R^\times_{>0}$ of the positive reals, 
and $A$ a subgroup of the additive group of $\R$ 
that is stable under multiplication by the elements of $P$; 
put differently, 
$A$ is required to be a $\Z[P]$-submodule of the additive group of $\R$, 
the action being the canonical one.
Consider now the set of all finitary PL-homeomorphisms $f \colon \R \iso \R$ 
which satisfy the three requirements:
\begin{align}
&f \text{ has support in $I$ and maps $A$ onto itself};
\label{eq:Condition-I}\\
&\text{the slopes of $f$ are elements of $P$,  and}
\label{eq:Condition-P}\\
&\text{the singularities of $f$ lie in $A$.}
\label{eq:Condition-A}
\end{align}
This set is closed under composition of functions, 
contains the identity and the inverses of its elements; 
it thus gives rise to a group 
that we denote by $G(I;A,P)$. 
In order to avoid that $G(I;A,P)$ is reduced to the identity 
we shall always impose the following non-triviality restrictions 
on the triple $(I,A,P)$:
\begin{equation}
\label{eq:Non-triviality-assumption}
\text{the interior of $I$ is non-empty, and $A$ as well as $P$ are non-trivial.}
\end{equation}
Note that requirement \eqref{eq:Non-triviality-assumption} implies 
that $A$, and each of its non-trivial $\Z[P]$-submod\-ules, 
contain arbitrarily small positive numbers
whence they are dense subgroups of the real line $\R$.
\index{Module A@Module $A$!density property}
\smallskip

We shall investigate five aspects of the groups $G = G(I;A,P)$:
\begin{enumerate}[(A)]
\index{Main themes of the memoir}%
\item the $G$-orbits $\Omega$ in $A$ and the multiple transitivity properties of $G$ on its orbits;
\item generating sets of $G$ and the question whether $G$ can be finitely generated;
\item the abelianization of the subgroup $B = B(I;A,P)$ of bounded homeomorphisms 
and the simplicity of its derived group $B'$; 
\item presentations of $G$ and the question whether $G$ admits a finite presentation; 
\item isomorphisms of certain subgroups of groups $G(I_1; A_1,P_1)$ and $G(I_2;A_2,P_2)$, 
and the automorphism groups of $B(I;A,P)$ and of $G(I;A,P)$.
\end{enumerate}

\subsection{The orbits of $G(I;A,P)$ in $A$}
\label{ssec:2.2}
\index{Group G(I;A,P)@Group $G(I;A,P)$!orbits in A cap I@orbits in $A \cap \Int(I)$}%
Fix $I$, $A$, $P$, 
and let $G$ be short for $G(I;A,P)$. 
It is clear from requirement \eqref{eq:Condition-I} 
that the intersection $A \cap  \Int(I)$ is stable under the action of $G$, 
but in general $G$ does not act transitively on this intersection. 
The reason for this failure is easily explained: 
if $a$, $a'$, $c$ and $c'$ are elements of $A$ 
and $f \colon [a,c] \iso  [a',c']$ is a PL-homeomorphism with slopes in $P$ and singularities in $A$, 
then $f$ is the affine interpolation of a sequence of the form
\[
 (b_0,b'_0)= (a,a'), \;
(b_1,b_1'), \ldots,  \;
(b_h, b'_h) =(c,c').
\]
So $c' - c = \sum_i p_i \cdot (b_i-b_{i-1})$ with  each $p_i \in P$,
while $a' - a = \sum_i  b_i-b_{i-1}$.
Consequently, 
the difference $(c'-c) - (a'-a)$ is in the subgroup $IP \cdot  A$ of $A $
generated by the elements of the form $(p-1)\cdot a$ with $p \in P$ and $a \in A$.

Actually, $IP \cdot A$ is a non-trivial $\Z{P}$-submodule of $A$, 
hence dense in $\R$; 
\index{Submodule IPA@Submodule $IP \cdot A$!density property}%
it is familiar from the Homology Theory of Groups.
\index{Homology Theory of Groups!submodule IPA@submodule $IP \cdot A$}%

A first basic result, 
Theorem \ref{TheoremA}, 
claims that this condition on $(c'-c) - (a'-a)$ is also sufficient 
for the existence of a PL-homeomorphism $f \colon [a,c] \iso [a',c']$ 
with properties as stated in the above. 
It follows that each $G$-orbit in $A \cap \Int(I)$ 
is the intersections of a coset $IP \cdot A+a$ with $\Int(I)$, 
provided the elements of $G$ have a common fixed point, 
\ie{}provided $I \neq \R$. 
\index{Submodule IPA@Submodule $IP \cdot A$!significance}%
Next the stabilizer of a point $a \in \Int(I)$ 
is the product $G(I_\ell;A,P) \cap G(I_r;A,P)$ 
where $I_\ell = I \cap (-\infty, a]$
and $I_r = I \cap [a, +\infty[$.
The above description of the $G$-orbits allows one to deduce 
that $G$ acts $\ell$-fold transitively on the $G$-orbits contained in $A \cap \Int(I)$ 
for  every $\ell \geq 1$, 
provided $I \neq \R$  or $A = IP \cdot A$
(see Corollary \ref{CorollaryA4}).
\index{Group G(I;A,P)@Group $G(I;A,P)$!multiple transitivity}%

Now to the orbits of $G(I;A,P)$ in $\Int(I)$.
If $I = \R$
they are the orbits of the affine group
\begin{equation}
\label{eq:Definition-Affine-group}
\index{Affine group Aff(A,P)@Affine group $\Aff(A,P)$!definition|textbf}%
\Aff (A,P) = (t \mapsto  p\cdot t + a \mid a \in A \text{ and  } p \in P);
\end{equation}
otherwise,
they are the intersections of $\Int(I)$ 
with the orbits of $\Aff(IP\cdot A,P)$; 
see Corollary \ref{CorollaryA1*}.
\index{Affine group Aff(A,P)@Affine group $\Aff(A,P)$!significance}%
%
\subsection{The subgroup of bounded homeomorphisms $B(I;A,P)$}
\label{ssec:2.3}
%
The space of orbits of $G(I;A,P)$ on $A \cap \Int(I)$ plays an unexpected rôle 
in the abelianization of the subgroup of bounded homeomorphisms $B(I;A,P)$. 
This rôle is explained by a homomorphism called $\nu$; 
as a preamble to its definition, 
we briefly discuss the homomorphisms $\lambda$ and $\rho$.
Both map the group $G = G(I;A,P)$ into $\Aff(A,P)$,
the affine group with slopes in $P$ and displacements in $A$. 

For the definition of $\rho$,
two cases will be distinguished 
depending on whether $\sup I$ is a real $c$ or $+\infty$. 
In the first case, 
let $\sigma_+ \colon G \to \R$ be the homomorphism 
that associates to $f$ the left derivative $f' (c_-)$ of $f$ in $c$
\index{Homomorphism!18-sigma-plus@$\sigma_+$}%
\label{notation:sigma-plus}%
and define $\rho(f)$ to be the linear map $t \mapsto \sigma_+(f) \cdot t$;
if $\sup I = +\infty$, 
define $\rho(f)$ to be the affine transformation 
that coincides with $f$ near $+\infty$. 
\label{notation:rho}%
\index{Homomorphism!17-rho@$\rho$}%
The homomorphisms $\sigma_-$  and $\lambda$ are defined analogously 
with regard to the left endpoint of the interval $I$ 
or with regard to $-\infty$. 
\label{notation:sigma-minus}%
\label{notation:lambda}%
\index{Homomorphism!11-lambda@$\lambda$}%
\index{Homomorphism!18-sigma-minus@$\sigma_-$}%

We turn now to the  subgroup  $B(I;A,P)$ of the ``bounded'' homeomorphisms  
in $G(I;A,P)$.
It is, by definition, the kernel of the homomorphism
\index{Subgroup B(I;A,P)@Subgroup $B(I;A,P)$!definition|textbf}%
\begin{equation}
\label{eq:2.5}
(\lambda, \rho) \colon  G = G(I;A,P) \to \Aff(A,P) \times \Aff(A,P).
\end{equation}

The definitions of $\sigma_-$ and $\sigma_+$ can be generalized as follows:
if $\Omega$ is a $G$-orbit in $A \cap \Int(I)$, 
define a function $\nu_\Omega \colon  G \to P$ by the formula
\begin{equation}
\label{eq:2.6}
\nu_\Omega(f) = \prod\nolimits_{a \in \Omega} f'(a_+)/f'(a_-),
\end{equation}
where $f'(a_+)$ and $f'(a_-)$ denote the right-hand and the left-hand derivatives of $f$ in $a$.
As each $f$ in $G$ has only finitely many singularities, 
the product on the right hand side of equation \eqref{eq:2.6} makes sense; 
the chain rule then shows that each $\nu_\Omega$ is a homomorphism. 

One can combine the various homomorphisms $\nu_\Omega$ 
into a single homomorphism $\nu$, 
defined by
\begin{equation}
\label{eq:2.7}
\index{Homomorphism!13-nu@$\nu$}
\nu \colon  G = G(I;A,P) \to \Z[(A \cap I)_\sim] \otimes P, 
\qquad f \mapsto  \sum\nolimits_\Omega \Omega \otimes \nu_\Omega(f).
\end{equation}
The existence of $\nu$ indicates
that the group $B(I;A,P)$ may not be perfect. 
The more detailed analysis of the restriction of $\nu$ to $B(I;A,P)$
is facilitated by two circumstances: 
the isomorphism type of $B(I;A,P)$ does not depend on $I$, 
it depends only on $A$ and $P$ (see Proposition \ref{PropositionC1}), 
\index{Subgroup B(I;A,P)@Subgroup $B(I;A,P)$!properties}%
and the groups
\[
G = G([0,\infty[\,;A,P) \quad\text{and}\quad G_1 = \ker(\sigma_- \colon  G \to P)
\]
admit very explicit infinite presentations. 
These presentations allow one to deduce 
that the abelianization
\[
\nu_{\ab} \colon (G_1)_{\ab} \longrightarrow \Z\left[\left(A\, \cap \,]0, \infty[ \right)_\sim \right] \otimes P
\]
of $\nu$ is an isomorphism.
Since $G_1$ acts trivially on $B_{\ab} = B([0,\infty[\, ;A,P)_{\ab}$
(by Lemma \ref{LemmaC3New}) 
the extension of groups $B \triangleleft G_1 \stackrel{\rho}{\epi} G_1/B$
gives rise to the exact sequence of homology groups
\[
\index{Homology Theory of Groups!5-term exact sequence}%
H_2(G_1) \xrightarrow {H_2(\rho\restriction{G_1})}
H_2(G_1/B) \xrightarrow {d_2}
 B_{ab} \xrightarrow{}
(G_1)_{\ab} \xrightarrow {(\rho\restriction{G_1})_{\ab}}
(G_1/B)_{\ab} \to 1.
\]
A closer look at $\ker(\rho \restriction{G_1})_{\ab}$ and $\coker H_2(\rho\restriction{G_1} )$ 
leads to
\begin{thmI}[see Theorem \ref{TheoremC10}]
\label{thm:C10}
\index{Submodule IPA@Submodule $IP \cdot A$!significance}%
If $B(I;A,P)$ is perfect then $A = IP \cdot A$. 
Conversely, $B(I;A,P)$ is perfect
if $A = IP \cdot A$ and $H_1(P,A) = 0$ and if, in addition, 
$A$ has rank 1  or $P$ contains a rational number $p = n/d > 1$ 
so that $A$ is divisible by $n^2 - d^2$.
\end{thmI}
Note that by Higman's result (mentioned in section \ref{ssec:1.4})
and the multiple transitivity of $B(I;A,P)$
pointed out in section \ref{ssec:2.2} imply
that the derived group of $B(I;A,P)$ is a torsion-free, non-abelian \emph{simple} group.
\index{Simplicity result for [B,B]@Simplicity result for $[B,B]$}%
\index{Subgroup B(I;A,P)@Subgroup $B(I;A,P)$!simplicity of [B,B]@simplicity of $[B,B]$}%

\subsection{Types of intervals}
\label{ssec:2.4}
 Two groups $G(I_1;A,P)$ and $G(I_2;A,P)$ with supports in distinct intervals
 may be isomorphic.
 The precise classification into isomorphism types requires a good deal of work 
 that will be carried out in Chapter \ref{chap:E}. 
But by lumping together intervals 
 which produce isomorphic groups for rather obvious reasons
 one obtains already a fairly good overview of the situation.
The present section furnishes such an overview. 
 
 Six types of intervals arise.
The first type is made up of the single interval
\begin{equation}
\label{eq:Type-1}
I = \R. \tag{1}
\end{equation}
Next, 
consider an unbounded, interval with one endpoint $a$.
If $a$ is in $A$, 
the interval belongs to the type of
\begin{equation}
\label{eq:Type-2}
I =[0, \infty[\tag{2}
\end{equation}
All the groups with an unbounded interval but one endpoint in $\R \smallsetminus A$
are isomorphic to each other. 
As they are also isomorphic to the kernel of the homomorphism 
$\sigma_- \colon G([0,\infty[\, ;A,P) \to P$ 
we introduce the open interval
\begin{equation}
\label{eq:Type-3}
I = \;]0,\infty[ \text{ with associated group }  \ker\left (\sigma_- \colon G([0,\infty[\, ;A,P) \to P\right).
\tag{3}
\end{equation}
For bounded intervals we distinguish three cases, 
according as to whether both endpoints are in $A$, 
a single one of them is in $A$ or none of them has this property. 
The first of these types gives rise at this preliminary stage to the class of intervals
\begin{equation}
\label{eq:Type-4}
I = [0,b] \text{ with $b$ a positive element of }  A .
\tag{4}
\end{equation}
The groups with an interval of the second kind are all isomorphic to each other. 
Since they are also isomorphic to the kernel of $\sigma_- \colon  G([0, b_0];A,P) \to P$
where $b_0$ is a fixed positive element of $A$, we introduce the half open interval
\begin{equation}
\label{eq:Type-5}
I =\, ]0,b_0]  \text{ with associated group } \ker \left(\sigma_- \colon  G([0, b_0];A,P) \to P\right).
 \tag{5}
\end{equation}
The group with an interval of the third kind coincide with the corresponding group of bounded homeomorphism, 
and are isomorphic to each other. 
They lead to
\begin{equation}
\label{eq:Type-6}
 I =\; ]0,b_0 [  \text{ with associated group } B([0,b_0];A,P)
 \tag{6}
\end{equation}
%
\subsection{Generating sets and presentations}
\label{ssec:2.5}
As pointed out in Section 2 of \cite{BrSq85} 
groups corresponding to the intervals $\R$ or $[0,\infty[$ 
can be generated by very explicit and natural subsets. 
For the intervals $[0,\infty[$ and $]0,\infty[$ 
one can take the homeomorphisms $g(a,p)$ with a single singularity defined by
\[
g(a,p) (t) = \begin{cases} t &\text{ if } t \leq a,\\ p(t-a) + a &\text{ if  }t > a.
\end{cases}
\]
Here $a \in A$ varies over $I$ and $p \in P \smallsetminus \{1\}$. 
For the interval $I = \R$, 
one has to add the elements with no singularity,
\ie{}the elements of the affine group $\Aff(A,P)$. 
To verify that these homeomorphisms generate the group in question,
one proceeds by a straightforward induction on the number of singularities. 
\index{Group G([0,infty[;A,P)@Group $G([0, \infty[\;;A,P)$!generating sets}%
\index{Group G(R;A,P)@Group $G(\R;A,P)$!generating sets}%

The mentioned groups admit also very explicit and pleasant presentations 
in terms of the given generators. 
For the interval $[0,\infty[$ or $]0,\infty[$ the relations
\begin{align*}
g(a,p) \circ g(a,p') &= g(a,p\cdot p'),\\
\act{g(a,p)}g(a',p') &= g(a+p(a'-a),p') \text{ for } a < a'
\end{align*}
suffice to define the group. 
\index{Group G([0,infty[;A,P)@Group $G([0, \infty[\;;A,P)$!infinite presentation}%

If $I = \R$,
one has to add the relations of the multiplication table presentation of  $\Aff(A,P)$, 
and the relations specifying the action of the affine transformations
$t \mapsto  pt+a$,
 on the set $\{g(a',p') \mid a' \in A \text{ and } p \in  P\smallsetminus \{1\}\; \}$. 
 The verification that the specified relations define the group in question 
 can be carried out by a normal form argument (see \cite[Section 2]{BrSq85}).
 \index{Group G(R;A,P)@Group $G(\R;A,P)$!infinite presentations}%
 
No generating sets of such a simple kind are known for the groups corresponding 
to the three types of bounded intervals. 
As mentioned in section \ref{ssec:1.1}, 
R. J. Thompson found generators for the group $F = G([0,1];\Z[1/2],\gp(2))$. 
In Section \ref{sec:9} we shall exhibit generating systems for the groups of the form
\[
G([0,b]; \Z[\gp(\PP)] , \gp(\PP))
\] 
where $\PP$ is a set of integers $> 1 $. 
Our proof will be in three steps: 
first we reduce the problem to the interval $[0,1]$ (see Theorem \ref{TheoremB7}). 
Next, we introduce the concepts of $\PP$-\emph{regular} 
and $\PP$-\emph{standard subdivisions of} $[0,1]$ 
\index{PP-regular subdivision@$\PP$-regular subdivision!definition|textbf}%
and establish Proposition \ref{prp:B8-intro} given below. 
The set of $\PP$-\emph{regular subdivisions} of $[0,1]$ is defined inductively:
a $\PP$-regular subdivision of level 1 is an equidistant subdivision of $[0,1]$ into p intervals 
where $p$ is an integer in $\PP$. 
If $D_\ell$ is a $\PP$-regular subdivision of level $\ell \geq 1$, 
every subdivision obtained from it 
by subdividing one of the intervals of $D_\ell$ into $p$ equidistant subintervals, 
with $p \in P$, is a $\PP$-regular subdivision of level $\ell + 1$ . 
A $\PP$-\emph{standard subdivision}
is a $\PP$-regular subdivision in whose formation process 
it is always the \emph{left most interval} that is subdivided. 
A crucial result, 
the basic idea of which goes back to Thompson \cite[pp.\;477--478]{McTh73}, 
is then
\begin{prpI}[see Proposition \ref{PropositionB8}]
\label{prp:B8-intro}
\index{PP-regular subdivision@$\PP$-regular subdivision!significance}%
 If $f$ is an element of the group 
 \begin{equation}
 \label{eq:Definition-G[PP]}
 G[\PP] = G([ 0,1] ;\Z[\gp(\PP) ], \gp(\PP))
 \end{equation} 
 there exist $\PP$-regular subdivisions $D$ and $D'$ with the same number of points 
 such that $f \restriction{[0,1]}$ is the affine interpolation of the list of points 
 with coordinates in  $D$ and  in $D'$, respectively. 
 Conversely, every couple of $\PP$-regular subdivisions with the same number of points 
 gives rise by affine interpolation to a function $f \colon [0,1] \to [0,1]$ 
 which extends uniquely to an element of $G[\PP]$.
 \end{prpI}

In the final step, 
we show that affine interpolations given by special $\PP$-regular 
and $\PP$-standard subdivisions, or by special $\PP$-standard subdivisions, 
suffice to generate $G[\PP]$ (see Theorem \ref{TheoremB9} for more details).
\index{PP-regular subdivision@$\PP$-regular subdivision}%
\index{PP-standard subdivision@$\PP$-standard subdivision}%
\subsubsection{Finite generation}
\label{ssec:2.6}
Our main result concerning finite generation is the following combination of Proposition \ref{PropositionB1} and Theorems \ref{TheoremB2} and \ref{TheoremB4}:
\begin{thmI}
\label{TheoremB}
\index{Group G(R;A,P)@Group $G(\R;A,P)$!finite generation}%
\index{Group G([0,infty[;A,P)@Group $G([0, \infty[\;;A,P)$!finite generation}%
\index{Homology Theory of Groups!H0(P,A)@$H_0(P,A)$}%
If $G(I;A,P)$ is finitely generated then
\begin{enumerate}[(i)]
\item $I$ equals $\R$ or it is of type $[0,\infty[$ or of type $[0,b]$ with  $b \in A$,
\item $A$ is a finitely generated $\Z[P]$-module, and
\item $P$ is a finitely generated group,
\item either $I = \R$ or $A/(IP \cdot A)$ is finite.
\end{enumerate}
Conversely, 
if conditions (ii), (iii) and (iv) are satisfied and if $I$ equals $\R$ or is of type $[0,\infty [$,
 the group $G(I;A,P)$ can be generated by finitely many elements.
\end{thmI}

If $I$ is an interval of type $[0,b]$ 
Theorem \ref{TheoremB} gives only necessary conditions for finite generation; 
the only sufficient condition we have been able to detect is described in
\begin{crlI}[see Corollary \ref{CorollaryB10}]
\label{crl:B10-intro}
\index{Group G([a,c];A,P)@Group $G([a,c];A,P)$!finite generation}%
Given a finite set $\PP$ of integers $\geq 2$,
set $P = \gp(\PP) $ and $A = \Z[P]$.
Then the group  $G([0,b];  A, P)$ is finitely generated for every $b \in A_{>0}$.
\end{crlI}

%
\subsubsection{Finite presentations}
\label{ssec:2.7}
If it comes to the question which of the groups $G(I;A,P)$ admit a finite presentation our results are meagre. 
We can prove that if $I$ is either $\R$ or is of type $[0,\infty[$
and if $G(I;A,P)$ is finitely presented, then
the metabelian group $\Aff(A,P) = A \rtimes P$ is finitely presented and $A/(IP \cdot A)$ is finite. 
As regards examples of finitely presented groups with $I$ being of one of these two types, 
we can show that for a finite set $\PP$ of integers $\geq 2$, 
the group $G(I;\Z[\gp(\PP)], \gp(\PP))$ has a finite presentation 
(see Propositions \ref{PropositionD4} and Theorem \ref{TheoremD6}). 
\footnote{If $I = \R$ one can do slightly better; 
see section \ref{sssec:13.3cNew}.}

If $I$ has type $[0,b]$ 
the only finitely presented examples we know of are given by
\begin{prpI} [Proposition \ref{PropositionD10}]
\label{PropositionD10-intro}
If $p\geq 2$ is an integer the group 
\begin{equation}
\label{eq:definition-G[p]}
\index{Group G[p]@Group $G[p]$!definition|textbf}%
\index{Group G[p]@Group $G[p]$!finite presentation}%
G[p] = G([0,1];\Z[1/p],\gp(p))
\end{equation}
is generated by elements $x = x_0$, $x_1$, \ldots , $x_{p-1}$
and defined in terms of these generators by the $p(p-1)$ relations
\begin{align*}
\act{x_i }x_j &= \act{x}x_j \text{ for } 1 \leq i < j \leq p-1,\\
\act{x_i\cdot x}x_j &= \act{x^2}x_j \text{ for } 1 \leq j < i+1 \leq p, \text{ and }\\
\act{x_{p-1}x^2}x_1 &= \act{x^3} x_1.
\end{align*}
\end{prpI}
This result covers slightly more ground than might appear at first sight,
since the group $G([0,b]; A,P)$ does not depend on  $b \in A$ 
if $P$ is cyclic
(see Theorem \ref{TheoremE07}).

Note that $G[p]_{\ab}$ is free abelian of rank $p$. 
This shows that $G[p]$ needs at least $p$ generators. 
Moreover, since $G[p]$ is a split extension $B[p]\rtimes  (P \times P)$ and as G[p] acts trivially on $B[p]_{\ab}$, 
it follows that $B(p]_{\ab}$ is free abelian of rank $p-2$;
so $B[p]$ is only perfect if $p=2$.
\index{Group G[p]@Group $G[p]$!abelianization}%
%
\subsection{Isomorphisms}
\label{ssec:2.8}
If it comes to isomorphisms of two groups $G(I_1;A_1,P_1)$ and $G(I_2;A_2,P_2)$, 
a better picture emerges if one considers not only the groups $G(I_i;A_i,P_i)$,
but all groups 
sandwiched between the simple groups $B(I_i;A_i,P_i)'$ and $G(I_i;A_i,P_i)$. 

For $i \in \{1,2\}$, 
let $(I_i,A_i,P_i)$ be a triple, as considered before, 
and let $G_i$ be a group satisfying  $B(I_i;A_i,P_i)' \leq G_i \leq G(I_i;A_i,P_i)$. 
Making heavy use of ideas of McCleary's paper \cite{McC78b} 
we shall prove in Theorem E4
that every group isomorphism $\alpha \colon G_1 \iso G_2$ is induced by conjugation 
by a unique homeomorphism $\varphi \colon \Int(I_1) \iso \Int(I_2)$.
One of the main questions left open by Theorem \ref{TheoremE04} is then
whether $\varphi$ is a, possibly \emph{infinitary}, PL-homeomorphism. 

Here the algebraic structure of $P$ plays a significant role: 
if $P$ is not cyclic, it is a dense subgroup of the multiplicative group $\R^\times_{> 0}$. 
Using this fact we have been able to determine the nature of $\varphi$:
\begin{thmI}[see Theorem \ref{TheoremE10} and Supplement \ref{SupplementE11New}]
\label{TheoremE10-intro}
\index{Group G(R;A,P)@Group $G(\R;A,P)$!isomorphisms}%
\index{Group G([0,infty[;A,P)@Group $G([0, \infty[\;;A,P)$!isomorphisms}%
\index{Group G([a,c];A,P)@Group $G([a,c];A,P)$!isomorphisms}%
Assume  $G_1$ and $G_2$ are as before, 
and $\alpha \colon G_1 \iso G_2$ is an isomorphism induced by conjugation by  
$\varphi \colon \Int(I_1) \iso \Int(I_2)$. 
If $P_1$ is not cyclic  then 
\begin{enumerate}[(i)]
\item $P_2 = P_1$;
\item there exists a positive real $s$ so that $A_2 = s\cdot A_1$
and $\varphi$ is a PL-homeomorphism with slopes in the coset $ s\cdot P$ of $P$ in $\R^\times$,
singularities in $A_1$ and set of singularities a discrete subset of $\Int (I_1)$;
\item if, in addition, $G_1 = G(I_1;A_1,P_1)$
and  $I_1$ and $I_2$ are both the line $\R$, or both of type $[0,\infty [$ 
or both of type $[0,b_i]$ with $b_i \in A_i$,
then $G_2 = G(I_2;A_2,P_2)$ and  $\varphi$ is a finitary PL-homeomorphism.
\end{enumerate}
\end{thmI}
This theorem permits one to decide when $G([0,b_1];A,P)$ and $G[0,b_2];A,P)$ are isomorphic.
\begin{crlI}[Corollary \ref{CorollaryE12}]
\label{crl:E12-intro}
\index{Subgroup Auto(A)@Subgroup $\Aut_o(A)$!definition|textbf}%
\index{Subgroup Auto(A)@Subgroup $\Aut_o(A)$!significance}%
\index{Submodule IPA@Submodule $IP \cdot A$!significance}%
\index{Group G([a,c];A,P)@Group $G([a,c];A,P)$!dependence on [a,c]@dependence on $[a,c]$}%
\index{Homology Theory of Groups!H0(P,A)@$H_0(P,A)$}%
Assume $P$ is not cyclic and $b_1$, $b_2$ are positive elements of $A$.
Then the groups $G([0,b_1];A,P)$  and $G([0,b_2];A,P)$ are isomorphic if, and only if, 
$b_1 + IP \cdot A$ and $b_2 + IP \cdot A$ lie in the same orbit of the group 
$\Aut_o(A) = \{s \in \R^\times_{>0} \mid s\cdot A = A\}$.
\end{crlI}

This leaves us with the case of a cyclic group $P$.
Here we have
\begin{thmI}[see Theorem \ref{TheoremE14}]
\label{thm:E14-intro} 
Assume $P_1$ is cyclic
and consider an isomorphism $\alpha \colon G(I_1;A_1, P_1) \iso G_2$
onto a subgroup of $G(I_2; A_2, P_2)$
which contains the derived group of $B(I_2; A_2, P_2)$.
Let $\varphi \colon \Int(I_1) \iso \Int(I_2)$ be the unique homeomorphism inducing $\alpha$.
\begin{enumerate}[(i)]
\item If $I_1 = \R$ 
then $I_2= \R$ and $\varphi$ is a finitary PL-homeomorphism 
with slopes in a coset $s \cdot P_1$ and singularities in $A_1$. 
Moreover, $P_2= P_1$, $A_2 = s \cdot A_1$ and $G_2 = G(\R; A_2 ,P_2)$. 
\item If $I_1$ equals $[0, \infty[$ or  $\,]0 ,\infty[$  and if  $I_2 = I_1$,
then $\varphi$ is an increasing  PL-homeomor\-phism 
with slopes in a coset $s \cdot P_1$ and singularities in $A_1$.
Moreover, $P_2= P_1$, $A_2 = s \cdot A_1$ and, for each $\varepsilon > 0$,
the interval $[\varepsilon, \infty[$ contains only finitely many singularities of $\varphi$.  
\end{enumerate}
\end{thmI}

The previous theorem does not mention intervals of type $[0,b]$ with $b \in A$.
We do not know whether the homeomorphism $\varphi$ 
is necessarily piecewise linear in such a situation.
We can, however,  construct infinitary PL-homeomorphisms $\varphi \colon ]0,b_1[\; \iso\; ]0,b_2[$ 
that allow us to prove the surprising
\begin{thmI}[Theorem \ref{TheoremE07}]
\label{thm:E7-intro}
If $P$ is cyclic the groups
$G([0,b];A,P)$ with $b \in A_{>0}$ are isomorphic to each other.
\end{thmI}

Intervals  $I_1$ and $I_2$ of different types may also
lead to isomorphic groups  $G_1$ and $G_2$, 
a fact illustrated by the group of bounded homeomorphisms 
discussed in section \ref{ssec:2.3}:
\begin{prpI}[Proposition \ref{PropositionC1}]
\label{prp:C1-intro}
For every interval $I$ that is not a singleton,
the group $B(I;A,P)$ is isomorphic to $B(\R; A,P)$.
\end{prpI}
%
\subsection{Automorphism groups}
\label{ssec:2.9}
Again two cases arise,
depending on whether $P$ is cyclic or not. 
If $P$ is \emph{not cyclic}, 
the automorphism groups can be determined by Theorem \ref{TheoremE04}. 
Generalizing a result of McCleary's \cite[p.\,526, Corollary 31]{McC78b}
we have:
\begin{crlI}[Corollary \ref{CorollaryE13}]
\label{crl:E13-intro}
\index{Group G(R;A,P)@Group $G(\R;A,P)$!automorphism group}%
\index{Group G([0,infty[;A,P)@Group $G([0, \infty[\;;A,P)$!automorphism group}%
\index{Group G([a,c];A,P)@Group $G([a,c];A,P)$!automorphism group}%
\index{Automorphism group of G(R;A,P)@Automorphism group of $G(\R;A,P)$!description}%
\index{Automorphism group of G([0,infty[;A,P)@Automorphism group of $G([0, \infty[\;;A,P)$!description}%
\index{Automorphism group of G([a,c];A,P)@Automorphism group of $G([a,c];A,P)$!description}%
\index{Group Aut(A)@Group $\Aut(A)$!significance}%
\index{Subgroup Auto(A)@Subgroup $\Aut_o(A)$!significance}%
Assume $P$ is non-cyclic and $I$ is either $\R$, 
or is of type $[0,\infty[$, 
or is of type $[0,b]$, and set $G = G(I;A, P)$.
Then the kernel of the homomorphism
\[
\index{Homomorphism!07-eta@$\eta$}%
\index{Group Aut(A)@Group $\Aut(A)$!significance}%
\index{Subgroup Auto(A)@Subgroup $\Aut_o(A)$!significance}%
\index{Subgroup Qb@Subgroup $Q_b$!significance}%
\eta \colon \Aut(G) \to \Aut(A)/P, \qquad \alpha \mapsto (\text{slopes of } \varphi) \cdot P
\]
is the group of inner automorphism of $G$. 
If $I=\R$, then $\eta$ is surjective; 
if $I$ is of type $[0,\infty[$ the image of $\eta$ is the subgroup $Aut_o(A)/P$ of index two. 
If, thirdly, $I = [0,b]$ the image of $\eta$ can depend on $b$ and is the canonical image of the group
\[
\{s \in \Aut(A) \mid (|s|-1)\cdot b \in IP \cdot A \}.
\]
\end{crlI}

We now move on to the case of a \emph{cyclic} group $P$. 
If $I$ is not bounded there are results similar to the corollary just stated 
(see Theorem \ref{TheoremE14}). 
If, however, $I$ is a compact interval we have not succeeded in proving 
that every automorphism is induced by conjugation by a PL-homeomorphism. 
Our investigation deals therefore only with two preliminary questions.

The first result describes the subgroup $\Out_{\PL}G$ of $\Out G$
whose elements are represented by automorphisms 
that are induced by conjugation by PL-homeo\-mor\-phisms of $]0,b[\,$.
\begin{crlI}[Corollary \ref{crl:PropositionE20New-part-II}]
\label{PropositionE20New-part-II-intro}
\index{Group Out-PLG@Group $\Out_{\PL}G([0,1];A,P)$!description}%
\index{Submodule IPA@Submodule $IP \cdot A$!significance}%
\index{Group T(R/Zpfr;A,P)@Group $T(\R/\Z \pfr;A,P)$!applications}%
Assume $P$ is cyclic, generated by $p > 1$ and set $G = G([0,1];A,P)$.
Then $\Out_{\PL}G $ is an extension of the form
\[
1 \to  L \longrightarrow \Out_{\PL}(G) \longrightarrow \Aut(A)/P  \to 1
\]
where $L$ is a subgroup of index $|A : IP \cdot A|$
 in the square of the generalized Thompson group $T(\R/\Z (p-1);A,P)$.
(The latter group consist of all PL-homeomorphisms of the circle $\R/\Z(p-1)$ with slopes in $P$ and vertices in $\R/\Z (p-1)$.)
\end{crlI}

The second result, Proposition \ref{prp:TheoremE21},
states that the full automorphism group $\Aut G$ is generated 
by the subgroup $\Aut_{PL} G$, 
consisting of all automorphisms induced by PL-homeomorphisms, 
and a subgroup $\Aut_{\per} G$ 
that seems more amenable to further study than the full automorphism group.
%
\subsubsection*{Acknowledgements}
\label{ssec:Acknowledgements}
%
We have been launched onto this research by the enthusiastic talk given by Ross Geoghegan 
at the \emph{Oberwolfach Meeting on Group Theory} in 1984, 
and by a long series of discussions with him, 
in the course of which he generously explained to us many
of the results, by then unpublished, of his and Ken Brown's, and
gave details of the work of Matthew Brin and Craig Squier, 
which, at that time, had not yet reached its final form given in \cite{BrSq85}. 
We thank Ross for putting into our hands this rich body of results, observations
and speculations when we got ready for our investigation.
\index{Brin, M. G.}%
\index{Squier, C. C.}%

The second author was fortunate in finding in Bernold Fiedler (Heidelberg)
\footnote{now professor 
at the Fachbereich für Mathematik und Informatik der Freien Universität Berlin}
\index{Fiedler, B.}%
a patient listener and reader 
when he first struggled with the problem explained in Section 20
\footnote{This section is not  included in this monograph; 
see section \ref{ssec:Changes-E-Section20} for an explanation}. 
Last, not least, he is grateful to the
\emph{Deutsche Forschungsgemeinschaft} for financial support.

\begin{center}
$\star$ \qquad $\star$  \qquad  $\star$ \\[5mm]

{\small Manuscript completed on 5th of September 1985}
\end{center}

%% file: chaptA_Construction.tex
\setcounter{chapter}{0}
%
\chapter{Construction of Finitary PL-homeomorphisms}
\label{chap:A}
%
%
\setcounter{section}{2}
%
\section{Preliminaries}
\label{sec:3}
%
The aim of this section is to fix notation and terminology to be used in the sequel 
and to collect a few simple results.
\subsection{Piecewise affine homeomorphisms}
\label{ssec:3.1}
%
Let J be a non-empty, \emph{open} interval of $\R$. 
A real valued function $f \colon  J \to \R$ will be called \emph{piecewise affine}, 
PL for short, 
if $f$ is continuous and if there exists a discrete subset $S \subset J$ 
so that $f$ is differentiable on $J \smallsetminus S$  
and its derivative $f'$ is constant on each component of $J \smallsetminus S$.
\index{PL-homeomorphism!definition|textbf}%
If $f$ is PL, 
a number $s \in J$ where $f$ is not differentiable 
will be called a \emph{singularity} or a \emph{break} of $f$ 
\index{PL-homeomorphism!singularities}%
\index{PL-homeomorphism!breaks}%
\index{Break|see{PL-homeomorphism!breaks}}%
and the point $(s,f(s))$ in the graph of $f$ will be called a \emph{vertex} of $f$.
\index{PL-homeomorphism!vertices}%
The value of $f'$ on a component $]b,b'[$ of $J\smallsetminus S$ 
will be referred to as the \emph{slope} of $f$ on $]b,b'[$.
\index{PL-homeomorphism!slopes}%

If $f \colon \R \iso \R$ is a homeomorphism its \emph{support}  
is the subset $\{t \in \R\mid f(t) \neq t\}$; 
\index{Support of a homeomorphism}%
\index{PL-homeomorphism!support}%
\label{notation:supp}%
it is an open subset of $\R$ and will be denoted by $\supp f$.
A PL-function is allowed to have infinitely many singularities; 
if it has only finitely many of them it will be referred to as \emph{finitary}.
\index{Finitary PL-homeomorphism!definition|textbf}%
%
\subsection{The homomorphisms $\lambda$ and $\rho$}
\label{ssec:3.2}
Let $A$ be a subgroup of $\R$ and $P$ a subgroup of $\R^\times_{>0}$,
as described in section \ref{ssec:2.1}. 

The group $G(\R;A,P$) contains the subgroup $\Aff(A,P)$ 
consisting of all affine homeomorphisms,
\ie the PL-homeomorphisms of the form $t \mapsto p\cdot t + a$. 
By requirement \eqref{eq:Condition-I} in section \ref{ssec:2.1}, 
the translation part $a$ of an element in $\Aff(A,P)$ must be in $A $; 
it follows that the assignment $(a, p) \longmapsto (t \mapsto p\cdot t + a)$
defines an isomorphism of $A \rtimes P$ onto $\Aff(A,P)$.
\index{Affine group Aff(A,P)@Affine group $\Aff(A,P)$!significance}%

Consider now a closed interval $I$ that satisfies the non-triviality assumption \eqref{eq:Non-triviality-assumption}. 
The group $G(I;A,P)$ admits two geometrically defined homomorphisms 
$\lambda$ and $\rho$  into $\Aff(A,P)$. 
We begin with the definition of $\lambda$; 
it relies on two auxiliary functions $\sigma_-$ and $\tau_-$.
The function $\sigma_- \colon  G(I;A,P) \to P$ records the slope of a PL-homeomorphism 
at the left end of its domain of definition; it is given by the formula
\[
\index{Homomorphism!18-sigma-minus@$\sigma_-$}%
\sigma_-(f) = \lim\nolimits_{t \searrow \inf(I)} f'(t).
\]
The function  $\tau_-$ takes values in $A$ and is defined by
\[
\label{notation:tau-minus}%
\index{Homomorphism!19-tau-minus@$\tau_-$}%
\tau_-(f) = \begin{cases}
\lim\nolimits_{t \searrow  -\infty} (f(t) - \sigma_-(f) \cdot t) &\text{ if } \inf I = -\infty,\\
0 &\text{ if } I \text{ is bounded from below}.
\end{cases}
\]

The function $\lambda$ is then defined by
\begin{equation}
\label{eq:Definition-lambda}
\index{Homomorphism!11-lambda@$\lambda$}
\lambda \colon G(I;A,P) \longrightarrow \Aff(A, P), 
\qquad f \longmapsto 
(t \mapsto \sigma_-(f) \cdot t+ \tau_-(f))
\end{equation}

The functions $\sigma_-$ and $\lambda$ are actually homomorphisms, 
for $\sigma_-$ this can be checked by a straightforward verification;
for the other function it follows from the observation
that $\lambda$ assigns to $f$ the unique affine homeomorphism
that coincides with $f$ near the left end point of $I$.

The homomorphism $\rho \colon  G(I;A,P) \to \Aff(A,P)$ is defined similarly;
it uses the auxiliary functions $\sigma_+$ and $\tau_+$. 
The first of them is given by
\[
\sigma_+(f) = \lim\nolimits_{t \nearrow \sup(I)} f'(t),
\index{Homomorphism!18-sigma-plus@$\sigma_+$}%
\]
the second by 
\[
\label{notation:tau-plus}%
\index{Homomorphism!19-tau-plus@$\tau_+$}%
\tau_+(f) = \begin{cases}
\lim\nolimits_{t \nearrow  _\infty} (f(t) - \sigma_+(f)\cdot t) &\text{ if } \sup I = \infty,\\
0 &\text{ if } I \text{ is bounded from above}.
\end{cases}
\]

The function $\rho$, finally, is defined by
\begin{equation}
\label{eq:Definition-rho}
\rho \colon G(I;A,P) \longrightarrow \Aff(A, P), \qquad f \longmapsto 
(t \mapsto \sigma_+(f) \cdot t+ \tau_+(f)).
\end{equation}
One sees as before that $\sigma_+$ and $\rho$ are homomorphisms.

\index{Homomorphism!17-rho@$\rho$}

\begin{remarks}
\label{remarks:Definitions-lambda-and-rho}
(i) In spite of the fact that the definitions of $\lambda$ and $\rho$ are straightforward, 
these homomorphisms are very useful. 
They show, in particular,
that each finitely generated subgroup $L \neq \{\id\}$ of $G(\R;\R,\R^\times_{>})$
maps homomorphically onto $\Z$.
Indeed, 
let $I$ be the smallest closed interval 
that contains the supports of the elements of $L$.
The homomorphism
\[
L \incl G(\R;\R,\R^\times_{>0}) \xrightarrow{\lambda} \Aff(R,\R^\times_{>0}).
\]
maps $L$ onto a non-trivial finitely generated subgroup of the locally indicable group 
$\R \rtimes \R^\times_{>0}$.
The group  $G(\R;\R,\R^\times_{>0}$  is thus \emph{locally indicable}; 
for a discussion of some consequences of this fact, 
see, \eg{}\cite[Section 13.1 and p.\,638, Ex.\,13.9 (ii)]{Pas77}.
\index{Local indicability!examples}%
\index{Local indicability!consequences}%
\index{Passman, D. S.}%

(ii) The definitions of $\sigma_-$ and $\sigma_+$ will be generalized in section \ref{ssec:11.1}.
\end{remarks}
%
\subsection{Geometrically induced isomorphisms}
\label{ssec:3.3}
%
Let $I_1$ and $I_2$ be two closed intervals of $\R$ 
and let  $\homeo (I_1 ,I_2)$ denote the set of all homeomorphisms $f$ of $\R$ with  $f(I_1) = I_2$; 
this set may be empty. 
If $\varphi_1$ and $\varphi_2$ are homeomorphisms of $\R$
they induce a  map
\[
\label{eq:Set-homeo(I1,I2)}
\homeo(\varphi_1, \varphi_2) \colon \homeo(I_1, I_2) 
\longrightarrow \homeo\left(\varphi_1(I_1)), \varphi_2(I_2)\right),
\]
taking $f $ to $ \varphi_2 \circ f \circ \varphi^{-1}_1$.
This map is bijective.

In the sequel,
we shall use this bijection in two different contexts. 
In Section \ref{sec:4},
 we shall investigate when the set $\PLhomeo(I_1,I_2;A,P)$, 
consisting of all finitary PL-homeomorphisms $f$ of $\R $ with $f(I_1) = I_2$, 
vertices in $A^2$ and slopes in $P$, is non-empty. 
The intervals $I_1$, $I_2$ will often be the form 
$[a_i, c_i]$ with $a_i$ and $c_i$ in $A$; 
so the translations $\varphi_1$ and $\varphi_2$,  
given by  $\varphi_i (t) = t-a_i$,  
induce a bijection
\begin{equation}
\label{eq:3.1}
\PLhomeo(I_1,I_2;A,P) \iso  \PLhomeo([0,c_1-a_1],[0,c_2-a_2 ] ;A,P).
\end{equation}

More generally, 
if $\Aut(A)$ denotes the group $\{p \in \R^\times  \mid p \cdot A = A\}$  
and if $\varphi_1$, $\varphi_2$ are in $\Aff(A, \Aut(A))$,
these homeomorphisms induce a bijection
\begin{equation}
\label{eq:3.2}
\PLhomeo(I_1,I_2;A,P) \iso  \PLhomeo (\varphi_1(I_1), \varphi_2(I_2) ;A,P).
\end{equation}
Another application will be given in section \ref{ssec:16.4}; 
see Proposition \ref{PropositionE6}.
%
\section{The basic result}
\label{sec:4}
%
One of the basic results of Chapter \ref{chap:A} is an algebraic characterization 
 of the orbits of the action of $G(I;A,P)$ on  $ A \cap \Int(I)$. 
It will involve the $\Z[A]$-submodule $IP \cdot A$ of $A$,
generated by the products $(p-1)\cdot a$ with $p \in P$ and $a \in A$. 

As a first step towards this characterization
we determine when there exists a PL-homeomorphism 
$f \colon \R \iso \R$ in $G(\R;A,P) $
that maps an interval $[a, c]$ with endpoints in $A$ onto another such interval $[a',c']$. 
The answer is provided by
\begin{theorem}
\label{TheoremA}
\index{Group G(R;A,P)@Group $G(\R;A,P)$!construction of elements}%
\index{Submodule IPA@Submodule $IP \cdot A$!significance}%
\index{Theorem \ref{TheoremA}!statement|textbf}%
Let $a$, $a'$, $c$ and $c'$ be elements of $A$ with $a < c$ and $a' < c'$.
Then there exists an element $f \in G(\R;A,P)$ mapping $[a,c]$ onto $[a', c']$ if, and only if,
$c'-a'$ is congruent to $c-a$ \emph{modulo} $IP \cdot A$.
\end{theorem}

\begin{proof}
Assume first that there exists $f \in G(\R;A,P)$ which maps $[a,c]$ onto $[a',c']$. 
Let $a = b_0$, $b_1$, \ldots, $b_h= c$  be an increasing sequence of elements of $A$ 
that includes all the singularities of $f$ inside $[a,c]$, 
and let $p _1$, \ldots, $p_h$ be the sequence of slopes of $f$ on the sequence of intervals
\[
]b_0, b_1[,\quad ]b_1, b_2[, \; \ldots, \; ]b_{h-1}, b_h[\;.
\]
Then $c'-a' = \sum\nolimits_{1 \leq i \leq h} p_i(b_i-b_{i-1})$.
As $c-a = \sum_{1 \leq i \leq h} b_i - b_{i-1}$
the difference $(c'-a') - (c-a) $ is an element of  $IP \cdot A$, as claimed.

Conversely, 
assume there exist elements $a_1$, \ldots, $a_h$ in $A $
and $p_1$, \ldots, $p_h$ in  $P$ so that
\[
c'-a'  = c-a  + \sum\nolimits_{1 \leq i \leq h} (1 - p_i) \cdot a_i.
\]
Set $b' = c' - a'$ and $b = c-a$.
By the bijection \eqref{eq:3.1} it suffices to find $f$ in $G(\R;A,P)$ 
which maps $[0,b]$ onto $[0,b']$.
If $b' = b$, the identity proves the contention;
so assume henceforth that $b' \neq b$ (and $h \geq 1)$.
There exists a permutation $\pi$ of $\{1, \ldots, h\}$ such that the partial sums
\[
b_j = b + \sum\nolimits_{1 \leq i \leq j} (1-p_{\pi(i)}) \cdot a_{\pi(i)}
\]
are positive for $j \in \{1,\ldots, h\}$.
If $f_1$, \ldots, $f_h$ are homeomorphisms in $G(\R;A,P)$ so 
that$ f_j([0, b_{j-1}]) = [0,b_j]$ for $j \in \{1,\ldots, h\}$, 
the composition $f_h \circ \cdots \circ f_1$
is a PL-homeomorphism mapping $[0,b] $ onto $[0,b']$.
It suffices therefore to prove the claim for $h = 1$. 
Moreover, as
$(1-p_1)a_1 = (1-p^{-1}_1) (-p_1 a_1)$,
we can assume that $p_1 > 1$. 
All taken together, 
it suffices to construct $f \in G(\R;A,P)$
 which maps $[0,b]$ onto $[0,b+(p-1)a]$  with  $p > 1$ , $a \neq 0$, 
 and $b$ as well as $b + (p-1)a$ positive.
 
 Suppose first that $a > 0$ and choose a large natural number $k$ 
 such that $a' = a/p^k$ 
is strictly smaller than $b$.
Lemma \ref{lem:TheoremA} below then shows 
that there is a PL-homeomorphisms $f \in G(\R;A,P)$, 
having singularities in 0, $a'/p$ and $a'$ and slopes 1, $p$, $p^{k+1}$, 1
which maps the interval $[0, b]$ onto $[0, b + (p-1)a]$.
If, on the other hand, $a$ is negative,
the lemma permits one to find an PL-homeomorphism  $f \in G(\R;A,P)$ 
which takes $[0,b+(p-1)a]$ onto 
\[
[0, (b+(p-1)a) + (p-1) \cdot (-a)] = [0,b]
\]
and so $f^{-1}$ establishes the validity of our contention.
\end{proof}

\begin{lemma}
\label{lem:TheoremA}
Assume $a$ and $b$ are positive elements of $A$ and $p > 1$ is a slope in $P$.
Choose a positive integer $k$ with $a/p^k < b$ and set $a' = a/p^k$.
Let $f$  be the PL-function
which interpolates  the assignments
\[
0 \mapsto 0, \quad (a'/p) \mapsto a', \quad a' \mapsto a' + (p-1)a
\]
linearly, and is the identity on $]-\infty, 0]$ and the translation with amplitude $(p-1)a$ on $[a', \infty[$.
Then $f$ is PL-homeomorphism with singularities in 0, $a'/p$, $a'$ 
and slopes 1, $p$, $p^{k+1}$, 1,
and $f$ maps the interval $[0,b]$ onto  $[0, b + (p-1)a]$.
\end{lemma} 

\begin{proof}
It suffices to check 
that the slope on the interval $[0,a'/p]$ is $p$,
and that the slope $s_2$  on $[a'/p, a']$ is $p^{k+1}$. 
The first claim is obvious and the calculation
\[
s_2 = \frac{(a' + (p-1)a) - a'}{a' -a'/p} 
= 
\frac{(p-1)\cdot p^k \cdot a'}{(p-1)/p \cdot a'} = p^{k+1}
\]
establishes the second assertion.
\end{proof}

\begin{illustration}
\label{illustration:4.3}
\index{Quotient group A/IPA@Quotient group $A/(IP \cdot A)$!examples|(}%
\index{Module A@Module $A$!select examples}%
We determine the submodule $IP \cdot A$ 
for some specific values of $P$ and $A$,
focussing attention on a particular kind of module $A$,
namely the subring $\Z[P]$ generated, 
inside the field of real numbers, 
by the group $P$. 
The module $A$ is then cyclic
and so $A/(IP \cdot A)$, being a trivial $\Z[P]$-module, is a cyclic group.
Next, let $\PP$ be a generating set of $P$.
The augmentation ideal $I{P}$, viewed as a left module, 
is then generated by the subset $\{p-1 \mid p \in \PP\}$
(see, \eg{}\cite[Lemma VI.1.2(ii)]{HiSt97}) 
and so its image $I[P] = IP \cdot A $ 
is generated by the set $\{p-1 \mid p \in \PP\}$.
\index{Quotient group A/IPA@Quotient group $A/(IP \cdot A)$!examples}%

(i)
The preceding remarks hold, in particular, if $P$ is a cyclic group,
generated by the real number $p > 1$.
Then $I[P]$ is a principal ideal, generated by $p-1$.
Our aim now is to determine the order of the quotient ring 
$\Z[P]/I[P] = \Z[P]/\Z[P]\cdot (p-1)$.

Assume first that $p$ is an integer. 
Then $\Z[P]$ is the ring $\Z[1/p]$,
the localization of the ring $\Z$ with respect to the monoid 
$\md(p) = \{p^k \mid k \in \N \}$.
We claim that  $\Z[P]/I[P]$ is cyclic of order $p-1$.
Indeed, 
the canonical projection $\pi \colon \Z \epi \Z_{p-1} = \Z/\Z(p-1)$ 
maps the element $p$ onto the class $1 + \Z(p-1)$ 
which is invertible in the ring $\Z_{p-1}$.
It follows that $\pi$ extends uniquely to a ring homomorphism $\tilde{\pi}$ 
of $\Z[P] = \Z[1/p]$ onto $\Z_{p-1}$
(see, \eg{}\cite[Proposition 3.1]{AtMa69}).
Its kernel is the ideal generated by $p-1$
and thus the quotient ring $\Z[P]/I[P]$ is isomorphic with $\Z_{p-1}$, 
a ring whose additive group is cyclic of order $p-1$.

Suppose next that $p$  is a rational number, 
say $p = n/d$ with $n$, $d$ relatively prime positive integers.
As the numbers $n$, $d$ are relatively prime
the generalized euclidean algorithm allows one to find integers $x$ and $y$
that satisfy the relation 
\begin{equation}
\label{eq:Generalized-euclidan-algo}
1 = x \cdot n+ y \cdot d.
\end{equation}
This relation implies, first of all,  
that $1 = x \cdot (n-d) + (y  + x) \cdot d)$,
a relation showing
that $d$ maps onto an invertible element of $\Z_{n-d}= \Z/\Z(n-d)$
under the canonical epimorphism $\pi \colon \Z \epi \Z_{n-d}$.
One sees similarly 
that $\pi(n)$ is invertible in the ring $\Z_{n-d}$,
whence $\pi(n \cdot d)$ is invertible in that ring.
Thirdly, 
relation \eqref{eq:Generalized-euclidan-algo} implies
that the rational number $\frac{1}{d}$ 
equals $x \cdot \frac{n}{d} + y$ inside $\Q$ 
and so $d$ is invertible in $\Z[P]$;
one verifies in the same way that $n$ is invertible in $\Z[P]$.
We conclude 
that $\Z[P]$ is the localized ring $\Z[\frac{1}{n \cdot d}]$
and that the ideal $I[P]$,
which is generated by $n/d - 1$,
is also generated by $n-d$.
It then follows, 
as in the preceding paragraph,
that the canonical ring epimorphism $\pi \colon \Z \epi \Z_{n-d}$ 
extends to a ring epimorphism $\tilde{\pi} \colon \Z[P] \epi \Z_{n-d}$
and that its kernel is the ideal generated by $n-d$.
But if so,
the quotient ring $\Z[P]/I[P]$ is isomorphic to $\Z_{n-d}$ 
and so has order $|n-d|$.

Thirdly, 
let $p$ be an algebraic number.
The additive group of $\Z[P]$ is then a torsion-free group of finite rank;
as $I[P] = \Z[P] \cdot (p-1)$ has the same finite rank, 
$A/(IP \cdot A)$ is a torsion group;
being cyclic it is therefore finite.
If, finally, $p$ is a transcendental  number 
the ring $\Z[P]$ is isomorphic to the group ring $\Z{P}$
and so $A/IP \cdot  A$,
being isomorphic to $\Z{P}/IP$,
is infinite cyclic.

(ii) Suppose $P$ is generated by finitely many integers $p_1$, \ldots, $p_f$,
each of which is greater than 1.
Then $I[P]$ is the ideal generated by the integers 
\[
p_1 - 1, \quad \ldots, \quad p_f - 1.
\]
This ideal is contained in the principal ideal 
generated by the greatest common divisor $d > 0$ of the $p_i - 1$;
the generalized euclidean algorithm, on the other hand, shows 
that $d \in I[P]$ and so $I[P] =  \Z]P] \cdot d$.
Moreover,
as the ring $\Z[P]$ is isomorphic to the localized ring $\Z[1/(p_1 \cdots p_f)]$,
the inclusion $\Z \incl \Z[P]$ induces an isomorphism 
$\Z / d \Z \iso \Z[P]/I[P]$,
and so $\Z[P]/I[P]$ is cyclic of order $d$.
\index{Quotient group A/IPA@Quotient group $A/(IP \cdot A)$!examples|)}%
 \end{illustration}

\section{Applications}
\label{sec:5}

\subsection{Orbits of $G(I;A,P)$}
\label{ssec:5.1}
%
The first corollary of Theorem \ref{TheoremA} 
gives an algebraic characterizations of the orbits of $G(I;A,P)$ 
acting on $A \cap \Int(I)$.
\begin{crl}
\label{CorollaryA1}
\index{Group G(I;A,P)@Group $G(I;A,P)$!orbits in A cap I@orbits in $A \cap \Int(I)$}%
\index{Submodule IPA@Submodule $IP \cdot A$!significance}%
\index{Theorem \ref{TheoremA}!consequences}%
If $I = \R$ the canonical action of $G(I;A,P)$ on $A$ is transitive; 
otherwise each orbit in $A \cap \Int(I)$ has the form $(IP \cdot A + a) \cap \Int(I)$  with $a \in A$.
\end{crl}

\begin{proof}
If $I =\R$ the affine group $\Aff(A,P)$ is a subgroup of $G(I;A,P)$
and it acts transitively on $A$. 
Otherwise, 
there exists an element $a_0 \in A\smallsetminus I$  which is fixed by $G(I;A,P)$. 
Let $b$ be an element of $A \cap  \Int(I)$ and consider $f \in G(I;A,P)$.
By Theorem \ref{TheoremA} the differences $f(b) - a$ and $b - a$ are congruent \emph{modulo} $IP \cdot  A$,
 whence $f(b) - b$ is in $IP \cdot A$. 
 This shows that each orbit of $G(I;A,P)$ is contained in a coset $IP \cdot A + a$ of $IP \cdot A$.
 Conversely, 
 assume $b$, $b'$ are elements of $A \cap \Int(I)$ which are congruent \emph{modulo} $IP \cdot A$. 
 Since $A$ is a dense subset of $\R$ there exist elements $a$, $c$ in $A \cap \Int(I)$ with
\[
a < \min(b,b') \quad \text{and}  \quad \max(b,b') < c .
\]
By Theorem \ref{TheoremA} 
there are homeomorphisms $f_1$  and $f_2$ in $G(\R;A,P)$ with
\[
f_1(a) = a, \quad f_1(b) = b' \quad \text{and} \quad f_2(b) = b', \quad f_2(c) = c.
\]
Then the function $f \colon \R \to \R$, given by
\[
f(t) 
= 
\begin{cases}  
t &\text{if }  t \notin [a,c],\\ 
f_1(t) & \text{if } a \leq  t \leq b ,\\ f_2(t) &\text{if } b < t \leq c
\end{cases} 
\]
is in $G(I;A,P)$ and takes $b$ to $b'$.
\end{proof}

Suppose $I$ is not the entire line $\R$.
Corollary \ref{CorollaryA1} tells one then
that each orbit $\Omega$ of $G = G(I;A,P)$ in $A \cap \Int(I)$
is the intersection of $\Int(I)$ and a coset $IP \cdot A + a_\Omega$ 
of the abelian group $IP \cdot A$.
There exists also an algebraic description of the $G$-orbits 
in $\Int(I) \smallsetminus A$.
It involves, however, not only the abelian group $A$, 
but also the canonical action of $P$ on $A$,
and reads thus:
\begin{crl}
\label{CorollaryA1*}
\index{Group G(I;A,P)@Group $G(I;A,P)$!orbits in int(I)@orbits in $\Int(I)$}%
\index{Orbits of G(I;A,P)@Orbits of $G(I;A,P)$}%
\index{Affine group Aff(A,P)@Affine group $\Aff(A,P)$!significance}%
\index{Theorem \ref{TheoremA}!consequences}%
If $I = \R$ 
the orbits of $G(I;A,P)$ coincide with those of the affine group  $\Aff(A,P)$.
If $I \neq \R$ 
the orbits of $G(I;A,P)$ are the intersections of $\Int(I)$ 
with the orbits of the group $\Aff(IP \cdot A,P)$.
\end{crl}

\begin{proof}
Assume first that $I \neq \R$
and consider a point $t_* \in \Int(I) \smallsetminus A$
and a function $f \in G = G(I;A,P)$.
Then $f$ is differentiable in $t_*$ 
and there exists a point $a_1 \in A \cap  \Int(I)$ with $a_1 < t_*$ 
so that $f$ has slope $f'(t_*)$ on $[a_1, t_*]$.
The image of $f$ in $t_*$ can then be expressed as follows:
\begin{align}
f(t_*) &=
\left(f(t_*)- f(a_1)\right)+ \left(f(a_1) - a_1\right) + a_1 
\notag\\
&=
f'(t_*) \cdot (t_*- a_1) +  \left(f(a_1) - a_1\right) + a_1 
\notag\\
&=
\label{eq:Expressing-f(tsubstar)}
 f'(t_*) \cdot t_* + \left(f(a_1) - a_1\right)  + (1 - f'(t_*))\cdot a_0
\end{align}
Since $I \neq \R$,
the difference $f(a_1) - a_1$ lies in the submodule $IP \cdot A$
by Corollary \ref{CorollaryA1} 
and the term $(1- f'(t_*))\cdot a_0$ lies in this submodule by definition.
The previous calculation and the fact that $f(t_*)$ must be inside $\Int(I)$ 
imply therefore that $f(t_*)$
and, more generally, the entire orbit $\Omega_{t_*}$ of $t_*$, 
is contained in the set
\[
O(t_*) = \left(P \cdot t_* + IP \cdot A\right) \cap \Int(I).
\]

We verify next $\Omega_{t_*}$ coincides with this set.
Let $p_* \in P$ and $b_* \in IP \cdot A$ be elements such 
that $p_* \cdot t_* + b_* \in \Int(I)$.
Choose elements $a_0$, $a_1$,  $a_2$ and  $a_3$ in $ A \cap \Int(I)$ 
which satisfy the restrictions 
\[
a_0 < a_1 < t_* < a_2 < a_3
\quad\text{and}\quad
a_0 < p_* a_1 + b_* < p_* a_2 + b_* < a_3.
\]
Since $p_* a_1 + b_* = a_1 + (p_* - 1) \cdot a_1 + b_* \in a_1 + IP \cdot A$,
Theorem \ref{TheoremA} allows one to find a PL-homeomorphism 
$f_1 \in G(\R; A, P)$
that maps the interval $[a_0, a_1]$ onto the interval $[a_0 , p_* a_1 + b_*]$.
Similarly, one can find a function $f_2 \in G(\R;A,P) $ with
$f_2([a_2, a_3]) = [p_* a_2 + b_*, a_3]$. 
The function
\begin{equation}
\label{eq:Constructing-f}
f =
\begin{cases}
f_1(t)                                                &\text{ if }   t \ \leq a_1\\ 
p_*(t-a_1) + p_* \cdot a_1 + b_*       & \text{ if } a_1 < t < a_2,\\ 
f_2(t)                                                &\text{ if }  t \geq a_2
\end{cases}
\end{equation}
is then a PL-homeomorphisms 
that belongs to $G(I;A, P)$  and
\[
f(t_*)  = p_* (t_* - a_1) + p_* a_1 + b_1 = p_* t_* + b_*.
\]
We conclude 
then $\Omega_{t_*} = \left( P \cdot t_* + IP \cdot A \right) \cap \Int(I)$.

Assume now that $I = \R$ and that $t_* \in \R \smallsetminus A$.
Then the first part of previous argument remains valid, save for the fact 
that the difference $f(a_1) - a_1$ occurring in the right hand side of equation 
\eqref{eq:Expressing-f(tsubstar)}
is an element of $A$, but not necessarily of $IP \cdot A$,
and so the first part shows only hat 
$\Omega_{t_*} \subseteq P \cdot t_* + A$.
The reverse of this inclusion is then an immediate consequence of the fact
that the affine map $t \mapsto p_* \cdot t + b_*$ lies in $G(\R;A,P)$
for every $(p_*, b_*) \in P \times A$
and so $\Omega_{t_*} = P \cdot t_* + A$.

So far we assumed that $t_* \in \Int(I)  \smallsetminus A$.
Consider now a point $t_* \in A \cap \Int(I)$.
Then $P \cdot t_* + IP \cdot A = t_* +  IP \cdot A$
and, by Corollary \ref{CorollaryA1}, 
this set is the orbit of $G$ through $t_*$ in case $I \neq \R$. 
If, on the other hand,
$I = \R$ then, according to Corollary \ref{CorollaryA1},
the orbit of $G$ through $t_*$ equals $A$,
and this set coincides with $O_{t_*} = P \cdot t_* + A$.

The proof is now complete except for the fact 
that the orbits have so far been described without recourse to affine groups.
This is easily put right,
for  
\[
P \cdot t_* + A =\Aff(A,P). t_*
\quad\text{and}\quad
P \cdot t_* + IP \cdot A   = \Aff(IP \cdot A,P). t_*. 
\] 
 \end{proof}
%
\subsection{The images of $\lambda$, $\rho$ and $(\lambda,\rho)$}
\label{ssec:5.2}
%
We begin by determining the image of the homomorphism 
\[
\lambda  \colon G(I;A,P) \to \Aff(A,P)
\]
defined in section \ref{ssec:3.2}. 
The answer depends on $I $, and of course on $( A,P)$,  
and is given by
\begin{crl}
\label{CorollaryA2}
\index{Image of the homomorphism!lambda@$\lambda$}%
The image of $\lambda  \colon G(I;A,P) \to \Aff(A,P) $ is as follows:
\begin{enumerate}[(i)]
\index{Homomorphism!11-lambda@$\lambda$}%
\item $\{\id \}$ if $I$  is bounded from below but $\inf I$ is not in $A$;
\item  $\{ p \cdot \id \mid p \in  P\}$  if $I$ is bounded from below and  $\inf I$ is in $A$;
\item 
$\Aff(IP\cdot A,P)$ if $I$ is not bounded from below but bounded from above;
\item
$\Aff(A,P)$ if $I = \R$.
\end{enumerate}
\end{crl}

\begin{proof}
(i) Set $a = \inf I$ and consider $f \in G(I;A,P)$. 
Then $\supp f \subseteq [a, \infty[$ and so $f$ is the identity on the interval $]-\infty, a]$.
Since $a \in \R \smallsetminus A$ the point $a$ is not a singularity of $f$; 
so $\sigma_-(f) = f' (a) = 1$ and thus $\lambda(f) = \id$.
\smallskip

(ii) Set $a = \inf I$ and  choose $c \in \Int(I)$. 
Given $p \in P$ find a positive element $\Delta \in  A$ with $(1+p)\Delta < c - a$. 
Then define a function $b(a, \Delta;p) \colon \R \to \R$
by setting
\begin{equation}
\label{eq:5.1}
b(a, \Delta;p)t) 
=
\begin{cases}  
t                                        &\text{ if }   t \notin[a, a + (1+p) \Delta]\\ 
p(t-a) + a                          & \text{ if } a \leq t < a+\Delta,\\ 
p^{-1} (t - (a+\Delta)) + (a + p \cdot \Delta) &\text{ if } a + \Delta \leq t \leq a + (1 + p)\Delta.
\end{cases} 
\end{equation}
Then $f$ belongs to $G(I;A,P)$ and $\lambda(f) = p \cdot \id$.
\begin{remark}
\label{remark:Usefulness-of-this-function}
The homeomorphism $f$ defined by formula \eqref{eq:5.1} will find various uses in this monograph,
in particular in section \ref{ssec:Action-TildeG} and in the proof of Theorem \ref{TheoremE04}.
In these sections $f$ is called $b(a,\Delta;p)$ 
to stress the parameters involved in $f$.
Its rectangle diagram is displayed in Example \ref{example:Elements-in-B}
on page \pageref{notation:b(a,Del;p)}.
\end{remark}

(iii) Given  $f \in G(I;A,P)$ there exists $a \in  A$ so that $f$ coincides with the affine map 
$\lambda (f) $ on $]-\infty, a[$. 
Then
\[
f (a) - a = \lambda(f) (a) - a = \left(\sigma_-(f) \cdot  a + \tau_-(f)\right) -a \in IP \cdot A + \tau_-(f).
\]
Since $f(a) - a$ is in $IP \cdot A$ by Corollary \ref{CorollaryA1}, 
this chain of equalities reveals 
that  $\tau_-(f)$ is in $IP \cdot  A$ and so $\lambda(f)$  in $\Aff(IP \cdot A,P)$. 
Conversely, if $g$ is in $\Aff(IP \cdot A, P)$,
choose $c \in A\, \cap\, \Int(I)$ and then $b \in A\, \cap \; ]-\infty,c[$ so that $g(b) < c$.
Next, use Theorem \ref{TheoremA} to find an element $h \in G(\R;A,P)$ 
that maps $[b,c]$ onto $[g(b),c]$,
and define the function $f \colon \R \to \R$ by
\begin{equation}
\label{eq:5.2}
\index{Homomorphism!17-rho@$\rho$}%
f(t)
=
\begin{cases}  
g(t )        &\text{ if }   t < b\\ 
h(t)         & \text{ if } b \leq t \leq c,\\ 
t              &\text{ if }  t > c.
\end{cases} 
\end{equation}
Then $f$ is in $G(I;A,P)$ and $\lambda(f) = g$.

(iv) follows from the fact that $\lambda$ retracts $G(\R;A,P)$ onto $\Aff(A, P)$.
\end{proof}

The situation for $\rho$  is symmetric  to that of $\lambda$,
\index{Homomorphism!17-rho@$\rho$}%
\index{Image of the homomorphism!-(rho@$\rho$}%
and so we move on to the discussion of the image of $(\lambda, \rho)$.
\begin{crl}
\label{CorollaryA3}
\index{Image of the homomorphism!lambda,rho)@$(\lambda, \rho)$}%
The image of $(\lambda, \rho) \colon G(I;A,P) \to \Aff(A,P)^2$ 
is the product $\im \lambda \times \im \rho$ if $I \neq \R$
and, otherwise, the subgroup  
\[
\{(f, g)  \in \Aff(A;P)^2 \mid   \tau_+(g)  - \tau_-(f)  \in IP \cdot A\}.
\]
\end{crl}

\begin{proof}
If $I \neq \R$ 
the constructions given by equations \eqref{eq:5.1} and \eqref{eq:5.2} 
occurring in the proof of Corollary \ref{CorollaryA2},
and the analogous constructions needed to establish the corresponding claims for $\rho$, 
can be carried out simultaneously, whence the first assertion.

If $I = \R$,
both $\lambda$ and $\rho$ retract onto the affine subgroup $\Aff(A,P)$ 
whence the diagonal group  $\{(g,g) \mid g \in \Aff(A,P) \}$ is contained in the image of $(\lambda, \rho)$.
The claim is therefore equivalent to the assertion that 
\[
\rho(\ker \lambda) = \Aff(IP \cdot A , P).
\]
This assertion holds by part (iii) in Corollary \ref{CorollaryA2} and the fact 
that $\ker \lambda$ is the union  $\bigcup\nolimits_{a \in A} G([a, \infty[\, ; A,P)$.
\end{proof}

\subsection{Multiple transitivity}
\label{ssec:5.3}
%
The groups $G(I;A,P)$ are \emph{ordered permutation groups} $(G,\Omega)$,
\ie{}groups acting faithfully on a totally ordered set $\Omega$ 
by order preserving permutations. 
\index{Ordered permutation groups!definition|textbf}%
\footnote{For details on \emph{Ordered Permutation Groups}
we refer the interested reader to the account of A. M. W. Glass \cite{Gla81}.}
From the point of view of ordered permutation groups 
it is useful to know the multiple transitivity properties of the groups $G(I;A,P)$. 
Suppose $(G,\Omega)$ is an ordered permutation group and $\ell$ is  a positive integer.
One says $G$ acts \emph{$\ell$-fold transitively on} $\Omega$ 
\index{Ordered permutation groups!multiple transitivity}%
if, 
given a couple of $\ell$-tuples 
$(\omega_1, \ldots, \omega_\ell)$  and $(\omega'_1, \ldots, \omega'_\ell)$ 
satisfying $\omega_1 < \omega_2 < \cdots < \omega_\ell$ 
and $\omega'_1 < \omega'_2 < \cdots <  \omega'_\ell$,
there exists an element $g \in G$ with 
$g(\omega_1 ) = \omega_1'$, 
$g(\omega_2) = \omega'_2$, 
\ldots, 
$g(\omega_\ell) = \omega'_\ell$.

\begin{crl}
\label{CorollaryA4}
\index{Ordered permutation groups!multiple transitivity}%
\index{Doubly transitive action}%
\index{Group G(I;A,P)@Group $G(I;A,P)$!multiple transitivity}%
Let $G$ be the group $G(I;A,P$) 
and let $\Omega$ be a $G$-orbit contained in $A \cap \Int(I)$. 
Then $G$ acts faithfully on $\Omega$; 
moreover, 
the $G$-action on $\Omega$ is $\ell$-fold transitive for every $\ell \geq  1$, 
if $I \neq \R$ or $IP \cdot A = A$. 
If, however, $I = \R$ and $IP \cdot A < A$, 
the $G$-action is transitive but not 2-transitive.
\end{crl}

\begin{proof}
By Corollary \ref{CorollaryA1}
the $G$-orbit $\Omega$  is the intersection of a coset of $IP \cdot A$ in $A$ 
and $\Int(I)$
whenever $I \neq \R$,  and $A$ itself  if $I = \R$; 
in both cases, $\Omega$ is thus dense in $\Int(I)$.
Since the support of each homeomorphism in $G$ is contained in $\Int(I)$,
the group $G$ thus acts faithfully on $\Omega$.

Fix now a point $\omega_0 \in \Omega \subseteq A\, \cap  \,\Int(I)$.
Its stabilizer $H$ is the direct product 
\[
H = G(I_1;A,P) \times G(I_2;A,P)
\]
with $I_1 = I \, \cap \,] - \infty, \omega_0]$ 
and $I_2= I \,\cap\, [\omega_0, +\infty[$.
Assume next that $\Omega$ has the form 
\[
(IP \cdot A + a_0)\, \cap\, \Int(I)
\] 
for some $a_0 \in A$.
Then $\Omega_2 = \Omega \, \cap \Int(I_2)$ 
equals $(IP \cdot A + a_0)\, \cap\, \Int(I_2)$ 
and so $G_2$ acts transitively on $\Omega_2$ 
by  Corollary \ref{CorollaryA1}.
The group $G$ itself acts therefore 2-fold transitively on $\Omega$.
Now $\Omega_2$ has the form $(IP \cdot A + a_0)\, \cap\, \Int(I_2)$;
so the previous argument can be applied to $(G_2, I_2)$ in place of $(G, I)$.
It follows, first, that $G$ acts 3-fold transitively on $\Omega$ 
and then by, iteration,
that  $G$ acts $\ell$-fold transitively on $\Omega$ for every $\ell \geq 1$.

In the previous argument we assumed 
that the $G(I;A,P)$-orbit $\Omega$ has the form $(IP \cdot A + a_0 ) \cap \Int(I)$;
by Corollary \ref{CorollaryA1} this assumption is fulfilled, if and only if, 
$ I \neq \R$ or $A = IP \cdot A$.
Suppose, finally, that $I = \R$ and $A >IP \cdot A$.
By the corollary, $G = G(\R;A,P)$ acts then transitively on $\Omega =  A$,
but $G_2 = G([0, \infty[\;;A,P)$ does not act transitively on 
$A_{>0} = A \cap \;]0, \infty[$, 
and so $G$ does not act 2-fold transitively on the orbit $A$.
\end{proof}

\begin{remark}
\label{remark:5.3}
If $\Int(I)$ is approximated by compact subintervals $[a, c]$
the preceding corollary implies
that $B(I;A, P)$ acts $\ell$-fold transitively on each of its orbits in $A \cap \Int(I)$,
the number $\ell$ being an arbitrary positive integer.
\index{Subgroup B(I;A,P)@Subgroup $B(I;A,P)$!properties}%

Suppose now that $f$ and $g$ are homeomorphisms of $\R$ with 
\[
\supp f < g (\supp f).
\]
Then the commutator $[f,g]  = f \cdot (\act{g}f)^{-1}$ coincides with $f$ 
on $\supp f$.
This fact and  the multiple transitivity  properties of $B = B(I;A, P)$ 
imply therefore 
that derived group of $B$ acts also $\ell$-fold transitively 
on every orbit of $B$ in $A  \cap \Int(I)$.
\index{Subgroup B(I;A,P)@Subgroup $B(I;A,P)$!properties}%
\end{remark}

\subsection{A density result}
\label{ssec:5.4}
 In this final section of Chapter \ref{chap:A}
 we view $G(I;A,P)$ as a topological group 
 and show that it is dense in a far larger group.
 \begin{corollary}
\label{CorollaryA5}
\index{Group G([a,c];A,P)@Group $G([a,c];A,P)$!density property}%
\index{Theorem \ref{TheoremA}!consequences}%
Let $\Homeo_o(I)$ denote the topological group of all orientation preserving homeomorphisms
$g \colon \R \iso \R$ with support in $I$, 
equipped with the topology of uniform convergence on compact intervals of $\R$. 
If $I$, $A$ and $P$ satisfy the non-triviality requirement \eqref{eq:Non-triviality-assumption} 
then $G(I;A,P)$ is a dense subgroup of $\Homeo_o(I)$.
\end{corollary}

\begin{proof}
Let $[a, c]$ be a compact interval contained in $I$ having endpoints in $A$.
If $a$ is the left endpoint of $I$, set $a_* = a$; 
otherwise let $a_* \in I \cap A$ be a point with $a_* < a$.
Define $c_* \in I \cap A$ with $c \leq c_*$ similarly.

Fix now a homeomorphism $g \in \Homeo_o(I)$ 
and let $\varepsilon$ be a given positive real.
Since $g\restriction{ [a,c]}$ is uniformly continuous and $A$ is dense in $\R$,
there exists a strictly increasing sequence $b_0 = a$, $b_1$, \ldots, $b_h = c$ of elements in $A$ 
so that  
\[
|g(b_i) - g(b_{i-1}) | \leq \varepsilon /2 \text{ for } i \in \{1,2, \ldots, h \}.
\]
Our next aim is to define 
a strictly increasing sequence $b'_0 < b'_1 < \cdots < b'_h$ of elements in $A$
that will be the values of the  PL-homeomorphisms $f \in G(\R;A, P)$
on the sequence $b_0$, $b_1$, \ldots, $b_h$.
Several cases arise.

Assume first that $I = [a, c]$.
Then the homeomorphism $g$ fixes both $a = b_0$ and $c = b_h$.
As the submodule $IP \cdot A$ is dense in $\R$ 
there exists therefore a strictly increasing sequence 
$b'_ 0 =b_0 < b'_1 < \cdots < b'_h = b_h$ of elements in $A$ 
with the following properties:
\begin{equation}
\label{eq:Condition-on-sequence}
| b'_i - g(b_i) | \leq \varepsilon/2 
\quad \text{and} \quad 
b'_i - b_i \in  IP \cdot A \text{ for } i \in \{1, \ldots, h-1\}.
\end{equation}

Consider now an index $i \in \{1, \ldots, h\}$. 
By the choices of $b'_{i-1}$ and $b'_i$
the difference 
\[
(b'_i - b'_{i-1}) - (b_i - b_{i-1}) = (b'_i - b_i) - (b'_{i-1} - b_{i-1})
\]
 lies in $IP \cdot A$;
so Theorem \ref{TheoremA} allows us to find a function $f_i \in G(\R;A,P)$ 
sending the interval $[b_{i-1}, b_i]$ onto the interval $[b'_{i-1}, b'_i]$.
Let $f \colon \R \to \R$ be the function 
whose restriction to each of the intervals  $[b_{i-1},b_i]$ is $f_i$, 
 and which fixes every point outside of $[a,c]$.
 Then $f \in G(I;A,P)$.
 Moreover,
 given $t \in [a, c]$, 
 there exists an index $i$ with $b_{i-1} \leq t \leq b_i$ and either $g(t) \geq f(t)$ or $f(t) >g(t)$.
 In the first case, 
 the chain of inequalities
 \begin{align*}
|g(t ) - f(t)| =  g(t) - f(t) &\leq g(b_i)   - f(b_{i-1}) \\ 
&\leq
|g(b_i) - g(b_{i-1}) | + |g(b_{i-1}) - b'_{i-1} |
\leq \varepsilon/2 + \varepsilon/2
 \end{align*}
 holds and so $|g(t ) - f(t)| \leq \varepsilon$.
 In the second case, the claim follows from the chain of inequalities
 \begin{align*}
|g(t ) - f(t)| =  f(t) - g(t) &\leq f(b_i)   - g(b_{i-1}) \\ 
&\leq
|b'_i - g(b_i) |  + |g(b_i) - g(b_{i-1})|
\leq \varepsilon/2 + \varepsilon/2.
 \end{align*}
 
Assume next that $a_* < a$ and $c_* = c$.
Choose $b'_0 < b'_1 < \cdots < b'_h$ so that 
\[
a_* < b'_0, \quad |b'_0 - g(b_0)| \leq \varepsilon/2, 
\quad
b'_0 - b_0 \in  IP \cdot A.
\]
and so that the requirements \eqref{eq:Condition-on-sequence} are satisfied.
The choices of $a_*$ and $b'_0$ permit one to construct,
with the help of Theorem \ref{TheoremA},
\index{Theorem \ref{TheoremA}!consequences}%
a PL-homeomorphism $f_0$
that maps the interval $[a_*, b_0]$ onto the interval $[a_*, b'_0]$.
Define the functions $f_1$, \ldots, $f_h$ as before,
and let $f \colon \R \to \R$ be the function 
whose restrictions to the intervals  
\[
[a_*, a_0], \quad [b_0, b_1], \ldots, [b_{h-1}, b_h]
\] 
are $f_0$, $f_1$, \ldots, $f_h$, 
and which fixes every point outside of $[a_*,c]$.
Then $f \in G(I;A,P)$.
Moreover,
one checks, as in the previous case, 
that $|f(t) - g(t)| \leq \varepsilon$ for every $t \in [a, c]$.
The cases 
where either $a_* = a$ and $c < c_*$, or $a_* < a$ and $c < c_*$,
can be handled similarly to the second case.
\end{proof}
%

%% file: chaptB_Generating-sets.tex
%
%
\chapter{Generating Sets}
\label{chap:B}
%
\setcounter{section}{5}
%
\section{Necessary conditions for finite generation}
\label{sec:6}
%
Suppose $P_1$ is a subgroup of $P$ 
and $A_1$ is a $\Z[P_1]$-submodule of $A$. 
Then $G(I;A_1,P_1 )$ is a subgroup of $G(I;A,P)$, 
and it is a proper subgroup unless $P_1 = P$ and $A_1 = A$.
Indeed, 
given $a \in  \Int(I) \cap A$ and $p \in P$, 
formula \eqref{eq:5.1} allows one to construct an element $f \in G(I;A,P)$ 
for which
\[
\label{notation:f-prime-at-a-minus}%
f'(a_-) = \lim_{t \nearrow a} f'(t) = 1 \quad \text{and} \quad f'(a_+) = \lim_{t \searrow a} f'(t) = p.
\]
Similarly, 
if $I_1$ is a proper, closed subinterval of $I$, 
then $G(I_1;A,P)$ is a proper subgroup of $G(I;A,P)$. 
Finally, if $X \subset  \Int(I) \cap A$ is a union of orbits
the set $G_X$ of all homeomorphisms $f \in G(I;A,P)$  with singularities in $X$
is a subgroup of $G$, 
and if $X_1 \subsetneqq X_2$  are two such subsets 
the group $G_{X_1}$ is properly contained in $G _{X_2}$.
These facts and Corollary \ref{CorollaryA1} yield a proof of
\begin{prp}
\label{PropositionB1}
\index{Group G(R;A,P)@Group $G(\R;A,P)$!finite generation}%
\index{Group G([0,infty[;A,P)@Group $G([0, \infty[\;;A,P)$!finite generation}%
\index{Group G([a,c];A,P)@Group $G([a,c];A,P)$!finite generation}%
\index{Submodule IPA@Submodule $IP \cdot A$!significance}
If $G(I;A,P)$ is finitely generated then
\begin{itemize}
\item $P$ is a finitely generated group;
\item $A$ is a finitely generated $\Z[P]$-module;
\item $A/(IP \cdot A)$ is finite or $I = \R$;
\item $\inf I$  and $\sup I$ lie  in $\{-\infty\} \cup  A \cup  \{+\infty\}$.
\end{itemize}
\end{prp}

%
\section{Generators for groups with supports in the line}
\label{sec:7}
%
For $a \in \R$  and $p \in \R_{>0}$,
we denote by $\aff(a,p)\colon \R \iso \R$ the affine homeomorphism 
which takes $t$ to $pt + a$, \label{notation:aff(a,p)}%
and by $g(a,p) \colon \R \iso \R$ the PL-homeomorphism given by
\[
\label{eq:g(a,p)}
g(a,p) (t) 
= 
\begin{cases} t &\text{ if } t \leq a,\\ 
p(t-a) + a &\text{ if } t \geq a. 
\end{cases}
\]
The group $G(\R;A,P)$ is generated by the subset
\begin{equation}
\index{Group G(R;A,P)@Group $G(\R;A,P)$!generating sets}%
\label{eq:7.1}
\Aff(A,P) \cup  \{g(a,p)\mid (a,p) \in A \times P \} ;
\end{equation}
this assertion can be verified by a straightforward induction on the number of singularities of the elements of $G(\R;A,P)$
(see \cite[Theorem 2.3]{BrSq85}). 
Note that the subset \eqref{eq:7.1} consists of all elements of $G(\R;A,P)$
with at most one singularity.

The generating set \eqref{eq:7.1} satisfies some simple relations:
\begin{align}
\aff(a,p) \cdot  \aff(a',p') &= \aff(a+pa',pp');
\label{eq:7.2}\\
g(a,p) \cdot  g(a,p') &= g(a,pp');
\label{eq:7.3}\\
\act{\aff(a,p)} g(a', p') &= g(a + p a', p');
\label{eq:7.4n}\\
\act{g(a,p)} g(a', p') &= g(a+p (a '-a), p') \quad \text{provided} \quad a < a'.
\label{eq:7.5n}
\end{align}

Actually, 
these relations define the group $G(\R;A,P)$ --- see section \ref{ssec:13.1}. 
\index{Group G(R;A,P)@Group $G(\R;A,P)$!infinite presentations}%
Our present interest in them lies in the fact 
that they enable us to prove the converse of Proposition \ref{PropositionB1}
for $I = \R$.

\begin{theorem}
\label{TheoremB2}
\index{Group G(R;A,P)@Group $G(\R;A,P)$!finite generation}%
\index{Finiteness properties of!G(R;A,P)@$G(\R;A,P)$}%
The group $G(\R;A,P)$ is finitely generated if, and only if, 
$P$ is a finitely generated group and $A$ is a finitely generated $\Z[P]$-module.
\end{theorem}

\begin{proof}
By Proposition \ref{PropositionB1}
the stated conditions are necessary  for finite generation.
To prove that they are sufficient, let $\PP$ be a finite set generating the group $P$ 
and $\Acal$ a finite set of $\Z[P]$-module generators of $A$. 
Relations \eqref{eq:7.2} then imply 
that the affine subgroup $\Aff(A,P)$ of $G(\R;A,P)$ is generated by the finite set
\begin{equation}
\label{eq:7.6}
\{\aff(0,p)\mid p \in \PP \} \cup \{\aff(a,1)\mid a \in  \Acal\}.
\end{equation}
Next, equation \eqref{eq:7.3} shows 
that the finite set
\begin{equation}
\label{eq:7.7} 
\{g(0,p)\mid p \in  \PP\}
\end{equation}
generates the subgroup $H = \{g(0,p) \mid p \in P\}$. 
Equation \eqref{eq:7.4n}, finally, reveals
that the translation $\aff(a, 1) \colon t\mapsto t +a$
conjugates $g(0,p')$ onto $g(a, p')$.
All taken together, 
this shows that $G(\R;A,P)$ is generated by the union of the finite sets 
\eqref{eq:7.6} and \eqref{eq:7.7}.
\end{proof}

\section{Generators for groups with supports in a half line}
\label{sec:8}
%
\subsection{Generators for $G = G([0, \infty[\;;A,P)$ 
and for $\ker(\sigma_- \colon G \to P)$}
\label{ssec:8.1}
%
Let $I$ be an interval which is bounded on one side but not on the other side. 
It then follows from bijection \eqref{eq:3.2}
that the group $G(I;A,P)$ is isomorphic to $G([a,\infty[\, ;A,P)$ 
for some $a \in \R$. 
Moreover, if $a \in A$ then the groups $G ([a,\infty[ \,;A,P)$ 
and $G ([0,\infty[\, ;A,P)$ are isomorphic.
A more detailed analysis, 
to be carried out in section \ref{ssec:16.4},
discloses 
that the groups $G([a, \infty [\;;A,P)$ with $a \in \R \smallsetminus A$ 
are all isomorphic to
\[
\ker(\sigma_- \colon  G([0,\infty[ \;;A,P)  \to P) 
\]
and that this kernel is not isomorphic to $G ([0,\infty[\;;A,P)$. 
We can therefore restrict attention to the groups 
$G = G ([ 0 ,\infty[ \,;A,P)$ and $G_1 = \ker \sigma_-$,
whenever this is convenient.

Given $b\in A$,
let $g(b,p)$ be the PL-homeomorphism
\begin{equation}
\label{eq:8.1}
g(b,p) (t) = \begin{cases} t & \text{ if } t \leq b,\\ p(t-b) + b & \text{ if } b \leq t. \end{cases}
\end{equation}
Fix $a \in A$.
An easy induction on the number of singularities shows 
that the group $G = G([a, \infty[\, ; A,P)$ is generated by the set
\begin{equation}
\label{eq:Generators-G}
\index{Group G([0,infty[;A,P)@Group $G([0, \infty[\;;A,P)$!generating sets}%
\XX_a = \{g(b,p) \mid b \geq a \text{ and } p \in P \}.
\end{equation}
Similarly, $G_1 = \ker(\sigma_- \colon G \to P)$ is generated by 
\begin{equation}
\label{eq:Generators-G1}
\XX_{>a} = \{g(b,p) \mid b > a \text{ and } p \in P \}.
\end{equation}
These generators satisfy the relations
\begin{align}
g(b,p)  \cdot g(b,p') &= g(b,pp') \label{eq:8.2}\\
\intertext{and}
\act{g(b,p)} g(b', p') &= g(b +p(b '- b),p')  \text{ whenever } b < b'.
\label{eq:8.3}
\end{align}
\index{Group G([0,infty[;A,P)@Group $G([0, \infty[\;;A,P)$!infinite presentation}%
The stated relations are actually defining, 
as can easily be deduced by a normal form argument 
due to M. G. Brin and C. C. Squier  
(see \cite[Thm.\,2.3]{BrSq85}). 

In this section, 
we use the above relations to obtain a more economical set of generators 
for $G([0, \infty[\;;A,P)$.
\begin{proposition}
\label{PropositionB3}
\index{Group G([0,infty[;A,P)@Group $G([0, \infty[\;;A,P)$!generating sets}%
Let $\PP$ be a set of generators of $P$ and
$\Acal$ a set of positive reals generating $A$ \emph{qua} $\Z[P]$-module, 
and let $\TT$ be a set of positive real numbers representing 
the cosets of $IP \cdot A$ in $A$. 
Then the set
\begin{equation}
\label{eq:8.4}
\{g(0,p)\mid p \in \PP\} \cup  \left\{g(b,p)\mid b \in \Acal \cup \TT  \text{ and  } p \in \PP \right\}
\end{equation}
generates the group $G([0, \infty[\;; A,P)$.
\end{proposition}

\begin{proof}
Let $H$ denote the subgroup generated by the set \eqref{eq:8.4} 
and let $\Afr$ denote the subset of $A\, \cap \;]0,\infty[$ defined by
\[
\Afr= \left\{a \in  A \, \cap \;]0, \infty[ \;\;  \big |  \; \;g(a, p) \in H \text{ for  each } p \in \PP \right\}.
\]
Our first aim is to show that $\Afr=  A \, \cap \;]0, \infty[$ ; 
the definition \eqref{eq:Generators-G} of $\XX_0$ 
and the facts that $\{g(0,p)\mid p \in P \}$ is included in the set \eqref{eq:8.4}
and that $\XX_0$ generates $G = G([0, \infty[\; ;A,P)$ 
will then immediately lead to the conclusion that $H = G$.

The set $\Afr$ enjoys the following properties.
\begin{enumerate}[(I)]
\item  $\Acal \cup  \TT \subset  \Afr$.
\item $P \cdot \Afr = \Afr$.
\item If $a$, $a'$ are in $\Afr$  and $a < a' $ 
then $(1-p)a + pa'$ belongs to $\Afr$ for each $p \in P$.
\end{enumerate}
Indeed, inclusion (I) holds by the definitions of $H$ and of $\Afr$.
Property (II) is a consequence of the inclusion of $\{g(0,p) \mid p \in P \}$ into the set \eqref{eq:8.4} 
 and of the relations
\[
\act{g(0,p)^\varepsilon} g(a,p') = g(p^\varepsilon \cdot a',p'),
\]
that are special instances of relations \eqref{eq:8.3}. 
Property (III) follows from relations  \eqref{eq:8.2} and \eqref{eq:8.3}.

Now to the proof of the equality $\Afr =  A\, \cap \;]0, \infty[$.
Let  $a_*$ be an element of $A\, \cap \;]0, \infty[$.
Find the representative $b_* \in \TT$ which is congruent to $a_*$ \emph{modulo} $IP \cdot  A$. 
For a reason that will only become clear at the end of the proof, 
$a_*$ should be larger than $b_*$.
To satisfy this requirement we replace $a_*$ by $p_\omega \cdot a_*$ 
with $p_\omega \in P$ a suitably large element.
If $p_\omega \cdot a_*$ can be shown to be in $\Afr$, then $a_* \in \Afr$ by property (II).
Note that $p_\omega\cdot a_* = a_* + (p_\omega-1) a_*$ is congruent to $b_*$
\emph{modulo} $IP \cdot  A$.

As an abelian group,
the submodule $IP \cdot A$ is generated by the products $(1-p) a$ 
with $p$  ranging over $P$ and $a$ varying over $\Acal$.
There exist therefore sequences
\[
(\varepsilon_1, \ldots,  \varepsilon_\ell), \quad (p_1, \ldots, p_\ell) \quad \text{and} \quad (a_1, \ldots, a_\ell)
\]
of elements in $\{1, -1\}$, $P$  or $\Acal$, respectively, 
so that
\begin{equation}
\label{eq:Representation-of-astar}
p_\omega \cdot a_* = b_* + \sum\nolimits_{1 \leq j \leq \ell} \varepsilon_j (1-p_j) a_j.
\end{equation}
Set $ b'_i  = b_*  + \sum\nolimits_{1 \leq j \leq i} \varepsilon_j (1-p_j) a_j$ 
for $i \in \{0, 1, \ldots, \ell\}$.
\smallskip

The idea now is to prove by induction on $i$ 
that each $b'_i$ lies in $\Afr$,  whence $p_\omega \cdot a_* = b'_\ell$ will be in $\Afr$. 
If $i = 0$ all is well,  for $b'_0 = b_*  \in \TT$ 
and so $b'_0$ lies in  $\Afr$  by property (I).
The inductive step, however, presents two problems: 
first of all, some of the members $b'_i$ may be negative, 
a difficulty that can be circumvented by rearranging the terms in equation 
\eqref{eq:Representation-of-astar};
secondly,  
in the proof of implication $b'_{i-1} \in \Afr \Longrightarrow b'_i \in \Afr$ 
we intend to use property (III).
Now 
\begin{equation}
\label{eq:Rewriting-inductive-relation}
b'_i = b'_{i -1}+ \varepsilon_i (1-p_i) a_i = 
\begin{cases} 
(1-p_i)a_i + p_i(p_i^{-1} b'_{i-1}) & \text{ if } \varepsilon_i = 1, \\
(1-p_i^{-1}) p_i a_i + p_i^{-1} (p_i b'_{i-1}) & \text{ if } \varepsilon_i = -1.
\end{cases}
\end{equation}
Property (III) is thus only of help if we can guarantee 
that $a_i < p_i^{-1} b'_{i-1}$ in the first case
and that $p_i a_i < p_i b'_{i-1}$ in the second case
or, equivalently,
that
\begin{equation}
\label{eq:Original-condition}
p_i a_i < b'_{i-1} \text{ if } \varepsilon_i = 1\quad\text{and} \quad a_i < b'_{i-1} \text{  if } \varepsilon_i = -1.
\end{equation}

This problem can be overcome as follows: 
one replaces $b_*$ by $p_\alpha \cdot b_*$ 
where $p_\alpha \in P$ is a large number and sets $b_i =  (p_\alpha-1)b_*+ b'_i$.
Then $b_0 = p_\alpha b_* \in \Afr$ by property (II) and condition \eqref{eq:Original-condition} becomes
\begin{equation}
\label{eq:New-condition}
p_i a_i < b_{i-1}  \text{ if } \varepsilon_i = 1\quad\text{and} \quad a_i < b_{i-1} \text{  if } \varepsilon_i = -1.
\end{equation}
Since the elements $\varepsilon_i$, $p_i$ and $a_i$ do not depend on the choice of $p_\alpha$,
these conditions are satisfied for every sufficiently large element $p_\alpha \in P$.
It follows that 
\[
b_\ell =  (p_\alpha - 1) b_* + b'_\ell =  (p_\alpha - 1) b_* + p_\omega \cdot a_* \in \Afr
\]
and so we are left with showing  
that $p_\omega \cdot a_* = (1-p_\alpha) b_* +  p_\alpha(p_\alpha^{-1}b_\ell) \in \Afr$.

According to property (III),  this conclusion will hold 
if  $b_* < p_\alpha^{-1} b_\ell $  
or, equivalently, if $p_\alpha b_* < b_\ell$.
 Now $b_\ell =   (p_\alpha-1) b_* + p_\omega a_*$  and $ b_* < p_\omega a_*$ by the choice of $p_\omega$,
 and thus the latter condition is satisfied.
\end{proof}

By combining Proposition \ref{PropositionB3}, 
the introductory remarks in section \ref{ssec:8.1} and Proposition \ref{PropositionB1}, 
one arrives at the following characterization of the finitely generated groups 
among the groups of the form $G = G([a, \infty[\;;A,P)$: 
\begin{theorem}
\label{TheoremB4}
\index{Group G([0,infty[;A,P)@Group $G([0, \infty[\;;A,P)$!finite generation}%
\index{Finiteness properties of!G([0,infty[;A,P)@$G([0, \infty[\;;A,P)$}\index{Submodule IPA@Submodule $IP \cdot A$!significance}%
Suppose $I$ is the half line $[a,\infty[$\,. 
Then $G(I;A,P)$ is finitely generated if, and only if, the following conditions hold:
\begin{enumerate}[(i)]
\item $P$ is finitely generated;
\item $A$ is a finitely generated $\Z[P]$-module;
\item $A/ (IP \cdot  A)$ is finite, and
\item $a \in A$.
\end{enumerate}
\end{theorem}
%
\subsection{Generators for the kernel of $(\sigma_-, \sigma_+)$}%
\label{ssec:8.2}
%
In section \ref{ssec:8.1},
generating sets $\XX_a$ and $\XX_{a,+}$ for  the groups $G = G([a, \infty[\,; A,P)$  
and $G_1 = \ker (\sigma_- \colon G \to P)$ have been found 
(see formulae \eqref{eq:Generators-G} and \eqref{eq:Generators-G1}).
The generating property of these subsets can be established 
by induction on the number of singularities. 
There exists a further subgroup where this approach is feasible,
the subgroup $G_2 = \ker (\sigma_-, \sigma_+)$. 
The groups $G_1$ and  $G_2$ are both infinitely generated:
the first of them is the union of the increasing chain of subgroups $G([a_1, \infty[\;; A, P)$ 
with $a_1 > a$ and the second is the union of a similarly defined chain.

Given positive elements $a$, $b$ of $A$ and an element $p \in P$,
let  $f(a,b;p)$ be the function given by the assignments
\begin{equation}
\label{eq:8.5}
t \longmapsto
\begin{cases}  
t &\text{if }  t <a,\\ 
p(t-a) + a & \text{if } a \leq t \leq  a+ b ,\\ 
t + (p-1)b &\text{if } t > a + b.
\end{cases} 
\end{equation}
Each of the functions $f(a,b;p)$ is in $G_2$;
 we claim they generate $G_2$.
Indeed,
if $f \in G_2$, there exists a sequence 
\begin{equation}
\label{eq:8.6}
S = (a; b_1, \ldots, b_h; p_1, \ldots, p_h),
\end{equation}
with $a$, $b_1$, \ldots, $b_h$ a strictly increasing list of elements of  $A$ 
and $p_1$, \ldots, $p_h$ elements in $P$, 
that describes $f$ in the sense that $f$ has
\begin{align*}
\text{slope } 1 \text{ on}\quad  &\, ]-\infty, a[\, \cup \,]a + b_1 + \cdots + b_h, +\infty[\, ,  \text{ and} \\
\text{slope } p_i \text{ on} \quad & \,]a + \sum\nolimits_{1 \leq j < i} b_j, a + \sum\nolimits_{1 \leq j \leq i} b_j[\, \text{ for } 
i \in \{1, \ldots, h\}.
\end{align*}
It is easily verified that $f$ is the composition
\[
f(a, b_1;p_1) \circ f(a+ b_1, b_2; p_2) \circ \cdots \circ f(a + b_1 + b_2 + \cdots + b_{h-1},b_h; p_h).
\]
If $f$ is described by the sequence \eqref{eq:8.6} 
its image under the homomorphism $\rho \colon G \to \Aff(A,P)$ is the translation with amplitude
\[
\sum\nolimits_i p_i b_i - \sum\nolimits_i b_i  =  \sum\nolimits_i (p_i-1) b_i.
\]

We claim that $(\rho \restriction{G_2}) \colon G_2 \to IP \cdot A$ factors over the multiplication map
\begin{equation}
\label{eq:Multiplication-map}
\mu \colon IP \otimes _{\Z[P]} A \epi IP \cdot A
\end{equation}
which takes $(p-1) \otimes a$ to $(p-1)a$.
To prove this contention,
we assign to each sequence \eqref{eq:8.6} an element of $IP \otimes_{\Z{P}}  A$,
namely $\sum_i (p_i-1) \otimes b_i$.
If a term $b_i$ is replaced by a sum $b'_i + b''_i$ of positive elements of $A$
and the sequence is replaced  by the subdivided sequence
\[
S' = (a; b_1,\ldots, b_{i-1}, b'_i, b''_i, b_{i+1}, \ldots, b_h; p_1, \ldots, p_{i-1}, p_i, p_i, p_{i+1}, \ldots,  p_h),
\]
then $S$ and $S'$ represent the same PL-homeomorphism
and they give rise to the same element in $IP \otimes_{\Z{P}}  A$.
This shows that there exists a function 
\begin{equation}
\label{eq:tilde(rho)}
\tilde{\rho} \colon G_2 \to IP \otimes_{\Z[P]} A,
\end{equation}
sending the composition 
$f(a, b_1;p_1) \circ \cdots \circ f(a + b_1 + b_2 + \cdots + b_{h-1}, b_h; p_h)$ 
to the sum $\sum_i (p_i-1) \otimes b_i$.
The function $\tilde{\rho}$ is a homomorphism;
for in computing $\tilde{\rho} (\bar{f} \circ f)$ 
we can arrange that $f$ and $\bar{f}$ are represented by sequences of the form
\begin{align*}
S &= (a; b_1, \cdots, b_h; p_1,\ldots, p_h),\\
\bar{S} &= (a; p_1b_1, \cdots , p_h b_h; q_1, \ldots, q_h).
\end{align*}
Then $\bar{f} \circ f$ can be represented by the sequence
\[
(a; b_1, \ldots, b_h; p_1q_1, \ldots,  p_hq_h),
\]
and the claim reduces to a very simple calculation.

Clearly, 
a sequence $S$ represents a bounded element 
if the sum 
\[
\sum\nolimits_i p_ib_i - \sum\nolimits_i b_i = \sum\nolimits_i (p_i-1)b_i
\] 
is equal to 0.
It follows that the subgroup $B([0,\infty[\,;A,P)$ maps under $\tilde{\rho}$ onto the kernel 
of the multiplication map $\mu \colon IP \otimes_{\Z{P}} A \epi IP \cdot A$, 
\ie{}onto the homology group $H_1(P,A)$.
The previous considerations establish therefore 
\begin{proposition}
\label{PropositionB5}
\index{Homology Theory of Groups!H1(P,A)@$H_1(P,A)$}%
The group $G_2 = \ker (\sigma_-, \sigma_+)$ is generated  by the homeomorphisms $f(a,b;p)$,
defined by equation \eqref{eq:8.5}, 
and the assignment 
\[
f(a,b;p) \mapsto (p-1)\otimes b
\] 
extends to an epimorphism $\tilde{\rho} \colon G_2 \epi IP \otimes_{\Z[P]} A$.

Let $\mu \colon IP \otimes_{\Z[P]} A  \epi IP \cdot A$ 
denote the multiplication map. 
The composition $\mu \circ \tilde{\rho}$ coincides then 
with $\rho {\restriction G_2} \colon G_2 \to IP \cdot A \incl \Aff(IP \cdot A, P)$.
In addition,
the subgroup $B([0,\infty[\, ; A,P)$ maps under $\tilde{\rho} $ 
onto the homology group $H_1(P,A) = \ker \mu$. 
\end{proposition}
%
\section{Generators for groups with supports in a compact interval}
\label{sec:9}
%
If the interval $I$ is compact 
we know of no general method for building up a PL-homeomorphism with many singularities 
from PL-homeomorphisms with fewer breaks. 
(This idea is the key to the results of Sections \ref{sec:7} and \ref{sec:8}.) 
However, 
if $P$ is generated by positive integers and $A$ is the subring $\Z[P]$ of the rationals, 
we can use equidistant subdivisions of $[a,c]$, 
and variations thereof, 
to define a procedure for obtaining elements with many singularities from simpler ones.

Our discussion starts with a comparison of generating sets for different intervals; 
for this comparison $A$ and $P$ can be arbitrary; 
we require, however, that the endpoints $a$ and $c$ lie in $A$. 
The existence of a comparison result is of interest 
since groups $G(I_1;A,P)$ and $G(I_2;A,P)$ with intervals $I_1$ and $I_2$ of different lengths
need not be isomorphic if $P$ is not cyclic (see Corollary \ref{CorollaryE12}).

\subsection{The Comparison Theorem}
\label{ssec:9.1}
%
In the proof of the comparison result
we shall make use of
\begin{lemma}
\label{LemmaB6}
Let $a$, $b$ and $c$ be elements of $A$ with $a < b < c$.
If $\PP$ is a generating set of $P$,
then $G([a,c];A,P)$ is generated by $G([a,b]; A,P)$ and $1 + \card \PP $ auxiliary elements.
(The set $\PP$ is allowed to be infinite.)
\end{lemma}

\begin{proof}
For each $p \in \PP$,
choose an element $g_p$ in $G = G([a,c];A,P)$ with $(g_p)'(c_-) = p$.
Then 
\[
\GG_c = \ker(\sigma_+ \colon G \to P) \cup \{ g_p \mid p \in \PP\}
\]
generates the group  $G = G([a,c];A,P)$.
Next, construct an element $h \in G$ so that $h(t) < t$ for each $t \in\, ](a+b)/2, c[$.
For every  $g \in \ker \sigma_+$ there exists then an exponent $m$ 
so that $h^m(\supp g)$ is contained in $]a,b[$\,,
whence $h^m \circ g \circ h^{-m} \in G([a,b]; A,P)$.
It follows that $G = G([a,c]; A,P)$ is generated by the set
\begin{equation*}
G([a,b];A,P) \cup \{h\} \cup \{g_p \mid p \in \PP\}. 
\end{equation*}
\end{proof}

As a first application of the previous lemma one has
\begin{theorem}
\label{TheoremB7}
\index{Group G([a,c];A,P)@Group $G([a,c];A,P)$!Comparison Theorem}%
\index{Group G([a,c];A,P)@Group $G([a,c];A,P)$!finite generation}%
\index{Group G([a,c];A,P)@Group $G([a,c];A,P)$!generating sets}%
\index{Finiteness properties of!G([a,c];A,P)@$G([a,c];A,P)$}%
Let $I_1$ and $I_2$ be compact intervals with endpoints in $A$.
Then the group $G(I_1;A,P)$ is finitely generated if, and only if, $G(I_2;A,P)$ is so.
This statement remains valid 
if the qualification ``finitely'' is replaced by ``countably''.
\end{theorem}

\begin{proof}
By Theorem \ref{TheoremA} and the density of $IP \cdot A$ in $\R$,
\index{Theorem \ref{TheoremA}!consequences}%
there exists an element $f \in G(\R;A,P)$ which maps $I_1$ 
onto an initial subinterval $[a_2, b_2]$ of $I_2 = [a_2,c_2]$.
If the group $G_1 = G(I_1;A,P)$ is generated by $\aleph$ elements,
so is its isomorphic copy $G([a_2, b_2]; A,P)$; 
by Lemma \ref{LemmaB6}  
the group $G_2 = G(I_2;A,P)$ is thus generated by $\aleph + (1 + \card \PP)$ elements. 
Here $\PP$ denotes a generating set of $P$;
if $G_1$ is finitely generated, 
$\PP$ can be chosen to be finite (see Proposition \ref{PropositionB1}),
if $G_1$ is countably generated, so is $P$.
As the rôles of $G_1$ and $G_2$ are interchangeable, 
the claim is established.
\end{proof}
%
\subsection{Definitions of $\PP$-regular and of $\PP$-standard subdivisions}
\label{ssec:9.2}
\index{Subdivisions!examples}
%
From now on we work in the special set-up mentioned 
in the preamble to Section \ref{sec:9}. 
Thus $P$ is generated by a set $\PP$ of natural numbers 
\label{notation:PP}%
and $A$ is the subring $\Z[P]$ of $\Q$.
In view of the Comparison Theorem (Theorem \ref{TheoremB7}) 
we shall assume that $I$ is the unit interval $[0,1]$. 
Our first aim is to construct a convenient set of generators of the group 
$G[\PP] = G( [ 0 ,1 ] ;\Z[\gp(\PP)] , \gp(\PP))$. 
\label{notation:G[PP]}%
\index{Group G[PP]@Group $G[\PP]$!definition|textbf}%

In the process of searching for such a set we shall adhere to 
\begin{convention}
\label{convention-in-9.2}
Given a sequence
\begin{equation}
\label{eq:9.1}
( (0,0), (b_1, b'_1), \ldots, (b_h, b'_h) , (1,1)) 
\end{equation} 
of points in the unit square $[0,1]^2$ with $0 < b_1 < \cdots < b_h < 1$,
let $\check{f} \colon [0,1] \to [0,1]$ denote the affine interpolation of the sequence \eqref{eq:9.1}
and let $f \colon \R \to \R$ be the extension of $\check{f}$ 
that is the identity outside of $[0,1]$.
\emph{By abuse of language, we shall call $f$, rather than $\check{f}$,
the affine interpolation of the sequence \eqref{eq:9.1}.}
\end{convention}

We proceed to the definition of a $\PP$-\emph{regular subdivision of} $[0,1]$; 
the definition will be recursive, the recursion parameter being called \emph{level}. 
A $\PP$-regular subdivision $D_1$ of level 1 is a finite (strictly) increasing sequence of the form
\[
(0,\Delta,2\Delta, . . . , (p-1)\Delta, 1),
\]
where $\Delta= 1/p$ for some $p \in \PP$.
If $D_{\ell-1} = (b_0 = 0, b_1,\ldots, b_h, b_{h+1}= 1)$ 
is a $\PP$-regular subdivision of level $\ell-1$,
each finite sequence $D_\ell$ obtained from $D_{\ell-1}$ 
by replacing a couple $(b_j$, $b_{j+1})$ of adjacent points in $D_{\ell-1}$
by the sequence
\[
(b_j, b_j + \Delta, b_j + 2 \Delta, \ldots, b_j+(p-1) \Delta, b_{j+1})
\]
with $\Delta  = p^{-1}(b_{j+1} - b_j)$ and $p \in \PP$ is,  
by definition,  
a $\PP$-regular subdivision of level $\ell$.
Stated more picturesquely, 
the subdivision $D_\ell$ is obtained from $D_{\ell-1}$ 
by subdividing the interval $[b_j,b_{j+1}]$ of $D_{\ell-1}$  into $p$ equal parts. 

The construction of a $\PP$-regular subdivision $D_\ell$ of level $\ell$ 
can be coded by a sequence of the form
\begin{equation}
\label{eq:9.2}
(1,p_1;n_2,p_2;n_3,p_3; \ldots ;n_\ell,p_\ell)
\end{equation}
where the couple $(n_i, p_i)$ specifies 
that in passing from level $i-1$ to level $i$ in the creation of $D_\ell$ 
the $n_i$-th interval $[b_{n_i-1} , b_{n_i}] $ of $D_{i-1}$ is subdivided into $p_i$ equal parts.
A subdivision $D$ is called $\PP$-\emph{regular} 
if it is $\PP$-regular of some level $\ell$.
\index{PP-regular subdivision@$\PP$-regular subdivision!definition|textbf}%

The $\PP$-\emph{standard} subdivisions are those $\PP$-regular subdivisions 
in whose formation process it is always the \emph{left most} interval that gets subdivided. 
A $\PP$-standard subdivision of level $\ell$ is thus coded 
by a sequence of the form
\begin{equation}
\label{eq:9.3}
\index{PP-standard subdivision@$\PP$-standard subdivision!definition|textbf}%
(1,p_1;1,p_2;1,p_3; \ldots ;1,p_\ell)
\end{equation}
The $\PP$-standard subdivision with parameters 
given by expression \eqref{eq:9.3}
will be denoted by $\St(p_1,p_2, \ldots, p_\ell)$.

\begin{illustration}
\label{illustration:Subdivisions}
\index{PP-regular subdivision@$\PP$-regular subdivision!examples}%
\index{PP-standard subdivision@$\PP$-standard subdivision!examples}%
\index{Rectangle diagram!examples}%
If $\PP = \{2\}$, the $\PP$-regular subdivisions of levels 1, 2 and 3 are as follows.
\psfrag{la1}{\hspace*{-8mm} \small  $\St(2)= (0,\tfrac{1}{2}, 1)$}
\psfrag{la2}{\hspace*{-14.3mm}  \small $\St(2,2)= (0,\tfrac{1}{4 }, \tfrac{1}{2}, 1)$}
\psfrag{la3}{\hspace*{0.7mm} \small  $(0,\tfrac{1}{2 },\tfrac{3}{4}, 1)$}
\psfrag{la4}{\hspace*{-27mm}  \small  $\St(2,2, 2)= (0,\tfrac{1}{8}, \tfrac{1}{4 }, \tfrac{1}{2}, 1)$}
\psfrag{la5}{\hspace*{-9mm}    \small                     $(0,\tfrac{1}{4 },\tfrac{3}{8 },\tfrac{1}{2 }, 1)$}
\psfrag{la6}{\hspace*{-9mm}    \small                      $(0,\tfrac{1}{4 },\tfrac{1}{2 },\tfrac{3}{4}, 1)$}
\psfrag{la7}{\hspace*{-9mm}    \small                      $(0,\tfrac{1}{2 },\tfrac{5}{8 },\tfrac{3}{4}, 1)$}
\psfrag{la8}{\hspace{-9mm}      \small                       $(0,\tfrac{1}{2 },\tfrac{3}{4 },\tfrac{7}{8}, 1)$}
\begin{equation*}
\includegraphics[width= 11cm]{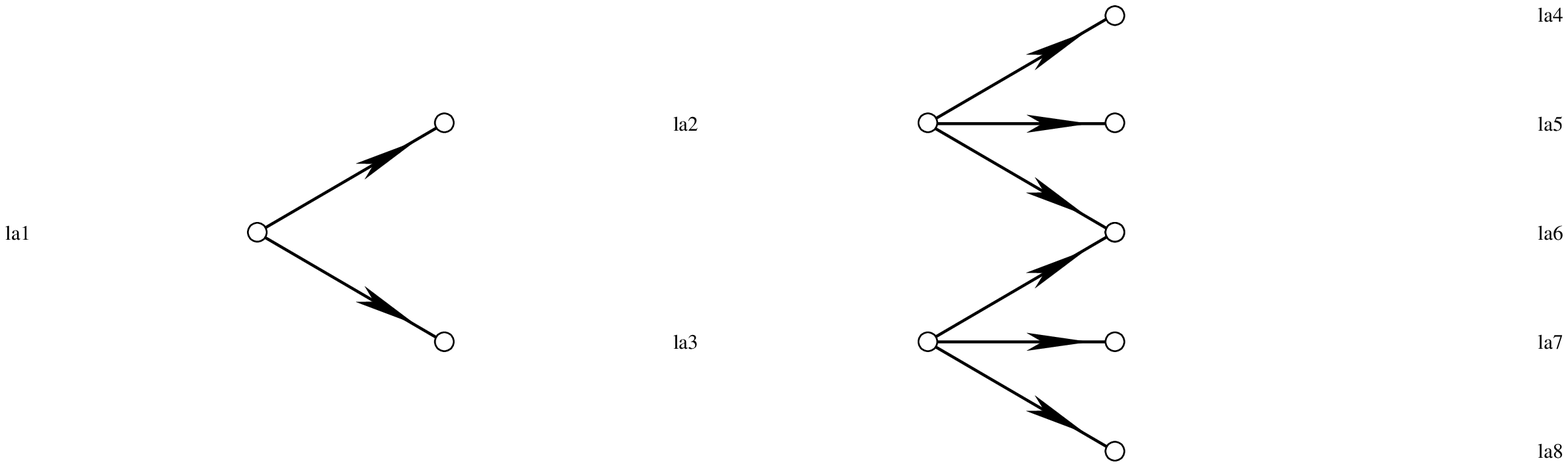}
 \end{equation*}
\end{illustration}
%
\subsection{Homeomorphisms associated to $\PP$-regular subdivisions}
\label{ssec:9.4}
%
As before, 
$\PP$ is a set of positive integers which generate the group $P$.
The reason why $\PP$-regular subdivisions are of interest is explained by
\begin{proposition}
\label{PropositionB8}
\index{PP-regular subdivision@$\PP$-regular subdivision!significance}%
\index{PP-regular subdivision@$\PP$-regular subdivision!associated homeomorphism}%
\index{Group G[PP]@Group $G[\PP]$!generating sets}%

Let
\[
D = (0 = b_0, b_1, \ldots, b_h, b_{h+1} = 1) 
\quad\text{and} \quad 
D' = (0 = b'_0, b'_1,\ldots,  b'_h, b'_{h+1} = 1)
\]
be two $\PP$-regular subdivisions with the same number of points.
Then the affine interpolation of the sequence of points
\begin{equation}
\label{eq:9.4}
D \boxtimes D' = \left((b_0, b'_0), (b_1, b'_1), \ldots, (b_h, b'_h), (b_{h+1}, b'_{h+1}) \right)
\end{equation}
is a PL-homeomorphism belonging to the group $G[\PP] = G([0,1]; \Z[\gp(\PP)], \gp(\PP))$.

Conversely,
every PL-homeomorphisms $f \in G[\PP]$ is the affine interpolation of some product $D \boxtimes D'$
of $\PP$-regular subdivisions of $[0,1]$.
\end{proposition}

\begin{proof}
If $D$ and $D'$ are $\PP$-regular subdivisions  
the lengths of the intervals $[b_i, b_{i+1}]$ and  $[b'_i, b'_{i+1}] $
 are numbers in $P = \gp(\PP)$
and so the affine map passing through the points $(b_i, b'_i)$ and $b_{i+1}, b'_{i+1})$ 
has slope in $P$.
It follows that the affine interpolation $\check{f}\colon [0,1] \iso [0,1]$ of $D \boxtimes D'$ 
has vertices in $\Z[P]^2$,
slopes in $P$ and that it maps $[0,1] \cap \Z[P]$ onto itself;
so it extends therefore uniquely to a PL-homeomorphism $f \in G[\PP]$.

Conversely, let $f$ be an element of $G[\PP]$.
A $\PP$-regular subdivision containing all the singularities of $f$ can then be found as follows.
To begin with,
let $\bar{p}_0$ be a common denominator of all the singularities of $f$.
Next choose a multiple  $p_0$ of $\bar{p}_0$ 
that belongs to the monoid $\md(\PP)$ 
\label{notation:md(XX)}
generated by the set $\PP$
and define $D$ to be the equidistant subdivision of the unit interval $[0,1]$ with $p_0 + 1$ points.
Then $D$ is a $\PP$-regular subdivision of $[0,1]$.
Let 
\[
f(D) = \{ f(b_0)= 0, f(b_1), \ldots, f(b_{p_0}), f(b_{p_0+1})= 1\}
\]
be the image of $D$ under $f$. 
Then $f$ is the affine interpolation of $D \boxtimes f(D)$ 
(extended by the identity outside of the unit interval)
and so $f$ is affine on each subinterval $[b_i,  b_{i+1}]$ with slope $s_i$, say.
The length of the  interval $[f(b_i), f(b_{i+1}]$  is thus $s_i (b_{i+1}-b_i)$ 
and this number lies in $P$.
Now, let $D'$ be a $\PP$-regular, equidistant subdivision of the interval $[0,1]$
which contains all the points of $f(D)$; 
let $\Delta$ be the step width of $D'$.
In passing from $f(D)$ to $D'$,
the interval $[f(b_i), f(b_{i+1})]$ is subdivided into $m_i$, say, equidistant subintervals of $D'$.
The number $m_i$ lies in $P$,
for it equals $s_i (b_{i+1}- b_i)/\Delta$.
Since each number in $P$ is a fraction of numbers belonging to the monoid $\md(\PP)$,
some $\md(\PP)$-multiple $q_i$ of $m_i$ belongs to $\md(\PP)$.
Subdivide now every subinterval of $[f(b_i), f(b_{i+1}]$ with step width $\Delta$
into $q_i$ equidistant intervals.
If this subdivision is carried out for each $i \in \{0,1,\ldots, h\}$,
one ends up with a $\PP$-regular subdivision $D''$ of $[0,1]$.
Its pre-image $f^{-1}(D'')$ is then $\PP$-regular.
Since $f$ is the affine interpolation of  $f^{-1}(D'') \boxtimes D''$, 
the contention is established.
\end{proof}

\begin{example}
\label{example:Search-for-regular-subdivisions}
\index{PP-regular subdivision@$\PP$-regular subdivision!examples}%
\index{PP-standard subdivision@$\PP$-standard subdivision!examples}%
The following example illustrates the search for $\PP$-regular subdivisions of the unit interval.
Suppose $\PP = \{4, 6\}$ and set $P = \gp(4,6)$.
Then  $\Z[P] = \Z[1/6]$. 
Note that $P \cap \N^\times_{>0}$ is the submonoid generated by 4, 6, and 9.

Let $f \colon \R \iso \R$ be the affine interpolation of
\[
\left((0,0), (\tfrac{1}{18}, \tfrac{1}{2}), (\tfrac{1}{2}, \tfrac{17}{18}), (1,1)\right ).
\]
Then $f$ has slopes $1$, $9$,  1, $\tfrac{1}{9}$, 1, 
and so it belongs the group $G([0,1]; \Z[1/6], \gp(4,6))$.

Let $D$ be the equidistant subdivision of $[0,1]$ with step width $\tfrac{1}{36}$.
This is a $\PP$-regular subdivision 
that contains the singularities 0,  $\tfrac{1}{18}$, $\tfrac{1}{2}$ and 1 of $f$;
it can be obtained by first subdividing the unit interval equidistantly into 6 intervals,
and then each of these 6 intervals into 6  subintervals of length $\tfrac{1}{36}$.

The image of $D$ under $f$ is the subdivision
\[
f(D) = \left(0, \tfrac{1}{4}, \tfrac{1}{2}, \tfrac{19}{36},  \tfrac{20}{36}, \ldots, \tfrac{34}{36}, 
\tfrac{34}{36} + \tfrac{1}{9 \cdot 36}, \ldots, 1- \tfrac{1}{9 \cdot 36}, 1 \right).
\]
Clearly, $f(D)$ is part of the equidistant subdivision of $[0,1]$ with step with $(9 \cdot 36)^{-1}$;
but as $9 \cdot 6^2$ does not lie in the monoid generated by $\PP = \{4,6\}$,
this subdivision is not $\PP$-regular. 
The equidistant subdivision $D'$ with step width $\Delta= (36)^{-2} = 6^{-4}$, however, is $\PP$-regular
and it contains all the points of $f(D)$.
Consider now the first subinterval $[b_0, b_1] = [0, \tfrac{1}{36}]$ of $D$ 
and its image $[f(b_0), f(b_1)] = [0, \tfrac{1}{4}]$ under $f$. 
In passing from $f(D)$ to $D'$ the interval $[0, \tfrac{1}{4}]$ is subdivided 
into intervals of length $\Delta= (36)^{-2}$ and so their number is $m_0 = 9 \cdot 36$.
The number $m_0$ lies in $P$, 
but it does not belong to the submonoid generated by $\PP = \{4, 6\}$.
We subdivide therefore each of the $m_0 = 9 \cdot 36$ 
small subintervals of $D'$ 
contained in $[0, \tfrac{1}{4}]$ into 4 equal parts, 
obtaining $q_0 = 36^2 = 6^4$ subintervals.
The resulting subdivision of $[f(b_0), f(b_1)] = [0, \tfrac{1}{4}]$
 is part of a $\PP$-regular subdivision of $D'$
and the preimages of these $6^4$ equidistant subintervals 
have lengths $\tfrac{1}{4} \cdot \tfrac{1}{36}$ 
and so they are part of a $\PP$-regular subdivision of $D$.

The preceding argument applies to each of the remaining 35 intervals of $f(D)$.
In the case of the last subinterval 
$[f(b_{35}), f(b_{36})] = [1- \tfrac{1}{9 \cdot 36}, 1]$, 
the details are as follows. 
In passing from $f(D)$ to $D'$ 
this interval is subdivided into 4 smaller subintervals.
Since 4 is a number in $\md(4, 6)$ 
the preimage of each of these intervals occurs in a suitable $\PP$-regular subdivision of $D$.
\end{example}
%
\subsection{A set of generators for $G[\PP] = G([0,1]; \Z[\gp(\PP)], \gp(\PP))$}
\label{ssec:9.5}
Our set of generators of $G[\PP]$ will involve four kinds of PL-homeo\-morphisms; 
they are defined next.
\subsubsection{Definitions of $f(q;p; r, p')$, 
$h(p_1, \ldots, p_\ell; p'_1, \ldots p'_k)$,  
$t(q;p,p')$ and $g(p_*,p)$}
\label{sssec:9.5.1}
%
We begin with affine interpolations of a $\PP$-regular and a $\PP$-standard  subdivision.
Let $(q,p,p')$ be a triple of elements in $\md(\PP) \times \PP \times \PP$,
with  $q = p_1 p_2 \cdots p_m$,
and let $r$ be an integer in $\{2,3, \ldots, p\}$.
Let $D$ to be the $\PP$-regular subdivision (of level $m + 2$) 
with parameter sequence
\[
(1,p_1; 1,p_2; \ldots; 1,p_m;  1, p; r,p')
\]
(see the lines following equation \eqref{eq:9.2} 
on page \pageref{eq:9.2}
for an explanation of this notation) 
and let $f(q;p;r,p')$ be the affine interpolation of   
\label{notation:f(q;p;r;p')}%
\[
D \boxtimes \St(p_1,p_2, \ldots,p_m, p, p').
\]
 As indicated by its name,
 the PL-homeomorphism $f(q;p; r, p')$ 
 does not depend on the order of the factors $p_i$ in  $q=p_1p_2 \cdots p_m$.
Note that the support of $f(q;p; r, p')$ is contained in the interval $[0, 1/q]$.

The remaining three kinds of PL-homeomorphisms 
are affine interpolations of two $\PP$-standard subdivisions.
Suppose $D = \St(p_1, \ldots, p_\ell)$ and $D' = \St(p'_1, \ldots k'_k)$
are $\PP$-standard subdivisions of $[0,1]$ with the same number of points;
this latter conditions holds, if and only if,
\begin{equation}
\label{eq:9.5}
2 +(p_1 -1) + \cdots +  (p_\ell -1)= 2 + (p'_1 -1) + \cdots  + (p'_k-1). 
\end{equation}
Define now $h(p_1, \ldots, p_\ell; p'_1, \ldots p'_k)$ to be the affine interpolation 
of $D \boxtimes D'$. 
\label{notation:h(p1,...,p'k)}

Two special cases of the type of PL-homeomorphism just defined 
will play an important rôle in the sequel.
The first of them is a kind of transposition.
Let $q = p_1 \cdots p_m$ be a product of elements in $\PP$ 
and assume $p$ and $p'$ are distinct elements of $\PP$. 
Set
\begin{equation}
\label{eq:9.6}
t(q;p,p') = h(p_1, \ldots, p_m, p, p'; p_1, \ldots, p_m, p', p).
\end{equation}
As is indicated by its name, 
the PL-homeomorphism $t(q;p,p')$ depends only on the triple $(q,p,p')$,
but not on the representation of $q$ as a product with factors in $\PP$.

The second special case involves two $\PP$-standard subdivisions of $[0,1]$
in each of which occurs only a single element of $\PP$.
Let $p_*$ be the smallest element of $\PP$.
\label{notation:p*}
(We assume, as we may, that 1 does not belong to $\PP$.)
Given $p \in \PP \smallsetminus \{p_*\}$, we put
\begin{equation}
\label{eq:9.7}
g(p_*, p) = h(p_*,\ldots, p_*; p, \cdots, p);
\end{equation}
here there are $p-1$ consecutive occurrences of $p_*$, 
followed by $p_*-1$ consecutive occurrences of $p$.
(Formula \eqref{eq:9.5} shows that this definition is licit.)
\begin{figure}
%
%
%
\begin{minipage}{12 cm}
%
%
\psfrag{1}{\hspace*{-1.7mm}   \footnotesize 0}
\psfrag{2}{\hspace*{-1.6mm}  \footnotesize $\tfrac{1}{3}$}
\psfrag{3}{\hspace*{-1.8mm} \footnotesize $\tfrac{1}{2}$}
\psfrag{4}{\hspace*{-0.5mm}\footnotesize  $\tfrac{2}{3}$}
\psfrag{5}{\hspace*{-1.1mm}  \footnotesize  1}
\psfrag{11}{\hspace*{-0.7mm}  \footnotesize  $0$}
\psfrag{12}{\hspace*{-0,9 mm}  \footnotesize$\tfrac{1}{6}$}
\psfrag{13}{\hspace*{-1.0mm}  \footnotesize $\tfrac{1}{3}$}
\psfrag{14}{\hspace*{-0.9mm}  \footnotesize$\tfrac{2}{3}$}
\psfrag{15}{\hspace*{-0.7mm}  \footnotesize $1$}
\psfrag{21}{\hspace*{-1.7mm} \footnotesize $\tfrac{1}{2}$}
\psfrag{22}{  \hspace*{-1.0mm}\footnotesize  1}
\psfrag{23}{\hspace*{-1mm} \footnotesize 2}
\psfrag{24}{\hspace*{-1.0mm} \footnotesize 1}
%
%
\psfrag{31}{\hspace*{-0.6mm}    \footnotesize $0$}
\psfrag{32}{\hspace*{-1.0mm}   \footnotesize $\tfrac{1}{3}$}
\psfrag{33}{\hspace*{-0.9mm}  \footnotesize $\tfrac{2}{3}$}
\psfrag{34}{\hspace*{-0.9mm} \footnotesize  $\tfrac{5}{6}$}
\psfrag{35}{\hspace*{-0.6mm}   \footnotesize  $1$}
\psfrag{41}{\hspace*{-0.2mm}  \footnotesize  $0$}
\psfrag{42}{\hspace*{-0.6mm}  \footnotesize$\tfrac{1}{6}$}
\psfrag{43}{\hspace*{-0.6mm}  \footnotesize $\tfrac{1}{3}$}
\psfrag{44}{\hspace*{-0.6mm}  \footnotesize$\tfrac{2}{3}$}
\psfrag{45}{\hspace*{-0.4mm}  \footnotesize $1$}
\psfrag{51}{\hspace*{-1.7mm}  \footnotesize $\tfrac{1}{2}$}
\psfrag{52}{  \hspace*{-2.5mm} \footnotesize  $\tfrac{1}{2}$}
\psfrag{53}{\hspace*{-1mm}  \footnotesize 2}
\psfrag{54}{\hspace*{-1.0mm}  \footnotesize 2}
%
%
\psfrag{61}{\hspace*{-0.5mm}   \footnotesize $0$}
\psfrag{62}{\hspace*{-0.7mm}  \footnotesize $\tfrac{1}{6}$}
\psfrag{63}{\hspace*{-0.9mm} \footnotesize $\tfrac{1}{3}$}
\psfrag{64}{\hspace*{-0.6 mm}\footnotesize   $\tfrac{5}{12}$}
\psfrag{65}{\hspace*{-1.0mm}  \footnotesize $\tfrac{1}{2}$}
\psfrag{66}{\hspace*{-0.6mm}  \footnotesize  $1$}
\psfrag{71}{\hspace*{-0.4mm} \footnotesize  $0$}
\psfrag{72}{\hspace*{-1.2mm} \footnotesize $\tfrac{1}{12}$}
\psfrag{73}{\hspace*{-0.4mm} \footnotesize $\tfrac{1}{6}$}
\psfrag{74}{\hspace*{-0.7mm} \footnotesize $ \tfrac{1}{3}$}
\psfrag{75}{\hspace*{-0.7mm}  \footnotesize $\tfrac{1}{2}$}
\psfrag{76}{\hspace*{-0.4mm} \footnotesize$1$}
\psfrag{81}{\hspace*{-0.8mm} \footnotesize $\tfrac{1}{2}$}
\psfrag{82}{  \hspace*{-1.3mm}\footnotesize  $\tfrac{1}{2}$}
\psfrag{83}{\hspace*{-0.5mm} \footnotesize 2}
\psfrag{84}{\hspace*{-0.5mm} \footnotesize 2}
\psfrag{85}{  \hspace*{-1.0mm}\footnotesize 1}
\psfrag{la1}{\hspace{-1mm}$f_1$}
\psfrag{la2}{\hspace{-2.5mm} $f_2$}
\psfrag{la3}{\hspace{-2mm}$f_3$}
\begin{gather*}
\includegraphics[width= 5.5cm]{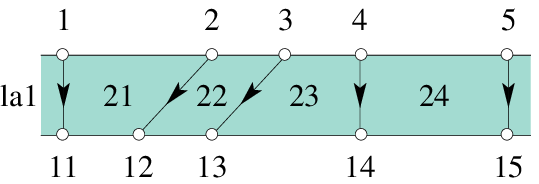}
\hspace*{6mm}
\includegraphics[width= 5.5cm]{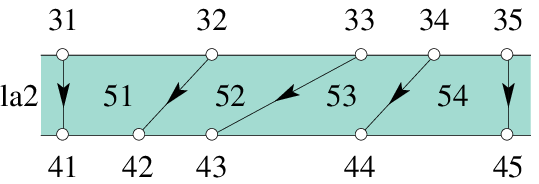}
\\
\includegraphics[width= 5.5cm]{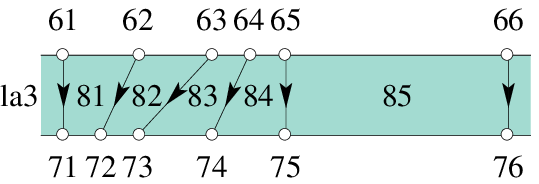}
\end{gather*}
\end{minipage}
%
\begin{minipage}{12cm}
%
%
\psfrag{1}{\hspace*{-1.4mm}   \small 0}
\psfrag{2}{\hspace*{-2.3mm}  \small $\tfrac{1}{15}$}
\psfrag{3}{\hspace*{-2.4mm} \small $\tfrac{2}{15}$}
\psfrag{4}{\hspace*{-1.2mm}\small  $\tfrac{3}{15}$}
\psfrag{5}{\hspace*{-2.5mm}  \small  $\tfrac{4}{15}$}
\psfrag{6}{\hspace*{-1.6mm}  \small  $\tfrac{1}{3}$}
\psfrag{7}{\hspace*{-1.7mm}  \small  $\tfrac{2}{3}$}
\psfrag{8}{\hspace*{-1.4mm}  \small  1}
\psfrag{11}{\hspace*{-0.7mm}  \small  $0$}
\psfrag{12}{\hspace*{-0.9mm}  \small$\tfrac{1}{7}$}
\psfrag{13}{\hspace*{-0.9mm}  \small $\tfrac{2}{7}$}
\psfrag{14}{\hspace*{-0.9mm}  \small$\tfrac{3}{7}$}
\psfrag{15}{\hspace*{-1.0mm}   \small $\tfrac{4}{7}$}
\psfrag{16}{\hspace*{-0.8mm}   \small$\tfrac{5}{7}$}
\psfrag{17}{\hspace*{-1.0mm}   \small $\tfrac{6}{7}$}
\psfrag{18}{\hspace*{-0.5mm}   \small $1$}
\psfrag{21}{\hspace*{-1.8mm}  \small $\tfrac{15}{7}$}
\psfrag{22}{\hspace*{-2.0mm} \small $\tfrac{15}{7}$}
\psfrag{23}{\hspace*{-1.8mm}   \small $\tfrac{15}{7}$}
\psfrag{24}{\hspace*{-2.0mm}  \small $\tfrac{15}{7}$}
\psfrag{25}{\hspace*{-1.0mm}  \small $\tfrac{15}{7}$}
\psfrag{26}{\hspace*{1.3mm}  \small $\tfrac{3}{7}$}
\psfrag{27}{\hspace*{1.5mm}  \small $\tfrac{3}{7}$}
\psfrag{laH357}{\hspace{-8mm}\small $h(3,5;7)$}
\begin{equation*}
\includegraphics[width= 10.5cm]{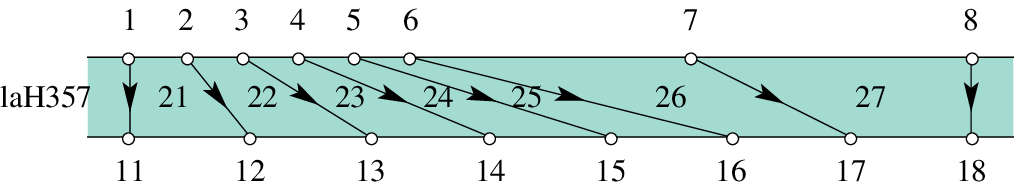}
\end{equation*}
\end{minipage}
%
%
\begin{minipage}{12 cm}
%
%
\psfrag{1}{\hspace*{-1.7mm}   \small 0}
\psfrag{2}{\hspace*{-2.5mm}  \small $\tfrac{1}{15}$}
\psfrag{3}{\hspace*{-2.5mm} \small $\tfrac{2}{15}$}
\psfrag{4}{\hspace*{-1.0mm}\small  $\tfrac{3}{15}$}
\psfrag{5}{\hspace*{-1.9mm}  \small  $\tfrac{4}{15}$}
\psfrag{6}{\hspace*{-1.8mm}  \small  $\tfrac{1}{3}$}
\psfrag{7}{\hspace*{-1.7mm}  \small  $\tfrac{2}{3}$}
\psfrag{8}{\hspace*{-1.5mm}  \small  1}
\psfrag{11}{\hspace*{-0.7mm}  \small  $0$}
\psfrag{12}{\hspace*{-1.7mm}  \small$\tfrac{1}{15}$}
\psfrag{13}{\hspace*{-1.7mm}  \small $\tfrac{2}{15}$}
\psfrag{14}{\hspace*{-1.0mm}  \small$\tfrac{1}{5}$}
\psfrag{15}{\hspace*{-1.0mm}   \small $\tfrac{2}{5}$}
\psfrag{16}{\hspace*{-0.8mm}   \small$\tfrac{3}{5}$}
\psfrag{17}{\hspace*{-1.0mm}   \small $\tfrac{4}{5}$}
\psfrag{18}{\hspace*{-0.7mm}   \small $1$}
\psfrag{21}{\hspace*{-0.8mm}  \small 1}
\psfrag{22}{\hspace*{-0.8mm} \small 1}
\psfrag{23}{\hspace*{-1mm}   \small 1}
\psfrag{24}{\hspace*{-2.0mm}  \small 3}
\psfrag{25}{\hspace*{-2.0mm}  \small 3}
\psfrag{26}{\hspace*{2.0mm}  \small $\tfrac{3}{5}$}
\psfrag{27}{\hspace*{-1.0mm}  \small $\tfrac{3}{5}$}
%
%
\psfrag{31}{\hspace*{-1.0mm}   \small 0}
\psfrag{32}{\hspace*{-1.5mm}  \small $\tfrac{1}{30}$}
\psfrag{33}{\hspace*{-2.7mm} \small }
\psfrag{34}{\hspace*{-0.5mm}\small  $\tfrac{1}{10}$}
\psfrag{35}{\hspace*{-1.7mm}  \small }
\psfrag{36}{\hspace*{-1.0mm}  \small  $\tfrac{1}{6}$}
\psfrag{37}{\hspace*{-0.8mm}  \small  $\tfrac{1}{3}$}
\psfrag{38}{\hspace*{-1.0mm}  \small  $\tfrac{1}{2}$}
\psfrag{39}{\hspace*{-0.5mm}  \small  1}

\psfrag{41}{\hspace*{-0.7mm}  \small  $0$}
\psfrag{42}{\hspace*{-1.5mm}  \small$\tfrac{1}{30}$}
\psfrag{43}{\hspace*{-1.6mm}  \small }
\psfrag{44}{\hspace*{-1.3mm}  \small$\tfrac{1}{10}$}
\psfrag{45}{\hspace*{-1.2mm}   \small $\tfrac{2}{10}$}
\psfrag{46}{\hspace*{-1.2mm}   \small$\tfrac{3}{10}$}
\psfrag{47}{\hspace*{-1.3mm}   \small $\tfrac{4}{10}$}
\psfrag{48}{\hspace*{-0.8mm}  \small  $\tfrac{1}{2}$}
\psfrag{49}{\hspace*{-0.3mm}  \small  1}
\psfrag{51}{\hspace*{-0.3mm}  \small 1}
\psfrag{52}{\hspace*{-0.3mm} \small 1}
\psfrag{53}{\hspace*{-0.3mm}   \small 1}
\psfrag{54}{\hspace*{-0.7mm}  \small 3}
\psfrag{55}{\hspace*{-1.0mm}  \small 3}
\psfrag{56}{\hspace*{-1.0mm}  \small $\tfrac{3}{5}$}
\psfrag{57}{\hspace*{-1.0mm}  \small $\tfrac{3}{5}$}
\psfrag{58}{\hspace*{-1.0mm}  \small 1}
\psfrag{laT135}{\small\hspace{-7mm}$t(1;3,5)$}
\psfrag{laT235}{\small \hspace{-7.5mm} $t(2;3,5)$}
\begin{gather*}
\includegraphics[width= 10.5cm]{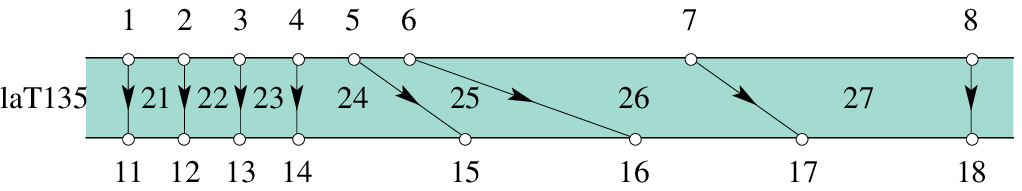}
\\
\includegraphics[width= 10.5cm]{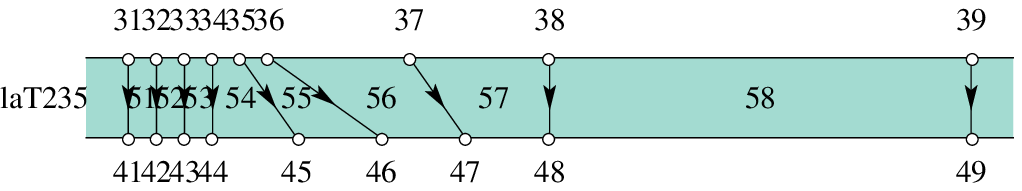}
\end{gather*}
\end{minipage}
%
%
\begin{minipage}{12cm}
%
%
\psfrag{1}{\hspace*{-1.5mm}   \small 0}
\psfrag{2}{\hspace*{-1.5mm}  \small $\tfrac{1}{4}$}
\psfrag{3}{\hspace*{-1.6mm} \small $\tfrac{1}{2}$}
\psfrag{4}{\hspace*{-1.2mm}  \small  1}
\psfrag{11}{\hspace*{-0.7mm}  \small  $0$}
\psfrag{12}{\hspace*{-0.9mm}  \small$\tfrac{1}{3}$}
\psfrag{13}{\hspace*{-0.9mm}  \small $\tfrac{2}{3}$}
\psfrag{14}{\hspace*{-0.6mm}   \small $1$}
\psfrag{21}{\hspace*{-1.7mm}  \small $\tfrac{4}{3}$}
\psfrag{22}{\hspace*{-1.5mm} \small $\tfrac{4}{3}$}
\psfrag{23}{\hspace*{-0.0mm}   \small $\tfrac{2}{3}$}
%
%
\psfrag{31}{\hspace*{-0.7mm}   \small 0}
\psfrag{32}{\hspace*{-1.6mm}  \small $\tfrac{1}{16}$}
\psfrag{33}{\hspace*{-0.8mm} \small $\tfrac{1}{8}$}
\psfrag{34}{\hspace*{0.35mm}\small  $\tfrac{1}{4}$}
\psfrag{35}{\hspace*{-0.7mm}  \small  $\tfrac{1}{2}$}
\psfrag{36}{\hspace*{-0.7mm}  \small  1}
\psfrag{41}{\hspace*{-0.2mm}  \small  $0$}
\psfrag{42}{\hspace*{-0.6mm}  \small$\tfrac{1}{5}$}
\psfrag{43}{\hspace*{-0.7mm}  \small $\tfrac{2}{5}$}
\psfrag{44}{\hspace*{-0.8mm}  \small$\tfrac{3}{5}$}
\psfrag{45}{\hspace*{-0.6mm}   \small $\tfrac{4}{5}$}
\psfrag{46}{\hspace*{-0.4mm}   \small $1$}
\psfrag{51}{\hspace*{-1.5mm}  \small $\tfrac{16}{5}$}
\psfrag{52}{\hspace*{-1.7mm} \small $\tfrac{16}{5}$}
\psfrag{53}{\hspace*{-0.8mm}   \small $\tfrac{8}{5}$}
\psfrag{54}{\hspace*{-0.0mm}  \small $\tfrac{4}{5}$}
\psfrag{55}{\hspace*{1.5mm}  \small $\tfrac{2}{5}$}
%
%
\psfrag{61}{\hspace*{-0.8mm}   \small 0}
\psfrag{62}{\hspace*{-2.7mm}  \small }
\psfrag{63}{\hspace*{-2.7mm} \small }
\psfrag{64}{\hspace*{-0.4mm}\small  $\tfrac{1}{27}$}
\psfrag{65}{\hspace*{-1.7mm}  \small  }
\psfrag{66}{\hspace*{-0.9mm}  \small  $\tfrac{1}{9}$}
\psfrag{67}{\hspace*{-1.2mm}  \small  $\tfrac{2}{9}$}
\psfrag{68}{\hspace*{-1.0mm}  \small  $\tfrac{1}{3}$}
\psfrag{69}{\hspace*{-1.0mm}  \small  $\tfrac{2}{3}$}
\psfrag{70}{\hspace*{-0.6mm}  \small  1}
\psfrag{71}{\hspace*{-0.7mm}  \small  $0$}
\psfrag{72}{\hspace*{-1.2mm}  \small$\tfrac{1}{25}$}
\psfrag{73}{\hspace*{-1.6mm}  \small }
\psfrag{74}{\hspace*{-1.6mm}  \small}
\psfrag{75}{\hspace*{-1.6mm}   \small }
\psfrag{76}{\hspace*{-0.4mm}   \small$\tfrac{1}{5}$}
\psfrag{77}{\hspace*{-0.6mm}   \small $\tfrac{2}{5}$}
\psfrag{78}{\hspace*{-0.6mm}   \small $\tfrac{3}{5}$}
\psfrag{79}{\hspace*{-0.6mm}   \small $\tfrac{4}{5}$}
\psfrag{80}{\hspace*{-0.5mm}   \small $1$}
\psfrag{81}{\hspace*{-2.2mm}  \small }
\psfrag{82}{\hspace*{-2.5mm}  \small }
\psfrag{83}{\hspace*{-2mm}   \small }
\psfrag{84}{\hspace*{-2.0mm}  \small }
\psfrag{85}{\hspace*{-1.5mm}  \small $\tfrac{27}{25}$}
\psfrag{86}{\hspace*{-1.05mm}  \small $\tfrac{9}{5}$}
\psfrag{87}{\hspace*{-1.7mm}  \small $\tfrac{9}{5}$}
\psfrag{88}{\hspace*{2.0mm}  \small $\tfrac{3}{5}$}
\psfrag{89}{  \hspace*{-1.2mm}\small $\tfrac{3}{5}$}
\psfrag{la2G3}{\hspace{-7mm} $g(2,3)$}
\psfrag{la2G5}{\hspace{-7mm} $g(2,5)$}
\psfrag{la3G5}{\hspace{-7mm} $g(3,5)$}
\begin{center}
\includegraphics[width= 5.2cm]{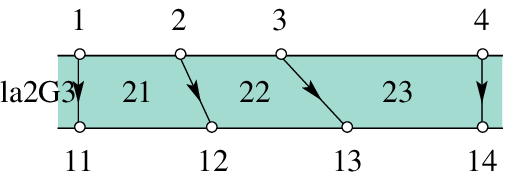}
\hspace*{8mm}
\includegraphics[width= 5.2cm]{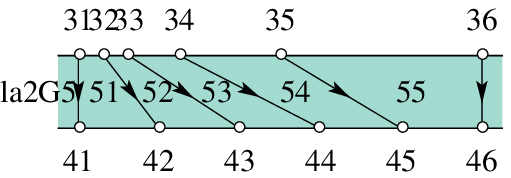}\par
\vspace*{2mm} \hspace*{5mm}
\includegraphics[width= 10.6cm]{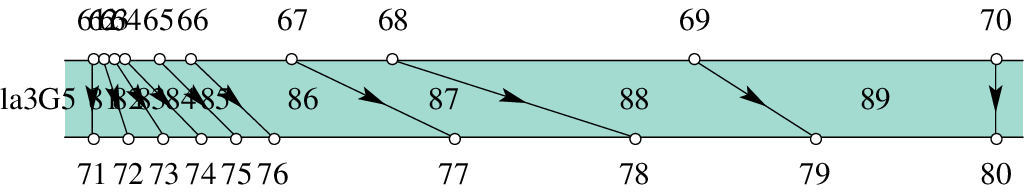}
\end{center}
\end{minipage}
\caption{PL-homeomorphisms of type $f$, $h$, $t$ and $g$}
\label{fig:Specimens-of-f-g-h-t}
\index{Rectangle diagram!examples}%
\index{PP-regular subdivision@$\PP$-regular subdivision!examples}%
\end{figure}

\begin{illustration}
\label{illustration:Prospective-generators}
The four kinds of PL-homeomorphisms just defined will be used repeatedly 
in the remainder of Section \ref{sec:9}.
In order to familiarize the reader with them, 
we exhibit some specimens. 
To do so, we use a graphical notation for PL-homeomorphisms,
called \emph{rectangle diagrams};
\index{Rectangle diagram!definition|textbf}%
these diagrams have already been used in section \ref{ssec:1.1}.
A rectangle diagram describes the affine interpolation $f$ of a sequence of points $(t_i,t'_i)$ 
by a rectangle
whose top and bottom lines are subdivided according to the coordinates $t_i$ and $t_i'$, respectively;
moreover, 
each point $t_i$ on the top line is joined to its image $t'_i$ on the bottom line
by a straight edge decorated by an arrow.
In the trapezoids bounded by these edges and bits of the top and bottom lines, 
the slopes can be inscribed.
\smallskip

\index{PP-regular subdivision@$\PP$-regular subdivision!examples|(}%
\index{PP-standard subdivision@$\PP$-standard subdivision!examples|(}%
\emph{1. PL-homeomorphisms $f$.}
These homeomorphisms are affine interpolations of a $\PP$-regular subdivision
and an associated $\PP$-standard subdivision of $[0,1]$. 
Three PL-homeomorphisms with $\PP = \{ 2,3\}$ are displayed in Figure
\ref{fig:Specimens-of-f-g-h-t},
namely the functions $f_1 = f(1,3; 2,2)$, $f_2 = f(1,3;3,2)$ 
and $f_3 = f(2;1,3; 3,2)$.

\emph{2. PL-homeomorphisms $h$.}
These homeomorphisms are affine interpolations of two $\PP$-standard subdivisions of $[0,1]$.
The two subdivisions are only restricted by the condition 
that they comprise the same number of intervals.
We illustrate this type of PL-homeomorphisms by the function $h(3,5;7)$;
it is the affine interpolation of $\St(3;5) \boxtimes \St(7)$.

\emph{3. PL-homeomorphisms $t$.}
This type of PL-homeomorphisms is a kind of transposition;
it permits one to exchange the order of subdivisions in $\PP$-standard subdivisions.
Every permutation of the set $\{1,2, \ldots, n\}$ 
is a composition of transpositions of adjacent numbers;
similarly, in order to express an arbitrary permutation of $\PP$-standard subdivisions
it suffices to consider PL-homeomorphisms of the form $t(q;p,p') = h(q;p,p')$.
In Figure \ref{fig:Specimens-of-f-g-h-t},
the rectangle diagrams of type $t$ depict transpositions with $q = 1$ and $q = 2$, respectively, 
and with $p = 3$, $p' = 5$.

\emph{4. PL-homeomorphisms $g$.}
The last kind of PL-homeomorphisms allows one to trade 
a $\PP$-standard subdivision involving only one element $p \in \PP$ 
for a $\PP$-standard subdivision
involving only the smallest element  $p_* \in \PP$.
In Figure \ref{fig:Specimens-of-f-g-h-t},
we display three such PL-homeomorphisms, 
with $p_* = 2$ for the first two functions
and with $p_* = 3$ for the third one.
\index{Rectangle diagram!examples}%
\index{PP-regular subdivision@$\PP$-regular subdivision!examples|)}%
\index{PP-standard subdivision@$\PP$-standard subdivision!examples|)}%
\end{illustration}

\begin{remark}
\label{remark:Number-of-singularities}
The rectangle diagrams of the PL-homeomorphisms  
displayed in Figure \ref{fig:Specimens-of-f-g-h-t}
have sometimes fewer singularities 
than the number of points in the subdivisions would lead one to expect.
We shall come back to this phenomenon in section \ref{sssec:preimages-yi}. 
\end{remark}
%
\subsubsection{Definition of a prospective set of generators}
\label{sssec:9.5.2}
We are now set for the definition of a prospective set of generators for the group $G[\PP]$.
This set will be the union of four subsets $\Fcal$, $\GG$, $\TT$, and $\RR$,
and it is defined as follows:
\begin{align*}
\Fcal       &= \{ f(q;p;r,p') \mid q \in \{1\} \cup \PP, (p,p') \in \PP^2 \text{ and } 1< r \leq p \};\\
\GG     &= \{ g(p_*,p)         \mid p \in \PP \smallsetminus \{p_*\}\, \}; \\
\TT      &= \{t(q;p,p') \mid q \in \{1\} \cup \PP, (p,p') \in \PP^2 \text{ and } p < p' \};
\intertext{and}
\RR &= \{ h(p_1,\ldots, p_\ell; p'_1, \ldots, p'_k) \mid  p_i \text{ and } p'_j \text{ satisfying a), b) and  c) } \}.
\end{align*}
The conditions referred to in the definition of the subset $\RR$ 
are the following ones:
\begin{enumerate}[a)]
\item $p_1 \leq p_2 \leq \cdots \leq p_\ell$ as well as $p'_1 \leq p'_2 \leq \cdots \leq p'_k$, 
and $p_1 < p'_1$;
\item the sets $\{p_i \mid 1 \leq i \leq \ell \}$ 
and $\{p'_j \mid 1 \leq j \leq k \}$ are disjoint;
\item if an element $p$ of $\PP$ occurs in $(p_1, p_2, \ldots, p_\ell)$ 
or in $(p'_1, \ldots, p'_k)$ at least $(p_*-1)$ times  then $p = p_*$.
\footnote{If $p_* = 2$ condition c) says 
that the members of $(p_1, \ldots, p_\ell)$ and $(p'_1, \ldots, p'_ k)$ are all equal to $p_*$,
whence condition b) implies that neither of the sequences has any member. 
The subset $\RR$ is thus empty.}
\end{enumerate}
Note that the subset $\RR$ is disjoint from $\Fcal \cup \GG \cup \TT$.
We shall prove
\begin{theorem}
\label{TheoremB9}
\index{Group G[PP]@Group $G[\PP]$!generating sets}%
\index{PP-regular subdivision@$\PP$-regular subdivision!significance}%
Assume $\PP$ is a set of integers greater than 1.
Then the group $G[\PP] = G([0,1]; \Z[\gp(\PP)], \gp(\PP))$ is generated by the set 
$\Fcal \cup \GG \cup \TT \cup \RR$.
\end{theorem}
The theorem will be established in sections \ref{ssec:9.6} through \ref{ssec:9.8}.
\smallskip

Consider now the special case where $\PP$ is \emph{finite}.
The subsets $\Fcal$, $\GG$ and $\TT$ are then obviously finite,
while $\RR$ is finite in view of conditions b) and c).
\footnote{Of course, condition \eqref{eq:9.5} is also required to hold.}
Theorem \ref{TheoremB9}, in conjunction with Theorem \ref{TheoremB7},
thus entails
\begin{corollary}
\label{CorollaryB10}
\index{Group G[PP]@Group $G[\PP]$!finite generation}%
\index{Finiteness properties of!G([a,c];A,P)@$G([a,c];A,P)$}%
If $\PP$ is a finite set of integers 
and $a$, $c$ are elements of  $\Z[\gp(\PP)]$, 
the group $G([a,c]; \Z[\gp(\PP)], \gp(\PP))$ is finitely generated.
\end{corollary}
%
\subsection{A reduction}
\label{ssec:9.6}
%
Let $f$ be an element of the group 
\[
G[\PP] = G([0,1];\Z[\gp(\PP)], \gp(\PP)).
\]
By Proposition \ref{PropositionB8} there exist then $\PP$-regular subdivisions $D$ and $D'$ of $[0,1]$ 
so that $f$, restricted to the unit interval, 
is the affine interpolation of $D \boxtimes D'$.
Suppose $D$ and $D'$ are described by the sequences
\[
(1,p_1; n_2, p_2; \ldots; n_\ell, p_\ell) \quad \text{and} \quad (1,p'_1; n'_2, p'_2; \ldots; n'_k,  p'_k),
\]
respectively, and let $D_0$ and $D'_0$ be the  $\PP$-standard subdivisions 
$\St(p_1, \ldots, p_\ell)$ and $\St(p'_1, \ldots, p'_k)$.
Then $D$, $D_0$,  $D'$ and $D'_0$ all have the same number of subintervals.
We can thus construct the affine interpolations of $D \boxtimes D_0$, 
of $D_0 \boxtimes D'_0$ and of $D' \boxtimes D'_0$;
call them $g$, $h = h(p_1, \cdots, p_\ell; p'_1, \ldots, p'_k)$ and $\bar{g}$.
By Proposition \ref{PropositionB8} these elements belong to $G[\PP]$; 
moreover, $f = \bar{g} ^{-1} \circ h \circ g$.

It suffices therefore to prove 
that each interpolation of a product $D \boxtimes D_0$ 
of a $\PP$-regular subdivision $D$,
say described by the sequence $(1,p_1; n_2,p_2; \ldots; n_\ell, p_\ell)$,
and the $\PP$-standard subdivision $\St(p_1, \ldots, p_\ell)$,  is an element of the group
\[
G_0 = \gp( \Fcal \cup \GG \cup \TT \cup \RR ),
\]
and that each PL-homeomorphism $h(p_1, \ldots, p_\ell; p'_1, \ldots, p'_k)$ belongs to $G_0$.
These verifications will be carried out in sections \ref{ssec:9.7} and \ref{ssec:9.8}.
Prior to moving on to them, we pause for an observation.

If $p_1$ and $\bar{p}$ are elements of $\PP$,
the PL-homeomorphism $f(1; p_1; p_1, \bar{p})$ is linear on $[0,1/p_1]$ with slope $1/\bar{p}$.
The PL-homeomorphisms $f(q;p;r,p')$ and $t(q;p,p')$,
on the other hand,
have supports contained in $[0, 1/q]$.
A glance at the definitions of these elements will then reveal 
that the relations
\begin{align*}
\act{f(1;p_1;p_1, \bar{p})} f(q;p; r, p') &= f(\bar{p} \cdot q; p;r, p')\\
\intertext{and}
\act{f(1;p_1; p_1, \bar{p})} t(q;p,  p') &=  t(\bar{p} \cdot q; p, p')
\end{align*}
are valid whenever $p_1 \leq q$.
It follows that  $G_0$ contains the subsets
\begin{align}
\Fcal^\sharp &= \{f(q;p;r, p') \mid q \in \md(\PP), \; (p,p') \in \PP^2 \text{ and } 1 < r \leq p \} 
\label{eq:FF-sharp}\\
\intertext{and}
\TT^\sharp &= \{t(q;p,p') \mid q \in \md(\PP), \; (p,p') \in \PP \text{ and } p < p' \}.
\label{eq:9.8}
\end{align}
%
\subsection[{Interpolations of $\PP$-standard subdivisions}]%
{Affine interpolations of $\PP$-standard subdivisions}
\label{ssec:9.7}
As before,  let $G_0$ denote the group generated by the subset $\Fcal \cup \GG \cup \TT \cup \RR$.
In section \ref{ssec:9.6},
the claim of Theorem \ref{TheoremB9} has been reduced to two verifications:
it suffices to prove
that each interpolation of a product $D \boxtimes D_0$  of a $\PP$-regular subdivision $D$
and an associated $\PP$-standard subdivision lies in $G_0$,
and that every PL-homeomorphisms $h(p_1, \ldots, p_\ell; p'_1, \ldots, p'_k)$ belongs to $G_0$

In this section the second claim will be established.
It is a consequence of
\begin{lemma}
\label{LemmaB11}
The group 
\[
H = \gp\left( \left\{ h(p_1,\ldots, p_\ell;p'_1,\ldots, p'_k) 
\mid 
\text{each } p_i \text{ and each } p'_j \text{ an element of } \PP \right\}\right)
\]
is generated by the subset $\GG \cup \TT^\sharp \cup \RR$.
\end{lemma}

\begin{proof}
Note first 
that the set $\GG \cup \TT^\sharp \cup \RR$ is a subset of $H$
and
that the subgroup $\gp( \TT^\sharp)$ contains all the ``permutations''
\[
h(p_1, \ldots, p_\ell; p_{\pi(1)}, \ldots, p_{\pi(\ell)})
\]
where $\ell \geq 2$, each $p_i$ is in $\PP$ and $\pi$ 
is a permutation of $\{1, 2,  \ldots, \ell\}$.

Let $L$ denote the group generated by $\GG \cup \TT^\sharp$.
Call two elements $h_1$ and $h_2$ of $H$ \emph{congruent}
if the double cosets $L \cdot h_1 \cdot L$ and $L \cdot h_2 \cdot L$ coincide.
We shall prove that every element $h(p_1, \ldots, p_\ell; p'_1, \ldots, p'_k)$ 
is congruent to an element in $\RR \cup \RR^{-1}\cup\, \{\id\}$.
 
 Let $h_0 = h(p_1, \ldots, p_\ell; p'_1, \ldots, p'_k)$ be given. 
 Suppose first there is some $p \in \PP \smallsetminus \{p_*\}$ 
 which occurs at least $p_*-1$ times in the sequence $(p_1, \ldots, p_\ell)$.
There exists then a ``permutation'' $t_1 $ in the subgroup $ \gp(\TT^\sharp)$
 so that the composition $h_0 \circ t_1$ takes on the form
 \[
 h(p,p,  \ldots p, \ldots; p'_1,\ldots,  p'_k) 
 \]
 with $p_* -1$ (or more) occurrences of $p$ in front.
 Then $\bar{h}_0 = h_0 \circ t_1 \circ  g(p_*, p)$ is congruent to $h_0$
 and the generator $p$ occurs less often in $\bar{h}_0$ than it does  in $h_0$.
 An analogous argument applies to the parameter sequence $(p'_1, \ldots, p'_k)$.
 By iteration we can therefore obtain an element $h_1$
 which is congruent to $h_0$
 and satisfies condition c) formulated 
 just before the statement of Theorem \ref{TheoremB9}.

  We may thus assume that the product
 $h_1 =h(p_1, \ldots,p_\ell; p'_1,\ldots,p'_k)$ 
 satisfies condition c).
 If $h_1  \neq \id$ and if $p_i = p'_j$ for some couple $(i,j)$,
 find two ``permutations'' in $\gp(\TT^\sharp)$ which transform $h_1$ into the congruent element
 \[
 h_2 = h(p_1,\ldots, p_{i-1}, p_{i+1}, \ldots p_\ell, p_i; p'_1, \ldots, p'_{j-1}, p'_{j+1}, \ldots, p'_k, p'_j);
 \]
 then observe that $h_2$ is the same PL-homeomorphism as is
 \[
 h(p_1,\ldots, p_{i-1}, p_{i+1}, \ldots p_\ell; p'_1, \ldots, p'_{j-1}, p'_{j+1}, \ldots, p'_k).
 \]
 By iteration we can then arrive at an element $h_2$ 
 that s congruent to $h_1$, hence to $h_0$,
 and which satisfies conditions b) and c).
 Finally,
 by reordering the parameters in $h_2$ by means of $\gp(\TT^\sharp)$, 
we can obtain an element $h_3  = h(q_1, \ldots, h_m ; q'_1, \ldots, q'_n)$ 
 which satisfies conditions b), c) and, in addition, the inequalities
 \[
 q_1 \leq q_2 \leq \cdots \leq q_m
 \quad \text{and} \quad
  q'_1 \leq q'_2 \leq \cdots \leq q_n.
  \]
  Thus $h_3$ or $h_3^{-1}$ will satisfy condition a);
  put differently, $h_3$ belongs to $\RR \cup \RR^{-1}$.
  All taken together,
  we have proved 
  that 
  \[
  H \subseteq L \cdot (\RR \cup  \RR^{-1} \cup \,\{\id\}) \cdot L \subseteq H.
  \]
  Since $L  = \gp(\GG \cup \TT^\sharp)$
  the group $H$ is thus generated by $\GG \cup \TT^\sharp \cup \RR$.
 \end{proof}
%
\subsection[Interpolation of $\PP$-regular and $\PP$-standard subdivisions]%
{Affine interpolation of a $\PP$-regular and a $\PP$-standard subdivision}
\label{ssec:9.8}
%
The next lemma takes care of the last missing bit in the proof of Theorem \ref{TheoremB9}.
\begin{lemma}
\label{LemmaB12}
Let $D$ be the $\PP$-regular subdivision described by the sequence
\begin{equation}
\label{eq:9.9}
S = (1,p_1; n_2, p_2; \ldots; n_\ell, p_\ell)
\end{equation}
and let $D_0$ be the $\PP$-standard subdivision $\St(p_1, p_2, \ldots, p_\ell)$.
Then the affine interpolation $f$ of $D \boxtimes D_0$ is an element of the group
generated by $\Fcal^\sharp \cup \TT^\sharp$.
\end{lemma}

\subsubsection{Proof of Lemma \ref{LemmaB12}: preliminary remark}
\label{sssec:9.8.1}
%
The subdivision $D$, given by the sequence \eqref{eq:9.9},
is a finite increasing sequence of points in the unit interval 
that is constructed in an iterative manner.
In the first step, one picks $p_1 \in \PP$,
adds the subdivision points $1/p_1$, $2/p_1$, \ldots, $(p-1)/p_1$  
to the couple $(0,1)$,
and obtains a $\PP$-regular subdivision $D_1$ of level 1.
In the next step, one picks $p_2 \in \PP$ and one of the $p_1$ subintervals created so far, 
and introduces in it $p_2-1$ equidistant subdivision points.
The result is a $\PP$-regular subdivision $D_2$ of level 2.
One continues in the same manner with $p_3$, \ldots, $p_\ell$ 
and ends up with a $\PP$-regular subdivision $D = D_\ell$ of level $\ell$.
This subdivision is a finite increasing sequence with
\[
\card(D) = 2 + (p_1-1) + (p_2-1) +\cdots+ (p_\ell-1)
\]
points. 
Let $D'$ be another $\PP$-regular subdivision with $\card(D)$ points.
By Proposition \ref{PropositionB8},
the affine interpolation of $D \boxtimes D'$ is then an element $f$ of 
\[
G = G([0,1]; \Z[\gp(\PP)], \gp(\PP)).
\]
This PL-homeomorphism $f$ depends only on the sequences $D$ and $D'$, 
but not on the sequential order in which these sequences are created.
This fact will be a crucial ingredient of the proof 
that we are about to describe.
%
\subsubsection{Proof of Lemma \ref{LemmaB12}: main argument}
\label{sssec:9.8.2N}
%
Let $f$ be the affine interpolation of the $\PP$-regular subdivision $D$, 
with parameter sequence \eqref{eq:9.9},
and of the $\PP$-standard subdivision $D_0 =\St(p_1, p_2, \ldots, p_\ell)$.
Our aim is to show that $f$ is a product $f_2 \circ f_1$ 
with $f_2 \in \gp(\TT^\sharp)$ and $f_1 \in \gp(F^\sharp)$.
The proof will be by induction on the level $\ell$ of $D$.
If $\ell = 1$, then $f =  \id$;
if $\ell = 2$, then $f \in \Fcal$.
Assume now that $\ell > 2$.
The subdivision $D$ contains then the points 
\[
0, \;1/p_1, \; 2/p_1,\;\ldots, \; (p_i-1)/p, \; 1
\]
and at least one of the intervals $[(k-1)/p_1, k/p_1]$ is again subdivided.
Let 
\[
[(m-1)/p_1,m/p]
\] 
be the right most of the subintervals that contain an additional point of $D$.

Assume first that $m > 1$.
Our aim is to construct $g\in  \gp(\Fcal)$ with the property 
that the intersection $g(D) \cap [1/p_1, 1]$ contains only the $p_1$ points 
\begin{equation}
\label{eq:Subdivision-points}
1/p_1, \; 2/p_1,\; \ldots, \; 1.
\end{equation}
Set $b= (m-1)/p_1$ and $c = m/p_1$.
Since the interval $[b, c]$ gets subdivided in the subdivision $D$,
there exists an index $i_2$ so that the intersection 
$D \cap [b,  c]$ contains the points
\[
b,\; b + \Delta, 
\;\ldots, 
\: b + (p_{i_2} - 1) \Delta, \; c, \quad \text{with} \quad
\Delta= \tfrac{1}{p_1 \cdot p_{i_2}}\,,
\]
and possibly some further points.
Set  $q_2 = p_{i_2}$,  put  $g_2 = f(1;p_1; m, q_2)$,
and consider $D_2 = g_2(D)$.
Since $g_2(b + (p_i-1) \Delta) = b$,
 the intersection $D_2 \cap [b, c]$ 
contains fewer points than does $D \cap [b, c]$.
By iteration we can thus find a sequence of maps $g_2$, $g_3$, \ldots $g_{k}$
and a sequence of elements $q_2$, \ldots, $q_k$, 
chosen among the members of the sequence $(p_2, \ldots, p_\ell)$,
so that the subdivision $D_k = (g_k \circ \cdots \circ g_2)(D)$ contains the points 
\[
b = (m-1)/p_1, \; m/p_1, \; \ldots, \; (p_1-1)/p_1, \; 1, 
\]
but no other point to the right of $b$.

The subdivision $D_k$ is a $\PP$-regular subdivision of level $\ell$
that involves the parameters $p_1$, $p_2$, \ldots, $p_\ell$,
not necessarily in the same order as in \eqref{eq:9.9},
and with $m(D_k) = m-1$.
By iteration we can therefore find an element $g \in \gp(\Fcal)$ with the properties
that $D' = g(D)$ is a $\PP$-regular subdivision of level $\ell$,
involving the numbers $p_1$, $p_2$, \ldots, $p_\ell$,
not necessarily in the same order as in \eqref{eq:9.9},
and that the intersection $D' \cap [1/p_1, 1]$ contains only the points \eqref{eq:Subdivision-points}.

The subdivision $D'$ is a $\PP$-regular subdivision of level $\ell$ with $m(D') = 1$. 
It is described by a sequence of the form
\begin{equation}
\label{eq:9.9-prime}
S' = (1,p_1; 1, q_2; \ldots; n'_\ell, q_\ell)
\end{equation}
where the list $(q_2,  \ldots, q_\ell)$ is a permutation of the list 
$(p_2,  \ldots, p_\ell)$.

Consider now the truncated sequence 
$\bar{S}' = (1,q_2;n'_3, q_3; \ldots ; n'_\ell, q_\ell)$. 
It describes a $\PP$-regular subdivision $\bar{D}'$ of level $\ell-1$;
let $\bar{D}'_0$ be the associated standard subdivision $\St(q_2, \ldots, q_\ell)$
and define $\bar{h} \colon \R \iso \R$ to be the affine interpolation  of the sequence of points
$\bar{D}' \boxtimes \bar{D}'_0$.
Since the level of $\bar{D}'$ is $\ell-1$,
the induction hypothesis applies
and guarantees that $\bar{h}$ is a product $\bar{h}_2 \circ \bar{h}_1$ of PL-homeomorphisms 
with $\bar{h}_2 \in \gp(\TT^\sharp)$ and $\bar{h}_1 \in \gp(\Fcal^\sharp)$.
Let's now go back to the $\PP$-regular subdivision $D'$ 
and its associated standard subdivision $D'_0 = \St(p_1, q_2, \ldots, q_\ell)$.
Let $h \colon \R \iso \R$ be the affine interpolation of the sequence of points $D' \boxtimes D'_0$.
Then $h$ is nothing but a rescaled version of $\bar{h}$;
more precisely, $h = (1/p_1) \id \circ \bar{h} \circ ((1/p_1) \id)^{-1}$.
Since conjugation by $(1/p_1) \id$ maps $\Fcal^\sharp$ into $\Fcal^\sharp$ 
and $\TT^\sharp$ into $\TT^\sharp$,
the PL-homeomorphism $h$ is therefore a product $h_2 \circ h_1$ 
with $h_2 \in \gp(\TT^\sharp)$ and $h_1 \in \gp(\Fcal^\sharp)$.

The preceding analysis discloses that $f$, 
the affine interpolation of $D \boxtimes D_0$, 
is a composition of three affine interpolations,
the interpolation $g$ of $D \boxtimes D'$,
the interpolation $h$ of $D' \boxtimes D'_0$,
and the affine interpolation $H$  of $ \St(p_1, q_2, \ldots, q_\ell) \boxtimes D_0$.
The claim thus  follows from the facts that
\[
g \in \gp(\Fcal),
\quad h =h_2 \circ h_1 \in \gp(\TT^\sharp) \circ \gp(\Fcal^\sharp)
\text{ and }
H \in \gp(\TT^\sharp)\,.
\]

\begin{example}
\label{example:Proof-LemmaB12}
We illustrate the reasoning in the proof of Lemma \ref{LemmaB12}
by an example.
Choose $\PP = \{2,3\}$ and  $A = \Z[1/2, 1/3] = \Z[1/6]$, 
and consider the $\PP$-regular subdivision
\[
D = (0,1/3, 1/2, 2/3, 5/6, 8/9, 17/18, 1);
\]
it is displayed in the first horizontal line of the following figure.
\smallskip

\input{chaptB.9.7.Illustration.tex}
\index{Rectangle diagram!examples}%

The subdivision $D$ can be obtained by dividing the unit interval first into 3 intervals of equal length,
then halving both the second and the third subinterval so obtained,
and finally subdividing the last of the created subintervals again into 3 intervals of equal length.
Sequences describing the subdivision $D$ are $(1,3;2,2; 4,2; 5,3)$, 
but also
\[
S = (1,3; 3,2;4, 3; 2,2) \quad \text{with}\quad  (p_1, p_2, p_3, p_4) =  (3,2,3,2).
\]
The level of the subdivision $D$ is 4.
The sequence $S$ tells us that the third subinterval of the initial subdivision $D_1 =(0,1/3,2/3,1)$ is cut into two subintervals of the same length.
We apply therefore the PL-homeomorphism $g_2 = f(1;3; 3,2)$ to $D$ 
and obtain the subdivision
\[
D_2 = (0, 1/6, 1/4, 1/3, 2/3, 7/9, 8/9, 1);
\]
it is shown by the second horizontal line of the figure.
It is again a subdivision of $D_1$, 
but this time the third subinterval is divided into three parts of equal length.
We transform therefore the subdivision $D_2 = g_2(D)$ 
with the help of the PL-homeomorphism $g_3 = f(1; 3; 3,3)$ 
and obtain the subdivision
\[
D_3 = (0, 1/18, 1/12, 1/9, 2/9, 1/3, 2/3, 1).
\]
In $D_3$, 
it is only the first interval of the division $D_1$ that is further subdivided; 
so $D_3 = D'$.
In the proof one passes therefore to the truncated subdivision 
\[
\bar{D}' = (0, 1/6, 1/4, 1/3, 2/3, 1).
\]
It has level 3 and is described by the sequence $\bar{S}' = (1, 3; 1, 2; 2, 2)$.
(The figure discloses a direct way of arriving at a standard subdivision.)
\end{example}
%
\subsection{Example 1: groups $G[\PP]$ with $\PP = \{p\}$ and $p \geq 2$}
\label{ssec:9.9}
\index{Group G[PP]@Group $G[\PP]$!examples}
If $\PP$ is a singleton,
 the generating set afforded by Theorem \ref{TheoremB9} simplifies considerably.
This is due to the fact 
that there is then only one $\PP$-standard subdivision for each level $\ell \geq 1$
and that the subsets $\GG$, $\TT$ and $\RR$, 
defined in section \ref{ssec:9.5},
are empty.
Hence $G[p] = G[\PP] = G([0,1]; \Z[1/p], \gp(p))$ is generated by the set 
\[
\Fcal = \{f(q; p;r;p) \mid q \in \{1, p\} \text{ and } 1 < r \leq p\};
\] 
it has $2(p-1)$ elements.
As will become clear later on,
it is useful to replace the finite generating set $\Fcal$ by the infinite set
\[
\Fcal^\sharp = \{ f(p^m;p; r, p) \mid m \in \N \text{ and } 1 < r \leq p\}.
\]
The PL-homeomorphism $f(m, r) = f(p^m; p; r,p)$ has support $]0, r/ p^{m+1}[$
and  it is linear on the interval $[0, (r-1)/p^{m+1}]$ with slope $1/p$.
The definition of $f(m,r)$ thus implies
that the relation
\[
\act{f(m,r)} f(m',r') = f(m' + 1,r')
\]
holds whenever, either $m < m'$, or $m = m'$ and $r > r'$.

The description of these relations simplifies 
if one enumerates the generators according to the decreasing size of their supports. 
We introduce therefore generators $x_i$ for $i \in \N$, 
express $i$ in the form $i = (p-r) + (p-1)m$ and set
\[
 x_i = f(m,r) = f(p^m;p, r, p).
\]
The above relations then take on the form $\act{x_i}x_j = x_{j + (p-1)}$;
they are valid for every couple of non-negative integers $(i,j)$ with $i < j$.
They imply, in particular, that $G[p]$ is generated by $p$ elements, 
namely $x_0$, $x_1$, \dots, $x_{p-1}$.

The following result summarizes the insights obtained so far:
\begin{corollary}
\label{CorollaryB13}
\index{Group G[p]@Group $G[p]$!generating set}%
\index{Group G[p]@Group $G[p]$!finite generation}%
For every integer $p$ greater than 1,
the group $G[p]$ is generated by elements $\{x_i \mid i \geq 0\}$
that satisfy the relations
\begin{equation}
\label{eq:9.12}
\act{x_i}x_j = x_{j +(p-1)}
\end{equation}
whenever $i < j$.
They show that $G[p]$ is generated by the subset $\{x_i \mid 0 \leq i < p \}$.

To relate the generators $x_i$ to the previously found generators, 
write $i$ in the form $(p-1)m + (p-r)$
with $1 < r \leq p$; then $x_i = f(p^m; p; r, p)$.
\end{corollary}

\begin{remarks}
\label{remarks:on-CorollaryB13}
(i) The  relations displayed in equation \eqref{eq:9.12} actually define $G[p]$ in terms of the $x_i$. 
A proof of this fact will be given in section \ref{ssec:15.1}.

(ii) For each $p \geq 2$,
the group $G[p]$ is not merely finitely generated, 
it is finitely presented (see section \ref{ssec:15.2}) 
and of type $\FP_\infty$ (K. S. Brown, unpublished.
\index{Finiteness properties of!G[p]@$G[p]$}%
\index{Brown, K. S.}%
\footnote{See section \ref{sssec:Notes-finiteness-properties-Brown-Stein} for an update.})

(iii) The group $G[2]$ is the notable group studied by R. J. Thompson (\cite{Tho74}),
P. Freyd and A. Heller (see \cite{FrHe93}), 
K. S. Brown and R. Geoghegan (\cite{BrGe84}),
and many others.
\index{Thompson, R. J.}%
\index{Brown, K. S.}%
\index{Geoghegan, R.}%
\index{Freyd, P.}%
\index{Heller, A.}%
It has a well-known finite presentation,
namely
\begin{equation}
\label{eq:9.13New}
\index{Thompson's group F@Thompson's group $F$!finite presentation}%
\index{Thompson's group F@Thompson's group $F$!local definition}%
\langle x = x_0, y = x_1 \mid \act{yx}y = \act{x^2} y, \act{yx^2}y = \act{x^3} y \rangle.
\end{equation}

(iv) The generators $x_i$ will crop up in section  \ref{sssec:preimages-yi}
in a quite different context.
\end{remarks}

\begin{illustration}
\label{illustration:9.9.Example1}
Corollary \ref{CorollaryB13} provides explicit generating sets with $p$ elements 
for each of the groups $G[p]$.
Below, the rectangle diagrams of the generators will be displayed 
for $p=2$ and $p = 3$. 

If $p= 2$, 
the generator $x_i$ coincides with $f(m, 2) = f(2^m;2;2,2)$; 
for $m = 0$ and $m = 1$, their rectangle diagrams are displayed below.
\index{Group G[p]@Group $G[p]$!examples}%
\index{PP-regular subdivision@$\PP$-regular subdivision!examples|(}%
\index{Rectangle diagram!examples}%
\smallskip

\input{chaptB.9.8.Illustration1.tex}

If $p=3$, the generator $x_i$ is equal to $f(m, 3) = f(3^m;3,3,3)$ if $i = 2m$
and equal to $f(3^m; 3; 2,3)$ if $i = 2m + 1$. 
The rectangle diagrams of the three generators $x_0$, $x_1$ and $x_2$ 
are then as shown next.
\index{Rectangle diagram!examples}%
\smallskip

\input{chaptB.9.8.Illustration2.tex}
\end{illustration}
%
\subsection{Example 2: group $G[\PP]$ with $\PP= \{2,3\}$}
\label{ssec:9.10}
\index{Group G[PP]@Group $G[\PP]$!examples}%
\index{Rectangle diagram!examples|(}%
%
If $\PP = \{2, 3\}$ the generating set provided by Theorem \ref{TheoremB9} is still so small 
that its elements can be displayed individually.
The set is the union of four subsets $\GG$, $\RR$,  $\TT$ and $\Fcal$.

The subset $\GG$ consists of a single element, namely $g(2,3)$.
The set $\RR$ has also only one element, namely $h(2,2;3)$, 
and this element coincides with $g(2, 3)$.
This homeomorphism is displayed next.
\smallskip

\begin{minipage}{10cm}
\psfrag{1}{\hspace*{-1.7mm}   \small $0$}
\psfrag{2}{\hspace*{-2.0mm} \small $\tfrac{1}{4}$}
\psfrag{3}{\hspace*{-1.7mm}  \small $\tfrac{1}{2}$}
\psfrag{4}{\hspace*{-1.2mm}  \small $1$}
\psfrag{11}{\hspace*{-0.8mm}   \small $0$}
\psfrag{12}{\hspace*{-0.8mm} \small $\tfrac{1}{3}$}
\psfrag{13}{\hspace*{-0.9mm}  \small $\tfrac{2}{3}$}
\psfrag{14}{\hspace*{-0.5mm}  $1$}
\psfrag{21}{\hspace*{-1mm}  \small   $\tfrac{4}{3}$}
\psfrag{22}{\hspace*{-1.0mm}  \small  $\tfrac{4}{3}$}
\psfrag{23}{\hspace*{-0.5mm}  \small  $\tfrac{2}{3}$}

\psfrag{g23}{\hspace{-7mm}$g(2,3)$}
\begin{equation*}
\includegraphics[width= 6.5cm]{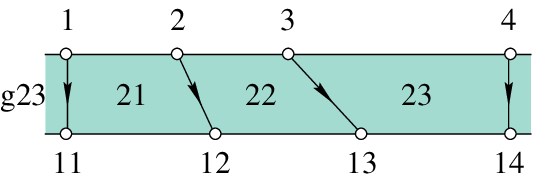}
\end{equation*}
\end{minipage}
\smallskip

The subset 
$\TT = \{t(q;p,p')
\mid 
q \in \{1,2,3\} \text{ and } (p,p') \in \PP^2  \text{ with } p < p' \}$
contains 3 elements; 
they are shown in Figure \ref{fig:PL-homeomorphisms-of-type-t}.
\begin{figure}[b]
\psfrag{1}{\hspace*{-1.7mm}  \small $0$}
\psfrag{2}{\hspace*{-2.8mm}  \small   $\tfrac{1}{12}$}
\psfrag{3}{\hspace*{-1.3mm}  \small    $\tfrac{1}{9}$}
\psfrag{4}{\hspace*{-1.5mm}  \small $\tfrac{1}{6}$}
\psfrag{5}{\hspace*{-2.0mm}  \small     $\tfrac{2}{9}$}
\psfrag{6}{\hspace*{-1.5mm}  \small   $\tfrac{1}{4}$}
\psfrag{7}{\hspace*{-1.5mm} \small  $\tfrac{1}{3}$}
\psfrag{8}{\hspace*{-1.8mm}   \small   $\tfrac{1}{2}$}
\psfrag{9}{\hspace*{-1.6mm}   \small   $\tfrac{2}{3}$} 
\psfrag{10}{\hspace*{-0.9mm}   \small $1$}
\psfrag{21}{\hspace*{-0mm}  \small   1}
\psfrag{22}{\hspace*{-0mm}  \small   1}
\psfrag{23}{\hspace*{-3.5mm}  \small  $\tfrac{3}{2}$}
\psfrag{24}{\hspace*{-0mm}  \small   $\tfrac{2}{3}$}
\psfrag{41}{\hspace*{-1.0mm} \small $1$}
\psfrag{42}{\hspace*{-1.0mm} \small $1$}
\psfrag{43}{\hspace*{-1.0mm}\small $\tfrac{3}{2}$}
\psfrag{44}{\hspace*{-0.7mm} \small    $\tfrac{2}{3}$}
\psfrag{45}{\hspace*{2.7mm} \small   $1$}
\psfrag{61}{\hspace*{-0.8mm} \small        $1$}
\psfrag{62}{\hspace*{-0.8mm} \small        $1$}
\psfrag{63}{  \hspace*{-2.3mm}   \small   $\tfrac{3}{2}$}
\psfrag{64}{\hspace*{-1.5mm} \small       $\tfrac{2}{3}$}
\psfrag{65}{\hspace*{-0.7mm} \small       $1$}
\psfrag{66}{\hspace*{-1.7mm} \small       $1$}
\psfrag{la1}{\hspace{-7mm} $t(1;2,3)$}
\psfrag{la2}{\hspace{-7mm} $t(2;2,3)$}
\psfrag{la3}{\hspace{-7mm} $t(3;2,3)$}
\begin{equation*}
\includegraphics[width= 8cm]{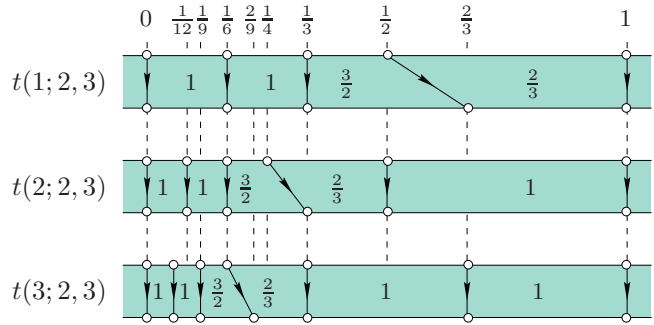}
\end{equation*}
\caption{The PL-homeomorphisms $t(1;2,3)$, $t(2;2,3)$ and $t(3;2,3)$}
\label{fig:PL-homeomorphisms-of-type-t}
\end{figure}

The subset
\[
\Fcal = \{ f(q; p; r, p') \mid q \in \{1,2,3\}, \quad (p,p') \in \{2,3\}^2 \text{ and } 1 < r \leq p \},
\]
finally,
has 18 elements.
The first six of them, with $q=1$, are shown in Figure \ref{fig:PL-homeomorphisms-of-type-f}.
The rectangle diagrams of the remaining 12 elements 
are obtained by reducing the above 6 diagrams 
by the factors $\tfrac{1}{2}$ and $\tfrac{1}{3}$, respectively, 
and extending by the identity on the right.
Altogether we see 
that the group $G[\{2,3\}]$ has a generating set with  $1 + 3 + 18= 22$ elements.
\begin{figure}
\psfrag{1}{\hspace*{-1.7mm}   $0$}
\psfrag{2}{\hspace*{-1.6mm}  $\tfrac{1}{9}$}
\psfrag{3}{\hspace*{-1.6mm}  $\tfrac{1}{6}$}
\psfrag{4}{\hspace*{-1.7mm}  $\tfrac{2}{9}$}
\psfrag{5}{\hspace*{-1.6mm} $\tfrac{1}{4}$}
\psfrag{6}{\hspace*{-1.6mm}  $\tfrac{1}{3}$}
\psfrag{7}{\hspace*{-1.7mm}  $\tfrac{4}{9}$}
\psfrag{8}{\hspace*{-1.6mm}  $\tfrac{1}{2}$}
\psfrag{9}{\hspace*{-1.5mm}  $\tfrac{5}{9}$}
\psfrag{10}{\hspace*{-1.2mm} $\tfrac{2}{3}$}
\psfrag{11}{\hspace*{-1.5mm}   $\tfrac{3}{4}$}
\psfrag{12}{\hspace*{-1.2mm} $\tfrac{7}{9}$}
\psfrag{13}{\hspace*{-1.2mm}  $\tfrac{5}{6}$}
\psfrag{14}{\hspace*{-1.2mm}  $\tfrac{8}{9}$}
\psfrag{15}{\hspace*{-1.2mm}   $1$}
\psfrag{21}{\hspace*{-5.5mm}  \small   $\tfrac{1}{2}$}
\psfrag{22}{\hspace*{-0.2mm}  \small  $1$}
\psfrag{23}{\hspace*{0.8mm}  \small   $2$}
\psfrag{41}{\hspace*{-4.2mm} \small $\tfrac{1}{3}$}
\psfrag{42}{  \hspace*{-1.3mm}\small $1$}
\psfrag{43}{\hspace*{-1.3mm} \small $1$}
\psfrag{44}{\hspace*{-0.4mm} \small $3$}
\psfrag{61}{\hspace*{-1.8mm} \small  $\tfrac{1}{2}$}
\psfrag{62}{  \hspace*{-1.7mm}   \small   $1$}
\psfrag{63}{\hspace*{0.5mm} \small   $2$}
\psfrag{64}{\hspace*{-0.3mm} \small   $1$}
\psfrag{81}{  \hspace*{-2.9mm}   \small $\tfrac{1}{2}$}
\psfrag{82}{\hspace*{-3.5mm} \small   $\tfrac{1}{2}$}
\psfrag{83}{\hspace*{2.0mm} \small   $2$}
\psfrag{84}{\hspace*{-3.0mm} \small   $2$}
\psfrag{101}{\hspace*{-3.2mm} \small   $\tfrac{1}{3}$}
\psfrag{102}{\hspace*{-0.50mm} \small   $1$}
\psfrag{103}{\hspace*{-0.5mm} \small   $1$}
\psfrag{104}{\hspace*{0.2mm} \small   $3$}
\psfrag{105}{\hspace*{-5mm} \small   $1$}
\psfrag{121}{\hspace*{-3.1mm} \small   $\tfrac{1}{3}$}
\psfrag{122}{\hspace*{-2.8mm} \small  $\tfrac{1}{3}$ }
\psfrag{123}{\hspace*{-1mm} \small   }
\psfrag{124}{\hspace*{0.4mm} \small   $3$}
\psfrag{125}{\hspace*{0.4mm} \small   $3$}
\psfrag{la1}{\hspace{-10mm}$f(1;2;2,2)$}
\psfrag{la2}{\hspace{-10mm} $f(1;2;2,3)$}
\psfrag{la3}{\hspace{-10mm}$f(1;3;2,2)$}
\psfrag{la4}{\hspace{-10mm}$f(1;3;3,2)$}
\psfrag{la5}{\hspace{-10mm}$f(1;3;2,3)$}
\psfrag{la6}{\hspace{-10mm}$f(1;3;3,3)$}
\begin{equation*}
\includegraphics[width= 8cm]{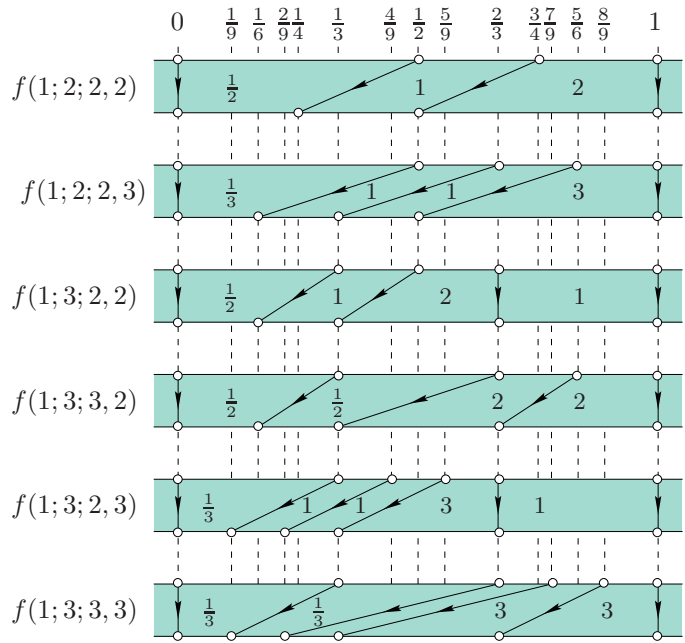}
\end{equation*}
\caption{PL-homeomorphisms of type $f$}
\label{fig:PL-homeomorphisms-of-type-f}
\end{figure}
\index{PP-regular subdivision@$\PP$-regular subdivision!examples|)}
\index{Rectangle diagram!examples|)}
\index{Group G[PP]@Group $G[\PP]$!generating sets}%
%

%% file: chaptB.9.7.Illustration.tex
\begin{equation*}
\psfrag{1}{\hspace*{-1.7mm}    \small  $0$}
\psfrag{6}{\hspace*{-1.7mm}    \small  $\tfrac{1}{6}$}
\psfrag{8}{\hspace*{-1.7mm}    \small   $\tfrac{1}{4}$}
\psfrag{9}{\hspace*{-1.7mm}    \small     $\tfrac{1}{3}$}
\psfrag{10}{\hspace*{-1.2mm}  \small   $\tfrac{1}{2}$}
\psfrag{11}{\hspace*{-1.2mm}  \small  $\tfrac{2}{3}$}
\psfrag{12}{\hspace*{-1.3mm}  \small  $\tfrac{7}{9}$}
\psfrag{13}{\hspace*{-1.1mm}  \small  $\tfrac{5}{6}$}
\psfrag{14}{\hspace*{-1.3mm}   \small   $\tfrac{8}{9}$}
\psfrag{15}{\hspace*{-1.9mm}  \small  $\tfrac{17}{18}$}
\psfrag{16}{\hspace*{-0.7mm}   \small   1}
\psfrag{21}{\hspace*{-1.0mm}   \small 0}
\psfrag{22}{\hspace*{-1.0mm} \small  }
\psfrag{23}{\hspace*{-1.5mm}     \small $\tfrac{1}{18}$}
\psfrag{24}{\hspace*{-2.0mm} \small }
\psfrag{25}{\hspace*{-0.5mm} \small $\tfrac{1}{9}$}
\psfrag{27}{\hspace*{-0.8mm} \small $\tfrac{2}{9}$}
\psfrag{29}{\hspace*{-0.5mm} \small $\tfrac{1}{3}$}
\psfrag{31}{\hspace*{-0.8mm} \small $\tfrac{2}{3}$}
\psfrag{36}{\hspace*{-0.4mm} \small $1$}
\psfrag{41}{\hspace*{-1.2mm}   \small $\tfrac{1}{2}$}
\psfrag{42}{\hspace*{-1.9mm} \small  $\tfrac{1}{2}$}
\psfrag{43}{\hspace*{-1.9mm}     \small $\tfrac{1}{2}$}
\psfrag{44}{\hspace*{-1.0mm} \small 2}
\psfrag{45}{\hspace*{-0.5mm} \small 2}
\psfrag{46}{\hspace*{-0.5mm} \small 2}
\psfrag{47}{\hspace*{-1.0mm} \small 2}
\psfrag{61}{\hspace*{-1.2mm}   \small $\tfrac{1}{3}$}
\psfrag{62}{\hspace*{-1.5mm} \small  }
\psfrag{63}{\hspace*{-1.7mm}     \small }
\psfrag{64}{\hspace*{-2.3mm} \small $\tfrac{1}{3}$}
\psfrag{65}{\hspace*{0.0mm} \small 1}
\psfrag{66}{\hspace*{-0.5mm} \small 3}
\psfrag{67}{\hspace*{-1.8mm} \small 3}
\psfrag{81}{\hspace*{-0.7mm}   \small $\tfrac{1}{2}$}
\psfrag{82}{\hspace*{-0.7mm} \small  1}
\psfrag{83}{\hspace*{-0.7mm}     \small 2}
\psfrag{84}{\hspace*{-1.0mm} \small 1}
\psfrag{85}{\hspace*{-1.5mm} \small 1}
\psfrag{86}{\hspace*{-1.5mm} \small 1}
\psfrag{87}{\hspace*{-1.8mm} \small 1}
\psfrag{la1}{\hspace{-3.5mm}$f(1;3;3,2)$}
\psfrag{la2}{\hspace{-4.5mm} $f(1;3;3,3)$}
\psfrag{la3}{\hspace{-3.8mm}$f(9;2;2,2)$}
\includegraphics[width= 10.5cm]{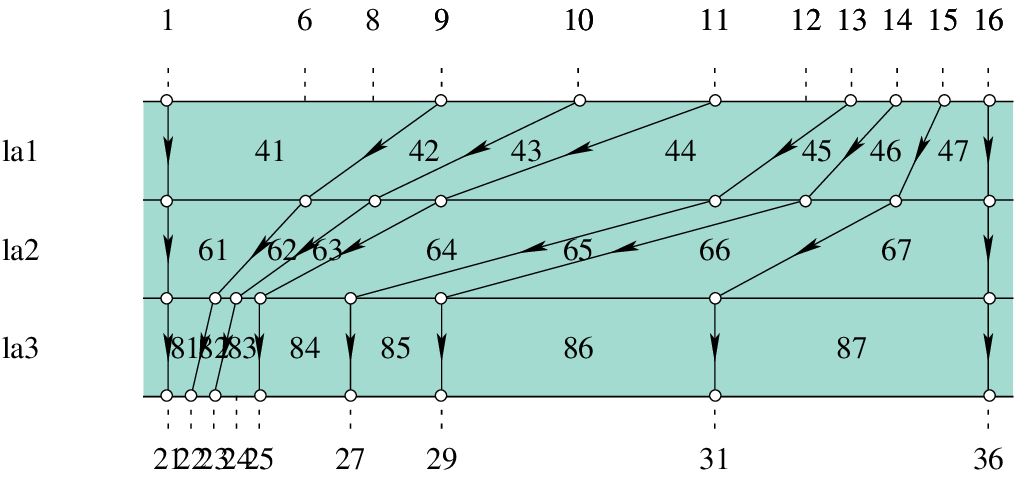}
\end{equation*}

%% file: chaptB.9.8.Illustration1.tex
\begin{equation*}
\psfrag{1}{\hspace*{-1.7mm}  \small $0$}
\psfrag{2}{\hspace*{-1.6mm}  \small $\tfrac{1}{2}$}
\psfrag{3}{\hspace*{-1.7mm}  \small $\tfrac{3}{4}$}
\psfrag{4}{\hspace*{-1.2mm}  \small  $1$}
\psfrag{11}{\hspace*{-0.7mm}  \small $0$}
\psfrag{12}{\hspace*{-1.0mm}  \small $\tfrac{1}{4}$}
\psfrag{13}{\hspace*{-0.8mm}  \small $\tfrac{1}{2}$}
\psfrag{14}{\hspace*{-0.7mm}  \small  $1$}
\psfrag{21}{\hspace*{-1mm}  \small   $\tfrac{1}{2}$}
\psfrag{22}{\hspace*{-1.0mm}  \small  $1$}
\psfrag{23}{\hspace*{-0mm}  \small   $2$}
\psfrag{31}{\hspace*{-0.7mm}  \small $0$}
\psfrag{32}{\hspace*{-0.8mm}  \small $\tfrac{1}{4}$}
\psfrag{33}{\hspace*{-1.0mm}  \small $\tfrac{3}{8}$}
\psfrag{34}{\hspace*{-0.8mm}  \small $\tfrac{1}{2}$}
\psfrag{35}{\hspace*{-0.7mm}  \small  $1$}
\psfrag{41}{\hspace*{-0.3mm}  \small $0$}
\psfrag{42}{\hspace*{-0.7mm}  \small $\tfrac{1}{8}$}
\psfrag{43}{\hspace*{-0.7mm}  \small $\tfrac{1}{4}$}
\psfrag{44}{\hspace*{-0.5mm}  \small $\tfrac{1}{2}$}
\psfrag{45}{\hspace*{-0.4mm}  \small  $1$}
\psfrag{51}{\hspace*{-1.3mm} \small $\tfrac{1}{2}$}
\psfrag{52}{  \hspace*{-0.5mm}\small $1$}
\psfrag{53}{\hspace*{-0.5mm} \small $2$}
\psfrag{54}{\hspace*{-1mm} \small $1$}
\psfrag{la1222}{\hspace*{-1mm} \small $x_0$}
\psfrag{la2222}{\hspace*{-1mm} \small $x_1$}
\includegraphics[width= 5.5cm]{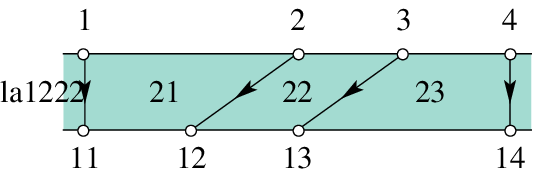}
\hspace*{12mm}
\includegraphics[width= 5.5cm]{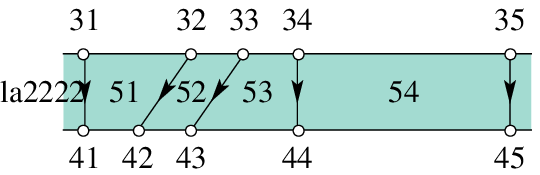}
\end{equation*}

%% file: chaptB.9.8.Illustration2.tex
\begin{minipage}{12cm}
%
\psfrag{61}{\hspace*{-0.7mm} \small   $0$}
\psfrag{62}{\hspace*{-1.0mm} \small  $\tfrac{1}{3}$}
\psfrag{63}{  \hspace*{-1.0mm}\small  $\tfrac{2}{3}$}
\psfrag{64}{\hspace*{-1.0mm} \small   $\tfrac{7}{9}$}
\psfrag{65}{\hspace*{-1.1mm} \small  $\tfrac{8}{9}$}
\psfrag{66}{\hspace*{-0.7mm} \small   $1$}
\psfrag{71}{\hspace*{-0.7mm} \small   $0$}
\psfrag{72}{\hspace*{-0.7mm} \small  $\tfrac{1}{9}$}
\psfrag{73}{\hspace*{-0.7mm} \small  $\tfrac{2}{9}$}
\psfrag{74}{\hspace*{-0.7mm} \small  $\tfrac{1}{3}$}
\psfrag{75}{\hspace*{-0.7mm} \small  $\tfrac{2}{3}$}
\psfrag{76}{\hspace*{-0.5 mm} \small   $1$}
\psfrag{81}{  \hspace*{-2.8mm}  \small   $\tfrac{1}{3}$}
\psfrag{82}{\hspace*{-2.5mm}    \small   $\tfrac{1}{3}$}
\psfrag{83}{  \hspace*{-2.7mm}   }
\psfrag{84}{\hspace*{-0.7mm} \small   $3$}
\psfrag{85}{\hspace*{-0.7 mm} \small   $3$}
%
\psfrag{91}{\hspace*{-0.6mm} \small   $0$}
\psfrag{92}{\hspace*{-0.9mm} \small  $\tfrac{1}{3}$}
\psfrag{93}{  \hspace*{-0.7mm}\small  $\tfrac{4}{9}$}
\psfrag{94}{\hspace*{-0.7mm} \small   $\tfrac{5}{9}$}
\psfrag{95}{\hspace*{-0.8mm} \small  $\tfrac{2}{3}$}
\psfrag{96}{\hspace*{-0.5mm} \small   $1$}
\psfrag{101}{\hspace*{-0mm} \small   $0$}
\psfrag{102}{\hspace*{-0.3mm} \small  $\tfrac{1}{9}$}
\psfrag{103}{\hspace*{-0.1mm} \small  $\tfrac{2}{9}$}
\psfrag{104}{\hspace*{-0.2mm} \small  $\tfrac{1}{3}$}
\psfrag{105}{\hspace*{-0.25mm} \small  $\tfrac{2}{3}$}
\psfrag{106}{\hspace*{-0.1 mm} \small   $1$}
\psfrag{111}{  \hspace*{-2.4mm}   \small   $\tfrac{1}{3}$}
\psfrag{112}{\hspace*{-0.0mm} \small   1}
\psfrag{113}{  \hspace*{-1.5mm}   \small   1}
\psfrag{114}{\hspace*{-0.5mm} \small   3}
\psfrag{115}{\hspace*{-0.8mm} \small   1}
%
\psfrag{121}{\hspace*{-0.3mm} \small   $0$}
\psfrag{122}{\hspace*{-0.4mm} \small  $\tfrac{1}{9}$}
\psfrag{123}{  \hspace*{-0.1mm}\small  $\tfrac{2}{9}$}
\psfrag{124}{\hspace*{-1.7mm} \small   }
\psfrag{125}{\hspace*{-1.8mm} \small  }
\psfrag{126}{\hspace*{-0.3mm} \small  $\tfrac{1}{3}$}
\psfrag{127}{\hspace*{-0.4mm} \small  $\tfrac{2}{3}$}
\psfrag{128}{\hspace*{-0.1mm} \small   $1$}
\psfrag{131}{\hspace*{-0.1mm} \small   $0$}
\psfrag{132}{\hspace*{-0.7mm} \small  }
\psfrag{133}{\hspace*{-0.7mm} \small  }
\psfrag{134}{\hspace*{-0.2mm} \small  $\tfrac{1}{9}$}
\psfrag{135}{\hspace*{-0.0mm} \small  $\tfrac{2}{9}$}
\psfrag{136}{\hspace*{-0.2mm} \small  $\tfrac{1}{3}$}
\psfrag{137}{\hspace*{-0.2mm} \small  $\tfrac{2}{3}$}
\psfrag{138}{\hspace*{0.1mm} \small   $1$}

\psfrag{141}{  \hspace*{-1.0mm}   \small   $\tfrac{1}{3}$}
\psfrag{142}{\hspace*{-0.6mm}   \small   $\tfrac{1}{3}$}
\psfrag{143}{  \hspace*{-2.7mm}  }
\psfrag{144}{\hspace*{0.2mm} \small   3}
\psfrag{145}{\hspace*{0.2mm} \small   3}
\psfrag{146}{\hspace*{-1.0mm} \small   1}
\psfrag{147}{\hspace*{-1.0mm} \small   1}

\psfrag{la1333}{\hspace{-0mm}$x_0$}
\psfrag{la1323}{\hspace{-0mm}$x_1$}
\psfrag{la3333}{\hspace{-0mm}$x_2$}
\begin{gather*}
\includegraphics[width= 5.5cm]{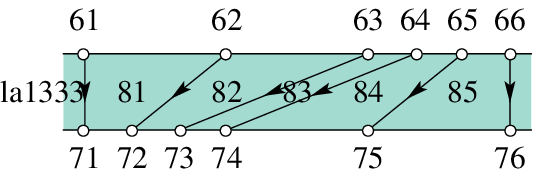}
\hspace*{7mm}
\includegraphics[width= 5.5cm]{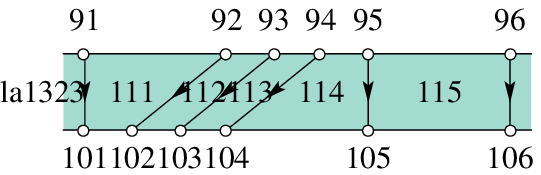}
\\
\includegraphics[width= 5.5cm]{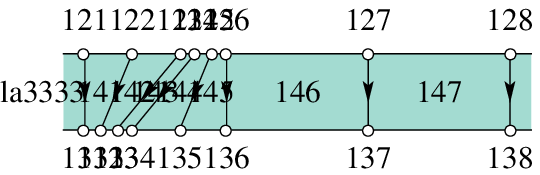}
\end{gather*}
\end{minipage}

%% file: chaptC_Bounded-homeos.tex
%
%
\chapter[The Subgroup of Bounded Homeomorphisms]%
{The Subgroup of Bounded Homeomorphisms $B$}
\label{chap:C}
\setcounter{section}{9}
%
Recall that $B(I;A,P)$,
the subgroup of bounded homeomorphisms,
is the kernel of the homomorphism
\begin{equation*}
(\lambda,\rho) \colon G(I;A,P) \longrightarrow \Aff(A,P) \times \Aff(A,P).
\end{equation*}
Alternatively, it can be described as the union
\begin{equation*}
\bigcup \, \{G(I_1;A,P)\mid I_1\ \text{a compact subinterval of int}(I)\}.
\end{equation*}
This description and Corollary \ref{CorollaryA1} show 
that the $B(I;A,P)$-orbits in $A \cap \Int(I)$ are of the form $(IP \cdot A + a) \cap \Int(I)$, 
even in case $I = \R$.
\index{Subgroup B(I;A,P)@Subgroup $B(I;A,P)$!characterization}%

The results of this chapter deal with three topics:
the independence of $B= B(I;A,P)$ on $I$, 
estimates of $B_{\ab} =B/[B,B]$ 
and the simplicity of $B' =[B,B]$.
%
\section{Simplicity of the derived group of the subgroup $B$}
\label{sec:10}
The main objective of this section is the proof
that the derived group of the subgroup $B(I;A,P)$ is simple,
no matter what the values of the parameters $I$, $A$ and $P$ are.
We begin by showing 
that the isomorphism type of $B(I;A,P)$ does not depend on the interval $I$.
%
\subsection{Independence on the interval containing the supports}
\label{ssec:10.1}
The fact that $B(I;A,P)$ does \emph{not} depend on $I$ is a consequence of
\begin{proposition}
\label{PropositionC1}
\index{Subgroup B(I;A,P)@Subgroup $B(I;A,P)$!independence on I@independence on $I$}%
\index{Isomorphisms!construction}%
\index{Proposition \ref{PropositionC1}!statement|textbf}%
The group $B(I;A,P)$ is isomorphic to $B(\R;A,P)$
for every  interval of positive length $I$. 
\end{proposition}

\begin{proof}
Let $a_0$ be a positive element of $A$ 
and choose an element $b_0$ in $A\cap \Int(I)$.  
Since each interval $]0,\frac{1}{m}[$ contains an element of $P$, 
we can find a doubly infinite increasing sequence $(b_i \mid i\in \Z)$ of elements in $A\cap \Int(I)$ 
satisfying properties (i) and (ii):
\begin{enumerate}[(i)]
\item  for each $i\in \Z$, the difference $b_{i+1} -b_i$ is in $P\cdot a_0$; 
\item  $\lim_{i\to-\infty}b_i = \inf(I)$ and $\lim_{i\to+\infty}b_i = \sup(I)$.
\end{enumerate}

Let $\varphi\colon \R \to \Int(I)$ be the affine interpolation of the doubly infinite sequence $((i\cdot a_0,b_i)\mid i\in\Z)$.  Then $\varphi$ is a $\PL$-homeomorphism of $\R$ onto $\Int(I)$ whose slopes are in $P$ 
and whose vertices are in $A^2$.  
If $f$ is an element of $B(\R;A,P)$ there exists a positive integer  $\ell$ such that 
the support of $f$ is contained in the interval $]-\ell a_0,\ell a_0[$.  
It follows that $\act{\varphi}f  = \varphi\circ f \circ \varphi^{-1}$
is a finitary $\PL$-auto-homeomorphism of $\Int(I)$
with support in $]b_{-\ell}, b_\ell[$, 
slopes in $P$ and vertices in $A^2$.  
Hence $(\act{\varphi} f)^\sim$,
the extension of $\act{\varphi}f$ that fixes every point outside of  $\Int(I)$, 
is an element of $B(I;A,P)$.  

Conversely,
if $g \in B(I;A,P)$ the conjugate $\varphi^{-1}\circ (g \restriction{\Int(I)})\circ \varphi$ is in $B(\R;A,P)$.  
From these facts we see 
that conjugation by $\varphi$ induces an isomorphism of $B(\R;A,P)$ onto $B(I;A,P)$.  
\end{proof}
%
\subsection{Proof of the simplicity of the derived group of $B$}
\label{ssec:10.2}
%
The simplicity of $B'$ will be a consequence of Proposition 
\ref{prp:Higmans-Theorem1} below;
its proof is a variation on G. Higman's proof of Theorem 1 in \cite{Hig54a}.
\index{Higman, G.}%
\begin{prp}
\label{prp:Higmans-Theorem1}
\index{Higman's simplicity result}%
\index{Simplicity result for [B,B]@Simplicity result for $[B,B]$}%
Let $B \neq \{1\}$ be a group satisfying the following condition.
\begin{equation}
\label{eq:Higmans-condition}
\left\{
\begin{minipage}{9.5cm}
For every ordered pair $(x,y) \in B^2$ and every $z \in B \smallsetminus \{1\}$ \\
there exists $u \in B$ such that the equation $[\act{u}x, \act{zu}y] = 1$ holds.
\end{minipage}
\right.
\end{equation}
Then $B'$ is simple and non-abelian, or $B'' = \{1\}$.
\end{prp}

\begin{proof}
We begin with a calculation.
Let $(x, y,z) \in B^3$  be a given triple.
If $z \neq 1$, there exists, by assumption,  an element $u \in B$ 
such that the equation  $[\act{u}x, \act{zu}y] = 1$ is valid;
by conjugating this equation by $u^{-1}$ and setting $w = u^{-1}zu$ 
one arrives at the equation $[x, \act{w}y] = 1$. 
It allows one to rewrite the commutator of $x$ and $y$ like this:
\begin{align*}
[x,y] 
&= 
xyx^{-1} \cdot \left( \act{w}y^{-1} \cdot \act{w}y \right) \cdot y^{-1}
=
xy  \cdot \left(x^{-1} \cdot \act{w}y^{-1} \right)\cdot \act{w}y \cdot   y^{-1}\\
&=
xy \cdot  \left(\act{w}y^{-1} \cdot x^{-1}\right) \cdot \act{w}y \cdot  y^{-1}
= xy w y^{-1} \cdot w^{-1} x^{-1} \cdot w \cdot y w^{-1} y^{-1}\\
&=
\act{xy} w \cdot \act{x}w^{-1} \cdot w \cdot \act{y} w^{-1}.
\end{align*}
Since $w$ is short for  $u^{-1}zu$ 
the calculation can be summarized by saying that
\begin{equation}
\label{eq:Commutator-relation}
[x,y] = \act{xyu^{-1}} z \cdot \act{xu^{-1}} z^{-1} \cdot \act{u^{-1}} z  \cdot \act{yu^{-1}} z^{-1}.
\end{equation}

Assume now that $N$ is a normal subgroup of $B$  and that $z \in N$.
Equation \eqref{eq:Commutator-relation} then shows 
that every commutator $[x,y]$ lies in $N$ and so $B' \subseteq N$.
It follows, first of all, that $B'$ is contained in every normal subgroup $N \neq \{1\}$
and is thus a minimal normal subgroup of $B$.
Two cases now arise.

If $B'' = \{1\}$ then $B$ is metabelian,
and either $B$ is abelian, or 
$B'$ is a minimal abelian normal subgroup of $B$ 
and thus a simple module over the ring $\Z{B_{\ab}}$.
Otherwise,
consider a normal subgroup $N \neq \{1\}$ contained in $B''$.
The calculation then implies
that $B' \subseteq N \subseteq B''$ and so $B' = B''$;
in addition, 
$B''$ is a minimal normal subgroup of $B$.
If we can show that $B''$ \emph{is a minimal normal subgroup of} $B'$
then $B' = B''$ will be a non-abelian simple group, as claimed by the proposition.

To achieve this goal,
we verify
that the element $u $ occurring in equation \eqref{eq:Commutator-relation} can be chosen inside $B'$.
Indeed,
given elements $x$, $y$  of $B$ and $z \in B \smallsetminus \{1\}$,
there exists, by statement \eqref{eq:Higmans-condition},
an element $u \in B$ with $[\act{u}x, \act{z u } y ] = 1$.
Next, with $(z, u, z)$ in the rôle of $(x,y,z)$, 
there exists $v \in B$ so that $[\act{v}z, \act{z v } u] = 1$. 
So $z$ commutes with $u_1 = \act{v^{-1} z v} u$.
Set $u_2 = u_1^{-1} \cdot u = [v^{-1} z v, u^{-1}]$.
Then $u_2 \in B'$
and the calculation
\begin{align*}
1 
&= 
u_1^{-1} \cdot  \left[\act{u}x, \act{z u } y \right]  \cdot u_1
=
\left[ u_1^{-1}\cdot\act{u} x\cdot u_1, u_1^{-1}\cdot \act{z u } y \cdot u_1 \right]
=
\left[\act{u_2} x, \act{z u_2} y\right ] 
\end{align*}
justifies our contention. 
\end{proof}

\begin{corollary}
\label{crl:Simplicity-of-B'}
\index{Subgroup B(I;A,P)@Subgroup $B(I;A,P)$!simplicity of [B,B]@simplicity of $[B,B]$}%
\index{Simplicity result for [B,B]@Simplicity result for $[B,B]$}%
For every choice of the triple $(I,A, P)$ the derived group of $B(I;A,P)$ is a simple, non-abelian group.
\end{corollary}

\begin{proof}
We begin by verifying 
that $B$ satisfies the hypotheses of Proposition \ref{prp:Higmans-Theorem1}.
The idea underlying this verification is simple: 
 the union of the supports of elements $f$ and $g$ lies in a compact interval, 
 say  $I =[b_1, b_2]$.
 If $h \neq \id $ there exists a small interval $I' = [b'_1, b'_2]$ 
 in the support of $h$ such that $I'$ and $h(I')$ are disjoint. 
 Use Theorem \ref{TheoremA} to construct $u \in B$ so that  $u(I) \subseteq I'$.
\index{Theorem \ref{TheoremA}!consequences}%
 Then the supports of $\act{u} f$ and $\act{hu}g$ will be disjoint 
 and so these functions commute.
 
Here are the details:
given $(f, g) \in B^2$ there exists numbers $ b_1 < b_2$  in $\Acal =(IP \cdot A) \cap \Int(I) $ 
so that $\supp f \,\cup \,\supp g \subset [b_1,b_2]$;
given $h \neq \id$ pick  $t_* \in \supp h$.
Two cases now arise. 
If $t_* < h(t_*)$ the continuity of $h$ allows one to find  $b_1'$, $ b'_2$ in  $ \Acal$
so that 
$
b'_1 < t_* < b'_2 < h(b'_1).
$
By Theorem \ref{TheoremA} 
\index{Construction of PL-homeomorphisms!applications}%
there exist then a PL-homeomorphism $u_2$  in $G(\R;A,P)$ 
such that 
\[
u_2(b_1) = b'_1, \; u_2(b_2) = b'_2.
\]
\emph{A priori},
the function $u_2$  is only known to belong to $G(\R;A,P)$;
its support, however may not be bounded.
We modify therefore $u_2$ like this.
Choose numbers $a$ and $c$ in $\Acal$ with 
$a < \min \{b_1, b'_1\}$ and $\max\{b_1, b'_2\} < c$.
By Theorem \ref{TheoremA} there exist then functions $u_1$ and $u_3$ in $G(\R;A,P)$ 
satisfying the restrictions
\[
u_1(a) = a, \; u_1(b_1) = b'_1, \quad 
 \;\text{ and }\;
 u_3(b_2) = b'_2, \; u_3(c) = c.
 \]
A look at the definitions of  $u_1$, $u_2$, $u_3$ shows next
that there exists a function $u$
 which agrees with $u_1$ on $[a, b_1]$, 
 is equal to $u_2$ on $[b_1, b_2]$, 
 agrees with $u_3$ on $[b_2, c]$ and is the identity outside of $[a, c]$.
 This function $u$ is then a PL-homeomorphism in $B(I;A,P)$
  and its construction guarantees 
  that $\act{u}f$ and $\act{hu} g$ commutate.
 
 If $ h(t_*) < t_*$ then  $t_* < h^{-1}(t_*)$; 
 so the previous argument holds with $h^{-1}$ in place of $h$.
 It implies  that there exists $u \in B(\R;A,P)$
so  that $\act{u}f$ and $\act{h^{-1}u} g$ commute,
 whence $\act{hu}f$ and $\act{u} g$ commute.
 It suffices thus  to exchange $f$ and $g$ at the beginning of the argument.
 
Proposition \ref{prp:Higmans-Theorem1} therefore applies in both cases 
and proves that $B'$ is a simple, non-abelian group unless $B'' = \{\id \}$.
We are left with verifying that $B$ is not metabelian.
Let $f$ be an element of $B \smallsetminus \{\id\}$.
Then its support is bounded and so there exists $g \in B$ with $ \supp f \neq  g(\supp f)$,
whence $[f,g] = f \circ \act{g} f^{-1} \neq \id$. 
Similarly, there exists $h \in B$ with $\supp [f,g] \neq h(\supp [f,g])$,
whence $ [[f,g], h] \neq \id$.
Since $h$ can be chosen in $B'$ we have proved that $B'' \neq \{\id\}$. 
\end{proof}
%
\section{Construction of  homomorphisms into the slope group}
\label{sec:11}
%
The homomorphisms $\sigma_-$ and $\sigma_+$,
introduced in  section \ref{ssec:3.2},
record the slopes of the elements $g \in G(I;A,P)$ near the end points.
In this section, we construct similarly defined homomorphisms $\nu_\Omega$.
In Section \ref{sec:12}, then, 
they will be used to describe the abelianization of the group $G([0, \infty[\;;A,P)$.
%
\subsection{Definition of the homomorphism $\nu$}
\label{ssec:11.1}

Let $G$ be one of the groups $G(I;A,P)$, 
the parameters $I$, $A$ and $P$ being arbitrary 
except that they satisfy the non-triviality restrictions \eqref{eq:Non-triviality-assumption}. 
Each $G$-orbit $\Omega$ in $A$ gives rise to a function
\begin{equation}
\label{eq:11.1}
\index{Homomorphism!13-nu-Omega@$\nu_\Omega$}%
\nu_\Omega \colon G\to P,\quad 
f \mapsto \prod\nolimits_{\alpha\in \Omega} f'(a_+)/f'(a_-).
\end{equation}

In the above,
 $f'(a_+)$ and $f'(a_-)$ are short for 
$\lim_{t\searrow a}f'(t)$ and $\lim_{t\nearrow a}f'(t)$, respectively; 
the product $\prod f'(a_+)/f'(a_-)$ is defined since all but finitely many factors are equal to 1.  
Since $\Omega$ is a $G$-orbit the chain rule implies that $\nu_\Omega$ is a homomorphism. 
As each homeomorphism $f \in G(I;A,P)$ has only finitely many singularities,
all but finitely many values $\nu_\Omega(f)$ are equal to 1.

In order to combine these homomorphisms $\nu_\Omega$ 
into a single homomorphism $\nu$, 
we introduce the set $(A \cap I)_\sim$ of $G$-orbits of $A\cap I$ 
\label{notation:orbit-set}%
and the free abelian group on this set.
We denote the latter group by  $\Z[(A\cap I)_\sim]$  
\label{notation:Z-orbit-set}%
and then define the homomorphism $\nu$ like this:
\begin{equation}
\label{eq:11.2}
\index{Homomorphism!13-nu@$\nu$}%
\nu \colon  
\begin{cases}
G(I;A,P) &\longrightarrow \Z[(A\cap I)_\sim] \otimes_{\Z}  P
\\
\phantom{aaaa}f &\longmapsto \sum\nolimits_{\Omega \subset A\cap I}
\Omega\otimes\nu_\Omega(f).
\end{cases}
\end{equation}
%
\subsection{Properties of $\nu$}
\label{ssec:11.2}
\index{Homomorphism!13-nu@$\nu$!properties}%
%
(i)  If $I = \R$ then $G$ acts transitively on $A\cap I = A$ 
and so $\Z[(A\cap I)_\sim]$ is infinite cyclic; 
moreover, for each $f$ in $G$ one has
\begin{equation*}
\prod_{a\in A}f'(a_+)/f'(a_-) =
(\lim_{t\nearrow \infty}f'(t)) \cdot  (\lim_{t\searrow -\infty} f'(t))^{-1} =\sigma_+(f) \cdot \sigma_-(f)^{-1},
\end{equation*}
and so $\nu$ can be recovered from $\sigma_-$ and $\sigma_+$.

(ii) If $I$ is bounded on one side but not on both sides, 
$\nu$ is surjective, as we shall see in section \ref{ssec:12.1};  
if, however, $I$ is bounded $\nu$ is never surjective.  
Indeed, if $f$ has bounded support 
the product $\prod\nolimits _{a\in A \cap I} f'(a_+)/f'(a_-)$ equals 1; 
thus the image of $\nu$ is contained in the kernel of the map
\begin{equation}\label{eq:11.3}
\varepsilon \colon \Z[(A\cap I)_\sim] \otimes P\longrightarrow P,\quad \Omega \otimes p \mapsto p.
\end{equation}

\subsubsection{Naturality of  $\nu$}
\label{sssec:11.2a}
%
The construction of $\nu$ is natural in two situations detailed by the next two lemmata.
\begin{lemma}
\label{lem:Commutative-square-1}
\index{Homomorphism!13-nu@$\nu$!properties|(}%
Let $I_1$ be a closed subinterval of $I$, 
let $P_1$ be a subgroup of $P$ and let $A_1$ be a $\Z[P_1]$-submodule of $A$.  
Let $\mu\colon  G_1 = G(I_1;A_1,P_1) \incl G = G(I,A,P)$ denote the obvious inclusion 
and let
$\iota_* \colon (A_1\cap I_1)_{\sim_1} \to (A\cap I)_\sim $ be the map 
that is derived from the inclusion $\iota \colon (A_1 \cap I_1) \incl A \cap I$ by passage to the quotients.  
Then the square
\begin{equation}
\label{eq:11.4}
\xymatrix{
G(I_1;A_1,P_1) \ar@{->}[r]^-{\nu_1} \ar@{->}[d]^-\mu 
&\Z[(A_1\cap I_1)_{\sim_1}] \otimes P_1 \ar@{->}[d]^-{\Z[\iota_*]\otimes (P_1\incl P)}\\
G(I;A,P) \ar@{->}[r]^-\nu &\Z[(A\cap I)_\sim] \otimes P
}
\end{equation}
is commutative.
\end{lemma}  

\begin{proof}
Let $\Omega_1$ be a $G_1$-orbit in $A_1 \cap \Int(I_1)$ 
and consider an element $f_1 \in G_1$.
By definition \eqref{eq:11.2} one has 
\begin{equation}
\label{eq:Formula-for-nu1-f1}
(\nu_1)_{\Omega_1}(f_1) = \prod\nolimits_{a_1 \in \Omega_1}  f_1'(a_{1, +})/ f'_1(a_{1, -}).
\end{equation}
If one passes from $f_1$ to $f = \mu(f_1)$, 
the factors in the above product do not change, 
but $\Omega_1$ may no longer be an orbit of the larger group $G$.
If this happens then several $G_1$-orbits, 
say $\Omega_{1,1}$, \ldots, $\Omega_{1,k}$,
 coalesce into a single orbit $\Omega$.
\footnote{Infinitely many orbits may coalesce, but as $f$ has only finitely many singularities,
only finitely many of these orbits need to be taken into account.}
Then 
\[
\nu_\Omega (f) = \nu_{\Omega_{1, 1}}(f_1) \dotsm  \nu_{\Omega_{1, k}}(f_1).
\]
Since the tensor product $-\otimes_{\Z} -$ is bi-additive,
this value coincides with the image of $\nu_1(f_1)$ under $\Z[\iota_*] \otimes (P_1 \incl P)$.
\end{proof}

\begin{lemma}
\label{lem:Commutative-square-2}
Suppose $I_1$, $I_2$ are intervals, $P$ is a subgroup of $\R^\times_{> 0}$ 
and $A_1$, $A_2$ are $\Z[P]$-submodules of $\R_{\add}$.  
For $i \in \{1,2\}$,
let $G_i\leq G(I_i;A_i,P)$ be a subgroup.
Assume there exists a \emph{finitary} PL-homeomorphism $\varphi \colon I_1\iso I_2$ 
which maps $A_1\cap I_1$ onto $A_2 \cap I_2$,
 and induces by conjugation an isomorphism
\begin{equation*}
\beta \colon  G_1\longrightarrow G_2,\quad g_1 \longmapsto \ \text{unique extension of } \act{\varphi}(g_1\restriction I_1)\ \text{in}\ G(I_2;A_2,P).
\end{equation*}
Then $\varphi$ sends $G_1$-orbits of $A_1\cap I_1$ onto $G_2$-orbits of $A_2\cap I_2$ 
and the square
\begin{equation}\label{eq:11.5Modified}
\xymatrix{
G_1\ar@{^{(}->}[r]  \ar@{->}[d]^-{\beta} 
& G(I_1;A_1,P) \ar@{->}[r]^-{\nu_1}
&\Z[(A_1\cap I_1)_{\sim_1}]\otimes P\ar@{->}[d]^-{\Z[\varphi_*]\otimes \id}\\
G_2 \ar@{^{(}->}[r]  
& G(I_2; A_2,P)  \ar@{->}[r]^-{\nu_2}
&\Z[(A_2\cap I_2)_{\sim_2}]\otimes P
}.
\end{equation}%
is commutative.
\end{lemma}  

\begin{proof}
The equation $\varphi(g_1( t_1)) = (\act{\varphi} g_1) ( \varphi(t_1))$ 
holds for every $(g_1,t_1) \in G_1 \times I_1$ 
and shows that $\varphi$ maps  $G_1$-orbits in $I_1$ onto $G_2$-orbits in $I_2$.
Moreover,
if $t_1 \in A_1 \cap I_1$ then $\varphi(t_1)$ in $A_2 \cap A_2$ by hypothesis.
Fix now a $G_1$-orbit $\Omega_1$  in $A_1 \cap \Int(I_1)$, 
set $\Omega_2 = \varphi(\Omega_1)$
and pick $a_1 \in \Omega_1$ and $g_1 \in G_1$.
Set $a_2 = \varphi(a_1)$. 
The chain rule justifies then the computation
\begin{align*}
\left(\act{\varphi} g_1\right)({a_2}_+) 
&= 
\varphi'(g_1(\varphi^{-1}(a_2))_+) 
\cdot 
g_1' (\varphi^{-1}(a_2))_+) 
\cdot  
(\varphi^{-1})'({a_2}_+)\\
&= 
\varphi'(g_1(a_1)_+) \cdot g_1'({a_1} _+) / \varphi'({a_1}_+)
\end{align*}
for the right-hand derivative at $a_2$.
A similar result holds for the left-hand derivative at $a_2$.

Let $S_1 \subset \Omega_1$ denote the union of the set of singularities of $g_1$, 
the set of singularities of $\varphi$ 
and the preimage of the set of singularities of $\varphi$ under $g_1$.
Then $S_1$ is a finite set and so the following calculation is licit:
\begin{align*}
\nu_{\Omega_2}\left( \act{\varphi}g_1 \right)
&=
\prod\nolimits _{a_1 \in S_1} 
\frac{
\varphi'(g_1(a_1)_+) \cdot g_1'((a_1) _+) / \varphi'((a_1)_+)
}%
{\varphi'(g_1(a_1)_-) \cdot g_1'((a_1) _-) / \varphi'((a_1)_-)}\\
&=
\prod\nolimits _{a_1 \in S_1}  \frac{\varphi'(g_1(a_1)_+) } {\varphi'(g_1(a_1)_-)}
\cdot
\prod\nolimits _{a_1 \in S_1} \frac{g_1'(a_1)_+) } {g_1'(a_1)_-)}
\cdot
\prod\nolimits _{a_1 \in S_1} \frac{ \varphi'((a_1)_-)}{ \varphi'((a_1)_+)} \\
&=
\prod\nolimits _{a_1 \in S_1}  \frac{\varphi'(g_1(a_1)_+) } {\varphi'(g_1(a_1)_-)}
\cdot
\prod\nolimits _{a_1 \in S_1} \frac{ \varphi'((a_1)_-)}{ \varphi'((a_1)_+)} 
\cdot \nu_{\Omega_1}(g_1).
\end{align*}
A factor  $\varphi'(g_1(a_1)_+) / \varphi'(g_1(a_1)_-)$ in the first product in the last line 
is distinct from $1 \in P$ 
if, and only if, $g_1(a_1)$ is a singularity of $\varphi$; 
so the first product in nothing but the reciprocal of the second one,
and thus $\nu_{\Omega_2} (\act{\varphi}g_1)$ coincides with $\nu_{\Omega_1}(g_1)$.
\end{proof}

\subsubsection{Comment on Lemma \ref{lem:Commutative-square-2}}
\label{sssec:11.2b}
%
The hypotheses of  Lemma \ref{lem:Commutative-square-2} require,
\emph{inter alia},
that $\varphi$ be a \emph{finitary} PL-homeomorphism.
The question arises 
whether the conclusion of Lemma \ref{lem:Commutative-square-2} 
remains valid if $\varphi$ is an infinitary PL-homeomorphism,
the remaining hypotheses being as before.
The answer turns out to be in the negative.
\index{Homomorphism!13-nu@$\nu$}%

Indeed,
let $I$ be the unit interval $[0,1]$ 
and  set $G = G(I;\Z[1/2], \gp(2))$.
Our first aim is to construct a non-surjective monomorphism $\mu$ of $G$.
Define $\varphi \colon \R \to \R$  to be the function 
that is the identity outside of $\Int(I)$ 
and   the affine interpolation of the decreasing sequence of points  $n \mapsto s_n$
with $s_0 = (1,1)$ and 
\[
s_n = (a_n, b_n) 
=
 \left(
 \tfrac{3}{2} \left(\tfrac{1}{2}\right)^n,  \left(\tfrac{1}{2} \right)^{2n-1} 
 \right)
\]
in the interior of $I$.
The sequence $s$ begins thus: 
\[
(1,1), \quad \left(\tfrac{3}{4}, \tfrac{1}{2}\right), 
\quad \left(\tfrac{3}{8}, \tfrac{1}{8}\right),
\quad \left(\tfrac{3}{16}, \tfrac{1}{32}\right),
\quad \ldots.
\]
For each $n \in \N$, 
the slope of $\varphi$ on the interval $[a_{n+1}, a_n]$ is $(1/2)^{n-1}$,
as one easily verifies.
It follows, first,  that $\varphi$ maps $A = \Z[1/2]$ onto itself  and then 
that conjugation by $\varphi$ induces an automorphism of the kernel $N$ of
the homomorphism $\sigma_- \colon G \epi \gp(2)$.
Let  $f \in G$ be the PL-homeomorphism
whose rectangle diagram is displayed 
in the middle rectangle of the following figure;
it generates a complement of $N$ in $G$.
We claim that the conjugate $g = \act{\varphi}f$ of $f$ is a finitary PL-homeomorphism
and that it lies in $G$.
To verify this assertion,
we study the rectangle diagram of $g = \varphi \circ f \circ \varphi^{-1}$.

%
\input{chaptC.11.2b.fig1.tex}


We claim $g$ is the finitary PL-homeomorphism
with rectangle diagram as shown on the right of the next figure.
To justify this assertion 
we compute the image of $g$ at $b_n$ for $n \geq 2$.
Using the definitions of $g$, and of the numbers $a_j$, $b_j$,
the image $g(b_n)$ of $b_n$ can be computed for $n \geq 2$ like this:
\[
b_n \longmapsto a_n 
\longmapsto \tfrac{1}{2} a_n = a_{n+1} 
\longmapsto b_{n+1} = \tfrac{1}{4} b_n.
\]
It follows 
that $g$ is linear with slope $1/4$ on the interval $]0, b_2] = \;]0, 1/8[$.
Moreover, the above rectangle diagram allows one to see 
that the slope of $g$ on the interval $[1/8, 1/4[$ is $1/4$,
on $[1/4, 1/2]$ it  is $1/2$,
that it equals 1 on $[3/8, 3/4]$  
and 2  on $[3/4, 1]$.
The rectangle diagram of $\act{\varphi}f$ is therefore 
as depicted on the right of the following figure.
%
\input{chaptC11.2b.fig2and3.tex}
%

We infer, first, 
that  $g = \act{\varphi}f \in G$ and so
$\im \mu = N \cdot \gp(\act{\varphi}f ) = N \cdot \gp(f^2)$
is a subgroup of index 2 in $G = N \cdot \gp(f)$.
Since $IP \cdot A = (2-1) \cdot A = A$, 
Corollary \ref{CorollaryA1} shows next
that the group $G = G([0,1];A,P)$ acts transitively on the set 
$\Omega = A \, \cap \; ]0,1[$;
as $\varphi (A) = A$ we see, in addition, that $\varphi$ maps the orbit $\Omega$ onto itself.

Set $G_1 = G([0,1]; \Z[1/2],\gp(2))$ and $G_2 = \im \mu$ 
and let $\beta \colon G_1 \iso G_2$ be the isomorphism induced by $\mu$.
Then $G_1$, $G_2$ and $\varphi$, $\beta$ satisfy the assumptions of Lemma 
\ref{lem:Commutative-square-2},
\emph{except for the fact that $\varphi$ has infinitely many breaks}.

The rectangle diagram of $f$ shows next that 
\[
(\nu_1)_{\Omega} (f) =  (1/\tfrac{1}{2}) \cdot (2/1)  = 4,
\]
while 
the rectangle diagram of $\act{\varphi} f$ shows
that
 \[
(\nu_2)_{\Omega} (\act{\varphi}f) 
=  
(\tfrac{1}{2}/\tfrac{1}{4})\cdot  (1/\tfrac{1}{2}) \cdot (2/1)  = 8.
\]
So $(\nu_1)_\Omega (f) \neq  (\nu_2)_\Omega (\act{\varphi}f)$.
Since $\Z[\varphi_*] \otimes \id$ is the identity of $\Z[(A \cap I)_\sim] \otimes P$
the square \eqref{eq:11.5Modified} is therefore not commutative.
\index{Homomorphism!13-nu@$\nu$!properties|)}%

\begin{remark}
\label{remark:Generalization-counterexample}
The above counter-example belongs to a larger class of examples.
Let $P$ be a cyclic group with generator $p < 1$ 
and $A$ a $\Z[P]$-module. 
Pick $b \in A_{>0}$ and consider the group $G_1 = G([0,b];A,P)$.
According to section \ref{sssec:18.5},
there exists then, for every integer $m \geq 2$, 
an increasing PL-homeomorphism $\varphi_m$, defined on the interval $[0,b]$,
that maps each $G_1$-orbit of $A \cap\, [0,b]$ onto itself
and induces by conjugation an (injective) endomorphism $\mu_m$ of $G_1$ 
with the following properties.
\begin{itemize}
\item $\mu_m$ maps the subgroup $B = B([0,b];A,P)$ onto itself;
\item $G_2= \im \mu_m$ has index $m$ in $G_1$;
\item $\mu_m(f)'(0_+)  = (f'(0_+)^m$ and $\mu_m(f)'(b_-)  = f'(b_-)$ 
for every $f \in G_1$;
\item $\Z[\varphi_*] \otimes \id$ is the identity of $\Z[(A \cap I)_\sim] \otimes P$.
\end{itemize}
If we set $I = I_1 = I_2 = [0,1]$, 
put $A = A_1 = A_2$ and $G_1 = G(I;A,P)$
all the hypotheses of Lemma \ref{lem:Commutative-square-2} are fulfilled,
except for the fact that the PL-homeomorphism $\varphi$ is infinitary.
The conclusion of the lemma does not hold;
indeed, as stated in point (ii) (at the top of page  \pageref{eq:11.3}),
the product $\prod\nolimits _{a\in A \cap I} f'(a_+)/f'(a_-)$ equals 1 for every $f \in G_1$.
Pick now $f \in G_1$ with $f'(0_+)  \neq 1$.
Then $\mu_m(f)'(0_+) \neq f'(0_+)$,
and as $\mu_m(f)'(b_-) = f'(b_-)$,
there must exist a $G_1$-orbit $\Omega_1$ in $A\; \cap \;]0, b[$
with image $\Omega_2$ so that
\[
(\nu_1)_{\Omega_1} (f) \neq (\nu_2)_{\Omega_2} (\act{\varphi}f).
\]
\end{remark}
%
\section{Investigation of the abelianization of the subgroup $B$}
\label{sec:12}
%
In this section, 
we determine the abelianizations of the groups
$G([0,\infty[\,;A,P)$ and $G_1  = \ker(\sigma_- \colon G([0,\infty[\,;A,P) \to P)$
and use then this knowledge to study the abelianization of $B([0,\infty[\; ; A, P)$.
\index{Homomorphism!13-nu@$\nu$!applications}%
\label{notation:G1}%
%
\subsection{Abelianization of $G([0, \infty[\, ; A,P)$}
\label{ssec:12.1}
%
The following result describes the abelianizations of $G = G([0,\infty[; A,P)$ and of $G_1$
with the help of the homomorphism $\nu$.
\begin{proposition}
\label{PropositionC2}
\index{Group G([0,infty[;A,P)@Group $G([0, \infty[\;;A,P)$!abelianization}%
\index{Homomorphism!13-nu@$\nu$!applications}%
\index{Submodule IPA@Submodule $IP \cdot A$!significance}%
Suppose $I$ is bounded on one side but not on both sides.  
Then the abelianization $\nu_{\ab}$ of $\nu \colon G(I;A,P)\to \Z[(A\cap I)_\sim]\otimes P$ is an isomorphism.
\end{proposition}

\begin{proof}  
Suppose first that \emph{the endpoint of $I$  is in $A$}.  
There exists then an \emph{affine} homeomorphism $\varphi$ in $\Aff(A,\{\pm 1\})$ 
with $\varphi(I) = [0,\infty[$ and such that 
\[
\act{\varphi}G(I;A,P) = G([0,\infty[\,;A,P).  
\]
The naturality of $\nu$, 
as described in Lemma  \ref{lem:Commutative-square-2}, 
implies thus
that it suffices to establish the claim for $I = [0,\infty[$.  

Set $G = G([0,\infty[;A,P)$ and let $\TT \subset \; ]0,\infty]$ be a set of representatives 
of the $G$-orbits $(IP \cdot A + b) \cap I$ in $A\, \cap \, \Int(I)$.
The homomorphism $\nu \colon G \to  \Z[(A \cap I)_\sim] \otimes P$
induces a homomorphism 
\[
\nu_{\ab} \colon G_{\ab} \longrightarrow \Z[(A \cap I)_\sim] \otimes P.
\]
We claim it is bijective.

To establish this contention,
we construct first a suitable generating set of $G_{\ab}$.
According to Proposition \ref{PropositionB3}, 
the group $G$ is generated by the subset
\begin{equation*}
\{g(0,p) \mid p\in P\}\cup \{g(b,p)\mid b\in \Acal \cup \TT \text{ and } p \in P\},
\end{equation*}
where $\Acal$ is a set of positive generators of $A$ 
and $\TT$ is a set of positive representatives of the cosets of $IP\cdot A$ in $A$.  
Our aim is now to verify
that the canonical image of the set 
\begin{equation}
\label{eq:Generating-set-Gab}
\{g(0,p) \mid p\in P\}\cup \{g(b,p)\mid b \in \TT \text{ and }p\in P\}
\end{equation}
generates $G_{\ab}$. 
Indeed, 
given $a_*$ in $\Acal$ there exists $b_* \in \TT$ so that  $b_* \in IP \cdot A + a_*$.
The proof of Proposition \ref{PropositionB3} furnishes next a sequence of elements 
\begin{align*}
&b_*,\\
&b_0= p_\alpha \cdot b_*, \quad
b_1 = 
b_0 + \varepsilon_1 (1-p_1) a_1, \ldots,
b_\ell = b_{\ell-1} + \varepsilon_\ell (1-p_\ell)a_\ell, \\
&p_\omega \cdot a_* = (1-p_\alpha) b_* + b_\ell,\\ 
& a_*
\end{align*}
whose members lie in  $A \, \cap \; ]0, \infty[$.
A closer look at the construction of this sequence discloses 
that, for each pair $(b,\bar b)$ of consecutive members of the sequence
and each $p \in P$,
the generator  $g(\bar b,p)$ is conjugate to $g(b,p)$.
It follows that $g(a_*, p)$ is conjugate to $g(b_*,p)$,
and so the canonical image of the set \eqref{eq:Generating-set-Gab} generates $G_{\ab}$.

The PL-homeomorphisms $g(0, p)$ and $g(t, p)$ have a single singularity at $0$, 
respectively at $t$,
and so the definition of  $\nu$, 
given by formulae  \eqref{eq:11.1} and \eqref{eq:11.2},
show that $\nu(g(0,p)) = \{0\} \otimes p$ and $\nu(g(t,p) = G \cdot t \otimes p$.
The co-domain of the homomorphism $\nu_{\ab}$ is the tensor product 
$\Z[(A \cap I)_\sim] \otimes P$
and its first factor $\Z[(A \cap I)_\sim] $ is free abelian with basis 
$\{0\} \cup \{G \cdot t \mid t \in \TT\}$.
The definitions of $\nu$ and of the set \eqref{eq:Generating-set-Gab} imply thus, first of all,
that $\nu_{\ab}$ is \emph{surjective}.
Consider now a non-trivial element of $g \cdot [G,G] \in G_{\ab}$.
In view of relations \eqref{eq:8.2}
there exists then finitely many elements 
\[
g(a_1, p_1), \ldots, g(a_k,p_k) \text{ with } 
a_1 < a_2 < \cdots < a_k,
\]
with each $a_j \in \{0\} \cup \TT$,
 each $p_j \in P\smallsetminus \{1\}$,
all in such a way that 
\[
g \cdot [G,G] 
= 
g(a_1, p_1) \cdot g(a_2, p_2) \cdots g(a_k, p_k) \cdot [G,G].
\]
The fact that the tensor product $\Z[(A \cap I)_\sim] \otimes P$ is a direct sum of the groups 
$\{0 \} \otimes P$ and $G \cdot t \otimes P$,
finally, 
allows us to conclude 
that $\nu_{ab} (g \cdot [G,G]) \neq 0$.
Thus $\nu_{ab}$ is \emph{injective}.
\smallskip

Suppose now that  \emph{the endpoint of $I$ is not in $A$}.  
If $I = I_1$ is bounded from above, 
the affine homeomorphism $\varphi = -\id$ sends $I_1$ onto the interval $I_2 = -I_1$,
which is bounded from below,
and it maps the orbit space 
\[
G(I_1;A,P) \backslash (A\cap I_1)
\]
onto the space $G(I_2; A, P) \backslash (A \cap I_2)$.
Lemma \ref{lem:Commutative-square-2} then shows
that the homomorphism 
\[
(\nu_1)_{\ab} \colon G(I_1;A,P)_{\ab} \to \Z[(A \cap I_1)_{\sim_1}] \otimes P
\]
is bijective
if  $(\nu_2)_{\ab} \colon G(I_2;A,P)_{\ab} \to \Z[(A \cap I_2)_{\sim_2}] \otimes P$ has this property.

We may, and shall, therefore assume 
that $I$ is bounded from below.
We discuss first the special case 
where $J$ is the  \emph{open} interval $ ]0, \infty[$ with endpoint $0$.
Then $G_1 = G(J;A,P)$ is nothing but the kernel of
\begin{equation*}
\sigma_-\colon  G([0,\infty[\, ;A,P) \to P.
\end{equation*}
We claim that $\nu \colon G([0, \infty]; A,P) \to  \Z[(A\ \cap [0,\infty[)_{\sim}] \otimes P$
induces an isomorphism
\begin{equation}
\label{eq:Isomorphism-for-G1}
(\nu_1)_{\ab} \colon (G_1)_{\ab} = \ker (\sigma_-)_{\ab}
\longrightarrow  \Z[(A\ \cap\ ]0,\infty[)_{\sim 1}] \otimes P.
\end{equation}

This assertion can be deduced from the case established before like this.
First,
 $G_1$ is the union of its subgroups $G([a, \infty[\;;A,P) $ 
with $ a \in A_{>0}$;
in view of Corollary \ref{CorollaryA1}, 
the $G_1$-orbits and the $G$-orbits coincide therefore on $A\ \cap\ ]0,\infty[$.
The resulting embedding 
\[
(A\ \cap\ ]0,\infty[)_{\sim_1} \mono (A \cap [0,\infty[)_\sim
\] 
gives then rise to an extension
\begin{equation}
\label{eq:12.2New}
\Z[(A\ \cap\ ]0,\infty)_{\sim 1}] \otimes P\mono \Z[(A \cap [0,\infty[)_\sim] \otimes P\epi P
\end{equation}
of abelian groups.
Next, the extension $G_1 \lhd G \epi P$ 
has the splitting $\sigma \colon P \mono G$  
that sends $p$ to $g(0,p)$. 
In addition,
$G $ acts trivially on $(G_1)_{\ab}$.  
Indeed, given $p \in P$ and an element $g_1 \in G_1$ 
there exists a positive number $b \in A$ 
such that the support of $g_1$ is contained in $[b, \infty[$.
Theorem \ref{TheoremA} 
\index{Theorem \ref{TheoremA}!consequences}%
and Corollary \ref{CorollaryA1}  allow us thus
to find an element $h \in G_1$
so that $h(t) = pt$ for each $t \geq b$.
It follows that $\act{g(0,p)}g_1(b,p') = \act{h}g_1(b,p')$. 
The fact just proved implies that the equation
\begin{equation*}
[G,G_1] = [G_1\cdot \sigma(P)] = [G_1,G_1]
\end{equation*}
is valid.
The fact that the extension $G_1 \lhd G \epi P$ splits
and 5-term exact sequence in homology
(see, \eg{}\cite[Corollary 8.2]{HiSt97})
\index{Homology Theory of Groups!5-term exact sequence}%
allows us therefore to infer
that this extension $G_1 \lhd G \epi P$ abelianizes to the exact sequence 
\begin{equation}\label{eq:12.3}
(G_1)_{\ab} \mono G_{\ab}\epi P.
\end{equation}
This extension, extension \eqref{eq:12.2New}, 
Lemma \ref{lem:Commutative-square-1} 
and the result for the  case $I = [0, \infty[$,
imply, finally,
that the map $(\nu_1)_{\ab}$, 
described by formula \eqref{eq:Isomorphism-for-G1},
is an isomorphism. 
Moreover,
since every translation $\tau_a \colon t \mapsto t+a$ is affine,
the assumptions on $\varphi$, 
stated in Lemma \ref{lem:Commutative-square-2},
are satisfied by every $ \tau_a$ with $a \in A$,
and so the preceding result allows us to infer
that the homomorphism 
$\nu_a \colon G(\,]a, \infty[\,; A,P) \to\Z[(A \,\cap \;]a, \infty[ )_{\sim}] \otimes P$
induces an isomorphism
\begin{equation}
\label{eq:Isomorphism-for-G1-general}
(\nu_{1, a})_{\ab}\colon  G(\,]a, \infty[\,;A,P)_{\ab}
\iso  \Z[(A\, \cap\, ]0,\infty[)_{\sim}] \otimes P.
\end{equation}
for every $a \in A$. 

These isomorphisms permit us 
to establish the claim of Proposition \ref{PropositionC2} 
by an approximation argument.
Suppose $b \in \R \smallsetminus A$ and set $G = G([b, \infty]\;;A,P)$. 
Then
\begin{equation*}
G  = \bigcup\,  \left\{G(\,]a, \infty[\,; A,P) \mid  b < a \in A \right\}.
\end{equation*}
Consider now a couple $(a, a') \in A^2$ with $a < a'$.
The inclusion $]a', \infty[ \,\subset\, ]a, \infty[$ induces then a bijection 
$(A\, \cap \;]a', \infty[\,)_\sim \iso (A \;\cap \;]a, \infty[\,)_\sim$.
On the other hand,
the functor $H \mapsto H_{\ab}$ commutes with colimits.
Since $G$ is the colimit of the subgroups $G(]a,\infty[\,;A,P)$,  
formula 
\eqref{eq:Isomorphism-for-G1-general}
therefore implies
that the homomorphism
\begin{equation}
\label{eq:Isomorphism-second-case}
\nu_{\ab} \colon G([b, \infty[\,;A,P)_{\ab} \longrightarrow 
\Z[(A \cap [b,\infty[)_{\sim}] \otimes P
\end{equation}
is bijective for $b \in \R \smallsetminus A$.
The proof of Proposition \ref{PropositionC2} is now complete.
\end{proof}

Since the group $\Z[(A\, \cap\, ]a, \infty[\,)_\sim]$ is free abelian of rank $\card A/(IP \cdot  A) $,
Proposition \ref{PropositionC2} has the following
\begin{corollary}
\label{crl:PropositionC2}
\index{Quotient group A/IPA@Quotient group $A/(IP \cdot A)$!applications}%
\index{Group G([0,infty[;A,P)@Group $G([0, \infty[\;;A,P)$!finite generation of Gab@finite generation of $G_{\ab}$}%
If $I$ is a half line,
the abelianization of $G = G(I;A,P)$ is \emph{finitely generated} if, and only if,
$P$ is finitely generated and and $IP \cdot A$ has finite index in $A$.
If $G_{\ab}$ is finitely generated, it is free abelian of rank 
\[
 \rk P \cdot \left(1 + \card A/(IP \cdot A)\right) 
\]
if the end point of $I$ lies in $A$, 
and otherwise of rank $ \rk P \cdot \left( \card A/(IP \cdot A) \right) $.
\end{corollary}

\begin{remarks}
\label{remarks:Explicit-description-nu-ab}
(i) The homomorphism $\nu \colon G = G(I;A,P) \to \Z[(A \cap I)_\sim] \otimes P$ 
has been introduced in section \ref{ssec:11.1} for arbitrary intervals $I$ of $\R$.
If $I$ is a half line of one of the forms $[a, \infty[$ or $]a, \infty[$ with $a \in A$,
the group $G$ is generated by elements with a single break, 
namely the PL-homeomorphisms $g(b,p)$ defined in formula \eqref{eq:8.1}.
This formula makes it plain 
that $b$ is the unique singularity of $g(b,p)$ 
and that $g(b,p)(b_+)/g(b,p)(b_-) = p$.
The definition of $\nu$ therefore implies that
\begin{equation}
\label{eq:valueof-nu-at-g(b,p)}
\nu \left(g(b,p) \right) = G \cdot b \otimes p
\end{equation}
for every couple $(b,p) \in (A \cap I) \times P$.

(ii) The above corollary complements Theorem \ref{TheoremB4} 
which states, \emph{inter alia}, 
that the group $G([a, \infty[\,;A,P)$ with $a \in A$ is finitely generated 
precisely if $P$ is finitely generated, $A$ is a finitely generated $\Z[P]$-module 
and $A/(IP \cdot A)$ is finite.
\end{remarks}
%
\subsection{Abelianization of $B(I;A,P)$}
\label{ssec:12.2New}
\index{Subgroup B(I;A,P)@Subgroup $B(I;A,P)$!abelianization}%
In view of Proposition \ref{PropositionC1} we assume that $I = [0,\infty[$.  
Set $B = B([0,\infty[\,;A,P)$.
\index{Proposition \ref{PropositionC1}!consequences}%
Then $B$ is a normal subgroup of 
\begin{equation*}
G_1 = \ker \left( \sigma_- \colon G([0,\infty[\; ;A,P)\to P\right).
\end{equation*}
The abelianization of $G_1$ can be computed by means of Proposition \ref{PropositionC2}.  
Corollaries \ref{CorollaryA2} and \ref{CorollaryA3} imply next
that $G_1/B$ is isomorphic to $\Aff(IP\cdot A,P) \cong (IP \cdot A) \rtimes P$.  
We shall use these facts and results from the Homology Theory of Groups 
to obtain estimates for $B_{\ab} = B/[B,B]$. 
 
We begin with 
\begin{lemma}
\label{LemmaC3New}
\index{Group G(I;A,P)@Group $G(I;A,P)$!action on Bab@action on $B_{\ab}$}%
If $I \neq \R$ or $IP\cdot A =A$, 
then $G(I;A,P)$ acts on $B(I;A;P)_{\ab}$ by the identity.
\end{lemma}

\begin{proof}
Let $f \in G(I;A,P)$ and $g \in B(I;A,P)$ be given.
Since $A$ is non-zero,
we can find elements $a_1 < c_1$ and $c_2 < a_2$ in $A \cap \Int(I)$ 
with $c_1 \leq \supp g \leq c_2$.
Our aim is now to construct an element $f_1 \in G([a_1, a_2]; A,P)$ 
such that 
\begin{equation}
\label{eq:Local-action-by-inner-automorphisms}
\act{f} g = \act{f_1} g.
\end{equation} 

Two cases arise.
If $f$ fixes a point of $A$ 
then $f(a_1) - a_1$ lies in  $IP \cdot A$ (by Theorem \ref{TheoremA}) 
\index{Theorem \ref{TheoremA}!consequences}%
and there exists a PL-homeomorphism $h_1 \in G(\R;A,P)$ with
$h_1(a_1) = a_1$ and $h_1(c_1) = f(c_1)$ (again by Theorem \ref{TheoremA}).
In the same way, one finds a PL-homeomorphism $h_2 \in G(\R;A,P)$ 
with $h_2(c_2) = f(c_2)$ and $h_2(a_2) = a_2$.
Consider now the function $f_1 \colon \R \to \R$ given by
\[
f_1(t) = 
\begin{cases} 
t &\text{ if } t  \in \R \smallsetminus [a_1,a_2],\\
h_1(t) &\text{ if } a_1 \leq t < c_1,\\
f(t) &\text{ if } c_1 \leq t \leq c_2,\\
h_2(t) &\text{ if } c_2 < t \leq  a_2.
\end{cases}
\]
Then $f_1$ is an element of $G([a_1,a_2]; A, P) \subset B(I;A,P)$ 
with property \eqref{eq:Local-action-by-inner-automorphisms}.

If, on the other hand, $f$ has no fixed point in $A$ then $I = \R$, 
whence $A = IP \cdot A$ by hypothesis 
and so we can construct $f_1$ as before.
Lemma \ref{LemmaC3New} now follows from the fact  that $f_1 \in B(I;A,P)$.
\end{proof}

\begin{remark}
\label{remark:Action-on-Bab-for-I=R}
\index{Group G(I;A,P)@Group $G(I;A,P)$!action on Bab@action on $B_{\ab}$}%
If $I = \R$ but $IP\cdot A\neq A$, 
the action of $G(\R;A,P)$ on $B(\R;A,P)_{\ab}$ is not trivial; 
see Proposition \ref{prp:Action-TildeG}.
\end{remark}
\subsubsection{Expressing $B_{\ab}$ as an extension}
\label{sssec:12.2.1}
%
For the next stage of our analysis of $B_{\ab}$ 
we consider the extension
\begin{equation}
\label{eq:12.4}
\xymatrix{1 \ar@{->}[r] &B \ar@{->}[r]^-\mu &G_1 \ar@{->}[r]^-\rho &
(IP\cdot A) \rtimes P\ar@{->}[r] &1.}
\end{equation}
Note that the image of $\rho \colon G_1 \to A \rtimes P$ is $(IP \cdot A) \rtimes P$, 
and not $A \rtimes P$ (see claim (iii) of Corollary \ref{CorollaryA2}). 
As we have seen in section \ref{ssec:8.1},
the group $G_1$ is generated by the PL-homeomorphisms $g(b,p)$ given by
\begin{equation*}
g(b,p)(t) = \begin{cases} t &\text{if $t \leq b$}\\
p(t-b)+b &\text{if $t\geq b$};\end{cases}
\end{equation*}
here $b$ is in $A_{>0}$ and $p$ is in $P$.
For future reference, we note that
\begin{equation}
\label{eq:Vaklur-rho-on-g(b,p)}
\rho(g(b,p))=((1-p)b,p).  
\end{equation}

Extension \eqref{eq:12.4} gives rise to an exact sequence in homology, 
namely
\begin{equation*}
\index{Homology Theory of Groups!5-term exact sequence}%
H_2(G_1)  
\xrightarrow{H_2 \rho}
H_2(IP\cdot A \rtimes P)
\xrightarrow  {d_2}  
B/[G_1,B]  
\xrightarrow {\iota_*} 
(G_1)_{\ab} 
\xrightarrow {\rho_{\ab}}
(IP\cdot A/IP^2\cdot A)\times P \to 1
\end{equation*}
(see, \eg \cite[p.\,47, Ex.6(a)]{Bro94}, or \cite[p. 203, Corollary 8.2]{HiSt97}).  
The abelian group $B_{\ab}$ is centralized by $G_1$ (by Lemma \ref{LemmaC3New}),
and so $[G_1,B] = [B,B]$. 
  
Proposition \ref{PropositionC2} allows us to express the kernel of $\rho_{\ab}$,
and hence the image of $\iota_*$. 
Indeed, 
this proposition and its proof show
that $\ker \rho_{\ab}$ is isomorphic to the kernel $K(A,P)$ of the epimorphism
\index{Subgroup K(A,P)@Subgroup $K(A,P)$!definition|textbf}%
\index{Homomorphism!17-rhobar@$\bar{\rho}$}%
\begin{equation}
\label{eq:12.5New}
\index{Homomorphism!17-rhobar@$\bar{\rho}$}%
\bar \rho \colon \left\{ 
\begin{aligned}
\Z[(A \,\cap \;]0,\infty[)_\sim] \otimes P 
&\longrightarrow 
\left( (IP \cdot A)/(IP^2\cdot A) \right)\times P\\
G_1 \cdot a \otimes p &\longmapsto ((1-p)a+IP^2\cdot A,p).
\end{aligned} \right.
\end{equation}
With the help of $K(A,P)$ and the above 5-term sequence the group $B_{\ab}$ 
can now be written as an extension of the form
\begin{equation}
\label{eq:Bab-as-extension}
\xymatrix{
1 \to \coker H_2(\rho) \ar@{->}[r]^-{d_2} 
&
B_{\ab} \ar@{->}[r]^-{\nu_*} 
&
K(A,P) \to 1,}
\end{equation}
the epimorphism $\nu_*$ being given by 
\begin{equation}
\label{eq;Description-nu-star}
f \cdot [B,B] \longmapsto
\sum\nolimits_{a \in A\,  \cap \,]0, \infty[\,} G_1 \cdot a \otimes f'(a_+)/f'(a_-).
\end{equation}
\begin{example}
\label{example:Elements-in-B}
Given positive elements $a$ and $\Delta$ of $A$ and $p \in P$, 
consider the PL-homeomorphism $b(a,\Delta;p)$ defined by the rectangle diagram 
\begin{equation*}
\label{notation:b(a,Del;p)}
\psfrag{1}{\hspace*{-1.7mm}  \small  $a$}
\psfrag{2}{\hspace*{-3.8mm}  \small    $a + \Delta$}
\psfrag{3}{\hspace*{-9mm}  \small   $a + (p+1)\cdot \Delta$}
\psfrag{11}{\hspace*{-0.6mm}  \small   $a$}
\psfrag{12}{\hspace*{-4mm}\small  $a + p \cdot \Delta$}
\psfrag{13}{\hspace*{-8mm}  \small   $a + (p+1)\cdot \Delta$}
\psfrag{21}{\hspace*{-1.0mm} \small $p$}
\psfrag{22}{  \hspace*{-2.0mm}\small $p^{-1}$}
\psfrag{la1}{\hspace{-8mm} \small $b(a,\Delta;p)$}
\includegraphics[width= 10.0cm]{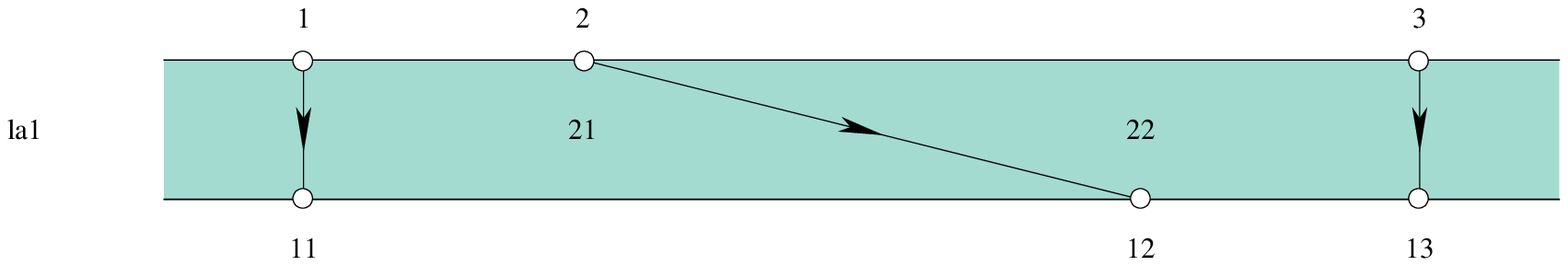}
\end{equation*}
Clearly $b(a,\Delta;p)$ is an element of the group $B([0, \infty[\,;A,P)$.
Its image under $\nu$ is
\begin{equation*}
 G_1 a \otimes p + G_1 (a + \Delta) \otimes p^{-2} 
+ G_1 (a + (p+1) \Delta) \otimes  p ;
\end{equation*}
\cf{}formula \eqref{eq:11.2}.
The description can be simplified,
as the $G_1$-orbit of $a + (p+1) \Delta$ is by Corollary \ref{CorollaryA1} equal to
 \[
 \left(a + (p+1) \Delta + IP \cdot A\right)\, \cap \;]0,\infty[ \;= 
 \left(a + 2 \Delta + IP \cdot A \right)\, \cap \; ]0, \infty[
 \]
 So $G_1 \cdot (a + (p+1) \Delta) = G_1 \cdot (a + 2 \Delta)$
 and thus
\begin{equation}
\label{eq:nu-of-bounded-element}
\nu(b(a,\Delta;P)) = G_1 \cdot a \otimes p + G_1 \cdot (a + \Delta) \otimes p^{-2} 
+ G_1 \cdot (a + 2 \Delta) \otimes p.
\end{equation}
\end{example}

\subsubsection{Analysis of $\coker H_2(\rho)$}
\label{sssec:Analysis-coker-rho}
The group $K(A,P)$ is fairly manageable --- see, \eg{} Proposition \ref{prp:LemmaC5} below --- 
but $\coker H_2(\rho) $ is hard to compute.  
An estimate, which is sometimes helpful, can be obtained as follows.  
Set
\begin{align*}
G_2 &= \ker\left((\sigma_-,\sigma_+) \colon G[(0,\infty[;A,P) \to P\times P\right)\\
&= \ker(\sigma_+ \colon G_1 \to P)
\end{align*}
and fix a positive element $b_0 \in A$. 
The relations \eqref{eq:8.2} imply
that the group  $P_b = \langle g(b_0,p)\mid p\in P\rangle$ is a complement of $G_2$ in $G_1$.  
Thus both $G_1 = G_2 \rtimes P_b$ and $IP\cdot A \rtimes P$ are split extensions, 
and $\rho$ maps $G_2$ onto $IP\cdot A$;  
note, however,  that $\rho$                                                                                                                                                                                                                                                                                                                                                                                                                                                                                                                                            does not map $P_b$ into $P$. 
The spectral sequences associated to the extensions $B\mono G_1 \to  P$ 
and $IP\cdot A \mono (IP\cdot A) \rtimes P\epi P$ 
lead then to the following two commutative ladders with exact rows:
\index{Homology Theory of Groups!spectral sequences}%
\begin{equation}\label{eq:12.6}
\xymatrix{
H_2(P,(G_2)_{\ab}) \ar@{->}[r]^-{d_2} \ar@{->}[d]^-{\rho_{*1}} 
&H_0(P,H_2(G_2)) \ar@{->}[r] \ar@{->}[d]^-{\rho_{*2}} 
&L \ar@{->}[r] \ar@{->}[d]^-{\rho_{*}} 
&H_1(P,(G_2)_{\ab}) \ar@{->}[r] \ar@{->}[d]^-{\rho_{*3}}&0\\
H_2(P, IP\cdot A) \ar@{->}[r]^-{d_2} 
&H_0(P,H_2(IP\cdot A)) \ar@{->}[r] 
&\bar L \ar@{->}[r] &H_1(P,IP\cdot A) \ar@{->}[r] 
&0
}
\end{equation}
and
\begin{equation*}
\xymatrix{0 \ar@{->}[r] 
&L \ar@{->}[r] \ar@{->}[d]^-{\rho_*} 
&H_2(G_1) \ar@{->}[r] \ar@{->}[d]^-{H_2 (\rho)} 
&H_2(P) \ar@{->}[r] 
\ar@{->}[d]^-{\id} &0\\
0 \ar@{->}[r] 
&\bar L \ar@{->}[r] 
&H_2(IP\cdot A\rtimes P) \ar@{->}[r] 
&H_2(P) \ar@{->}[r] &0.}
\end{equation*}
(For more details, see page 362, in particular formula (4.4), in \cite {HiSt74a}.)

The second ladder implies 
that $\coker H_2(\rho)$ is isomorphic to  $\coker \rho_*$ 
and the first one provides one with an estimate for $\coker \rho_*$. 
The resulting estimate for $B_{\ab}$ is summarized by
\begin{proposition}
\label{PropositionC4}
\index{Subgroup B(I;A,P)@Subgroup $B(I;A,P)$!abelianization}%
\index{Subgroup K(A,P)@Subgroup $K(A,P)$!significance}%
If $B = B([0, \infty[\;; A, P)$ 
and if $K(A, P)$ is the kernel of  the homomorphism $\bar{\rho}$
given by formula \eqref{eq:12.5New},
 there exist exact sequences
\begin{gather}
0 \to \coker( L \xrightarrow{\rho_*}  \bar{L}) \to B_{\ab} \xrightarrow {\nu_*} K(A,P) \to 0, 
\label{eq:ses-for-Bab}\\
0 \to K(A,P) 
\longrightarrow 
\Z[A/ (IP \cdot A)] \otimes_{\Z} P \longrightarrow (IP \cdot A/ IP^2 \cdot A) \times P \to 1, 
\label{eq:Description-K(A,P)}\\
\intertext{and}
H_2(P, IP\cdot A)  \xrightarrow{d_2} H_0(P, H_2(IP \cdot A)) \longrightarrow \bar{L} \longrightarrow H_1(P, IP \cdot A) \to 0.
\label{eq:Description-Lbar}
\end{gather}
\end{proposition}
%
\subsubsection{Vanishing results for $H_*(P, IP \cdot A)$}
\label{sssec:12.2c}
The exact sequence \eqref{eq:Description-Lbar} provides an estimate of the group $\bar{L}$.
Under suitable hypotheses on $P$ and $A$,
the first and the last term of this sequence vanish; 
the sequence yields then an explicit description of $\bar{L}$.
Here are some samples.
\begin{lemma}
\label{LemmaC6}
\index{Submodule IPA@Submodule $IP \cdot A$!significance}%
\index{Homology Theory of Groups!H1(P,A)@$H_1(P,A)$}%
\index{Homology Theory of Groups!H2(P,A)@$H_2(P,A)$}%
Suppose one of the following conditions is satisfied.
\begin{enumerate}[(i)]
\item $P$ is (infinite) cyclic. 
\item $A = IP \cdot A$ and $I[P] = \sum \{ (1-p) \Z[P]  \mid p \in P\}$ is a cyclic $\Z[P]$-module.
\item $A = IP \cdot A$ and $A$ is a finitely generated $\Z[P]$-module.
\item $A = IP\cdot A$ and $A$ is noetherian.
\end{enumerate}
Then the homology groups $H_j(P,IP\cdot A)$ vanish  for every $j > 0$.
\end{lemma}

\begin{proof}
(i) If $P$ is cyclic, say $P = \gp(p)$, 
then  $0 \to \Z P \xrightarrow{1-p} \Z{P} \to \Z  \to 0$ 
is a $\Z P$-free resolution of $\Z$ 
and  so $H_1(P,A)\cong \ker( A \xrightarrow{1-p} A)$ vanishes, 
as do all higher dimensional homology groups.

(ii) and (iii). Assume first that $I[P]$ is a cyclic $\Z[P]$-module. 
There exists then $\bar{\lambda} \in \Z[P]$ with $I[P] = \bar{\lambda} \cdot \Z[P]$;
note that $\bar{\lambda} \in I[P]$.
The chain of equalities
\[
A = IP \cdot A = I[P] \cdot A = \bar{\lambda} \cdot \Z[P] \cdot A 
= \bar{\lambda} A
\]
shows then that multiplication by $\bar{\lambda}$ defines a surjective endomorphism 
$\lambda_* \colon A \to A$. 
Since $A$ is a torsion-free $\Z[P]$-module, $\lambda_*$ is injective and hence an automorphism.

If, on the other hand, $A$ is a finitely generated $\Z{P}$-module, 
it follows from $A = IP\cdot A$ and the Cayley-Hamilton Theorem 
that $A$ is annihilated by an element of the form $1-\lambda$ with $\lambda$ in $IP$.
The endomorphism $\lambda_*$ induced by $\lambda$  is thus bijective. 

We exploit next the fact 
that $\lambda$ lies in the center of the (commutative) ring $\Z{P}$.
Let $B$ a right $\Z{P}$-module and $A$ a left $\Z{P}$-module.
Multiplication by $\lambda$ induces an endomorphism of $A$, 
but also of $B$,
and the two endomorphism $\id \otimes \lambda_*$ and $\lambda_* \otimes \id$ 
induce the same endomorphism of $A \otimes_{\Z{P}} B$,
but also the same endomorphism of the torsion functors $\Tor^{\Z P}_j (A,B)$ for $j > 0$ (see, \eg \cite[10.3.3]{LeRo04}).

We apply this fact to the homology groups 
$H_j(P, A) = \Tor^{\Z P}_j ( \Z, A)$. 
Since $\lambda \in IP$, the map $\Tor_j^{\Z P} (\lambda_*, \id)$ is the zero map;
since $\lambda_*$ is an automorphism of $A$,
the map $\Tor_j^{\Z P} (\id, \lambda_*)$ is bijective. 
As the two maps coincide, 
$\Tor_j^{\Z P}(\Z, A)$ must be trivial for $j > 0$. 

(iv) Since $A = IP \cdot A$ the homology group $H_0(P, A)$ is trivial;
as $A$ is noetherian all  higher homology groups $H_j(P, A)$ vanish therefore
by a result of Derek Robinson's  (see \cite[Theorem A]{Rob76} or \cite[10.3.1]{LeRo04}).
\end{proof}

\begin{example}
\label{example:LemmaC6}
\index{Module A@Module $A$!select examples}
In the literature on PL-homeomorphism groups of the form $G(I;A,P)$,
the module $A$ is usually taken to be $\Z[P]$.
To require that $A = IP \cdot A$ 
is then tantamount with demanding 
that $\Z[P| = I[P] \cdot \Z[P]$.
Thus for $A = \Z[P]$,
the requirement that $A = IP \cdot A$ 
implies that $I[P]$ is a cyclic $\Z[P]$-module.
So all homology groups $H_j(P,A)$ vanish by part (ii) of Lemma
\ref{LemmaC6}. 
\end{example}
%
\subsection{Analysis of $K(A;P)$}
\label{ssec:12.3New}
\index{Subgroup K(A,P)@Subgroup $K(A,P)$!analysis}%
\index{Homomorphism!17-rhobar@$\bar{\rho}$}%
We turn now to the  computation of the abelian group  $K(A,P)$, 
the kernel of the epimorphism 
\[
\label{notation:K(A,P)}
\index{Homomorphism!17-rhobar@$\bar{\rho}$}%
\bar\rho \colon\Z[(A \,\cap \,]0,\infty[)_\sim] \otimes P 
\epi (IP\cdot A/IP^2\cdot A) \times P;
\]
this epimorphism is described by formula \eqref{eq:12.5New}.  
Our main result is
\begin{proposition}
\label{prp:LemmaC5}
\index{Quotient group A/IPA@Quotient group $A/(IP \cdot A)$!significance}
\begin{enumerate}[(i)]
\item $K(A,P)$ is trivial if, and only if, $A = IP\cdot A$.
\item  $K(A,P)$ is finitely generated if, and only if, 
either $A = IP\cdot A$, 
or if $A/(IP\cdot A)$ is finite and $P$ is finitely generated.  
\item 
if $K(A,P)$ is finitely generated and $A \neq IP \cdot A$, 
then $K(A,P)$ is free abelian and its rank equals 
 $\rk(P) \cdot \left(\card\left (A/(IP\cdot A)\right)-1\right)$.
\end{enumerate}
\end{proposition}

In the proof of Proposition  \ref{prp:LemmaC5} 
we shall make use of an auxiliary result;
it will be stated and established first.

\subsubsection{On the quotient modules $A/(IP \cdot A)$ 
and $(IP \cdot A)/(IP^2 \cdot A)$}
\label{sssec:12.3a}
%
\begin{lemma}
\label{lemma:Quotients-A/IPA-and-IPA/IP2A}
\index{Quotient group A/IPA@Quotient group $A/(IP \cdot A)$!properties}%
Assume $P$ is an abelian group and  $A$ is an arbitrary  $\Z{P}$-module.
\begin{enumerate}[a)]
\item If $P$ is finitely generated and $IP \cdot A$ has finite index in $A$, 
then $IP^2 \cdot A$ has finite index in $IP \cdot A$.

\item If $A/(IP \cdot A)$ is a torsion group, 
so is $A/(IP^2 \cdot A)$.
\end{enumerate}
\end{lemma}

\begin{proof}
a) Let $\TT$ be a transversal of $IP \cdot A$ in $A$
 and $\PP$ a finite set generating $P$.
 Consider an element $a \in A$.
 Since $\TT$ is a transversal of $A_1 = IP \cdot A$ in $A$,
 there exists $\tau_a \in \TT$ and $a_1  \in A_1$ with $a = \tau_a + a_1$.
 
 As an abelian group,
 $A_1$ is generated by the products $(q-1) \cdot  a$ 
 with $q \in P$ and $a \in A$.
 The identities 
 $q_1 q_2 - 1 = q_1(q_2 - 1) + (q_1 - 1)$ 
 and $q^{-1} - 1 = - q^{-1} ( q-1)$
 then imply that $A_1$ is generated,
 as a $\Z{P}$-module, by the products $(p-1) \cdot a$ with $p \in \PP$ and $a \in A$.
 So the element $a_1$ found before has a representation of the form
 $a_1 = \sum_{p \in \PP} (p-1) \cdot a(p)$.
 Each of the elements $a(p)$, finally, can be written as a sum $\tau_{a(p)} + a_1(p)$.
 
 The representations of $a$ obtained so far give rise to the following chain of equalities:
 \begin{align}
 a  
 = 
 \tau_a + a_1 
  &= 
 \tau_a + \sum\nolimits_{p \in \PP} (p-1) \cdot a(p)\notag\\
 &=
 \tau_a + \sum\nolimits_{p \in \PP} (p-1) \cdot (\tau_{a(p)} + a_1(p))\notag\\
 &=
\left(\tau_a + \sum\nolimits_{p \in \PP} (p-1) \cdot \tau_{a(p)}\right) 
 + 
 \sum\nolimits_{p \in \PP} (p-1) \cdot  a_1(p). \label{eq:Expressing-a}
 \end{align}
 The first summand in equation \eqref{eq:Expressing-a} 
 can take on at most $|A/A_1| + |A/A_1|^{\card \PP}$  values
 and the second summand lies in $IP^2 \cdot A$.
 The equation thus proves that  $IP^2 \cdot A$ has finite index in $A$, 
 and hence also in $A_1 = IP \cdot A$.

b) Fix $a \in A$.
 By hypothesis, there exists a positive integer $n_a$ with $n_a \cdot a \in A_1$.
 Set $a_1 = n_a \cdot a$.
 Then $a_1$ can be written as a sum  $a_1 = \sum_{p \in \PP} (p-1) \cdot a(p)$.
 Each of the elements $a(p)$ lies in $A$ 
 and so there exists, for each $p$, a positive integer $m_{a,p}$ with  $m_{a,p} \cdot a(p) \in A_1$.
 Let $m$ be a common multiple of the integers $m_{a,p}$.
 Then
 \[
 (m \cdot n_a) \cdot a = m \cdot (n_a \cdot a)   = m \cdot \sum_{p \in \PP} (p-1) \cdot a(p)
 = \sum_{p \in \PP} (p-1) \cdot (m \cdot a(p)).
 \]
 As each summand $m \cdot a(p)$ lies in $A_1 = IP \cdot A$,
 the sum $\sum_{p \in \PP} (p-1) \cdot (m \cdot a(p))$ belongs to $IP^2 \cdot A$.
 This proves that $A /(IP^2 \cdot A)$ is a torsion group.
\end{proof}
%
\subsubsection{Proof of Proposition \ref{prp:LemmaC5}}
\label{sssec:12.3b}
%
We begin with an observation.
The set $(A \,\cap \;]0, \infty[)_\sim$ is the space of $G_1$-obits
\footnote{Recall 
that $G_1$ denotes the subgroup $\ker (\sigma_- \colon G([0, \infty[\;;Q,P)$} 
of the interval $]0,\infty[$.
Now each $G_1$-orbit has the form $(IP \cdot A+  a) \cap \,]0, \infty[$
(by Corollary \ref{CorollaryA1});
so the assignment $G_1 \cdot a \mapsto a + IP \cdot A$ defines a bijection 
of $(A\, \cap \,]0, \infty[)_\sim$ onto $A/(IP \cdot A)$.

We embark now on the proof of Proposition \ref{prp:LemmaC5},
starting out with the case where $A = IP \cdot A$. 
Then $(A\,  \cap \,]0, \infty[)_\sim$ is a singleton; 
in addition, $IP \cdot A = IP^2 \cdot A$. 
The map $\bar{\rho}$ thus takes on the form $P \to \{0\} \times P$ and sends $p$ to $(0, p)$.
It is therefore injective and so its kernel $K(A,P)$ is trivial.
Claims (i) and (ii) hold therefore, 
while claim (iii) is valid since its assumption is not fulfilled.

Assume next that $A/(IP \cdot A)$ is finite and $P$ is finitely generated.
Lemma \ref{lemma:Quotients-A/IPA-and-IPA/IP2A} then implies 
that $(IP \cdot A )/( IP^2 \cdot A)$ is finite.
The group $P$ is free abelian, for it is finitely generated by hypothesis 
and torsion-free as a subgroup of $\R^\times_{>0}$.
Hence $\Z[(A\, \cap \;]0, \infty[)_\sim] \otimes P$ is free abelian,
while $(IP \cdot A/IP^2 \cdot A) \times P$ is the direct product of a finite and a free abelian group.
As $\card((A \cap \,]0, \infty[)_\sim)$ coincides with the order of $A/(IP \cdot A)$ 
by the preliminary remark, a straightforward count of ranks proves
that $K(A,P)$ has torsion-free rank $(\card(A/IP\cdot A)-1) \cdot \rk(P)$.
Moreover, being a subgroup of a free abelian group, $K(A,P)$ is free abelian, too.

We continue with the remaining parts of the proof of claims (i) and (ii),
assuming that $A \neq IP \cdot A$. 
Pick a positive number $a \in A \smallsetminus IP \cdot A$,
set $\Delta = a$, fix $p \in P$,  
and consider the PL-homeomorphism $b(a,\Delta;p)$
discussed in Example \ref{example:Elements-in-B}.
Formula \eqref{eq:nu-of-bounded-element}
the shows that
\begin{equation}
\label{eq:Image-in-K(A,P)}
\nu(b(a,\Delta;p)) = G_1 \cdot a \otimes p + G_1 \cdot 2a \otimes p^{-2} 
+ G_1 \cdot 3a \otimes p \in \Z[(A \, \cap \, ]0,\infty[)_\sim] \otimes P
\end{equation}
and this image lies in $K(A,P)$ 
(to see this, use  the exactness of the sequence \eqref{eq:ses-for-Bab}
and the fact that $b(a,\Delta;p) \in B(]0,\infty[\,;A,P)$).

We claim that \emph{$\nu(b(a,\Delta;p))$ is non-trivial for every $p \in P \smallsetminus \{1\}$}.
Two cases arise.
If the order of $a + IP \cdot A$ in $A/(IP \cdot )$ is 3 or higher,
the three orbits $G_1 \cdot a$, $G_1 \cdot 2a$ and $ G_1\cdot 3a$ are pairwise distinct;
the projection $\pi \colon \Z[(A\, \cap \, ]0,\infty[)_\sim \otimes P \to P$
which sends $G_1\cdot a \otimes p$ to $p \in P$
and the remaining direct summands to $1 \in P$,
is therefore surjective.
If, on the other hand,   
$a + IP \cdot A$ has order 2, then  $G_1 \cdot a = G_1 \cdot 3a$,
and the image of $\pi$ is the subgroup $\{p^2 \mid p \in P\}$ of $P$.
It follows that $K(A,P)$ is non-trivial
and that it is infinitely generated  if $P$ is so.
\smallskip

The arguments in the preceding paragraphs 
establish claims (i) and (iii) of Proposition \ref{prp:LemmaC5},
as well as part of its claim (ii).
It remains to be proved that 
\emph{$A/(IP \cdot A)$ is finite
if $K(A,P)$ is finitely generated and, as before, $A \neq IP \cdot A$}.
By the previous paragraph the group $P$ must then be finitely generated. 
We next verify 
that $A/(IP \cdot A)$ is a torsion group.
Suppose $n$ is a positive integer and $a \in A$ an element
so that the numbers $a$, $2a$, \ldots, $4^n a$ 
represent pair-wise distinct cosets of $A/(IP \cdot A)$.
Set $\Delta = a$, choose $p_0 \in P\smallsetminus \{1\}$  
and consider the sequence $n \mapsto a_n $ defined 
by  $a_n = b(4^n \cdot a, \Delta; p)$.
It then follows, 
as in the previous paragraph, 
that this sequence generates a subgroup of $K(A;P)$ 
which is free abelian of rank $n$.
As the group $K(A,P)$ is finitely generated, 
by hypothesis,
its rank is an upper bound for $n$ 
and hence for the order of the elements of $A/(IP\cdot A)$.
If follows, in particular, that $A/(IP\cdot A)$ is a torsion group,
whence $(IP \cdot A) / (IP^2 \cdot A)$ is so
by part b) of Lemma \ref{lemma:Quotients-A/IPA-and-IPA/IP2A}.
But if so,
the exact sequence \eqref{eq:Description-K(A,P)} 
yields an upper bound of the rank of the free abelian group
$\Z[(A \,\cap \,]0,\infty[)_\sim] \otimes P$
and hence of the cardinality of the orbit space $(A \,\cap \,]0,\infty[)_\sim$.
The bijection  $(A\,  \cap \,]0, \infty[)_\sim \iso A/(IP \cdot A)$,
described at the very beginning of the proof,
permits one, finally, to conclude that the group $ A/(IP \cdot A)$ is finite
%
\subsubsection{Application to groups with cyclic group $P$}
\label{sssec:12.3c}
Proposition \ref{prp:LemmaC5} gives concrete information on the group $K(A,P)$ 
under fairly mild assumptions. 
Proposition \ref{PropositionC4} then allows one 
to turn this information into a description of $B_{\ab}$, 
provided one can control the cokernel of the map $\rho_* \colon L \to \bar{L}$.
A situation where such a control is available
arises if $P$ is a cyclic subgroup of  $\Q^\times_{>0}$.
\begin{corollary}
\label{crl:Bab-for-P-cyclic-and-A-locally-cyclic}
\index{Subgroup B(I;A,P)@Subgroup $B(I;A,P)$!abelianization}%
\index{Quotient group A/IPA@Quotient group $A/(IP \cdot A)$!significance}%
Assume $P$ is a cyclic subgroup of $\Q^\times_{>0}$ 
and $A/(IP \cdot A)$ is a torsion group.
Then $B([0, \infty[\,;A, P)_{\ab}$ is free abelian of rank $\card(A/IP \cdot A)-1$ 
if $A/(IP\cdot A)$ is finite, and otherwise of rank $\aleph_0$.
\end{corollary}

\begin{proof}
Since $P$ is cyclic, 
the homology groups $H_j(P, IP \cdot A)$ vanish  in every positive dimension $j$
(see Lemma  \ref{LemmaC6}).
By Proposition \ref{PropositionC4}
the group $B_{\ab}$ is thus isomorphic to 
\[
K(A,P) = \ker \left( \bar\rho \colon\Z[(A \,\cap \,]0,\infty[)_\sim] \otimes P 
\epi (IP\cdot A/IP^2\cdot A) \times P  \right).
\]
Since $A/(IP \cdot A)$ is a torsion group, 
so is $(IP \cdot A)/(IP^2 \cdot A)$ 
by Lemma \ref{lemma:Quotients-A/IPA-and-IPA/IP2A};
the co-domain of $ \bar{\rho}$ is thus torsion-by-infinite cyclic;
the domain of $ \bar{\rho}$,
on the other hand,
is free abelian of rank $\card A/(IP \cdot A)$ 
by the first paragraph of section \ref{sssec:12.3b}.
The claim now follows from the fact 
that $ \bar\rho $ is surjective.
\end{proof}

\begin{example}
\label{example:P-cyclic-A-locally-cyclic}
We begin with a general remark.
Given a group $P $ and a $\Z[P]$-module $A$,
set $G = G([0,1]; A, P)$ and $B = B([0,1];A,P)$.
Consider now the extension $B \xrightarrow{\iota} G \epi G/B$.
This extension gives rise to an exact sequence in homology,
namely
\begin{equation*}
\index{Homology Theory of Groups!5-term exact sequence}%
H_2(G) \xrightarrow{} H_2(G/B) \xrightarrow{d_2} B/[B,G] \xrightarrow{\iota_*} G_{\ab} 
\xrightarrow{} (G/B)_{\ab} \to 0.
\end{equation*}
The group $G$ acts trivially on $B_{\ab}$ (by Lemma \ref{LemmaC3New}), 
whence $[B, G] =[B,B]$.
Since $G/B$ is isomorphic to $P\times P$, 
the above sequence can therefore be rewritten as follows:
\begin{equation}
\label{eq:5-term-sequence-involving-Bab}
H_2(G) \xrightarrow{} H_2(P^2) \xrightarrow{d_2} B_{\ab} \xrightarrow{\iota_*} G_{\ab} 
\xrightarrow{} P^2 \to 0.
\end{equation}
The maps $H_2(G) \to H_2(P^2)$ and $G_{\ab} \to P^2$ occurring in this sequence are induced 
by the projection $(\sigma_-, \sigma_+) \colon G \epi P \times P$.

Assume now that $P$ is cyclic.
Then the projection $(\sigma_-, \sigma_+)$ has a one-sided inverse $\mu \colon P\times P \to G$
(this follows from part (ii) of Corollary \ref{CorollaryA2} 
and its analogue for $\sigma_+$)
and so $H_2(G) \to H_2(P^2)$ is a split epimorphism.
The five term sequence \eqref{eq:5-term-sequence-involving-Bab} implies therefore 
that the sequence $B \xrightarrow{\iota} G \epi P^2$ abelianizes to an exact sequence
\footnote{see sections \ref{sssec:Notes-ChapterC-Brown-Stein} 
and \ref{sssec:Notes-ChapterC-Gal-Gismatullin}
for far better results.}
\begin{equation}
\label{eq:Exact-abelianized-sequence-P-cyclic}
0 \to B_{\ab} \xrightarrow{\iota_*} G_{\ab}  \epi P^2.
\end{equation}
Note that the injectivity of the induced map $\iota_* \colon B_{\ab} \to G_{\ab}$ 
amounts to say that $B' = G'$ or, equivalently, that $G/B'$ is abelian.
\index{Subgroup B([a,c];A,P)@Subgroup $B([a,c];A,P)$!derived group}%
\index{Group G([a,c];A,P)@Group $G([a,c];A,P)$!derived group}%
The conclusion of Corollary \ref{crl:Simplicity-of-B'} can then be stated in a more memorable form:
\emph{the derived group of $G([0,1];A,P)$ is a non-abelian simple group}.

Suppose, finally, that $P$ is generated by a positive rational number $p > 1$
and that $A$ is the $\Z[P]$-module $\Z[P] = \Z[p,p^{-1}]$. 
Then $A$ is locally cyclic and so $A/(IP \cdot A)$ is a torsion group.
By Proposition \ref{PropositionC1} 
\index{Proposition \ref{PropositionC1}!consequences}%
the group $B= B([0,1];A,P)$ is therefore isomorphic 
to the group $B([0, \infty[\,;A,P)$ studied so far. 
By Corollary \ref{crl:Bab-for-P-cyclic-and-A-locally-cyclic},
the group $B_{\ab}$ is thus free abelian 
and its rank is 1 less than the order of $A/(IP \cdot A)$.
The order of this quotient has been determined in Illustration \ref{illustration:4.3}:
if $p = D/N$  with $D$ and $N$ positive, relatively prime integers,
then $IP \cdot A  = (N-D) \cdot A$ and the quotient $A/(IP \cdot A) $ is cyclic of order $N-D$.
We deduce that $B_{\ab}$ is free abelian of rank $|N-D| -1$.
The exact sequence \eqref{eq:Exact-abelianized-sequence-P-cyclic}
then allows us to compute $G_{\ab}$: 
it is free abelian of rank $|N-D|+ 1$.
\end{example}
%
%
\subsection{When is $B_{\ab}$ trivial?}
\label{ssec:12.4New}
\index{Subgroup B(I;A,P)@Subgroup $B(I;A,P)$!vanishing of Bab@vanishing of $B_{\ab}$|(}%
%
If $B = B([0,\infty[;A,P)$ is perfect 
then $A = IP\cdot A$ by Propositions \ref{PropositionC4} and \ref{prp:LemmaC5} (i).  
A further necessary condition is provided by Proposition \ref{PropositionB5}: it states, inter alia, 
that $H_1(P,A)$ is a homomorphic image of $B$.  
By Lemma \ref{LemmaC6} this second condition holds
if, \eg $A =IP\cdot A$ is noetherian.
We don't know whether 
this consequence of Lemma \ref{LemmaC6} holds without any extra hypothesis.

Conversely, assume that $A = IP\cdot A$. 
Then $K(A,P)$ is trivial by part (i) of Proposition \ref{prp:LemmaC5}
whence Proposition \ref{PropositionC4} yields the following commutative diagram 
with exact rows and an exact column.

\begin{equation}
\label{eq:Diagram-for-Bab}
\xymatrix{H_2(P,(G_2)_{\ab}) \ar@{->}[r]^-{d_2} \ar@{->}[d]^-{\rho_{*1}} 
&H_0(P,H_2(G_2)) \ar@{->}[r] \ar@{->}[d]^-{\rho_{*2}} 
&L\ar@{->}[r] \ar@{->}[d]^-{\rho_*} 
&H_1(P,(G_2)_{\ab}) \ar@{->}[r]\ar@{->}[d] &0\\
H_2(P,A) \ar@{->}[r]^-{d_2} 
&H_0(P,H_2(A)) \ar@{->}[r] 
&\bar L\ar@{->}[r] \ar@{->>}[d] 
&H_1(P,A) \ar@{->}[r] 
& 0\\
& &B_{\ab} & &}
\end{equation}
This diagram implies, in particular,  that $B_{\ab}$ is trivial 
if $H_1(P,A) = 0$ and if the map
\begin{equation*}
\rho_{*2} = H_0(P,H_2(\rho\restriction{G_2})) \colon  H_0(P,H_2(G_2))\longrightarrow H_0(P,H_2(A))
\end{equation*}
is \emph{surjective}.  
The only situation for which we can guarantee that this condition is fulfilled
occurs when the group $H_0(P,H_2(A))$ vanishes.  

Our final result singles out two cases in which this latter condition is satisfied.
\begin{theorem}
\label{TheoremC10}
Assume $P$ is a subgroup of $\R^\times_{>0}$ 
and $A $ is a $\Z[P]$-submodule of $\R_{\add}$.
Let $B$ denote the subgroup $B([0, \infty[\,;A,P)$.
If $B_{\ab}$ is trivial then $A = IP\cdot A$.  
Conversely, $B_{\ab}$ is trivial if $A = IP\cdot A$,  $H_1(P,A) = 0$
and if one of the following two conditions holds:
\begin{enumerate}[(i)]
\item $A$ is locally cyclic; 
\item  $P$ contains a rational number $p = n/d > 1$  so that $A$ is divisible by $n^2-d^2$.
\end{enumerate}
\end{theorem}

\begin{proof}
If $B_{\ab}$ is trivial 
then $K(A,P) = 0$ by the exact sequence \eqref{eq:ses-for-Bab}
and so $A = IP \cdot A$ by part (i) of  Proposition \ref{prp:LemmaC5}.
Conversely, assume that $A = IP \cdot A$.  
Then $K(A,P)$ vanishes by part (i) of  Proposition \ref{prp:LemmaC5}.
If, in addition,  $H_1(P,A) = 0$ and $A$ is locally cyclic
the group $H_2(A)$, 
being isomorphic to the exterior square of $A$, 
vanishes and so $B_{\ab} = 0$ by diagram \eqref{eq:Diagram-for-Bab}.
If condition (ii) is fulfilled
the rational number $p$ acts on $H_2(A) \approx A \wedge A$ by multiplication by $p^2$.  
The group $H_0(P,H_2(A))$ is therefore trivial,
for it is a quotient of  
\[
H_2(A)/ H_2(A) \cdot (p^2- 1) =H_2(A)/ H_2(A) \cdot (n^2-d^2).
\] 
The desired conclusion  follows again from diagram \eqref{eq:Diagram-for-Bab}.
\end{proof}

\begin{examples}
\label{examples:TheoremC10}
\index{Module A@Module $A$!select examples}
\index{Module A@Module $A$!select examples}
The PL-homeomorphism groups of the form $G(I;A,P)$
that are discussed in the literature
typically use a module of the form $A = \Z[P]$. 
Then the condition $A = IP \cdot A$ implies that $H_1(P,A) =0$  
(see Example \ref{example:LemmaC6}). 
\emph{The group $B$ is therefore perfect and hence simple} 
(by Corollary \ref{crl:Simplicity-of-B'}) 
\emph{provided hypothesis (i) or (ii) is fulfilled}.

Here are some concrete examples.
a) Assume first that $P$ is the cyclic group generated by $p = \tfrac{3}{2}$,  
and set $A= \Z[P]$.
Then $IP \cdot A = (\tfrac{3}{2}- 1) \cdot A = \tfrac{1}{2} \cdot A = A$.
Moreover,
as $A$ is locally cyclic hypothesis (i) is fulfilled and so $B$ is simple.

b)  Assume next
that $P$ contains the integers 2 and 3 and  set $A = \Z[P]$.
Then  $IP$ contains $2-1$ and so $IP \cdot A = A$;
in addition,  $2^2 - 1^2 = 3 \in P$. 
So hypothesis (ii) is fulfilled,
whence $B(\R;\Z[P],P)$ is simple.
Moreover, Theorem \ref{TheoremE10}) implies
that distinct groups $P$, $\bar{P}$, containing both 2 and 3, 
give rise to non-isomorphic simple groups $B(\R;\Z[P], P)$ and 
$B(\R;\Z[\bar{P}], \bar{P})$.
\index{Theorem \ref{TheoremE10}!consequences}%

c) From the point of view of an analyst, uncountable groups are also of interest,
for example the group $B(\R;\R,\R^\times_{> 0})$;
here $P = \R^\times_{>0}$ and $A = \Z[P]$.
Since 2 and 3 lie in $P$
the group $B(\R;\R,\R^\times_{> 0})$ is simple by the previous example b),
and we have recovered a result originally proved by Chehata in \cite{Che52}
\index{Chehata, C. G.}%
and later on by D. B. A. Epstein in \cite[Theorem 3.2]{Eps70}.
\index{Epstein, D. B. A.}
\end{examples}
\index{Subgroup B(I;A,P)@Subgroup $B(I;A,P)$!vanishing of Bab@vanishing of $B_{\ab}$|)}%
%
\subsection{Action of $G(\R;A,P)$ on  $B(\R;A,P)_{\ab}$}
\label{ssec:Action-TildeG}
%
We begin with an observation.
If $I$ is a compact interval 
the orbits of the group $G(I;A,P)$ contained in $A\, \cap\, \Int{I}$ 
have the form  $(IP \cdot A + a) \cap  \Int(I)$ with $a \in A$
(see Corollary \ref{CorollaryA1}). 
Now every bounded PL-homeomorphism $f \in G(\R;A,P)$ 
lies in $G(I;A,P)$ for some compact interval $I$ 
and $B(G;A, P)$ is the union of the groups $G(I; A,P)$ with $I$ a compact interval.
It follows that the orbits of $\tilde{B} = B(\R;A,P)$ in $A$ are the cosets of $IP \cdot A$ in $A$. 

Fix a $\tilde{B}$-orbit contained in $A$, 
say $\Omega = IP \cdot A + a_0$.
The function
\begin{equation}
\label{eq:nu-for-B}
\index{Homomorphism!13-nu-Omega@$\nu_\Omega$}%
\nu_\Omega \colon \tilde{B} \longrightarrow P, 
\quad
f \longmapsto \prod\nolimits_{a \in \Omega} f'(a_+)/f'(a_-)
\end{equation}
is then defined
and, as $\Omega$ is an orbit of $\tilde{B}$, it is a homomorphism
(\cf{}section \ref{ssec:11.1}).

Consider now the PL-homeomorphism $f = b(a, \Delta; p)$ 
discussed in Example \ref{example:Elements-in-B};
here $a \in A$ and $\Delta$ is a positive element of $A$
and $p \in P$ is a non-trivial element. 
This homeomorphism has 3 singularities,
namely 
\[
a, \quad a + \Delta \quad \text{and}\quad  a + (p+1) \cdot \Delta \in a + 2 \Delta + IP \cdot A.
\]
Its image $\nu(f)$ is given 
by formula \eqref{eq:nu-of-bounded-element};
\index{Homomorphism!13-nu@$\nu$!applications}%
it is
\begin{equation*}
\label{eq:nu-of-bounded-element-2} 
\index{Homomorphism!13-nu@$\nu$}%
\nu(f) = \tilde{B} \cdot a \otimes p + \tilde{B} \cdot (a+ \Delta) \otimes p^{-2}
+ \tilde{B} \cdot (a+ 2\Delta) \otimes p.
\end{equation*}

Suppose now 
that $A/(IP \cdot A)$ has more than one element
and choose $\Delta \in A \smallsetminus IP \cdot A$.
Then the orbit represented by $a + \Delta$ is distinct from the orbit represented by $a$,
but also distinct from $\tilde{B} \cdot (a + 2 \Delta)$.
If $\pi \colon  \Z[A/(IP \cdot A)] \otimes P \epi P$ denotes the projection onto the direct summand corresponding to  $IP \cdot A + (a+ \Delta)$,
the image of $ f$ under $\pi \circ \nu$ is therefore $p^{-2}$.

Consider, finally, the translation $\tau \colon \R \to \R$ with amplitude $\Delta$
and the conjugated element 
$
\act{\tau} f =\tau \circ b(a, \Delta; p) \circ \tau^{-1} )  
= 
b(a + \Delta, \Delta; p)$.
Then 
\[
(\pi \circ \nu) ( \act{\tau}f)
=
\begin{cases} 
p &\text{ if  }2 \Delta \notin IP \cdot A,\\
p^2 & \text{ if } 2 \Delta \in IP \cdot A.
\end{cases}
\]
In both cases, the homomorphism $\pi \circ \nu$ maps $f = b(a;\Delta;p)$ 
and its conjugate $\act{\tau}f$ to distinct elements  of $P$.
Since $P$ is commutative 
$\pi \circ \nu$ induces a homomorphism of $\tilde{B}_{\ab}$ into $P$;
it maps the cosets $f \cdot [\tilde{B},\tilde{B}]$ 
and $\act{\tau}f\cdot [\tilde{B},\tilde{B}]$ 
to different elements of $P$.
 
The previous calculations establish
\begin{proposition}
\label{prp:Action-TildeG}
\index{Group G(I;A,P)@Group $G(I;A,P)$!action on Bab@action on $B_{\ab}$}%
If $IP \cdot A \neq A$ the group $\tilde{G} = G(\R;A, P)$  acts non-trivially 
on the abelianization $\tilde{B}_{\ab}$ of its subgroup of bounded elements $\tilde{B}$.
\end{proposition}

%

%% file: chaptC.11.2b.fig1.tex
\begin{minipage}{12 cm}
\psfrag{1}{\hspace*{-1.6mm}     \small $0$}
\psfrag{2}{\hspace*{-1.5mm}    \small $\tfrac{1}{8}$}
\psfrag{3}{\hspace*{-1.8mm}     \small $\tfrac{1}{4}$}
\psfrag{4}{\hspace*{-2mm}     \small $\tfrac{1}{4}$}
\psfrag{4}{\hspace*{-1.6mm}     \small $\tfrac{3}{8}$}
\psfrag{5}{\hspace*{-1.8mm}     \small $\tfrac{1}{2}$}
\psfrag{6}{\hspace*{-1.9mm}     \small $\tfrac{3}{4}$}
\psfrag{7}{\hspace*{-1.3mm}     \small $1$}

\psfrag{11}{\hspace*{-2.5mm}     \small }
\psfrag{12}{\hspace*{-1.7mm}    \small $1$}
\psfrag{13}{\hspace*{-0.30mm}     \small $1$}
\psfrag{14}{\hspace*{-1.0mm}        \small $1$}
\psfrag{15}{\hspace*{1.9mm}     \small $\tfrac{1}{2}$}
\psfrag{16}{\hspace*{1mm}     \small $\tfrac{1}{2}$}
\psfrag{21}{\hspace*{-2mm}     \small }
\psfrag{22}{\hspace*{-1.3mm}     \small }
\psfrag{23}{\hspace*{-2.3mm}     \small }
\psfrag{24}{\hspace*{-1.6mm}     \small $\tfrac{5}{8}$}
\psfrag{25}{\hspace*{-4.5mm}     \small }
\psfrag{26}{\hspace*{-1.4mm}     \small $\tfrac{7}{8}$}
\psfrag{27}{\hspace*{-4.5mm}     \small }
\psfrag{31}{\hspace*{-4mm}     \small $\tfrac{1}{2}$}
\psfrag{32}{\hspace*{-1.7mm}    \small $\tfrac{1}{2}$}
\psfrag{33}{\hspace*{-1.0mm}     \small 1}
\psfrag{34}{\hspace*{-1.0mm}     \small 1}
\psfrag{35}{\hspace*{2.0mm}     \small 2}
\psfrag{36}{\hspace*{1.7mm}     \small 2}
\psfrag{41}{\hspace*{-3mm}     \small }
\psfrag{42}{\hspace*{-7.8mm}    \small $\tfrac{3}{16}$}
\psfrag{43}{\hspace*{-3.3mm}     \small}
\psfrag{44}{\hspace*{-3.3mm}     \small }
\psfrag{45}{\hspace*{-4.5mm}     \small }
\psfrag{46}{\hspace*{-4.5mm}     \small }
\psfrag{47}{\hspace*{-4.5mm}     \small }
\psfrag{51}{\hspace*{-2mm}     \small }
\psfrag{52}{\hspace*{-1.3mm}    \small $\tfrac{1}{2}$}
\psfrag{53}{\hspace*{-2.5mm}     \small $\tfrac{1}{2}$}
\psfrag{54}{\hspace*{-1.2mm}     \small $1$}
\psfrag{55}{\hspace*{-1.0mm}     \small $1$}
\psfrag{56}{\hspace*{-7.0mm}     \small $2$}
\psfrag{61}{\hspace*{-0.7mm}     \small 0}
\psfrag{62}{\hspace*{-1.7mm}    \small $\tfrac{1}{32}$}
\psfrag{63}{\hspace*{-1.5mm}     \small $\tfrac{1}{16}$}
\psfrag{64}{\hspace*{-1.4mm}     \small $\tfrac{1}{8}$}
\psfrag{65}{\hspace*{-1.2mm}     \small $\tfrac{1}{4}$}
\psfrag{66}{\hspace*{-1.3mm}     \small $\tfrac{1}{2}$}
\psfrag{67}{\hspace*{-0.4mm}     \small 1}

\psfrag{la1}{\hspace*{-6.5mm}  \small $\varphi^{-1}$}
\psfrag{la2}{\hspace*{-5.7mm}  \small $f$}
\psfrag{la3}{\hspace*{-5.7mm}  \small $\varphi$}
\begin{equation*}
\includegraphics[width= 11cm]{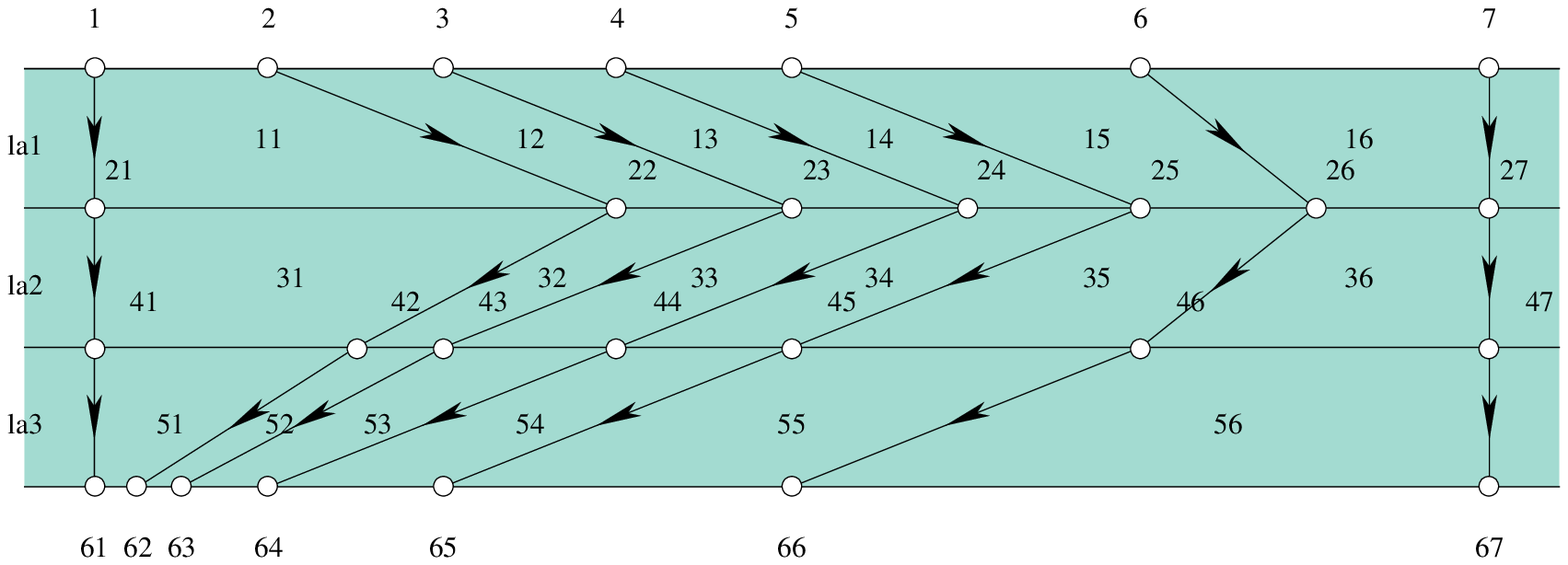}\\
\end{equation*}
\end{minipage}

%% file: chaptC11.2b.fig2and3.tex
\[
\begin{minipage}[c]{12cm}
\psfrag{1}{\hspace*{-1.5mm}     \small $0$}
\psfrag{2}{\hspace*{-1.9mm}    \small $\tfrac{1}{2}$}
\psfrag{3}{\hspace*{-2.0mm}     \small $\tfrac{3}{4}$}
\psfrag{4}{\hspace*{-1.3mm}     \small $1$}

\psfrag{11}{\hspace*{-3mm}     \small $\tfrac{1}{2}$}
\psfrag{12}{\hspace*{-1.0mm}    \small $1$}
\psfrag{13}{\hspace*{-0.5mm}     \small $2$}
\psfrag{21}{\hspace*{-0.6mm}     \small 0}
\psfrag{22}{\hspace*{-0.8mm}    \small $\tfrac{1}{4}$}
\psfrag{23}{\hspace*{-0.9mm}     \small $\tfrac{1}{2}$}
\psfrag{24}{\hspace*{-0.5mm}     \small 1}
%
%
\psfrag{31}{\hspace*{-1.3mm}     \small $0$}
\psfrag{32}{\hspace*{-1.2mm}    \small $\tfrac{1}{4}$}
\psfrag{33}{\hspace*{-1.1mm}     \small $\tfrac{3}{8}$}
\psfrag{34}{\hspace*{-1.3mm}     \small $\tfrac{1 }{2}$}
\psfrag{35}{\hspace*{-1.7mm}     \small $\tfrac{3 }{4}$}
\psfrag{36}{\hspace*{-1.3mm}     \small 1}

\psfrag{41}{\hspace*{-1.8mm}     \small $\tfrac{1}{4}$ }
\psfrag{42}{\hspace*{-1.7mm}    \small $\tfrac{1}{2}$}
\psfrag{43}{\hspace*{-1.2mm}     \small $1$ }
\psfrag{44}{\hspace*{-0.9mm}    \small $1$}
\psfrag{45}{\hspace*{0.0mm}    \small $2$}
\psfrag{51}{\hspace*{-1.0mm}     \small $0$}
\psfrag{52}{\hspace*{-1.5mm}    \small $\tfrac{1}{16}$}
\psfrag{53}{\hspace*{-0.8mm}     \small $\tfrac{1}{8}$}
\psfrag{54}{\hspace*{-1.3mm}    \small $\tfrac{1}{4}$}
\psfrag{55}{\hspace*{-1.4mm}    \small $\tfrac{1}{2}$}
\psfrag{56}{\hspace*{-1.1mm}     \small 1}

\psfrag{la1}{\hspace*{-1.5mm}  \small $f$}
\psfrag{la2}{\hspace*{-3.5mm}  \small $\act{\varphi}f$}

\[
\includegraphics[width= 5.5cm]{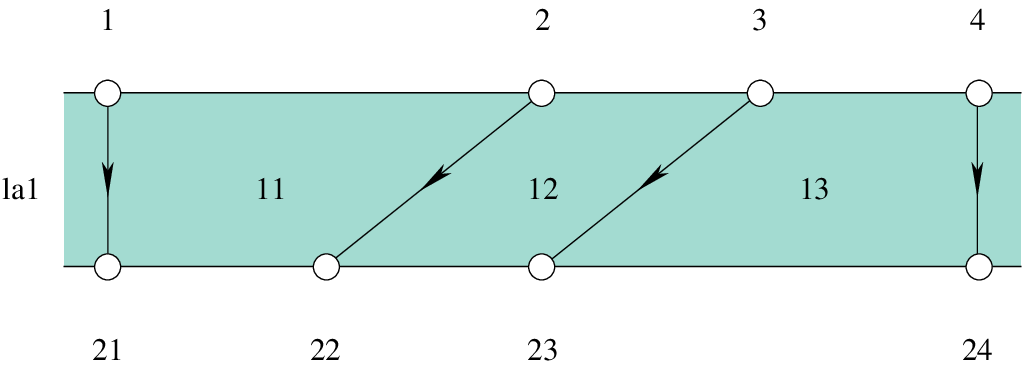}
\hspace{7mm}
\includegraphics[width= 5.5cm]{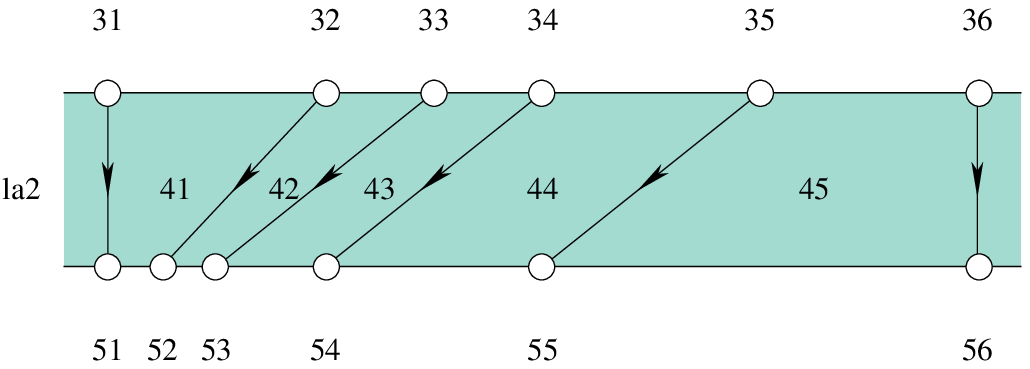}
\]
\end{minipage}
\]

%% file: chaptD_Presentations.tex
%
%
\chapter{Presentations}
\label{chap:D}
%
%
Let $I$ be the line $\R$ or the half line $[0,\infty[$.
The group $G = G(I;A,P)$ admits then a simple infinite presentation:
in case $I = \R$,
its generators are the affine homeomorphisms 
and PL-homeomorphisms with one break,
if $I = [0,\infty[$, 
the generating set consists of all PL-homeomorphisms with one singularity.

In Chapter \ref{chap:B} the triples $(I,A,P)$ 
giving rise to a \emph{finitely generated groups} $G$ have been determined:
the group $P$ must then be finitely generated and $A$ must be a finitely generated $\Z[P]$-module;
moreover, 
if $I = [0,\infty[$ the submodule $IP \cdot A$ is to be of finite index in $A$
(see Theorems \ref{TheoremB2} and \ref{TheoremB4}).
\index{Quotient group A/IPA@Quotient group $A/(IP \cdot A)$!significance}%
The necessity of the stated conditions follows directly from the observation
that the group $G$ is the union of an infinite chain
if either $P$ is an infinitely generated group, 
or if $A$ is an infinitely generated module,
or if $I =[0,\infty[$ and $A/(IP \cdot A)$ is infinite 
(clearly only countable groups present a problem).
In the case of a line the sufficiency of the stated conditions
can easily be read off from the infinite  presentation;
in the other case, 
the claim has been established by a rather lengthy calculation.

In this chapter, 
we address the problem of determining the triples $(I,A,P)$ 
which lead to \emph{finitely presented groups} $G$.
Our results are far less complete than those in Chapter \ref{chap:B}.
The necessary conditions we have found are easy to state:
if $G$ admits a finite presentation 
so does the metabelian group  $\Aff(A,P) \iso A \rtimes P$;
\index{Affine group Aff(A,P)@Affine group $\Aff(A,P)$!significance}
moreover, if $I = \R$ then $A/(IP \cdot A)$ must be finite
(see Propositions \ref{PropositionD2} and \ref{PropositionD5}).
\index{Quotient group A/IPA@Quotient group $A/(IP \cdot A)$!significance}%
Our sufficient conditions are very exacting; 
the best result is obtained in the case of a line (see Proposition \ref{PropositionD4});
in the case of a half line 
the group $P$ is required to be generated by positive integers 
and $A = \Z[P]$ (see Theorem \ref{TheoremD6}).
\index{Group P@Group $P$!select examples}%
\index{Module A@Module $A$!select examples}%

In the above, only intervals of infinite length have been mentioned.
There are results for intervals of the form $[0, b]$ with $b \in A$
(see Section \ref{sec:15})
but, much as in Chapter \ref{chap:B}, 
these results are far less general 
than the findings for intervals of infinite length.
%
\setcounter{section}{12}
\section{Presentations of groups with supports in the line}
\label{sec:13}
%
Here is a summary of the results in this section.
The group $G = G(\R;A,P)$ is a semi-direct product of the form $N \rtimes \Aff(A,P)$.
If $G$ is finitely generated, $N$ is the normal closure of a finite subset 
and $G$ admits a tractable infinite presentation (see Proposition \ref{PropositionD1}).
If $G$ admits a \emph{finite presentation}, 
so does its quotient group $\Aff(A,P)$; 
moreover, the quotient module $A/(IP \cdot A)$ is then finite 
(see Proposition \ref{PropositionD2}).
Proposition \ref{PropositionD4}, finally, gives a partial converse to these results.
%
\subsection{An infinite presentation}
\label{ssec:13.1}
In Section \ref{sec:7}  
the group $G = G(\R;A,P)$ has been shown to be generated 
by the affine subgroup $\Aff(A,P)$ 
and by the subset of the kernel of $\lambda \colon G\to \Aff(A,P)$ 
consisting of elements with one singularity.
If we include,  
for ease of notation,  the identity into this subset 
we arrive at the set
\begin{equation}
\label{eq:13.1}
\GG = \{g(a,p)\mid (a,p) \in A\times P\};
\end{equation}
here  $g(a,p)$ denotes the PL-homeomorphism that fixes $t$ for $t\leq a$ 
and sends $t$ to $p(t-a) + a$ for $t > a$.  

Relations in terms of the generating set $\Aff(A,P)\,\cup\, \GG$ 
\index{Group G(R;A,P)@Group $G(\R;A,P)$!generating sets}%
are the multiplication table relations
\begin{align}
\aff(a,p)\circ \aff(a',p') &= \aff(a+pa',pp'), \label{eq:13.2}\\
g(a,p)\circ g(a,p') &= g(a,pp'), \label{eq:13.3}
\intertext{and the conjugation relations}
\act{\aff(a,p)}g(a',p') &= g(a+pa',p'), \label{eq:13.4}\\
\act{g(a,p)}g(a',p') &= g(a+p(a'-a), p') \text{ if }  a < a'.\label{eq:13.5}
\end{align}
\index{Group G(R;A,P)@Group $G(\R;A,P)$!infinite presentations}%

These relations define $G$ in terms of the generating set $\Aff(A,P)\cup \GG$
(see \cite[Section 2]{BrSq85}).  
Indeed, let $F$ be the free group on  the set $\Aff(A,P)\,\cup \,\GG$
and let $\equiv$ denote the congruence relation on $F$
generated by relations \eqref{eq:13.2} through \eqref{eq:13.5}.  
Given a word $w\in F$
there exists, by relations \eqref{eq:13.4} and \eqref{eq:13.5}, 
a congruent word $v_{a_m} \cdots  v_{a_1} \cdot u$, 
where $u$ is a word in $\Aff(A,P)$, 
each $v_{a_i}$ is a word in $\{g(a_i,p)\mid p\in P\}$ and $a_m > a_{m-1} > \cdots > a_1$.
Using relations \eqref{eq:13.2} and  \eqref{eq:13.3}, 
each of these subwords can be replaced by a generator $g(a_i, p_i)$,
respectively by a generator $\aff(a_0,g_0)$, without leaving the congruence class.  
The resulting word
\begin{equation*}
w' = g(a_m,p_m) \cdots g(a_1,p_1)\aff(a_0,p_0)
\end{equation*}
represents the  identity $\id \colon \R \iso \R$
only if $\aff(a_0, p_0) = \id$ and $p_i = 1$ for each $i$.

The above presentation has a simple form, 
but its generating set contains many redundant elements;  
a smaller generating set can be obtained as follows.
We begin by introducing the PL-homeomorphisms
\begin{equation}
\label{eq:New-generators}
\index{Group G(R;A,P)@Group $G(\R;A,P)$!generating sets}%
f(p) = \aff(0,p), \quad  g(p) = g(0,p) \text{ and } h(a) = \aff(a,1).
\end{equation}
Relations \eqref{eq:13.2} and  \eqref{eq:13.4} 
allow one to express the old generators in terms of the new ones:
$\aff(a,p) = h(a) \circ f(p)$ and $g(a,p) = \act{h(a)}g(p)$.
If one substitutes the new generators into relations  \eqref{eq:13.3}
one finds that all of them are consequences of the relations
\begin{equation}
\label{eq:13.3p}
g(p) \cdot  g(p') =g(pp') \quad\text{for all} \quad (p, p')  \in P^2. 
\end{equation}
If one expresses  relation \eqref{eq:13.4} in the new generators
one arrives at the relations
\[
\act{h(a) f(p)}\left( \act{ h(a')}g(p') \right) = \act{h(a + p a')} g(p').
\] 
The conjugating element of the right hand side of this relation
equals $h(a) \act{f(p)}h(a')$ (by relations  \eqref{eq:13.2}). 
So relations \eqref{eq:13.4} boil down to the conjugation relations
\begin{equation}
\label{eq:13.4p}
\act{f(p)} g(p') =g(p') \quad\text{for all} \quad (p, p') \in P^2.
\end{equation}
Now to relations \eqref{eq:13.5}.
The insertion of the new generators produces the relations
\[
\act{h(a) g(p) h(a)^{-1} \cdot h(a')} g(p') = \act{h(a + p(a'-a))}g(p').
\]
Upon setting $b = a'-a$ and using  relations \eqref{eq:13.2},
the previous relations become
\begin{equation}
\label{eq:13.5p}
\act{g(p)h(b)}g(p') = \act{h(pb)}g(p') 
\quad\text{for all} \quad (b, p, p') \in A_{>0} \times P^2.
\end{equation}
Note that $h(pb) = f(p) \cdot h(p)\cdot f(p)^{-1}$; 
since $\act{f(p)}g(p') = g(p')$ in view of relations \eqref{eq:13.4p}, 
relations \eqref{eq:13.5p} can also be written in the form
\begin{equation}
\label{eq:13.5pp}
\act{g(p)h(b)}g(p') = \act{f(p)h(b)}g(p').
\end{equation}

In the sequel,
we shall mainly be interested  in presentations of finitely generated groups.  
Then $P$ will be finitely generated, hence free abelian.  
Suppose therefore that $P$ is free abelian and let $\PP$ is a basis of $P$.  
Relations \eqref{eq:13.2} and \eqref{eq:13.3p} can then be used 
to eliminate the generators $f(p)$ and $g(p)$  with $p\in P\smallsetminus \PP$. 
The presentation derived from relations \eqref{eq:13.2}, \eqref{eq:13.4p} and \eqref{eq:13.5p} 
in terms of this smaller set of generators is then as described in
\begin{proposition}
\label{PropositionD1}
\index{Group G(R;A,P)@Group $G(\R;A,P)$!infinite presentations}%
\index{Group G(R;A,P)@Group $G(\R;A,P)$!generating sets}%
Suppose $P$ is free abelian with basis $\PP$. 
Introduce the PL-homeomorphisms $f(p) = \aff(0,p)$, $g(p) = g(0,p)$ and $h(a) = \aff(a,1)$.
The group $G(\R;A,P)$ is then generated by the (infinite) set
\begin{equation*}
\{f(p) \mid p\in \PP\} \cup \{g(p)\mid p\in \PP\} \cup \{h(a) \mid a\in A\},
\end{equation*}
and defined in terms of this generating set by the relations
\begin{equation}
\label{eq:13.6}
[f(p),f(p')] = [f(p),g(p')] = [g(p),g(p')] = \id,
\end{equation}
\begin{align}
\act{f(p)}h(a) &= h(pa),\label{eq:13.7}\\
h(a)\cdot h(a') &= h(a+a'), \text{ and} \label{eq:13.8}\\
\act{g(p)h(b)}g(p') &= \act{f(p)h(b)}g(p').\label{eq:13.9}
\end{align}
Here $(a,a',b)$ ranges over $A\times A \times A_{>0}$ and $(p,p')$ over $\PP^2$.
\end{proposition}

\begin{remark}
\label{remark:PropositionD1}
The first commutator relations in \eqref{eq:13.6} guarantee 
that the subset $\{f(p) \mid p \in \PP\}$ generates an abelian group;
the relations \eqref{eq:13.7} and \eqref{eq:13.8} define then the module $A$.
If the metabelian subgroup (and quotient group) $Q = A \rtimes P$ of $G$ admits a finite presentation,
the infinite set of generators $\{f(p)\mid p\in \PP\} \cup \{h(a) \mid a\in A\}$ 
can be replaced by a finite subset $\QQ$, 
and finitely many relations in terms of this subset $\QQ$ 
will then define the group $Q$.

The group $G$ itself will then be generated by its subgroup $Q$
and the finite subset $\{ g(p) \mid p\in \PP\}$
and defined by the relations of $Q$,
the remaining relations listed in equation \eqref{eq:13.6} 
and the set of relations \eqref{eq:13.9}.
This last set has a strange feature: 
it involves the subset $A_{>0}$ 
which depends not only on the module structure on $A$,
but, in addition,  on the order relation inherited from $\R$.
In section \ref{ssec:13.3New} 
we shall present an approach to deal with this problem.
\end{remark}
%
\subsection{Necessary conditions for a finite presentation}
\label{ssec:13.2}
In Chapter \ref{chap:B} necessary, as well as sufficient, conditions
for the finite generation of $G = G(\R;A,P)$ have been stated.  
If it comes to finite presentations, we cannot do equally well: 
we can only derive necessary conditions and give sufficient conditions of a rather different sort.
Necessary conditions are spelled out by
\begin{proposition}
\label{PropositionD2}
\index{Group G(R;A,P)@Group $G(\R;A,P)$!finite presentation}%
\index{Quotient group A/IPA@Quotient group $A/(IP \cdot A)$!significance}%
\index{Affine group Aff(A,P)@Affine group $\Aff(A,P)$!significance}%
If $G(\R;A,P)$ admits a finite presentation 
the metabelian quotient group $\Aff(A,P)$ is finitely related
and $A/(IP\cdot A)$ is finite.
\end{proposition}

\begin{proof}
According to Theorem \ref{TheoremB2} 
the group $G = G(\R;A,P)$ is finitely generated if, and only if, $P$ is finitely generated 
and $A$ is a finitely generated $\Z[P]$-module; 
these two conditions hold precisely 
if the homomorphic image $\im \lambda = \Aff(A,P) \cong A\rtimes P$ is finitely generated.
Moreover, if $P$ is finitely generated, say by the finite set $\PP$,
relations \eqref{eq:13.4} imply
that the kernel of $\lambda$ is the normal closure of the finite set $\{g(0,p)\mid p\in \PP\}$.  
Thus $\Aff(A,P)$ is finitely presented if $G$ is so.
\smallskip

The proof of the assertion 
that finite presentability of $G$ forces $A/(IP\cdot A)$ to be finite
is more involved and relies on Theorem A in \cite{BiSt78}.
\index{Bieri, R.}%
\index{Strebel, R.}%
This theorem has the following consequence:
suppose $G$ is a finitely related group without non-abeian free subgroups,
$\chi \colon G \epi \Z$ is an epimorphism 
and $x  \in G$ is an element generating a complement of $\ker \chi$.
Then there exist a \emph{finitely generated} subgroup $B$ of $\ker \chi$
and a sign $\varepsilon \in \{1, -1\}$ with the following properties:
\begin{equation}
\label{eq:13.10}
\act{x^{-\varepsilon}}B \subseteq B 
\quad \text{and}\quad 
\bigcup\nolimits_{j\geq 0} \act{x^{j\cdot \varepsilon}} B = \ker\chi
\end{equation}
(see, \eg{}  Theorem B$^*$ on page 266 and the Deduction of Theorem B from Theorem B$^*$ on page 267 in \cite{Str84}).

Suppose now
that \emph{$G = G(\R;A,P)$ is finitely generated 
and that $A/(IP\cdot A)$ is infinite}. 
The image $\Aff(A,P) \iso A \rtimes P$ of $G$ is then finitely generated,
whence $A$ is a finitely generated $\Z[P]$-module 
over the finitely generated group $P$,
and so the quotient module $A/(IP \cdot A)$ is a finitely generated abelian group.
Being infinite,
it maps therefore onto an infinite cyclic group;
so there exists a subgroup $A_1$ of $A$
that contains $IP \cdot A$ and so that $A/A_1$ is infinite cyclic.
The subgroup $A_1$ is, in point of fact,
a finitely generated $\Z[P]$-submodule of $A$ 
and $A_1 \rtimes  P$ is a finitely generated normal subgroup of  
$A \rtimes P \isoinv \Aff(A,P)$ 
with infinite cyclic quotient $(A \rtimes P)/(A_1 \rtimes P)$.
There exists therefore an epimorphism $\bar{\chi} \colon A \rtimes P \epi \Z$ 
with kernel $A_1 \rtimes P$;
let $\chi$ denote the composition
\begin{equation*}
G \xrightarrow{\lambda} \Aff(A,P)  
\iso A \rtimes P 
\xrightarrow{\bar\chi} \Z.
\end{equation*}
The kernel of $\chi$ is then the semi-direct product 
\begin{equation}
\label{eq:Relation-chi-lambda}
\ker\chi = \ker\lambda \rtimes \Aff(A_1,P).
\end{equation}

Next, 
let $x_0 \in \Aff(A, P) \subset G$ be a translation with $\chi (x_0) \in \{ 1, -1\}$
and let $B_0$ be a subgroup of $\ker \chi$ with the following properties:
\begin{equation}
\label{eq:13.10plus}
B_0 \text{ is finitely generated and } B_0 \subseteq \act{x_0} B_0. 
\end{equation}
Lemma \ref{lem:Sigularities-in-ascending-union} below assures us
that
$\bigcup\nolimits_{j\geq 0} \act{x^j} B_0 \subsetneqq \ker\chi$.
As the group of all finitary PL-homeomorphisms of the real line
and, in particular, $G$
have no non-abelian free subgroups (see \cite[Theorem (3.1)]{BrSq85}),
the consequence of \cite[Theorem A]{BiSt78} stated in the above
thus allows us to conclude  
that $G$ is infinitely related.
\end{proof}

We are left with proving
\begin{lemma}
\label{lem:Sigularities-in-ascending-union}
Suppose the epimorphism $\chi \colon G(\R; A, P)  \epi \Z$,
the finitely generated subgroup $B_0 \in \ker \chi$
and the translation $x_0 \colon t \mapsto t + a_0$ are as before.
Then
\[
\bigcup\nolimits_{j\geq 0} \act{x_0^j} B_0\subsetneqq \ker\chi.
\]
\end{lemma}

\begin{proof}
We first have a closer look at $K = \ker \chi$.
Equation \eqref{eq:Relation-chi-lambda} shows
that $K$ is the semi-direct product of $\ker \lambda$ and $\Aff(A_1, P)$.
The kernel of $\lambda$ contains the subset $\XX_{-\infty}$
made up of all PL-homeomorphisms $g(a,p)$ given by the formula
\begin{equation}
\label{eq:8.1bis}
g(a,p) (t) = \begin{cases} t & \text{ if } t \leq a,\\ p(t-a) + a & \text{ if } a \leq t ;\end{cases}
\end{equation}
here $a$ runs over $A$ and $p$ over $P$.
An easy induction on the number of singularities of the elements in $\ker \lambda$ 
then discloses
that $\ker \lambda$ is generated by the subset $\XX_{-\infty}$
(cf.{} \cite[Corollary (2.5)]{BrSq85}).
As  $\XX_{-\infty}$ is closed under inversion
it follows
that every element in $K = \ker \chi$ is a composition of the form
\begin{equation}
\label{eq:Representation-elements-K}
f = g(a_{f,k}, p_{f, k}) \circ \cdots \circ g(a_{f,1}, p_{f,1}) \circ \aff(a'_f, p'_f).
\end{equation}

We next turn our attention to the subgroup $B_0$.
It is,
by hypothesis, 
a finitely generated subgroup of $K$.
Choose a finite generating set $\BB_0$ of $B_0$ that is stable under inversion
and then select, for every $f \in \BB_0$, a representation of the form
\eqref{eq:Representation-elements-K}.
Let $\Acal$ be the set of all elements $a_{f, j}$ occurring
in the chosen representations of the elements $f \in \BB_0$;
it is a finite set.
Consider now the set  
\begin{equation}
\label{eq:Set-of-singularities}
\Sing(B_0) 
= \{t \in \R \mid \exists f \in B_0 \text{ with } f'(t_-) \neq f'(t_+) \}
\end{equation}
of all singularities of the elements in $B_0$.
Our immediate aim is to establish an upper bound for $\Sing(B_0)$,
namely 
\begin{equation}
\label{eq:Upper-bound-B0}
\Sing(B_0) \subseteq A_1 + \Acal = \bigcup\nolimits_{a \in \Acal} A_1 + a.
\end{equation}

Fix an element $h \in B_0$.
Since the generating set $\BB_0$ of $B_0$ is closed under inversion,
$h$ is a product $f_\ell \circ \cdots \circ f_1$  of elements $f_j \in \BB_0$,
each of which has the form \eqref{eq:Representation-elements-K}.
If $h$ is not differentiable at $s$ 
there exists a factor $f_j$ of $h$ 
and in this factor a function $g(a_{f_j, i}, p_{f_j,i})$,
all in such a way
that the PL-homeomorphism
\[
h_* 
=   
\left( g(a_{f_j, i-1}, p_{f_j,i-1}) \circ \cdots \circ g(a_{f_j, 1}, p_{f_j,1})  
\circ \aff(a'_{f_j}, p'_{f_j}) \right)
\circ  
\left(f_{j-1} \circ \cdots \circ f_1\right)
\]
maps $s$ onto the unique singularity $a_{f_j, i}$ of $g(a_{f_j, i}, p_{f_j,i})$.
Now, $a_{f_j, i} \in \Acal$ 
and $h_*$ is a composition of PL-homeomorphisms 
that belong to  $\KK =\XX_\infty \cup \Aff(A_1, P)$.
Since $\KK$ is closed under inversion, 
it suffices to check
that each element  of $\KK$ 
maps the set $A_1 + \Acal$ into itself.
This verification is easy:
fix $a_* \in  A_1 + \Acal$. 
If $\aff(a_1, p)$ is an element of  $\Aff(A_1, P)$ 
then
\[
\aff(a_1, p)(a_*) = a_1 + p \cdot a_* 
= (a_1 + (p-1) \cdot a_*)) + a_* \in A_1 + a_*.
\]
If  $g(a, p) \in \XX_\infty$ then
\[
g(a, p)(a_*) = \begin{cases}
a_*  & \text { if } a_* \leq a \\
a + p(a_*-a) = (p-1)(a_*-a) + a_* \in A_1 + a_* & \text{ if } a_* > a.
\end{cases}
\]
The preceding calculations use both the fact 
 that  $IP \cdot A$ is contained in $A_1$.

 So far we know
 that $\Sing(B_0)$ is contained in $A_1 + \Acal$
where $\Acal$ is a certain finite subset of $A$.
Let's now pass to the subgroup $\act{x_0} B_0$.
By assumption,
$x_0$ is the translation with amplitude $a_0 \in A$.
The set of singularities $\Sing(\act{x_0}B_0)$ of $\act{x_0}B_0$
equals therefore $\Sing(B_0) + a_0$.
Moreover,
\[
\Sing(\act{x_0^k}B_0) = \Sing(B_0) + k \cdot a_0
\]
for every positive integer $k$, 
whence inclusion \eqref{eq:Upper-bound-B0}
allows us to infer that
\begin{equation}
\label{eq:Upper-bound-Binfty}
\Sing \left(\bigcup\nolimits_{j\geq 0} \act{x_0^j} B_1\right) 
\subseteq A_1 + \Acal + \N \cdot a_0.
\end{equation}
The right hand side of this inclusion is a proper subset of $A$;
for $A/A_1$ is infinite cyclic,
$a_0$ generates a complement of $A_1$ in $A$ and $\Acal$ is a finite set.
The assertion of the lemma now follows from inclusion 
\eqref{eq:Upper-bound-Binfty}
and the fact 
that the set of singularities $\Sing(\XX_\infty)$ of $\XX_\infty$ 
is, of course, all of $A$.
\end{proof}

\begin{remarks}
\label{remarks:Section-13.3}
(i) The problem of characterizing the \emph{finitely related} met\-abelian groups
among the \emph{finitely generated} metabelian groups
 is addressed in \cite{BiSt80}.  
The proposed solution involves a subset 
$\Sigma_{[Q,Q]}$ of a sphere $S(Q_{\ab})$; 
thanks to a number of subsequent articles this subset is fairly well-understood.  
A survey of results on $\Sigma_{[Q,Q]}$ is given in \cite[Section 6]{Str84}.
\index{Bieri, R.}%
\index{Strebel, R.}%

(ii) If $G$ is finitely presented, 
so is the soluble quotient $G/[B,B]$ of $G$, the group $B'$ being simple by section \ref{ssec:10.2}.  
\emph{Suppose now that $A/(IP\cdot A)$ is finite.
Then $G/[B,B]$ is finitely presented if, and only if, $G/\ker\lambda$, 
or equivalently $\Aff(A,P)$,  is so}.  
Indeed, if $G/[B,B]$ is finitely presented 
then so is $\Aff(A,P)$ by the first part of the proof of Proposition \ref{PropositionD2}
and the fact that $B \subset \ker \lambda$. 
The converse follows from three facts.  
First of all,
$G/[B,B]$ is an extension of $B_{\ab}$ by $G/B$, 
which by the proof of Corollary \ref{CorollaryA3} 
is an extension of $IP\cdot A \rtimes P$ by $A\rtimes P$.  
Secondly, $B_{ab}$ is finitely generated by Lemma \ref{LemmaD3} below 
and $IP\cdot A \rtimes P$ is finitely presented, 
\eg{}because it has finite index in $A\rtimes P$ which is finitely related by assumption.  
Finally, extensions of finitely presented groups can be finitely presented 
by a well-known observation of P.\ Hall's  (see, \eg{}\cite[p.\,52]{Rob96}).

(iii) Lemma \ref{LemmaD3} below states a  condition 
which implies that $B_{\ab}$ is finitely generated.
Another sufficient condition can be gleaned 
from  Proposition \ref{PropositionC1}
\index{Proposition \ref{PropositionC1}!consequences}%
and the exact sequence \eqref{eq:5-term-sequence-involving-Bab}: 
\emph{the group $B_{\ab}$ is finitely generated 
if $G([a,c]; A, P)_{ab}$, and hence its image $P$,
are finitely generated (for some $a < c$ in $A$).}
\footnote{see sections \ref{sssec:Notes-ChapterC-Brown-Stein}
and \ref{sssec:Notes-ChapterC-Gal-Gismatullin} for a more revealing argument}
\index{Subgroup B(I;A,P)@Subgroup $B(I;A,P)$!finite generation of Bab@finite generation of $B_{\ab}$}%
\end{remarks}

In part (ii) of the preceding remarks,
the following result has been used.
\begin{lemma}\label{LemmaD3}
\index{Subgroup B(I;A,P)@Subgroup $B(I;A,P)$!finite generation of Bab@finite generation of $B_{\ab}$}%
\index{Quotient group A/IPA@Quotient group $A/(IP \cdot A)$!significance}%
If $\Aff(A,P)$ is finitely presented and $A/(IP \cdot A)$ is finite
then the abelian group $B(\R;A,P)_{\ab}$ is finitely generated.
\end{lemma}

\begin{proof}
The claim is a consequence of results in Chapter \ref{chap:C}. 
Proposition \ref{PropositionC1} first shows that $B(\R;A,P)$ is isomorphic 
to the group $B =B([0,\infty[\,;A,P)$.
\index{Proposition \ref{PropositionC1}!consequences}%
The abelianization of the latter group is an extension of a quotient of a group $\bar{L}$ 
by a group $K(A;P)$ (see formula \eqref{eq:ses-for-Bab}).
By assumption,
$A/(IP \cdot A)$ is finite and $\Aff(A,P)$ is finitely presented,
whence $P$ is finitely generated and $A$ is a finitely generated $\Z[P]$-module;
by part (ii) of Proposition \ref{prp:LemmaC5} 
the group $K(A,P)$ is  thus finitely generated. 
It suffices therefore to show
that  $\bar{L}$ is finitely generated.

By formula \eqref{eq:Description-Lbar} 
the group $\bar{L}$ is the middle term of an exact sequence
\[
H_0(P, H_2(IP \cdot A))  \longrightarrow \bar{L} \longrightarrow H_1(P, IP \cdot A) \to 0.
\]
As pointed out before,
$P$ is finitely generated and $A$ is a finitely generated $\Z[P]$-module.
Since the group ring of a finitely generated abelian group is noetherian,
the submodule $A_1 = IP \cdot A$ of $A$ is finitely generated
and so the homology group $H_1(P, A_1)$ is a finitely generated group
(see, \eg{}\cite[Theorems 4.2.3 and 10.1.5]{LeRo04}).
To complete the proof it suffices thus to show
that the iterated homology group $H_0(P, H_2(A_1))$  is finitely generated.

We claim that $H_2(A_1)$ is finitely generated as a $\Z{P}$-module.
Indeed,
the group $A_1 \rtimes P$ is (isomorphic to) a subgroup of $\Aff(A,P)$ having finite index
and so finitely presentable since $\Aff(A,P)$ is so.
Choose a free presentation $\pi \colon  F\epi A_1 \rtimes P$  
with $F$ a finitely generated free group.
Set $R = \ker \pi$ and  $S = \pi^{-1}(A_1)$.
Then $S$ is a free group and so the Schur-Hopf formula 
permits one to express $H_2(A_1)$ as $(R\cap S')/[R,S]$
(see, \eg \cite[section VI.9]{HiSt97}).
\index{Homology Theory of Groups!Schur-Hopf formula}%
Since $(R\cap S')/[R,S]$  is a submodule of the $\Z{P}$-module $R/[R,S]$
and as the latter is finitely generated and $\Z{P}$ is noetherian, 
the $\Z{P}$-modules $(R\cap S')/[R,S]$ and $H_2(A_1, \Z)$ are finitely generated over $\Z{P}$, 
and so $H_0(P, H_2(A_1))$ is a finitely generated abelian group.
\end{proof}
%
\subsection{Sufficient conditions for finite presentability}
\label{ssec:13.3New}
%
In this section we formulate conditions 
which guarantee that a group of the form $G= G(\R;A,P)$ admits a finite presentation.
As these conditions might puzzle the reader 
we begin by motivating them.
%
\subsubsection{Motivation}
\label{sssec:13.3aNew}
%
The homomorphism $\lambda \colon G \epi \Aff(A,P) \iso A \rtimes P$ has a splitting,
given by the obvious inclusion of $\Aff(A,P)$ into $G$.
The group $G$ is thus a semidirect product $N \rtimes \Aff(A,P)$
with $N$ the normal subgroup generated by the subgroup $G([0,\infty[\,;A,P)$.
Assume now that $G$ is \emph{finitely generated}.
Then so is $P$; as $P$ is torsion-free abelian, it is free abelian; let $\PP$ be a basis of  $P$. 
Proposition \ref{PropositionD1} thus applies 
and provides one with an infinite presentation of $G$.
This presentation has three types of generators,
homotheties $f(p) = \aff(0,p)$ with $p \in \PP$,
translations $h(a) = \aff(a,1)$ with $a \in A$, 
and the generators $g(p) = g(0,p)$ with $p \in \PP$.
This generating set is infinite,
but this is solely due to the translations $h(a)$ 
with $a$ ranging over the infinite $\Z[P]$-module $A$.
Note that the module $A$ is finitely generated, 
for the group $G$, and hence its quotient $\Aff(A,P) \iso A \rtimes P$, are so.

Assume next that $G$ is \emph{finitely presented}.
By Proposition  \ref{PropositionD2}
the quotient group $\Aff(A,P) \iso A \rtimes P$ of $G$ is then finitely presented, too. 
Let $\Acal \subset A$ be a finite set generating $A$ as a $\Z[P]$-module.
Then $\Aff(A,P)$ has a finite presentation of the form
\begin{equation}
\label{eq:Presentation-Aff(A,P)}
\Aff(A,P) 
= 
\left\langle \Acal \cup \Fcal 
\,\Big|\, 
\{\, [f(p), f(p')]  \mid (p,p') \in \PP^2 \} \cup \RR \} \right\rangle
\end{equation}
where $\Fcal = \{f(p) \mid p \in \PP \}$ and $\RR$ is a finite set of relators
which ensure that $A$ is an abelian group with the correct module structure.
In the presentation \eqref{eq:Presentation-Aff(A,P)},
each homothety $f(q)$ is represented 
by a product of homotheties $h(p_j)^{\varepsilon_j}$ 
with $p_j \in \PP$ and $\varepsilon_j \in \{1,-1\}$,
while the translation $h(a)$ with amplitude $a$ is represented by a product of the form
$\act{h(q_1)}a_1 \cdots \act{h(q_m)} a_m$ 
with each $q_\ell \in P$
and each $a_\ell \in \Acal \cup \Acal^{-1}$.
The relations \eqref{eq:13.7} and \eqref{eq:13.8} are then satisfied,
and hence in principle redundant.
But they are very convenient in later calculations.

Let's enlarge presentation  \eqref{eq:Presentation-Aff(A,P)}
to a presentation of $G$.
Proposition \ref{PropositionD1} shows how this can be done:
one has to add the (finite) set of generators 
$\GG = \{g(p) = g(0, p)\mid p \in \PP\}$,
and add the subset 
\begin{equation}
\label{eq:Remaining-commutators}
[f(p), g(p')] = 1 = [g(p), g(p')] \quad\text{for}\quad  (p, p') \in \PP^2
\end{equation}
of the set of commutator relations \eqref{eq:13.6},
and the infinite set of conjugation relations \eqref{eq:13.9},
namely
\begin{equation}
\label{eq:Conjugation-relations}
\act{g(p) h(b)}g(p') = \act{f(p) h(b)}g(p').
\end{equation}
In this last set of relations, the elements $p$ and $p'$ range over $\PP$ 
and $b $ ranges over the subset $B = A_{> 0}$.
To deal with this infinite set of relations
we introduce a condition that will be formulated and investigated next.
%
\subsubsection{The result and its proof}
\label{sssec:13.3bNew}
%
The subset $A_{>0}$ inherits two algebraic operations 
from the $\Z[P]$-module $A$:
is is closed under addition, thus a sub-semi-group of $A_{\add}$,
and it is stable under multiplication by $P$.
We can thus view $A_{>0}$ as \emph{a semi-group with operators in $P$}.
\index{Sub-semi-group with operators!definition|textbf}%
If this structure is finitely generated 
the group $G$ can conceivably be shown to be finitely presented.

We have only been able to establish a weaker result of this kind,
namely
\begin{proposition}
\label{PropositionD4}
\index{Group G(R;A,P)@Group $G(\R;A,P)$!finite presentation}%
\index{Finiteness properties of!G(R;A,P)@$G(\R;A,P)$}%
\index{Sub-semi-group with operators!application}
Suppose $\Aff(A,P)$ is finitely presented 
and $B = A_{>0}$ is finitely generated 
as a semi-group with operators in $P$.
If $P$ contains a rational number distinct from 1
then $G(\R;A,P)$ admits a finite presentation.
\end{proposition}

\begin{proof}
Let $\PP$ be a basis of the group $P$ and set $\Fcal = \{f(p) \mid p \in P \}$.
The module $A$ is a finitely generated over $\Z[P]$;
let $\Acal \subset A$ be a finite set of generators. 
The group $G$ is then generated by the finite set
\[
\XX = \Fcal \cup \Acal \cup\;  \{g(p) \mid p \in \PP\}.
\] 
Its subset $\Fcal \cup \Acal$ generates the complement $\Aff(A,P)$,
and this complement admits a finite presentation of the form 
\eqref{eq:Presentation-Aff(A,P)}.
Our task is to show that the infinite set of relations \eqref{eq:Conjugation-relations},
namely
\begin{equation}
\label{eq:Conjugation-relations-bis}
\act{g(p) h(b)}g(p') = \act{f(p) h(b)}g(p'),
\end{equation}
follow from finitely many among them, 
from the relations in presentation \eqref{eq:Presentation-Aff(A,P)}
and the commutativity relations \eqref{eq:Remaining-commutators}.

Let $\BB$ be a finite set that generates $B = A_{>0}$ 
as a semi-group with operators in $P$,
and let $n/d \in P \smallsetminus \{1\}$ be a rational number;
we may assume that $n > d > 0$.
Put
\begin{equation*}
\tilde\BB 
=  
\left\{
\sum\nolimits_b m(b) \cdot b \;\Big|\; b \in \BB \text{ and } m(b) \in \{0,1,\hdots,n-1\} 
\right\}.
\end{equation*}
Then $\tilde\BB $ is a finite set.
 
Let $F(\XX)$ be the free group on $\XX$
and let $\equiv$ denote the congruence relation on $F(\XX)$ 
that is generated by the relators in presentation \eqref{eq:Presentation-Aff(A,P)},
by the commutator relations \eqref{eq:Remaining-commutators}
and by those relations of the form \eqref{eq:Conjugation-relations-bis}
with $(p,b,p')$ in $\PP \times \tilde{\BB} \times \PP$. 
\emph{We claim that the congruence}
\begin{equation}\label{eq:13.12}
\act{g(p) h(b)}g(p') = \act{f(p) h(b)}g(p')
\end{equation}
\emph{holds for each element $b \in B = A_{>0}$ and each pair $(p,p')$ in $\PP^2$}.  

Let $\Bfr$ denote the subset of those elements $b \in B = A_{>0}$ 
for which congruence \eqref{eq:13.12} holds for all $(p,p')$ in $\PP^2$.  
Then $\Bfr$ enjoys the following properties.
\begin{enumerate}[(I)]
\item  $\tilde{\BB} \subset \Bfr$.
\item  $P\cdot \Bfr = \Bfr$.
\item  If $b$, $b'$ and $b+b'$ are in $\Bfr$ and if $q$ lies in $P$ 
then $b + qb'$ belongs to  $\Bfr$.
\end{enumerate}

Property (I) holds by the definition of the congruence relation $\equiv$;
property (II) is valid  because of the commutativity relations 
\eqref{eq:Remaining-commutators}, 
relations \eqref{eq:13.7}
and the following calculation,
valid for every $q \in \PP \cup \PP^{-1}$:
\begin{align*}
\act{g(p) h(q b)} g(p')
&=
\act{g(p) \act{f(q)}h(b)} g(p')
=
\act{g(p) (f(q)h(b))} g(p')
=
\act{f(q) (g(p) h(b))} g(p')\\
&=
\act{f(q) (f(p) h(b))} g(p')
=
\act{f(p) (f(q) h(b))} g(p')
=
\act{f(p) h(qb)} g(p').
\end{align*}
 Property (III) will be established by induction 
on the minimal number $\ell$ of letters in $\PP \cup \PP^{-1}$ 
which are needed to express $q$.
If $\ell = 0$ then $b + q b' = b +  1 \cdot b'$ 
and this sum is in $\Bfr$ by hypothesis.

The case $\ell = 1$ requires a rather lengthy argument 
that will be given below.
Suppose now that $\ell = 2$ 
and that $q = p_2^{\varepsilon_2}p_1^{\varepsilon_1}$ 
with $\varepsilon_i \in \{1, -1\}$ and $i \in \{1,2\}$.
Assume $b$, $b'$ and $b+ b'$ lie in $\Bfr$.
The case $\ell = 1$ (to be proved below) then guarantees 
that $b + p_1^{\varepsilon_1} b' \in \Bfr$.
On the other hand, 
$p_1^{\varepsilon_1} b' \in \Bfr$ by property (II) 
and the assumption that $b' \in \Bfr$.
But if so,
the elements $b$, $p_1^{\varepsilon_1}b' $ 
and  $b + p_1^{\varepsilon_1} b'$ lie all three in $\Bfr$,
whence 
$b + q b' = b + p_2^{\varepsilon_2}(p_1^{\varepsilon_1} b')$
lies in $\Bfr$, again by the case $\ell = 1$.
Iterate.

Now to the verification of the case $\ell = 1$.
It will involve a calculation in which $p'$ is fixed 
and only the conjugating elements in the exponents are modified.
In order to render the calculation more readable, 
we write $w_1\approx w_2$ to express 
that  $w_1$ and $w_2$ are words of $F(\XX)$ 
for which the relation 
\begin{equation*}
\act{w_1}g(p') \equiv \act{w_2}g(p')
\end{equation*} 
holds for every choice of $p' \in \PP$; 
the relations then also hold for every $p' \in \PP^{-1}$. 
In the calculation,
we shall also make use of the relations
\eqref{eq:13.7} and \eqref{eq:13.8},
namely
\[
h(p \cdot b) = \act{f(p)} h(b) \quad \text{and}\quad 
h(b_1 + b_2) = h(b_1) \cdot h(b_2). 
\]


Assume now that $b$, $b'$ and $b+b'$ belong to $\Bfr$, 
and fix $q \in \PP \cup \PP^{-1}$. 
We aim at proving the equivalence 
\[
g(p)\cdot h(b +q \cdot b') \approx f(p)\cdot h(b+q \cdot  b').
\]
The proof will consist of a chain of equivalences.
At one point in the chain
one uses the hypothesis that $b \in \Bfr$,
but written in the form
\begin{equation}
\label{eq:Special-form-of-hypothesis}
(g(p) h(b)) \cdot g(p') \cdot (g(p) h(b))^{-1} \equiv h(pb)  \cdot g(p') \cdot  h(pb)^{-1}.
\end{equation}
This relation is valid every $p' \in \PP \cup \PP^{-1}$ 
and so, in particular, for $q$.
The following chain of equivalences now holds:
\begin{align*}
g(p)\cdot h(b+q \cdot b')
&\approx 
g(p)\cdot h(b) h(q \cdot b') 
\approx 
g(p) h(p)  \cdot 1 \cdot  h(q \cdot b') 
\qquad \text{(use \eqref{eq:13.8})}\\
&\approx
 (g(p)h(b)) \cdot 
 \left( g(q) \cdot (g(p)h(b))^{-1}\cdot g(p)h(b) \cdot  g(q)^{-1}\right) 
 \cdot  h(q \cdot b')\\
&\approx 
 \left(g(p)h(b) \cdot g(q) \cdot (g(p)h(b))^{-1}\right) 
 \cdot \left( g(p)h(b)  g(q)^{-1} \cdot  g(q) \cdot h( b')\right) \\
&\approx 
\left( h(p b) \cdot  g(q) \cdot h(p b)^{-1}\right)
\cdot 
\left(g(p)h(b) \cdot h(b')\right) 
\qquad \text{(by \eqref{eq:Special-form-of-hypothesis}) }\\
&\approx
\left(h(pb)g(q)h(pb)^{-1} \right) \cdot g(p)h(b+b')\\
&\approx
\left( h(pb)g(q)h(pb)^{-1} \right)\cdot h(p(b+b'))
\qquad \text{(since $b + b' \in \Bfr$)}\\
&\approx
h(pb)\cdot g(q)\cdot h(-pb + p(b + b'))
\approx h(pb)\cdot g(q)\cdot h(p b')
\\
 &\approx
h(pb)\cdot h(q \cdot  pb')
\qquad \text{(use (II) and $b' \in \Bfr$)}\\
& \approx h(pb + p q b')
\approx
h(p(b+q b'))\\
&\approx
f(p)\cdot h(b+q \cdot  b')
\qquad \text{(by relation \eqref{eq:13.7}).}
\end{align*}

The claim  that $\Bfr = B$ will now be deduced from property (III) 
by a double induction, 
the induction parameter being an ordinal $k\omega_0 + \ell$,
called \emph{weight}.  
This induction parameter is defined like this:  
every $b \in B$ has a representation of the form
\begin{equation}
\label{eq:13.13}
\left(p_1b_1 + p_2b_2 + \cdots + p_k b_k \right)
+ 
\left(1 \cdot b_{k+1} + \cdots + 1\cdot b_{k+\ell} \right),
\end{equation}
where the elements $b_1$,\ldots, $b_{k+\ell}$ lie in $\BB$ 
and the elements $p_1$,\ldots, $p_k$ in $P\smallsetminus  \{1\}$.  
The weight of representation \eqref{eq:13.13} is defined to be $k\omega_0 + \ell$.  

The induction starts with elements having a representation of finite weight, 
\ie{}with sums of the form 
\begin{equation*}
b_* = \sum\nolimits_{b \in \BB} m(b) \cdot b
\end{equation*}
where each $m(b)$ is a non-negative integer.
If $m(b) < n $ for each $b \in \BB$
then $b_* \in \Bfr$ by the definition of $\tilde{\BB}$ and property (I).  
If, on the other hand,  $m(b_1) \geq n$ for some $b_1 \in \BB$,
set $b = b_* - n \cdot b_1$ and $b' = d \cdot b_1$.  
Then $b$, $b'$ and $b+b'$ have representations of smaller weight 
than has $b_*$, 
and so they lie in $\Bfr$ by the induction hypothesis 
while $b_* = b + (n/d)b'$ is in $\Bfr$ by property (III)
and the assumption that $n/d \in P$.

Suppose, finally, that $k > 0$ 
and that each element $b_0 \in B$ with weight less than $k\omega_0$ is in $\Bfr$.  
If $b_*$ has a representation of weight $k\omega_0 + \ell$, 
set $b = b_*-p_1b_1$ and $b' = b_1$.  
Then $b$, $b'$ and $b+b'$ have representations with weight
smaller than $k\omega_0$ and $b_* = b+p_1b'$, 
whence $b_*$ is in $\Bfr$ by property (III) and the inductive assumption.  
So  $\Bfr = B = A_{> 0}$; 
in other words, 
the congruence \eqref{eq:13.12} is valid for every $b \in A_{>0}$ and each couple $(p,p') \in \PP^2$
in the finitely presented group defined just before equation \eqref{eq:13.12}.
So we have constructed a finite presentation of the group $G(\R;A,P)$
and thus Proposition \ref{PropositionD4} is established.
\end{proof}
%
\subsubsection{Some examples of finitely presented groups}
\label{sssec:13.3cNew}
\index{Group G(R;A,P)@Group $G(\R;A,P)$!finite presentation|(}%
%
We shall discuss two kinds of examples.
The metabelian quotients $\Aff(A,P)$ of the groups of the first kind will be \emph{constructible} 
(in the sense of Baumslag-Bieri \cite{BaBi76}) 
and so their finite presentability will be obvious.
\index{Baumslag, G.}%
For the examples of the second kind,  
the finite presentability of $\Aff(A,P)$ will rely on the theory developed
by Bieri-Strebel in \cite{BiSt80} and \cite{BiSt81a}.
\index{Sub-semi-group with operators!examples}%
\index{Bieri, R.}%
\index{Strebel, R.}%
\smallskip

1)  Let $p_1$, $p_2$,\ldots, $p_k$ be positive integers greater than 1
with the property that $\PP = \{p_1,\hdots,p_k\}$ is a basis of the group $P$ generated by $\PP$.
Set $A = \Z[P]$.  
Then each positive element of $A$ is a positive multiple of some power of
\begin{equation*}
p = (p_1\cdot p_2 \cdots p_k)^{-1}
\end{equation*}
and so $b_1 = 1$ generates  $B = A_{>0}$ 
as a semi-group with operators in $P$.

The group $\Aff(A,P)\cong A \rtimes P$ is finitely presented, 
as it can be obtained by a sequence of $k$ ascending HNN-extensions, 
with $\Z$ as the base group  of the first HNN-extension 
and the group obtained in step $i -1$ as the base group of the $i$-th HNN-extension
(see, \eg{}\cite[Section 1.2]{Str84} for further details).
\index{Strebel, R.}%
By Proposition \ref{PropositionD4} the group $G(\R;\Z[P],P)$ admits therefore a finite presentation.

2) In the preceding examples,
the metabelian group $A\rtimes P$ has been built up in a manner which implies directly 
that it admits a finite presentation. 
Now many more subgroups $\Z[P] \rtimes P$ of $\Q \rtimes \Q^\times_{> 0}$ 
admit a finite presentation.  
\footnote{The necessary and sufficient conditions
are spelled out in \cite[p.\ 299, Example 34]{Str84}.}  
\index{Bieri, R.}%
\index{Strebel, R.}%
For some of these finitely presented groups $\Z[P] \rtimes P$
the semi-group  $B = Z[P]_{>0}$ can be shown to be finitely generated
with $P$ as set of operators.
Let $J$ denote the product of all the (positive) primes 
occurring in the prime factorizations of the elements of $P$.
Then $\Z[P] = \Z[1/J]$. 
Suppose the fraction $1/J$ can be written in the form
\begin{equation}
\label{eq:13.14}
1/J = q_1 + q_2 + \cdots + q_m
\end{equation}
where every  $q_i$ lies in $P$.
Each positive power of $1/J$ has then a representation of the same form
and so $B$ is generated by 1 as a semi-group with operators in $P$.

Here is an illustration. 
Suppose $n_1$, $n_2$ and  $d$ are relatively prime integers greater than 1
and $P$ is the group $\gp(\{ n_1/d, n_2/d\})$. 
Set $A = \Z[P] = \Z[1/(n_1n_2 d)]$.  
Then $\{n_1/d ,n_2/d \}$ is a basis of $P$ and  $A\rtimes P$ is finitely presented 
(by Example 34 in \cite{Str84}).
\index{Strebel, R.}%
It suffices now to find a representation of the form \eqref{eq:13.14},
for one of numbers $1/n_1$, $1/n_2$ or $1/d$,
as the remaining two numbers are $P$-multiple of any one of them.

Let us consider first the special case where $n_1 = 2$ and $ n_2 = 3$.
Then $1/d$ has a representation of the form \eqref{eq:13.14}
for all odd integers $d \geq 11$.
Indeed, $11^2 = 121 = 5 \cdot 2^3 + 3 \cdot 3^3$
and so
\[
(11 + 4k)^2 = 11^2 + 88 k + 16 k^2 = (5 + 11k+ 2k^2) \cdot 2^3 + 3 \cdot 3^3.
\]
Every denominator $d$ of the form $11 + 4k$ satisfies therefore the relation
\[
1/d =   (5 + 11k + 2k^2) \cdot (n_1/d)^3 + 3 \cdot (n_2/d)^3.
\]
The relation $13 = 2^2 + 3^2$, on the other hand,
leads to a representation of $1/d$ for all $d$ of the form $13 + 4k$.

A similar conclusion holds for arbitrary positive, 
relatively prime integers $n_1$ and $n_2$.
The squares $n_1^2$, $n_2^2$ are then relatively prime
and so a simple argument, 
based on the generalized Euclidean algorithm, 
discloses
that every integer $d \geq (n_1^2-1) \cdot (n_2^2-1)$ has a representation of the form
\[
d = c_1 \cdot n_1^2 + c_2 \cdot n_2^2 \quad \text{with} \quad c_1, c_2 \text{ in } \N
\] 
whence $1/d = c_1 (n_1/d)^2 + c_2 (n_2/d)^2$.
It follows that \emph{the group} 
\[
G\left(\R;\Z[1/(n_1n_2 d)], \gp(n_1/d, n_2/d)\right)
\] 
\emph{admits a finite presentation
whenever  $n_1$, $n_2$, $d$ are relatively prime integers greater than 1
and $d \geq (n_1^2-1) \cdot (n_2^2-1)$}.
\index{Group G(R;A,P)@Group $G(\R;A,P)$!finite presentation|)}%
 
\begin{remarks}
\label{remarks:Fp-groups-with-interval-a-line}
(i) A further class of groups deserves to be drawn to the attention of the reader.  
Let $K\subseteq \R$ be an algebraic number field with ring of integers $\OO$.  
If $K \neq \Q$ the group of positive units $P = \OO^\times_{> 0}$ is non-trivial 
and so we can consider $G = G(\R; \OO,P)$.  
The quotient group $\Aff(\OO,P)$ is polycyclic, hence finitely presented, 
and $\OO/(IP\cdot \OO)$ is finite.  
So the necessary conditions, given by Proposition \ref{PropositionD2}, hold; 
but as $P$ contains no rational number $\neq 1$,
our sufficient condition is not satisfied. 
We do not know 
whether any of these groups $G(\R;\OO,P)$ admits a finite presentation.

(ii) The module $A$ of every group $G(\R;A,P)$ 
which, so far, has been shown to be finitely presented, 
 is a locally cyclic as a group.
\end{remarks}
%
\section{Presentations of groups with supports in a half line}
\label{sec:14}
%
The group $G = G([a,\infty[\,;A,P)$ is generated by its elements with one singularity
and it admits an infinite presentation 
in terms of this generating set  
that is easy to describe (see section \ref{ssec:14.1}).
If $G$ admits a finite presentation so does its quotient group $\Aff(IP \cdot A, P)$ 
and $A/(IP \cdot A)$ is finite
(see Propositions \ref{PropositionB1} and  \ref{PropositionD5}).
In this section, we shall also establish the following converse:
\emph{if $P$ is generated by finitely many integers 
and $A = \Z[P]$ then $G$ is finitely presented} (see Theorem \ref{TheoremD6}).
\subsection{An infinite presentation}
\label{ssec:14.1}
By section \ref{ssec:8.1}
the group $G([a,\infty[\; ;A,P)$ is generated by the subset of its elements with (at most) one singularity,
namely the set
\begin{equation}
\label{eq:14.1}
\XX_a = \{g(b,p) \mid a\leq b \in A\ \text{ and }\ p \in P\}.
\end{equation}
In this formula,
$g(b, p)$ denotes the PL-homeomorphism
given by equation \eqref{eq:8.1}.
These generators satisfy the relations 
\begin{align}
g(b,p) \cdot g(b,p') &= g(b,pp') \quad \text{and } \label{eq:14.2}\\
\act{g(b,p)}g(b',p') &= g(b+p(b'-b),p')\ \text{ if }\ b < b', \label{eq:14.3}
\end{align}
\index{Group G([0,infty[;A,P)@Group $G([0, \infty[\;;A,P)$!infinite presentation}%
where $(p,p')$ ranges over $P^2$ 
and $b$, $b'$ are elements of $A\cap [a,\infty[$.  
One can show by a normal form argument
that relations \eqref{eq:14.2} and \eqref{eq:14.3} 
are defining relations of $G([a,\infty[\,;A,P)$
in terms of the generating set $\XX_a$; 
see \cite[Section 2]{BrSq85}.

Suppose $P$ is free abelian with basis $\PP$.  
Then the generators $g(b,p)$ with $p \in P\smallsetminus \PP$ can be eliminated.  
If this is done, 
one ends up with a presentation of $G([a,\infty[\,;A,P)$ with generating set
\begin{equation}
\label{eq:14.1p}
\YY_a = \{g(b,p)\mid a \leq b\in A\ \text{ and }\ p \in \PP\}
\end{equation}
and defining relations
\begin{gather}
[g(b,p),g(b,p')] = 1,\label{eq:14.2p}\\
\act{g(b,p)}g(b',p') = g(b+p(b'-b),p') \text{ if }\ b < b' \label{eq:14.3p}
\end{gather}
where $(p,p')$ runs over $\PP^2$ and $b$, $b'$ are in $A\cap [a,\infty[\,$.

%
\subsection{Necessary conditions for a finite presentation}
\label{ssec:14.2}
By Theorem \ref{TheoremB4} the group $G = G([a,\infty[\,;A,P)$ is finitely generated precisely 
if $A\rtimes P$ is finitely generated, $A/(IP\cdot A)$ is finite and $a$ lies in $A$.  
If $G$ is finitely presented, $A\rtimes P$ must be finitely presented according to 
\begin{proposition}
\label{PropositionD5}
\index{Group G([0,infty[;A,P)@Group $G([0, \infty[\;;A,P)$!finite presentation}%
\index{Affine group Aff(A,P)@Affine group $\Aff(A,P)$!significance}%
If $G([a,\infty[\;;A,P)$ has a finite presentation 
the same is true for $A\rtimes P$.  
\end{proposition}

\begin{proof}
By \cite[Section 3]{BrSq85}
the group $G = G([a,\infty[\,;A,P)$ has no non-abelian free subgroups.  
\index{Brin, M. G.}%
\index{Squier, C. C.}%
The theory developed in \cite{BiSt80} therefore applies 
and shows  that every metabelian quotient of $G$ is finitely presented 
if $G$ is so (see \cite[p.\ 462, Theorem 5.5]{BiSt80}).  
\index{Bieri, R.}%
\index{Strebel, R.}%
In particular, 
the quotient group $\im\rho = \Aff(IP\cdot A,P)\cong IP\cdot A\rtimes P$ is then finitely presented.  
As $IP\cdot A$ has finite index in $A$ (by Proposition \ref{PropositionB1}), 
the group $A\rtimes P$ is finitely presented, too.  
\end{proof}

\subsection{A Sufficient condition for finite presentability}
\label{ssec:14.3}
%
M. G.\ Brin and C. C.\ Squier prove in Section 2 of \cite{BrSq85} 
\index{Brin, M. G.}%
\index{Squier, C. C.}%
that, for every integer $p\geq 2$, the group
\begin{equation*}
\index{Group G({p})@Group $G(\{p\})$!definition|textbf}%
G(\{p\}) = G([0,\infty[\,;\Z[1/p],\gp(p))
\end{equation*}
has a finite presentation with $p$ generators and $p(p-1)^2$ defining relations
(see \cite[p.\,492]{BrSq85}).  
We shall improve on this result and show 
that, for every non-empty finite set $\PP$ of integers $p \geq 2$, 
the group
\begin{equation}
\label{eq:14.4}
\index{Group G(PP)@Group $G(\PP)$!definition|textbf}%
G(\PP) = G([0,\infty[\;;\Z[\gp(\PP)],\gp(\PP))
\end{equation}
is finitely related. 
More precisely, 
we shall establish
\begin{theorem}
\label{TheoremD6}
\index{Group G(PP)@Group $G(\PP)$!finite generating sets}%
\index{Group G(PP)@Group $G(\PP)$!finite presentation}%
\index{Finiteness properties of!G([0,infty[;A,P)@$G([0, \infty[\;;A,P)$}%
\index{Finiteness properties of!G(PP)@$G(\PP)$}%
Let $\PP$ be a finite (non-empty) set of integers $p \geq 2$.
Fix an element $p_0$ in $\PP$, 
\eg{}$p_0 = \min\{p\mid  p \in \PP\}$.
Then $G(\PP)$ is generated by the set
\begin{equation}
\label{eq:Finite-generating-set}
\GG(\PP,  p_0) =
\{g(i, p) \mid i \in \{0,1, \ldots, p_0 - 1 \} \text{ and } p \in \PP \}
\end{equation}
and it admits a finite presentation.
\end{theorem}

The proof will comprise four stages.  
In the first one 
we show that the subpresentation obtained from equations 
\eqref{eq:14.1}, \eqref{eq:14.2} and \eqref{eq:14.3}
by restricting the parameters $b$, $b'$ to $\N$ and $p$, $p'$ to $\PP$ 
gives a presentation of $G(\PP)$.  
\footnote{Here, and in the remainder of Section \ref{sec:14},
$\N$ is assumed to contain the number 0.}
In the next stage 
we verify that the group $G(\PP)$ is generated by the finite set 
\eqref{eq:Finite-generating-set} 
and recall the notion of Tietze transformations.
In the third stage,
we derive a convenient presentation of $G(\PP)$ 
with a finite set of generators which has one more element
than the set \eqref{eq:Finite-generating-set}.
In the final stage we establish 
that all but finitely many defining relations of the presentation 
arrived at in the third stage
can be discarded.
%
\subsubsection{Proof of Theorem \ref{TheoremD6}: first stage}
\label{sssec:14.3aNew}
%
The aim of this stage is to obtain a presentation of $G(\PP)$ with generating set
\begin{equation}
\label{eq:Generating-set-XX}
\index{Group G(PP)@Group $G(\PP)$!infinite generating set}%
\index{Group G(PP)@Group $G(\PP)$!infinite presentations}%
 \XX =  \{f(i,p) \mid i\in \N\ \text{ and }\ p \in \PP \}.
\end{equation}
Let $F(\XX)$ denote the free group on $\XX$.  
As a first step towards the presentation of $G(\PP)$ with generating set $\XX$, 
we choose a presentation of $P = \gp(\PP)$ as an abelian group with set of generators $\PP$.
Let $p_1$, $p_2$,\ldots, $p_m$ be an enumeration of the elements of $\PP$ 
and let $\alpha \colon  \Z^m\epi P$ be the epimorphism 
which sends the $i$-th standard basis element of $\Z^m$ to $p_i$.  
Then find a finite subset 
\begin{equation*}
\{(x(1,1),\hdots,x(m,1)),\quad (x(1,2), \hdots,x(m,2)),
\quad \hdots\quad,  (x(1,k), \hdots  ,x(m,k))\}
\end{equation*}
of $\Z^m$  which generates the kernel of $\alpha$.  
For each $i\geq 0$ and each $j \in \{1, 2. \ldots, k\}$,
set
\begin{equation}
\label{eq:14.5}
w(i,j) = f(i,p_1)^{x(1,j)}f(i,p_2)^{x(2,j)}\cdots f(i,p_m)^{x(m,j)}.
\end{equation}

Having defined the elements $w(i,j)$ 
we can describe the set of relators:  
it is the union of the three infinite sets

\begin{align}
\RR_0 &= \{w(i,j)\mid i \in \N\ \text{ and }\ j \in \{1,2,\hdots,k\}\},
\label{eq:Relators-RR0}\\
\RR_1 &= \{[f(i,p),f(i,p')] \mid i\in \N\ \text{ and }\ (p,p') \in \PP^2\}
\label{eq:Relators-RR1}
\end{align}
and
\begin{equation*}
\RR_2 
= 
\{\act{f(i,p)}f(i',p') \cdot f(i+p(i'-i),p')^{-1}
\mid 
(i , i') \in \N^2, \; i < i' \text{ and }(p,p') \in \PP^2 \}.
\end{equation*}

Set $\RR = \RR_0 \cup \RR_1 \cup \RR_2$,
let $R$ be the normal subgroup of $F(\XX)$ generated by $\RR$
and let $\pi \colon F(\XX)/R \to G(\PP)$ denote the homomorphism
sending, 
for $i \in \N$ and $p \in \PP$,
 the generator $f(i,p)$ to $g(i,p)$.  
In view of relations \eqref{eq:14.2} and \eqref{eq:14.3} 
the PL-homeomorphism
\begin{equation*}
\act{(g(0,p_1)\cdots g(0,p_m))^{-k}}g(i,p')
\end{equation*}
coincides with  $g(i/(p_1\cdots p_m)^k,p')$. 
The image of $\pi$ contains therefore the generating set  
\eqref{eq:14.1} of $G(\PP)$   
and so $\pi$ is \emph{surjective}.

The \emph{injectivity} of $\pi$ will be established by a normal form argument. 
The set $\RR_2$ consists of relators having the form 
$\act{f(i,p)}f(i',p') \cdot f(i'',p')^{-1}$
where $i < i'$ and $i'' = i + p(i' - i)$.
 Then $i < i''$.
The relator $\act{f(i,p)}f(i',p') \cdot f(i'',p')^{-1}$
gives rise to the following four equivalent relations
\begin{align}
f(i,p) f(i',p')&=  f(i'',p') f(i,p),
\label{Moving-to-the-right-1} \\
 f(i,p)  f(i',p')^{-1} &=  f(i'',p')^{-1}f(i,p),
 \label{Moving-to-the-right-2} \\
f(i',p') f(i,p) ^{-1}&=  f(i,p)^{-1} f(i'',p'),
 \label{Moving-to-the-right-3} \\
f(i',p')^{-1} f(i,p)^{-1}&=  f(i,p)^{-1}  f(i'',p')^{-1}.
\label{Moving-to-the-right-4}
\end{align}

Consider now a non-trivial word $w$, with letters in $\XX \cup \XX^{-1}$.
Relations \eqref{Moving-to-the-right-2} and \eqref{Moving-to-the-right-3},
along with free reduction and the commutator relations in $\RR_1$,
permit one to move every$f(i_1,p_1)$ in $w$ 
to the right of every inverse $f(i_2,p_2)^{-1}$ of a generator. 
It follows that $w$  is congruent,
modulo the relators in $\RR_1 \cup \RR_2$, to a freely reduced word $v_1^{-1} \cdot v_2$,
where both $v_1$ and $v_2$  are positive words in $\XX$.
Relations \eqref{Moving-to-the-right-1} 
next allow one to transform the second factor $v_2$ into a positive word
\begin{equation} 
\label{eq:Word-u2}
u_2 = f(i'_{k'}, p'_{k'}) \cdots f(i'_2,p'_2) \cdot f(i'_1, p'_1)
\end{equation}
with
$i'_{k'} \geq \cdots \geq i'_2 \geq i'_1$ and each $p'_i \in \PP$.
Similarly,
relations \eqref{Moving-to-the-right-4} 
can be used to change the factor $v_1^{-1}$ 
into  the inverse of a positive word 
\begin{equation} 
\label{eq:Word-u1}
u_1 = f(i_{\ell}, p_{\ell}) \cdots f(i_2,p_2) \cdot f(i_1, p_1)
\end{equation}
with
$i_{\ell} \geq \cdots \geq i_{2} \geq i_1$ and with each $p_i \in \PP$.

Assume now that $w$ represents an element in $\ker \pi$.
Then so does $w' = u_1^{-1} \cdot u_2$. 
The positive words $u_1$ and $u_2$ represent therefore the same PL-homeomorphism $h \in G(\PP)$.
If both words are empty, all is well.
If one of them is non-empty we may assume 
that $u_1$ is non-empty 
and that $i_1$ is the smallest value of an index $i$ of a generator $f(i,p)$ 
which occurs in $u_1$ or in $u_2$.
Let $u_{12}$ be the terminal segment of $u_1$ formed by all the letters $f(i, p)$ with $i > i_1$
and let $u_{11}$ be the initial segment preceding $u_{12}$.
Then the functions $h = \pi(u_1R)$ and $h_{12} = \pi(u_{12}R)$ are the identity on $[0, i_1]$ 
and have the same slope on $]i_1, i_1 + 1[\,$. 
If this slope is 1 the relators in $\RR_0 \cup \RR_1$ show 
that $u_{12} \in R$ and so $\pi(u_1 R) =  \pi(u_{11} R)$.
If it is non-zero then $i_1' = i_1$; 
let $u_{22}$ be the largest terminal segment of $u_2$ 
formed by the letters $f(i, p)$ with $i = i_1$.
Then the slopes of the functions $h$, $\pi(u_{12}R)$ and $\pi(u_{22}R)$ 
agree on the interval $]i_1, i_1 + 1[$;
since the relators $\RR_0 \cup \RR_1$ define the group $P$ 
and are contained in $R$
the terminal segments $u_{12}$ and $u_{22}$ are congruent \emph{modulo} $R$
and so  $\pi(u_{11}R) =\pi(u_{21}R)$;  
here $u_{21}$ denotes the initial segment of $u_2$ which precedes $u_{22}$.

The above arguments are the ingredients of a proof by induction 
which allows one to conclude that $\pi \colon F(\XX)/R \to G(\PP)$ is \emph{injective}.
%
\subsubsection{Application: abelianization of $G(\PP)$}
\label{sssec:14.3bNew}
%
The presentation of $G(\PP)$ found in the previous section 
allows one to determine the abelianization of $G(\PP)$ 
in the important case where $\PP$ is a basis of $P$.
\footnote{Actually, a result, due to Ken S. Brown, shows 
that every group $P$ generated by finite set of integers admits a basis $\PP_1$ 
that is made up of integers; 
see Proposition 1.1 in \cite{Ste92}.}
\index{Brown, K. S.}%
\index{Stein, M.}%

We begin with the case of a cyclic group $P$ generated by the integer $p > 1$.
The generators $f(i,p)$ have then a fixed second argument;
we set $f(i) = f(i,p)$ to simplify the notation.
The sets of relators $\RR_0$ and $\RR_1$ are redundant in this case.
Consider now a relator  $r = \act{f(i)} f(i') \cdot f(i + p(i'-i))^{-1}$ in $\RR_2$.
Since
\begin{equation}
\label{eq:Useful-relation}
i + p(i'-i) = i' + (p-1)(i'-i),
\end{equation}
the relator $r$ is equivalent to the relation $\act{f(i)} f(i') = f(i' + (p-1)(i'-i))$.
Replacing $i'$ by $j$
we conclude that $G(\{p\}$ admits the presentation
\begin{equation}
\label{eq:Infinite-presentation-of-G(PP)}
\left\langle 
\{f(i)\}_{ i \in \N } \mid \act{f(i)}f(j) = f(j  + (p-1)(j-i))
\text{ for } (i,j) \in \N^2 \text { and } i < j 
\right\rangle.
\end{equation}
The abelianization of $G(\{p\})$, viewed as an abelian group, 
has thus the presentation
\begin{equation*}
\label{eq:Infinite-presentation-of-G(PP)abenianized}
\left\langle 
\{x(i)\}_{ i \in \N } \mid x_j = x_{j + (p-1)(j-i)} 
\text{ for } (i,j) \in \N^2 \text { and } i < j 
\right\rangle.
\end{equation*}
This presentation shows 
that abelianized group is free abelian 
with a basis that corresponds bijectively 
to the set
$\{0\} \cup \left\{[j] \mid j \in \N \smallsetminus \{0\} \right\}$
where $[j]$ denotes the equivalence class, represented by $j$, 
of the equivalence relation $\equiv$ generated by the relation 
$j \approx  j + (p-1)(j-i)$.
It follows, first of all, 
that the equivalence class $[j]$ is contained in the equivalence class of the congruence relation \emph{modulo} $p-1$; 
upon setting $i = j -1$ one sees that $j + (p-1) \in [j]$. 
This proves that the set of equivalence classes is represented by
$0$, $1$, $2$, \ldots, $p-1$,
and implies 
that $G(\{p\})_{ab}$ \emph{is free abelian of rank $ p$}.

Let's now pass to the case 
where $P$ is generated by a basis $\PP$ consisting of integers greater than 1.
The previous arguments can be adopted readily to this more general case 
and show that, as an abelian group,
$G(\PP)_{\ab}$ has the presentation with generating set $\{x_{i,p} \mid (i,p) \in \N \times \PP\}$
and defining relations 
\[
x_{j,p'} = x_{j + (p-1)(j-i), p'}.
\]
These defining relations identify only generators with the same parameter $p'$
and the equivalence relation $\equiv$ generated by the relations $j  \approx_p j + (p-1)(j-i)$,
does not depend on the parameter $p' \in \PP$.
Every equivalence class of $\equiv$ is contained in the congruence class modulo 
$n = \gcd\{p-1 \mid p \in \PP\}$
and  $j + (p-1) \in [j]$ for every $p \in \PP$. 
The generalized euclidean algorithm thus implies 
that the equivalence classes of $\equiv$ 
are the congruence classes \emph{modulo} $n$ in $\N \smallsetminus \{0\}$.
So we have proved
\begin{lemma}
\label{lem:Abelianization-of-G(PP)}
\index{Group G(PP)@Group $G(\PP)$!abelianization}%
Let $P$ be a subgroup of $\Q^\times_{>0}$ generated by a finite set $\PP$ of integers.
Assume $\PP$ is a basis of the abelian group $P$.
Then the group $G(\PP)_{\ab}$ is free abelian of rank
\begin{equation} \label{eq:Rank-abelianized-group}
\card( \PP) \cdot \left(1 + \gcd\left\{p-1 \mid p \in \PP\right\}\right)
\end{equation}
\end{lemma}
%
\subsubsection{Proof of Theorem \ref{TheoremD6}: 
second stage}
\label{sssec:14.3cPreliminaries}
%
We begin by showing 
that the group $G(\PP)$ can be generated 
by a finite set $\GG(\PP,  p_0) $ with a very simple description. 
The presentation arrived at in section \ref{sssec:14.3aNew}
implies, in particular,
that $G(\PP)$ is generated by the infinite set 
\[
\pi(\XX) = \{g(i, p) \mid (i,p) \in \N \times \PP \}.
\]
Fix an element $p_0 \in \PP$.  
Given $i\geq p_0$,
let $r \in \{0,1,\hdots,p_0-1\}$ denote the remainder of $i$ divided by  $p_0$.
Then  $i = p_0 \cdot s + r$ for some positive integer $s$.  
Relations \eqref{eq:14.3} show
that $g(i,p)$ equals $\act{g(r,p_0)}g(r+s,p)$;
indeed,
\[
\act{g(r,p_0)} g(r+s, p) = g(r + p_0((r + s) - r), p) = g(i,p).
\]
Since $r + s < i$ the preceding calculation allows us to set up 
an induction with respect to $i$
which proves that $G(\PP)$ is generated by the finite set 
\begin{equation*}
\index{Group G(PP)@Group $G(\PP)$!finite generating sets}%
\GG(\PP, p_0) = \{g(r,p)\mid 0 \leq r < p_0\ \text{ and }\ p\in \PP\}
\end{equation*}
and so the first assertion of Theorem \ref{TheoremD6}
is established.

If one uses $\GG(\PP, p_0)$ as the generating set of a presentation $F_0/R_0$
that is derived from the presentation $F(\XX)/R$ 
by elimination of redundant  generators, 
one ends up, in general,  with an unwieldy set of relators.
The situation improves if one adds the extra generator
\begin{equation}
\label{eq:14.6} 
g =g(p_0,p_0)^{-1} g(0,p_0).
\end{equation}
The calculation 
\begin{align*}
\act{g(p_0,p_0)}g(i'+(p_0-1),p') &=g(p_0+p_0(i'+(p_0-1)-p_0),p')\\
&= g(p_0\cdot i',p') = \act{g(0,p_0)}g(i',p'),
\end{align*}
holds in the group $G(\PP)$ for each $i' > 1$ and each  $p' \in \PP$.
The relation just found continues to hold if $i' = 1$,
as one sees from the computation
\[
\act{g(p_0,p_0)}g(1+(p_0-1),p')
=
\act{g(p_0,p_0)}g(p_0,p')
=
g(p_0,p')
= 
\act{g(0,p_0)}g(1,p').
\]
Conjugation by $g$ satisfies therefore the equation
\begin{equation}
\label{eq:14.7}
\act{g}g(i,p) = g(i+(p_0-1),p)
\end{equation}
for every $i \geq 1$ and each $p \in \PP$.
With the help of this extra generator $g$
 the presentation 
\begin{equation}
\label{eq:Original-presentation-of-G(PP)}
\pi_* \colon \langle \XX \mid \RR = \RR_0 \cup \RR_1 \cup \RR_2 \rangle \iso G(\PP),
\end{equation}
will be transformed into a presentation with finitely many generators
and infinitely many, explicitly describable relators.
This will be achieved by a sequence of transformations
traditionally called \emph{Tietze transformations}.
\index{Tietze transformations|(}%
\footnote{see, \eg{}\cite[Section 1]{MKS04} for details on this concept}

Here is a conceptual version of this kind of transformations.
Let $\YY$ and $\ZZ$ be (non-empty) disjoint sets,
and let $F$ denote the free group on $\YY$ 
and $\tilde{F}$ the free group on $\YY \cup \ZZ$.
The assignements $y \mapsto y$
give rise to an embedding $\iota \colon F \mono \tilde{F}$.
Choose, for each $z \in \ZZ$, a word $w_z \in F$
and define $\WW$ to be the set of words 
$\{z \cdot w_z^{-1}\mid z \in \ZZ\}$.
Let $\rho \colon \tilde{F} \to F$ be the homomorphism 
that sends $ y \in \YY$ to $y$ and $z \in \ZZ$ to $w_z$.
The composition $\rho \circ \iota  \colon F \to F$ is the identity, 
while  $\iota \circ \rho$ fixes each $y \in \YY$ 
and sends each $z$ to $w_z$. 
Consider now the group given by the presentation 
\[
\Phi = \langle \YY \cup \ZZ \mid \WW \rangle
\]
and let $\psi \colon \tilde{F} \to \Phi$ be the obvious projection.
Since $\rho$ maps each relator $z \cdot w_z^{-1}$ to the neutral element of $F$,
it induces an epimorphism $\rho_* \colon \Phi \epi F$ 
that is the inverse of
\[
\iota_* = \psi \circ \iota \colon F \mono \tilde{F} \epi \Phi.
\]

Suppose now 
that $\RR$ is a set of relators in $F$
and let $R$ be the normal subgroup of $F$ generated by $\RR$.
Then $\iota_*$ induces an isomorphism of $F/R$ onto $\Phi/ \iota_*(R)$.
It follows that the presentations
\[
\label{eq:Tietze-transformation}
\langle \YY \mid \RR \rangle 
\quad \text{and} \quad 
\langle \YY \cup \ZZ \mid \RR \cup  \WW \rangle
\]
define isomorphic groups.
In traditional parlance this fact is rendered by saying:
\emph{by introducing a new set of generators $\ZZ$ 
and adding,
for each $z \in \ZZ$, 
a relator of the form $z \cdot w_z^{-1}$ with $w_z$ a word in the old generators,
one obtains a presentation $ \langle \YY \cup \ZZ \mid \RR \cup  \WW \rangle$
that defines the same group 
as does the presentation $\langle \YY \mid \RR \rangle$}.

What we have described so far is a first type of Tietze transformations.
There is a second type 
that is easy to understand from the conceptual point of view:
in a presentation, say $\langle \YY \mid \RR \rangle$,
the set of relators $\RR$ generates a normal subgroup $R \triangleleft F(\YY)$.
If $\RR_1$ is another set of relators with $R$ as its normal closure,
the identity on $\YY$ will induce an isomorphism of the group
with the presentation $\langle \YY \mid \RR \rangle$ 
onto the group with the presentation  $\langle \YY \mid \RR_1 \rangle$.
%
\subsubsection{Proof of Theorem \ref{TheoremD6}: third stage}
\label{sssec:14.3cNew}
This stage has three parts.
One begins by adding a new generator $f$ to the generating set $\XX$,
defined by equation \eqref{eq:Generating-set-XX},
and adding the relator 
\begin{equation}
\label{eq:Extra-relator-w}
w = f \cdot f(0,p_0)^{-1} f(p_0,p_0)
\end{equation}
to the set of relators
$\RR = \RR_0 \cup \RR_1 \cup \RR_2$,
specified in the second paragraph of section \ref{sssec:14.3aNew}.
In the resulting group the analogue of the relations \eqref{eq:14.7} will hold,
the relations
\begin{equation}
\label{eq:14.7New}
\act{f}f(i,p) = f(i+(p_0-1),p),
\end{equation}
valid for every $i \geq 1$ and each $p \in \PP$.
In the second part,
these relations will be used 
to reexpress the generators $f(i, p)$ 
in terms of a finite subset among them and of $f$.
The new expressions for the generators $f(i,p)$
permit one then to replace $\RR$ 
by a set of relators in which occur only finitely many of the original generators.
In the third part,
one discards all but finitely many of the generators $f(i,p)$
and arrives at a presentation of $G(\PP)$ with finitely many generators 
and an explicitly known infinite set of relators.

Let $F$ denote the free group on $\XX$ and $\pi \colon F \to G(\PP)$ 
the homomorphism 
that sends $f(i,p)$ to $g(i,p)$ for every $(i,p) \in \N \times \PP$.
By the verifications in section  \ref{sssec:14.3aNew},
the map $\pi$ is surjective 
and its kernel is the normal closure $R = \gp_F(\RR)$ 
of an explicitly known subset $\RR$ of $F$.
Next, 
let $\tilde{F}$ be the free group on $\XX \cup\, \{f\}$
and extend $\pi$ to the homomorphism $\tilde{\pi} \colon \tilde{F} \to G(\PP)$
which maps $f$ to $g = g(p_0,p_0)^{-1} g(0,p_0)$.
Then $\tilde{\pi}$ is surjective 
and its kernel  is the normal closure of the subset
\begin{equation}
\label{eq:Relators-RR-tilde}
\tilde{\RR} = \RR \cup \,\{w\} 
\quad \text{with} \quad
w = f \cdot f(0,p_0)^{-1} f(p_0,p_0).
\end{equation}
Let $\equiv$ denote the congruence relation on $\tilde{F}$
generated by $\tilde{\RR}$.
Equations  \eqref{eq:14.6}  and \eqref{eq:14.7},
in conjunction with the fact that $\tilde{\pi}$ is an isomorphism,
imply then that the congruence
\begin{equation}
\label{eq:14.7bis}
\act{f}f(i,p)\equiv f(i+(p_0-1),p)
\end{equation}
is valid for every $i \geq 1$ and $p \in \PP$.


We next express the generators in $\XX  \cup \, \{f\}$
in terms of finitely many among them.
Set
\begin{equation}
\label{eq:Enlarged-finite-set}
\tilde{\XX}(\PP, p_0, f) 
= 
\{f(i, p) \mid (i, p) \in \{0, 1, \ldots, p_0-1\} \times \PP \} \cup \, \{f\}.
\end{equation}
Given an integer $i\geq p_0$ 
write it in the form $i = r + s \cdot (p_0 - 1)$ 
with $s$ a positive integer and  $ r \in \{1, 2, \ldots, p_0-1\}$.
By the congruences \eqref{eq:14.7bis} 
the congruence
\begin{equation}
\label{eq:Congruences-for-eliminating}
f(i, p) \equiv \act{f^s} f(r, p)
\end{equation}
will the hold for every $i \geq p_0$.

With the help of these congruences we rewrite now
the relators in the set $\tilde{\RR}$ 
(defined by equation \eqref{eq:Relators-RR-tilde}).
This set is the union of the four subsets
\[
\RR_0, \quad \RR_1, \quad \RR_2 
\quad\text{and} \quad
\{f \cdot f(0,p_0)^{-1} f(p_0, p_0) \}.
\]
The subset $\RR_0$ is given by equation \eqref{eq:14.5};
its elements have the form
\[
w(i,j) = f(i,p_1)^{x(1,j)}f(i,p_2)^{x(2,j)}\cdots f(i,p_m)^{x(m,j)}.
\]
The generators involved in each one of these words 
have a fixed first parameter $i$
and so $w(i+ (p_0 - 1),j)$ is congruent to $\act{f} w(i, j)$ 
for every $i \geq 1$ and $j \in \{1, \ldots, k\}$.
In the presence of the congruences 
\eqref{eq:Congruences-for-eliminating},
the set $\RR_0$ can therefore be replaced by the set
\[
\left\{
\act{f^s} w(r, j) \mid (s, r, j) \in \N \times \{1, \ldots, p_0-1\} \times \{1, \ldots, k\} \right\}
\cup \left\{w(0,j) \mid j \in \{1, \ldots, k\} \right\}.
\]
Since we are looking for a set of relators in terms of the generating set 
$\XX \cup \, \{f\}$, 
we need only retain one relator in every conjugacy class of $\gp(f)$,
and so it suffices to retain the finite set
\begin{equation}
\label{eq:tilde-RR0}
\tilde{\RR}_0
= \{w(r, j \mid (r, j) \in \{0, 1, \ldots, p_0-1\} \times \{1, 2 \ldots, k\} \}.
\end{equation}
The subset $\RR_1$  consists of all the commutators  
$[f(i,p),f(i,p')]$ with $i \in \N$ and $(p,p') \in \PP^2$.
Each of these relators is a consequence of the relators in the finite set
\begin{equation}
\label{eq:tilde-RR1}
\tilde{\RR}_1
= \{[f(r, p), f(r, p')]\mid (r, p, p') \in \{0, 1, \ldots, p_0-1\} \times \PP^2 \}.
\end{equation}

Now to the subset of relators $\RR_2$.
Its elements have the form
\[ 
\act{f(i,p)}f(i',p') \cdot f(i'',p')^{-1} \text{ with }  (i,i') \in \N^2, i < i' \text { and } i'' = i + p(i'-i).
\]
Since $i'' >i' >  i$ and as $i''$ is given by the formula $i'' = i + p(i'-i)$, 
the relators in $\RR_2$ are congruent to conjugates of relators in $\RR_2$
with $i \leq p_0-1$.
Fix $r < p_0$,  set $ i = r$ and write $i'$ in the form $i' = (p_0-1)s' + r'$  
with $r' \in \{1,2, \ldots, p_0-1\}$
(this is possible since $i' > i \geq 0$).
Similarly, 
write $i'' = (p_0-1)s'' + r''$ with $r'' \in \{1,2, \ldots, p_0-1\}$.
In view of the equations 
\[
i'' = r+ p(i'-r) 
\quad \text{and} \quad
i' = (p_0-1)s' + r',
\]
the calculation
\begin{align*}
(p_0-1)s'' + r'' &= i'' = r + p(i'-r) = r + p((p_0-1) s' + r' - r) \\
&= r + p(r'-r) + (p_0-1) \cdot p  s'
\end{align*}
is valid.
It implies that 
\begin{align}
r'' &\equiv r + p(r'-r)  = r' + (p-1)(r'-r) \pmod{p_0-1},
\label{eq:14.8} \\
s'' &= ps' + \frac{r + p(r'-r) - r''}{p_0-1}\,.
\label{eq:14.8plus}
\end{align}
We are left with the relator $w = f \cdot f(0, p_0)^{-1} f(p_0, p_0)$.
The generator $f(p_0, p_0)$ occurring in it
is congruent to $\act{f} f(1, p_0)$
(see the congruence \eqref{eq:Congruences-for-eliminating}).
So $w$ can be replaced by
$f \cdot (f(0, p_0)^{-1}\cdot  f \cdot f(1,p_0)) \cdot f^{-1}$,
hence by $f(0, p_0)^{-1}\cdot  f \cdot f(1,p_0)$, or by
\begin{equation}
\label{eq:Replacing-additional-relator}
 f \cdot f(1,p_0) \cdot f(0, p_0)^{-1}. 
\end{equation}

So far we have rewritten the set of relators in
$\RR \cup\, \{ f \cdot f(0,p_0)^{-1} f(p_0,p_0) \}$
in terms of the finite set of generators $\tilde{\XX}(\PP, p_0, f)$ 
(given by  equation \eqref{eq:Enlarged-finite-set}),
and we have obtained a set of relators
$\tilde{\RR} \cup \{f \cdot f(1,p_0) \cdot f(0, p_0)^{-1} \}$.
In order to obtain this new set of relators 
the congruences 
\eqref{eq:Congruences-for-eliminating}
have been used;
they are consequences of the set of relators 
\begin{equation}
\label{eq:Elimination-redundant-generators}
 f(r + s(p_0-1), p) \cdot \left(\act{f^s} f(r, p)\right)^{-1} 
 \quad \text{for } s \geq 1 \text{ and } 
 r \in \{1, \cdots,  r_0-1\}.
\end{equation}
The generators $f(r + s(p_0-1), p)$ with $s > 0$
are no longer needed and 
they can be eliminated by means of a Tietze transformation
\footnote{see page \pageref{eq:Tietze-transformation}}
based on the set of relators  
\eqref{eq:Elimination-redundant-generators}.
If this is done,
one arrives at the presentation described in  
\begin{proposition} 
\label{prp:Summary-step-2}
\index{Group G(PP)@Group $G(\PP)$!finite generating sets}%
\index{Group G(PP)@Group $G(\PP)$!infinite presentations}%
The group $G(\PP)$ admits a presentation with the finite set of generators
\begin{equation}
\label{eq:XX-red}
\XX(\PP,p_0, f) =\{f(r,p) \mid 0 \leq r \leq p_0-1 \text{ and } p\in \PP\} \cup  \{f\}
\end{equation}
and the defining set of relators
\begin{equation}
\label{eq:Reduced-relators}
\RR(\XX, p_0, f) = \{ f \cdot f(1,p_0) \cdot f(0, p_0)^{-1} \}
\cup \tilde{\RR}_0 \cup \tilde{\RR}_1 \cup \tilde{\RR}_2.
\end{equation}
Here $\bar{\RR}_0$ and $\bar{\RR}_1$ denote the finite sets 
defined by formulae \eqref{eq:tilde-RR0} and \eqref{eq:tilde-RR1},
respectively,
while $\bar{\RR}_2$ is the infinite set consisting of all the words
\begin{equation}
\label{eq:Rewritten-conjugation-relations}
\act{f(r,p) f^{s'}}f(r',p') \cdot \left( \act{f^{s''}} f(r'', p') \right)^{-1},
\end{equation}
the parameters satisfying the restrictions
\[
0 \leq r < p_0, \quad  1 \leq r' < p_0, \quad  1 \leq r'' < p_0, \quad s' \in \N,
\text{ and } r' > r \text{ or } s' > 0,
\]
and equations \eqref{eq:14.8} and \eqref{eq:14.8plus}.
\end{proposition}

\begin{illustration}
\label{illustration:Lemma-summary-step-2}
If the slope group $P$ is cyclic,
the presentation afforded by Proposition \ref{prp:Summary-step-2}
can be described more concisely. 
Indeed, suppose $\PP$ is generated by  $p = p_0$.  
The sets $\bar{\RR}_0$ and $\bar{\RR}_1$ are then not needed,
formula \eqref{eq:14.8} simplifies to $r'' = r'$ 
and so formula \eqref{eq:14.8plus} becomes 
\begin{equation}
\label{eq:14.8plusNew}
s'' = ps' +\frac{ r + p(r'-r) - r'}{p-1} = ps' + r'-r.
\end{equation}

We show next
that every relator in $\bar{\RR}_2$ with $r = 0$ 
is a consequence of the remaining relators in $\bar{\RR}_2$.
If $r = 0$,
such a relator has the form
\[
\act{f(0,p)f^{s'}}f(r',p) \cdot \act{f^{s''}}f(r', p)^{-1}
\]
(recall that $r'' = r' $).
If one expresses $s''$ with the help of formula  \eqref{eq:14.8plusNew}
and uses that $r = 0$,
the relator becomes
\[
\act{f(0,p)g^{s'}}f(r',p)  \cdot \act{f^{ps' + r'}}f(r', p)^{-1}.
\]
If, finally, the generator $f(0,p)$ is replaced by $f \cdot f(1,p)$,
one arrives at the relator
\begin{equation}
\label{eq:Redandant-relator}
\act{f \cdot f(1,p) \cdot f^{s'}}f(r',p)  \cdot \act{f^{ps' + r'}}f(r', p)^{-1}.
\end{equation}
We claim that this relator is a consequence of the relators 
\begin{equation}
\label{eq:Relators-at-disposal}
\act{f(1,p)g^{s_1'}}f(r_1',p) \cdot \act{f^{ps_1' + r_1'-1}}f(r_1', p)^{-1}
\end{equation}
that are contained in the set $\bar{\RR}_2$. 
Two cases arise.
If $r' = 1$ and $s' = 0$ the word \eqref{eq:Redandant-relator}
becomes $\act{ f \cdot f(1,p)}f(1,p) \cdot \act{f} f(1,p) ^{-1}$ 
and so it is freely equivalent to the empty word;
if $r' > 1$ or $s' > 0$ then $r' + (p-1)s' > 1$
and so the word \eqref{eq:Relators-at-disposal} 
with $(r_1', s_1') = (r', s')$
is conjugated to a relator contained in the set $\bar{\RR}_2$
and thus conjugation by $f$ transforms the relator \eqref{eq:Relators-at-disposal} 
into the word \eqref{eq:Redandant-relator}.
The generator $f(0, p)$ is now no longer needed
and can be eliminated by a Tietze transformation 
based on the relator $ f \cdot f(1,p_0) \cdot f(0, p_0)^{-1}$.
\index{Brin, M. G.}%
\index{Squier, C. C.}%
If this is done,
one arrives at the presentation detailed in the following
\begin{corollary}
\label{LemmaD7N}
\index{Group G({p})@Group $G(\{p\})$!finite generating set}%
\index{Group G({p})@Group $G(\{p\})$!infinite presentation}%
(\cf{}(\cite[Section 2]{BrSq85}) 
If $p$ is an integer $\geq 2$, 
the group $G(\{p\})$ is generated by 
\begin{equation}
\label{eq:14.10}
x \mapsto f(0,p)f(1,p)^{-1},\quad  x_1 \mapsto f(1,p),\hdots,x_{p-1} \mapsto f(p-1,p),
\end{equation}
and is defined in terms of the generating set $\{x, x_1, \ldots, x_{p-1} \}$
by the relations
\begin{equation}
\label{eq:14.11}  \act{x_ix^n}x_j = \act{x^{pn+j-i}}x_j,
\end{equation}
where $n\geq 0$ and $(i,j)$ ranges over $\{1,\hdots,p-1\}$ with the restriction 
that either $j > i$ or $n > 0$.
\end{corollary}
\end{illustration}
\index{Tietze transformations|)}%
%
\subsubsection{Proof of Theorem \ref{TheoremD6}: fourth stage}
\label{sssec:14.3dNew}
%
In this final stage,
we show that finitely many among the relators listed in 
Proposition \ref{prp:Summary-step-2} suffice to define the group $G(\PP)$.
The proof will consists in a lengthy calculation
that is based on some properties of formulae
\eqref{eq:14.8} and \eqref{eq:14.8plus}.

According to Proposition \ref{prp:Summary-step-2}
the group $G(\PP)$ admits a presentation with 
\begin{equation}
\label{eq:Generating-set-X(P,p0)}
\XX(\PP,p_0, f) =\{f(r,p) \mid 0 \leq r \leq p_0-1\text{ and }p\in \PP\} \cup  \{f\}
\end{equation}
as set of generators and $\RR(\PP,p_0, f)$ as set of defining relators.
The set $\RR(\PP,p_0, f)$ in the union of three finite subsets and 
the infinite subset $\bar{\RR}_2$.
The relators in the infinite subset are equivalent 
to relations of the form
\begin{equation}
\label{eq:Some-defining-relations-G(PP)}
\act{f(r,p) f^{s'}}f(r',p')  =  \act{f^{s''}} f(r'', p'), 
\end{equation}
the parameters satisfying the following restrictions:
\begin{enumerate}[a)]
\item the elements $p$ and $p'$ range over $\PP$,
\item the index $r$ varies over $\{0,1, \ldots, p_0-1\}$,
\item the integer $s'$ takes on all values in $\N$;
\item the triple  $(r, r', s')$ takes on all values in  
$\{0,\ldots, p_0-1\} \times \{1,\ldots, p_0-1\} \times \N$ 
with $r' > r$ or $s' > 0$;
\item the index $r''$ lies in $\{1, \ldots, p_0-1\}$ 
and satisfies the congruence \eqref{eq:14.8},
namely 
\begin{equation}
\label{eq:14.8bis}
r'' \equiv r + p(r'-r)  \pmod{p_0-1};
\end{equation} 
\item the integer $s''$ lies in $\N$ and satisfies equation \eqref{eq:14.8plus},
namely
\begin{equation}
\label{eq:14.8plusbis}
s'' = p \cdot s' + \frac{r + p(r'-r) - r''}{p_0-1}\,.
\end{equation}
\end{enumerate}

We continue with three observations.
\begin{enumerate} [i)]
\item The right hand side of relation \eqref{eq:Some-defining-relations-G(PP)} 
involves the parameter $r''$;
according to formula \eqref{eq:14.8bis} 
this parameter depends on $r$, $r'$ and $p$, 
but it does not depend on $s'$.
\item By formulae \eqref{eq:14.8bis} and \eqref{eq:14.8plusbis},
 the parameter $s''$ depends on $r$, $r'$, $p$, $p_0$ and $s'$,
 but its dependence is of the form
 \begin{equation}
 \label{eq:14.8plus-aux}
 s'' = p \cdot s' + d(r, r', p, p_0).
 \end{equation}
 \item If $r = r'$ (and hence $s' > 0$) 
 then $r'' = r$ by  \eqref{eq:14.8bis}
 and so $s'' = p' \cdot s'$;
relation   \eqref{eq:Some-defining-relations-G(PP)} simplifies 
therefore to
 \begin{equation}
 \label{eq:Defining-relation-with-r=r'}
 \act{f(r',p') f^{s'}}f(r',p')  =  \act{f^{p' \cdot s'}} f(r', p').
 \end{equation}
 \end{enumerate}
 
 At present, 
 we are set for completing the proof of Theorem \ref{TheoremD6}.
 Let $p_\star$ be the largest element of $\PP$
 and let $F$ be the free group on the finite set $\XX(\PP, p_0, f)$
 defined by formula \eqref{eq:Generating-set-X(P,p0)}.
 Let $\equiv$ denote the congruence relation on $F$ 
 generated by the relations $f \equiv f(p,p_0) \cdot f(1, p_0)^{-1}$,
the relations $r \equiv 1$ with $r \in \bar{\RR}_0 \cup \bar{\RR}_1$ 
and all the relations of the form \eqref{eq:Some-defining-relations-G(PP)}
with 
\[
(p,p') \in \PP^2, \quad (r,r') \in \{0,1,\ldots, p_0-1\} \times \{1,\ldots, p_0-1\}\text{ and } s' \leq p_\star.
\]
Assume now that $n > p_*$ 
and that the congruences \eqref{eq:Some-defining-relations-G(PP)}
namely
\[
\act{f(r,p) f^{s'}}f(r',p')  \equiv  \act{f^{s''}} f(r'', p'), 
\]
hold for all $(p,p') \in \PP^2$, 
all $(r,r') \in \{0,1,\ldots, p_0-1\} \times \{1,\ldots, p_0-1\}$ 
and every $s' < n$.
Fix a couple $(p,p') \in \PP^2$ and find integers $n_1$, $n_2$ so that
\begin{equation}
\label{eq:Definition-n1-n2}
n = n_1 +p' \cdot n_2 \text{ with } n_1 \in \{1,\ldots, p'\}.
\end{equation}
Then  $n_1$, $n_1 + n_2$ and $p'\cdot n_2$ are all three smaller than $n$. 

Our aim is to prove the validity of the congruence
\begin{equation}
\label{eq:Relation-with-parameter-n}
\act{f(r,p) f^{n}}g(r',p')  \equiv  \act{f^{n''}} f(r'', p') 
\end{equation}
where $r''$ is given by formula   \eqref{eq:14.8bis} 
and $n''$ is given by the analogues of formulae 
\eqref{eq:14.8plusbis} and \eqref{eq:14.8plus-aux}
with $s'$ replaced by $n$. 
Thus
\begin{equation}
\label{eq:14.8plustris}
n'' = p \cdot n + \frac{r + p(r'-r) - r''}{p_0-1}
= p \cdot n + d(r, r', p, p_0).
\end{equation}
To reach this goal, 
we rewrite the left hand side $\act{f(r,p) f^{n}}f(r',p')$ 
of congruence \eqref{eq:Relation-with-parameter-n}
repeatedly in such a way
that the induction hypothesis can be applied to certain parts of the rewritten words.
In the course of these rewritings a situation will arise
where an auxiliary result for the case $r = r'$ is needed.

We move on to the first part of the calculation.
To ease notation, we write $x \downarrow y$ for $xyx^{-1}$ 
\label{notation:x-downarrow-y}%
and set $f_1 = f(r,p)$, $f_2 = f(r',p')$ and  $f_3 = f(r'',p')$.
The left hand side of congruence \eqref{eq:Relation-with-parameter-n} becomes then 
$f_1 f^n \downarrow f_2$. 
Since $n_1 < n$,
the induction hypothesis guarantees 
that $\act{f_1f^{n_1}} f_2 \equiv \act{f^{n_1''}}f_3$;
moreover, as $n_2 < n_1 + n_2 <   n$,
the induction hypothesis and observation iii) show that
\[
f_2 f^{n_2} \downarrow f_2 \equiv f^{p'\cdot n_2}\downarrow f_2 
\quad\text{and}\quad
f_3 f^{n_1 +n_2} \downarrow f_3 \equiv f^{p'(n_1 + n_2)} \downarrow f_3.
\] 
These relations and formulae \eqref{eq:Definition-n1-n2}
and \eqref{eq:14.8plustris}
are used in the next calculation.
\begin{equation}
\label{eq:Rewriting-exponents}
\begin{split}
\act{f(r,p) \cdot f^n}f(r',p') 
&= 
\left(f_1 f^{n_1} \cdot f^{p'n_2}\right) \downarrow f_2\\
&=
\left( (f_1f^{n_1} \downarrow f_2) \cdot f_1 f^{n_1} \cdot f_2^{-1}  f^{p'n_2}\right) 
\downarrow f_2\\
&\equiv
(f^{n_1''} \downarrow f_3) \cdot f_1 f^{n_1} 
\downarrow 
\left(f_2^{-1}  f^{p'n_2} \downarrow f_2 \right)\\
&\equiv
(f^{n_1''} \downarrow f_3) \cdot    f_1 f^{n_1} 
\downarrow 
(f^{n_2}\downarrow f_2)\\
&=
(f^{n_1''} \downarrow f_3) \downarrow  (f_1 f^{n_1 + n_2}) \downarrow f_2)\\
&\equiv
f^{n_1''}  f_3  f^{-n_1''} \downarrow  (f^{(n_1 + n_2)''})  \downarrow f_3\\
&=
\left (f^{n_1''}  f_3   f^{(n_1 + n_2)'' - n_1''}\right) \downarrow f_3.
\end{split}
\end{equation}
Observation ii) next allows one to rewrite the exponent $(n_1 + n_2)''-n_1''$
like this:
\begin{align*}
(n_1 + n_2)''-n_1''
&=
\left( p(n_1 + n_2) + d(r,r', p,p_0) \right) -  \left( p n_1 + d(r,r', p,p_0) \right)
= p n_2.
\end{align*}
So the result of calculation \eqref{eq:Rewriting-exponents} is the relation 
\[
\act{f(r,p) \cdot f^n}f(r',p') \equiv (f^{n_1''} f_3 f^{pn_2}) \downarrow f_3.
\]
Suppose we know that
$(f_3 f^{m}) \downarrow f_3   \equiv   f^{p'm}\downarrow f_3$
for every exponent $m > 0$.
We can then deduce that
\begin{align*}
(f^{n_1''} f_3 f^{pn_2}) \downarrow f_3
&=
f^{n_1''} \downarrow\left(f_3 f^{pn_2} \downarrow f_3  \right)
\equiv 
f^{n_1''} \downarrow \left(f^{p'\cdot pn_2} \downarrow f_3 \right)\\
&= 
f^{(pn_1  + d(r,r',p,p_0)) + p' pn_2} \downarrow f_3\\
&=
f^{p(n_1  + p' n_2) + d(r,r',p,p_0)} \downarrow f_3
 = \act{f^{n''}}f(r'',p').
\end{align*}

We are left with establishing the relations
$(f_3 f^{m}) \downarrow f_3   \equiv   f^{p'm}\downarrow f_3$ for  $m > 0$.
This can be done 
by repeating the calculation  \eqref{eq:Rewriting-exponents}
with $f_2 =f(r',p')$ in the rôle of $f_1 =f(r,p)$.
Then $p = p'$ and $r = r'$, hence $r' = r''$ and so $f_1 = f_2 = f_3$.
If $m \leq p_* = \max \{p \mid p \in \PP\}$ 
the claimed relation is one of the defining relations 
\eqref{eq:Some-defining-relations-G(PP)};
if $m > p_*$ 
we set $m = m_1 + p' m_2$ with  $m_1 \in \{1,\ldots, p'\}$ 
and compute:
\begin{equation}
\label{eq:Rewriting-exponents-for-p=pprime}
\begin{split}
\act{f_3 \cdot f^m}f_3 
&= 
\left(f_3 f^{m_1} \cdot f^{p'm_2}\right) \downarrow f_3\\
&=
\left( (f_3 f^{m_1} \downarrow f_3) \cdot f_3 f^{m_1} \cdot f_3^{-1}  f^{p'm_2}\right) 
\downarrow f_3\\
&\equiv
(f^{p'm_1} \downarrow f_3) \cdot f_3 g^{m_1} 
\downarrow 
\left( f_3^{-1}  f^{p'm_2} \downarrow f_3 \right)\\
&\equiv
f^{p' m_1}  f_3  f^{-p'm_1} \cdot    f_3 f^{m_1} 
\downarrow 
(f^{m_2}\downarrow f_3)\\
&\equiv
\left(f^{p'm_1}  f_3  f^{-p'm_1}  f^{p' (m_1 + m_2)}\right)  \downarrow f_3\\
&=
\left (f^{p'm_1}  f_3   f^{p'm_2}\right) \downarrow f_3\\
&\equiv
(f^{p'm_1}   f^{p' (p'm_2)}) \downarrow f_3 
= 
\act{f^{p'm}} f_3.
\end{split}
\end{equation}
The end of the fourth stage of the proof has now been reached;
with it, 
the proof of Theorem  \ref{TheoremD6} is complete.
%
\subsection{A more concise finite presentation of $G(\{p\})$}
\label{ssec:14.4New}
We close Section \ref{sec:14} with a word on the finite presentation 
of the group $G(\PP)$,
obtained before, in the special case where $\PP$ is a singleton.
By Corollary \ref{LemmaD7N} 
the group $G = G(\{p\})$ is then generated by the elements
\begin{equation}
\label{eq:14.10bis}
x \mapsto g(0,p)g(1,p)^{-1}, \;  x_1 \mapsto g(1,p),\hdots,  
x_{p-1}\mapsto g(p-1,p)
\end{equation}
and defined in terms of these generators by the infinite set of relations
\begin{equation} \label{eq:14.11bis}  
\act{x_ix^n}x_j = \act{x^{pn+j-i}}x_j,
\end{equation}
where $n\geq 0$ and $(i,j)$ ranges over $\{1,\hdots,p-1\}^2$ with the restriction 
that either $j > i$ or $n > 0$.
The results in section \ref{sssec:14.3dNew} show, in addition,  
that the relations with $n > p_\star = p$ are redundant.
The group $G$ has therefore a finite presentation with $p$ generators 
and $(p-1)(p-2)/2 + p(p-1)^2$ relations.
If $p = 2$, 
this is a  presentation with 2 generators and  2 relations;
if $m > 2$,
one can do better.
\begin{proposition}
\label{PropositionD8}
\index{Group G({p})@Group $G(\{p\})$!finite presentation}%
\index{Finiteness properties of!G({p})@$G(\{p\})$}%
If $p \geq 2$,
the group $G(\{p\}) = G([0,\infty[\,;\Z[1/p],\gp(p))$ 
is generated by the elements
\begin{equation*}
x \mapsto g(0,p)g(1,p)^{-1},
\quad x_1 \mapsto g(1,p),
\ldots,  \quad x_{p-1}\mapsto g(p-1,p)
\end{equation*}
and defined in terms of these generators by the relations
\begin{align}
\act{x_i}x_j &= \act{x^{j-i}}x_j \hspace*{10mm}\text{ with } 1 \leq i < j \leq p-1,\phantom{and }
\label{eq:14.13}\\
\act{x_ix}x_j &= \act{x^{p+j-i}}x_j  \hspace*{7mm}
\text{ with } (i,j) \in \{1, 2, \ldots, p-1\}^2 \text{ and } j \leq  i+1,
\label{eq:14.14}\\
\act{x_{p-1}x^2}x_1 &= \act{x^{p+2}}x_1.
\label{eq:14.15}
\end{align}
The number of these relations is 
$\tbinom{p-1}{2} + \left(\tbinom{p-1}{2} + (p-1) + (p-2)\right) + 1 =  p(p-1)$.
\end{proposition}

\begin{proof}
\label{Start-of-proof-of-Prp-D8}
The claim will be established by a refinement of the calculations carried out in section
\ref{sssec:14.3dNew}.
As before, $u \downarrow v$ denotes the triple product $uvu^{-1}$. 
We have to verify 
that all the relations \eqref{eq:14.11bis} are consequences of the relations 
listed in formulae  \eqref{eq:14.13}, \eqref{eq:14.14} and \eqref{eq:14.15}.

Let $\Scal_n$ be the assertion: 
\emph{all relations \eqref{eq:14.11bis} with parameter $n$
are valid in the group with generators  \eqref{eq:14.10bis} and relations
\eqref{eq:14.13}, \eqref{eq:14.14} and \eqref{eq:14.15}.} 

Assertion $\Scal_0$ holds since all relations with $n= 0$ are included 
into the set \eqref{eq:14.13}.
If $n=1$ the relations for $j \leq i+1$ are listed in \eqref{eq:14.14}. 
Suppose now that $j > i+1$.
The three relations 
\[
x_{j-1} \downarrow x_j = x \downarrow x_j, 
\quad
x_{i} \downarrow x_{j-1} = x^{(j-1)-i} \downarrow x_{j-1}
\text{ and  }
x_{i} \downarrow x_{j} = x^{j-i} \downarrow x_j
\]  
hold by assertion $\Scal_0$,
while relation $x_{j-1}x \downarrow x_j = x^{p+1} \downarrow x_j$ 
is covered 
by \eqref{eq:14.14}. 
The cited four relations justify the following calculation:  
 \begin{align*}
x_i x \downarrow x_j 
&=
((x_i \downarrow x_{j-1}) \cdot  x_i \cdot x_{j-1}^{-1} x) \downarrow x_j\\
&=
(x_i \downarrow x_{j-1}) \cdot x_i \downarrow (x_{j-1}^{-1} x \downarrow  x_j)\\
&=
(x_i \downarrow x_{j-1}) \downarrow (x_i \downarrow  x_j)
 \tag{\text{since  $x_{j-1} \downarrow x_j = x \downarrow x_j$}}\\
&=
(x^{(j-1)-i} \downarrow x_{j-1}) \downarrow (x^{j-i} \downarrow  x_j)\\
&=
(x^{(j-1)-i} \cdot x_{j-1} \cdot x^{-(j-1) + i} \cdot  x^{j-i}) \downarrow x_j\\
&=
x^{(j-1)-i} \downarrow (x_{j-1}  x \downarrow x_j)\\
& = 
(x^{(j-1)-i} \cdot x^{p + 1}) \downarrow x_j
=
x^{p + j-i} \downarrow x_j.
\end{align*}
This proves assertion $\Scal_1$.

The verification of assertion $\Scal_2$ divides into three parts.
If $j = 1$ and $i = p-1$,  the relation holds by assumption \eqref{eq:14.15}.
If $j = 1$ and $i < p-1$,  
assumption \eqref{eq:14.15} and assertion $\Scal_1$ 
justify the following calculation:
\begin{align*}
x_i x^2 \downarrow x_1
&=
((x_i  \downarrow x_{p-1}) \cdot  x_i  \cdot x_{p-1}^{-1} x^2) \downarrow x_1\\
&=
(x_i  \downarrow x_{p-1}) \downarrow ( x_i x \downarrow  x_1) 
\tag{\text{since $x_{p-1}x \downarrow x_1 = x^2 \downarrow x_1$}}
\\
&=
(x^{(p-1) - i}  \downarrow x_{p-1}) \downarrow  (x^{p + 1 - i} \downarrow x_1)\\
&=
(x^{(p - 1)- i} \cdot  x_{p-1} \cdot x^{-p + 1 + i} x^{p + 1 - i}) \downarrow x_1\\
& = 
(x^{p -1 -i} \cdot x_{p-1} x^2) \downarrow x_1 \tag{\text{by relation \eqref{eq:14.15}}}\\
&=
(x^{(p-1) -i} \cdot x^{p+2}) \downarrow x_1 
= 
x^{2p + 1- i} \downarrow x_1.
\end{align*}
If, finally, $n= 2$ and $j > 1$,
the index $j-1$ exists. 
The relations 
\[
x_{j-1}x \downarrow x_j = x^{p+1} \downarrow x_j
\quad \text{and}\quad
x_i x \downarrow x_{j-1} = x^{p +( j-1)- i}\downarrow x_{j-1}
\]
hold by assertion $\Scal_1$
and relation $x_{j-1} \downarrow x_j = x \downarrow x_j$ 
by statement  $\Scal_0$.
Therefore:
 \begin{align*}
x_i x^2 \downarrow x_j 
&=
((x_i x \downarrow x_{j-1}) \cdot  x_i  x \cdot x_{j-1}^{-1} x) \downarrow x_j\\
&=
(x_i x  \downarrow x_{j-1}) \downarrow (x_i  x\downarrow  x_j)
\tag{\text{since  $x_{j-1} \downarrow x_j = x \downarrow x_j$}} \\
&=
(x^{p + (j-1)- i} \downarrow x_{j-1})  \downarrow (x^{p + j- i}\downarrow  x_j)\\
&=
(x^{p + (j-1)- i}  \cdot x_{j-1} \cdot x^{-p -j + 1 + i} \cdot  x^{p + j - i} )\downarrow x_j
\\
&=
x^{p + (j-1)- i}  \downarrow (x_{j-1}  x \downarrow x_j)\\
& = 
(x^{p + (j-1)-i} \cdot x^{p + 1}) \downarrow x_j
=
x^{2p + j-i} \downarrow x_j.
\end{align*}
The verification of assertion $\Scal_2$ is now complete.

Suppose, next, that $n \geq 2$ and that assertion $\Scal_n$ has been established.
If $j > 1$ the index $j-1$ exists and the following calculation is valid:
 \begin{align*}
x_i x^{n+1} \downarrow x_j 
&=
((x_i x ^n \downarrow x_{j-1}) \cdot  x_i  x^n \cdot x_{j-1}^{-1} x) \downarrow x_j\\
&=
(x_i x ^n \downarrow x_{j-1}) \downarrow (x_i  x^n\downarrow  x_j)
\tag{\text{since  $x_{j-1} \downarrow x_j = x \downarrow x_j$}}\\
&=
(x^{pn + (j-1)- i} \downarrow x_{j-1})  \downarrow (x^{pn + j- i}\downarrow  x_j)\\
&=
(x^{pn + (j-1)- i}  \cdot x_{j-1} \cdot x^{-pn -j + 1 + i} \cdot  x^{pn + j - i} )\downarrow x_j
\\
&=
x^{pn + (j-1)- i}  \downarrow (x_{j-1}  x \downarrow x_j)\\
& = 
(x^{pn + (j-1)-i} \cdot x^{p + 1}) \downarrow x_j
=
x^{(n+1) p + j-i} \downarrow x_j.
\end{align*}
The  relation $x_i x^{n+1} \downarrow x_j = x^{p(n+1) + j - i} \downarrow x_j$
holds thus for $j > 1$.
Assume, finally,  that $j = 1$ and, in addition, 
that assertions $\Scal_{n-1}$ and $\Scal_n$ hold.
The calculation
\begin{align*}
x_i x^{n+1} \downarrow x_1
&=
((x_i x ^{n-1} \downarrow x_{p-1}) \cdot  x_i  x^{n-1} \cdot x_{p-1}^{-1} x^2) \downarrow x_1\\
&=
((x_i x ^{n-1} \downarrow x_{p-1}) \cdot x_i  x^{n}) \downarrow x_1
\tag{\text{since  $x_{p-1} x \downarrow x_1 = x^2 \downarrow x_1$}}\\
&=
(x^{p(n-1) + (p-1)- i} \downarrow x_{p-1})  \downarrow (x^{pn +1- i}\downarrow  x_1)\\
&=
(x^{pn -1 - i}  \cdot x_{p-1} \cdot x^{-pn + 1 + i} \cdot  x^{pn + 1 - i} )\downarrow x_1
\\
&=
x^{pn - 1- i}  \downarrow (x_{p-1}  x^2 \downarrow x_1)\\
& = 
(x^{pn -1 -i} \cdot x^{p + 2}) \downarrow x_1
=
x^{(n+1) p + 1-i} \downarrow x_1
\end{align*}
is then valid and establishes the induction step for $j = 1$.
 The proof of Proposition \ref{PropositionD8} is now complete.
 \label{End-of-proof-of-Prp-D8}
\end{proof}

\begin{examples}
\label{examples:Two-numerical-examples}
We work out the presentation given by Proposition  \ref{PropositionD8}
for $p=2$ and $p=3$.
If $p =2$,
the set \eqref{eq:14.13} is empty, 
while each of the sets \eqref{eq:14.14} and \eqref{eq:14.15} contributes a single relation.
The outcome is the 2-generator 2-relator presentation 
\begin{equation}
\label{eq:14.19}
\langle x,x_1; \act{x_1x}x_1 = \act{x^2}x_1,\act{x_1x^2}x_1 = \act{x^4}x_1\rangle;
\end{equation}
it resembles closely the presentation \eqref{eq:9.13New} for $G([0,1];\Z[1/2],\gp(2))$.
The PL-homeomorphisms associated to $x$ and $x_1$ are 
$g(0,2)g(1,2)^{-1}$ and  $ g(1,2)$;
they are indicated by the following two diagrams.
\input{chaptD.14.Illustration1}
\index{Group G({p})@Group $G(\{p\})$!finite presentation}%

If $p = 3$, Proposition \ref{PropositionD8} gives a presentation 
with three generators $x$, $x_1$, $x_2$ and six defining relations.
The relations are as follows:
\begin{align}
\act{x_1}x_2 &= \act{x}x_2,
\label{eq:14.13p}\\
\act{x_1x}x_1 = \act{x^3}x_1,\quad
\act{x_2x}x_1 = \act{x^2}x_1, 
\quad &\quad
\act{x_1x}x_2 = \act{x^4}x_2, \quad 
\act{x_2x}x_2 = \act{x^3}x_2,
\label{eq:14.14p}\\
\act{x_2x^2}x_1 &= \act{x^5}x_1.
\label{eq:14.15p}
\end{align} 
The PL-homeomorphisms corresponding to $x$, $x_1$ and $x_2$ are 
\[
g(0,3)g(1,3)^{-1}, \quad g(1,3) \quad \text{and}\quad  g(2,3).
\]
\index{Group G({p})@Group $G(\{p\})$!finite presentation}%
\end{examples}
%
\section{Presentations of groups with supports in a compact interval}
\label{sec:15}
%
Suppose $I$ is a compact interval and $G$ is the group $G(I;A,P)$.
If  $G$ is \emph{finitely generated} then $P$ is finitely generated, 
$A$ is a finitely generated $\Z[P]$-module, 
the quotient module $A/(IP \cdot A)$ is finite 
and both end points of $I$ are in $A$ (see Proposition \ref{PropositionB1}). 
No further necessary condition is known to hold 
if $G$ is \emph{finitely presented}.  
\index{Quotient group A/IPA@Quotient group $A/(IP \cdot A)$!significance}%

The group $G([a,c];A,P)$ is isomorphic to $G([a',c'];A,P)$ 
whenever $P$ is cyclic and $a$, $a'$, $c$ and $c'$ are all four in $A$
(see Theorem \ref{TheoremE07}).
In the case of a cyclic group $P$ with $A = \Z[P]$, 
it suffices to therefore to study the group $G([0,1];A,P)$.
In this section,
we shall do this for the groups
\begin{equation}
G[p] = G([0,1]; \Z[1/p], \gp(p)) 
\end{equation}
with $p \geq 2$ an integer 
and prove that they are finitely presented.
To date
\footnote{\ie{}in 1985; 
see sections \ref{sssec:Notes-finiteness-properties-Brown-Stein}
and \ref{sssec:Notes-Finiteness-properties-Cleary} for newer results.}
they are the only groups of the form $G([a,c]; \Z[P], P)$
that are known to be finitely presented 
%
\subsection{An infinite presentation of $G[p]$}
\label{ssec:15.1}
%
By Theorem \ref{TheoremB9} and the preamble to section \ref{ssec:9.9}
the group $ G = G[p]$ is generated by the infinite set
$\Fcal^\sharp = \{f(p^m;p; r, p) \mid m \in \N \text{ and } 1 < r \leq p\}$.
For $m > 0$, 
these generators are rescaled copies of the $p-1$ generators
\[
f(1;p; p,p), \quad f(1; p; p-1, p), \ldots, f(1; p; 2, p)
\]
with supports $]0,1[$\,, $]0, (p-1)/p[$\,, \ldots, $]0, 2/p[$, respectively.
\footnote{For $p \in \{2,3\}$ and $m \leq 1$, 
the rectangle diagrams of some of these functions are displayed in 
Illustration \ref{illustration:9.9.Example1}.}

It is useful to enumerate the generators $f(p^m;p;r,p)$
by decreasing length of their support.
To that end, we set $i = (p-r) + (p-1) m$ and $x_i = f(p^m; p;r,p)$.
By Corollary \ref{CorollaryB13}
the generators $x_i$ satisfy then the relations
\begin{equation}
\label{eq:15.2}
\index{Group G[p]@Group $G[p]$!infinite presentation}%
\index{Thompson's group F@Thompson's group $F$!infinite presentation}%
\index{Thompson, R. J.}%
\act{x_i}x_j =  x_{j+(p-1)}\; \text{ for }\; i < j.
\end{equation}

These relations actually define the group,
as we now prove by a normal form argument.
Let $F$ be the free group with basis $\{x_i \mid i \geq 0\}$
and let  $\equiv$ denote the congruence relation on $F$
that is induced by the relations \eqref{eq:15.2}.  
The relation  with parameters $i < j$ has the following four equivalent forms
\begin{gather}
\label{eq:First-two-equivalent-forms}
x_i \cdot x_j \equiv  x_{j+(p-1)} \cdot x_i, 
\quad 
x_j^{-1} \cdot x_i^{-1} \equiv x_i^{-1} \cdot x_{j+(p-1)}^{-1}, \\
\label{eq:Last-two-equivalent-forms}
x_i \cdot x_j^{-1} \equiv  x_{j+(p-1)}^{-1} \cdot x_i, 
\quad
x_j \cdot x_i ^{-1} \equiv  x_i^{-1} \cdot  x_{j+(p-1)}.
\end{gather}
Relations \eqref{eq:Last-two-equivalent-forms},  
together with free reduction,
allow one to move each generator $x_k$ 
to the right of each inverse $x_\ell^{-1}$ of a generator.
Every word $w \in F$ is therefore congruent 
to a word $v_1^{-1} \cdot v_2$ with $v_1$ and $v_2$ 
both positive or empty words in the generators $x_k$.
Next relations \eqref{eq:First-two-equivalent-forms} permit one
to reorder the letters in $v_2$, respectively those in $v_1^{-1}$, 
so us to obtain a word $w' = u_1^{-1} \cdot u_2$,
that is congruent to $v_1^{-1} \cdot v_2$ and satisfies:
\begin{equation*}
u_1 = x^{h(k)}_{i_k} \cdots x^{h(2)}_{i_2} x^{h(1)}_{i_1}
\quad \text{and} \quad
u_2 =x^{n(\ell)}_{j_\ell} \cdots x^{n(2)}_{j_2} x^{n(1)}_{j_1}
\end{equation*}
with $i_1 < i_2 < \cdots < i_k$ and $j_1 < j_2 < \cdots < j_\ell$, 
and where each $h(i)$ and each $n(j)$ is positive, 
unless $u_1$ or $u_2$ is empty
(\cf{}the injectivity proof in section \ref{sssec:14.3aNew}).

Assume now that $w$ represents an element in the kernel of the obvious projection
$ \pi \colon F \epi G[p]$.
Then so does $w'$.
By replacing, if need be,
$w' = u_1^{-1} \cdot u_2$, by a cyclic reduction and/or  its inverse,
one arrives at a word $\bar{w} = \bar{u}_1^{-1} \cdot \bar{u}_2$
that is, either empty, 
or if non-empty has the property 
that the index $\bar{j}_1$ of the last letter in $\bar{u}_2$ 
is smaller than the index of the last letter in $\bar{u}_1$. 
The second case cannot arise.
Indeed,
write $\bar{j}_1$ in the form $(p - r_1) + (p-1)m_1$  
with $r_1 \in \{0,1, \ldots, p-2\}$.
Then $f =  \pi(w) = \pi(\bar{w})$ is the identity on $[(p - r_1)/p^{m_1 + 1}, 1]$ 
and the right-hand derivative of $f$ at $(p-r_1)/p^{m_1 + 1}$ is $p^{\bar{n}_1} \neq 1$,
a finding which contradicts the assumption 
that $w$ is in the kernel of $\pi$.
%
\subsection{A finite presentation of $G[p]$}
\label{ssec:15.2}
\setcounter{thm}{9}
Relations \eqref{eq:15.2} make it clear 
that the group $G[p]$ is generated by the finite set 
\begin{equation}
\label{eq:Generating-set-XX[p]}
\XX[p] = \{x = x_0,x_1,\hdots,x_{p-1}\}.
\end{equation}
Conjugation by the distinguished generator $x$ permits one to eliminate the other generators 
by means of the relations 
\begin{equation*}
x_{r+m(p-1)} = \act{x^m}x_r,
\end{equation*}
where $r$ is in $\{1,2,\hdots,p-1\}$.  
A short calculation then shows
that the rewritten relations \eqref{eq:15.2} have the form
\begin{equation}
\label{eq:15.3}
\act{x_ix^n}x_j = x^{n+1}x_j;
\end{equation}
here $n\geq 0$ and $(i,j)$ ranges over $\{1,2,\hdots,p-1\}^2$ with the restriction 
that either $j > i$ or $n > 0$.  

So far we know that the group $G[p]$ is generated 
by the set $\XX[p]$ with $p$ elements
and that it has an easily described infinite presentation in terms of this set.  
But more is true: 
\begin{proposition}
\label{PropositionD10}
\index{Group G[p]@Group $G[p]$!finite presentation}%
\index{Finiteness properties of!G([a,c];A,P)@$G([a,c];A,P)$}%
\index{Finiteness properties of!G[p]@$G[p]$}%
If $p$ is an integer $\geq 2$  
the group $G[p]$ is generated by the set 
$\XX[p] = \{x = x_0, x_1, \ldots, x_{p-1}\}$, 
and defined in terms of $\XX[p]$ by the  relations
\begin{align}
\act{x_i} x_j  &= \act{x}x_j \quad \text{ where} \quad 1 \leq i < j \leq p-1, \label{eq:15.4}\\
\act{x_ix} x_j  &= \act{x^2}x_j \quad \text{where} \quad 
(i,j) \in \{1, \ldots, p-1\} \text{ and } j \leq i+1,
\label{eq:15.5}\\
\act{x_{p-1}x^2} x_1  &= \act{x^3}x_1. \label{eq:15.6}
\end{align}
The number of these relations is $p(p-1)$.
\end{proposition}

\begin{proof}
Our argument follows closely the procedure employed in the proof of Proposition \ref{PropositionD8},
given on pages  \pageref{Start-of-proof-of-Prp-D8}--\pageref{End-of-proof-of-Prp-D8}. 

Let $\Scal_n$ be the assertion: 
\emph{all relations \eqref{eq:15.3} with parameter $n$
are valid in the group with generators  \eqref{eq:Generating-set-XX[p]} 
and relations \eqref{eq:15.4}, \eqref{eq:15.5} and \eqref{eq:15.6}.} 

Assertion  $\Scal_0$ holds
since all relations with $n= 0$ are included in the set \eqref{eq:15.4}.
If $n=1$ the relations for $j \leq i+1$ are listed in \eqref{eq:15.5}. 
Suppose now that $j > i+1$ and write, to ease notation, 
$u \downarrow v$ for the conjugate $u v u^{-1}$ of $v$. 
By statement $\Scal_0$,
the three relations 
$x_{i} \downarrow x_{j-1} = x \downarrow x_{j-1}$
and
$x_{i} \downarrow x_{j}  = x_{j-1} \downarrow x_j = x \downarrow x_j$, 
hold,
while relation $x_{j-1}x \downarrow x_j = x^2 \downarrow x_j$ holds by assumption \eqref{eq:15.5}. 
The cited four relations justify then the following calculation:  
 \begin{align*}
x_i x \downarrow x_j 
&=
((x_i \downarrow x_{j-1}) \cdot  x_i \cdot x_{j-1}^{-1} x) \downarrow x_j\\
&=
((x_i \downarrow x_{j-1}) \cdot x_i )\downarrow (x_{j-1}^{-1} x \downarrow  x_j)\\
&=
(x_i \downarrow x_{j-1}) \downarrow (x_i \downarrow  x_j)
 \quad \text{(since  $x_{j-1} \downarrow x_j = x \downarrow x_j$)}\\
&=
(x \downarrow x_{j-1}) \downarrow (x \downarrow  x_j)
=
(x\cdot x_{j-1} \cdot x^{-1}\cdot  x)\downarrow x_j\\
&=
x \downarrow   (x_{j-1} \downarrow x_j)
=
x \downarrow (x \downarrow x_j)
=
x^2 \downarrow x_j.
\end{align*}
It follows that assertion $\Scal_1$ is valid.

The verification of assertion $\Scal_2$ divides into three parts.
If $j = 1$ and $i = p-1$,  the relation holds by hypothesis \eqref{eq:15.6}.
If, secondly, $j = 1$ and $i < p-1$,  
assumption \eqref{eq:15.6} 
and statement $\Scal_1$ justify the following calculation:
\begin{align*}
x_i x^2 \downarrow x_1
&=
((x_i  \downarrow x_{p-1}) \cdot  x_i  \cdot x_{p-1}^{-1} x^2) \downarrow x_1\\
&=
(x_i  \downarrow x_{p-1}) \downarrow ( x_i x \downarrow  x_1) 
\quad \text{(since  $x_{p-1} x \downarrow x_1 = x^2 \downarrow x_1$)}\\
&=
(x  \downarrow x_{p-1}) \downarrow  (x^2 \downarrow x_1)\\
&=
(x \cdot  x_{p-1} \cdot x^{-1} x^{2}) \downarrow x_1
= 
x \downarrow( x_{p-1} x  \downarrow x_1) \\\
&=
x \downarrow (x^2 \downarrow x_1) 
= 
x^3 \downarrow x_1.
\end{align*}
If, finally, $n= 2$ and $j > 1$,
the index $j-1$ exists. 
The relations 
\[
x_{j-1}x \downarrow x_j = x^2 \downarrow x_j
\quad \text{and}\quad
x_i x \downarrow x_{j-1} = x^2\downarrow x_{j-1}
\]
hold by statement $\Scal_1$
and relation $x_{j-1} \downarrow x_j = x \downarrow x_j$ 
by statement $\Scal_0$.
The following calculation is thus valid.
 \begin{align*}
x_i x^2 \downarrow x_j 
&=
((x_i x \downarrow x_{j-1}) \cdot  x_i  x \cdot x_{j-1}^{-1} x) \downarrow x_j\\
&=
(x_i x  \downarrow x_{j-1}) \downarrow (x_i  x\downarrow  x_j)
\quad \text{(since  $x_{j-1} \downarrow x_j = x \downarrow x_j$)} \\
&=
(x^2 \downarrow x_{j-1})  \downarrow (x^2\downarrow  x_j)\\
&=
(x^2  \cdot x_{j-1} \cdot x^{-2} \cdot  x^2 )\downarrow x_j
=
x^2  \downarrow (x_{j-1}  \downarrow x_j)\\
& = 
x^2 \downarrow (x \downarrow x_j)
=
x^3\downarrow x_j.
\end{align*}
The verification of assertion $\Scal_2$ is now complete.

Suppose, finally, that $n \geq 2$ 
and that assertions $\Scal_{n-1}$ and $\Scal_n$ have been established.
The following calculation is then valid:
 \begin{align*}
x_i x^{n+1} \downarrow x_j 
&=
((x_i x ^{n-1} \downarrow x_j) \cdot  x_i  x^{n-1} \cdot x_j^{-1} x^2) \downarrow x_j\\
&=
(x_i x ^{n-1} \downarrow x_{j}) \downarrow (x_i  x^{n-1} \cdot x\downarrow  x_j)
\quad \text{(since  $x_{j} x\downarrow x_j = x^2 \downarrow x_j$)}\\
&=
(x^n \downarrow x_{j})  \downarrow (x^{n+1}\downarrow  x_j)\\
&=
(x^n  \cdot x_{j} \cdot x^{-n } \cdot  x^{n+1} )\downarrow x_j
=
x^n  \downarrow (x_{j}  x \downarrow x_j)\\
& = 
x^n \downarrow (x^2 \downarrow x_j)
=
x^{(n+1) +1} \downarrow x_j.
\end{align*}
This completes the induction step 
and with it the proof of Proposition \ref{PropositionD8}.
 \end{proof}

\begin{examples}
\label{examples:Fp-groups-generalizing-F}
The group $G[2]$ is the celebrated group studied by R. J. Thompson in \cite{Tho74} 
(see also \cite{McTh73}), 
\index{Thompson, R. J.}%
and by many other mathematicians.
Proposition \ref{PropositionD10} yields the well-known 2-generator 2-relator presentation 
\eqref{eq:9.13New},
namely
\begin{equation}
\label{eq:15.8}
\langle x, x_1 \mid \act{x_1x}x_1 = \act{x^2}x_1, \act{x_1x^2}x_1 = \act{x^3}x_1 \rangle.
\end{equation}
The rectangle diagrams of the generators $x$, $x_1$ are 
\[
\begin{minipage}[c]{12cm}
\psfrag{1}{\hspace*{-1.7mm}  \small  $0$}
\psfrag{2}{  \hspace*{-3mm} \small  $\tfrac{1}{2}$}
\psfrag{3}{\hspace*{-2mm}  \small  $\tfrac{3}{4}$}
\psfrag{4}{\hspace*{-1.5mm} \small   $1$}
\psfrag{11}{\hspace*{-1.7mm}  \small   $0$}
\psfrag{12}{  \hspace*{-2.8mm} \small   $\tfrac{1}{4}$}
\psfrag{13}{\hspace*{-1.45mm}  \small   $\tfrac{1}{2}$}
\psfrag{14}{\hspace*{-1.15mm}  \small   $1$}
\psfrag{21}{\hspace*{-2.5mm}  \small   $\tfrac{1}{2}$}
\psfrag{22}{\hspace*{-0.3mm}  \small   $1$}
\psfrag{23}{\hspace*{-0mm}  \small   $2$}
\psfrag{31}{\hspace*{-1.3mm} \small $0$}
\psfrag{32}{\hspace*{-0.4mm}\small $\tfrac{1}{4}$}
\psfrag{33}{\hspace*{-1.5mm} \small $\tfrac{3}{8}$}
\psfrag{34}{\hspace*{-0.4mm}\small $\tfrac{1}{2}$}
\psfrag{35}{\hspace*{-0.3mm}\small  $1$}
\psfrag{41}{\hspace*{-1.2mm} \small  $0$}
\psfrag{42}{  \hspace*{-2.7mm} \small   $\tfrac{1}{8}$}
\psfrag{43}{\hspace*{-1.2mm} \small   $\tfrac{1}{4}$}
\psfrag{44}{\hspace*{-1.3mm} \small   $\tfrac{1}{2}$}
\psfrag{45}{\hspace*{-1mm} \small   $1$}
\psfrag{51}{\hspace*{-1.4mm} \small   $\tfrac{1}{2}$}
\psfrag{52}{\hspace*{-0.8mm} \small   1}
\psfrag{53}{\hspace*{-0.6mm} \small   $2$}
\psfrag{54}{\hspace*{-1mm} \small   $1$}

\psfrag{la1}{\hspace{-1.5mm}$x=$}
\psfrag{la2}{\hspace{-4.5mm} $x_1=$}
\begin{equation*}
\includegraphics[width= 5.5cm]{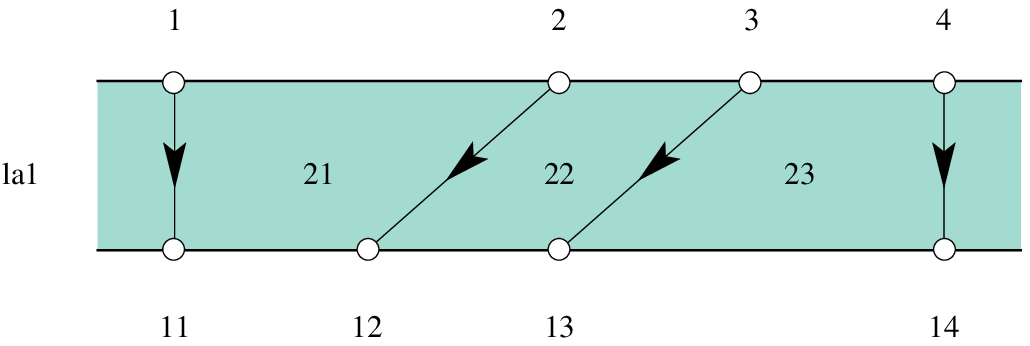} 
\hspace*{7mm}
\includegraphics[width= 5.5cm]{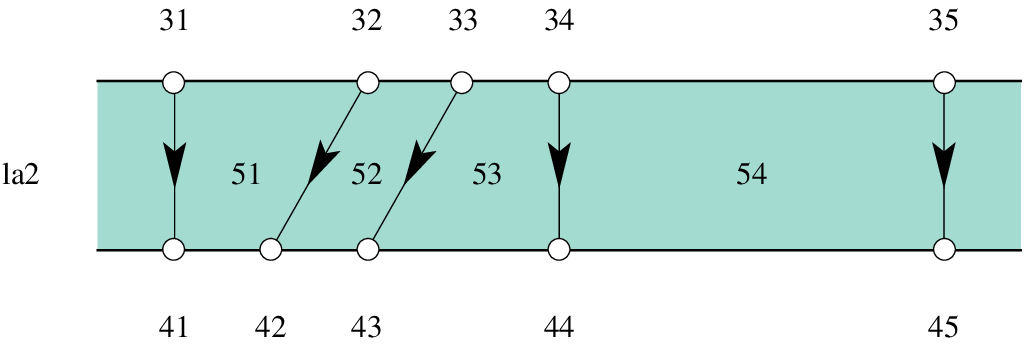} 
\end{equation*}
\end{minipage}
\]
\index{Group G[p]@Group $G[p]$!finite presentation}%
\index{Thompson's group F@Thompson's group $F$!finite presentation}%
\index{Thompson, R. J.}%

As a second example, consider the group $G[3]$. 
Proposition \ref{PropositionD10} furnishes  3 generators $x$, $x_1$ and $x_2$;
their rectangle diagrams are
\[
\begin{minipage}[c]{12cm}
\psfrag{1}{\hspace*{-1.7mm}  \small  $0$}
\psfrag{2}{  \hspace*{-3.2mm} \small  $\tfrac{1}{3}$}
\psfrag{3}{\hspace*{-1.8mm}  \small  $\tfrac{2}{3}$}
\psfrag{4}{\hspace*{-1.6mm}  \small  $\tfrac{7}{9}$}
\psfrag{5}{\hspace*{-1.8mm}  \small  $\tfrac{8}{9}$}
\psfrag{6}{\hspace*{-1.6mm} \small   $1$}
\psfrag{11}{\hspace*{-1.2mm}  \small   $0$}
\psfrag{12}{  \hspace*{-2.5mm} \small   $\tfrac{1}{9}$}
\psfrag{13}{  \hspace*{-2.7mm} \small   $\tfrac{2}{9}$}
\psfrag{14}{\hspace*{-1.4mm}  \small   $\tfrac{1}{3}$}
\psfrag{15}{\hspace*{-1.5mm}  \small   $\tfrac{2}{3}$}
\psfrag{16}{\hspace*{-1.4mm}  \small   $1$}
\psfrag{21}{\hspace*{-1.5mm}  \small   $\tfrac{1}{3}$}
\psfrag{22}{\hspace*{-2.6mm}  \small   $\tfrac{1}{3}$}
\psfrag{23}{\hspace*{-1.0mm}  \small  }
\psfrag{24}{\hspace*{-0.5mm}  \small   $3$}
\psfrag{25}{\hspace*{-0.2mm}  \small   $3$}
\psfrag{31}{\hspace*{-1.1mm} \small $0$}
\psfrag{32}{\hspace*{-0.2mm}\small $\tfrac{1}{3}$}
\psfrag{33}{\hspace*{-1.3mm} \small $\tfrac{4}{9}$}
\psfrag{34}{\hspace*{-1.5mm} \small $\tfrac{5}{9}$}
\psfrag{35}{\hspace*{-0.2mm}\small $\tfrac{2}{3}$}
\psfrag{36}{\hspace*{-0.0mm}\small  $1$}
\psfrag{41}{\hspace*{-1.2mm} \small  $0$}
\psfrag{42}{  \hspace*{-2.4mm}   \small   $\tfrac{1}{9}$}
\psfrag{43}{  \hspace*{-2.4mm}   \small   $\tfrac{2}{9}$}
\psfrag{44}{\hspace*{-1mm} \small   $\tfrac{1}{3}$}
\psfrag{45}{\hspace*{-1.5mm} \small   $\tfrac{2}{3}$}
\psfrag{46}{\hspace*{-1.2mm} \small   $1$}
\psfrag{51}{\hspace*{-2mm} \small   $\tfrac{1}{3}$}
\psfrag{52}{\hspace*{-1.2mm} \small}
\psfrag{53}{\hspace*{-1.2mm} \small}
\psfrag{54}{\hspace*{-1mm} \small   $3$}
\psfrag{55}{\hspace*{-1mm} \small   $1$}

\psfrag{la1}{\hspace{-2mm}$x=$}
\psfrag{la2}{\hspace{-5mm} $x_1=$}
\begin{equation*}
\includegraphics[width= 5.5cm]{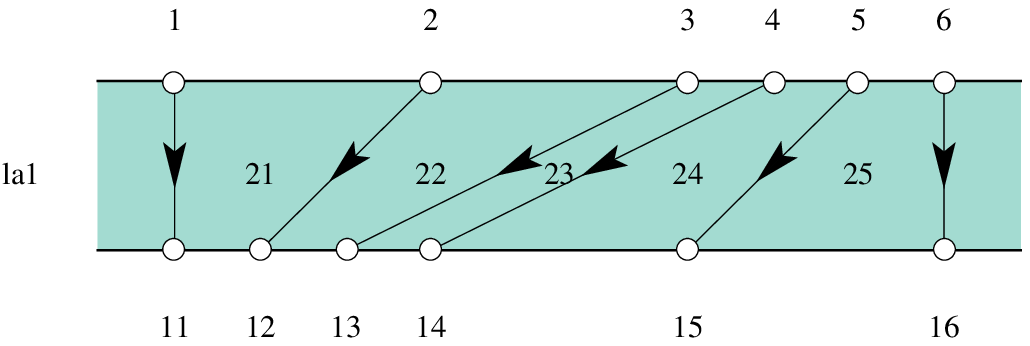} 
\hspace*{6mm}
\includegraphics[width= 5.5cm]{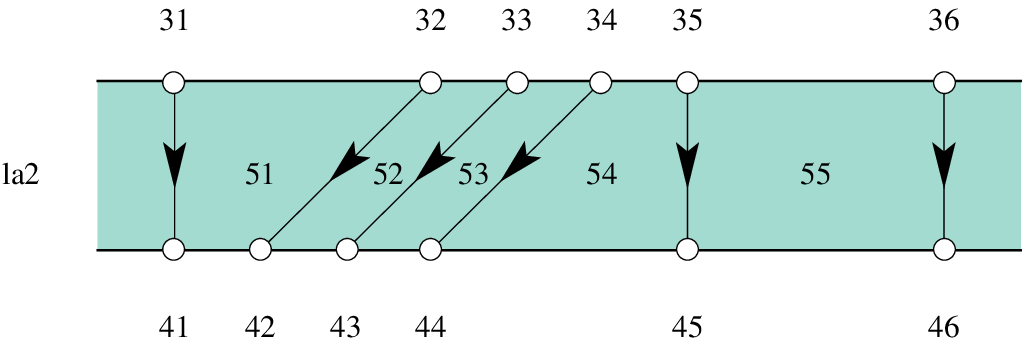} 
\end{equation*}
\end{minipage}
\]
and
\[
\begin{minipage}[c]{12cm}
\psfrag{1}{\hspace*{-1.7mm}  \small  $0$}
\psfrag{2}{  \hspace*{-2.8mm} \small  $\tfrac{1}{9}$}
\psfrag{3}{  \hspace*{-3.0mm} \small  $\tfrac{2}{9}$}
\psfrag{4}{\hspace*{-2.2mm}  \small  }
\psfrag{5}{\hspace*{-2.2mm}  \small  }
\psfrag{6}{\hspace*{-1.6mm} \small   $\tfrac{1}{3}$}
\psfrag{7}{\hspace*{-1.6mm} \small   $\tfrac{2}{3}$}
\psfrag{8}{\hspace*{-1.4mm} \small   $1$}

\psfrag{11}{\hspace*{-0.7mm}  \small   $0$}
\psfrag{12}{  \hspace*{-3.0mm} \small   }
\psfrag{13}{  \hspace*{-3.0mm} \small   }
\psfrag{14}{  \hspace*{-2.0mm} \small   $\tfrac{1}{9}$}
\psfrag{15}{  \hspace*{-2.2mm} \small   $\tfrac{2}{9}$}
\psfrag{16}{\hspace*{-0.8mm}  \small   $\tfrac{1}{3}$}
\psfrag{17}{\hspace*{-0.9mm}  \small   $\tfrac{2}{3}$}
\psfrag{18}{\hspace*{-0.5mm} \small   $1$}
\psfrag{21}{\hspace*{-1.0mm}  \small   $\tfrac{1}{3}$}
\psfrag{22}{\hspace*{-1.5mm}  \small  }
\psfrag{23}{\hspace*{-1mm}  \small   }
\psfrag{24}{\hspace*{-1.1mm}  \small   }
\psfrag{25}{\hspace*{-0.7mm}  \small   $3$}
\psfrag{26}{\hspace*{-1.3mm}  \small   $1$}
\psfrag{27}{\hspace*{-1.3mm}  \small   $1$}
\psfrag{la3}{\hspace{-5mm}$x_2=$}
\begin{equation*}
\includegraphics[width= 5.5cm]{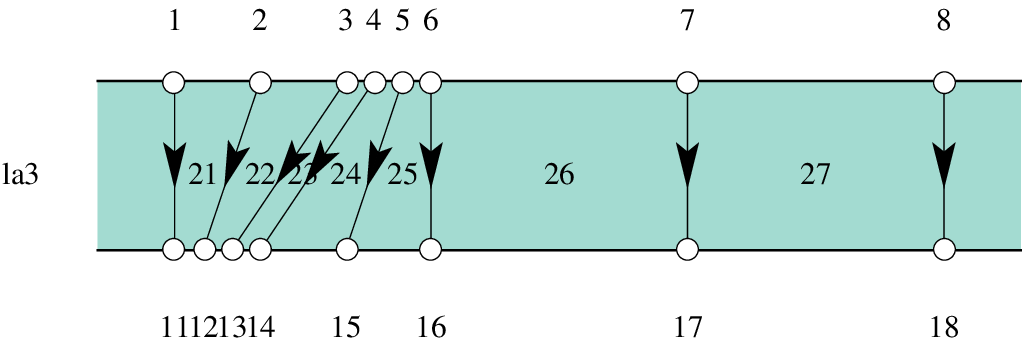}
\end{equation*}
\index{Group G[p]@Group $G[p]$!finite presentation}%
\smallskip
\end{minipage}
\]
%
Proposition \ref{PropositionD10} lists 6 relations, 
namely
\begin{gather*}
\act{x_1}x_2   = \act{x} x_2, \\
\act{x_1x}x_1 = \act{x_2x}x_1 = \act{x^2}x_1, 
\qquad 
\act{x_1x}x_2 = \act{x_2x}x_2 = \act{x^2}x_2,\\
\act{x_2x^2}x_1 = \act{x^3}x_1.
\end{gather*}
\end{examples}
%

%% file: chaptD.14.Illustration1.tex
\begin{figure}[htb]
\psfrag{1}{\hspace*{-1.54mm} \small   $0$}
\psfrag{2}{\hspace*{-1.7mm} \small  1}
\psfrag{3}{  \hspace*{-1.7mm}\small  $\infty$}

\psfrag{11}{\hspace*{-0.6mm} \small   $0$}
\psfrag{12}{\hspace*{-1.5mm} \small  $2$}
\psfrag{13}{\hspace*{-1.7mm} \small  $\infty$}
\psfrag{21}{\hspace*{-3.7mm} \small  2}
\psfrag{22}{\hspace*{-1.7mm} \small  1}

\psfrag{31}{\hspace*{-0.9mm} \small   $0$}
\psfrag{32}{\hspace*{-1.2mm} \small  1}
\psfrag{33}{  \hspace*{-1.7mm}\small  $\infty$}

\psfrag{41}{\hspace*{-0.9mm} \small   $0$}
\psfrag{42}{\hspace*{-0.9mm} \small  $1$}
\psfrag{43}{\hspace*{-1.3mm} \small  $\infty$}
\psfrag{51}{\hspace*{-1.7mm} \small  1}
\psfrag{52}{\hspace*{1.7mm} \small  2}

\psfrag{la1}{\hspace{1mm}$x$}
\psfrag{la2}{\hspace{1mm}$x_1$}
\psfrag{la3333}{\hspace{-2mm}$x_2$}
\begin{center}
\includegraphics[width= 11cm]{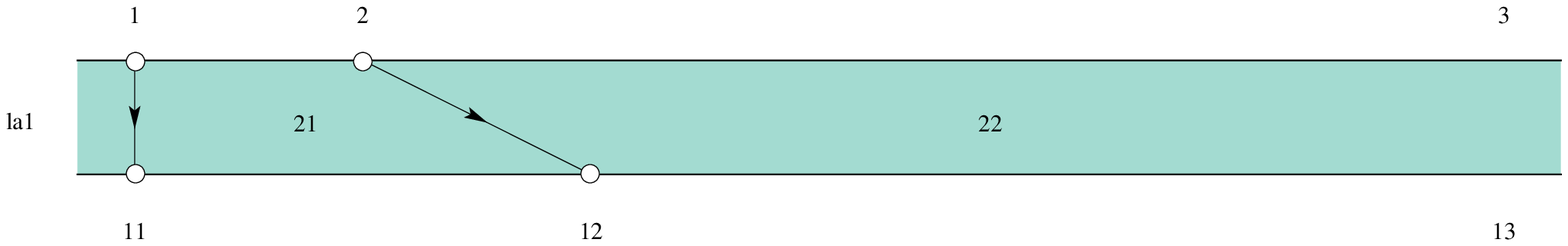}\par
\vspace*{0mm}
\includegraphics[width= 11cm]{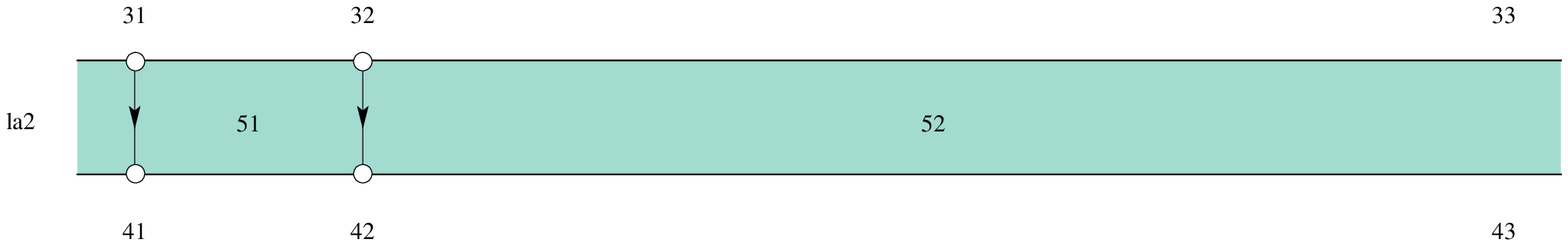}\\
\end{center}
\end{figure}
\index{Rectangle diagram!examples}

%% file: chaptE_Isos-and-autos.tex
%
%
\chapter{Isomorphisms and Automorphism Groups}
\label{chap:E}
\setcounter{section}{15}
%
%
In this chapter,
we investigate isomorphisms among groups $G =G(I;A,P)$ 
and $\bar{G} = G(\bar{I}; \bar{A}, \bar{P})$.
We first show 
that every such isomorphism $\alpha$ is induced by conjugation 
by a unique homeomorphism $\varphi \colon \Int(I) \iso \Int(\bar{I})$.
This conclusion holds for arbitrary pairs of triples 
$(I,A,P)$, $(\bar{I},\bar{A},\bar{P})$
and it remains valid for isomorphisms among certain subgroups of $G$ 
and of $\bar{G}$
(see Theorem \ref{TheoremE04} for a precise statement). 
In the remainder of the chapter 
the nature of the homeomorphism $\varphi$ is studied.
Best results are obtained if the group $P$ is not cyclic.
Then $\bar{P} = P$, 
the homeomorphism $\varphi$ is piecewise linear 
with slopes in a coset $s \cdot P$ of $P$ 
and $\bar{A} = s \cdot A$
(see Theorem \ref{TheoremE10}). 
A similar conclusion holds for isomorphisms 
$\alpha \colon G(I;A,P)  \iso G(\bar{I}; \bar{A}, \bar{P})$
of groups with cyclic $P$ and $I$ the line or a half line with endpoint in $A$
(see Theorem \ref{TheoremE14}).
%
%
\section{General results}
\label{sec:16}
%
%
 S. H. McCleary investigates group isomorphisms $\alpha \colon G \iso \bar G$ 
of ordered permutation groups  $(G,\Omega)$ and $(\bar G,\bar\Omega)$
in \cite{McC78b}.
\index{McCleary, S. H.}%
He proves that $\alpha$ is induced by a unique homeomorphism 
$\varphi \colon \Omega^{\cpl} \iso \bar \Omega^{\cpl}$ 
of the Dedekind completion $\Omega^{\cpl}$ of $\Omega$ 
onto the Dedekind completion $\bar\Omega^{\cpl}$ of $\bar{\Omega}$, 
provided the following conditions hold:
$G$ acts 3-fold transitively on $\Omega$ 
and contains strictly positive elements of bounded support, 
and the pair $(\bar G,\bar\Omega)$ has the analogous properties
(see \cite[p.\;505, Main Theorem]{McC78b}).  
Here $f$ is called \emph{positive}, if $f(\omega) \geq \omega$ for all $\omega \in \Omega$, 
and \emph{strictly positive} if it is positive and distinct from $\id$.  
In section \ref{ssec:16.2} we shall establish a variation of McCleary's result.  
Our proof will exploit ideas used in \cite{McC78b} in a set-up 
that is adapted to the case at hand; 
the new set-up enables one to simplify McCleary's  argument. 
\index{Ordered permutation groups!isomorphisms}%
\index{McCleary, S. H.}%
%
\subsection{Set-up}
\label{ssec:16.1}
In what follows we deal with isomorphisms of certain kinds of groups of order preserving homeomorphisms of an open interval $J$ of $\R$.  
We require that $G$ satisfies the following four axioms.
\begin{enumerate}[{Ax} 1:]
\item $G$ contains a strictly positive element.
\item $G$ contains a non-trivial element with bounded support.
\item $G$ contains non-trivial elements $g_\ell$ and $g_u$ with $\sup(\supp g_\ell) = \inf(\supp g_u)$. 
\item  $G$ is approximately 6-fold transitive in the sense that, given strictly increasing six-tuples $(t_1,\hdots,t_6)$ and $(t'_1,\hdots,t'_6)$ of elements in $J$, there
exists $h$ in $G$ with 
\end{enumerate}
\begin{equation}
\label{eq:16.1} 
h(t_1) < t'_1, \quad  t'_2 < h(t_2), \quad 
h(t_3) < t'_3, \quad t'_4 < h(t_4), \quad 
h(t_5) < t'_5, \quad t'_6 < h(t_6).
\end{equation}

Note that if the pair $(G,J)$ satisfies these axioms and if $\tilde G$ is a supergroup of $G$, made up of homeomorphisms of the interval $J$, then $(\tilde G,J)$ satisfies the axioms, too.  
Moreover, if $\varphi \colon J \iso \bar J$ is an order preserving or order reversing
homeomorphism and if $(G,J)$ fulfills the axioms, so does $(\act{\varphi}G,\bar{J})$. 

%
\subsection{The basic result}
\label{ssec:16.2}
In the sequel $G$ and $\bar G$ are assumed to be groups made up of order preserving homeomorphisms of the open intervals $J$ and $\bar J$, respectively.  

We begin with a technical result 
that will be the main tool for constructing the homeomorphism $\varphi$.
\begin{lemma}
\label{LemmaE1}  
Assume $G$ satisfies axiom Ax4 and $g_1$, $g_2$ are elements of $G \smallsetminus \{\id\}$.  
Suppose, in addition, that $G$ contains an element $f$  so that 
\begin{equation}
\label{eq:16.2}
[g_1, \act{hfh^{-1}}g_2] = \id \quad \text{for all }h \in G.
\end{equation}
If the inequality  $t < f(t)$ holds for some $t \in J$  then 
\begin{equation}
\label{eq:16.3}
\supp g_1 < \supp g_2 \text{ and  $f$  is strictly positive}.
\end{equation} 
\end{lemma}

\begin{proof}
We begin by proving a contraposition of the first assertion in \eqref{eq:16.3}.
Suppose $s_1$, $s_2$ and $t$ are elements of  $J$ such that 
\begin{equation*}
 s_1 \in \supp g_1,\quad  s_2 \in \supp g_2 \quad \text{and}\quad  s_2 < s_1.
\end{equation*}
and $f$ is an element of $G$ with $t < f(t)$.
Equation \eqref{eq:16.2} holds for $g_1$ if, and only if, 
it holds with $g_1$ replaced by $g^{-1}_1$, 
and the analogous statement is true for $g_2$.
We can therefore assume that $s_1 < g_1(s_1)$ and that $g_2(s_2) < s_2$.  
Choose reals $t'$, $t''$ with $t < t' < t'' < f(t)$ 
and use them to define the six-tuples
\begin{equation*}
(t_1,t_2,t_3,t_4,t_5,t_6) = (t,t',t'',f(t),f(t'),f(t'') )
\end{equation*}
and
\begin{equation*}
(t'_1,t'_2,t'_3,t'_4,t'_5,t'_6)= \left(g^2_2(s_2),g_2(s_2),s_2,s_1,g_1(s_1),g^2_1(s_1) \right).
\end{equation*}
These six-tuples are strictly increasing;  
by axiom Ax4 there exists therefore an element $h_0$ in $G$ 
satisfying the inequalities \eqref{eq:16.1}.  
These inequalities imply 
that
\begin{align*}
\left(\act{h_0fh^{-1}_0}g^{-1}_2 \circ g_1\right) (t'_4) 
&= 
(h_0 f h^{-1}_0 \circ g^{-1}_2\circ h_0 f^{-1} h^{-1}_0) (t'_5) \\
&>
( h_0 f h^{-1}_0 \circ g^{-1}_2 \circ h_0f^{-1})(t_5)
= 
(h_0 f h^{-1}_0 \circ g^{-1}_2 \circ h_0) (t_2)\\
&> 
(h_0 f h^{-1}_0 \circ g^{-1}_2) (t'_2) = (h_0fh^{-1}_0) (t'_3) \\
&> h_0 (f(t_3))=  h_0(t_6) > t'_6
\end{align*}
and that
\begin{align*}
\left(g_1 \circ \act{h_0fh^{-1}_0} g^{-1}_2\right)(t'_4) 
&< 
(g_1 \circ h_0fh^{-1}_0 \circ g^{-1}_2 \circ h_0f^{-1})(t_4) \\
&= 
(g_1 \circ h_0 f h^{-1}_0 \circ g^{-1}_2 \circ h_0)(t_1)
< 
(g_1 \circ h_0fh^{-1}_0 \circ g^{-1}_2)(t'_1)\\
&= 
(g_1 \circ h_0fh^{-1}_0)(t'_2) < (g_1 \circ h_0f )(t_2)\\ 
&= (g_1\circ h_0) (t_5) < g_1(t'_5) = t'_6.
\end{align*}
These calculations reveal
that $g_1$ and $\act{h_0fh_0^{-1}}g_2^{-1}$ do not commute.

We come now to the second assertion of Lemma \ref{LemmaE1}.
Suppose that $f \in G$ satisfies hypothesis \eqref{eq:16.2}
and $t \in J$ is an element with $f < f(t)$. 
Then $\supp g_1 < \supp g_2$.
Since $g_1 \neq \id $ and  $g_2\neq \id $ by hypothesis, 
their supports are non-empty 
and so the inequality $\supp g_2 < \supp g_1$ does not hold.
On the other hand,
if $f$ were not strictly positive 
there would exists an element $t' \in J$ with $f(t') < t'$.
Then $t' < f^{-1}(t')$ and so the first assertion of \eqref{eq:16.3} would allow us to conclude
that  $\supp g_2 < \supp g_1$, a contradiction.
\end{proof}

The main goal of this section is to show that if $G$ and $\bar{G}$ satisfy suitable hypotheses, 
each isomorphism $\alpha \colon G\iso \bar{G}$ is induced by conjugation 
by a unique homeomorphism $\varphi \colon J \iso \bar{J}$.  
The uniqueness part is a consequence of Ax4. 
\begin{lemma}
\label{LemmaE2}
If $G$ satisfies Ax2 and Ax4 the only homeomorphism 
$\varphi \colon J \iso J$ centralizing $G$ is the identity.
\end{lemma}

\begin{proof}
Suppose $\varphi(t) \neq t$ for some $t$ in $J$.
By the continuity of $\varphi$ 
there exists then a small interval $J' =\,]t_2', t_3']$ 
that contains $t$ and is disjoint from its image $\varphi(J')$.
Use axiom Ax2 to find an element $g \neq \id$ in $G$ with bounded support,
say contained in $]t_2, t_3[$, 
and then axiom Ax4  to find $h \in G$ 
so that $ t'_2 < h(t_2) < h(t_3) < t'_3$.
Set $g_1 = \act{h} g$. 
Then  $\supp g_1$ and  $\varphi(\supp g_1)$ are disjoint 
and so $g_1 \neq \act{\varphi} g_1$.
\end{proof}

We are now ready to establish the announced modification of McCleary's result \cite[Main Theorem 4]{McC78b}:
\begin{theorem}
\label{TheoremE3}
\index{Representation Theorem for isomorphisms!general result}%
\index{Isomorphisms!representation}%
Assume $G$ satisfies axioms Ax1, Ax3 and Ax4,
and $\bar{G}$  satisfies axioms Ax2 and Ax4.
If $\alpha \colon G \iso \bar{G}$ is an isomorphism of groups 
there exists a unique homeomorphism $\varphi \colon J \iso \bar{J}$
which induces $\alpha$ by conjugation.
\end{theorem}

\subsubsection{Proof of Theorem \ref{TheoremE3}}
\label{sssec:Proof-of-TheoremE3}
%
If $\varphi$ exists, it is unique by Lemma \ref{LemmaE2}.
As a first step on the way to $\varphi$,
we introduce the subset $\Bfr$ of $J$ defined by
\begin{equation}
\label{eq:Definition-Bfr}
\Bfr = \{b \in J \mid \text{there exist }  g_\ell, g_u \text{ in } G \text{ with } \sup(\supp g_\ell) = b = \inf(\supp g_u) \},
\end{equation}
and the analogously defined subset $\bar{\Bfr}$ of $\bar J$.  
Since $G$ satisfies Ax3 the set $\Bfr$ is non-empty; 
hypothesis Ax4 then implies that $\Bfr$ is a dense subset of $J$.

Our next aim is to construct a function $\varphi_1 \colon \Bfr \to \bar{\Bfr}$. 
Consider $b \in \Bfr$. 
By definition,
$G$ contains elements $g_\ell$ and $g_u$ such that
\begin{equation}
\label{eq:16.4}
\sup(\supp g_\ell) = b = \inf(\supp g_u).
\end{equation}
We claim that the images $\alpha(g_\ell)$ and $\alpha(g_u)$
have the property that
\begin{equation}
\label{eq:16.4plus}
\sup(\supp \alpha(g_\ell)) = \inf(\supp \alpha(g_u))
\quad \text{or} \quad
\sup(\supp \alpha(g_u) )= \inf(\supp \alpha(g_\ell)).
\end{equation}
Once this is proved,
$\bar{b} = \sup(\supp \alpha(g_\ell))$
or  $\bar{b} = \sup(\supp \alpha(g_u))$, respectively,
will be our candidate for $\varphi_1(b)$.

To establish property \eqref{eq:16.4plus} we proceed as follows.
Formula \eqref{eq:16.4} implies, in particular,  
that $\supp g_\ell \leq \supp g_u$, 
\ie{}that every element in the support of $g_\ell$ is smaller 
than every element in the support of $g_u$.
The group $G$ contains, by axiom Ax1 a strictly positive element $f$. 
Every conjugate of $f$ is then strictly positive and so the formula
\begin{equation*}
\supp g_\ell \leq \left(\act{h}f \right) (\supp g_u) = \supp\left(\act{hfh^{-1}}g_u\right)
\end{equation*}
holds for every $h \in G$.
It implies that $g_\ell$ commutes with each conjugate $\act{hfh^{-1}}g_u$ of $g_u$.
Since $\alpha$ is a homomorphism
this commutativity relation is also true for the images 
$\bar{g}_\ell = \alpha(g_\ell)$,  $\bar{g}_u = \alpha(g_u)$ 
and for the images of the conjugates $hfh^{-1}$ of $f$;
as $\alpha$ is surjective this consequence can be stated as follows:
\[
\text{$\bar{g}_\ell$ commutes, 
for each $\bar{h} \in \bar{G}$, 
with the conjugate $\act{\bar{h} \bar{f} \bar{h}^{-1}}\bar{g}_u$  of $\bar{g}_u$} .
\]
Note that $\bar f\neq \id$, for  $f\neq \id$ and $\alpha$ is injective. 
  
Two cases now arise: 
if there exists $\bar t \in \bar J$ with $\bar t < \bar{f}(\bar t)$ 
then $\bar f$ is strictly positive by Lemma \ref{LemmaE1};
otherwise, $\bar f$ is strictly negative.

\emph{Suppose first that $\bar{f} = \alpha(f)$ is strictly positive}.
The preceding argument can then be summarized by the implication
\begin{equation}
\label{eq:Summary1-proof-TheoremE3}
\text{if $(g_\ell, g_u) \in G^2$ satisfies $\supp g_\ell  \leq \supp g_u$ 
then $\supp \alpha(g_\ell ) \leq \supp \alpha(g_u)$.}
\end{equation}

Assume now that the triple $(g_\ell, b,  g_u) \in G \times \Bfr \times G$
satisfies equality \eqref{eq:16.4}. 
Then the images $\bar g_1 =\alpha(g_1)$ and $\bar g_u = \alpha(g_u)$ 
satisfy relation $\supp \bar g_\ell \leq \supp \bar g_u$.  
But more is true:
the equality
\begin{equation}
\label{eq:16.5}
\sup(\supp \bar g_\ell) = \inf(\supp \bar g_u)
\end{equation}
holds.
Indeed, 
if equation \eqref{eq:16.5} were not valid
there would exist, by axioms Ax2 and Ax4, an element $\bar g = \alpha(g)$ in 
$\bar G\smallsetminus \{\id\}$, 
having bounded support and satisfying the inequalities
$\supp \bar g_\ell \leq \supp \bar g\leq \supp \bar{g}_u$.
Since $\bar G$ satisfies Ax4 and $\bar{f}$ is strictly positive,
Lemma \ref{LemmaE1} would allow us to deduce the inequalities
\begin{equation}
\label{eq:Triple-inequality}
\supp g_\ell \leq \supp g\leq \supp g_u
\end{equation}
in contradiction to formula \eqref{eq:16.4}.  

The preceding argument shows, first of all, 
that the subset $\bar{\Bfr}$ is non-empty 
and so dense in $\bar{J}$ (by axiom Ax4).
It implies next
that each triple $(g_\ell,b, g_u)$ satisfying condition \eqref{eq:16.4} 
defines an element $\bar b$ of $\bar{\Bfr}\subset \bar J$, given by 
\begin{equation}
\label{eq:Definition-of-b-bar}
\sup(\supp \alpha(g_\ell)) = \bar b = \inf(\supp\alpha(g_u)).
\end{equation}
If $g_\ell$ is fixed,  
the element $\bar b$ does not depend on the choice of $g_u$;
if $g_u$ is fixed, it does not depend on $g_\ell$. 
The assignment $b \mapsto \bar b$ defines therefore a \emph{function} 
$\varphi_1 \colon \Bfr \to \bar{\Bfr}$.  
This function is \emph{strictly increasing}.
Indeed,  if $b < b'$ and $b = \sup(\supp g_\ell)$, $b' = \inf(\supp g_u)$
there exists (by Ax2 and Ax4) an element $g$ satisfying relation \eqref{eq:Triple-inequality} 
whence implication \eqref{eq:Summary1-proof-TheoremE3} entails that
\begin{equation*}
\varphi_1(b) = \sup(\supp \bar g_\ell)) < \inf(\supp \bar g_u) = \varphi_1(b').
\end{equation*}

So far we know 
that $\varphi_1 \colon \Bfr \to \bar{\Bfr}$  is a strictly increasing function, 
defined on a dense subset of $J$,
and that $\bar{G}$ satisfies not only axioms Ax2 and Ax4,
but also axiom Ax1 (for $\bar{f} = \alpha(f)$ is strictly positive)
and axiom Ax3. 
The roles of $G$ and $\bar{G}$ are therefore interchangeable.
It follows, in particular, 
that  $\bar{\Bfr}$ is a dense subset of $\bar J$
and that $\varphi_1\colon \Bfr \to \bar{\Bfr}$ is bijective.
Let $\tilde{\varphi}_1\colon \Bfr \to \bar{J}$ 
denote the composition of $\varphi_1$ and the inclusion $\bar{\Bfr} \incl \bar{J}$.
Then $\tilde{\varphi}_1$ is a strictly increasing function
whose image is dense in $\bar{J}$
and so it extends uniquely to a strictly increasing and continuous function 
$\varphi \colon J \to \bar{J}$ 
(\cf{}\cite[IV.4.2.1,  III.3.15.5 and III.3.15.3]{Die69}).
This extension is actually surjective (by the Intermediate Value Theorem)
and thus bijective. 
By exchanging the rôles of $G$ and $\bar{G}$ one sees 
that $\varphi^{-1}$ is continuous, too.
All taken together,
this shows
that $\varphi \colon J \to \bar{J}$ is a homeomorphism 
with image $\bar{J}$.

We are left with proving 
that the isomorphism $\alpha \colon G \iso \bar G$ is induced by conjugation by $\varphi$.
For this task we recall the definition of $\Bfr$. 
Given $b \in \Bfr$, there exist homeomorphisms $g_\ell$ and $g_u$ in $G$ such that
\[
\sup ( \supp g_\ell) = b = \inf ( \supp g_u).
\]
Since each $g \in G$ is orientation preserving 
it follows that
\begin{equation*}
g(b)
=
g(\sup ( \supp g_\ell)) 
= 
\sup ( g(\supp g_\ell) ) = \sup \left(\supp \act{g} g_\ell \right).
\end{equation*}
On the other hand, 
if $b'$ is any point of $\Bfr$ and $g'_\ell \in G$ is any element with  $b' = \sup (\supp g'_\ell)$
then $\varphi(b') = \sup ( \supp \alpha( g'_\ell))$  holds by equation \eqref{eq:Definition-of-b-bar}.
These two formulae then allow one to justify the following chain of equations:
\begin{align*}
\left( \varphi \circ g \right) ( b) 
=\varphi(g(b))
&=
\varphi \left( \sup \left( \supp \act{g} g_\ell\right) \right)\\
&=
\sup \left( \supp \alpha\left(\act{g} g_\ell\right) \right)
=
\sup \left( \supp   \act{\alpha(g)} \alpha(g_\ell) \right)\\
&=
\alpha(g) \left( \sup \left( \supp  \alpha(g_\ell)\right) \right) \\
&=
\alpha(g) ( \varphi(b))=
\left(\alpha(g) \circ \varphi \right) (b).
\end{align*}
The chain shows that the homeomorphisms $\varphi \circ g$ and $\alpha(g) \circ \varphi$ 
agree on the subset $\Bfr$ of $J$.
As this subset is dense in the interval $J$ the two homeomorphisms agree everywhere 
whence $\alpha (g) = \varphi \circ g \circ \varphi^{-1}$.
The first case of the proof of Theorem \ref{TheoremE3} is now complete
 \smallskip

\emph{Suppose, secondly,  that $\bar f = \alpha(f)$ is strictly negative.}
Define $\breve{f}$ to be the function 
$\act{-\id} \bar f \colon -t \mapsto -\bar{f}(t)$.
Then $\breve{f}$ is strictly positive, for the inequality $\bar{f}(t) \leq t$ entails
that $-t \leq -\bar{f}(t) = \breve{f}(-t)$.
The homeomorphism $\breve{f}$ is an element of the group $\breve{G} = \act{-\id} \bar{G}$
and this group satisfies axioms Ax2 and Ax4.
Indeed, 
the hypothesis that $\bar{G}$ satisfies axiom Ax2 obviously implies 
that $\breve{G}$ does so.
Suppose now 
that $(-t_1, \ldots, -t_6)$ and $(-t'_1, \ldots, -t'_6)$ are two strictly increasing six-tuples
of elements in $-\bar{J}$.
Then $(t_6, \ldots, t_1)$ and $(t'_6, \ldots, t'_1)$ are two strictly increasing six-tuples 
with elements in $\bar{J}$. 
Since $\bar{G}$ satisfies axiom Ax4 
there exists therefore a homeomorphism $h \in \bar{G}$ satisfying the chain of inequalities
\[
h(t_6) < t'_6 < t'_5 < h(t_5) < h(t_4) < t'_4 < t'_3 < h(t_3) < h(t_2) < t'_2 < t'_1 < h(t_1).
\]
The negatives of this chain of numbers satisfy then a chain of inequalities of the form 
\[
-h(t_1) < -t'_1 < \cdots  < -t'_5 < -t'_6 < -h(t_6);
\]
it shows 
that $\breve{h} = \act{-\id} h \colon -\bar{J} \iso -\bar{J}$ satisfies the chain of inequalities 
\eqref{eq:16.1} with $h$ replaced by $\breve{h}$ 
and the six-tuples $(-t_1, \ldots, -t_6)$ and $(-t'_1, \ldots, -t'_6)$.
The group $\breve{G}$ satisfies therefore axiom Ax4.

Consider now the isomorphism 
$\breve{\alpha} 
\colon 
G \xrightarrow{\alpha} \bar{G} \xrightarrow{\bar{g} \mapsto (-\id) \circ \bar{g} \circ (-\id)} \breve{G}$.
It maps the strictly negative homeomorphism $f$ onto $\breve{f}$ 
which is a strictly positive element of $\breve{G}$.
By the first case of the proof 
the isomorphism $\breve{\alpha}$ is therefore induced 
by conjugation by a homeomorphism $\breve{\varphi} \colon J \iso -\bar{J}$,
and so $\alpha$ itself is induced by conjugation by $\varphi = (-\id) \circ \breve{\varphi}$.
%
\subsection[Consequences for groups of PL-homeomorphisms]%
{Consequences of Theorem \ref{TheoremE3}}
\label{ssec:16.3}
%
We begin by showing 
that Theorem \ref{TheoremE3} applies 
to every isomorphism of groups $\alpha \colon G \iso \bar{G}$
where $G$ is a subgroup of a group $G(I;A,P)$ containing the derived group of $B(I;A,P)$
and $\bar{G}$ a subgroup of a group $G(\bar{I};\bar{A}, \bar{P})$ 
containing the derived group of $B(\bar{I};\bar{A}, \bar{P})$.
\emph{A priori}, 
there may not be any relations 
between the triples $(I,A,P)$ and $(\bar{I},\bar{A}, \bar{P})$;
it will turn out, however,
that $\alpha$ induces always an isomorphism $P \iso \bar{P}$ of a specific type
(see Proposition \ref{PropositionE8})
and that a much tighter relationship holds if the group $P$ is not cyclic: 
then $\bar{P} = P$ and there exists a positive real number $s$ 
so that $\bar{A} = s \cdot A$
(see Theorem \ref{TheoremE10}).

\subsubsection{The main result and its proof}
\label{sssec:TheoremE04}
%
\begin{thm}
\label{TheoremE04}
\index{Isomorphisms!main result}%
\index{Representation Theorem for isomorphisms!main result}%
\index{Theorem \ref{TheoremE04}!statement|textbf}%
Assume $G$ is a subgroup of $G(I;A,P)$ containing $B(I;A,P)'$
and that $\bar{G}$ is a subgroup of $G(\bar{I};\bar{A},\bar{P})$ containing 
$B(\bar{I}; \bar{A},\bar{P})'$.
If $\alpha \colon  G \iso \bar{G}$ is an isomorphism of groups
there exists a unique homeomorphism $\varphi \colon \Int(I) \iso \Int(\bar{I})$ 
which induces $\alpha$ by conjugation; more precisely,
the equation
\begin{equation}
\label{eq:16.6}
\alpha(g) \restriction{\Int(\bar{I}})  = \varphi \circ (g \restriction{\Int(I)}) \circ \varphi^{-1}
\end{equation}
holds for every $g \in G$.
Moreover, $\varphi$ maps $A \cap \Int(I) $ onto $\bar{A} \cap \Int (\bar{I})$.
\end{thm}

\begin{proof}
We first verify that $G$ satisfies axioms Ax1 through Ax4, 
listed in section \ref{ssec:16.1}. 
The group $B' = B(I;A,P)'$ consists of elements of bounded support 
and it is not reduced to the unit element 
(\cf{}Corollary \ref{crl:Simplicity-of-B'}), 
and thus axiom Ax2 holds for every group $G$ containing $B'$.
Theorem \ref{TheoremA} and remark \ref{remark:5.3} imply 
\index{Theorem \ref{TheoremA}!consequences}%
next that $B'$, and hence $G$,  satisfy axiom Ax4.
\smallskip

Now to axiom Ax1. 
It postulates the existence of a strictly positive element in $G$;
we establish its validity by an explicit construction.
Choose $p \in P$ with $p > 1$ and $a \in A \cap \Int{I}$.
Find a positive element $\Delta \in A$ that is so small 
that 
\[
c = a + (2p + 1) \cdot \Delta \in \Int(I).
\]
These choices being made, we construct PL-homeomorphisms $f$ and $g$;
they will have the form of the PL-homeomorphism $b(a, \Delta; p)$ 
discussed in Example \ref{example:Elements-in-B}.
The PL-homeomorphism $g$ is the affine interpolation of the points
\[
(a,a), \quad (a + \Delta, a + p \cdot \Delta), \quad (a + (p+1) \cdot \Delta, a + (p+1) \cdot \Delta)
\]
(extended by the identity outside of the interval $[a, a + (p+1) \cdot \Delta]$);
so $g = b(a, \Delta; p)$ in the notation of Example  \ref{example:Elements-in-B}.
The PL-homeomorphism $f$  is a translated copy of $g$, namely $b(a + p \Delta, \Delta; p)$.
The commutator $[f,g] = f \circ g \circ f^{-1} \circ g^{-1}$ 
has then the rectangle diagram depicted in the next figure.
\medskip

\begin{minipage}{12 cm}
\psfrag{1}{\hspace*{-1.5mm}     \small $0$}
\psfrag{2}{\hspace*{-1.3mm}    \small $1$}
\psfrag{3}{\hspace*{-1.5mm}     \small $p$}
\psfrag{4}{\hspace*{-6mm}     \small $\tfrac{p^2 +p - 1}{p}$}
\psfrag{5}{\hspace*{-2mm}     \small $p+1$}
\psfrag{6}{\hspace*{-2.0mm}  \small $2p$}
\psfrag{7}{\hspace*{-5.5mm}     \small $2p+1$}
\psfrag{11}{\hspace*{-3.5mm}     \small $\tfrac{1}{p}$}
\psfrag{12}{\hspace*{-1.4mm}    \small $p$}
\psfrag{13}{\hspace*{-1.4mm}     \small $p$}
\psfrag{14}{\hspace*{-2mm}     \small $1$}
\psfrag{15}{\hspace*{-2.5mm}     \small $1$}
\psfrag{21}{\hspace*{1.7mm}     \small $1$}
\psfrag{22}{\hspace*{-1.7mm}    \small $1$}
\psfrag{23}{\hspace*{-2.0mm}     \small $\tfrac{1}{p}$}
\psfrag{24}{\hspace*{-4.0mm}     \small $\tfrac{1}{p}$}
\psfrag{25}{\hspace*{2.5mm}     \small $p$}
\psfrag{31}{\hspace*{-3mm}     \small $p$}
\psfrag{32}{\hspace*{1.7mm}    \small $\tfrac{1}{p}$}
\psfrag{33}{\hspace*{-0.5mm}     \small }
\psfrag{34}{\hspace*{0.5mm}     \small $\tfrac{1}{p}$}
\psfrag{35}{\hspace*{7.3mm}     \small $1$}
\psfrag{41}{\hspace*{-7.4mm}     \small $1$}
\psfrag{42}{\hspace*{-1.0mm}    \small $p$}
\psfrag{43}{\hspace*{-1.7mm}     }
\psfrag{44}{\hspace*{-0.3mm}     \small $p$}
\psfrag{45}{\hspace*{2.4mm}     \small $\tfrac{1}{p}$}
\psfrag{101}{\hspace*{0.2mm}     \small $0$}
\psfrag{102}{\hspace*{0.1mm}    \small $p$}
\psfrag{103}{\hspace*{-6.5mm}     \small $2p-1$}
\psfrag{104}{\hspace*{-5mm}     \small $\tfrac{2p^2 -p + 1}{p}$}
\psfrag{105}{\hspace*{-0.8mm}     \small $2p$}
\psfrag{106}{\hspace*{-4.0mm}  \small $2p+1$}
\psfrag{ginv}{\hspace*{-5.5mm}  \small $g^{-1}$}
\psfrag{finv}{\hspace*{-5.5mm}  \small $f^{-1}$}
\psfrag{g}{\hspace*{-5.8mm}  \small $g$}
\psfrag{f}{\hspace*{-5.8mm}  \small $f$}
\begin{equation*}
\index{Rectangle diagram!examples}%
\includegraphics[width= 11cm]{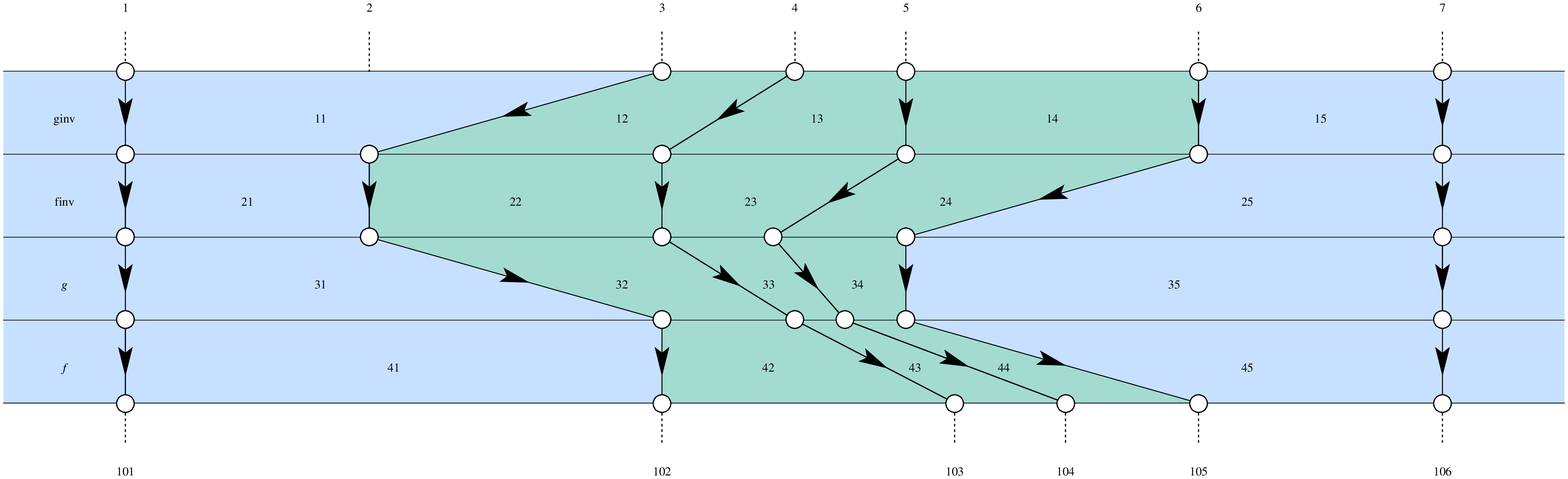}
\end{equation*}
\end{minipage}
\medskip

To simplify notation in this figure, but also in the calculations below,
we consider only the case where $a = 0$ and $\Delta = 1$;
the general case can be reduced to this one by a suitable affine map.

The diagram will be used to compute the commutator $h = [f,g]$.
To render legitimate this verification by a figure, 
we have to check that the four homeomorphisms 
and the arrangements of the auxiliary subdivisions
are as indicated by the figure.
For the first verification, 
we recall that $1 < p$,
that $f$ is the identity on $[0, p]$,
has slope $p$ on the interval $]p, p+1[$ 
and slope $1/p$ on the interval $]p+1, 2p+1[$,
and that $g$ has slope $p$ on the interval $]0, 1[$,  
slope $1/p$ on the interval $]1, p+1[$
and that is the identity on $[p + 1, 2p +1]$.
It follows that the support of $[f,g]$ is contained in the interval 
$]p,  2p[$.

The second verification amounts to check 
that $g(p) = p + 1 - 1/p$ 
and that $p < p + 1 - 1/p < p + 1  < 2p$ for every $p > 1$;
both verifications are straightforward.
The figure next permits one to infer 
that $[f,g]$ fixes the points $p$ and $2p$,
and that it has slope $p$ on the interval $]p, p + 1 - 1/p[$,
slope 1 on the interval $]p + 1 - 1/p, p + 1[$
and slope $1/p$ on the interval $]p+1, 2p[$,
We conclude that $[f,g]$ is strictly positive and 
that its support is the interval $]p,2p[$.

Conjugation by affine transformations then shows the following:
\emph{whenever $a$ is a point in $A \cap \Int(I)$, $p > 1$ is in $P$ 
and $\Delta$ is a small positive element of $A$ such that $a + (2p+1) \Delta \in \Int(I)$, 
then the derived group of $B(I;A,P)$ contains a strictly positive element 
with support $]a + p\cdot \Delta, a + 2p\cdot \Delta[$}.
So $B'$ satisfies axiom Ax1 and also axiom Ax3.

So far we know that the group $G$ satisfies axioms Ax1 through Ax4;
as the group $\bar{G}$ has the same properties as $G$, 
it satisfies these axioms, too.
Theorem \ref{TheoremE3} therefore implies 
that the given isomorphism $\alpha \colon G \iso \bar{G}$ is induced 
by conjugation by a homeomorphism $\varphi \colon \Int(I) \iso \Int(\bar{I})$.

We are left with proving 
that $\varphi$ maps the subset $A \cap \Int(I)$ onto $\bar{A} \cap \Int (\bar{I})$.
Let $a$ be a point in $A \cap \Int(I)$ and choose $p \in P$ with $p > 1$. 
There exists then a positive number $\Delta \in A$ which is so small 
that the interval $[a - p \cdot \Delta, a + (p+1) \cdot \Delta]$ is contained in $\Int(I)$.
By the construction carried out at the beginning of the proof,
there exists thus a strictly positive element $h \in B(I;A,P)' \subseteq G$ 
which moves every point in the interval $]a, a + p \cdot \Delta[$.
The image $\alpha(h) = \varphi \circ h \circ \varphi^{-1}$ is therefore a PL-homeomorphism
with support $]\varphi(a), \varphi(a + p \Delta)[$ 
or $] \varphi(a + p \Delta), \varphi(a)[$.
In both cases, the derivative of $\alpha(h)$ is 1 on one side of $\varphi(a)$ 
and distinct from 1 at the other side.
Thus $\varphi(a)$ is a singularity of $\alpha(h) \in \bar{G} \subset G(\bar{I}; \bar{A}, \bar{P})$,
so $\varphi(a) \in \bar{A}$.
This shows that 
$\varphi(A \cap \Int(I)) \subseteq \bar{A}\, \cap \,\Int(\bar{I})$.
As the rôles of $G$ and $\bar{G}$ can be reversed in the previous argument,
we have established 
that $\varphi(A \cap \Int(I)) = \bar{A} \cap \Int(\bar{I})$.
\end{proof}

\subsubsection{Application to the kernels of $\lambda$ and of $\rho$}
\label{sssec:Kernels-lambda-rho}
Theorem \ref{TheoremE04} will provide the basis of many results of this chapter.  
We open the host of consequences with an application
that relies on the fact 
that $\varphi$ is a homeomorphism mapping an open interval onto another open interval 
and so is, either orientation preserving, or orientation reversing.  
Prior to stating the result, we introduce a parlance.

Let $G$ be a subgroup of $G(I;A,P)$, 
let $B$ be the group of bounded homeomorphism $B(I;A,P)$,
and let $\bar G$, $\bar B$ be defined analogously. 
Suppose that $G$ contains $B'$, that $\bar{G}$ contains $\bar{B}'$ and 
that $\alpha \colon G \iso \bar G$ is an isomorphism; 
let $\varphi \colon \Int(I) \iso \Int(\bar I)$ 
be the unique homeomorphism inducing $\alpha$ by conjugation. 
\emph{We say that $\alpha$ is increasing if $\varphi$ is so};
otherwise $\alpha$ will be called \emph{decreasing}
\index{Increasing isomorphism}%
\index{Decreasing isomorphism}%
\begin{corollary}
\label{CorollaryE5}
\index{Group G(I;A,P)@Group $G(I;A,P)$!characteristic subgroups}%
\index{Group G(R;A,P)@Group $G(\R;A,P)$!characteristic subgroups}%
\index{Group G([0,infty[;A,P)@Group $G([0, \infty[\;;A,P)$!characteristic subgroups}%
\index{Group G([a,c];A,P)@Group $G([a,c];A,P)$!characteristic subgroups}%
\index{Representation Theorem for isomorphisms!consequences}%
\index{Theorem \ref{TheoremE04}!consequences}%
\begin{enumerate}[(i)]
\item If $G$ contains $B$ and $\bar G$ contains $\bar B$ 
every isomorphism $\alpha \colon G \iso \bar G$ maps $B$ onto $\bar B$.
 \item $B$ is a characteristic subgroup of every group $G$ 
 with $B\leq G \leq G(I;A,P)$.  
 \item Suppose $G$ contains the derived group of $B$ 
 and $\bar G$ contains the derived group of $\bar B$.  
 If $\alpha \colon G \iso \bar G$ is \emph{increasing} 
 it maps the kernel of the homomorphism 
 $(\lambda \restriction{G}) \colon G\to \Aff(A,P)$ 
 onto the kernel of
$(\bar\lambda\restriction{\bar G}) \colon \bar{G} \to  \Aff(\bar A,\bar P)$ 
and induces by passage to the quotients an isomorphism
\begin{equation*} 
\alpha_\ell \colon  G/(G \cap \ker \lambda) \iso \bar{G}/(\bar{G} \cap \ker \bar{\lambda}).
\end{equation*}
Similarly, 
$\alpha$ maps $\ker(\rho\restriction{G})$ onto $\ker(\bar\rho \restriction{\bar G})$ 
and induces an isomorphism
\begin{equation*}
\alpha_r \colon  G/(G \cap \ker \rho) \iso \bar{G}/(\bar{G} \cap \ker \bar{\rho}).
\end{equation*}
\end{enumerate}
\end{corollary}

\begin{proof}  
If $g \in G$ has bounded support and $\varphi\colon \Int(I) \iso \Int(\bar{I})$ is a homeomorphism 
then $\act{\varphi}g$ has bounded support.  
If $\varphi$ is increasing and $g \in G$ is bounded from below 
the same is true of $\act{\varphi}g$; 
similarly, if $g$ is bounded from above, so is $\act{\varphi}g$.  
The assertions of the corollary are immediate consequences of Theorem \ref{TheoremE04},
the listed facts and the analogous facts 
obtained by reversing the rôles of $G$ and $\bar{G}$.
\end{proof}
%
\subsection{Explicit constructions of isomorphisms: part  I}
\label{ssec:16.4}
\index{Group G(R;A,P)@Group $G(\R;A,P)$!isomorphisms}%
Many useful isomorphisms $\alpha \colon G \iso \bar G$ can be obtained 
by constructing explicit homeomorphisms $\varphi \colon \Int(I) \iso \Int(\bar I)$.  
Evidence for this statement has been given already in Chapter \ref{chap:A}, 
where $\varphi$ has been chosen to be a restriction of an affine homeomorphism of $\R$.  
In Proposition \ref{PropositionC1} then a PL-homeomorphism 
with infinitely many singularities has been employed to show 
that $B(I;A,P)$ is isomorphic to $B(\R;A,P)$ for each interval $I$ (with non-empty interior).  

By varying the construction given in the proof of Proposition \ref{PropositionC1},
\index{Proposition \ref{PropositionC1}!analogues}%
one can also show
that each group $G(I;A,P)$ belongs to one of the six types of groups
enumerated in section \ref{ssec:2.4} of the introduction; 
these groups correspond to the intervals
\begin{enumerate}[(1)]
\item  $\R$,
\item $[0,\infty[$,
\item $]0,\infty[$ with associated group 
$\ker(\sigma_- \colon G([0,\infty[\;;A,P) \to P)$,
\item  $[0,b]$,
\item  $]0,b]$ with associated group $\ker(\sigma_-\colon G([0,b];A,P)\to P)$, and
\item  $]0,b[$ with associated group $B([0,b];A,P)$.
\end{enumerate}

\begin{proposition}
\label{PropositionE6}
\index{Intervals|see{Classification of intervals}}%
\index{Classification!of intervals}%
\index{Representation Theorem for isomorphisms!consequences}%
\index{Isomorphisms!construction}%
\index{Proposition \ref{PropositionC1}!consequences}%
Each group $G(I;A,P)$ is isomorphic to a group of one of the above-listed six types; the type is uniquely determined by $I$.
\end{proposition}

\begin{proof}
We first deal with the \emph{existence} of an isomorphism; 
it is clear for $I = \R$.  
If $I$ is bounded on one side with endpoint in $A$, or bounded on both sides with
endpoints in $A$, there is an affine homeomorphism $f$ in $\Aff(A,\{1, -1\})$ 
so that $f(I)$ is the interval $[0,\infty[$ is the first case 
and an interval of the form $[0,b]$ with $b$ in $A$ in the second case.  
If $I$ is bounded, but none of the endpoints is in $A$, 
then $G(I;A,P) = B(I;A,P)$ and $B(I;A,P)$ is isomorphic to a group $B([0,b];A,P)$
(by Proposition \ref{PropositionC1}).  
There remain the cases
where exactly one endpoint is in $\R \smallsetminus  A$.  
Then one can modify the proof of Proposition \ref{PropositionC1} 
\index{Proposition \ref{PropositionC1}!analogues}%
so as to obtain a homeomorphism $\varphi \colon \Int(I) \iso\ ]0,\infty[$ 
provided $I$ or $-I$ has the form $[a,\infty[$ with $a \notin A$, 
or a homeomorphism $\varphi \colon \Int I \iso\ [0,b[$ 
if $I$ or $-I$ is of the form $[a,c]$ with $a \in A$ and $c \in \R\smallsetminus A$.
\smallskip

Now to the \emph{uniqueness} of the type.
Suppose $I_1$, $I_2$ are intervals of one of the above six types,
and $G_j =G(I_j;A,P)$ and $B_j = B(I_j; A, P)$ for $j \in \{1,2\}$.
Let $\alpha \colon G_1 \iso G_2$ be an isomorphism.
Then part (i) of Corollary \ref{CorollaryE5} implies 
that the quotient groups $G_1/B_1$ and $G_2/B_2$ are isomorphic.
These quotient groups are metabelian,  but non-abelian 
if the interval is of one of the types (1), (2) or (3);
abelian but not trivial if the interval is of type (4) or (5),
and trivial in the case of type (6).
It follows
that the intervals $I_1$ and $I_2$ are, either both of one of the types (1), (2) or (3),
or both of one of the types (4) or (5), or both of type (6).

Consider next the case where the intervals $I_1$ and $I_2$ are of one of the types (1), (2) or (3)
and let $\varphi \colon \Int(I_1) \iso \Int(I_2)$ 
be a homeomorphism inducing $\alpha \colon G_1 \iso G_2$.
Assume first that $\varphi$ is \emph{increasing}.
Then $\varphi$ induces an isomorphism of the image of $\lambda_1$ 
onto the image of $\lambda_2$ (by part (iii) of Corollary \ref{CorollaryE5}). 
If $I_1$ is of type (1) the image of $\lambda_1$ is isomorphic to $A \rtimes P$ 
(by Corollary \ref{CorollaryA2});
if $I_2$ is of type (2) this image is isomorphic to  $P$ (again by Corollary \ref{CorollaryA2});
if, thirdly, $I_1$ is of type (3), this image is trivial.
It follows that the groups $G_1$ and $G_2$ can only be isomorphic via an increasing isomorphism
if both are of the same type.
If, on the other hand, $\varphi$ is \emph{decreasing}
then $\alpha$ will map the kernels of $\lambda_1$ and $\rho_1$
onto the kernels of $\rho_2$ and $\lambda_2$, respectively.
This fact implies, much as before, that the groups $G_1$ and $G_2$ must be of the same type.

We are left with the case where the intervals are of type (4) or (5).
If $I_1$ is of type (4),  
the images of both $\lambda_1$ and $\rho_1$ are isomorphic to $P$,
and if it is of type (5) the image of $\lambda_1$ is trivial.
These facts and part (iii) of Corollary \ref{CorollaryE5} allow us to deduce, 
much as before,
that the groups can only be isomorphic if the intervals $I_1$ and $I_2$ are, either both of type (4), or both of type (5).
\end{proof}

Each of the above types of intervals $I$ gives rise to a single isomorphism type of groups $G(I;A,P)$, except possibly for types (4) and (5). 
Actually, type (5) is no exception; 
indeed, given $b$, $b'$ in $A_{>0}$ 
there exists an infinitary PL-homeomorphism 
$\varphi \colon \,]0, b] \iso \,]0,b']$ 
that maps $A$ onto itself, 
has only finitely many singularities in every compact subinterval of $]0, b]$ 
and which induces an isomorphism 
$\alpha$ of $G(\,]0, b]; A, P)$ onto $G(\,]0, b']; A, P)$
(\cf{}the proof of Proposition \ref{PropositionC1}).
\index{Proposition \ref{PropositionC1}!analogues}%

For an interval of type (4),
the case where
$P$ is cyclic differs radically from the case where $P$ is non-cyclic.  
To gain an insight into the problem 
consider two intervals $[0,b]$ and $[0,b']$ with endpoints in $A$.  
If there exists a (finitary) PL-homeomorphism $\varphi$ in $G(\R;A,P)$ mapping $[0,b]$ onto $[0,b']$ then conjugation by $\varphi$ induces an isomorphism of $G([0,b];A,P)$ onto $G([0,b'];A,P)$.  
The same conclusion holds 
if $\varphi$ can be found in $G(\R;A,\Aut(A))$, 
where $\Aut(A) = \{s\in \R^\times \mid sA = A\}$, 
and the slopes of $\varphi \restriction{[0,b]}$ 
lie in a single coset $s\cdot P$ of $P$. 
\index{Group Aut(A)@Group $\Aut(A)$!significance}%

Assume now that the group $P$ is \emph{not cyclic}.
Then every  isomorphism 
\[
\alpha \colon G([0,b];A,P) \iso G([0,b'];A,P)
\] 
is induced by a (finitary) PL-homeomorphism as described before  
(for details, see Supplement \ref{SupplementE11New}); 
\index{Supplement \ref{SupplementE11New}!consequences}%
it follows that the family of groups  
\[
\{G([0,b];A,P\mid 0 < b\in A\}
\]
can contain non-isomorphic members 
(see Corollary \ref{CorollaryE12}). 
\smallskip 

For \emph{cyclic} groups $P$ the situation is different.
\begin{theorem}
\label{TheoremE07}
\index{Representation Theorem for isomorphisms!consequences}%
\index{Group G([a,c];A,P)@Group $G([a,c];A,P)$!dependence on [a,c]@dependence on $[a,c]$}%
\index{Isomorphisms!construction}%
\index{Group P cyclic@Group $P$ cyclic!consequences}%
\index{Theorem \ref{TheoremE07}!statement|textbf}%
If $P$ is cyclic the groups in the family
\[
\left\{G([0,b]; A,P) \mid b \in A_{>0}\right\}
\]
are isomorphic to each other.
\end{theorem}
\begin{figure}[b]
\vspace*{3mm}
\psfrag{pp0}{\hspace*{3.5mm}     \small $(b_0, \bar{b}_0)$}
\psfrag{pp1}{\hspace*{1.5mm}    \small $(b_1, \bar{b}_1)$}
\psfrag{pp2}{\hspace*{1.5mm}     \small $(b_2, \bar{b}_2)$}
\begin{center}
\includegraphics[width= 8cm]{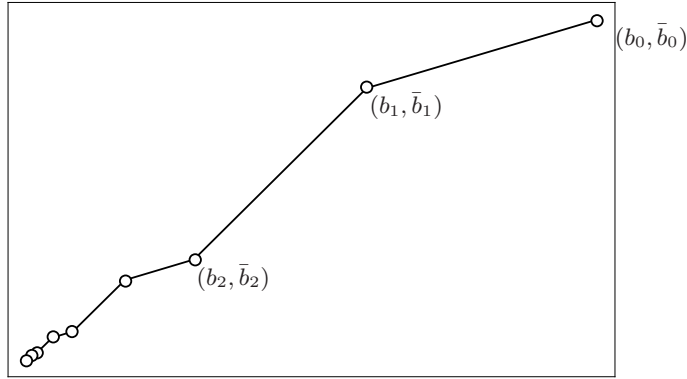}
\end{center}
\caption{Graph of the infinitary PL-homeomorphism $\varphi_1$}
\label{fig:Graph-varphi}
\end{figure}

\begin{proof}
Let $p$ be the generator of $P$ with $p < 1$.
Since the homothety $t \mapsto p \cdot t$ maps the interval $[0, b'|$ 
onto the interval $[0, p \cdot b']$,
it suffices to consider intervals $[0,b]$ and $[0,\bar{b}]$ with $p\cdot b < \bar b < b$.
The affine interpolation of 
\begin{equation*}
\left((pb,p\bar b),\quad (\bar b,\bar b + p(\bar b-b)),\quad (b,\bar b)\right)
\end{equation*}
has slopes $1$ and $p$, and vertices in $A$. 
With its help we now construct an infinitary PL-homeomorphism 
$\varphi \colon \,]0,b] \iso \,]0, \bar{b}]$.
Let $(b_i\mid i \geq 0)$ and $(\bar b_i\mid i\geq 0)$ be the sequences defined  by 
\begin{equation*}
b_{2i} = p^i\cdot b,\quad 
b_{2i+1} = p^i\cdot \bar b, \quad\text{and}\quad
\bar{b}_{2i}= p^i \cdot \bar b, 
\quad
\bar b_{2i+1} = p^i(\bar b+p(\bar b-b)).
\end{equation*}
These sequences are strictly decreasing 
and they have the property that multiplication by $p$ 
sends $b_i$ to $b_{i + 2}$ and $\bar{b}_i$ to $\bar{b}_{i+2}$.
The affine interpolation of the product $((b_i,\bar b_i) \mid i\geqq 0)$ of these sequences is an infinitary PL-homeomorphism $\varphi_1 \colon \,]0,b]\ \iso \,]0,\bar{p}]$ 
with slopes in $\{1,p\}$ and vertices in $A^2$;
it is indicated in Figure \ref{fig:Graph-varphi}.
The homeomorphism extends uniquely to a homeomorphism 
$\varphi \colon [0, b] \iso [0,\bar{b}]$ that fixes 0,
and this extension satisfies then the commutativity relation 
\begin{equation}
\label{eq:Commutativity-relation-for-phi}
(\varphi \circ p \cdot \id) \restriction{[0,1]} = (p \cdot \id \circ \varphi) \restriction{[0,1]}.
\end{equation}

We claim
that conjugation by $\varphi$ induces an isomorphism of $G = G([0,b];A,P)$ 
onto $\bar G = G([0,\bar b];A,P)$.  
Indeed, 
if $f \in G$ there exists an integer $k$ and an index $j\geq 0$ 
so that $f$ is linear on $[0,p^j\cdot b]$ with slope $p^k$. 
Consider now a positive real number $t$ in the interval
$[0, \min \{1, p^{-k}\} \cdot p^j \cdot \bar{b}]$.
Then $\varphi^{-1} (t) <  \min \{1, p^{-k}\} \cdot p^j \cdot b$
and so the chain of equalities
\begin{equation*}
(\varphi \circ f \circ \varphi^{-1})(t) = (\varphi \circ (p^k \cdot \id) \circ \varphi^{-1})(t)  = p^k \cdot t
\end{equation*}
holds by relation \eqref{eq:Commutativity-relation-for-phi}.
It shows that $\act{\varphi}f$ is linear with slope $p^k$ 
on the interval  $[0, \min \{1, p^{-k}\} \cdot p^j \cdot \bar{b}]$.
As the singularities of $\varphi$ accumulate only in 0
and as the slopes of $\varphi$ lie in $\{1, p\}$, 
it follows that $\act{\varphi}f  \in \bar G$.

The preceding argument proves 
that $\act{\varphi} G \subseteq \bar{G}$.
By reversing the rôles of $G$ and $\bar{G}$, 
one establishes similarly that $\act{\varphi^{-1}} \bar{G} \subseteq G$ 
and thus $\act{\varphi} G = \bar{G}$, as asserted
\end{proof}

\subsection[Auxiliary results helping one in showing that $\varphi$ is PL]%
 {Results helping one in showing  that $\varphi$ is PL}
\label{ssec:16.5}
The usefulness of Theorem \ref{TheoremE04} increases considerably 
if $\varphi$ is known to be a finitary or infinitary PL-homeomorphism.  
In this section we establish two auxiliary results that help one in proving 
that a homeomorphism $\varphi$ inducing an isomorphism is piecewise linear.
The set-up will be as in the statement of Theorem \ref{TheoremE04}.
Thus:  
\begin{itemize}
\item
\emph{$G$ is a subgroup of $G(I;A,P)$ containing the derived group of $B(I;A,P)$},
\item 
\emph{$\bar G$ is a subgroup of $G(\bar I;\bar A,\bar P)$ containing $B(\bar I;\bar A,\bar P)'$},
and 
\item \emph{$\varphi \colon \Int(I) \iso \Int(\bar{I})$ is a homeomorphism 
that induces by conjugation an isomorphism of groups $\alpha \colon G \iso \bar{G}$}.
\end{itemize}
\begin{proposition}
\label{PropositionE8}
\index{Representation Theorem for isomorphisms!supplements}%
\index{Theorem \ref{TheoremE04}!consequences}%
Assume $\varphi$ is \emph{order preserving}.  
Then there exist,
for every $a \in A \cap \Int(I)$,
positive real numbers $\delta$ and $r$ such that the formula
\begin{equation}
\label{eq:16.7}
\varphi(a + p \cdot t) = \varphi(a) + p^r \cdot \left(\varphi( a + t) - \varphi(a)\right)
\end{equation}
holds for every $t  \in [0,\delta]$ and each $p \in P \, \cap \,]0,1]$.  
Moreover, 
the automorphism $\theta_r \colon \R^\times_{> 0} \iso \R^\times_{> 0}$, 
taking $t$ to $t^r = e^{r\cdot \log(t)}$, 
maps $P$ isomorphically onto $\bar P$.
\end{proposition}

\begin{proof}
The proof consists of three parts.
In the first one,
the fact that $\varphi$ maps $A \cap \Int(I)$ onto $\bar{A} \cap \Int(\bar{I})$
and the hypothesis 
that $\alpha$ sends PL-homeomorphisms with slopes in $P$ 
to PL-homeomorphisms with slopes in $\bar{P}$ 
is shown to imply 
that $\alpha$ induces an isomorphism  $\vartheta_a \colon P \iso \bar{P}$ 
taking $p$ to $p^r$.
In the second part,
we establish formula \eqref{eq:16.7} for a single slope $p_1 \in P\, \cap \,]0,1]$.
In the third part, 
this result is then extended to all slopes $p \in P \, \cap \,]0,1]$.

We embark now on the first part.
The last statement of Theorem \ref{TheoremE04} guarantees 
that $\bar a = \varphi(a)$ lies in $\bar{A}$.
Consider the subgroups
\[
G_a = G\cap G(I\cap [a, \infty[\;;A,P) \quad\text{and} \quad 
\bar G_{\bar a} = \bar G \cap G(\bar I \cap [\bar a, \infty[\;;\bar A,\bar P).
\]
Part (ii) of Corollary \ref{CorollaryA2} 
and Remark \ref{remark:5.3} imply
that $G_a$ contains, for every $p \in P$,  an element $f_a$ 
whose right-hand derivative in $a$ equals $p$.
The assignment $f \mapsto \lim_{t\searrow a} f'(t)$ 
yields therefore an epimorphism  $\sigma_-\colon G_a \epi P$.
Similarly, one constructs an epimorphism
$\bar\sigma_- \colon \bar G_{\bar a} \epi \bar P$.
On the other hand,
the hypothesis that $\alpha$ is induced by an order preserving homeomorphism
implies 
that $\alpha$ maps the kernel of $\sigma_-$ onto the kernel of $\bar{\sigma}_-$
and thus gives rise to an isomorphism  $\theta_a \colon P\to \bar P$ 
which renders the square
\begin{equation*}
\xymatrix{G_a \ar@{->}[r]^-{\alpha_a} \ar@{->>}[d]^-{\sigma_-} &\bar G_{\bar a} \ar@{->>}[d]^-{\bar\sigma_-}\\
P \ar@{->}[r]^-{\vartheta_a} &\bar P}
\end{equation*}
commutative.

Our next aim is to work out the form of the isomorphism $\vartheta_a$.
As a first step towards this goal 
we verify 
that $\vartheta_a$ maps the set $P \, \cap \,]0, 1[$ 
onto $\bar{P} \, \cap \,]0, 1[$\,.
Consider $p \in P$ with $p < 1$ and find $f_p \in G_a$ with $\sigma_-(f_p) = p$.  
Then $\alpha_a(f_p) = \alpha(f_p)$ is affine 
on some interval $[\bar{a}, \bar{a} + \varepsilon_p]$ 
and has there slope $\bar{\sigma}_-(\alpha(f_p)) = \vartheta_a(p)$.  
Find $\delta_p > 0$ 
so that $f_p$ is affine on $[a, a + \delta_p]$ 
and that $\varphi(a + \delta_p) \leq \bar{a} + \varepsilon_p$.
The hypothesis that $\alpha$ is induced by conjugation by $\varphi$ 
then leads to the following chain of equalities
\begin{align*}
\varphi (a + p \cdot t)  
&=
(\varphi \circ f_p)(a + t) \\
&=
(\alpha(f_p) \circ \varphi)(a + t)
=
\alpha(f_p) (\varphi(a + t))\\
&=
\varphi(a) + \vartheta_a(p) \cdot (\varphi(a + t) - \varphi(a));
\end{align*}
it is valid for every  $t \in [0,  \delta_p]$.
Upon setting $\bar{a} = \varphi(a)$, 
the result of the above calculation takes on the form
\begin{equation}
\label{eq:Representation-varphi}
\varphi(a + p \cdot t) = \bar{a} + \vartheta_a(p) \cdot \left(\varphi(a + t) - \varphi(a)\right)
\quad\text{for } t \in [0,\delta_p].
\end{equation}
Since  $\varphi$ is increasing and $p < 1$,
this representation implies
that $\vartheta_a(p) < 1$.
It follows that $\vartheta_a(P\, \cap\, ]0,1[)  \subseteq \bar{P}\, \cap\, ]0,1[$
and then, by reversing the rôles of $G$ and $\bar{G}$, 
that $\vartheta_a(P\, \cap\, ]0,1[)  = \bar{P}\, \cap\, ]0,1[$.

This fact allows one to show next
that $\theta_a(p) = p^r$ for all $p \in P$ and for some positive real $r$. 
If $P$ is cyclic, this claim is clear,
for the generator $p \in  P\, \cap\, ]0,1[$ is sent to the generator 
$\bar{p} \in \bar{P}\, \cap\, ]0,1[$ 
and thus $\bar{p} = p ^r$ with $r = \ln \bar{p}/\ln p > 0$.
Otherwise, 
$P$ and $\bar P$ are dense subgroups of $\R^\times_{> 0}$ 
and $\theta_a \colon P \iso \bar{P}$ is an order preserving homomorphism.
Set $L_0 =  \ln \circ\ \tilde\theta_a\circ \exp \colon \ln{P} \iso \ln(\bar{P})$.
Then $L_0$ is an order preserving homomorphism which 
maps the dense subgroup $\ln(P)$ of $\R$ onto the dense subgroup $\ln(\bar{P})$ of $\R$. 
It follows
\footnote{see, \eg{}\cite[IV.4.2.1,  III.3.15.5 and III.3.15.3]{Die69}} 
that $L_0$ extends uniquely to an order preserving, continuous endomorphism  
$L \colon \R_{\add} \to \R_{\add}$.
The homomorphism $L$ is thus linear, given by multiplication by some real number $r > 0$,
whence $\vartheta \colon P \iso \bar{P}$ has the form 
$p \mapsto \exp{r \cdot \ln (p)} = p ^r$, as asserted.
\smallskip 

We come now to the second part of the proof.
In it we establish formula \eqref{eq:16.7} for a single element $p_1 < 1$.
Translations allow one to reduce to the case  where $a = 0 = \bar a$.  
With this reduction and the formula for $\vartheta_a $ just obtained,
representation \eqref{eq:Representation-varphi}  becomes for $p = p_1$:
\begin{equation}
\label{eq:16.8}
\varphi(p_1 \cdot t)= p_1^r  \cdot  \varphi (t) \;
\text{ for every }\; t\in [0,\delta_{p_1}].
\end{equation}

In the third part we show that formula \eqref{eq:16.8} holds not only for $p_1$,
but for every $p \in P \, \cap \, ]0,1]$.
Given $p < 1$,  one finds, as before, a real number $\delta > 0$ 
so that $\varphi(p\cdot t) = p^r\varphi(t)$ for all $t\in [0,\delta]$.  
Choose $m\in \N$ so large that $p_1^m \cdot \delta_{p_1} \leq \delta$.  
Then the following chain of equalities holds for each $t \in [0,\delta_{p_1}]$:
\begin{equation*}
p^{m\cdot r}_1 \cdot \varphi(p\cdot t) 
= 
\varphi(p_1^m\cdot p \cdot t) 
= 
\varphi( p \cdot p_1^m \cdot t) 
= 
p^r \cdot \varphi(p^m_1\cdot t)
= 
 p^r \cdot p^{m\cdot r}_1 \cdot \varphi(t).
\end{equation*}
This calculation shows
that $\varphi(p\cdot t) = p^r\cdot \varphi(t)$ for every $t\in [0,\varepsilon_1]$. 
Formula \eqref{eq:16.7} now follows by an affine change of coordinates  
(and setting $\delta = \delta_{p_1}$).  
\end{proof}

The second result shows 
that a mild additional hypothesis on $\varphi$ implies
that $\varphi$ is PL, 
that $\bar P = P$ 
and that $\bar A = s\cdot A$ for some positive real number $s$. 
\begin{proposition}
\label{PropositionE9} 
\index{Representation Theorem for isomorphisms!supplements}%
\index{Proposition \ref{PropositionE9}!statement|textbf}%
Assume there exists a point $a_0$ in $A\cap \Int(I)$ 
at which $\varphi$ has a one-sided derivative with non-zero value $s$.  
Then $\bar P = P$, $\bar A = s\cdot A$ and
$\varphi$ is $\PL$ with slopes in the coset $s\cdot P$ and singularities in $A$.  
Moreover, $\varphi$ has only finitely many singularities in every compact subinterval of $\Int(I)$.  
\end{proposition}

\begin{proof}  
Since $a_0$ is in $A$ and $\varphi(a_0)$ in $\bar{A}$, 
conjugations by affine homeomorphisms bring us to the special case 
where $a_0 = 0 = \varphi(a_0)$ and $\varphi$ is order preserving 
with right-hand derivative of value $s > 0$ at 0. 
Formula \eqref{eq:16.7} then simplifies to
\begin{equation*}
\varphi(p \cdot t) = p^r \cdot \varphi (t);
\end{equation*}
it holds for every $t \in [0,\delta]$ for some $\delta > 0$,
and $p \in P$ with $p< 1$.   

The hypothesis that the right-hand derivative of $\varphi$ exists in 0 and is positive
(and the fact that set $P\,  \cap \, ]0,1[$ is not empty) 
implies therefore that $r = 1$
and so \emph{$\varphi$ is linear on $[0,\delta]$ with slope $s$} 
and $\bar{P} = P$ (by Proposition \ref{PropositionE8}).
Moreover, $\bar A = s A$.
Indeed,  
the last statement in Theorem \ref{TheoremE04} guarantees
that $\varphi$ maps $A \cap \Int(I)$ onto $\bar{A}\cap  \Int(I)$
whence $s(A \,\cap \,[0, \delta]\,) = \bar{A}\, \cap \,[0, s \cdot \delta]$.
Next,
given a positive number $a \in A$, 
there exists a positive number $p \in P$  with $p \cdot a < \varepsilon$ 
and so 
\[
s \cdot a = p^{-1} \cdot (s \cdot p \cdot a) \in \bar{A}, 
\]
(because $p^{-1} \cdot \bar{A} = \bar{A}$).
Finally, if $\bar{a}$ is a positive element of $\bar{A}$
there exists $\bar{p} \in \bar{P} = P$ 
with $\bar{p} \cdot \bar{a} < s \cdot \delta$
and thus $ (\bar{p} \cdot \bar{a})/s$ is an element $a$, say, of $A$ 
and so $\bar{a} = s \cdot (a / \bar{p})$ lies in $s \cdot A$.

We are left with proving 
that $\varphi$ is a PL-homeomorphism  with the stated properties.
By the previous paragraph, there exists a positive real number $\varepsilon$
such that $\varphi$ is linear on the open interval  $]0, \delta[$\,.
Consider now an arbitrary point $ t_0 \in \Int(I)$.
By Remark \ref{remark:5.3} 
the group $B' = B(I;A,P)'$ acts 2-fold transitively 
on the intersection $(IP \cdot A) \cap \Int(I)$.
There exists therefore an element $f \in B'$ 
which maps an open neighbourhood $J_{t_0}$ of the given point 
$t_0$ into  $]0, \delta[$.
 
We claim that the restriction of $\varphi$ to $J_{t_0}$ is piecewise linear.
Indeed,
conjugation by $\varphi$ induces the isomorphism $\alpha \colon G \iso \bar{G}$.
It follows, in particular, 
that the relation $\varphi \circ f^{-1} \circ \varphi^{-1} = \alpha(f^{-1})$ is valid
or, equivalently,
that the formula  
\begin{equation}
\label{eq:Revealing-that-phi-is-PL}
\varphi  =  \alpha(f)^{-1} \circ \varphi \circ f
\end{equation}
 holds.
 The homeomorphisms $f$ and $\alpha(f)^{-1}$ are finitary piecewise linear
 with slopes in $P = \bar{P}$;
 in addition,  $f$ maps $J_{t_0}$ into an  interval 
 on which $\varphi$ is linear with slope $s$.
Formula \eqref{eq:Revealing-that-phi-is-PL} therefore proves 
that $\varphi$ is finitary piecewise linear on an open neighbourhood of $t_0$
 and that its slopes lie in $P^{-1} \cdot s \cdot P = s \cdot P$. 
\end{proof}
%
%
\section{Isomorphisms of groups with non-cyclic slope groups}
\label{sec:17}
%
In this section, Theorem \ref{TheoremE04} is applied in the situation 
where the group $P$  \emph{is not cyclic}.
The homeomorphisms $\varphi$ inducing isomorphisms $\alpha$ 
are then necessarily piecewise linear and they can be determined explicitly.
%
\subsection{The main results}
\label{ssec:17.1}
If $P$ is not cyclic, 
the conclusion of Proposition \ref{PropositionE8} can be sharpened 
so as to imply the hypothesis of Proposition \ref{PropositionE9}.  
In this way one obtains the following basic result:
\begin{thm}
\label{TheoremE10}
\index{Representation Theorem for isomorphisms!supplements}%
\index{Automorphisms of G(I;A,P)@Automorphisms of $G(I;A,P)$!properties}%
\index{Group P not cyclic@Group $P$ not cyclic!consequences}%
\index{Proposition \ref{PropositionE9}!consequences}%
\index{Theorem \ref{TheoremE10}!statement|textbf}%
Let $G$ be a subgroup of $G(I;A,P)$ containing $B(I;A,P)'$
and  $\bar{G}$ a subgroup of $G(\bar{I}; \bar{A}, \bar{P})$ with the analogous property. 
If there exists an isomorphism $\alpha \colon G \iso \bar{G}$ and if $P$ is not cyclic,
the following assertions hold:
\begin{enumerate} 
\item[(i)] $ \bar{P} = P$.
\item[(ii)] There exists a non-zero real number $s$ such that  $\bar{A} = s \cdot A$
and $\varphi$ is a PL-homeo\-mor\-phism 
which maps $A \cap \Int(I)$ onto $\bar{A} \cap \Int(\bar{I})$ 
and whose slopes lie in the coset $s \cdot P$.
Each compact subinterval of $\Int(I)$ contains only finitely many singularities of $\varphi$. 
\end{enumerate}
\end{thm}

\begin{proof}
By Theorem \ref{TheoremE04}, the isomorphism $\alpha$ is induced by a homeomorphism $\varphi$.  
Replacing, if need be, 
$\varphi$ by $(-\id) \circ \varphi \colon \Int(I) \iso  \Int(-\bar{I})$, 
we can reduce to the case where $\varphi$ is increasing.  
Choose $a_0 \in A\cap \Int(I)$ and set $\bar{a}_0 = \varphi(a_0)$.  
By Proposition \ref{PropositionE8} 
there exist positive real numbers  $r$ and $\delta$ 
so that the equation
\begin{equation}
\label{eq:17.1}
\varphi(a_0+ p\cdot \delta) 
= 
\bar{a}_0  +  p^r \cdot \left(\varphi (a_0 +\delta) - \bar{a}_0 \right)
\end{equation}
is valid for every $p \in  P\, \cap \,]0, 1]$.  
Since $P$ is not cyclic, hence dense in $\R_{> 0}$,  and $\varphi$ is continuous,
equation \eqref{eq:17.1} is valid for each $p \in \,]0,1[$ 
and so $\varphi$ is differentiable on $]a_0,a_0+\delta[$. 
The claim now follows from Proposition \ref{PropositionE9}
and the fact that $]a_0,a_0+\delta[$ contains a point of $A$.
\end{proof}

\subsubsection{Supplement to Theorem \ref{TheoremE10}}
\label{sssec:Supplement-TheoremE10}
%
Theorem \ref{TheoremE10} reveals 
that the hypothesis that $P$ is not cyclic 
forces the parameters $\bar{P}$ and $\bar{A}$ 
to be closely related to $P$ and $A$.
Indeed, one has then $\bar{P} = P$ and $\bar{A} = s \cdot A$ for some positive real $s$.
This is as good as one can expect,
for every homothety $\varphi \colon t \mapsto s \cdot t$ induces 
by conjugation an isomorphism $\alpha_s \colon G(I;A,P) \iso G(\varphi(I); s \cdot A, P)$
and this isomorphism yields, of course, also isomorphisms between subgroups.

We turn next to the question as to 
\emph{how closely related the intervals $I$ and $\bar{I}$ must be.}
The group of bounded PL-homeomorphisms $B(I;A,P)$ provides a telling negative answer,
for this group does not depend on the interval $I$ 
(up to isomorphism; see Proposition \ref{PropositionC1}).
\index{Proposition \ref{PropositionC1}!consequences}%
 A positive answer is provided by
 \begin{lemma}
 \label{lem:Isomorphisms-and-type-of-intervals-2}
 \index{Group G(R;A,P)@Group $G(\R;A,P)$!isomorphisms}%
 \index{Group G([0,infty[;A,P)@Group $G([0, \infty[\;;A,P)$!isomorphisms}%
 Suppose that $G = G(I;A,P)$, 
 that $\bar{G}$ is a subgroup of $G(\bar{I}; \bar{A}, \bar{P})$ 
 containing the derived group of $B(\bar{I}; \bar{A}, \bar{P})$
 and that $G$ is isomorphic to $\bar{G}$.
 Then the following implications hold:
 \begin{enumerate}[(i)]
 \item if $I = \R$ then $\bar{I} = \R$;
 \item if $I = [0,\infty[$
 then $\bar{I}$ is either a half line with endpoint in $\bar{A}$ or $\bar{I} = \R$;
 \item if $I =\, ]0,\infty[$
 then $\bar{I}$ is either a half line with endpoint in $\R \smallsetminus A$
 or a line.
 \end{enumerate}
 \end{lemma}
 
 \begin{proof}
Let $\alpha$ be in isomorphism of $G$ onto $\bar{G}$.
If $\alpha$ is not increasing 
then 
\[
\act{-\id}(-) \circ \alpha \colon G \iso \bar{G} \iso \act{-\id}\bar{G}
\] 
is an increasing isomorphism. 
We can thus assume that $\alpha$ is increasing.
If $I = \R$ both $\lambda$ and $\rho$ map $G$ onto a non-abelian, metabelian group;
 if $I = [0, \infty[$ the image of $\lambda = \sigma_-$ is a non-trivial abelian group
 and the image of $\rho$ is metabelian, but not abelian;
 if $I = \;]0, \infty[$ the image of $\rho$ is metabelian, but not abelian.
In all three cases, the claim thus follows from Corollary \ref{CorollaryE5}.
 \end{proof}
\index{Homomorphism!11-lambda@$\lambda$}%
\index{Homomorphism!17-rho@$\rho$}%
  
Another type of question is whether $\mu \colon G(I;A,P) \mono G(\bar{I};\bar{A},\bar{P})$
can be a monomorphism without being an isomorphism.
An answer is provided by
\begin{supplement}
\label{SupplementE11New}
\index{Representation Theorem for isomorphisms!supplements}%
\index{Group G([a,c];A,P)@Group $G([a,c];A,P)$!isomorphisms}%
\index{Group G([0,infty[;A,P)@Group $G([0, \infty[\;;A,P)$!isomorphisms}%
\index{Group G(R;A,P)@Group $G(\R;A,P)$!isomorphisms}%
\index{Automorphisms of G(R;A,P)@Automorphisms of $G(\R;A,P)$!properties}%
\index{Automorphisms of G([0,infty[;A,P)@Automorphisms of $G([0, \infty[\;;A,P)$!properties}%
\index{Automorphisms of G([a,c];A,P)@Automorphisms of $G([a,c];A,P)$!properties}%
\index{Group P not cyclic@Group $P$ not cyclic!consequences}%
\index{Proposition \ref{PropositionE9}!consequences}%
\index{Theorem \ref{TheoremE10}!consequences}
\index{Supplement \ref{SupplementE11New}!statement|textbf}%
\index{Group P not cyclic@Group $P$ not cyclic!consequences}%
Assume $P$ is \emph{not cyclic},
$\alpha \colon G(I;A, P) \iso \bar{G}$
is an isomorphism onto a subgroup $\bar{G}$ of $G(\bar{I}; \bar{A}, \bar{P})$
which contains the derived group of $B(\bar{I}; \bar{A}, \bar{P})$.
Let $\varphi \colon \Int(I) \iso \Int(\bar{I})$ 
be the unique homeomorphism inducing $\alpha$.

Assume, furthermore,
that $I$ is of type $\R$, or of type $[a, \infty[$ with $a \in A$, 
or of finite length with endpoints in $A$.
If $\bar{I}$ has the same type as $I$ 
then $\varphi$ is a finitary PL-homeomorphism and $\bar{G} = G(\bar{I}; \bar{A}, \bar{P})$.
\end{supplement}

\begin{proof}
The verification will consist of three parts.
Note first that
Theorem \ref{TheoremE10} applies to $\alpha$ 
\index{Theorem \ref{TheoremE10}!consequences}%
and so $\varphi$ is a PL-homeomorphism with slopes in a coset $s \cdot P$
and $\bar{A} = s \cdot A$.
In addition, $\bar{P} = P$ 
and each compact subinterval of $\Int(I)$ contains only finitely many singularities of $\varphi$.
By replacing, if need be $\bar{G}$ by $\act{-\id}\bar{G}$, 
we can arrange that $\alpha$ is increasing. 

\paragraph{Part 1: $I$ is bounded from below.}
\label{para:Part1-SupplementE11New}
%
So $I$ is either of type $[a,c]$ or of type $[a, \infty[$ with $a \in A$.
The homeomorphism $\varphi$ has then a unique continuous extension 
$\tilde{\varphi} $ which sends $a$ to the left end point $\bar{a}$ of $\bar{I}$.
It turns out 
that the reasoning in the proof of Proposition \ref{PropositionE8} 
can be repeated in the present situation.
So there exist positive real numbers  $\delta$ and $r$ with the property 
that the formula
\begin{equation}
\label{eq:16.7bis}
\tilde{\varphi}(a + p \cdot t) = \bar{a} + p^r \cdot \left(\tilde{\varphi}( a + t) - \bar{a}\right)
\end{equation}
holds for every $t \in [0, \delta]$ and each $p \in P\, \cap \, ]0,1]$.
Next we use the hypothesis that $P$ is not cyclic.
It implies that $P\, \cap \, ]0,1]$ is dense in the unit interval $[0,1]$.
Since $\varphi$ and $t \mapsto t^r$ are orientation preserving and continuous,
the representation \eqref{eq:16.7bis} holds therefore for every $p \in \,]0,1[$
and shows that $\varphi$ is differentiable 
on the open interval $]a, a + \delta[$.
Proposition \ref{PropositionE9} thus applies and guarantees 
that $\varphi$ is piecewise linear on $]a, a + \delta[$.
Therefore $r = 1$ and so $\varphi$ is affine on $[a, a + \delta]$.
%
\paragraph{Part 2: $I$ is not bounded from above.}
\label{para:Part2-SupplementE11New}
%
The quotient $G/\ker \rho \iso \im \rho \subseteq A \rtimes P$ 
is then a metabelian, but non-abelian group; 
the derived group of $\im \rho$ consists of translations with amplitudes in $IP \cdot A$ if $I = \R$,
respectively in  $IP^2 \cdot A$ if $I = [a, \infty[\,$.
In both cases, 
this derived group is a non-trivial $\Z[P]$-submodule of $\R_{\add}$,
hence a dense subgroup of $\R_{\add}$.
Now, 
according to part (iii) of Corollary \ref{CorollaryE5} 
the isomorphism $\alpha$, being increasing,
induces an isomorphism of $G/\ker \rho$  onto $\bar{G}/(\bar{G} \cap \ker \bar{\rho})$.
If follows 
that $\alpha$ induces an isomorphism
\[
\alpha_2 \colon \im(\rho)' \iso  (\im \bar{\rho} \restriction{\bar{G}})';
\]
it maps the amplitude of a translation $f$ to the amplitude of the translation $\rho(\alpha(f))$.

Our next aim is to imitate the proof of Proposition \ref{PropositionE8}
in the present additive setting. 
We claim, first, 
that the isomorphism \emph{$\alpha_2$ is strictly increasing}.
Let $b \in IP^2 \cdot A$ be an arbitrary  positive element 
and  let $f_b \in G$ be a translation with amplitude $b$.
Then $\alpha(f_b) \in \bar{G}$ is a PL-homeomorphism 
which is a translation with amplitude $\alpha_2(b)$ for large enough values of $t$, 
say for $t \geq \bar{t}_b$.
Set $t_b = \varphi^{-1}(\bar{t}_b)$.
Since $\alpha$ is induced by conjugation by $\varphi$,
one has $\alpha(f_b) = \act{\varphi} f_b$; 
so $\varphi \circ f_b = \alpha(f_b) \circ  \varphi$.
By evaluating this equality at $t \geq t_b$ one obtains the chain of equations
\[
\varphi(t + b) = (\varphi \circ f_b)(t) = (\alpha(f_b) \circ \varphi)(t) = \alpha_2(b) + \varphi(t).
\]
It implies that $\alpha_2(b)$ is positive, 
for $b$ is so by assumption and $\varphi$ is increasing.

We show next
that $\alpha_2$ is given by multiplication by a positive real number $s_2$.
Let $A_2$ be the group of amplitudes of the translations in $(\im \rho)'$.
Then $A_2$ contains the submodule $IP^2 \cdot A$, 
hence arbitrary small positive numbers. 
It follows that $A_2$ is not cyclic and so dense in $\R$.
The isomorphism $\alpha_2$ maps $A_2$  
injectively onto a group $\bar{A}_2$ of amplitudes contained in $\bar{A}$.
Its image $\bar{A}_2$ is not cyclic, and hence dense in $\R$.
So $A_2$ and $\bar{A}_2$  are both dense subgroups of $\R_{\add}$;
as $\alpha_2$ is strictly increasing  
it extends therefore to a (unique) strictly increasing monomorphism  
$\tilde{\alpha}_2 \colon \R \mono \R$.
The extended monomorphism is continuous and hence an $\R$-linear map,
given by multiplication by some positive real $s_2$.

We come now to the third stage of the analysis of $\varphi$.
In it we show 
that the restriction of $\varphi$ to a suitable interval of the form $[t_*, \infty[$ is \emph{affine}.
Choose a positive element $b_* \in IP^2 \cdot A$ 
and let $f_{b_*} \in G$ be a translation with amplitude $b_*$.
It then follows, as in the above,
that there is a number $t_* \in \Int(I)$ 
so that the equation
\begin{equation}
\label{eq;Describing-varphi}
\varphi(t + b_*) = \alpha_2(b_*) + \varphi(t) =  \varphi(t) + s_2 \cdot b_*
\end{equation} 
holds for every $t \geq t_*$.
Consider now an arbitrary positive element $b \in IP^2 \cdot A$.
Then there exists  a number $t_{b} \in \Int(I)$ 
such that
\[
\varphi(t + b) = \alpha_2(b) + \varphi(t) = \varphi(t) +  s_2 \cdot b  \text{ for } t \geq t_{b}.
\]
Choose a positive integer $m$ which is so large that $t_{b} \leq t_* + m \cdot b_*$.
For every $t \geq t_*$ the following calculation is then valid:
\begin{align*}
 \varphi(t + b) +s_2 \cdot m b_*
&=
\varphi( t +b +  m \cdot b_* )\\
&=
\varphi( t +  m \cdot b_*) + s_2 \cdot b
= 
\varphi(t) + s_2 \cdot m b_* +s_2 \cdot b.
\end{align*}
It follows, in particular,
that the equation
\begin{equation}
\label{eq:Representation-.varphi-bis}
\varphi (t_* + b) = \varphi (t_*) + s_2 \cdot b
\end{equation}
holds for every positive element $b \in IP^2 \cdot A$ and $t \geq t_*$.
Since $\varphi$ is continuous and $IP^2 \cdot A$ is dense in $\R$,
this equation allows us to deduce 
that $\varphi$ is affine with slope $s_2$ on the half line $[t_*, \infty[$.

%
\paragraph{Part 3: conclusion of the proof.}
\label{para:Part3-SupplementE11New}
%
Having at one's disposal the analysis of $\varphi$
for $I$ an interval that is bounded from below, carried out in Part 1,
and for $I$ an interval that is not bounded from above, given in Part 2,
it is easy to establish that $\varphi$ is a finitary PL-homeomorphism.
Recall that \emph{$\varphi$ can be assumed to be increasing}.

Suppose first that $I$ and $\bar{I}$ are \emph{half lines} with endpoints in $A$,
say $I = [a, \infty[$ and $\bar{I} = [\bar{a}, \infty]\,$.
By part 1 
there exist a positive real $\delta$ with the property
that  $\varphi$ is affine on the interval $[a, a + \delta]$;
by part 2 there exist a number $t_* \in \Int(I)$ 
so that $\varphi$ is affine on the half line $[t_*, \infty[$;
we may assume that $a + \delta \leq t_*$. 
As $\varphi$ has only finitely many singularities in  $[a + \delta, t_*]$
(by claim (iii) in Theorem \ref{TheoremE10}) it is finitary piecewise linear.
\index{Theorem \ref{TheoremE10}!consequences}%

Assume next that $I$ and $\bar{I}$ 
are \emph{compact intervals} with endpoints in $A$,
say $I = [a, c]$ and $\bar{I} = [\bar{a}, \bar{c}]$.
By part 1 there exist a positive real $\delta_\ell$ 
so that $\varphi$ is affine on the interval $[a, a + \delta_\ell]$.
Replace now $I$ and $\bar{I}$ by $-I$ and $-\bar{I}$, 
and $\varphi$ by $(-\id) \circ \varphi \circ (-\id)$.
It then follows that there exists a positive real $\delta_r$ 
so that $\varphi$ is affine on the interval $]c - \delta_r, c[$.
We conclude, as in the previous case, 
that $\varphi$ is a finitary PL-homeomorphism.

Thirdly, assume $I$ and $\bar{I}$ are lines.
By part 2 there exists a positive real $t_r$ 
so that $\varphi$ is affine on $]t_r, \infty[$.
By replacing $\varphi$ by $(-\id) \circ \varphi \circ (-\id)$ 
one finds that there exists $t_\ell$ so that $\varphi$ is affine on $]-\infty, t_\ell[$.
It follows, as before,  that $\varphi$ is finitary PL.

We are left with showing that $\bar{G}$ is all of $G(\bar{I}; \bar{A}, \bar{P})$.
This follows from the fact that $\varphi$ is finitary PL.
Indeed,
by Theorem \ref{TheoremE10} we know that $\bar{P} = P$,
\index{Theorem \ref{TheoremE10}!consequences}
that the slopes of $\varphi$ lie in a single coset $s \cdot P$ 
and that $\bar{A} = s \cdot A$.
Since, by the preceding analysis, $\varphi$ is a \emph{finitary} PL-homeomorphism  
it induces by conjugation an isomorphism 
that maps $G = G(I;A,P)$ onto all of $G(\bar{I};s \cdot A, P)$.
Since $\bar{G} $ is known to be the image of the isomorphism induced by conjugation by $\varphi$,
we see that $\bar{G} = G(\bar{I};s \cdot A, P) = G(\bar{I}; \bar{A}, \bar{P})$.
The proof is now complete.
\end{proof}

\begin{remarks}
\label{remarks:SupplementE11New}
(i) If $I = \R$
then $\bar{I} = \R$  by claim (i) in Lemma \ref{lem:Isomorphisms-and-type-of-intervals-2}.
The analysis of $\varphi$ carried out in the above 
uses then only the analysis of Part 2 in the proof.
This part does not need the hypothesis that $P$ be non-cyclic,
for $IP^2 \cdot A$ is a non-trivial $\Z[P]$-module and hence dense in $\R_{\add}$ 
even in case $P$ is cyclic. 
It follows that $\varphi$ is affine on some interval $J$ of the form $[t_*, \infty[\,$,
so differentiable on $J$,
whence $P = \bar{P}$ and $\varphi$ is piecewise affine 
with slopes in a coset $s \cdot P$  by Proposition \ref{PropositionE9}.
Supplement \ref{SupplementE11New} has therefore the pleasant
\begin{corollary}
\label{crl:Isomorphism-I-line}
\index{Representation Theorem for isomorphisms!supplements}%
\index{Group G(R;A,P)@Group $G(\R;A,P)$!isomorphisms}%
\index{Proposition \ref{PropositionE9}!consequences}%
\index{Supplement \ref{SupplementE11New}!consequences}%
Assume $\alpha \colon G(\R;A, P) \iso \bar{G}$
is an isomorphism onto a subgroup $\bar{G}$ of $G(\bar{I}; \bar{A}, \bar{P})$
that contains the derived group of $B(\bar{I}; \bar{A}, \bar{P})$.
Then $\bar{I} = \R$ and $\alpha$ is induced 
by conjugation by a \emph{finitary} PL-homeomorphisms with slopes in a coset $s \cdot P$.
Moreover, $\bar{P} = P$, $\bar{A} = s \cdot A$ and $\bar{G} = G(\R; s \cdot A, P)$.
\end{corollary}

(ii) The hypotheses of Supplement \ref{SupplementE11New} require
\index{Supplement \ref{SupplementE11New}!discussion}%
that the intervals $I$ and $\bar{I}$ have the same type. 
As we have seen in the preceding part (i),
this requirement can be dispensed with if $I$ is a line.
I do not know if the requirement is also superfluous 
in case $I$ is a half line or an interval of finite length.
\end{remarks}
%
\subsection{Comparison of groups having intervals of finite length}
\label{ssec:17.2}
The question 
when two groups of the form $G(I;A,P)$ and $G(\bar I;A,P)$ are isomorphic 
has been discussed briefly in section \ref{ssec:2.4} 
and, more fully, in section \ref{ssec:16.4}; 
the answer is that they are isomorphic if, and only if, 
the intervals $I$ and $\bar I$ belong to the same type, 
with the possible exception that groups $G([0,b_1];A,P)$ and
$G([0,b_2];A,P)$ need not be isomorphic if $P$ is not cyclic. 
In this brief section we study the exceptional case.

Let $b_1$, $b_2$ be positive elements of $A$,
and let $G_1$, $G_2$ be the groups $G([0,b_1]; A, P)$ and $G([0,b_2]; A, P)$, respectively.
Assume $P$ is not cyclic.
Suppose first $G_1$ and $G_2$ are isomorphic
and let $\alpha \colon G_1 \iso G_2$ be an isomorphism.
Since conjugation by the reflection $t \mapsto b_2-t$ maps $G_2$ onto itself,
we may and shall assume that $\alpha$ is increasing.
By Supplement \ref{SupplementE11New} 
\index{Supplement \ref{SupplementE11New}!consequences}%
there exists then a finitary PL-homeomorphism 
$\varphi \colon ]0, b_1[ \, \iso \, ]0, b_2[$ 
that induces $\alpha$ by conjugation and has slopes in a coset $s \cdot P$.
The real $s$ is positive and $s \cdot A = A$ (by part (ii) of Theorem \ref{TheoremE10}).
\index{Theorem \ref{TheoremE10}!consequences}
It follows that $\varphi$ is the composition of the homothety $s \cdot \id$, 
mapping $[0,b_1]$ onto $[0, s \cdot b_1]$ 
and a PL-homeomorphism $f \in G(\R; A, P)$ 
which maps the interval $[0, s\cdot b_1]$ onto the interval $[0, b_2]$.
By Theorem \ref{TheoremA}, 
\index{Theorem \ref{TheoremA}!consequences}%
the difference $s \cdot b_1- b_2$ lies therefore in $IP \cdot A$.
Conversely, assume there exists a positive real number $s \in \Aut_o(A)$ 
so that $s \cdot b_1 - b_2 \in IP \cdot A$.
By Theorem \ref{TheoremA} 
\index{Theorem \ref{TheoremA}!consequences}%
there exists then $f \in G(\R; A, P)$ with $f([0, s \cdot b_1]) = [0, b_2]$
and conjugation by $f \circ (s \cdot \id)$ induces an isomorphisms of $G_1$ onto $G_2$.
We conclude that $G_1$ and $G_2$ are isomorphic if, and only if, 
there exists $s \in \Aut_o(A)$ so that $s \cdot b_1 - b_2 \in IP \cdot A$.
Now $IP \cdot A$ is invariant under multiplication by $\Aut_o(A)$
and so the abelian group $A/(IP\cdot A)$ has a canonical $\Aut_o(A)$-action.
The previous finding can therefore be summarized as follows:
\begin{corollary}
\label{CorollaryE12}
\index{Group G([a,c];A,P)@Group $G([a,c];A,P)$!isomorphisms}%
\index{Subgroup Auto(A)@Subgroup $\Aut_o(A)$!significance}%
\index{Group G([a,c];A,P)@Group $G([a,c];A,P)$!dependence on [a,c]@dependence on $[a,c]$}%
\index{Group P not cyclic@Group $P$ not cyclic!consequences}%
Let $\Aut_o(A)$ be the group $\{s \in \R^\times_{> 0}\mid s\cdot A = A\}$ 
and let $b_1$, $b_2$ be positive elements of $A$.
If $P$ is not cyclic, 
the groups $G([0,b_1];A,P)$ and $G([0,b_2];A,P)$ are isomorphic if, and only if,
$b_1 + IP \cdot A$ and $b_2 + IP \cdot A$ lie in the same orbit of $\Aut_o(A)$.
\end{corollary}

An illustration of Corollary \ref{CorollaryE12} will be given
in section \ref{sssec:Automorphism-groups-some-examples-non-cyclic-P}.

%
\subsection{Automorphism groups}
\label{ssec:17.3}
In this section, 
we study the implications of Theorem \ref{TheoremE10} 
and its Supplement \ref{SupplementE11New} 
\index{Supplement \ref{SupplementE11New}!consequences}%
for the automorphism groups $\Aut G$  of the groups of $G = G(I;A,P)$,
\emph{assuming that the slope groups $P$ are non-cyclic}.
\label{notation:AutG}

%
\subsubsection{Basic result}
\label{sssec:Automorphism-Result-non-cyclic}
We begin by introducing an auxiliary homomorphism
that will help us in describing the outer automorphism groups.
Let $G$ be a subgroup of $G(I;A,P)$ 
that contains the derived group of $B(I;A,P)$. 
If $P$ is non-cyclic,
Theorem \ref{TheoremE10} applies
\index{Theorem \ref{TheoremE10}!consequences}
and shows that every automorphism $\alpha$ of $G$ is induced 
by conjugation by a unique PL-automorphism $\varphi_\alpha$ 
whose slopes lie in a coset $s \cdot P$ of $P$
with  $s$ belonging to the group
\begin{equation}
\label{eq:Definition-Aut(A)}
\Aut(A) = \{s\in \R^\times \mid s \cdot A= A\}
\end{equation}
The assignment $\alpha \mapsto \varphi_\alpha \mapsto s\cdot P$ 
yields then a homomorphism.
\begin{equation}
\label{eq:17.2}
\index{Homomorphism!07-eta@$\eta$}%
\index{Group Aut(A)@Group $\Aut(A)$!significance}%
\eta \colon \Aut G  \to \Aut(A)/P.
\end{equation}

Henceforth we shall assume that  $G$ is the group $G(I;A,P)$.
The kernel and the image of $\eta$ can then be described as follows.
\begin{corollary}
\label{CorollaryE13}
\index{Group Aut(A)@Group $\Aut(A)$!significance|(}%
\index{Subgroup Auto(A)@Subgroup $\Aut_o(A)$!significance}%
\index{Group G(R;A,P)@Group $G(\R;A,P)$!automorphism group|(}%
\index{Group G([0,infty[;A,P)@Group $G([0, \infty[\;;A,P)$!automorphism group|(}\index{Group G([a,c];A,P)@Group $G([a,c];A,P)$!automorphism group|(}%
\index{Automorphism group of G(R;A,P)@Automorphism group of $G(\R;A,P)$!description|(}%
\index{Automorphism group of G([0,infty[;A,P)@Automorphism group of $G([0, \infty[\;;A,P)$!description|(}%
\index{Automorphism group of G([a,c];A,P)@Automorphism group of $G([a,c];A,P)$!description|(}%
\index{Homomorphism!07-eta@$\eta$}%
\index{Subgroup Qb@Subgroup $Q_b$!significance}%
\index{Group P not cyclic@Group $P$ not cyclic!consequences}%
\index{Theorem \ref{TheoremE10}!consequences}%
\index{Supplement \ref{SupplementE11New}!consequences}%
Assume $P$ is non-cyclic and $G = G(I;A,P)$.  
Then the following statements hold.
\begin{enumerate}[(i)]
\item  If $I = \R$, or $I = [a,\infty[$ with $a \in A$, 
or $I = [a,c]$ with $(a, c) \in A^2$,
then $\ker \eta$ is the group of inner automorphisms 
and $\im \eta$ is the group of outer automorphisms of $G$.
\index{Outer automorphism group of!G(R;A,P)@$G(\R;A,P)$}%
\index{Outer automorphism group of!G([0,infty[;A,P)@$G([0, \infty[\;;A,P)$}%
\index{Outer automorphism group of!G([a,c];A,P)@$G([a,c];A,P)$}%
\index{Group P not cyclic@Group $P$ not cyclic!consequences}%
\item  If $I = \;]a, \infty[$, 
or $I =\;]a, c]$ with $c \in A$,
or $ I =\;]a, c[$,
then $\ker \eta$ consists of all automorphism $\alpha$
which are induced by PL-auto-homeo\-mor\-phisms $\varphi_\alpha$ of $\Int(I)$ 
with slopes in $P$ and vertices in $A^2$ 
and whose sets of singularities have the following property: 
\begin{enumerate}
\item if $I\, = \, ]a, \infty[$ or $I = \,]a, c]$ 
there are only finitely many singularities above every $b > a$;
\item if $I\, =\; ]a, c[$ the singularities accumulate only in the endpoints.
\end{enumerate}
\item if $I = \R$ or $I = \;]a,c[$ then $ \im\eta =\Aut(A)/P$; 
\item if $I = [a,\infty[\,$ with $a \in A$, 
or $I =\;]a,\infty[\,$,
or $I =\; ]a,c]$ with $c \in A$,
then $\im \eta$ is the subgroup $\Aut_o(A)/P$ of $\Aut(A)/P$ having index 2;
\item if $I$ equals $[a,c]$ with $(a, c) \in A^2$ 
then the image of $\eta$ may depend on $b = c-a$ and it is the subgroup
\begin{equation}
\label{eq:17.3}
Q_b/P \text{ with } Q_b = \{s\in \Aut(A) \mid (|s| -1)b \in IP\cdot A\}.
\end{equation}
\end{enumerate}
\index{Group G(R;A,P)@Group $G(\R;A,P)$!automorphism group|)}%
\index{Group G([0,infty[;A,P)@Group $G([0, \infty[\;;A,P)$!automorphism group|)}\index{Group G([a,c];A,P)@Group $G([a,c];A,P)$!automorphism group|)}%
\index{Theorem \ref{TheoremE10}!consequences}%
\index{Supplement \ref{SupplementE11New}!consequences}%
\index{Group Aut(A)@Group $\Aut(A)$!significance|)}%
\index{Group P not cyclic@Group $P$ not cyclic!consequences}%
\end{corollary}

\begin{proof}
(i)  is a direct consequence of Theorem \ref{TheoremE10}
and Supplement \ref{SupplementE11New} 
(and of the definition of $\ker \eta$).
Now to (ii).  
Suppose $\alpha$ is an automorphism of $G$.
By Theorem \ref{TheoremE10} 
it is then induced by a unique  PL-auto-homeomorphism $\varphi_1$ 
with slopes in $P$ and vertices in $A^2$;
in addition, 
every compact subinterval contains only finitely many singularities of $\varphi$.
It follows, in particular, 
that statement (ii)a) holds for $I = \;]a, c]$  
and that (ii)b) is valid for $I = \;]a, c[$.
Consider now an interval of the form $]a, \infty[$.
Then the argument carried out in Part 2 of the proof of Supplement \ref{SupplementE11New} applies 
and shows 
that $\varphi_\alpha$ is affine on a subinterval of the form $[t_*, \infty]$,
and so statement a) holds.

So far we know 
that every automorphism of $\ker \eta$ is induced by an auto-homeo\-mor\-phism 
with properties as described in statement (ii).
Conversely, these properties ensure 
that conjugation by such a PL-homeomorphism induces an automorphism of $G(I;A,P)$ that lies in  $\ker \eta$.
\smallskip

(iii) Assume first that $I = \R$. 
Then each element $s \in \Aut(A)$ gives rise to a homothety  $\vartheta_s \colon t\mapsto s\cdot t$ 
which normalizes $G(\R;A,P)$ and also $B(\R;A,P)$;
let $\alpha_s$ the induced automorphism of $G(\R;A, P)$.
Clearly, $\eta(\alpha_s) = s \cdot P$.

Consider now an open bounded interval $I =\; ]a,c[\,$.
Then the group $G(I; A, P)$ coincides with $ B(I;A,P)$ (by the definition of $G(I;A,P)$; 
see section \ref{ssec:16.4}).
\index{Proposition \ref{PropositionC1}!consequences}%
Next, the proof of Proposition \ref{PropositionC1} provides one
with an infinitary PL-homeomorphism $\varphi \colon \R \iso I$, 
with slopes in $P$, that induces,  by conjugation,  
the isomorphism  $\alpha_\varphi \colon B(\R;A, P) \iso  B(I;A,P)$. 
The composition $\psi = \varphi \circ \vartheta_s \circ \varphi^{-1}$ 
is then an auto-homeo\-mor\-phism of $I$ 
that is PL and has slopes in the coset $s \cdot P$.
Conjugation by $\psi$, finally,  induces the automorphism 
$\beta_s = \act{\alpha_\varphi} \alpha_s$ of the group $B(I;A,P)$ 
and $\eta(\beta_s) = s \cdot P$.
\smallskip

(iv)  If $I$ is either $[0,\infty[$ or $]0,\infty[$,
the images of $\sigma_-$ and  of $\rho$ are not isomorphic 
whence each isomorphism of $G = G(I;A,P)$ must be increasing.  
If $I = ]a, c]$ with $c \in A$ 
the image of $\sigma_-$ is trivial and that of $\sigma_+$ is $P$, 
and so every automorphism is increasing.
It follows in all these cases 
that the image of $\eta$ is contained in the subgroup  $\Aut_o(A)/P$
represented by the positive elements in $\Aut(A)$.

Fix now $s \in \Aut_o(A)$. 
If $I$ is either $[a, \infty[$ or $]a, \infty[$, 
there exists an affine map with slope $s$ that fixes $a$ and maps $I$ onto itself.
The induced automorphism then proves that $s \cdot P \in \im \eta$.
If, on the other hand, 
$I$ is the half-open interval $]a,c]$ 
pick a positive element $a_0 \in A$
and construct strictly decreasing sequences 
$(c_j \mid j \in \N)$ and $(c'_j \mid j \in \N)$ 
with the following properties:
\begin{itemize}
\item $c_0 =  c'_0 = c$ and $\lim_{j \to \infty} c_j= \lim_{j \to \infty} c'_j = a$;
\item $c_j - c_{j+1} \in P \cdot a_0$ 
and $c'_j - c'_{j+1} \in s \cdot P \cdot a_0$ for every $j \in \N$.
\end{itemize}
The affine interpolation of the sequence of points $(c_j, c'_j) \in A^2$ 
is then a PL-homeomorphism with slopes in $s \cdot P$ 
and it induces by conjugation an automorphism of the group $G(\,]a,c];A,P)$.
\smallskip

(v) Assume, finally, that  $I = [a,c]$. 
The affine map $t\mapsto a + c -t$ maps $I$ onto itself 
and induces by conjugation an order reserving automorphism; 
so it suffices to consider increasing automorphisms $\alpha$ of $G([a,c];A,P)$.  
If $\alpha$ is induced by $\varphi$, 
it follows, as in the proof of Corollary \ref{CorollaryE12}, 
that $\varphi$ is the composition of a homothety 
$\vartheta_s \colon \,]a, c[\, \iso\, ]s \cdot a,s \cdot c[$, 
taking $t$ to $s\cdot t$, 
and the restriction of an element $f$ in $G(\R;A,P)$ with $f([s \cdot a,s \cdot c]) = [a,c]$. 
By Theorem \ref{TheoremA} the difference 
\index{Theorem \ref{TheoremA}!consequences}%
$s \cdot (c-a) - (c-a) = (s - 1 )b$ lies therefore in $IP \cdot A$.  
The preceding argument can be reversed
and so we have established assertion (v).
\end{proof}

Corollary \ref{CorollaryE13} details the outer automorphism groups 
of the groups $G(I;A,P)$ with $P$ non-cyclic and $I$ the line, a closed half line, 
or an interval $[a, c]$ with $(a,c) \in A^2$.
The next result describes the full automorphism groups of these groups. 
\begin{corollary}
\label{crl:Explicit-description-of-Aut(G)}
\index{Subgroup Auto(A)@Subgroup $\Aut_o(A)$!significance}%
\index{Group G(R;A,P)@Group $G(\R;A,P)$!automorphism group}%
\index{Group G([0,infty[;A,P)@Group $G([0, \infty[\;;A,P)$!automorphism group}%
\index{Group G([a,c];A,P)@Group $G([a,c];A,P)$!automorphism group}%
\index{Automorphism group of G(R;A,P)@Automorphism group of $G(\R;A,P)$!description|)}%
\index{Automorphism group of G(R;A,P)@Automorphism group of $G(\R;A,P)$!properties}%
\index{Automorphism group of G([0,infty[;A,P)@Automorphism group of $G([0, \infty[\;;A,P)$!description|)}%
\index{Automorphism group of G([0,infty[;A,P)@Automorphism group of $G([0, \infty[\;;A,P)$!properties}%
\index{Automorphism group of G([a,c];A,P)@Automorphism group of $G([a,c];A,P)$!description|)}%
\index{Automorphism group of G([a,c];A,P)@Automorphism group of $G([a,c];A,P)$!properties}%
\index{Group Auto G(I;A,P)@Group $\Aut_o G(I;A,P)$!properties}%
\index{Theorem \ref{TheoremE04}!consequences}%
\index{Supplement \ref{SupplementE11New}!consequences}%
Suppose $P$ is non-cyclic 
and $I$ is one of the closed intervals 
$\R$, $[0, \infty[$ or $[a, c]$ with $(a,c) \in A^2$.
Set $G = G(I;A,P)$ 
and let $\Aut_o G $ be the group consisting of all increasing automorphisms of $G$. 
Then $\Aut_o G $ has index at most 2 in $\Aut G$ 
and it is isomorphic to the subgroup of $G(I;A, \Aut_o(A))$ 
consisting of all finitary PL-homeomorphisms $f$  
with slopes in a single coset $s_f \cdot P$ of $\Aut_o(A)/P$.
It follows, in particular, 
that $\Aut_o G $ is locally indicable 
and that it contains no free subgroups of rank 2.
\end{corollary}

\begin{proof}
Each automorphism $\alpha$ of $G$ is induced 
by conjugation by a \emph{unique} homeomorphism $\varphi_\alpha$ of $\Int(I)$
(see Theorem \ref{TheoremE04}).
This homeomorphism is a finitary PL-homeomorphism 
with slopes in a coset of $\Aut(A)/P$ and maps $A \cap \Int(I)$ onto itself
(by Supplement \ref{SupplementE11New}).
\index{Supplement \ref{SupplementE11New}!consequences}%
If $\alpha$ is increasing, 
the slopes of $\varphi_\alpha$ are in $\Aut_o(A)/P$ and so $\varphi_\alpha $
is an element of $\tilde{G} = G(I;A, \Aut_o(A))$.
Conversely, 
suppose $\tilde{g} \in \tilde{G}$ is an element whose slopes lie in a single coset $s \cdot P$. 
Then conjugation by $\tilde{g}$ maps $G$ onto the subgroup of all finitary PL-homeomorphisms 
which map $A \cap \Int(I) $ onto $s\cdot A \cap \Int(I)$ and have slopes in $P$;
since $s \cdot A = A$ the equality $\act{\tilde{g}} G = G$ holds.

The previous reasoning implies, in particular, 
that $\Aut_o G $ is isomorphic to a subgroup of $G(\R;\R, \R^\times_{>0})$ 
hence it is locally indicable by Remark \ref{remarks:Definitions-lambda-and-rho}(i)
and contains no non-abelian free subgroups by \cite[Thm.\,3.1]{BrSq85}.
\end{proof}
\index{Group G(R;A,P)@Group $G(\R;A,P)$!automorphism group}%
\index{Group G([0,infty[;A,P)@Group $G([0, \infty[\;;A,P)$!automorphism group}%
\index{Group G([a,c];A,P)@Group $G([a,c];A,P)$!automorphism group}%
\index{Local indicability!examples}%
%
\subsubsection{Some examples}
\label{sssec:Automorphism-groups-some-examples-non-cyclic-P}
\index{Group Aut(A)@Group $\Aut(A)$!examples|(}%
%
(i) We begin with specimens of \emph{automorphism groups $\Aut(A)$ 
of modules $A$}.
The group $\Aut(A)$ can be far larger than $P$,
as is brought home by the following construction.
Let $Q$ be a subgroup of $\R^\times_{>0}$ which contains the given group $P$.
Set $A =\Z[Q]$ and view $A$ as a $\Z[P]$-module.
Then $\Aut_o (A)$ contains $Q$, but it may actually be larger.
Specific examples of this construction are obtained
by choosing  $Q = \R^\times_{>0}$ and $A = \Z[Q] = \R$
or by selecting $Q = \Q^\times_{>0}$ and $A = \Z[Q] = \Q$.

In the previous construction,
$A$ is the additive group of a the ring generated by a super group $Q$ of $P$
and so $A$ is typically a non-cyclic $\Z[P]$-module.
But there exist also examples with $\Aut_o(A)$ larger than $P$ 
where $A$ is a cyclic $\Z[P]$-module.  
Consider, for instance, a subring $R \subset \R$ of the form $\Z[P]$
and let $A$ be the additive group of $R$.
The automorphism group $\Aut(A)$ is then nothing 
but the \emph{group of units $U(R)$ of the ring $R$}.
This group of units can be larger than $P$,
as is shown by the examples that follow.

a)  Let $p_1$, \ldots, $p_k$  be integers greater than 1.
Define $P$ to be the group generated by $p_1$, \ldots, $p_k$
and set $R = \Z[P]$.
Then $R$ is the ring $\Z[1/(p_1 \cdots p_k)]$.
Let $q_1$, \ldots, $q_\ell$ be the list of prime numbers 
dividing the product $p_1 \cdots p_k$.
We claim that
\begin{equation}
\label{eq:Determination-Aut(A)-1}
\Aut_o(A) = U(R) \cap \R_{>0} =\gp(q_1,\ldots, q_\ell).
\end{equation}
Indeed, 
suppose that $s  \in \Aut(A)$. 
Then $s \cdot 1$ and $s^{-1} \cdot 1$ both belong to $A$ 
and so there exist positive integers  $n_+$, $n_-$ and exponents $e_+$,  $e_-$ 
such that
\[
s = n_+/ (q_1 \cdots q_\ell)^{e_+} 
\quad \text{and} \quad 
s^{-1} = n_-/ (q_1 \cdots q_\ell)^{e_-}.
\]
The equation $s \cdot s^{-1} = 1$ gives next
that $n_+ \cdot n_- = (q_1 \cdots q_\ell) ^{e_+ e_-}$
and so $n_+$ is a divisor of a power of $(q_1 \cdots q_\ell)$;
the uniqueness of the prime factorization in $\N$ then implies
that $n_+$ is a product of powers of the various primes $q_j$ 
and so $s \in \gp(q_1, \ldots, q_\ell)$.
The preceding argument shows that $\Aut_o(A) \subseteq \gp(p_1, \ldots,p_\ell)$;
the reverse inclusion is clear.  

b) In the preceding example, 
the ring $R = \Z[P]$ is a localization of $\Z$ 
and so $A$ is infinitely generated as as an abelian group
(recall that $P$ is assumed to be non-trivial). 
Subrings $R$ of $\R$
whose additive group is finitely generated, hence free abelian,
are rings made up of algebraic integers. 
If $P$ is generated by finitely many algebraic integers that are units, 
the additive group of $R = \Z[P]$ is finitely generated
and $R_{\Q} = \Q[P]$ is an algebraic number field;
let $\OO$ be its ring of integers.
Dirichlet's Unit Theorem
(see, \eg{}\cite[Theorem 13.1.1]{AlWi04}) then guarantees 
that the group of units of $\OO$ 
is the direct product of a free abelian group of finite rank  and the cyclic group $\{1,-1\}$.
It follows that the group of units $U(R)$ of $R$, 
being a subgroup of $U(\OO)$ and containing $-1$,
has the same form.

Here is an explicit example.
Let $K$ be the algebraic number field $\Q(\sqrt{2}, \sqrt{3}\,)$. 
Then $K$  is generated by $p_1 = \sqrt{3} + \sqrt{2}$ 
and this number is a root of the polynomial $X^4 - 10X^2 + 1$,
which is irreducible in $\Q[X]$
(see, \eg{}\cite[Example 5.6.2]{AlWi04}).
As this polynomial has four real roots,
namely $p_1$, $-p_1$ and $\pm(\sqrt{3}- \sqrt{2})$,
the field $K$ is totally real 
and so the $\Z$-rank of  $U(\OO)$ is 3 by  Dirichlet's Unit Theorem.

Consider now the subring $ R = \Z[\sqrt{2}, \sqrt{3}\,]$ of $\OO$.
It is easy to find 3 units in $R$
that have a good chance of generating a free abelian group of rank 3.
namely
\begin{align*}
p_1 &= \sqrt{3} +\sqrt{2} &\text{with } p_1^{-1} &= \sqrt{3} -\sqrt{2}, \\
p_2 &= \sqrt{2} +1    &\text{with } p_2^{-1} &= \sqrt{2} -1, \\
p_3 &=2 +\sqrt{3}   &\text{with } p_3^{-1} &= 2 -\sqrt{3}.
\end{align*}
Note that $p_1 + 3 = p_2 + p_3$;
if $P$ is a group generated by two of these numbers,
the ring $R =\Z[P]$ will therefore contain the remaining one.

We finally show that $p_1$, $p_2$, $p_3$ generate a free abelian group of rank 3.
Note first that all three numbers are greater than 1 and so of infinite order.
To show that $\gp(p_2, p_3)$ has rank 2,  it suffices to verify 
that no positive power of $p_2$ can be a positive power of $p_3$;
this follows from the facts 
that $p_2^{m_2}$ is a linear combination of 1 and $\sqrt{2}$ with positive coefficients,
from the analogous statement for $p_3^{m_3}$, 
and from the linear independence of 1, $\sqrt{2}$ and $\sqrt{3}$. 
To establish that $\gp(p_1, p_2, p_3)$ has rank 3 it suffices now to verify 
that no positive power $p_1^{m_1}$ of $p_1$
is a product of the form $a_2 \cdot a_3$ with $a_2 \in \Z[\sqrt{2}\,]$ 
and $a_3 \in \Z[\sqrt{3}\,]$.
As there is no harm in assuming that $m_1$ is even, 
we shall do so; then $p_1^{m_1} \in \Z[\sqrt{6}\,]$
and so the claim follows by a straightforward calculation 
based on the fact 
that the numbers 1, $\sqrt{2}$, $\sqrt{3}$ and $\sqrt{6}$ are linearly independent over $\Q$.
\index{Group Aut(A)@Group $\Aut(A)$!examples|)}%
\smallskip

(ii) We continue with \emph{specimens of functions $b \mapsto Q_b$} 
that crop up in statement (v) of Corollary \ref{CorollaryE13}.
We start with two observations.
Let $b$ be a positive element of $A$.
The group $Q_b$ is then given by the formula
\[
Q_b = \{s\in \Aut(A) \mid (|s| -1)b \in IP\cdot A\}
\]
It is the direct product $\{1,-1\} \times (Q_b \, \cap\, \R^\times_{>0})$;
in order to determine $Q_b(A)$  it suffices thus to find the positive elements of $Q_b(A)$.
Secondly, if $b_1$, $b_2$ are two elements  in $A_{>0}$ 
which are congruent \emph{modulo} $IP \cdot A$ 
the groups $Q_{b_1}$ and $Q_{b_2}$ are equal.
So we need only find the groups $Q_b$ 
for $b$ running through a system of positive representatives of $A/(IP \cdot A)$.

Consider now the case 
where $P$ is generated by positive integers $p_1$, \ldots, $p_k$ 
and set $R = \Z[P]$.
By the example in part (i) a) above,
the group $\Aut_o(R)$ is then
generated by the positive prime divisors $q_1$, \ldots, $q_\ell$ 
of the product $p_1 \cdots p_k$.

The ideal $I[P] = IP \cdot R$ is generated by the integers $p_1-1$, \ldots, $p_k-1$; 
it is actually the principal ideal generated by the greatest common divisor $d$ of these numbers.
The quotient module $A/(IP \cdot A)$ is therefore represented 
by the list of integers $\{1,2, \ldots, d\}$.
If $b$ is one of them,
a fraction $s = d/n \in \gp(q_1, \ldots, q_\ell)$ lies in $Q_b$ if, 
and only if,
the congruence
\begin{equation}
\label{eq:Definition-of-b-example1}
(d-n) \cdot b \equiv 0 \pmod{d} 
\end{equation}
is satisfied. 

Here is a numerical example.
Choose $p_1 = 65 = 5 \cdot 13$ and $p_2= 97$.
Then $P = \gp(65, 97)$, $\Aut_o(A) = \gp(5, 13, 97)$, and $d = \gcd\{65-1, 97-1 \} = 32$.
The cyclic group generated by 5 is a complement of $P$ in $\Aut_o(A)$;
in order to find $Q_b$ is suffices thus to find the minimal positive exponent $k$
with  $(5^k - 1) \cdot b  \equiv 0 \pmod{32}$. 
Since $5^k - 1$ and 32 are both multiples of 4, the preceding congruence 
is equivalent to the congruence
\[
(1 + 5 + \cdots + 5^{k-1}) \cdot b \equiv 0 \pmod{8}.
\]
So we need only find the minimal exponents $k_{\min}$ for the 8 values 1,  2,  \ldots, 8 of $b$.
The results are collected in the following table.
\begin{center}
\begin{tabular}{r|ccccccccc}
$b$&
1&2&3&4&5&6&7&8\\
\hline
$k_{\min}$&
8&4& 8&2&8&4&8&1\\
\end{tabular}
\end{center}
\smallskip

(iii) We conclude with an illustration of Corollary  \ref{CorollaryE12}.
As in the previous part(ii),
we consider $P = \gp(65, 97)$ and $A = \Z[P]$.
Then  $\gcd\{65-1, 97-1 \} = 32$ 
and so the inclusion of $\Z$ into $A$ induces an isomorphism $\Z/32\Z \iso A / (IP \cdot A)$.
The group $\Aut_o(A)$ is generated by $\{5, 13, 97\}$ or, alternatively,  by $\{5, 65, 97\}$.
Since the numbers 65 and 97 act on $\Z/32\Z$ by the identity, 
the isomorphism classes of the groups $G_b = G([0, b]; \Z[1/65 \cdot 97], \gp(65,97))$ 
correspond to the orbits of $\gp(5)$ on the abelian group $\Z/32\Z$.
There are 10 orbits,
represented by the subsets
\begin{gather*}
\{1,5,9,13,17,21,25, 29\}, \quad \{3,7,11,15, 19, 23, 27, 31\}\\
\{2,10, 18,26\},\quad\{6,14,22,30\}, \\
\{4,20\},\quad \{12, 28\},\quad \{8\}, \quad \{16\}, \quad \{24\}, \quad \{32\}.
\end{gather*}
%
%
\section{Isomorphisms of groups with cyclic slope groups}
\label{sec:18}
%
If $P$ is cyclic 
the homeomorphisms $\varphi \colon \Int(I) \iso \Int(\bar I)$
that induce isomorphisms $\alpha \colon G \iso \bar G$ are of a more varied nature 
than in the non-cyclic case, 
studied in Section \ref{sec:17}.  
On the one hand, 
$\varphi$ can be an infinitary PL-homeomorphism even 
when $G = G([0,b];A,P)$ and $\bar G = G([0,\bar b];A,P)$, 
as has already been brought to light by the proof of Theorem \ref{TheoremE07}. 
In addition,
there exists embeddings
\begin{equation*}
\index{Homomorphism!12-mu(1,b)@$\mu_{1,b}$}%
\index{Homomorphism!12-mu(2,b)@$\mu_{2,b}$}%
\xymatrix{G([0,b];A,P) \ar@{>->}[r]^-{\mu_{1,b}} &G([0,\infty[\;;A,P) \ar@{>->}[r]^-{\mu_{2,b}} &G(\R;A,P)}
\end{equation*}
induced by infinitary PL-homeomorphisms, 
which map the corresponding subgroups of bounded homeomorphisms isomorphically onto each other;
they will be constructed in section \ref{ssec:18.2}.
Moreover,
$G([0,b];A,P)$ admits injective endomorphisms $\mu_m$ and $\nu_n$,
whose images are subgroups that contain the bounded subgroup $B([0,b];A,P)$
having finite indices $m$ and $n$, respectively, in $G/[0,b];A,P)$
see section \ref{sssec:18.5}. 
\smallskip

Prior to giving details about these surprising isomorphisms, 
we deal in section \ref{ssec:18.1} with a case 
where the situation is much as in Section \ref{sec:17}.

\subsection{Unbounded intervals}
\label{ssec:18.1}
If the intervals are unbounded, 
most of the conclusions of Theorem \ref{TheoremE10} and Supplement \ref{SupplementE11New}
continue to hold for cyclic $P$:
\begin{thm}
\label{TheoremE14}
\index{Automorphisms of G(R;A,P)@Automorphisms of $G(\R;A,P)$!properties}%
\index{Group G(R;A,P)@Group $G(\R;A,P)$!isomorphisms}%
\index{Group G(R;A,P)@Group $G(\R;A,P)$!endomorphisms}%
\index{Automorphisms of G([0,infty[;A,P)@Automorphisms of $G([0, \infty[\;;A,P)$!properties}%
\index{Group G([0,infty[;A,P)@Group $G([0, \infty[\;;A,P)$!isomorphisms}%
\index{Group G([0,infty[;A,P)@Group $G([0, \infty[\;;A,P)$!endomorphisms}%
\index{Representation Theorem for isomorphisms!supplements}%
\index{Proposition \ref{PropositionE9}!consequences}%
\index{Supplement \ref{SupplementE11New}!variation}%
Assume $P$ is cyclic
and $\alpha \colon G(I;A, P) \iso \bar{G}$
is an isomorphism onto a subgroup $\bar{G}$ of $G(\bar{I}; \bar{A}, \bar{P})$
which contains the derived group of $B(\bar{I}; \bar{A}, \bar{P})$.
Let $\varphi \colon \Int(I) \iso \Int(\bar{I})$ be the unique homeomorphism inducing $\alpha$.
\begin{enumerate}[(i)]
\item If $I = \R$ 
then $\bar{I} = \R$ and $\varphi$ is a finitary PL-homeomorphism 
with slopes in a coset $s \cdot P$ and singularities in $A$.
Moreover, $P = \bar{P}$, $\bar A = s \cdot A$ and $\bar{G} = G(\R; \bar A ,P)$. 
\item If $I = [0, \infty[$ or $I = \,]0 ,\infty[$, and  if $\bar{I} = I$,
then $\varphi$ is an increasing  PL-homeomor\-phism 
with slopes in a coset $s \cdot P$ and singularities in $A$.
Moreover, $P = \bar{P}$, $\bar A = s \cdot A$ and for each $k \in \N_{>0}$,
the interval $[1/k, \infty[$ contains only finitely many singularities of $\varphi$
and $\bar{G}$ contains 
$\ker (\bar{\sigma}_- \colon  G(\bar{I}; \bar{A}, \bar{P}) \to \bar{P})$.
\end{enumerate}
\end{thm}

\begin{proof}
The crucial ingredient in the proofs of assertions (i) and (ii) is the analysis carried out in Part 2 
of the proof of Supplement \ref{SupplementE11New} 
(see pages \pageref{para:Part2-SupplementE11New}ff).
\index{Supplement \ref{SupplementE11New}!consequences}%

Assertion (i) is a restatement of Corollary \ref{crl:Isomorphism-I-line}.
Assume now that $I$ is one of the half lines $[0, \infty[$ or $]0, \infty[$  
and that $\bar{I} = I$. 
Part 2 in the proof of Supplement \ref{SupplementE11New} then shows 
that $\varphi$ is affine on a subinterval $J \subset I$ of the form $]t_*,\infty[$
and thus differentiable on $J$.
Proposition \ref{PropositionE9}
\index{Proposition \ref{PropositionE9}!consequences}%
tells one next
that $\varphi$ is piecewise affine with slopes in a coset $s \cdot P$, 
that it has only finitely many singularities in each interval of the form $[1/k, \infty[$
and that $P = \bar{P}$. 
Note that $\varphi$ is increasing;
indeed,
 $\bar{I} = I$ and $\im (\rho \colon G(I;A, P) \to \Aff (A,P))$ 
 and $P = \im (\bar{\sigma}_- \colon \bar{G} \to P)$ are not isomorphic).
\end{proof}

\subsection{Explicit constructions of isomorphisms: part  II}
\label{ssec:18.2}
\index{Group G([a,c];A,P)@Group $G([a,c];A,P)$!embeddings}
The isomorphisms constructed below make explicit use of the fact that $P$ is cyclic.
In that respect they are akin to the isomorphisms 
used in the proof of Theorem \ref{TheoremE07}.
\index{Theorem \ref{TheoremE07}!analogues}%

\emph{Notation. }
In section \ref{ssec:18.2} the group $P$ is assumed to be cyclic
and  $p$ denotes its generator with $p > 1$.

\subsubsection{The embedding $\mu_{1, b} \colon G([0,b];A,P) \mono G([0,\infty[\,;A,P)$.}
\label{sssec:Embedding-mu1}
%
Given a positive number $ b \in A$,
let $\varphi_b \colon [0,\infty[\ \to\ [0,b[$ 
be the affine interpolation of the sequence of points
\begin{equation}
\label{eq:18.3}
\left(j\cdot (p-1)b,(1-p^{-j})b \right) \quad \text{ for $j = 0$, 1, \ldots }.
\end{equation}
Then $\varphi$ is a PL-homeomorphism with slopes $p^{-1}$, $p^{-2}$, \ldots 
that maps $A \cap [0, \infty[$ onto $A \cap [0, b[$\,.
It is illustrated by the following diagram.

\input{chaptE.18.illustration-mu1.tex}
\begin{lemma}
\label{LemmaE15}
\index{Group G([a,c];A,P)@Group $G([a,c];A,P)$!embeddings}%
\index{Isomorphisms!construction}%
\index{Homomorphism!12-mu(1,b)@$\mu_{1,b}$}%
For every positive number $b \in A$
the PL-homeomorphism $\varphi_b^{-1}$ induces by conjugation an embedding
 $\mu_{1,b} \colon G([0,b]; A, P) \mono \bar{G} = G([0,\infty[\, ; A, P)$. 
Its image is the subgroup consisting of all elements $\bar g \in \bar G$ with 
\[
\sigma_+(\bar g) = 1 \quad \text{and} \quad  \tau_+(\bar g)  \in \Z(p-1) b.
\]
Let $\bar{\mu}_{1,b} \colon P \times P \to  P \times ((IP \cdot A) \rtimes P)$ 
denote the function defined by  
\[
(p^m,p^n)   \mapsto  \left(p^m,(-n(p-1)\cdot b, 1)\right).
\]
Then $\bar{\mu}_{1,b}$ is a monomorphism which renders the square 
\begin{equation}
\label{eq:18.4}
\xymatrix{G([0,b];A,P) \ar@{->>}[r]^-{(\sigma_-,\sigma_+)} \ar@{>->}[d]^-{\mu_{1,b}} &P\times P\ \ar@{>->}[d]^-{\bar\mu_{1, b}}\\
G([0,\infty[\,;A,P) \ar@{->>}[r]^-{(\sigma_-,\rho)} &P\times (IP\cdot A \rtimes P)}
\end{equation}
commutative. 
\end{lemma}

\begin{proof}
For every $f \in G = G([0,b];A,P)$
the composition $\varphi_b^{-1} \circ f \circ \varphi_b$ 
is a (possibly infinitary) PL-auto-homeomorphism of $ I =[0, \infty[$ with slopes in $P$ 
and vertices in $A^2$.
Moreover, since $f \in G = G([0,b];A,P)$,
there exists a small number $\varepsilon > 0$ 
so that $f$ is affine on the interval  $[(1-\varepsilon)b, b] $ 
and has there slope $p^n$ for some $n \in \Z$.
Choose $j_0$ so large that $1/p^{j_0} \leq \varepsilon$.
It then follows, for every index 
\[
j \geq j_1 = \max\{j_0, j_0 + n\},
\]
that $f$ sends the number $t_j = (1-p^{-j})b$ 
to $t_{j-n} =(1-p^{-j + n})b$.
The composite $\act{\varphi_b^{-1}}f$ thus maps
$ \bar{t}_j = j \cdot (p-1)b$ to 
$\bar{t} _{j-n} = (j- n) \cdot (p-1)b$ for every $j \geq j_1$
and so it is a translation with amplitude $- n \cdot (p-1)b$ for $t \geq j_1 (p-1)b$.

Conversely, if $\bar{f} \in G([0, \infty[\, ;A,P)$ has the property 
that $\rho (\bar{f})$ is a translation with amplitude in $\Z(p-1)\cdot b$,
then $\varphi_b\circ \bar{f} \circ \varphi_b^{-1}$ is a PL-auto-homeomorphism of $[0,b[$ 
with slopes in $P$ and vertices in $A^2$
that is affine near $b$ 
and so it is an element of $G([0,b]; A, P)$.
We are left with verifying the commutativity of diagram \eqref{eq:18.4}.
It is clear that $f$ and its conjugate $\mu_{1,b}(f) = \act{\varphi_b^{-1}}f$ have the same slope in 0.
On the other hand,
the above calculation has shown that $\rho(\mu_{1,b}(f))$ is a translation 
with amplitude  $- n (p-1)b$ if $f$ has slope $p^n$ in 1.
So the diagram commutes.
Since $\mu_{1,b}$ and $\bar{\mu}_{1,b}$ are obviously injective,
the proof is complete.
\end{proof}
%
\subsubsection{The embedding $\mu_{2,b} \colon G([0,\infty[\,;A,P) \mono G(\R;A,P)$.}
\label{sssec:Embedding-mu2}
%
Given $b \in A_{>0}$,
let $\psi_{2,b} \colon \R \iso\ ]0,\infty[$ be the affine interpolation of the sequence
\begin{equation}
\label{eq:18.5}
\left(-j(p-1)b,p^{-j}\cdot b\right) \quad \text{ for $j = 0$, 1, \ldots }
\end{equation}
and the translation by $b$ on $[0,\infty[$.  
Then $\psi_b$ has slopes $\cdots,p^{-2},p^{-1},1$, 
and it maps $A$ onto $A \,\cap\, ]0, \infty[$; 
it is indicated by the following diagram.

\input{chaptE.18.illustration-mu2.tex}

\begin{lemma}
\label{LemmaE16}
\index{Group G([0,infty[;A,P)@Group $G([0, \infty[\;;A,P)$!embeddings}%
\index{Isomorphisms!construction}%
\index{Homomorphism!12-mu(2,b)@$\mu_{2,b}$}%
For every positive number $b \in A$
the PL-homeomorphism $\psi_{2,b}^{-1}$ induces by conjugation an embedding
$\mu_{2, b} \colon G([0,\infty[\,; A, P) \mono G(\R ; A, P)$. 
Its image consists of all elements $\bar g \in \bar G = G(\R;A,P)$ 
with 
\[
\sigma_-(\bar{g}) = 1, \tau_-(\bar{g}) \in \Z(p-1)b \quad \text{and} \quad  
\rho(\bar{g}) \in (IP \cdot A) \rtimes P.
\]
Let  $\bar{\mu}_{2, b} \colon 
P \times (IP \cdot Q \rtimes P) \to  (A \rtimes P)  \times (A \rtimes P)$ 
denote the function defined by 
\[
(p^m,y) \mapsto \left(m(p-1)b,(-b, 1) \cdot y \cdot (b,1)\right).
\]  
Then $\bar{\mu}_{2,b}$ is a monomorphism and the following square commutes.
\begin{equation}
\label{eq:18.6}
\xymatrix{G([0,\infty[;A,P) \ar@{->}[r]^-{(\sigma_-,\rho)} \ar@{>->}[d]^-{\mu_{2,b}} &P\times (IP\cdot A \rtimes P) \ar@{>->}[d]^-{\bar\mu_{2,b}}\\
G(\R;A,P) \ar@{->}[r]^-{(\lambda,\rho)} &( A\rtimes P) \times ( A\rtimes P)}
\end{equation}
\end{lemma}

\begin{proof}
For every $f \in G = G([0,\infty[\,;A,P)$
the composition $\psi_b^{-1} \circ f \circ \psi_b$ 
is a (possibly infinitary) PL-auto-homeomorphism of $ I =\R$ with slopes in $P$ 
and vertices in $A^2$.
Since $f$ lies in $G$
there exists a small number $\varepsilon > 0$ 
so that $f$ is linear on the interval  $[0, \varepsilon] $,
say with slope $p^m$.
Choose $j_0$ so large that $b/p^{j_0} \leq \varepsilon$.
It then follows
that $f$ sends,
for every index $j \geq j_1 = \max\{j_0, j_0 + m\}$,
the point $t_j = b/p^j$  to $t_{j-m} =b/ p^{j-m}$.
The composition $\act{\psi_b^{-1}}f$ maps,
for every $j \geq j_1$, 
the point  $ \bar{t}_j = -j (p-1)b$ to $\bar{t} _{j-m} = -(j- m) (p-1)b$,
and so it is the translation with amplitude $m (p-1)b$ 
on the interval $]-\infty,  - j_1(p-1)b]$.

There exists, furthermore, a positive number $t_*$ 
so that $f$ is affine on the interval $[t_*, \infty[$ and $\rho(f) \in \Aff(IP \cdot A, P)$.
The composition $\act{\psi_b^{-1}}f$ is then affine on the interval $[\max\{0, t_*-b\}, \infty[$
and $\rho(\act{\psi_b^{-1}}f)$ lies again in $\Aff(IP \cdot A, P)$.

Conversely, 
if $\bar{f} \in G(\R ;A,P)$ has the property 
that $\lambda(\bar{f})$ is a translation with amplitude in $\Z(p-1) b$
and $\rho(\bar{f}) \in \Aff(IP \cdot A, P)$,
the composition $\psi_b \circ \bar{f} \circ \psi_b^{-1}$ is a PL-homeomorphism with slopes in $P$
that maps $A \,\cap\, ]0, \infty[$ onto itself and belongs to $G([0, \infty[\,;A, P)$.
The commutativity of the diagram follows from the calculations carried out on the way.
\end{proof}
%
\subsubsection{The embedding $\mu_{2,b}\circ \mu_{1,b} \colon G([0,b];A,P) \mono G(\R;A,P)$.}
\label{sssec:Embedding-mu2-mu1} 
The composition $\varphi_{1,b} \circ \psi_{2,b}\colon \R\to\ ]0,b[$ is a PL-homeomorphism; 
its description is unwieldy unless $(p-1)b = b$, \ie{}unless $p =2$. 

By combining Lemmata \ref{LemmaE15} and \ref{LemmaE16} one obtains
\begin{lemma}
\label{LemmaE17}
Let $\bar\mu \colon P^2 \to (A \rtimes P)^2$ be the function defined by the formula
\[
\bar{\mu}(p^m,p^n) = (m(p-1)b,-n(p-1)b).  
\]
Then $\bar{\mu}$ is a monomorphism.
The composition $(\varphi\circ\psi_b)^{-1}$ induces by conjugation  an embedding 
$\mu = \mu_{2,b}\circ \mu_{1, b}$ 
which makes the square
\begin{equation}
\label{eq:18.7}
\xymatrix{G([0,b];A,P) \ar@{>->}[r]^-{(\sigma_-,\sigma_+)} \ar@{->}[d]^-\mu &P\times P \ar@{->}[d]^-{\bar\mu}\\
G(\R;A,P) \ar@{->}[r]^-{(\lambda,\rho)} &A\rtimes P \times A\rtimes P}
\end{equation}
commutative.
The image of $\mu$ consists of all elements $\bar{g} \in G(\R;A,P)$ 
which satisfy the conditions
\[
\sigma_-(\bar{g}) = 1 = \sigma_+(\bar{g}) \quad \text{ and } \quad 
\left(\tau_-(\bar{g})  , \tau_+(\bar{g})\right) \in \Z(p-1)b \times  \Z(p-1)b.
\]
\end{lemma}
\index{Group G([a,c];A,P)@Group $G([a,c];A,P)$!embeddings}
\index{Isomorphisms!construction}
%
\subsection{Application to the groups $G[p] = G([0,1];\Z[1/p], \gp(p))$}
\label{ssec:Application-embeddings-mi1-mu2-etc}
\index{Group G[p]@Group $G[p]$!endomorphisms}
Let $P$ be the cyclic group generated by an integer $p\geq 2$ 
and choose $A = \Z[1/p]$ and $b = 1$. 
The group $G[p] = G([0,1]; A, P)$ has been investigated in section \ref{ssec:9.9}
and in Section \ref{sec:15}. 
In this section,
we work out the images of $G[p]$ 
for the embeddings $\mu_{1,1}$ and $\mu =  \mu_{2,1} \circ \mu_{1,1}$.

According to Lemma \ref{LemmaE15},
the embedding $\mu_{1,1}$ maps $G[p]$ isomorphically onto the subgroup of
$G([0,\infty[\,;A,P)$ consisting of all elements $\bar{g}$ which satisfy the requirements
\begin{equation*}
\sigma_+(g) = 1 \text{ and } \tau_+(g) \in \Z(p-1).
\end{equation*}
We claim 
that the image $\mu_{1,1}(G[p])$ is generated 
by the PL-homeomorphisms $y_i$ defined by the formula
\begin{equation}
\label{eq:18.8}
y_i(t) = 
\begin{cases}
t &\text{if $t < i$}\\
p(t-i)+i &\text{if $i\leq t \leq i+1$}\\
t+(p-1) &\text{if $t > i+1$}
\end{cases},
\end{equation}
the index $i$ varying over $\N$.  
\footnote{In section \ref{ssec:8.2},
the function $y_i$ is called $f(i,i+1;p)$.} 
  Let $H$ denote subgroup generated by $\{y_i \mid i \in \N\}$.  
For every $m\geq 1$, 
the product
\begin{equation*}
y_0 \circ y_1 \circ \cdots \circ y_m
\end{equation*}
has slope $p$ on $[0,m]$.  
It follows that there exist, for every $g$ in $\im \mu_{1,1}$, 
an element $h$ in $H$ 
such that the singularities of $hgh^{-1}$ are contained in $\N$. 
The argument detailed in section \ref{ssec:8.2} then shows 
that $\act{h}g$ is in $H$.

One verifies painlessly that $\act{y_i}y_j = y_{j+(p-1)}$ whenever $i < j$, 
and the usual normal form argument then permits one to establish
\begin{lemma}
\label{lemma:New18.3}
The image of $G([0,1];\Z[1/p], \gp(p))$ under $\mu_{1,1}$ is generated by the set 
$\{y_i\mid i \in \N\}$ specified by formula \eqref{eq:18.8}, 
and defined in terms of this set by the relations
\begin{equation}
\label{eq:Presentation-for-mu1(G[p])}
\index{Group G[p]@Group $G[p]$!infinite presentation}%
\act{y_i}y_j = y_{j+(p-1)} \text{ whenever } i < j.
\end{equation}
\end{lemma}
%
\subsubsection{Images of the generators $y_i$ in $G(\R; A, P)$.}
\label{sssec:images-yi}
%
For the determination of the automorphisms of the group $G[p]$ 
it is advantageous to use their realizations in $G(\R;A,P)$ afforded
by the embeddings 
\[
\mu_{2,1}\colon \im \mu_{1,1} \mono G(\R;A,P),
\qquad y_i \longmapsto \psi_{2,1}^{-1} \circ y_1 \circ \psi_{2,1}.
\]
In this context, 
it is useful to dispose of a description of the images $z_i$ of the generators $y_i$.
It is not hard to find one: 
using definition \eqref{eq:18.5} of $\psi_{2,1}$
and the fact that $z_i =\mu_{2,1} (y_1)$ is the composition $\psi_{2,1}^{-1} \circ y_i \circ \psi_{2,1}$,
one easily verifies 
that $z_0$ is the translation with amplitude $p-1$,
while $z_i = y_{i-1}$ for positive indices $i$.
%

\subsubsection{Preimages of the generators $y_i$ in $G([0,1];A,P)$.}
\label{sssec:preimages-yi}
Presentation \eqref{eq:Presentation-for-mu1(G[p])} 
coincides with presentation \eqref{eq:9.12} of the group $G[p]$,
except for the fact 
that the generators are now called $y_i$ 
while they are denoted by $x_i$ in section \ref{ssec:9.9}. 
This makes one wonder 
what the preimages $\mu_{1,1}^{-1}(y_i)$ of the generators $y_i$  look like.

To find the answer,
we recall the definition of the generators $x_i$.
The generator $x_i$ is the PL-homeomorphism $f(m, r) = f(p^m;p;r,p)$,
the parameters $i$ and $(m,r)$ being related by the formula 
$i = (p-1)m + (p-r)$ with $1 < r \leq p$;
see Corollary \ref{CorollaryB13}.
The PL-homeomorphism $f(m, r)$ is the affine interpolation of 
two $\PP$-regular subdivisions of level $m + 2$.
In the sequel it will suffice to have a closer look at the case where $m=0$;
then $f(r) = f(1,p; r, p)$ is obtained as follows.
One divides the unit interval into $p$ subintervals of length $1/p$ 
and then the $r$-th of the subintervals so obtained again into $p$ intervals of length $1/p^2$,
ending up with the $\PP$-regular subdivision $D_r$.
Let $D'$ be the $\PP$-standard subdivision of level 2;
it arises by first subdividing the unit interval into $p$ parts of equal length 
and then dividing the first interval $[0, 1/p]$ into $p$ subintervals of length $1/p^2$.
The PL-homeomorphism $f(r)$, finally, is the affine interpolation 
of the sequence of points $D_r \boxtimes D'$.

Now,
in the definition of a $\PP$-standard subdivision it is always the \emph{left most} subinterval attained
at an intermediate level $\ell$ that gets subdivided in the passage from level $\ell$ to level $\ell + 1$.
Since conjugation by the reflection $\imath \colon t \mapsto 1- t$
\label{notation-imath}%
 maps the group $G([0,1]; A, P)$ onto itself and thus induces an automorphism of this group,
one could equally well work with subdivisions 
where it is always the \emph{right most} of the intervals already created that gets subdivided.
If one proceeds in this new manner 
one obtains the preimages of the elements $y_i$ 
used in the presentation  \eqref{eq:Presentation-for-mu1(G[p])}.
Indeed, one has
\begin{proposition}
\label{prp:Preimages-yi}
Let $\imath$ denote the reflection at the midpoint of the unit interval 
and let $x_i$ be the generator of the group $G([0,1]; \Z[1/p], \gp (p))$ 
introduced in section \ref{ssec:9.9}. 
For each $i \in \N$ set $f_i = \act{\imath} x_i$.
Then $f_i$ coincides with $\mu_{1,1}^{-1}(y_i)$ for each $i \in \N$.
\end{proposition}

The diagrams below illustrate the generators $f_i$ for $p = 5$ and $ i \in \{0,1,2,3\}$.
\smallskip 
\input{chaptE.18.generators-tilde-x}
\index{Rectangle diagram!examples}%
%
\subsubsection{Proof of Proposition \ref{prp:Preimages-yi}.}
\label{sssec:Proof-Proposition-Preimiges-yi}
%
We begin by collecting the salient features of the generators $f_i$, 
for $i$ running from 0 to $p-2$.
The function $f_i$ is the identity on the interval $[0, i/p]$.
The next few subdivision points of the interval $[i/p, (i+1)/p]$ 
are mapped to $(i+1)/p$, $(i+2)/p$, up to $(p-1)/p$;
there are $p- 1 -i$ such points.
It follows that $f_i$ has slope $p$ on the interval from $s_1=i/p$ 
to $s_2 = s_1 + (p-1-i)/p^2$.  
The remaining $i+1$ subdivision points in the interval $[i/p, (i+1)/p]$
are mapped to the points 
\[
(p-1)/p + 1/p^2,  \quad (p-1)/p + 2/p^2, \ldots, (p-1)/p + (i+1)/p^2;
\]
on the interval from $s_2 = s_1 + (p-1-i)/p^2$ to $s_3 = (i+1)/p$ 
the function $f_i$ has therefore slope 1.
On the interval from $s_3= (i+1)/p$ to the endpoint 1, finally, the function is affine with slope $1/p$.
We conclude that $f_i$ has 4 singularities,
in
\begin{equation}
\label{eq:List-of-singularities}
s_ 1 =\frac{i}{p}, \quad 
s_2 = s_1 + \frac{p-1-i}{p^2} = \frac{i+1}{p}- \frac{i+1}{p^2}, \quad 
s_3 = \frac{i+1}{p}, \quad 1,
\end{equation}
and slope $p$ on $]s_1,s_2[$, slope 1 on $]s_2,s_3[$ and slope $1/p$ on the interval $]s_3, 1[$.

We next determine the singularities $t_j$
of the PL-homeo\-mor\-phism 
$\mu_{1,1}^{-1}(y_i)$, again for $i \leq p-2$. 
The function $\mu_{1,1}^{-1}(y_i)$ is the composition 
$\varphi_1 \circ y_i \circ \varphi_1^{-1}$;
see section \ref{sssec:Embedding-mu1}.
The rectangle diagram of this composition (restricted to the interval $[0,(p-1)/p]$) 
is displayed in Figure \ref{fig:PL-homeomorphism-mu1-invers-i-small}.
The generator $y_i$ is the identity on the interval $[0,i]$,
it has slope $p$ on the adjacent interval $]i, i+1[$ and is the translation with amplitude $p-1$ 
on the half line $[i+1, \infty[\,$.
The infinitary PL-homeomorphism $\varphi_1$ is linear with slope $1/p$ on the interval $]0, p-1[$, 
affine with slope $1/p^2$ on the interval $]p-1, 2(p-1)[$
and affine with slope $1/p^3$ on the interval $]2(p-1), 3(p-1)[$.
\input{chaptE.18.preimages1}
\index{Rectangle diagram!examples}%

Using the stated facts, 
it is easy to determine the singularities $t_i$ of $\mu_{1,1}^{-1}(y_i)$.
Note first that the chain of inequalities
\[
i < p-1 < p + i \leq 2(p-1)
\]
holds for every couple of integers $(i,p)$ with $0 \leq i \leq p-2$ and $2 \leq p$.
Figure  \ref{fig:PL-homeomorphism-mu1-invers-i-small}
thus allows one to see 
that the first singularity $t_1$ equals $ i/p$,
and so it coincides with $s_1$,
the first singularity of $f_i$ (see formulae  \eqref{eq:List-of-singularities}).
The function $\mu_{1,1}^{-1}(y_i)$ 
is affine with slope $p$ on the interval from $t_1$ to
\[
t_2 = \varphi_1(y_i^{-1}(p-1) )
= 
\varphi_1\left(i + \frac{(p-1-i)}{p} \right) 
= 
\frac{1}{p} \cdot \left(i + \frac{p-1-i}{p}\right),
\]
and so $t_2 = s_2$.
On the interval from $t_2$ to $t_3 = \varphi_1(i+1) = (i + 1)/p$ 
the map $\mu_{1,1}^{-1}(y_i)$ is a translation;
moreover,
$t_3 = s_3$.
On the interval from $t_3$ to $t_4 = (p-1)/p$ the function has slope $1/p$.
(Note
that $t_3 = t_4$  if $i= p-2$.)
 Finally,
 $\mu_{1,1}^{-1}(y_i)$ has the same slope $1/p$ on the interval from 
$t_4 = 1-p^{-1}$ to 1.
Indeed, 
definition \eqref{eq:18.3} of $\varphi_1$ justifies the following calculation for $k \geq 1$:
\[
\varphi_1 \circ y_i \circ \varphi_1^{-1} \colon
1- p^{-k} \longmapsto  k(p-1) \longmapsto (k+1)(p-1) \longmapsto 1 - p^{-(k+1)}.
\]
It follows 
that the function $\mu_{1,1}^{-1}(y_i)$ has no singularity in the interval $]t_3, 1[$ 
and slope $1/p$ on this interval.

The previous calculations show 
that the preimage $ \mu_{1,1}^{-1}(y_i)$ of $y_i$ coincides 
with the function $f_i$ for $i \in \{0,1, \ldots, p-2\}$.
We intend to deduce the general result by induction, 
the verifications just carried out being our basis,
but to do so we need to check our contention also for $i = p-1$.
This new case is depicted in Figure 
\ref{fig:PL-homeomorphism-mu1(invers)-i=p-1}.
\input{chaptE.18.preimages2}
\index{Rectangle diagram!examples}%

Note first that the chain of inequalities $p-1 < 2(p-1) < 2p-1 \leq 3(p-1)$
holds for every integer $p \geq 2$.
The function $ \mu_{1,1}^{-1}(y_{p-1})$ is the identity on the interval 
$[0, \varphi_1(p-1)]$ 
and its first singularity $t_1$ equals $\varphi_1(p-1) =(p-1)/p$.

The second singularity is  $\varphi_1\left(y_{p-1}^{-1}(2(p-1))\right)$
and  this number works out at
\[
t_2 = \varphi_1(p- 1 + (p-1)/p) = \frac{p-1}{p} + \frac{p-1}{p^3}.
\]  
The slope on the interval $]t_1,t_2[$  is $p$.
The third singularity is 
\[
t_3 = \varphi_1(p) = \frac{p-1}{p} + \frac{1}{p^2}
\]
and the slope on the interval $]t_2,t_3[$  is 1.
The value of $t_4 = \varphi_1(2(p-1))$, finally, is $(p-1)/p + (p-1)/p^2$.
If $p = 2$ the interval $]t_3, t_4[$ empty; 
otherwise, its length is positive and the function $\mu_{1,1}^{-1}(y_{p-1})$ has slope $1/p$ on it.
One verifies, as in the previous part, 
that this function has also slope $1/p$ on the interval from $t_4$ until 1.
Altogether one finds 
that the function $\mu_{1,1}^{-1}(y_{p-1})$ has four singularities in 
\[
t_1 = \frac{p-1}{p}, \quad 
t_2 = \frac{p-1}{p} + \frac{p-1}{p^3}, \quad
t_3 =  \frac{(p-1)}{p} + \frac{1}{p^2}, \quad  1
\]
and that its slopes are $p$ on $]t_1,t_2[$, 
then 1 on $]t_2, t_3]$ and finally $1/p$ on  $]t_3,1[$.

We claim these singularities and slopes are also those of the generator $f_{p-1}$.
To do so,
one has to recall 
that the generator $x_{p-1}$ is the conjugate of $x_0$ 
by the linear function $\lambda = (1/p)\id$. 
Therefore 
$
f_{p-1} 
= 
\act{\imath} x_{p-1} 
= 
\act{\imath \circ \lambda}x_0 
=
\act{\imath \circ \lambda \circ \imath^{-1}}f_0
$
is the conjugate of $f_0$ by the affine function 
$\vartheta = \imath \circ \lambda \circ \imath^{-1}$;
one finds that 
\[
\vartheta(t) = \frac{p-1}{p} + \frac{t}{p}.
\]
It follows that the singularities of $f_{p-1}$ 
are the images of the singularities of $f_0$  under $\vartheta$.
The singularities of $f_0$ can be obtained from the list   \eqref{eq:List-of-singularities}
by setting $i = 0$;
they are $s_1 =0$, $s_2 = (p-1)/p^2$, $s_3 = 1/p$ and 1, 
and so the list of singularities of $f_{p-1}$ is
\[
\vartheta(s_1) 
= 
\frac{p-1}{p}, 
\quad
\vartheta(s_2) 
= 
\frac{p-1}{p} + \frac{p-1}{p^3}, 
\quad
\vartheta(s_3) 
= 
\frac{p-1}{p} + \frac{1}{p^2}, 
\quad
1.
\]
This list tallies with that of the singularities of $\mu_{1,1}^{-1}(y_{p-1})$; 
so $f_{p-1} = \mu_{1,1}^{-1}(y_{p-1})$.
\smallskip

The proof is now quickly completed.
The previous calculation show 
that $f_{i} = \mu_{1,1}^{-1}(y_i)$ for each index $i < p$.
Suppose 
the equality $f_{j-(p-1)} = \mu_{1,1}^{-1}(y_{j-(p-1)})$  has been established
for some index $j \geq p$,
and set $i = j- (p-1)$.
Then $i \geq 1$.
Relations \eqref{eq:Presentation-for-mu1(G[p])} and \eqref{eq:9.12} 
therefore justify the following chain of equalities:
\[
\mu_{1,1}^{-1}(y_j) 
= \mu_{1,1}^{-1}\left( \act{y_0}y_i\right) 
=
 \act{f_0} f_i = \act{\imath} \left( \act{x_0} x_i \right) = \act{\imath} x_j = f_j.
\]
We conclude 
that the equality $f_i= \mu_{1,1}^{-1}(y_i)$ holds for each index $i \in \N$
and thus Proposition \ref{prp:Preimages-yi} is established. 
%
\subsection{Isomorphism types of normal subgroups containing $B$}
\label{ssec:18.4}
%
We consider a group $G([0,b];A,P)$ with $P$ cyclic 
and two subgroups $G$ and $\bar G$ containing $B([0,b];A,P)$.  
Our aim is to determine when $G$ and $\bar G$ are isomorphic.  
For ease of notation we set $I = [0,b]$ and $B = B(I;A,P)$.  
Note that $G$, $\bar G$ can be specified by giving their images  $Q$, $\bar Q$ 
under the epimorphism
\begin{equation*}
\xymatrix{\pi \colon G(I;A,P) \ar@{->>}[r]^-{(\sigma_-,\sigma_+)} 
&
P \times P \,\ar@{>->>}[r]^-\omega &\Z^2};
\end{equation*}
here $\omega$ takes $(p^j,p^\ell)$ to $(j,\ell)$ 
and $p$ denotes the generator of $P$ with $p < 1$.  

By Corollary \ref{CorollaryE5} each isomorphism $\alpha \colon G \iso \bar{G}$
maps the subgroup $B$ of $G$ onto the subgroup $B$ of $\bar{G}$
and so it induces an isomorphism 
\[
\bar\alpha \colon Q = \pi(G) \iso \bar Q = \pi(\bar{G}).  
\]
This fact prompts one to introduce two classes of groups,
defined by 
\begin{enumerate}[(I)]
\item  $Q$ has rank 1, and 
\item  $Q$ has rank 2, \ie{}$G$ has finite index in $G(I,A,P)$.  
\end{enumerate}
We shall see that class (I) divides into 3 isomorphism types, 
while class (II) comprises infinitely many such types.

We begin our analysis with an easy consequence of Theorem \ref{TheoremE04};
it holds for arbitrary groups $A$ and $P$ and reads thus:
\begin{corollary}
\label{crl:Respecting-attractors}
\index{Theorem \ref{TheoremE04}!consequences}%
Suppose $G \subseteq G(I;A,P)$ contains the derived group of $B = B(I;A,P)$
and $\bar{G} \subseteq G(\bar{I};\bar{A}, \bar{P})$ contains the derived group 
of $\bar{B} = B(\bar{I};\bar{A}, \bar{P})$.
Assume, in addition,
that $I$ and $\bar{I}$ are bounded from below with endpoints $a \in A$ and $\bar{a} \in \bar{A}$.
Define subsets
\begin{equation}
\label{eq:Attracting-subsets}
\Att_a(G) = \{g \in G \mid g'(a_+) < 1  \} 
\;\text{ and } \;
\Att_{\bar{a}}(\bar{G}) = \{\bar{g} \in \bar{G} \mid \bar{g}'(\bar{a}_+) < 1  \}.
\end{equation}
Then every \emph{increasing} isomorphism $\alpha \colon G \iso \bar{G}$ maps $\Att_a(G)$ 
onto $\Att_{\bar{a}}(\bar{G})$.
\end{corollary}
\begin{proof}
The set $\Att_a(G)$ consists of all $g \in G$ 
for which there exist a small positive number $\varepsilon_g$ so that $g(t) < t $ 
for every $t \in \;]a, a + \varepsilon_g[$.
Every increasing homeomorphism $\varphi \colon \Int(I) \iso \Int(\bar{I})$ 
leads therefore to an inclusion 
of $\act{\varphi}\Att_a(G)$ into $ \Att_{\bar{a}}(\bar{G})$.
As the rôles of $G$ and $\bar{G}$ can be interchanged,
this inclusion is an equality.
\end{proof}
\subsubsection{The endomorphisms $\mu_m$ and $\nu_n$}
\label{sssec:18.5}
We continue with the  construction of endomorphisms $\mu_m$ 
and $\nu_n$ of $G([0,b];A,P)$ that map $B$ onto itself.  
Theorem \ref{TheoremA} allows one to find, 
\index{Theorem \ref{TheoremA}!consequences}%
for every integer $m \geq 1$,
an element $f \in G(\R;A,P)$
that maps the interval $[p\cdot b,b]$ onto $[p^m\cdot b,b]$. 
\footnote{Recall that $p$ is the generator of $P$ with $p < 1$.} 
For $j\geq 0$,
let $f_j$ denote the composition $(p^{mj}\cdot \id) \circ f \circ (p^{-j}\cdot \id)$.  
Then $f_j$ induces by restriction a homeomorphism
\begin{equation*}
f_{j*} \colon [p^{j+1}b,p^jb] \iso [(p^m)^{j+1}b,(p^m)^j b]
\end{equation*}
and by taking the union of these homeomorphisms
one ends up with an infinitary PL-homeomorphism  
$\varphi_m\colon \,]0,b] \iso\, ]0,b]$.  
One verifies, 
as in the proof of Theorem \ref{TheoremE07}, 
that $\varphi_m$ induces an endomorphism $\mu_m$ of $G([0,b];A,P)$
\label{Endomorphism-mu-m}%
\index{Endomorphism mum@Endomorphism $\mu_m$!definition|textbf}%
which maps $B$ onto itself,
and that it induces the endomorphism
\begin{equation}
\label{eq:Morphism-mu-m}
\bar\mu_m \colon \Z^2 \mono \Z^2,\quad (j,\ell) \mapsto (mj,\ell).
\end{equation}
Similarly, one can construct, 
for each $n \geq 1$, 
an endomorphism $\nu_n$ of $G([0,b];A,P)$
\label{Endomorphism-nu-n}%
which takes $B$ onto itself and gives rise to the endomorphism
\begin{equation}
\label{eq:Morphism-nu-n}
\index{Endomorphism nun@Endomorphism $\nu_n$!definition|textbf}%
\bar\nu_n \colon \Z^2 \mono \Z^2,\quad (j,\ell) \mapsto (j,n\ell).
\end{equation}
Note 
that the auto-homeomorphism $\imath$ of $I = [0,b]$, 
given by $t\mapsto b-t$, 
induces an automorphism $\iota$ of $G([0,b];A,P)$ 
\label{notation:iota}%
which gives rise to the automorphism
\begin{equation*}
\index{Group G([a,c];A,P)@Group $G([a,c];A,P)$!endomorphisms}%
\index{Homomorphism!09-iota-bar@$\bar{\iota}$}%
\bar\iota \colon \Z^2 \iso\Z^2,\quad  (j,\ell) \mapsto (\ell,j).
\end{equation*}
%
\subsubsection{Isomorphism types of groups with $G/B$ infinite cyclic}
\label{sssec:18.6}
We consider now a group $G$  in the class (I); 
by definition,
the image $Q$ of 
\[
\index{Homomorphism!16-pi-bar@$\bar{\pi}$}%
\pi \colon G \incl G([0,b];A,P) \epi \Z^2
\]
 is then infinite cyclic 
and so it is generated by an element of the form 
\begin{equation*}
(j_0,\ell_0)\text{ with } j_0 > 0, \text{ or with } j_0 = 0 \text{ and } \ell_0 > 0.
\end{equation*}
Two cases arise: 
if $j_0 = 0$ then the homomorphism $\nu_{\ell_0}$ maps 
\begin{equation*}
G_1 = \ker(\sigma_- \colon G([0,b];A,P) \to P)
\end{equation*}
isomorphically onto $G$. 
Similarly, if $j_0 > 0$ but $\ell_0 = 0$ 
then $\mu_{j_0}$ maps $\ker \sigma_+$ isomorphically onto $G$
and   $\ker \sigma_+$ is isomorphic to $G_1$.
In the second case, $j_0 > 0$, and $\ell_0 \neq 0$.  
If $\ell_0 > 0$ then the group
\begin{equation*}
G_+ = \left\{g \in G([0,b];A,P) \mid \sigma_-(g) = \sigma_+(g)\right\}
\end{equation*}
is mapped by $\nu_{\ell_0} \circ \mu_{j_0}$ onto $G$; 
if $\ell_0 < 0$ it is the group 
\begin{equation*}
G_- = \{g\in G([0,b];A,P) \mid \sigma_-(g) = (\sigma_+(g))^{-1}\}
\end{equation*}
that is mapped by $\nu_{|\ell_0|} \circ \mu_{j_0}$ onto the given group $G$.  
We conclude that each group in class (I) is isomorphic to one of the groups $G_1$, $G_+$ or $G_-$. 
The next lemma shows 
that no two of these three groups are isomorphic.

\begin{lemma}
\label{lemma:The-three-groups-are-not-isomorphic}
\index{Group G([a,c];A,P)@Group $G([a,c];A,P)$!classification of normal subgroups}%
\index{Classification!of normal subgroups}%
\index{Theorem \ref{TheoremE04}!consequences}%
No two of the three groups  $G_+$, $G_-$ and $G_1$ are isomorphic,
for each satisfies a characteristic property not enjoyed by the other two groups,
namely:
\begin{enumerate}[(i)]
\item $G_+$ is an ascending HNN-extension with a stable letter $g_+$
and a base group of the form $B_0 = G([a_1, b-a_1];A,P)$
where $a_1 \in A$ is a small positive number with $3a_1 < b$;
\item $G_-$ is generated by an element $g_-$
and the subgroup $B_0 = G([a_1, b-a_1];A,P)$
where $a_1 \in A$ is a small positive number with $3a_1 < b$,
but it is not an ascending HNN-extension of the type described in (i);
\item $G_1$ cannot be generated by a subgroup $B_0 = G([a_1, b-a_1];A,P)$
with $a_1 > 0$ and an additional element $g_1$.
\end{enumerate}
\end{lemma}

\begin{proof}
The proof has two parts.
In the first one, the three statements (i), (ii) and (iii) are established.
In the second part we deduce from the first part and Theorem \ref{TheoremE04} 
that no two of the three groups are isomorphic.

We begin by fixing the notation
that will be used in the proofs of both (i) and (ii).
Choose a positive number $a_0 \in A$ with $3a_0 < b$. 
By Corollary \ref{CorollaryA2}
there exists then an element $g_\ell \in G(I;A,P)$  
with $\supp g_\ell \subset [0, a_0]$ 
whose right-hand derivative in 0 is $p$;
similarly, 
there exists an element $g_r \in G(I;A,P)$ 
with support contained in $[b-a_0, b]$
and whose left-hand derivative in $b$ is $p$.
Set $g_+ = g_\ell \circ g_r$ and $g_- = g_\ell \circ g_r^{-1}$.
Then $\sigma_-(g_+) = \sigma_+(g_+) = p$ and so $g_+ \in G_+$. 
On the other hand,
 $\sigma_-(g_-) =  p$ and $\sigma_+(g_-) = p^{-1}$,
 so $g_- \in G_-$.
 
Let  $a_1 \in A$ be a positive number which is so small 
that $g_\ell$ is linear on $[0, a_1]$ and $g_r$ is affine on $[b-a_1, b]$;
clearly $a_1 \leq a_0$.
Then both $g_+$ and $g_-$ are linear on $[0, a_1]$ and affine on $[b-a_1, b]$.
\smallskip
 
(i) The element $g_+$ generates a complement of $B = B([0,b];A,P)$ in $G_+$.
Set $B_0 = G([a_1, b- a_1];A,P)$.
Then $B_0$ is a subgroup of $B$ 
and it enjoys the following property:
\[
B_0 \subset \act{g_+} B_0 \subset \act{g_+^2} B_0 \subset \cdots 
\quad \text{and} \quad 
\bigcup\nolimits_{n \in \N} \act{g_+^n}B_0 = B.
\]
As $G_+$ is generated by  $B\, \cup\, \{g_+\}$, 
this property implies 
that $G_+$ is an ascending HNN-extension 
with stable letter $g_+$ and base group $B_0 = G([a_1, b-a_1];A,P)$.
\smallskip

(ii) The group $G_-$ is generated by  the subset $B \cup \{g_-\}$.
We want to show that $B$ is generated 
by the conjugates $B_0= G([a_1, b-a_1];A,P)$ by the powers of $g_-$.
By the choices of $a_1 \leq a_0$ and $g_-$, 
the generator $g_-$ fixes every point in the interval $[a_0, 2a_0]$
and it pushes the points in the interval $[0,a_1]$ towards 0.
Moreover, the inverse of $g_-$ moves every point of $[b-a_1, b]$ towards $b$.
It follows that the union of the conjugated subgroups $g_-^k B_0 g_-^{-k}$ 
for $k \in \N$ 
contains the subgroup $B([0, 2a_0];A,P)$.
Similarly, the subgroup $B([a_0, b];A,P)$ is contained in the union
of all the conjugated subgroups $g_-^{-j} B_0 g_-^{j}$ with $j \in \N$.
Lemma \ref{lem:Piecing-together} below allows us to deduce from these two facts 
that $B$ is generated by the conjugates of $B_0$ by the powers of $g_-$.

We are left with showing 
that $G_-$ is not an ascending HNN-extension with a base group of the form 
$B([\varepsilon , b- \varepsilon]; A, P)$ with $\varepsilon > 0$
and some stable letter $g_*$.
Consider an element $g_* \in G_-$  that generates a complement of $B$ in $G_-$.
Then $\sigma_-(g_*) \sigma_+(g_*) = 1$. 
So $g_*$ either moves points near 0 towards 0 and points near $b$ away from $b$,
or it moves points near 0 away from 0 and points near $b$ towards $b$.
The group $G_-$ thus cannot be an ascending HNN-extension with a base group 
whose elements have supports in an interval $[ \varepsilon, b- \varepsilon]$ and stable letter $g_*$.
\smallskip

(iii) Let $g_1$ be an element of $G_1$. 
Then $\sigma_-(g_1) = 1$ 
and so $g_1$ is the identity on a small interval $[0, a_1]$ with $a_1 > 0$.
So $G_1$ cannot be generated by a subgroup $B([a_1, b- a_1];A,P)$
with $a_1> 0$ and an additional element $g_1 \in G_1$.
\smallskip

We move on to the second part of the proof.
In the first part each of the three groups $G_+$, $G_-$ and $G_1$
has been shown to enjoy a characteristic property.
These properties involve the realizations of the groups as groups of PL-homeomorphisms. 
To deduce the claim that no two of the three groups are isomorphic,
we shall use these characteristic properties, 
along with Theorem \ref{TheoremE04}, Corollary \ref{CorollaryE5} 
and Corollary \ref{crl:Respecting-attractors}.
Let $G$, $\bar{G}$ be two of the groups $G_+$, $G_-$ and $G_1$,
not necessarily distinct, and let $\alpha \colon G \iso \bar{G}$ be an isomorphism.
Part (i) of Corollary \ref{CorollaryE5} guarantees 
that $\alpha$ maps the subgroup $B$ of $G$ 
onto the subgroup $B$ of $\bar{G}$.
Now both $G$ and $\bar{G}$ are generated by the subgroup $B$ 
and an additional element $g$, respectively $\bar{g}$;
so $\alpha(g) \in B \cdot \bar{g} \cup B \cdot \bar{g}^{-1}$.
The isomorphism $\alpha$ is induced by conjugation 
by some homeomorphism $\varphi \colon ]0,b[\, \iso \, ]0,b[$
(by Theorem \ref{TheoremE04}).

We compare now $\alpha(g) = \varphi \circ g \circ \varphi^{-1}$ 
with the elements in $B \cdot \bar{g}\, \cup \, B \cdot \bar{g}^{-1}$.
Suppose first that $G = G_1$ and choose $g = g_1$.
Then $g_1$ and hence $\alpha(g_1)$ are the identity near 0, 
while no element in $G_+ \smallsetminus B$ or in $G_- \smallsetminus B$
has this property.
So $G_1$ is neither isomorphic to $G_+$ nor to $G_-$.
Suppose next that $G = G_+$ and choose $g = g_+$.
If $\alpha$ is increasing, 
Corollary \ref{crl:Respecting-attractors} implies 
that the right-hand derivative of $\alpha(g_+)$ in 0 is smaller then 1;
the analogous result for the right end point $b$ proves, in addition, 
that $\alpha(g_+)'(b_-) < 1$.
So $\bar{G} = G_+$.
If $\alpha$ is decreasing, 
say induced by conjugation by $\varphi$,
replace $\varphi$  by $\imath \circ \varphi$ 
where $\imath$ is the reflection at the midpoint of $[0,b]$.
It then follows, as before, that $\bar{G} = G_+$.
\end{proof}

\begin{lem}
\label{lem:Piecing-together}
\index{Subgroup B([a,c];A,P)@Subgroup $B([a,c];A,P)$!generating sets}%
\index{Theorem \ref{TheoremA}!consequences}%
Let $A$ and $P$ be arbitrary and consider elements $a_1 < a_2 < c_ 1< c_2 $ in $A$.
Then the group $B([a_1,c_2];A, P)$ is generated by the set 
\[
B([a_1,c_1];A , P) \cup B([a_2,c_2];A, P).
\]
\end{lem}

\begin{proof}
Set $B =  B([a_1,c_2];A, P)$ and $B_i = B([a_i,c_i];A, P)$ for $i \in \{1,2\}$.
We intend to show that every element of $B$ lies in the complex product $B_1 \circ B_2$.

Let $g$ be a given element of $B$. 
Then $g$ fixes $a_1$ and so $f(t)- t \in IP \cdot A$ for every $t \in [a_1, c_2]$ 
(by Theorem \ref{TheoremA}). 
\index{Construction of PL-homeomorphisms!applications}%
Next,
choose $b$ and $b'$ in $IP \cdot A$ with $a_2 < b < b' < c_1$
and construct then, with the help of Theorem \ref{TheoremA}, 
\index{Theorem \ref{TheoremA}!consequences}%
an element $h \in G(\R;A, P)$ satisfying the requirements
$h(b) = g(b)$ and $h(b') = b'$.
Define, finally,  $g_1 \in G(\R; A,P)$ to be the PL-homeomorphism 
that coincides with $g$ on $[a_1, b]$,
with $h$ on $[b, b']$ and that  is the identity outside of $[a_1, b']$.
Then $g_1$ belongs to $B_1$,
while $g_2 = g_1^{-1} \circ g$ lies in $B_2$ and  so $g = g_1 \circ g_2$.
\end{proof}

\begin{remark}
\label{remark:ssec-18.7}
\footnote{This remark is due to R.\ Bieri and R.\ Geoghegan.}
\index{Bieri, R.}%
\index{Geoghegan, R.}%
If $\tilde G = G(I;A,P)$ is \emph{finitely generated} 
the subgroups $G_+$ and $G_-$ are finitely generated, too.
Indeed,
by Lemma \ref{lemma:The-three-groups-are-not-isomorphic}
each of these groups is generated 
by a subgroup of the form $G([a_1, b-a_1]; A,  P)$ and an additional element,
and the subgroup $G([a_1, b-a_1];A,P)$ is isomorphic to $\tilde G$
by Theorem \ref{TheoremE07}.
\index{Theorem \ref{TheoremE07}!consequences}%

If $\tilde G$ is \emph{finitely presented}
$G_+$ is finitely presented, 
being an ascending HNN-extension with base group $G([a_1,b-a_1];A,P)$ 
isomorphic to $\tilde G$ (by Theorem \ref{TheoremE07}),
\index{Theorem \ref{TheoremE07}!consequences}%
and stable letter generating a complement of $B(I;A,P)$.  
The group $G_-$, however, 
cannot be finitely presented.
Indeed, 
$G_-$ is a subgroup of $G(\R;\R_{\add}, \R^\times_{>0})$
and so it contains no non-abelian free subgroup (by \cite[Thm.\ 3.1]{BrSq85}).
If it admitted a finite presentation, 
the proof of Theorem A in \cite{BiSt78} would therefore allow us to find a 
\emph{finitely generated} subgroup $B_0$ in the kernel of $\sigma_- \colon G_- \epi P$
and a stable letter $x$ 
so that $G_-$ is an ascending HNN-extension with base group $B_0$.
\index{Bieri, R.}%
\index{Strebel, R.}%
Now,
the kernel of $\sigma \restriction{G_-}$ is nothing but the subgroup $B = B(I;A,P)$ 
and every finitely generated subgroup $B_1$ of $B$ 
is contained in $B_2 =G([\varepsilon, b-\varepsilon];A,P)$ 
for some $\varepsilon > 0$.
The stable letter $x$ would therefore have 0 and 1 as attracting fixed points
in contradiction to the fact 
that $B \cdot g_- \,\cup\,  B \cdot g_-^{-1}$ does not contain such an element.  
\end{remark}

\subsubsection{Investigation of subgroups of $\Z^2$ with finite index}
\label{sssec:18.8a}
\index{Group G([a,c];A,P)@Group $G([a,c];A,P)$!classification of subgroups with finite index|(}%
\index{Classification!of subgroups with finite index|(}%
%
If $G$ is in class (II), the subgroup  $Q$ has finite index in $\Z^2$.
We collect here some properties of such subgroups $Q$ 
that will be used in the next section.

Let $Q \subseteq \Z^2$ be a subgroup of finite index.
Then $Q$ gives rise to four groups
\[
Q_1= Q \cap (\Z \times \{0\}), \quad
Q_2 = Q \cap (\{0\} \times \Z), \quad
\im (\pi_1 \colon Q  \epi \Z) \text{ and } \im (\pi_2 \colon Q  \epi \Z);
\]
here $\pi_1$ and $\pi_2$ denote the canonical projections of $\Z^2$ onto its two factors.
All four groups are infinite cyclic; 
let $(m', 0)$, $(0,n')$ and $m$, $n$ be their positive generators.
The group $\inner(Q) = Q_1 \oplus Q_2$ is then a subgroup of $Q$
while $\Outer(Q ) =  \Z m \oplus \Z n$ contains $Q$;
these subgroups will be called  
the \emph{inner} and \emph{outer rectangles} of $Q$.
\footnote{These expressions are borrowed from \cite{BlWa07}; 
see the comments in \ref{sssec:Notes-ChapterE-Bleak-Wassink}.} 
\label{notation:Group inner(Q)}%
\label{notation:Group outer(Q)}%
\index{Subgroup inner(Q)@Subgroup $\inner(Q)$!definition|textbf}%
\index{Subgroup outer(Q)@Subgroup $\Outer(Q)$!definition|textbf}%

We consider now  the images of $\inner(Q)$ and $Q$ under the projection $\pi_1$.
The first image is $\Z \cdot m'$, 
the second  $\Z \cdot m$; 
as $\pi_1(\inner(Q))$ is a subgroup of $\pi_1(Q)$,
the integer  $m'$ is a multiple of $m$, 
say $m' = c \cdot m$.
Now the kernels of the restrictions of $\pi_1$ to the subgroups $\inner(Q)$ 
and to $Q$ 
are both equal to $\{0\} \times \Z \cdot  n'$; 
the integer $c$ is thus nothing but the index of $\inner(Q)$ in $Q$.
Similarly, one finds that $n' = c \cdot n$.

Next, we construct generating sets of $Q$.
By the previous paragraph the group $Q$ is an extension of 
$Q_1 = \Z \cdot m' \times \{0\}$ 
by a cyclic group generated by an element of the form $(j, n)$;
so $Q$ is generated by $(m',0)$ and $(j,n)$.
Similarly, $Q$ is generated by $(0,n')$ 
and an element of the form $(m, \ell)$.
There exist therefore integers $x$, $y$ and $d$, $e$ 
that satisfy the equations
\[
(j,n)= x \cdot  (0, n') + d\cdot ( m, \ell) \quad \text{and} \quad 
(m,\ell) = y \cdot (m', 0) + e \cdot (j,n).
\]
The first component of the first equation shows that $j = d \cdot m$;
the second component of the second equation implies that $\ell = e \cdot n$.
The second component of the first equation then states that  
\[
n = x \cdot n' + d \cdot \ell = x \cdot c \cdot n + d \cdot e \cdot n
\]
or, equivalently, that $1 = x \cdot c + d \cdot e$.
Two cases now arises.

If $c = 1$ 
then $Q = \inner(Q) = \Outer(Q)$ and $d$ and $e$ can be arbitrary integers;
if $c > 1$ then $\Z/\Z c$ is a non-trivial ring 
and so equation $1 = x \cdot c + d \cdot e$ states
that $d$ and $e $ are inverses of each other \emph{modulo} $c$.
Moreover,
the facts
that $m' = c \cdot m$, that  $j = d \cdot m$, and
that $Q$ is generated by $(m',0) = (c\cdot m, 0)$ 
and $(j,n) = (d \cdot m, n)$
implies that $d$ is  unique \emph{modulo} $c$.

The previous insights are summarized in
\begin{lemma}
\label{lem:Summary-analysis-Q} 
\index{Subgroup inner(Q)@Subgroup $\inner(Q)$!properties}
Suppose $Q$ is a subgroup of $\Z \times \Z$ with finite index,
and set $Q_1= Q \cap (\Z \times \{0\})$ 
and $Q_2 = Q \cap (\{0\} \times \Z)$.
Then there exists positive integers $m$, $n$ and $c$, $d$, $e$ with the following properties:
\begin{enumerate}[(i)]
\item  $\im (\pi_1 \colon Q \epi \Z) = \Z \cdot m$ 
and $\im (\pi_2 \colon Q \epi \Z) = \Z \cdot n$;
\item $c$ is the index of  $\inner(Q) = Q_1 \oplus Q_2$ in $Q$;
\item $Q_1 = \Z (c\cdot m,0)$ and $Q_2 = \Z (0, c\cdot n)$;
\item $Q = Q_1 + \Z (d  \cdot m, n) =  Q_2 + \Z (m, e \cdot  n)$.
\end{enumerate}
The numbers $m$, $n$ and $c$ are uniquely determined by $Q$; 
the numbers $d$ and $e$ are unique modulo $c$. 
Moreover, if $c > 1$ then $1 \equiv d \cdot e \pmod{c}$.
\end{lemma}

\subsubsection{Isomorphism types of groups with $G/B$ free abelian of rank 2}
\label{sssec:18.8b}
%
Suppose $G$ and $\bar{G}$ are subgroups of $G([0,b];A,P)$ 
which contain both the subgroup $B = B([0,b];A,P)$ 
and have finite index in $G([0,b];A,P)$.
Let $Q$ and $\bar{Q}$ be the images of $G$ and $\bar{G}$ under the epimorphism
\begin{equation}
\label{eq:Projection-pi}
\index{Homomorphism!16-pi@$\pi$}%
\xymatrix{\pi \colon G(I;A,P) \ar@{->>}[r]^-{(\sigma_-,\sigma_+)} 
&
P \times P \,\ar@{>->>}[r]^-\omega &\Z^2};
\end{equation}
here $\omega$ takes $(p^i, p^k)$ to $(i, k)$ and $p$ denotes the generator of $P$ with $p < 1$.

The aim of this section is to obtain necessary, as well as sufficient, conditions, 
for the existence of an isomorphism $\alpha \colon G \iso \bar{G}$. 
We begin with a remark on the strategy of our approach.
Suppose $G$ and $\bar{G}$ are subgroups of $G([0,b];A,P)$ as described before,
and $\alpha \colon G \iso \bar{G}$ is an isomorphism.
Since $G$ and $\bar{G}$ contain the subgroup $B = B(I; A, P)$,
Corollary \ref{CorollaryE5} implies then 
that $\alpha$ induces an isomorphism $\alpha_* \colon G/B \iso \bar{G}/B$.
\index{Theorem \ref{TheoremE04}!consequences}%
Next,
let $Q$ and $\bar{Q}$ denote the images of $G$ and $\bar{G}$, respectively,
under the projection $\pi$.
Then $\pi$  induces isomorphisms $\pi_* \colon G/B \iso Q$ 
and $\bar{\pi}_{* }\colon \bar{G}/B \iso \bar{Q}$.
Let $\alpha_* \colon Q \iso \bar{Q}$ denote the isomorphism 
that renders the square
\begin{equation}
\label{eq:Commuativity-square-new}
\xymatrix{
G/B \ar@{->}[r]^-{\pi_\star} \ar@{->}[d]^-{\alpha}
&Q \ar@{->}[d]^-{{\alpha}_*}\\
\bar{G}/B \ar@{->}[r]^-{\bar{\pi}_\star} &\bar{Q}
}
\end{equation}
commutative. 

The isomorphism $\alpha$ is induced 
by conjugation by an auto-homeo\-mor\-phism $\varphi$ of $\Int(I)$ 
(by Theorem \ref{TheoremE04}); 
\index{Theorem \ref{TheoremE04}!consequences}%
let $\tilde{\varphi}$  denote the unique extension of $\varphi$ to $I = [0, b]$.
Assume now that $\tilde{\varphi} $ is \emph{increasing} 
and that its \emph{right-hand derivative} in $0$ 
and its \emph{left-hand derivative} in  $b$ exist.
The chain rule permits one then to deduce 
that $\bar{Q} = Q$ and that $\alpha_* = \id_Q$.
This conclusion holds without the hypothesis 
on the differentiability of $\tilde{\varphi}$,
\emph{provided} $\Outer(Q)$ and $\Outer(\bar{Q})$ are equal to $\Z^2$.

More precisely, one has
\begin{lemma}
\label{lem:Isomorphisms-special-case}
\index{Homomorphism!09-iota@$\iota$}%
\index{Homomorphism!09-iota-bar@$\bar{\iota}$}%
Assume the rectangles $\Outer(Q)$ and $\Outer(\bar{Q})$ coincide with $\Z^2$,
and let $\alpha \colon  G \iso \bar{G}$ be an isomorphism.
Then the parameter $c$ of $Q$ coincides 
with the parameter $\bar{c}$ of $\bar{Q}$.
Moreover,
if $\alpha$ is increasing
then $\bar{Q} = Q$ and $\alpha_* = \id$, 
whereas $\alpha_* = \bar{\iota} \restriction {Q}$ 
and $\bar{Q} = \bar{\iota}(Q)$
if $\alpha$ is decreasing.
(Here $\bar{\iota}$ denotes the automorphism of $\Z^2$ 
mapping $(i, k)$ to $(k,i)$.)
\end{lemma}
\index{Subgroup outer(Q)@Subgroup $\Outer(Q)$!application}%
\index{Homomorphism!09-iota-bar@$\bar{\iota}$}%

\begin{proof}
Let $\varphi $ be the auto-homeomorphism of $]0,b[$ inducing $\alpha$;
it exists by Theorem \ref{TheoremE04}.
\index{Theorem \ref{TheoremE04}!consequences}%
\emph{Assume first that $\varphi$ is increasing}
and consider the subgroups
\begin{equation}
\label{eq:Definitions-H1-and-of-barH-1}
H_ 1 = \{g \in G \mid g'(b_-) = 1 \}
\quad \text{and} \quad 
\bar{H}_ b = \{\bar{g} \in \bar{G} \mid \bar{g}'(b_-) = 1 \}
\end{equation}
of $G([0,b];A,P)$.
The left-hand derivative $g'(b_-)$ of an element $g$ equals 1 if, and only if,
$g$ is the identity in a neighbourhood of $1$. 
Since $\varphi$ is increasing,
the image $\alpha(H_1) = \varphi \circ H_1 \circ \varphi^{-1}$ 
of $H_1$ coincides therefore with $\bar{H_1}$. 
The definitions of $H_1$ and of the projection $\pi$,
on the other hand, imply
that $\pi(H_1)$ equals $Q_1 = Q\, \cap\, ( \Z \times \{0\})$.
The latter subgroup can be identified with the help of claim (iii) in Lemma 
\ref{lem:Summary-analysis-Q}:
in view of the assumption on $\Outer(Q)$ one has $m = n = 1$.
It follows that
\[
Q_1 = \pi(H_1) = Q\, \cap\, ( \Z \times \{0\})  = \Z (c,0);
\]
here $c$ denotes the positive integer characterized in claim (ii) of the lemma.

So far we dealt with the subgroup $H_1$ of $G$.
The situation for the subgroup $\bar{H}_1$, 
defined in equation \eqref{eq:Definitions-H1-and-of-barH-1},
is entirely analogous and so
\[
\bar{Q}_1 = \pi(\bar{H}_1)= \bar{Q}\; \cap\; ( \Z \times \{0\})  = \Z (\bar{c}, 0)
\]
by statement (iii) of Lemma \ref{lem:Summary-analysis-Q}.

Now we invoke the commutativity of the square 
\eqref{eq:Commuativity-square-new}.
By the first paragraph of the proof, 
the isomorphism $\alpha \colon G \iso \bar{G}$ maps $H_1$ onto $\bar{H}_1$;
by the commutativity of the square this fact implies
that $\alpha_* \colon Q \iso \bar{Q}$ maps $Q_1$ onto $\bar{Q}_1$.
As $Q_1$ is infinite cyclic generated by $(c,0)$ 
and $\bar{Q}_1$ 
is generated by $(\bar{c}, 0)$, 
and as the parameters $c$ and $\bar{c}$ are positive 
one infers 
that  $\alpha_*$ maps $(c,0)$ onto  $(\bar{c}, 0)$.
Now to the subgroups
\begin{equation}
\label{eq:Definitions-H2-and-of-barH-2}
H_ 2 = \{g \in G \mid g'(0_+) = 1 \}
\quad \text{and} \quad 
\bar{H}_ 2 = \{\bar{g} \in \bar{G} \mid \bar{g}'(0_+) = 1 \}
\end{equation}
of $G([0,b];A,P)$.
Statement (iii) of Lemma \ref{lem:Summary-analysis-Q} implies for them 
that
\[
Q_2 = \pi(H_2) = Q\, \cap\, ( \{0\} \times \Z)  = \Z (0,c)
\quad \text{and} \quad
Q_2 = \pi(\bar{H}_2) =  \Z (0,\bar{c}).
\]

The proof is now quickly completed.
By the first paragraph,  $\alpha$ maps $H_1$ onto $\bar{H}_1$;
similarly, one proves that $\alpha(H_2) = \bar{H}_2$.
It follows that 
\[
\alpha_*(\inner(Q) )
= 
\alpha_* (Q_1 + Q_2) = \bar{Q}_1 +  \bar{Q}_2 = \inner(\bar{Q}),
\]
and thus $\alpha_*$ induces an isomorphism of 
$Q/\inner(Q)$ onto $\bar{Q}/\inner(\bar{Q})$.
In view of claim (ii) in Lemma \ref{lem:Summary-analysis-Q},
the first of these groups has order $c$,
the second order $\bar{c}$, whence $c = \bar{c}$.
But if so, $\alpha_*$ fixes every point of the subgroup 
$\inner(Q) = \Z(c,0) \oplus \Z(0,c)$; 
as $Q$ is torsion-free,
$\alpha_*$ fixes therefore every point $Q$
and so it is the identity on $Q$.
\smallskip

Assume, \emph{secondly},
that $\varphi $ is \emph{decreasing}
and let $\imath  \colon ]0,b[ \iso ]0,b[$ be  the reflection 
at the midpoint $b/2$ of the interval $[0,b]$.
Then $\psi = \imath \circ \varphi$ is increasing; 
let $\alpha_1 \colon G \iso \act{\vartheta}\bar{G}$ 
be the isomorphism induced by conjugation by $\psi$
and set $\tilde{Q}= \pi(\act{\imath}\bar{G})$.
Then $\tilde{Q} = Q$ by the previous part 
and so $\alpha_* = \iota \restriction{Q}$ and $\bar{Q} = \iota(Q)$.
\end{proof}

We are now set for establishing the announced characterization.
\begin{theorem}
\label{thm:Isomorphism-types-class-II}
Suppose $P$ is cyclic
and $G$, $\bar{G}$ are subgroups of $G([0,b];A,P)$ with finite index,
both containing $B$.
Set $Q = \pi(G)$ and $\bar{Q} = \pi(\bar{G})$,
and let $m$, $n$ and $c$, $d$, $e$ be the parameters associated to $Q$,
and defined in section \ref{sssec:18.8a}, 
and let $\bar{m}$, $\bar{n}$ and $\bar{c}$, $\bar{d}$, $\bar{e}$ 
be the analogous parameters associated to $\bar{Q}$. 
The following statements are then equivalent:
\begin{enumerate}[(i)]
\item $G$ and $\bar{G}$ are isomorphic;
\item $c = \bar{c} = 1$, or $c = \bar{c} > 1$ 
and, either $d \equiv \bar{d} \pmod{c}$ or $d \equiv \bar{e} \pmod{c}$.
\end{enumerate}
\end{theorem}
\index{Group G([a,c];A,P)@Group $G([a,c];A,P)$!subgroups of finite index}
\index{Group G([a,c];A,P)@Group $G([a,c];A,P)$!isomorphic subgroups}
\index{Thompson's group F@Thompson's group $F$!subgroups of finite index}
\begin{proof}
We begin with a remark.
By assumption,
the group $Q = \pi(G)$ has the parameters $m$, $n$ and $c$, $d$, $e$,
and so it is generated by the ordered pairs $(c \cdot m,0)$ and $(d \cdot m, n)$
(see items (iii) and (iv) in Lemma \ref{lem:Summary-analysis-Q}).
Define $\check{Q}$ to be the subgroup of $\Z^2$ generated by 
$(c, 0)$ and $(d,1)$.
Then $\Outer(\check{Q}) = \Z^2$
(by the final claim of Lemma \ref{lem:Summary-analysis-Q})
and $Q$ is the image of $\check{Q}$ under the ``rescaling'' map 
$ \zeta_{m,n} \colon \Z^2 \to \Z^2$
sending $(x,y)$ to $(mx, ny)$.
We assert 
that  $(\zeta_{m, n})_* \colon \check{Q} \iso Q$ can be lifted to an isomorphism 
$\eta_* \colon \check{G} \iso G$ 
with $\check{G}$ a subgroup of $G([0,b];A,P)$ containing $G$.
Indeed,
let $\mu_m$ and $\nu_n$ be the monomorphisms of  $G([0,b];A,P)$ 
constructed in section \ref{sssec:18.5}
and set
\[
\eta = \mu_m \circ \nu_n \colon G([0,b];A, P) \mono G([0,b];A,P).
\]
The image of $\eta$ contains the subgroup $B$  
and formulae \eqref{eq:Morphism-mu-m} and \eqref{eq:Morphism-nu-n}
reveal that $\pi(\im \eta)$ is nothing  but $\Z m \times \Z n = \Outer(Q)$
and so $\im \eta$ contains $G$.

Set $\check{G} = \pi^{-1}(\check{Q})$;
then $\pi(\check{G}) = \check{Q}$.
Let $\eta_* \colon \check{G} \iso G$ be the isomorphism 
obtained from $\eta$ by restricting both the domain and the range suitably.
The isomorphisms  $\eta_*$ and $(\zeta_{m,n})_*$ 
render commutative the left square 
displayed in \eqref{eq:Relation-eta-bar-eta}.
\begin{equation}
\label{eq:Relation-eta-bar-eta}
%
\xymatrix{
\check{G} \ar@{->}[r]^-{\check{\pi}_*} \ar@{->}[d]^-{\eta_*}
&\check{Q}\ar@{->}[d]^-{(\zeta_{m,n})_*}\\
G \ar@{->}[r]^-{\pi_\star} &Q
}
\qquad \qquad
\xymatrix{
\hat{G} \ar@{->}[r]^-{\hat{\pi}_*} \ar@{->}[d]^-{\bar{\eta}_*}
&\check{Q}\ar@{->}[d]^-{(\zeta_{\bar{m},\bar{n}})_*}\\
\bar{G} \ar@{->}[r]^-{\pi_\star} &\bar{Q}
}
\end{equation}
In this square,
$\check{\pi}_* \colon \check{G} \epi \check{Q}$ and $\pi_*  \colon G \epi Q$ 
denote the epimorphisms obtained from $\pi$ 
by restricting domain and range.
In the same way, 
one can construct a monomorphism $\bar{\eta} = \mu_{\bar{m}} \circ \nu_{\bar{n}}$ 
and a subgroup $\hat{G}$ in $G([0,b];A, P)$ 
that is mapped by $\bar{\eta}$ onto $\bar{G}$
and that renders commutative the square shown on the right.
\index{Endomorphism mum@Endomorphism $\mu_m$!applications}
\index{Endomorphism nun@Endomorphism $\nu_n$!applications}
\index{Isomorphisms!construction}
\smallskip

$(i) \Rightarrow (ii)$.
Let $\alpha \colon G \iso \bar{G}$ be an isomorphism 
and set 
\[
\alpha_1 = (\bar{\eta}_*)^{-1} \circ \alpha \circ \eta_* \colon \check{G} \longrightarrow \hat{Q}.
\]
Then $\alpha_1$ is an isomorphism and it is increasing precisely if $\alpha$ is so.
\footnote{The monomorphisms $\mu_m$  and $\nu_n$ are induced by conjugation 
by homeomorphisms which are obviously orientation preserving; 
see section \ref{sssec:18.5}.}
Lemma \ref{lem:Isomorphisms-special-case} shows 
that  $c$ and $\bar{c}$ coincide and
that $\check{Q} =\hat{Q}$ if $\alpha$ is increasing 
or that $\check{Q} = \iota(\hat{Q})$ if $\alpha$ is decreasing.
Moreover, if $c > 1$ and $\check{Q} = \hat{Q}$ then  $d \equiv \bar{d} \pmod{c}$;
if $c > 1$ and $\check{Q} = \iota(\hat{Q})$ then $d \equiv \bar{e} \pmod{c}$. 
\smallskip

$(ii) \Rightarrow (i)$.
If $c = \bar{c} = 1$ then $\check{Q} = \hat{Q}$; 
if $c = \bar{c} > 1$ and $d \equiv \bar{d} \pmod{c}$
the equality $\check{Q} = \hat{Q}$ holds again.
In both cases, 
the groups $\check{G}$ and $\hat{G}$ are equal;
the isomorphism $\bar{\eta}_* \circ \eta^{-1}_*$ is therefore defined 
and it maps $G$ onto $\bar{G}$.
Assume, finally, that $c = \bar{c} > 1$ and that $d \equiv \bar{e} \pmod{c}$.
Then $\hat{Q} = \iota(\check{Q})$.
Now $G([0,b];A,P)$ admits an automorphism that corresponds under $\pi$ to this exchange of factors,
namely the automorphism $\iota$  induced by conjugation 
by the reflection $\imath$ at the midpoint of the interval $[0,b]$.
It follows that $\hat{G} = \act{\imath}\check{G}$ 
and so the composition
$\bar{\eta}_* \circ \iota \circ \eta^{-1}_*$
maps $G$ onto $\bar{G}$.
\end{proof}

\begin{remark}
\label{remark:Analogues-of-thm-ismorphism-types-class-II}
Theorem  \ref{thm:Isomorphism-types-class-II} shows, in particular,
that a group of the form $G =G(I;A, P)$ with $I = [0,b]$, 
endpoint $b \in A$ and cyclic slope group $P$
has many subgroups of finite index 
that contain $B([0,b];A, P)$ and are isomorphic to $G$.
By Supplement \ref{SupplementE11New},
\index{Supplement \ref{SupplementE11New}!consequences}%
this cannot happen 
if $P$ is \emph{non-cyclic} and $I$ is one of the intervals $\R$, $[0,\infty[$ 
or $[0,b]$ with $b \in A$.
\end{remark}
\index{Group G([a,c];A,P)@Group $G([a,c];A,P)$!classification of subgroups with finite index|)}%
\index{Classification!of subgroups with finite index|)}%
%
%
\section{Automorphism groups of groups with cyclic slope groups}
\label{sec:19}
%
In this section,
we have assembled our findings 
on the automorphism groups of groups $G =G(I;A,P)$ 
with cyclic slope group $P$.
These results show
that the complexity of the automorphism groups increases markedly 
if one passes from groups with $I$ a line, 
to groups with $I$ a closed half line,
to groups with $I$ an open half line.
The outer automorphism groups are an indicator of this growing complexity:
if $I$ is the line, $\Out G $ is abelian and isomorphic to $\Aut(A)/P$;
if $I$ is a closed half line, $\Out G $ contains a copy of $G$;
if, thirdly, $I$ is an open half line,
$\Out G $ contains the cartesian product $\prod_{n \in \N} G_n$ of countably many copies of $G$.
The three cases share, on the other hand, a common feature: 
\emph{every automorphism is induced  by conjugation by a PL-homeomorphism} defined on $\Int(I)$.
We have not been able to determine
whether this property continues to hold 
if $I$ is a compact interval with endpoints in $A$;
all we can offer is a satisfactory description of the subgroup of $\Out G $ 
induced by conjugation by PL-homeomorphisms.
\footnote{See sections \ref{sssec:Notes-ChapterE-Brin96}
and  \ref{sssec:Notes-ChapterE-BrGu98} for updates.}

The section divides into three parts.
In the first one,
some subgroups of the group of all PL-homeomorphisms of $\R$ are introduced
that help one to describe the automorphism groups.
The second part contains our results on the automorphism groups 
for $I$ a line or a half line.
The third part, finally, collects our findings 
on the automorphism groups of groups $G(I;A,P)$ 
with $I$ a compact interval. 
%
\subsection{Preliminaries}
\label{ssec:19.1}
%
We begin with a reminder.
Let $(I,A,P)$ be an arbitrary triple 
that satisfies the non-triviality assumptions \eqref{eq:Non-triviality-assumption} 
and set $G = G(I;A,P)$. 
By Theorem \ref{TheoremE04},
the automorphism group of $G$ is then isomorphic to the normalizer of $G$ 
in the ambient group $\Homeo(\Int(I))$, \label{Homeo(J)}
provided one views the elements of $G$ as homeo\-mor\-phisms of $\Int(I)$.
We shall denote this normalizer by $\Autfr G$. 
\label{notation:Autfr}%
Conjugation induces then an isomorphism 
\begin{equation}
\label{eq:19.1}
\index{Homomorphism!08-theta@$\vartheta$}%
\vartheta \colon \Autfr G \iso \Aut G.
\end{equation}

The automorphism group of $G$ will be described 
by giving details about its normalizer $\Autfr G$ of $G$.
In the discussion of this normalizer
various subgroups turn out to be useful,
in particular the two subgroups defined next.

Let $J$ be a non-empty \emph{open} interval of the real line,
let $A$ be a subgroup of $\R_{\add}$ and let $P$ a subgroup of $\Aut_o(A)$.
Define $G_\infty(J;A,P)$ to be the group 
\label{definition:Ginfty(J;A,P)}%
\index{Group Ginfty(J;A,P)@Group $G_\infty(J;A,P)$!definition|textbf}%
consisting all (finitary or infinitary) PL-homeomorphisms  
$f \colon J\iso J$ 
which map $A \cap J$ onto itself, have \emph{slopes in} $P$,
singularities in $A$, and no more than finitely many singularities 
in any given compact subinterval of $J$.

In the sequel, 
$G_\infty = G_\infty(\R;A,P)$ will play the role of a containing group.
One subgroup that is worth being singled out at this stage
is the subgroup of \emph{periodic elements} with period $\pfr \in A_{>0}$;
it is defined by
\begin{equation}
\label{eq:Definition-periodic-subgroup}
\index{Group Ginfty(R;pfr,A,P)@Group $G_\infty(\R;\pfr;A,P)$!definition|textbf}%
\index{Group Autfr G([0,1];A,P)@Group $\Autfr G([0,1];A,P)$!subgroups}%
G_\infty(\R;\pfr;A,P) = 
\{f\in G_\infty \mid f(t+\pfr) = f(t) + \pfr \text{ for all }t \in \R\}
\end{equation}

For every period $\pfr \in A_{>0}$, 
the group $G_\infty(\R;\pfr;A,P)$ gives rise to a group 
made up of homeomorphisms of the circle $\R/\Z \pfr$.
Indeed,
the covering map $\pi \colon \R \epi \R/\Z  \pfr$ induces,
by passage to the quotients,
a homomorphism
\begin{equation}
\label{eq:19.2}
\index{Homomorphism!10-kappa@$\kappa$}%
\kappa \colon G_\infty(\R;\pfr;A,P) \to \Homeo(\R/\Z \pfr)
\end{equation}
into the group of auto-homeomorphisms of the circle $\R/\Z  \pfr$.  
The image of $\kappa$ will be denoted by $T(\R/\Z \pfr;A,P)$ 
\label{definition:T(R/Zpfr;A,P)}%
\index{Group T(R/Zpfr;A,P)@Group $T(\R/\Z \pfr;A,P)$!definition|textbf}%
and called the  \emph{group of finitary PL-homeomorphisms of the circle 
$\R/\Z \pfr$ with slopes in $P$ and singularities in $A/\Z\pfr$}.  
The kernel of $\kappa$ is infinite cyclic, 
generated by the translation with amplitude $\pfr$.
\begin{remark}
\label{remark:Homomorphism-gamma}%
\index{Thurston, W. P.}%
\index{Thompson's group T@Thompson's group $T$!Thurston's characterization}\index{Thompson, R. J.}%
W. P. Thurston discovered 
that $T(\R/\Z;\Z[1/2],\gp(2))$ is isomorphic to Thompson's simple groups $T$.  
\label{definition:T}%
Its generalization $T(\R/\Z \pfr;A,P)$, however, is not always simple.  
To see this,
we go back to the group $G_\infty(\R;A,P)$.
Each of its elements $f$ maps $A$ onto itself.
Moreover, if $a_1 < a_2$ are in $A$ 
then $f$ has only finitely many singularities in the compact interval $[a_1, a_2]$,
and so $f(a_2) - f(a_1)$ is congruent to $a_2-a_1$ \emph{modulo} $IP \cdot A$
by Theorem \ref{TheoremA}.
\index{Construction of PL-homeomorphisms!applications}%
This finding can be restated by saying that
\begin{equation}
\label{eq:Consequence-TheoremA}
f(a_1) - a_1 \equiv f(a_2)-a_2 \pmod{IP \cdot A} 
\quad \text{for all} \quad
(a_1, a_2) \in A^2.
\end{equation}

The above findings allow us to define a homomorphism $\gamma$ of 
$G_\infty(\R;A,P)$ onto the abelian group $A/(IP \cdot A)$.
To do so, we fix $a_0 \in A$ and consider the function
\begin{equation}
\label{eq:Definition-gamma}
\index{Group Ginfty(J;A,P)@Group $G_\infty(J;A,P)$!abelianization}%
\index{Quotient group A/IPA@Quotient group $A/(IP \cdot A)$!applications}%
\index{Homomorphism!03-gamma@$\gamma$}%
\gamma \colon G_\infty(\R; A,P) \to A/(IP\cdot A), 
\quad f \longmapsto (f(a_0) - a_0 )+ IP \cdot A.
\end{equation}
Statement  \eqref{eq:Consequence-TheoremA} shows 
that the function $\gamma$ does not depend on the choice $a_0$, 
a fact that implies that $\gamma$ is a homomorphism.
Indeed, if $f_1$, $f_2$ are elements of $G_\infty(\R; A,P)$,
the following calculation is valid:
\begin{align*}
\gamma (f_2 \circ f_1)
=
(f_2 \circ f_1) (a_0) -a_0 
&= (f_2(f_1(a_0)) - f_1(a_0)) + (f_1(a_0) - a_0)
\\
&=
f_2(f_1(a_0)) - f_1(a_0)) + \gamma(f_1)
=
\gamma(f_2)+ \gamma(f_1).
\end{align*}
As $G_\infty(\R; A,P) $ contains every translation with amplitude in $A$,
the homomorphism $\gamma$ is surjective.

\emph{Assume now that $\pfr$ lies in $IP \cdot A$}
and let $\gamma_{\pfr} \colon G_\infty(\R;\pfr;A,P) \to A/(IP \cdot A)$
be the restriction of $\gamma$ to the subgroup 
$G_\infty(\R;\pfr;A,P)$ of $G_\infty(\R; A,P) $.
Since the translation $h_\pfr$ with amplitude $\pfr$ 
lies in the kernel of $\gamma_\pfr$
the assignment
\begin{equation}
\label{eq:Definition-gammap-bar}
\index{Homomorphism!03-gamma-bar-p@$\bar{\gamma}_\pfr$}%
\bar{\gamma}_\pfr \colon T(\R/ \Z \pfr; A,P) \epi A/(IP \cdot A), 
\quad \bar{f} \longmapsto \bar{f}(a_0 + \Z \pfr) - (a_0 + \Z \pfr)+ IP \cdot A.
\end{equation}
is licit and defines an epimorphism onto the abelian group  $A/(IP \cdot A)$
\end{remark}
\index{Group T(R/Zpfr;A,P)@Group $T(\R/\Z \pfr;A,P)$!abelianization}%
\index{Submodule IPA@Submodule $IP \cdot A$!significance}%
\index{Group Ginfty(R;pfr,A,P)@Group $G_\infty(\R;\pfr;A,P)$!properties}%
%
\subsubsection{Hypotheses imposed for the remainder of Section \ref{sec:19}}
\label{sssec:19.1a} 
From now on, $P$ denotes a cyclic group with generator $p > 1$.
We assume 
that $A$ contains the integers $\Z$ and that $[0,b]$ is the unit interval $[0, 1]$;  
by Theorem \ref{TheoremE07} 
\index{Theorem \ref{TheoremE07}!consequences}%
this last assumption entails no loss of generality.  
Note that $\pfr = p-1$ is then a positive element of $IP\cdot A$.
%
\subsection{Groups with interval the line or a half Line}
\label{ssec:19.2New}
In this section, 
we study the automorphism groups of those groups 
for which we can prove 
that every automorphism is given by conjugation by a PL-homeomorphism.
%
\subsubsection{The automorphism group of $G(\R;A,P)$}
\label{ssec:19.2a}
\index{Group G(R;A,P)@Group $G(\R;A,P)$!automorphism group}%
\index{Automorphism group of G(R;A,P)@Automorphism group of $G(\R;A,P)$!description}%
%
The automorphism group of $G(\R;A,P)$ has the form 
described by items (i) and (iii) of Corollary \ref{CorollaryE13}.
Indeed, 
by part (i) of Theorem \ref{TheoremE14} 
the group $\Autfr G(\R;A,P)$ 
is made up of those \emph{finitary} PL-homeomorphisms $f$ 
that map $A$ onto itself, have slopes in a coset $s_f \cdot P$,
with $s$ an element of 
\begin{equation}
\label{eq:19.3}
\Aut(A) = \{s\in \R^\times \mid s\cdot A = A\},
\end{equation}
and singularities in $A^2$.
In addition, 
the outer automorphism group of $G(\R;A,P)$ is isomorphic to $\Aut(A)/P$.
\index{Outer automorphism group of!G(R;A,P)@$G(\R;A,P)$}%
Note that $\Autfr G(\R;A,P) $ contains a familiar copy of $\Aut(A)$, 
the group of homotheties
\begin{equation*}
\Aut(A)^\wedge = \{s\cdot \id\mid s\in \Aut(A)\},
\end{equation*}
and so $\Autfr G(\R;A,P)  = G(\R;A,P)\cdot \Aut(A)^\wedge$.
(Note that  $\Autfr G(\R;A,P) $ has the same description
if $P$ is a non-cyclic subgroup of $\R^\times_{>0}$.)
\index{Group G(R;A,P)@Group $G(\R;A,P)$!outer automorphism group}%
\index{Group P cyclic@Group $P$ cyclic!consequences}%

%
\subsubsection{Automorphism group of $G([0, \infty[\,;A,P) $}
\label{sssec:19.2b}
\index{Group G([0,infty[;A,P)@Group $G([0, \infty[\;;A,P)$!automorphism group|(}
\index{Automorphism group of G([0,infty[;A,P)@Automorphism group of $G([0, \infty[\;;A,P)$!description|(}%
\index{Theorem \ref{TheoremE04}!consequences}%
We shall study both the automorphism group and the outer automorphism group 
of $G =G([0,\infty[\;;A,P)$. 
We begin with some preliminaries.
By Theorem \ref{TheoremE04} and part (ii) of Theorem \ref{TheoremE14},  
every element $\varphi \in \Autfr G $ is an increasing  PL-homeomorphism of $]0,\infty[$
that maps $A\, \cap \, ]0, \infty[$ onto itself,
with slopes in a coset $s_\varphi \cdot P$ and singularities in $A$;
the singularities may be infinite in number, 
but if so they accumulate only in 0.
The number $s_\varphi$ is an element of
\begin{equation*}
\Aut_o(A) = \{s\in \R^\times_{>0} \mid s\cdot A = A\}.
\end{equation*}
If follows that there exists a homomorphism 
\[
\label{notation:eta}
\index{Homomorphism!07-eta@$\eta$}%
\eta \colon \Autfr G  \to \Aut_o (A)/P
\] 
\index{Group Auto(A)@Group $\Aut_o(A)$!significance}%
which sends a PL-auto-homeomorphism $\varphi$ to the coset $s_\varphi \cdot P$;
\cf{}section \ref{sssec:Automorphism-Result-non-cyclic}.
This homomorphism is surjective;
indeed, $\Autfr G$ contains a copy of $\Aut_o(A)$ made up of the PL-homeomorphisms of the form
\begin{equation}
\label{eq:Definition-copy-Auto(A)}
\index{Subgroup Auto(A)@Subgroup $\Aut_o(A)$!significance}%
t \mapsto \begin{cases} 
s\cdot t &\text{ if } t \geq 0,\\ 1 \cdot t&\text{ if } t \leq 0, \end{cases}
\end{equation}
with $s \in \Autfr_o(A)$.

For the more detailed analysis of $\Autfr G$
it is convenient to replace $G$ by its image $\bar{G}$
\label{notation:bar(G)}%
under the embedding $\mu_{2,1} \colon G([0, \infty[\,;A,P) \mono G(\R;A,P)$ 
detailed in section \ref{sssec:Embedding-mu2}. 
This embedding is induced by the PL-homeomorphism  $\psi_{2,1}^{-1} \colon \,]0, \infty[\, \iso \R$
that is increasing, maps $A \; \cap \; ]0, \infty[$ onto $A$,
has slopes in $P$ and infinitely many singularities which,
however, accumulate only in 0. 
Moreover, 
the homomorphism $\eta$ gives rise to an epimorphism 
\[
\index{Homomorphism!07-eta-bar@$\bar{\eta}$}%
\bar{\eta} \colon \Autfr \bar{G}  \to \Aut_o (A)/P.
\] 

The elements in its kernel have therefore the properties listed in
\begin{lemma}
\label{lem:Properties-Autfr-bar-G}
\index{Group Ginfty(J;A,P)@Group $G_\infty(J;A,P)$!significance}%
Each element $\bar{\varphi}\in \ker \bar{\eta}$ lies in the group
$G_\infty(\R;A,P)$ (defined in section \ref{ssec:19.1})
and its singularities can only accumulate in $-\infty$.
\end{lemma}

The group $\Autfr \bar{G}$ contains three subgroups 
that we describe next.
Firstly,  the subgroup $\bar{G}$ itself; 
according to Lemma  \ref{LemmaE16}, 
it consists of all elements $\bar{g} \in G(\R;A,P)$ satisfying the restrictions
\begin{equation}
\label{eq:19.4}
\sigma_-(\bar{g}) = 1, \quad \tau_-(\bar{g}) \in \Z (p-1) \text{ and }  \rho(\bar{g}) \in \Aff(IP \cdot A, P);
\end{equation}
in stating these restrictions,
the assumptions 
that $b = 1$ and that $1 \in A$ have been taken into account.
The second subgroup is a copy $\Aut_o(A)^\wedge$ of the group $\Aut_o(A)$, 
made up of all PL-homeomorphisms of the form \eqref{eq:Definition-copy-Auto(A)}.

The third subgroup $H$ will turn out to be the kernel of $\bar \eta$.
Let $H$ be the set that consists of those elements $h \in G_\infty(\R;A,P)$
for which there exists an element $t_h \in A$ with the following property:
\begin{enumerate}[a)]
\item $h(t+(p-1)) = h(t) + (p-1)$ for each $t \leq t_h$,
\item $h_{\restriction{[t_h, \infty[}}$ is a finitary PL-function.
\end{enumerate}
\begin{proposition}
\label{prp:LemmaE18New}
\index{Group Ginfty(J;A,P)@Group $G_\infty(J;A,P)$!significance}%
\index{Homomorphism!07-eta-bar@$\bar{\eta}$}%
\index{Automorphism group of G(R;A,P)@Automorphism group of $G(\R;A,P)$!description}%
The set $H$ is a subgroup of $G_\infty(\R;A,P)$ 
and it is the kernel of the epimorphism $\bar \eta \colon \Autfr \bar{G}  \epi \Aut_o(A)/P$.
\end{proposition}

\begin{proof}
We first verify that $H$ is a subgroup.
Suppose $h \in H$.
Properties a) and b) entail
that $h$ is a finitary PL-function on every half line of the form $[t_2, \infty[$
and periodic with period $p-1$ 
on every sufficiently small half line of the form $]-\infty, t_1]$.
These features are preserved by composition and passage to the inverse,
and so $H$ is a subgroup.

We next show 
that $H$ is a subgroup of $\Autfr \bar{G}$.
Fix $\bar{g} \in \bar G$ and $h \in H$.
Then $\bar{g}$ has the properties stated in equation \eqref{eq:19.4};
in particular, $\tau_-(\bar{g})= k(p-1)$ for some $k \in \Z$. 
Then $\act{h} \bar g$ is a PL-homeomorphism in $G_\infty(\R;A,P)$
whose singularities are finite in number 
on every half line of the form $[t_2, \infty[$\,.
So $\act{h} \bar g$ belongs to $\bar G$ if, 
and only if,
conditions \eqref{eq:19.4} are satisfied.
Let $t_1$ be a sufficiently small number.
Then
\begin{align*}
\left(\act{h} \bar g\right) (t_1) 
= 
h \left(\bar g (h^{-1}(t_1)) \right)  
&= 
h\left( h^{-1}(t_1) + k(p-1)\right)\\
&=
h\left( h^{-1}(t_1 + k(p-1))\right)
= t_1 + k(p-1).
\end{align*}
This calculation shows 
that $\bar g_1 = \act{h} \bar g$ is an element of $G(\R;A,P)$
with $\sigma_-(\bar{g}_1) = 1$ and $\tau_-(\bar{g}_1) \in \Z(p-1)$.
The third property of $\bar g$ listed in equation \eqref{eq:19.4} 
and the fact 
that $(IP \cdot A) \rtimes P$ is a normal subgroup of $A \rtimes P$
then imply that $\rho(\bar{g}_1) \in \Aff(IP \cdot A, P)$. 
All taken together, we have shown that  $\bar{g}_1 \in \bar G$.

The preceding reasoning establishes 
that $H$ is a subgroup of $\Autfr \bar{G}$;
its definition then shows 
that it is contained in the kernel of $\bar \eta$.
Conversely, 
assume that $\bar \varphi \in \ker \bar \eta$.
Then $\bar \varphi$  lies in $G_\infty(\R;A,P)$  
(by Lemma \ref{lem:Properties-Autfr-bar-G}).
Conditions \eqref{eq:19.4} show next
that $\bar{G}$ contains the translation 
$\bar{h}_*$ with amplitude $p-1$.
As $\tau_-(\bar h_*)$ generates the image of $\tau_-$ 
and $\bar \varphi$ is increasing
the conjugated translation $\act{\bar{\varphi}}\bar{h}_*$ 
is a translation  near $-\infty$, again with amplitude $p-1$.
There exists therefore a number $t_1 \in A$ so that the equation
\begin{equation}
\label{eq:19.5}
  \left(\varphi \circ \bar{h}_* \circ \varphi^{-1}\right) (t) = t  + (p-1)
\end{equation}
holds for every  $t \leq t_1$. 
Set $u_1 = \bar \varphi^{-1} (t_1)$ and let $u$ be an element of $]-\infty, u_1]$.
Equation \eqref{eq:19.5} gives then rise to the chain of equations
\[
\bar \varphi (u + (p-1)) = \bar \varphi ( \bar h_*(u)) 
= 
 \left(\act{\bar{\varphi}} \bar{h_*}\right) (\bar{\varphi}(u)) 
 = \bar \varphi (u) + (p-1);
\]
it is valid for each $u \leq u_1$ and proves that $\bar{\varphi}$ 
satisfies property a) characterizing the elements of $H$.
Property b) holds since the singularities of $\bar \varphi$ can only accumulate in $-\infty$
(by Lemma \ref{lem:Properties-Autfr-bar-G})
and thus $\bar{\varphi} \in H$.
 \end{proof}

\begin{corollary}
\label{crl:PropositionE19}
\index{Group G([0,infty[;A,P)@Group $G([0, \infty[\;;A,P)$!outer automorphism group}%
\index{Outer automorphism group of!G([0,infty[;A,P)@$G([0, \infty[\;;A,P)$}%
\index{Group P cyclic@Group $P$ cyclic!consequences}%
\index{Group T(R/Zpfr;A,P)@Group $T(\R/\Z \pfr;A,P)$!applications}%
The outer automorphism group of $\bar G$ is an extension of the form
\begin{equation}
\label{eq:19.6}
\xymatrix{1 \ar@{->}[r] &T(\R/\Z (p-1);A,P) \ar@{->}[r]^-{\bar{\iota}_*} 
&\Outfr \bar{G} \ar@{->}[r]^-{\bar{\eta}_*} &\Aut_o(A)/P \ar@{->}[r] &1.}
\end{equation}
In the above,
$\bar{G}$ denotes the subgroup of $G(\R;A,P)$ 
satisfying the restrictions \eqref{eq:19.4}; 
it is isomorphic to $G([0, \infty[\;;A,P)$.
\end{corollary}

\begin{proof} 
The conclusion of Proposition \ref{prp:LemmaE18New} can be restated by saying 
that $\Autfr \bar G$ is the extension of the form
\begin{equation*}
\xymatrix{1 \ar@{->}[r] &H \ar@{->}[r]^-{\bar{\iota}} 
&\Autfr \bar{G} \ar@{->}[r]^-{\bar{\eta}} &\Aut_o(A)/P \ar@{->}[r] &1.}
\end{equation*}
Since $\bar{G}$ is a normal subgroup of $H$ it suffices therefore to establish 
that the quotient group $H/\bar G$ is isomorphic to $T(\R/\Z(p-1);A,P)$.

To reach this goal, 
we construct an epimorphism $\zeta_\ell \colon H \epi G_\infty(\R;p-1;A,P)$.
\label{notation_zeta-ell}%
Consider $h \in H$. 
According to property b) there exists a number $t_h \in A$
so that the equation $h(t+(p-1)) = h(t) + (p-1)$ holds for every $t \leq t_h$;
intuitively speaking, the PL-homeomorphism  $h$ is thus
\emph{periodic with period $p-1$ to the left of $t_h$}. 
It follows that there exists a function $\tilde{h} \in G_\infty(\R;p-1);A,P)$ 
which agrees with $h$ on the half line $]-\infty,t_h]$ 
(and is periodic with period $p-1$).
The PL-homeomorphism $\tilde{h}$ does not depend on the choice of $t_h$;
it is thus uniquely determined by $h$.
The assignment  $h \mapsto \tilde{h}$ is then a homomorphism,
say  $\zeta_\ell$,
of  $H$ into $G_\infty(\R;p-1;A,P)$.
\label{notation:zeta-ell}%
\index{Homomorphism!06-zeta-ell@$\zeta_\ell$}%
It is surjective;
indeed,
given $\tilde{h} \in G_\infty(\R;p-1;A,P)$ 
define $h \colon \R \iso \R$ by
\[
h (t) = \begin{cases} \tilde{h}(t)& \text{ if } t \leq 0,\\ t + \tilde{h}(0) &\text{ if } t \geq 0. \end{cases}
\]
Then $h$ is a PL-homeomorphism with slopes in $P$, singularities in $A$
which maps $A$ onto $A$.
In addition, it satisfies properties a) and b) characterizing the elements of $H$ 
and so it belongs to $H$.
 
The homomorphism  $\zeta_\ell \colon H \to G_\infty(\R;p-1;A,P)$ 
is therefore surjective.
Its kernel consists of all elements in $h \in H$ 
which are the identity to the left of a suitably small number $t_h$;
in view of property b) of $h$ and Lemma \ref{LemmaE16},
it coincides thus with the kernel of $\tau_- \colon \bar {G} \epi \Z(p-1)$.
Moreover,
the image of $\bar{G}$ under $\zeta_\ell$ consists of all the translations with amplitudes in $\Z(p-1)$; 
this image equals the kernel of canonical epimorphism
\[
\kappa_* \colon G_\infty(\R;p-1;A,P) \epi T(\R/\Z(p-1); A, P).
\]
The epimorphism $\zeta_\ell$ induces therefore an isomorphism 
\[
H/\bar{G} \iso T(\R/\Z(p-1);A,P).
\]
The proof of Corollary \ref{crl:PropositionE19} is now complete.
\index{Group G([0,infty[;A,P)@Group $G([0, \infty[\;;A,P)$!automorphism group|)}
\index{Automorphism group of G([0,infty[;A,P)@Automorphism group of $G([0, \infty[\;;A,P)$!description|)}%
\end{proof}
%
\subsubsection{Automorphism group of 
$G_1 = \ker(\sigma_- \colon G([0, \infty[ ;A,P) \epi P)$}
\label{sssec:19.2c}
\index{Theorem \ref{TheoremE04}!consequences}%
%
The investigation of the automorphism group of $G_1$ will parallel 
that of the automorphism group $\Autfr G([0,\infty[\,;A,P)$.
By Theorem \ref{TheoremE04} and part (ii) of Theorem \ref{TheoremE14}  
every element $\varphi_1 \in \Autfr G_1 $ is an increasing  PL-homeomorphism of $]0,\infty[$
that maps $A\; \cap \; ]0, \infty[$ onto itself
and has slopes in a coset $s_{\varphi_1} \cdot P$ and singularities in $A$;
the singularities may be infinite in number, 
but if so they accumulate only in 0.
It follows that there exists a homomorphism
\[
\label{eq:Definition-eta-1}
\index{Homomorphism!07-eta1@$\eta_1$}%
\eta_1 \colon \Autfr G_1 \to \Aut_o(A)/P
\]
which takes $\varphi_1$ to $s_{\varphi_1} \cdot P$.
It is surjective,
since $\Autfr G_1$ contains all the homeomorphisms of the form
\begin{equation}
\label{eq:Definition-copy-Auto(A)-1}
t \mapsto \begin{cases} 
s\cdot t &\text{ if } t \geq 0,\\ 1 \cdot t&\text{ if } t \leq 0, \end{cases}
\end{equation}
with $s \in \Autfr_o(A)$.

For a more detailed analysis we replace $G_1$ by its image $\bar{G}_1$ 
under the embedding $\mu_{2,1} \colon G([0, \infty[\,;A,P) \mono G(\R;A,P)$ 
(defined in \ref{sssec:Embedding-mu2}). 
\label{notation:bar(G)1}
The epimorphism $\eta_1 \colon \Autfr G_1 \to \Aut_o(A)/P$ gives then rise 
to the epimorphism
\[
\label{notation:bar(eta)1}
\index{Homomorphism!07-eta1-bar@$\bar{\eta}_1$}%
\bar{\eta}_1  \colon \Autfr \bar{G}_1 \epi \Aut_o(A)/P.
\]
One sees next,
just as in the paragraph preceding the statement of Lemma  
\ref{lem:Properties-Autfr-bar-G},
that the conclusion of this lemma holds also 
for the elements of the kernel of $\bar{\eta}_1$.
But more is true, 
this conclusion actually describes the elements of $\Autfr \bar{G}_1$;
indeed, 
the elements $\bar{g}_1 \in \Autfr \bar{G}_1$ 
are characterized  by the requirements
\begin{equation}
\label{eq:Characterization-bar(g)1}
\sigma_-(\bar{g}_1) = 1, \quad \tau_-(\bar{g}) = 0 
\quad \text{and}\quad  
\rho(\bar{g}) \in \Aff(IP \cdot A, P)
\end{equation}

We have thus established
\begin{proposition}
\label{prp:Characterisation-Aut-barG1}
\index{Homomorphism!07-eta1-bar@$\bar{\eta}_1$}%
The kernel of $\bar{\eta}_1 \colon \bar{G}_1 \epi \Aut_o(A)/P$ 
consists of all PL-homeomorphism in $G_\infty(\R;A,P)$ 
whose set of singularities, if infinite in number,  accumulate only in $-\infty$.
\end{proposition}

\begin{remark}
\label{remark:Characterisation-Aut-barG1}
The group described in Proposition \ref{prp:Characterisation-Aut-barG1}
does not seem to have an alternative description that is better known.
The group contains the cartesian product $\prod_n G([-n+1),-n];A,P)$ 
of countably many copies of the group $G([0,1];A,P)$.
The outer automorphism group $\Outfr \bar{G}_1$ of $\bar{G}_1$
contains also a cartesian  product $\prod_j L_j$ of countably many copies of $L =G([0,1];A,P)$.
\index{Group P cyclic@Group $P$ cyclic!consequences}%
To see this, 
choose a bijection $\beta \colon \N \times \N \iso \N$
and realize $L_j$ as the diagonal copy in the cartesian product 
$\prod_{i \in \N} G([-(\beta(i,j) - 1, -\beta(i,j)]; A, P)$.
Then no nontrivial element of $\bar{G}_1$ lies in $\prod_j L_j$. 
\end{remark}
%
\subsection[Groups with $I$ a compact interval]%
{Automorphism groups of groups with $I$ a compact interval}
\label{ssec:19.3New}
%
We come finally to groups of PL-homeomorphisms 
with supports in a compact interval having endpoints in $A$;
as the isomorphism type of $G([0,b]; A, P)$ 
does not depend on $b \in A$ (see Theorem \ref{TheoremE07}),
we may and shall assume that $I = [0,1]$.
\index{Theorem \ref{TheoremE07}!consequences}%

Set $G = G([0,1];A,P)$.
We have not been able to determine
whether every element of $\Autfr G$ is a PL-homeomorphism; 
\index{Group G([a,c];A,P)@Group $G([a,c];A,P)$!exotic automorphisms}%
the situation is even worse for $G_1 = \ker(\sigma_- \colon G\to P)$ or for $B([0,1];A,P)$
and so we shall say nothing about these groups. 
\footnote{See sections \ref{sssec:Notes-ChapterE-Brin96}
and  \ref{sssec:Notes-ChapterE-BrGu98} for updates.}
The subgroup $\Autfr_{\PL}G$ of all PL-homeomorphisms in $\Autfr G$
can, however, be investigated in some detail.

We carry out this study in section \ref{sssec:19.3a}, 
obtaining descriptions of $\Autfr_{\PL} G$ and of $\Outfr_{\PL}  G$,
that are akin to those of 
$\Autfr G([0,\infty[\,;A,P)$ and of $\Outfr G([0,\infty[\,;A,P)$ 
found in section \ref{sssec:19.2b}.
In section \ref{sssec:19.3b}
 we address then the question whether $\Autfr_{\PL} G$ 
 coincides with the full group $\Autfr G$. 
 As a first step towards an answer we show 
 that $\Autfr G$ is generated by $\Autfr_{\PL} G$ 
 and a second subgroup $\Autfr_{\per}G$ with special properties.
%
\subsubsection{The subgroup of PL-automorphisms}
\label{sssec:19.3a}
%
Let $\Autfr_{\PL}G$ denote the intersection $\Autfr G\, \cap\, G_\infty(]0,1[\,;A,P)$
\footnote{The definition of $G_\infty(J;A,P)$ is given in section \ref{ssec:19.1}.}
and consider a PL-homeomorphism in $\varphi \in \Autfr_{\PL} G$,
By definition,  
$\varphi$ has then only finitely many singularities in every compact subinterval of 
$\Int(I) =\; ]0,1[$;
so it is differentiable at a point $a \in A \cap \Int(I)$
whence Proposition \ref{PropositionE9} allows one to deduce
that $\varphi$ has slopes in a single coset $s_\varphi \cdot P$ of $\Aut(A)/P$.
\index{Proposition \ref{PropositionE9}!consequences}%
This fact implies 
that there exists a homomorphism
\begin{equation}
\label{eq:Definition-eta-PL}
\index{Homomorphism!07-etaPL@$\eta_{\PL}$}%
\index{Group AutfrPLG@Group $\Autfr_{\PL}G([0,1];A,P)$!properties}%
\index{Group Aut(A)@Group $\Aut(A)$!significance}%
\eta_{\PL} \colon \Autfr_{\PL}G \to \Aut(A)/P
\end{equation}
which sends $\varphi$ to the coset $s_\varphi \cdot P$.
This homomorphism is \emph{surjective}.
Indeed, 
the reflection $\beta$ at the midpoint of $I$ is in $\Autfr_{\PL}G$ 
and $\eta(\beta) =(-1) \cdot P$. 
On the other hand,
if $s$ is a positive real with $s\cdot A = A$ 
then $s = s\cdot 1$ lies in $A$. 
By Theorem \ref{TheoremE07}
\index{Theorem \ref{TheoremE07}!consequences}%
there exists therefore an infinitary PL-homeomorphism 
$\psi_s \colon ]0,1[\, \iso\, ]0,s[$ 
with slopes in $P$ 
and vertices in $A\times (s\cdot A) = A^2$ 
which, by conjugation,  induces  an isomorphism
\begin{equation*}
(\psi_s)_* \colon G([0,1];A,P) \iso G([0,s];A,P).
\end{equation*}
The composition $\varphi_s = (\psi_s)_*^{-1} \circ (t\mapsto s\cdot t)$ 
belongs then to $\Autfr_{\PL} G $ and $\eta(\varphi_s) = s\cdot P$.
(Recall that $\eta$ may not be surjective 
if the slope group $P$ of $G = G([0,b];A,P)$ is not cyclic,
as is shown by part (v) of Corollary \ref{CorollaryE13}).

Our next aim is to describe the kernel of $\eta$.
For this task it is convenient to replace the group $G$ 
by its image $\bar G$ of $G$  under the embedding $\mu_{2,1} \circ \mu_1$. 
In view of Lemma \ref{LemmaE17}, 
this image consists of all PL-homeomorphism in $G(\R;A,P)$
that satisfy the restrictions
\begin{equation}
\label{eq:19.7}
\sigma_-(\bar{g}) = 1, \quad  \tau_-(\bar{g}) \in \Z(p-1)
\quad \text{and}\quad
 \sigma_+(\bar{g}) = 1,\quad \tau_+(\bar{g}) \in\Z(p-1).
\end{equation}

We are now ready to define a subgroup $H \subset G_\infty(\R;A,P)$
that will turn out to be the kernel of 
$\bar \eta \colon \Autfr_{\PL} \bar G  \epi \Aut(A)/P$.
Let $H$ be the set consisting of all PL-homeomorphisms $h \in G_\infty(\R;A,P)$ 
for which there exists a positive number $t_h \in A$ with the following properties:
\begin{enumerate}
\item[a1)] $h(t - (p-1)) = h(t) - (p-1)$ for each $t \leq -t_h$;
\item[a2)] $h(t+(p-1)) = h(t) + (p-1)$ for each $t \geq t_h$.
\end{enumerate}
\begin{proposition}
\label{PropositionE20New-part-I}
\index{Homomorphism!07-eta-bar@$\bar{\eta}$}%
\index{Group AutfrPLG@Group $\Autfr_{\PL}G([0,1];A,P)$!properties}%
The set $H$ is a subgroup of $G_\infty(\R;A,P)$ 
and it is the kernel of the epimorphism 
$\bar \eta \colon \Autfr_{\PL} \bar{G}  \epi \Aut(A)/P$.
\end{proposition}

\begin{proof}
We first verify that $H$ is a subgroup.
Every element $h \in H$ maps $A$ onto itself,
has slopes in $P$,
singularities in $A$ and it lies in $G_\infty$ by definition.
Properties a1),  a2)  can thus be restated by saying
that $h$ is a finitary PL-function on every symmetric interval $[-b, b]$,
and periodic with period $p-1$ 
outside of every sufficiently large  interval of this kind.
It follows
that the defining properties of the elements of $H$
are preserved by composition and passage to the inverse.

We show next
that $H$ is a subgroup of $\Autfr_{\PL} \bar G$.
Let $h$ be an element in $H$.
Suppose that $\bar{g} \in \bar G$ 
and that $\tau_-(\bar{g})= k_\ell(p-1)$ and  $\tau_+(\bar{g})= k_r(p-1)$.
Then $\bar{g}_1 = \act{h} \bar g$ is a PL-homeomorphism in $G_\infty(\R;A,P)$.
The element $\bar{g}_1$ lies therefore in $\bar G$ if, 
and only if,
conditions \eqref{eq:19.7} are satisfied.
Let $t$ be a sufficiently small, negative number.
Then
\begin{align*}
\bar{g}_1 (t) 
= 
h \left(\bar g \left(h^{-1}(t) \right) \right)  
&= 
h\left( h^{-1}(t) + k_\ell(p-1)\right)\\
&= 
h\left( h^{-1}(t + k_\ell(p-1))\right)
= 
t + k_\ell(p-1).
\end{align*}
This calculation shows 
that $\bar g_1 = \act{h} \bar g$ is an element of $G(\R;A,P)$
with $\sigma_-(\bar{g}_1) = 1$ and $\tau_-(\bar{g}_1) \in \Z(p-1)$.
One sees similarly 
that $\sigma_+(\bar{g}_1) = 1$ and $\tau_+(\bar{g}_1) \in \Z(p-1)$.

The preceding reasoning shows
that the group $H$ is a subgroup of $\Autfr_{\PL} \bar{G}$;
the definition of $H$ then implies 
that $H$ is contained in the kernel of $\bar \eta$.
Conversely, assume that $\bar \varphi \in \ker \bar \eta$.
Then $\bar \varphi$ is a PL-homeomorphism 
which maps $A$ onto itself and has singularities in $A$ 
(by the definition of $\Autfr_{\PL} \bar{G}$)
and which has slopes in $P$ (since $\bar \varphi \in \ker \bar \eta$).
It follows, in particular, that $\bar{\varphi}$ is increasing.
The group $\bar{G}$ contains the translation  $\bar g_*$ with amplitude $-(p-1)$
(see conditions \eqref{eq:19.7}). 
Its image $\tau_-(\bar g_*)$ generates the image of $\tau_-$; 
as $\bar \varphi$ is increasing the conjugated translation 
$ \act{\bar{\varphi}}\bar{g}_*$ 
is therefore a translation  near $-\infty$ with amplitude $-(p-1)$.
There exists thus a number $t_\ell \in A$ so that the equation
\begin{equation}
\label{eq:19.5bis}
  \left(\act{\bar{\varphi}} \bar g_*\right) (t) = t - (p-1)
\end{equation}
holds for every  $t \leq t_\ell$. 
Set $u_\ell =\bar  \varphi^{-1} (t_\ell)$ and fix $u \in \;]-\infty, u_\ell]$.
Equation \eqref{eq:19.5bis} leads then to the chain of equations
\[
\bar \varphi (u - (p-1)) = \bar \varphi ( \bar g_*(u)) 
= 
 \left(\act{\bar{\varphi}} \bar g_*\right) (\bar{\varphi}(u) 
 = \bar \varphi (u) - (p-1).
 \]
They prove that $\bar{\varphi}$ 
satisfies property a1),
the first of the properties characterizing the elements of $H$;
property  a2) can be established similarly.
Thus   $\ker \bar{\eta} \subseteq H$. 
 \end{proof}
 
The previous proposition is the analogue of Proposition \ref{prp:LemmaE18New}.
That proposition has a corollary 
which describes the group $\Outfr_{\PL} \bar G$ of $\bar{G}$ 
as an extension of groups that are better known than the subgroup $H$ itself.
Our next goal is to obtain an analogous corollary 
for the subgroup $\bar{G}$ of $G(\R;A,P)$ 
made up of all elements  that satisfy condition \eqref{eq:19.7};
this subgroup is isomorphic to $G([0,1];A,P)$.
The corollary will describe $\Outfr_{\PL} \bar G$ as an extension of groups.
The kernel of the extension is a subgroup, 
called $\bar{L}$, 
of the square of $T(\R/\Z(p-1);A,P)$.
To define $\bar{L}$, 
we need the homomorphism
\begin{equation}
\label{eq:definition-gamma-bar-1}
\index{Homomorphism!03-gamma-bar-p@$\bar{\gamma}_\pfr$}%
\bar{\gamma}_{p-1} \colon T(\R/ \Z(p-1); A,P) \epi A/(IP \cdot A), 
\quad \bar{f} \mapsto \bar{f}(a_0 + \Z \pfr) - (a_0 + \Z \pfr)+ IP \cdot A,
\end{equation}
introduced in Remark \ref{remark:Homomorphism-gamma}.
Here $\R/\Z(p-1)$ denotes the circle of circumference $p-1$
and $a_0$ is an element of $A$.
A crucial property of the function $\bar{\gamma}_{p-1}$ is the fact
that it does not depend on the choice of  $a_0$.

The group $\bar{L}$ can now defined like this:
\begin{equation}
\label{eq:Definition-subgroup-L}
\index{Group T(R/Zpfr;A,P)@Group $T(\R/\Z \pfr;A,P)$!applications}%
\bar{L} = \left\{
(\bar{f}_1, \bar{f}_2) \in T(\R/ \Z(p-1); A,P) ^2  
\mid 
\bar{\gamma}_1(\bar{f}_1) =  \bar{\gamma}_1(\bar{f}_2)
\right\}.
\end{equation}
Since $\bar{\gamma}_1$ is an epimorphism 
onto the quotient group $A/(IP \cdot A)$,
the group $\bar{L}$ is a subgroup of index $|A : IP \cdot A|$ in the square of the group of finitary PL-homeomorphisms of the circle $\R/\Z(p-1)$ with slopes in $P$ and singularities in $A/\Z(p-1)$.

\begin{corollary}
\label{crl:PropositionE20New-part-II}
\index{Group AutfrPLG@Group $\Autfr_{\PL}G([0,1];A,P)$!description}%
\index{Outer automorphism group of!G([a,c];A,P)@$G([a,c];A,P)$}%
\index{Outer automorphism group of!G([a,c];A,P)@$G([a,c];A,P)$}%
\index{Group Out-PLG@Group $\Out_{\PL}G([0,1];A,P)$!description}%
\index{Group P cyclic@Group $P$ cyclic!consequences}%
Let $\bar{G}$ denote the subgroup of $G(\R; A, P)$ 
consisting of the PL-homeomorphisms satisfying the requirements \eqref{eq:19.7}.
Then $\Outfr_{\PL} \bar{G}$,
\ie{} the image of $\Autfr_{\PL} \bar{G}$ 
in the outer automorphism group of $\bar{G}$, 
 is an extension of the form
\begin{equation}
\label{eq:PL-outer-automorphism-group-for-I-compact}
\xymatrix{1 \ar@{->}[r] &\bar{L} \ar@{->}[r]^-{} 
&\Outfr_{\PL} \bar{G} \ar@{->}[r]^-{\bar{\eta}_*} &\Aut(A)/P \ar@{->}[r] &1}.
\end{equation}
\end{corollary}
\index{Group G([a,c];A,P)@Group $G([a,c];A,P)$!outer automorphisms}%
\index{Group P cyclic@Group $P$ cyclic!consequences}%

\begin{proof} 
The conclusion of Proposition \ref{PropositionE20New-part-I}
can be restated by saying 
that $\Autfr_{\PL} \bar G$ is an extension having the form
\begin{equation*}
\xymatrix{1 \ar@{->}[r] &H \ar@{->}[r]^-{\bar{\iota}} 
&\Autfr_{\PL} \bar{G} \ar@{->}[r]^-{\bar{\eta}} &\Aut(A)/P \ar@{->}[r] &1.}
\end{equation*}
Now, $\bar{G}$ is a normal subgroup of $\Autfr \bar{G}$
and every element of $\bar G$ has slopes in $P$;
so the epimorphism $\bar{\eta}$
gives rise to an epimorphism
\[
\bar{\eta}_* \colon \Outfr_{\PL} \bar{G} \epi \Aut(A)/P.
\]
It suffices therefore to establish 
that $\ker \bar{\eta}_*$ is isomorphic to $\bar{L}$.

As a first step on the way towards this goal,
we construct two epimorphisms $\zeta_\ell$ and $\zeta_r$ 
\label{notation:zeta-r}%
from  $H$ onto  $G_\infty(\R;p-1;A,P)$.
Let $h$ be  an element in $H$. 
According to property a1) of the elements of $H$
there exists then a positive number $t_h \in A$
such that the equation $h(t-(p-1)) = h(t) - (p-1)$ holds for every $t \leq -t_h$.
There exists thus a PL-homeomorphism $\tilde{h}_\ell \in G_\infty(\R;p-1);A,P)$ 
which agrees with $h$ on the half line $]-\infty,-t_h]$. 
As this PL-homeomorphism is not changed if $t_h$ is replaced  
by a larger number in $A$,
it follows,
as in the proof of Corollary \ref{crl:PropositionE19},
that the assignment  $h \mapsto \tilde{h}_\ell$ is a homomorphism,
\index{Homomorphism!06-zeta-ell@$\zeta_\ell$}%
say $\zeta_\ell \colon H \to G_\infty(\R;p-1;A,P)$,
and that this homomorphism is surjective.
Property a2) allows one to construct a similar epimorphism
$\zeta_r$ of $H$ onto $G_\infty(\R;p-1;A,P)$.
\index{Homomorphism!06-zeta-r@$\zeta_r$}%
Taken together,
these homomorphisms lead to the homomorphism
\begin{equation}
 \index{Homomorphism!06-zeta@$\zeta$}%
 \label{eq:Homomorphism-zeta}
\zeta = (\zeta_\ell, \zeta_r) \colon 
H \longrightarrow G_\infty(\R; p-1; A, P)^2,
\quad
h \longmapsto (\tilde{h}_\ell, \tilde{h}_r).
\end{equation}
By composing $\zeta$ with the square of  the projection
\[
\kappa_*\colon G_\infty(\R;p-1;A,P) \epi T(\R/\Z(p-1);A,P)
\]
one obtains, finally, the homomorphism
\begin{equation}
\label{eq:Homomorphism-zeta-bar}
\bar{\zeta} = (\bar{\zeta}_\ell, \bar{\zeta}_r) \colon 
H \longrightarrow T(\R/\Z(p-1); A, P)^2,
\quad
h \longmapsto (\bar{h}_\ell, \bar{h}_r).
\end{equation}

Its kernel consists of all elements of $H$ 
that are translations with periods in $\Z(p-1)$, 
both near $-\infty$ and near $+\infty$;
a glance at equation \eqref{eq:19.7} reveals then
that this kernel is nothing but the group $\bar{G}$.

Now to the image of $\bar{\zeta}$.
Let $h$ be an element of $H$. 
Since $H$ is a subgroup of $G_\infty(\R;A,P)$ 
the restriction of $h$ to every compact interval $[a_1,a_2]$ is a finitary PL-homeomorphism with singularities in $A$ and  slopes in $P$.
Suppose now that $a_1 \in A$ is so small 
and $a_2 \in A$ is so large 
that $h$ is periodic with period $p-1$ 
to the left of $a_1$ and also to the right of $a_2$.
Then $h(a_1) - a_1$  is congruent to $h(a_2) - a_2$ \emph{modulo}  $IP \cdot A$
(by Theorem \ref{TheoremA}).
In view of the definitions of 
$\tilde{h}_\ell = \zeta_\ell(h)$ and of $\tilde{h}_r = \zeta_r(h)$
this finding can be restated by saying 
that 
\begin{align*}
\gamma_{p-1} (\zeta_\ell(h)) 
&=
\tilde{h}_\ell(a_1) - a_1 + (IP \cdot A)\\
&=  
\tilde{h}_r(a_2) - a_2 + (IP \cdot A)
=
\gamma_{p-1} (\zeta_r(h)).
\end{align*}
So the image of $\zeta$ is contained in the group
\begin{equation*}
\tilde{L} 
= \left\{
(\tilde{f_1}, \tilde {f_2}) \in G_\infty(\R;p-1;A,P)^2 
\mid 
\bar{\gamma}_{p-1}(\bar{f}_1) = \bar{\gamma}_{p-1}(\bar{f}_2)
\right\}.
\end{equation*}

Conversely, 
let  $\bar{h}_\ell$  and $\bar{h}_r$  be elements of  $T(\R/\Z(p-1);A,P)$
with $\bar{\gamma}_{p-1} (\bar{h}_\ell) = \bar{\gamma}_{p-1} (\bar{h}_r)$.
Let $\tilde{h}_\ell$  and $\tilde{h}_r$ 
be lifts of these elements in $G_\infty(\R; p-1;A,P)$.
Then $\gamma_{p-1} (\tilde{h}_\ell) = \gamma_{p-1} (\bar{h}_r)$.
Pick elements $a_1 < a_2$ in $A$.
Then $\tilde{h}_\ell (a_1) - a_1$ is congruent to $\tilde{h}_r (a_2) - a_2$,
and so $a_2 - a_1$ is congruent to $\tilde{h}_r (a_2)- \tilde{h}_\ell (a_1)$.
If the latter difference is positive,
Theorem \ref{TheoremA} allows one to find a PL-homeomorphism 
$f \in G(\R;A,P)$ with $f(a_1) = \tilde{h}_\ell(a_1)$ 
and $f(a_2) = \tilde{h}_\ell(a_2)$.
Let $h \colon \R \iso \R$ be the PL-homeomorphism 
that agrees with $h_\ell$ to the left of $a_1$,
with $f$ on the interval $[a_1, a_2]$ 
and with $\tilde{h}$ to the right of $a_2$.
Then $h$ is a PL-homeomorphism lying in $H$ 
that maps onto the given ordered pair $(\bar{h}_\ell, \bar{h}_r)$ in $\bar{L}$.
If, however, $\tilde{h}_r (a_2) \leq \tilde{h}_\ell (a_1)$
replace $\tilde{h}_r$ by a lift $\widetilde{h_1}_r$ of $\bar{h}_r$ 
with $\widetilde{h_1}_r (a_2) >\tilde{h}_\ell(a_1)$ and proceed as before.
We conclude, first, that $\bar{\zeta}$ maps $H$ onto $\bar{L}$
and  then that $\bar{\zeta}$ induces an isomorphism $H/\bar{G} \iso \bar{L}$.
The proof of Corollary \ref{crl:PropositionE20New-part-II}
is now complete.
\end{proof}
%
\subsubsection{Exotic automorphisms}
\label{sssec:19.3b}
%
We have not succeeded in determining 
whether the automorphism group $\Autfr \bar{G}$ is larger 
than its subgroup $\Autfr_{PL} \bar{G}$.
\footnote{See sections \ref{sssec:Notes-ChapterE-Brin96}
and  \ref{sssec:Notes-ChapterE-BrGu98} for updates.}
In this section, we prove, as a meagre substitute, Proposition \ref{prp:TheoremE21}
which brings to light 
that $\Autfr \bar{G}$ is generated by $\Autfr_{\PL} \bar{G}$ 
and a subgroup $\Autfr_{\per} \bar{G}$ consisting of periodic homeomorphisms.
This result implies that if $\Autfr \bar{G}$ contains an element outside of $\Autfr_{\PL} \bar{G}$,
dubbed \emph{exotic},
then the subgroup $\Autfr_{\per} \bar{G}$ will also contain such an element.
\index{Exotic automorphisms!definition|textbf}%

\begin{definition}
\label{definition:ssec-19.6}
Let $\Autfr_{\per}\bar G$ denote the subgroup of $\Autfr \bar{G}$ 
which consists of all homeomorphisms $\varphi \in \Autfr \bar G$
that satisfy the restrictions
\begin{equation}
\label{eq:19.8}
\varphi(0) = 0,\text{ and } \varphi(t+(p-1)) = \varphi(t) + (p-1)\text{ for all }t\in \R.
\end{equation}
\end{definition}

\begin{proposition}
\label{prp:TheoremE21}
As before,
let  $\bar G$ denote the subgroup of $G(\R;A,P)$ 
defined by the conditions \eqref{eq:19.7}.
The automorphism group $\Autfr \bar{G}$ is generated by its subgroups 
$\Autfr_{\PL}\bar G$ and  $\Autfr_{\per} \bar{G} $.
\end{proposition}

\begin{proof}
Let $\varphi$ be an element of $\Autfr \bar G$.  
Since $-\id_{\R}$ is in $\Autfr_{\PL} \bar G $ 
we may assume that $\varphi$ is increasing. 
It then follows,
as in the proof of Proposition \ref{PropositionE20New-part-I},  
that there is a real $t_0 \in A$ so that 
\begin{equation}
\label{eq:19.9}
\varphi(t-(p-1)) = \varphi (t) - (p-1) \text{ for all } t\leq t_0.
\end{equation}
Let $\tilde{\varphi}_\ell$ denote the unique homeomorphism of $\R$ 
which agrees with $\varphi$ on the half line $]-\infty, t_0]$ 
and is periodic with period $p-1$.

We claim that $\tilde\varphi$ is in $\Autfr \bar G$.  
To prove this, 
we bring into play the translation $\bar{h}= (t\mapsto t + (p-1) )$ with amplitude $p-1$.
This translation lies in $\bar{G}$
and it commutes with $\tilde{\varphi}$.
Consider now an element $\bar{g} \in \bar{G}$.
It satisfies the restrictions \eqref{eq:19.7}
and so $\tau_+(\bar{g}) = k \cdot (p-1)$ for some $k \in \Z$.
The translation $\bar{h}^k$ coincides therefore with $\bar{g}$ 
on a half line of the form $[t_1, \infty[$\,;
let $m$ be an integer with  $t_1- m(p-1) \leq t_0$.
The composition 
\[
\bar{g}_1 
= 
\bar{h}^{-m} \circ \left(\bar{g} \circ (\bar{h}^k)^{-1}\right) \circ  \bar{h}^{m}
\]
is then an element of $\bar{G}$ 
whose support is contained in the half line $]-\infty, t_0]$.
The calculation
\begin{equation*}
\act{\varphi}\bar{g}_1 
= 
\act{\tilde{\varphi}} \bar{g}_1
=
\act{\tilde{\varphi}} \left(\bar{h}^{-m} \circ \bar{g} \circ  \bar{h}^{ m-k}\right)
=
\bar{h}^{-m} \circ \act{\tilde{\varphi}}\bar{g}  \circ \bar{h}^{ m-k}
\end{equation*}
and the facts that $\act{\varphi} g_1$ and  $\bar{h}$ lie both in $\bar{G}$ 
imply then
that  $\act{\tilde{\varphi}}\bar{g} \in \bar{G}$.
It follows that $\act{\tilde{\varphi}} \bar{G} \subseteq \bar{G}$.
One verifies similarly, that $\act{\tilde{\varphi}^{-1}} \bar{G} \subseteq \bar{G}$,
and so $\bar{\varphi} \in \Autfr \bar{G}$, as claimed.

The proof is now quickly completed.
Let $h_0$ be the translation with amplitude $\bar{\varphi}(0)$.
Then $\bar{\varphi}(0) \in A$ by Theorem \ref{TheoremE04}
\index{Theorem \ref{TheoremE04}!consequences}%
and so $h_0 \in \Autfr \bar{G}$.
The composition 
$\tilde{\varphi}_1 = \tilde{\varphi}\circ h_0^{-1}$ fixes 0,
lies in $\Autfr \bar{G}$
and is periodic with period $p-1$,
and so it is an element of $\Autfr_{\per} \bar G $.
The composition 
$ \varphi_2 = \tilde{\varphi}_1^{-1} \circ \varphi$,
on the other hand,  is an element of $\Autfr \bar G$ 
which is a translation on the half line $]-\infty,t_0]$
and hence differentiable in a point of $A$. 
Proposition \ref{PropositionE9} thus allows us to infer
\index{Proposition \ref{PropositionE9}!consequences}%
that $\varphi_2$ lies in $\Autfr_{\PL} \bar{G}$,
and so $\varphi$,
being the product $\tilde{\varphi}_1 \circ \varphi_2$,
lies in  $\Autfr_{\per} \circ \Autfr_{\PL}$. 

All taken together, 
we have shown
that $\Autfr \bar{G} $ is  the complex product 
\[
\{\id_{\R}, -\id_{\R}\} \circ \Autfr_{\per} \circ \Autfr_{\PL}
\]
and so the proof is complete.
\end{proof}

%% file: chaptE.18.illustration-mu1.tex
\begin{minipage}{11.5cm}
\psfrag{1}{\hspace*{-2mm} \small   $0$}
\psfrag{2}{\hspace*{-6.3mm} \small $(p-1)b$}
\psfrag{3}{  \hspace*{-5.0mm}\small  }
\psfrag{4}{  \hspace*{-1.7mm}\small  }
\psfrag{5}{  \hspace*{-1.7mm}\small }
\psfrag{6}{  \hspace*{-1.7mm}\scriptsize}
\psfrag{7}{  \hspace*{-3mm}\scriptsize $\infty$}

\psfrag{11}{\hspace*{-0.7mm} \small   $0$}
\psfrag{12}{\hspace*{-2.9mm} \small $\tfrac{p-1}{p}b$}
\psfrag{13}{  \hspace*{-5.7mm}\scriptsize  }
\psfrag{14}{  \hspace*{-1.7mm}\small  }
\psfrag{15}{  \hspace*{-1.7mm}\small }
\psfrag{16}{  \hspace*{-1.2mm}\scriptsize}
\psfrag{17}{  \hspace*{-0.5mm}\small$b$}

\psfrag{21}{\hspace*{-1.2mm} \scriptsize   $\tfrac{1}{p}$}
\psfrag{22}{\hspace*{-2.2mm} \scriptsize $\tfrac{1}{p^2}$}
\psfrag{23}{  \hspace*{-2.9mm}\scriptsize $\tfrac{1}{p^3}$}
\psfrag{24}{  \hspace*{-1.7mm}\small  }
\psfrag{25}{  \hspace*{-1.7mm}\small }
\psfrag{26}{  \hspace*{-1.7mm}\small }

\psfrag{81}{\hspace*{-9mm} \small  3}
\psfrag{82}{\hspace*{-1.7mm} \small  1}
\psfrag{31}{\hspace*{-1.2mm} \small   $0$}
\psfrag{32}{\hspace*{-1.2mm} \small  1}
\psfrag{33}{  \hspace*{-1.7mm}\small  $\infty$}
\psfrag{41}{\hspace*{-1mm} \small   $0$}
\psfrag{42}{\hspace*{-1.0mm} \small  $1$}
\psfrag{43}{\hspace*{-1.7mm} \small  $\infty$}
\psfrag{51}{\hspace*{-1.7mm} \small  1}
\psfrag{52}{\hspace*{-1.7mm} \small  3}
\psfrag{1}{\hspace*{-1.5mm} \small   $0$}
\psfrag{92}{\hspace*{-1.2mm} \small  2}
\psfrag{93}{  \hspace*{-1.7mm}\small  $\infty$}

\psfrag{101}{\hspace*{-1mm} \small   $0$}
\psfrag{102}{\hspace*{-0.7mm} \small  $2$}
\psfrag{104}{\hspace*{-1.2mm} \small  $\infty$}
\psfrag{111}{\hspace*{-1.7mm} \small  1}
\psfrag{112}{\hspace*{-1.7mm} \small  3}
\psfrag{phi1}{\hspace{1mm}\small $\varphi_b$}
\psfrag{la2}{\hspace{-1mm}$x_1$}
\psfrag{la4}{\hspace{-2mm}$x_2$}
\begin{equation*}
\includegraphics[width= 11cm]{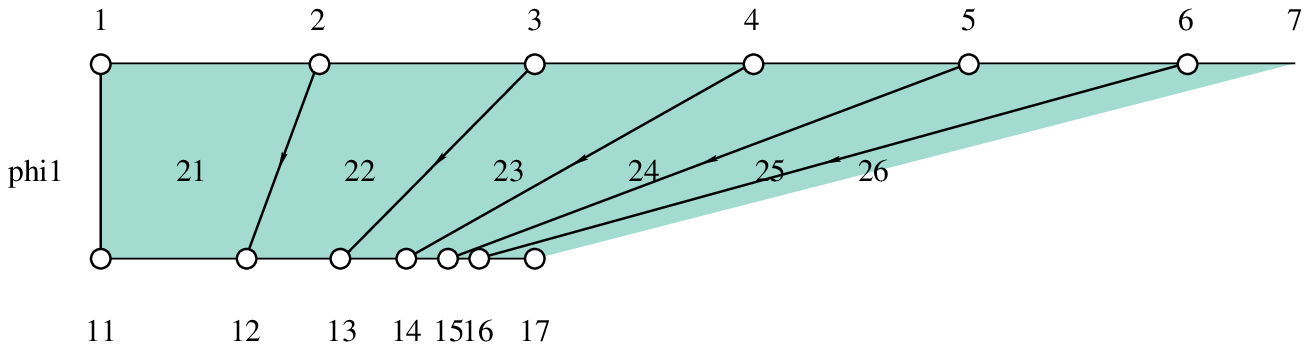}
\end{equation*}
\end{minipage}
\index{Rectangle diagram!examples}

%% file: chaptE.18.illustration-mu2.tex
\begin{minipage}{11.5cm}
\psfrag{1}{\hspace*{-3mm}    \small   $-\infty$}
\psfrag{2}{\hspace*{-6.7mm} \small }
\psfrag{3}{  \hspace*{-5.7mm}\small  }
\psfrag{4}{  \hspace*{-6.7mm}\small  }
\psfrag{5}{  \hspace*{-5.2mm}\small  $(1-p)b$ }
\psfrag{6}{  \hspace*{-1.4mm}\small $0$}
\psfrag{7}{  \hspace*{-4.7mm}\small $\infty$}

\psfrag{11}{\hspace*{-0.6mm} \small  $0$}
\psfrag{12}{\hspace*{-3.7mm} \small }
\psfrag{13}{  \hspace*{-5.7mm}\small  }
\psfrag{14}{  \hspace*{-1.3mm}\small $\tfrac{b}{p^2}$}
\psfrag{15}{  \hspace*{-0.8mm}\small $\tfrac{b}{p}$}
\psfrag{16}{  \hspace*{-0.2mm}\small $b$}
\psfrag{17}{  \hspace*{-3.7mm}\small $\infty$}

\psfrag{21}{\hspace*{-1.2mm} \small   }
\psfrag{22}{\hspace*{-2.7mm} \small }
\psfrag{23}{  \hspace*{-3.9mm}\small }
\psfrag{24}{  \hspace*{-0.0mm}\small $\tfrac{1}{p^2}$}
\psfrag{25}{  \hspace*{-0.6mm}\small $\tfrac{1}{p}$ }
\psfrag{26}{  \hspace*{-0.7mm}\small 1 }
\psfrag{psi}{\hspace{2mm}\small $\psi_{2,b}$}
\begin{equation*}
\includegraphics[width= 11cm]{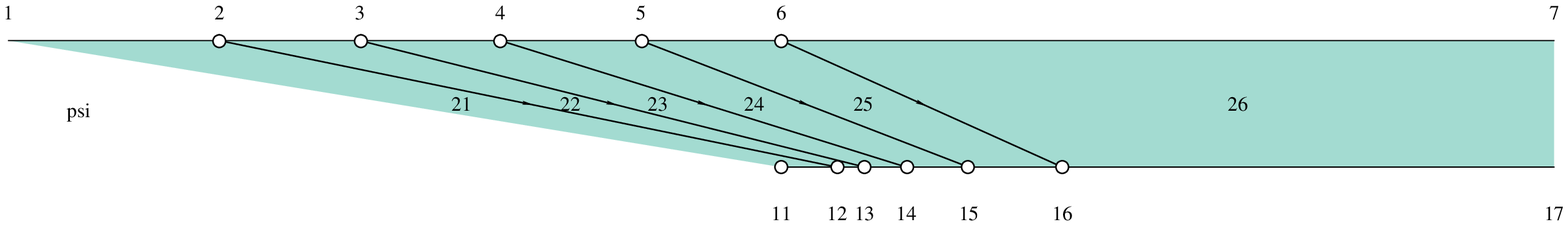}
\end{equation*}
\end{minipage}
\index{Rectangle diagram!examples}

%% file: chaptE.18.generators-tilde-x.tex
\begin{equation*}
%
%
\psfrag{1}{\hspace*{-1.4mm}   \small 0}
\psfrag{2}{\hspace*{-2.0mm}  \small }
\psfrag{3}{\hspace*{-2.3mm} \small }
\psfrag{4}{\hspace*{-1.0mm}\small  }
\psfrag{5}{\hspace*{-1.7mm}  \small  }
\psfrag{6}{\hspace*{-1.5mm}   \small $\tfrac{1}{5}$}
\psfrag{7}{\hspace*{-1.6mm}  \small $\tfrac{2}{5}$}
\psfrag{8}{\hspace*{-1.7mm} \small $\tfrac{3}{5}$}
\psfrag{9}{\hspace*{-0.4mm}\small  $\tfrac{4}{5}$}
\psfrag{10}{\hspace*{-1mm}  \small  1}
\psfrag{11}{\hspace*{-0.5mm}  \small  $0$}
\psfrag{12}{\hspace*{-0.9mm}  \small$\tfrac{1}{5}$}
\psfrag{13}{\hspace*{-0.8mm}  \small $\tfrac{2}{5}$}
\psfrag{14}{\hspace*{-1mm}  \small$\tfrac{3}{5}$}
\psfrag{15}{\hspace*{-0.9mm}  \small $\tfrac{4}{5}$}
\psfrag{16}{\hspace*{-1.7mm}  \small  }
\psfrag{17}{\hspace*{-1.8mm}  \small}
\psfrag{18}{\hspace*{-1.6mm}  \small }
\psfrag{19}{\hspace*{-1.6mm}  \small}
\psfrag{20}{\hspace*{-0.5mm}  \small 1}
\psfrag{21}{\hspace*{-0.7mm} \small 5}
\psfrag{22}{\hspace*{0.4mm}\small 5}
\psfrag{23}{\hspace*{-0.0mm}\small 5}
\psfrag{24}{\hspace*{-1.5mm} \small 5}
\psfrag{25}{\hspace*{-1.7mm} \small }
\psfrag{26}{\hspace*{-1.0mm} \small }
\psfrag{27}{\hspace*{-1.0mm} \small }
\psfrag{28}{\hspace*{-0.0mm} \small $\tfrac{1}{5}$}
\psfrag{29}{\hspace*{-0.0mm} \small $\tfrac{1}{5}$}
%
%
\psfrag{31}{\hspace*{-0.7mm}   \small 0}
\psfrag{32}{\hspace*{-1.0mm}  \small $\tfrac{1}{5}$}
\psfrag{33}{\hspace*{-2.3mm} \small }
\psfrag{34}{\hspace*{-1.0mm}\small  }
\psfrag{35}{\hspace*{-1.7mm}  \small  }
\psfrag{36}{\hspace*{-2.2mm}   \small }
\psfrag{37}{\hspace*{-0.3mm}  \small $\tfrac{2}{5}$}
\psfrag{38}{\hspace*{-0.7mm} \small $\tfrac{3}{5}$}
\psfrag{39}{\hspace*{-0.2mm}\small  $\tfrac{4}{5}$}
\psfrag{40}{\hspace*{-0.7mm}  \small  1}
\psfrag{41}{\hspace*{-1.0mm} \small 1}
\psfrag{42}{\hspace*{0.6mm}\small  5}
\psfrag{43}{\hspace*{-0.6mm}  \small 5}
\psfrag{44}{\hspace*{-1.0mm} \small 5}
\psfrag{45}{\hspace*{-1.7mm} \small }
\psfrag{46}{\hspace*{-1.7mm} \small }
\psfrag{47}{\hspace*{-1.4mm} \small }
\psfrag{48}{\hspace*{-0.7mm} \small $\tfrac{1}{5}$}
\psfrag{49}{\hspace*{0.1mm} \small $\tfrac{1}{5}$}
\psfrag{la1}{\hspace{-2.5mm} \small$f_0$}
\psfrag{la2}{\hspace{-2.5mm} \small $f_1$}
\includegraphics[width= 5.5cm]{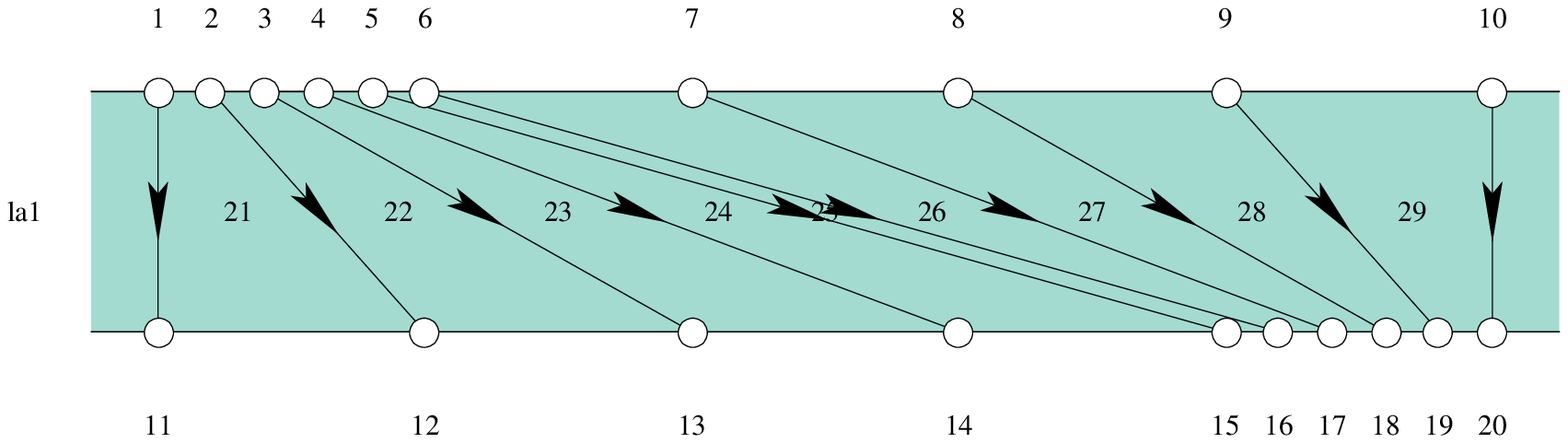}
\hspace*{6mm}
\includegraphics[width= 5.5cm]{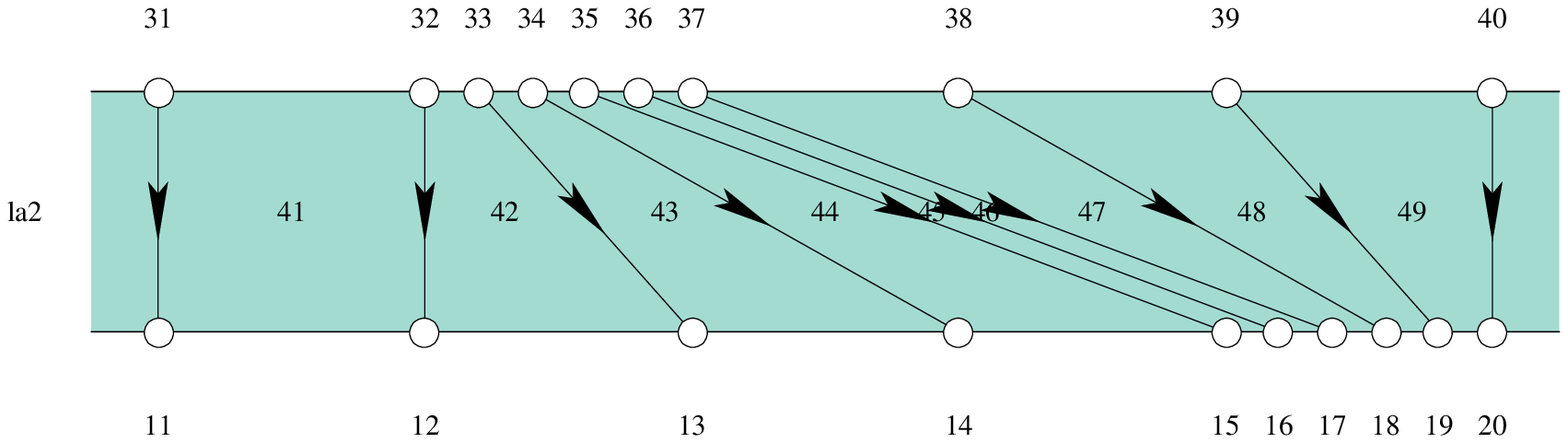}
\end{equation*}

\begin{equation*}
%
%
%
\psfrag{11}{\hspace*{-0.5mm}  \small  $0$}
\psfrag{12}{\hspace*{-0.9mm}  \small$\tfrac{1}{5}$}
\psfrag{13}{\hspace*{-0.8mm}  \small $\tfrac{2}{5}$}
\psfrag{14}{\hspace*{-1mm}  \small$\tfrac{3}{5}$}
\psfrag{15}{\hspace*{-0.9mm}  \small $\tfrac{4}{5}$}
\psfrag{16}{\hspace*{-1.7mm}  \small  }
\psfrag{17}{\hspace*{-1.8mm}  \small}
\psfrag{18}{\hspace*{-1.6mm}  \small }
\psfrag{19}{\hspace*{-1.6mm}  \small}
\psfrag{20}{\hspace*{-0.5mm}  \small 1}
\psfrag{51}{\hspace*{-0.7mm}   \small 0}
\psfrag{52}{\hspace*{-1.0mm}  \small $\tfrac{1}{5}$}
\psfrag{53}{\hspace*{-0.7mm} \small $\tfrac{2}{5}$}
\psfrag{54}{\hspace*{-1.0mm}\small  }
\psfrag{55}{\hspace*{-1.7mm}  \small  }
\psfrag{56}{\hspace*{-2.2mm}   \small }
\psfrag{57}{\hspace*{-2.0mm}  \small }
\psfrag{58}{\hspace*{-0.9mm} \small $\tfrac{3}{5}$}
\psfrag{59}{\hspace*{0.4mm}\small  $\tfrac{4}{5}$}
\psfrag{60}{\hspace*{-1mm}  \small  1}
\psfrag{61}{\hspace*{-0.8mm} \small 1}
\psfrag{62}{\hspace*{-0.0mm}\small  1}
\psfrag{63}{\hspace*{-0.7mm} \small 5}
\psfrag{64}{\hspace*{-0.7mm} \small 5}
\psfrag{65}{\hspace*{-1.7mm} \small }
\psfrag{66}{\hspace*{-1.7mm} \small }
\psfrag{67}{\hspace*{-1.7mm} \small }
\psfrag{68}{\hspace*{-0.3mm} \small $\tfrac{1}{5}$}
\psfrag{69}{\hspace*{-0.1mm} \small $\tfrac{1}{5}$}
%
%
%
\psfrag{71}{\hspace*{-0.7mm}   \small 0}
\psfrag{72}{\hspace*{-0.9mm}  \small $\tfrac{1}{5}$}
\psfrag{73}{\hspace*{-0.9mm} \small $\tfrac{2}{5}$}
\psfrag{74}{\hspace*{0.3mm}\small $\tfrac{3}{5}$}
\psfrag{75}{\hspace*{-1.7mm}  \small  }
\psfrag{76}{\hspace*{-2.2mm}   \small }
\psfrag{77}{\hspace*{-2.0mm}  \small }
\psfrag{78}{\hspace*{-2.1mm} \small }
\psfrag{79}{\hspace*{0.3mm}\small  $\tfrac{4}{5}$}
\psfrag{80}{\hspace*{-0.6mm}  \small  1}
\psfrag{81}{\hspace*{-1.7mm} \small 1}
\psfrag{82}{\hspace*{-0.0mm}\small  1}
\psfrag{83}{\hspace*{-1mm} \small 1}
\psfrag{84}{\hspace*{-1.0mm} \small 5}
\psfrag{85}{\hspace*{-1.7mm} \small }
\psfrag{86}{\hspace*{-1.7mm} \small }
\psfrag{87}{\hspace*{-1.7mm} \small }
\psfrag{88}{\hspace*{-1.7mm} \small }
\psfrag{89}{\hspace*{-0.4mm} \small $\tfrac{1}{5}$}
%
\psfrag{la3}{\hspace{-2.5mm}\small $f_2$}
\psfrag{la4}{\hspace{-2.5mm}\small $f_3$}
\includegraphics[width= 5.5cm]{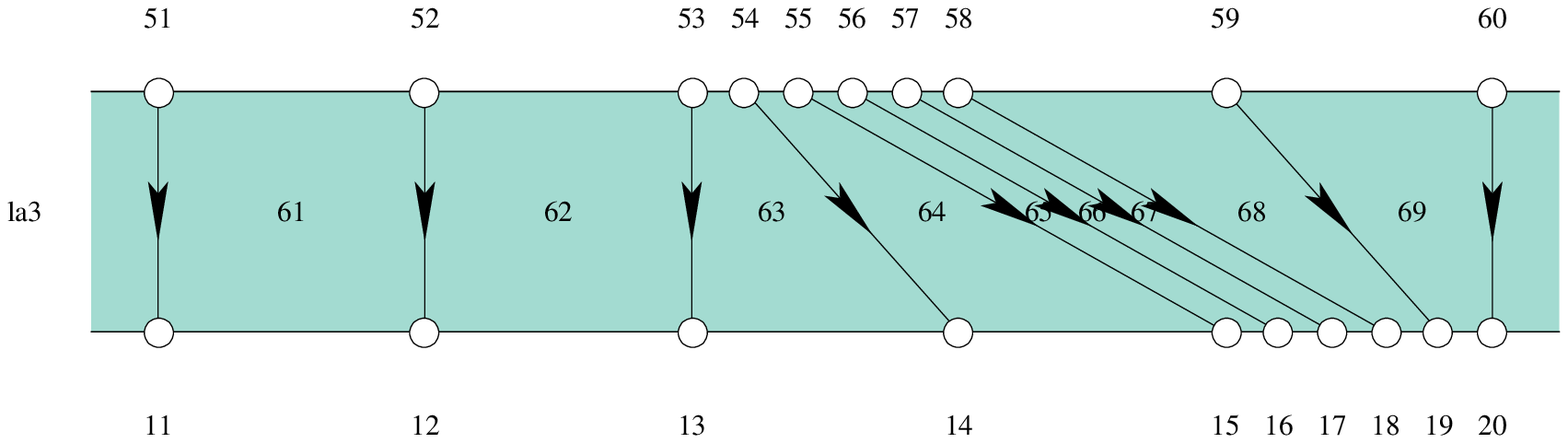}
\hspace*{6mm}
\includegraphics[width= 5.5cm]{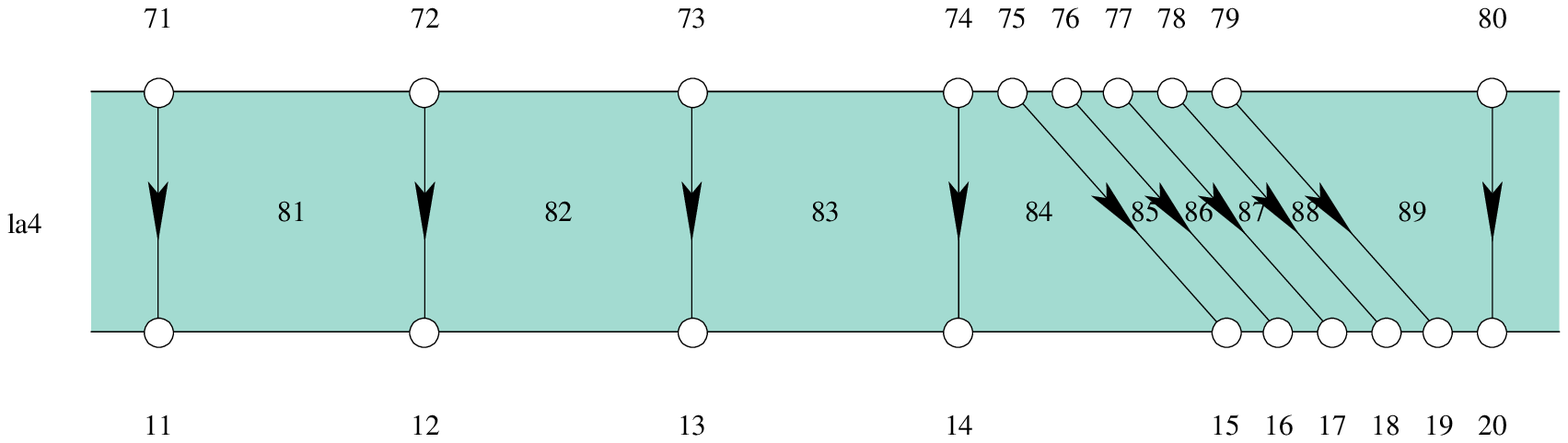}
%
\end{equation*}

%% file: chaptE.18.preimages1.tex
\begin{figure}
\psfrag{1}{\hspace*{-1.2mm} \small   0}
\psfrag{2}{\hspace*{-1.7mm} }
\psfrag{3}{\hspace*{-0.7mm}\small$t_2$ }
\psfrag{4}{ \hspace*{-1.5mm}}
\psfrag{5}{  \hspace*{-1.7mm}\small $t_4$}
\psfrag{6}{  \hspace*{-0.7mm}\small }
\psfrag{11}{  \hspace*{-1.7mm}}
\psfrag{12}{  \hspace*{-1.9mm}\small $i$}
\psfrag{13}{  \hspace*{-1.7mm}\small }
\psfrag{14}{  \hspace*{-3.2mm}\small $i+1$}
\psfrag{15}{ \hspace*{-1.0mm}\small  $p-1$}
\psfrag{16}{ \hspace*{-1.8mm}\small  $2(p-1)$}

\psfrag{21}{\hspace*{-1.2mm} }
\psfrag{22}{\hspace*{-2.2mm} \small $i$}
\psfrag{23}{\hspace*{-2.7mm}\small  $ p-1$}
\psfrag{24}{\hspace*{-6.5mm}\small $p + i$}
\psfrag{25}{\hspace*{-1.7mm}\small $2(p-1)$}
\psfrag{26}{\hspace*{-2.5mm}\small $3(p-1)$}
\psfrag{41}{\hspace*{-0.2mm} \small $p$  }
\psfrag{42}{\hspace*{-0.2mm} }
\psfrag{43}{  \hspace*{0.1mm}}
\psfrag{44}{  \hspace*{0.5mm}\small $p$}
\psfrag{45}{  \hspace*{-1.2mm}\small $p^2$}
\psfrag{51}{\hspace*{-0.7mm} \small $1$  }
\psfrag{52}{\hspace*{-1.7mm} \small $p$}
\psfrag{53}{  \hspace*{-1.8mm}\small $p$}
\psfrag{54}{  \hspace*{-0.8mm}\small 1}
\psfrag{55}{  \hspace*{0.7mm}\small 1}

\psfrag{61}{\hspace*{-1.7mm} \small $p^{-1}$  }
\psfrag{62}{\hspace*{-2.7mm} \small $p^{-1}$}
\psfrag{63}{  \hspace*{-1.9mm}\small $p^{-2}$}
\psfrag{64}{  \hspace*{-0.7mm}\small $p^{-2}$}
\psfrag{65}{  \hspace*{-0.7mm}\small $p^{-3}$}
\psfrag{varphiI}{\hspace{-1mm}\small $\varphi_1^{-1}$}
\psfrag{yy}{\hspace{-2.8mm}\small $y_i$}
\psfrag{varphi}{\hspace{-0.5mm}\small $\varphi_1$}
\begin{center}
\includegraphics[width= 11cm]{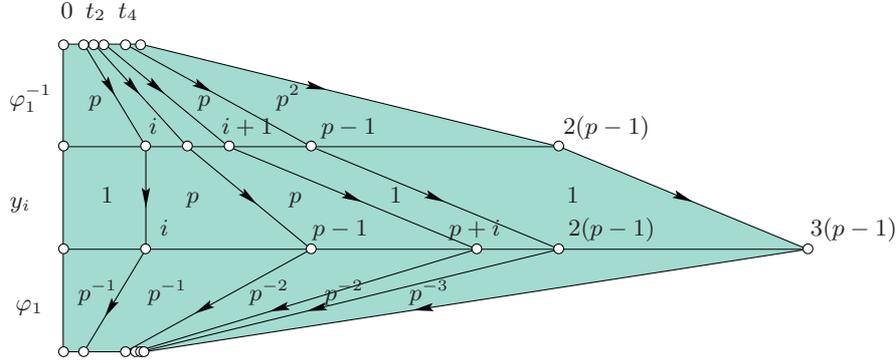}
\caption{PL-homeomorphism $\mu_1^{-1}(y_i)$ for $0 \leq i \leq p-2$}
\label{fig:PL-homeomorphism-mu1-invers-i-small}
\end{center}
\end{figure}

%% file: chaptE.18.preimages2.tex
\begin{figure}[b]
\psfrag{1}{\hspace*{-2.2mm} \small   0}
\psfrag{2}{\hspace*{-2.7mm} \small $t_1$}
\psfrag{3}{\hspace*{-0.9mm}\small$t_2$ }
\psfrag{4}{ \hspace*{-1.0mm}\small $t_3$}
\psfrag{5}{  \hspace*{-1.0mm}\small $t_4$}
\psfrag{11}{  \hspace*{-1.7mm}}
\psfrag{12}{  \hspace*{-3.7mm}\small}
\psfrag{13}{  \hspace*{-1.7mm}\small }
\psfrag{14}{  \hspace*{-1.2mm}\small $p$}
\psfrag{15}{  \hspace*{-1.7mm}\small $2(p-1)$}
\psfrag{21}{\hspace*{-1.2mm} }
\psfrag{22}{\hspace*{-2.2mm} \small $p-1$}
\psfrag{23}{  \hspace*{-5.5mm}\small $2(p -1)$}
\psfrag{24}{\hspace*{-2.7mm}\small $2p-1$}
\psfrag{25}{\hspace*{-2.7mm}\small $3(p-1)$}
\psfrag{41}{\hspace*{-1.2mm} \small $p$  }
\psfrag{42}{\hspace*{0.99mm} \small $p^2$}
\psfrag{43}{  \hspace*{0.2mm}\small }
\psfrag{44}{  \hspace*{-0.4mm}\small $p^2$}
\psfrag{51}{\hspace*{-1.2mm} \small $1$  }
\psfrag{52}{\hspace*{-0.7mm} \small $p$}
\psfrag{53}{  \hspace*{-0.3mm}\small $p$}
\psfrag{54}{  \hspace*{-1.2mm}\small 1}

\psfrag{61}{\hspace*{-1.2mm} \small $p^{-1}$  }
\psfrag{62}{\hspace*{-0.2mm} \small $p^{-2}$}
\psfrag{63}{  \hspace*{3.5mm}\small $p^{-3}$}
\psfrag{64}{  \hspace*{5mm}\small $p^{-3}$}
\psfrag{varphiI}{\hspace{-0mm}\small $\varphi_1^{-1}$}
\psfrag{yy}{\hspace{-3mm}\small $y_{p-1}$}
\psfrag{varphi}{\hspace{0.8mm}\small $\varphi_1$}
\begin{center}
\includegraphics[width= 11cm]{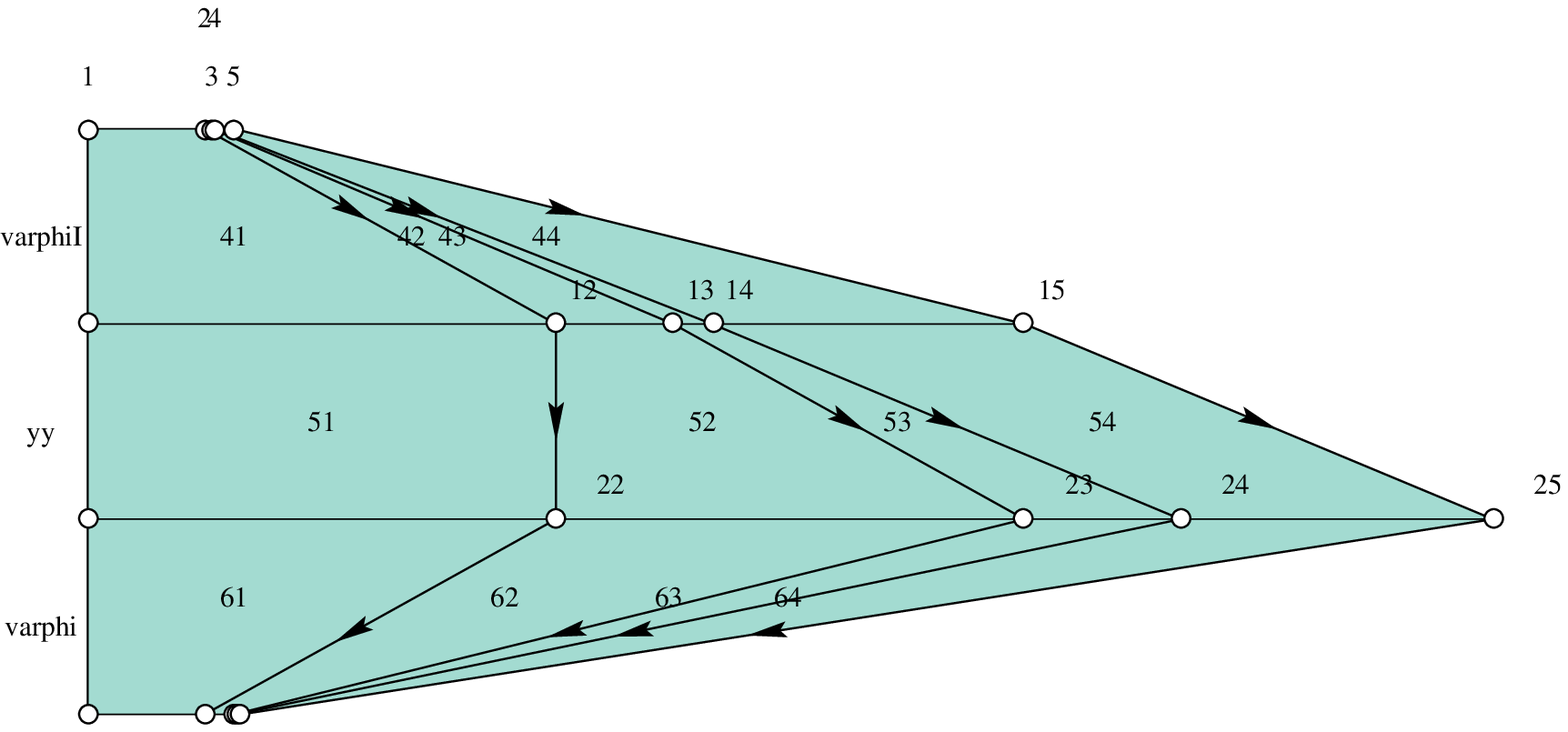}
\caption{PL-homeomorphism $\mu_1^{-1}(y_{p-1})$}
\label{fig:PL-homeomorphism-mu1(invers)-i=p-1}

\end{center}
\end{figure}

%% file: chaptN_Notes.tex
\renewcommand\descriptionlabel[1]%
{\hspace{\labelsep} \textit{#1}}
\thispagestyle{plain}
\renewcommand{\theequation}{\thesection.\arabic{equation}}
%
%
\chapter*{Notes}
\label{chap:Notes}
\markboth{Notes}{}
\addcontentsline{toc}{chapter}{Notes}

This chapter, written by Ralph Strebel, discusses two topics: 
the differences between the Bieri-Strebel memoir of 1985 and this monograph,
and a brief survey of newer articles related to the themes of the memoir.
\renewcommand{\thesection}{N\arabic{section}}

\section{Differences between memoir and monograph: summary}
\label{sec:Differences-summary}
%
When I started to revise the Bieri-Strebel memoir
I intended to produce a new version of the text with few changes,
limited to corrections of misprints and minor errors, 
and to elucidations of obscure passages.
Little by little, 
though, 
I detected that more profound modifications were called for.
%
\subsection{Numbering of results, remarks, equations and the likes}
\label{ssec:Numbering-results}
%
First, there were errors in the numbering of equations,
then comments and examples without a name.
In addition,
some names given to  results look strange for someone familiar with \LaTeX{}.
In Chapter A, for instance, there is a main result, 
called Theorem A
and a sequence of consequences, 
labelled Corollaries  A1, A1*,  A2,  A3 and A4, 
then a remark and finally Corollary A5.
I decided to use a sequential numbering throughout 
and so the names of the previously listed results are now:
Theorem \ref{TheoremA}, the first result in Section 4 inside Chapter A,
and then Corollaries 
\ref{CorollaryA1}, 
\ref{CorollaryA1*}, 
\ref{CorollaryA2}, 
\ref{CorollaryA3}, 
\ref{CorollaryA4},
and 
\ref{CorollaryA5}.
The item called ``remark'' in the original version is now 
Remark \ref{remark:5.3}. 
Section 5 contains also a new item, 
Remark \ref{remark:Usefulness-of-this-function}.

The previously listed results and remarks show
that their names refer both to the chapter 
and to the section in which they occur.
The reason for this double reference is this:
each of the five chapters of the memoir has a very specific theme.
I think it is useful for a reader to be reminded of this theme 
when encountering a reference to a result or remark 
and so I have included the names of the chapters into the names of the results.
On the other hand,
some of the chapters, in particular the last one, are quite long.
In locating a result referred to it is therefore useful 
to know in advance the section in which it is to be found.
\footnote{The counter for results and the likes is reset at the beginning of each section.}
%
\subsection{Subdivision of the memoir}
\label{ssec:Subdivision-etc}
%
I maintained the subdivision of the memoir into an Introduction followed by Chapters A through E,
and the subdivision of the chapters into sections, 
with numbers running from 1 to 19.
\footnote{The final Section 20 has not been included into the  monograph
for reasons adduced in section  \ref{ssec:Changes-E-Section20}.}
The results of these chapters are reproduced in the same order,
but I have inserted new results, remarks and examples.
The formulations of the results are mostly as before, 
disregarding some stylistic polishing, 
with two important exceptions:  
the statement of the core result Theorem E4 contained a hypothesis 
that is redundant;
I have removed it in Theorem \ref{TheoremE04}, 
the equivalent of Theorem E4,
 and in all of its corollaries;
secondly,
Supplement E11 states a claim that does not seem to follow from the offered proof;
it may be correct, but as I could not mend the original proof
this claim is no longer stated in Supplement \ref{SupplementE11New} 
and the results based on it.
%
\section{Differences between memoir and monograph: details}
\label{ssec:Changes-details}
%
In this section I list the more important changes, corrections as well as additions,
that distinguish this monograph from the memoir of 1985.
%
\subsection{Introduction}
\label{ssec:Changes-Intro}
%
There are few changes here, disregarding some rewordings.
The discussed results are, of course, cited in their corrected form; 
in addition,
I changed the names of the groups $G$ and $S$,
used in the memoir,  into $F$ and $T$, 
the names that have become standard for two of Richard Thompson's well-known groups.

The monograph has now two prologues,
the \emph{Introduction} and the \emph{Preface};
their subject matters overlap considerably.
I decided to maintain the old introduction as it reflects the outlook on the subject
Robert and I had when we wrote the memoir,
while the preface, besides surveying the results of the memoir,
tries also to assess its achievements.
%
\subsection{Chapter A}
\label{ssec:Changes-A}
%
The differences between the versions of Chapter A are quickly detailed.
Theorem A has been split into Theorem \ref{TheoremA} 
and Lemma \ref{lem:TheoremA}; 
the lemma states the basic step in the proof of Theorem A 
in a more memorable form.
The illustration in section 5.1 has been transferred to Section \ref{sec:4} 
and expanded into Illustration \ref{illustration:4.3}.
It prepares for the applications of Theorem  \ref{TheoremA} 
made in \cite{Ste92} and in \cite{Lio08}; 
see section \ref{sssec:Notes-Choices-of-A-and-P} for some comments.
\index{Stein, M.}%
\index{Liousse, I.}%
The proof of Corollary \ref{CorollaryA1*} contains now more details
than does the proof of Corollary A1$^*$,
Remark \ref{remark:Usefulness-of-this-function} has been added,
the statement of Corollary \ref{CorollaryA4} has been simplified, 
the illustration following Corollary A4 in the first version has been removed
and the proof of Corollary \ref{CorollaryA5} has been slightly expanded.
%
\subsection{Chapter B}
\label{sssec:Changes-B}
%
The changes in this chapter are of greater weight
than those in the Introduction and in Chapter \ref{chap:A};
they clarify proofs and correct errors.

\subsubsection{Section \ref{sec:7}} 
In the memoir,
formulae \eqref{eq:7.4n} and \eqref{eq:7.5n} 
are erroneously contracted into a single one,
namely
\[
\act{\aff(a,p)} g(a',p') = g(a + p(a'-a),p')
\]
and labelled (7-4). 
The lines following formula (7-7) 
become then incomprehensible.
In this monograph, 
the error has been corrected 
and the final part of the proof of Theorem \ref{TheoremB2} has been polished.

\subsubsection{Section \ref{sec:8}}
The original proof of Proposition \ref{PropositionB3} is complete, 
but the ideas behind its calculations do not stand out clearly.
I rearranged therefore the reasoning 
so as to make the strategy of the proof more transparent.

\subsubsection{Section \ref{sec:9}}
In the proof of the comparison result Lemma \ref{LemmaB6},
I replaced the original PL-homeomorphism by a PL-homeomorphism 
that is easier to define and which leads to a simpler justification.

$\PP$-regular subdivisions are finite sequences of points in $A \cap [0,1]$
of a particular kind and are constructed inductively.
They can be coded by a sequence of the form 
$(1,p_1;n_2.p_2; \ldots, n_\ell, p_\ell)$ 
(see the paragraph containing formula \eqref{eq:9.2} for explanation).
In the memoir, a subdivision is denoted by the letter $S$,
in this monograph,
it is denoted by $D$ and its code by $S$.
Illustration \ref{illustration:9.9.Example1} 
corresponds to the very short section 9.3 in the memoir.

In the memoir,
the proof of  Proposition \ref{PropositionB8}
and the Example \ref{example:Search-for-regular-subdivisions}
illustrating it are both  sketchy; 
the monographs  gives more details.
Remark \ref{remark:Number-of-singularities} does not figure in the memoir.

\subsubsection{Proof of Theorem \ref{TheoremB9}}
This proof is tricky. 
To help the reader in digesting it,
I reworked the description of the prospective set of generators 
(section \ref{ssec:9.5}),
added illustrations \ref{illustration:Prospective-generators} 
and \ref{illustration:9.9.Example1},
and rearranged the proof of Lemma \ref{LemmaB12}.

\subsubsection{Rectangle diagrams}
\index{Rectangle diagram!originator}%
Robert and I became acquainted 
with these diagrams through notes 
taken by Ross Geoghegan. 
\index{Geoghegan, R.}%
These notes recorded unpublished work of Matt Brin and Craig Squier
and contained examples of rectangle diagrams.
\index{Brin, M. G.}%
\index{Squier, C. C.}%
Brin and Squier, in turn, had learned the notion of rectangle diagram 
in an early version of the paper \cite{CFP96} by Cannon et al. 
The rectangle diagrams in the Bieri-Strebel memoir 
are thus also an offspring of Bill Thurston. 
\index{Thurston, W. P.}%
%
\subsection{Chapter C}
\label{ssec:Changes-C}
In the memoir,
Chapter \ref{chap:C} contains several passages 
which today seem rather sketchy, if not obscure;
some of them are erroneous.
Most of the results formulated in the chapter are, however, correct.
I have therefore decided to rewrite much of chapter 
and to round it off by several new examples.
Here is the list of the main changes and corrections. 

\subsubsection{Section \ref{sec:10}}
\index{Higman's simplicity result}%
In the memoir, 
a consequence of Higman's Theorem 1,
namely the first example of \cite[Section 3]{Hig54a},
is stated as ``Theorem'', 
but no proof is supplied.
As Higman's argument is fairly short, 
I decided to include a variant of Higman's Theorem 1 as 
Proposition \ref{prp:Higmans-Theorem1} and to explain 
how this result entails the simplicity of the derived group of $B(I;A,P)$.
(The proof of Proposition \ref{prp:Higmans-Theorem1} given in \cite{BiSt14}
has a gap, pointed out to me by Jakub Gismatullin.)
\index{Higman, G.}%
\index{Gismatullin, J.}%

\subsubsection{Section \ref{sec:11}}
In the 1985 text, 
the assertions stated now in Lemmata \ref{lem:Commutative-square-1}
and \ref{lem:Commutative-square-2} are given without proofs.
I decided to promote them into lemmata followed by proofs,
all the more so as the second assertion is erroneous:
the claim enunciated in the lower part of page C-5  
is that of Lemma \ref{lem:Commutative-square-2},
but with the hypothesis 
``$\varphi$ is a finitary PL-homeomorphism''
weakened to
``$\varphi$ is a PL-homeomorphism'';
this weaker hypothesis does not allow one to deduce the conclusion,
as is shown in section \ref{sssec:11.2b}.

\subsubsection{Section \ref{ssec:12.1}}
Lemma \ref{lem:Commutative-square-2} is used in the proof of Proposition \ref{PropositionC2}.
In the \emph{first part}
the PL-home\-omorphism $\varphi$ is affine and so the lemma applies.
This first part quotes also some details of the proof of Proposition \ref{PropositionB3}.
As this proof has been rearranged so as to become more readable,
I adopted the quotation accordingly.
The \emph{second part} of the proof 
relies in the memoir on the mistaken analogue of \ref{lem:Commutative-square-2}.
In this monograph, 
an entirely different approach is used. 
Corollary \ref{crl:PropositionC2} and 
Remarks \ref{remarks:Explicit-description-nu-ab}
do not figure in the memoir.

\subsubsection{Section \ref{ssec:12.2New}}
The second topic of Section \ref{sec:12}
is the \emph{abel\-ianization of the bounded group} $B = B([0,\infty[\,;A,P)$.
In a first step, $B_{\ab}$ is written as an extension of the cokernel of 
$H_2(\rho) \colon H_2(G_1) \to H_2(IP \cdot A \rtimes P)$ by an auxiliary  group $K(A,P)$;
here $G_1$ is the kernel of the surjection
\[
\sigma_- \colon G([0, \infty[\,; A, P) \to P.
\]
This first step relies on  Lemma \ref{LemmaC3New} 
which asserts that $G(I;A,P)$ 
acts by the identity on the abelian group $B(I;A,B)_{\ab}$
as long as $I\neq \R$.
The analogue of this lemma figuring in the memoir asserts more,
namely 
that the action of the group $G(I;A,P)$ 
on the homology groups  $H_j(B,\Z)$ is trivial in every dimension $j$.
This claim is correct, but its proof is rather sketchy.
As the general result is only needed in a minor point of the memoir
and as the length of a complete proof seems out of proportion 
with the interest in the claim,
I have omitted the more general assertion.

In the memoir, 
the first step is followed directly by a study of the cokernel,
the outcome being summarized in Lemma C5.
Here the study of the cokernel is given in section \ref{sssec:Analysis-coker-rho} 
and it is summarized by Proposition \ref{prp:LemmaC5}.
The description of $\coker H_2(\rho)$
calls for an investigation of the vanishing of the groups $H_j(P, IP \cdot A)$.
In the memoir, 
this analysis is the topic of section 12.4, 
here that of section \ref{sssec:12.2c}.
The findings are collected in Lemma \ref{LemmaC6};
this lemma improves on Lemma C6 in the memoir 
and is followed by a new illustration, 
Example \ref{example:LemmaC6}.

\subsubsection{Section \ref{ssec:12.3New}}
This section corresponds to section 12.4 in the memoir.
Its topic is the study of the abelian group $K(A, P)$.
The results obtained therein are then used 
in sections \ref{sssec:12.3c} and \ref{ssec:12.4New} of the monograph,
sections which replace sections 12.5 and 12.6 in the memoir.

The group $K(A,P)$ is (isomorphic to) the kernel of the homomorphism
\begin{equation}
\label{eq:12.5New-bis}
\index{Homomorphism!17-rhobar@$\bar{\rho}$}%
\bar \rho \colon \left\{ 
\begin{aligned}
\Z[(A \,\cap \;]0,\infty[)_\sim] \otimes P 
&\longrightarrow 
\left( (IP \cdot A)/(IP^2\cdot A) \right)\times P,\\
G_1 \cdot a \otimes p &\longmapsto ((1-p)a+IP^2\cdot A,p).
\end{aligned} \right.
\end{equation}
In this formula,
$(A\, \cap\, ]0, \infty[\,)_\sim$ denotes the space of orbits in  $A\, \cap\, ]0, \infty[$,
of the group $G_1 = \ker \left( \sigma_- \colon G([0, \infty[\;;A,P)\right)$.
This orbit space has the cardinality of $A/(IP \cdot A)$
and so the domain of the map $\bar{\rho}$ 
is written $\Z[A/(IP \cdot A)]$ in the memoir.
This is, however, a dangerous substitution: 
the elements of the orbit space are represented 
by \emph{positive} elements of $A$,
while the cosets in $A/(IP\cdot A)$ are represented by arbitrary elements of $A$.
The formula describing $\bar{\rho}$ in the memoir does not respect this fact.

In the monograph,
the domain of $\bar{\rho}$ is always written $\Z[(A \,\cap\,]0,\infty[\,)_\sim]$
and great care is exercised 
when the facts 
that $(A \,\cap\,]0,\infty[\,)_\sim$ and $A/(IP \cdot A)$ have the same cardinality 
and that $A/(IP \cdot A)$ has the structure of an abelian group, are used.

The proof of Proposition  \ref{prp:LemmaC5}
relies on an auxiliary result,
called Lemma \ref{lemma:Quotients-A/IPA-and-IPA/IP2A}.
The proof of this lemma makes use of an elementary argument
which has been communicated to me by J. R. J. Groves;
\index{Groves, J. R. J.}%
it replaces the more conceptual argument,
based on properties of the tensor product,  given in the memoir.
The second part of the proof of Proposition \ref{prp:LemmaC5} is new 
and relies on the explicit construction of non-trivial elements of $B_{\ab}$ 
carried out in Example \ref{example:Elements-in-B}.

Corollary \ref{crl:Bab-for-P-cyclic-and-A-locally-cyclic}
is an application of Proposition  \ref{prp:LemmaC5};
it is called Proposition C9 in the memoir.
Its proof establishes a claim 
that is part of Lemma C8.

Example \ref{example:P-cyclic-A-locally-cyclic} is an addition to this text;
it replaces Corollary D12,
a result stated in Remark 15.3 
which figures no longer in the monograph.
\footnote{The Note \ref{sssec:Remark-15.3} gives the reason for this deletion}
In this example I also point out a fact, 
stated by M. Stein on page 478 of her article \cite{Ste92}:
the derived group of $G([a,c];A,P)$ 
coincides with the derived group of $B([a,c];A,P)$, and so it is simple,
for every slope group $P \neq \{\id \}$ and every $\Z[P]$-module $A$.
\index{Stein, M.}%
\footnote{see also sections \ref{sssec:Notes-ChapterC-GaGi16}
and \ref{sssec:Notes-ChapterC-Brown-Stein}}

\subsubsection{Section \ref{ssec:12.4New}}
This section follows closely section 12.6 of the memoir,
but supplements the final result by a more detailed discussion of its utility.
Examples \ref{examples:TheoremC10} have no counterpart in the memoir.

\subsubsection{Section \ref{ssec:Action-TildeG}}
Let $\tilde{G}$ denote the group $G(\R;A,P)$ and $\tilde{B}$ its subgroup of bounded elements.
The aim of sections 12.7 and 12.8 in the memoir is to show,
by way of examples, 
that $\tilde{G}$ can act non-trivially on the abelianization of $\tilde{B}$,
in contrast to the conclusion of Lemma \ref{LemmaC3New}.
This aim is attained in a roundabout manner.
One considers first the extension $\tilde{B} \incl \tilde{G} \epi \tilde{G}/\tilde{B}$
and its associated 5-term sequence
\begin{equation*}
\label{eq:12.12}
H_2(\tilde G)  \xrightarrow{H_2(\lambda,\rho)} 
H_2(\tilde G/\tilde B)  \xrightarrow{d_2} 
\tilde B/[\tilde B,\tilde G]  \xrightarrow{\iota_*}
\tilde G_{\ab} \xrightarrow{(\lambda,\rho)_{\ab}} 
(\tilde{G}/\tilde B)_{\ab}  \to 1
\end{equation*}
and verifies then
that the map $(\lambda,\rho)_{\ab}$  is bijective.
The group $\tilde B/[\tilde B,\tilde G]$ is thus a quotient of the multiplicator 
of the quotient group $\tilde{G}/\tilde B$; 
this group is metabelian.
In the special case where $P$ is cyclic,
generated by an integer $p \geq 2$, and $A = \Z[ 1/p]$,
a presentation of $\tilde{G}/\tilde B$ is then constructed
that allows one to deduce
 that $H_2(\tilde G/ \tilde B)$, and hence $\tilde B / \tilde B , \tilde G]$, 
are generated by at most 2 elements.
The group ${\tilde B}_{\ab}$, on the other hand, is isomorphic to $B_{\ab}$ 
(by Proposition \ref{PropositionC1}),
hence free abelian of rank $p-2$ 
by Illustration \ref{illustration:4.3} and Corollary \ref{crl:Bab-for-P-cyclic-and-A-locally-cyclic}.
So the commutator subgroup $[\tilde B, \tilde B]$ of $\tilde{B}$
is a proper subgroup of $[\tilde B , \tilde G]$ whenever $p > 4$.

The proof given in the monograph is straightforward:
the orbits $\Omega$ of the action of $\tilde{B}$ on $A$ 
are the cosets of $IP \cdot A$ of $A$.
The homomorphism $\nu$ 
and the PL-homeomorphisms discussed in Example \ref{example:Elements-in-B}
imply then
that every translation with amplitude $a \in A \smallsetminus IP \cdot A$ 
acts non-trivially on $\tilde{B}_{\ab}$.
%
\subsection{Chapter D}
\label{ssec:Changes-D}
The changes in this chapter are less important 
than those in Chapter \ref{chap:C}.
They aim mainly at making the proofs more readable.

\subsubsection{Preamble to Chapter \ref{chap:D}} 
This introductory text is new.

\subsubsection{Section \ref{sec:13}}
This section studies presentations of the group $G(\R;A,P)$.
The starting point are relations
\eqref{eq:13.3p}, \eqref{eq:13.4p}, \eqref{eq:13.5p} and \eqref{eq:13.5pp}.
In the memoir,
few details about the derivation of these relations are given;
for the monograph,
I have therefore supplied some details.
Remark \ref{remark:PropositionD1} does not figure in the memoir.

Section \ref{ssec:13.3New} is a combination of sections 13.5 and 13.6 
in the memoir.
The proof of Proposition \ref{PropositionD4} is rather tricky.
I have therefore added an introductory section \ref{sssec:13.3aNew};
it explains where the crux of the matter lies.
Moreover,
I have replaced the term \emph{$\PP$-adic expansion}, 
used in the memoir,
by the locution \emph{generated as a semi-group with operators in $\PP$}.
The examples in section \ref{sssec:13.3cNew} 
are based on Proposition \ref{PropositionD4} mentioned before.
They are identical with those discussed in the memoir.

\subsubsection{Section \ref{sec:14}}
In the memoir, 
the analogue of Theorem \ref{TheoremD6} 
and the proof of this analogue make up sections 14.3 through 14.7.
For the monograph,
I have combined these sections into a single one 
and I have inserted an application into the proof;
it is given in section  \ref{sssec:14.3bNew}.
In addition,
I have expanded the third stage of the proof and corrected an error.

\subsubsection{Proofs of Propositions \ref{PropositionD8} and \ref{PropositionD10}}
The proofs of these results involve calculations that are rather similar. 
In the memoir, 
the two proofs are shown to be consequences of a single general argument.
Here I give, in the interest of clarity, direct proofs for each of the two cases.

\subsubsection{Remark 15.3 in the memoir} 
\label{sssec:Remark-15.3}%
This remark describes an approach to the computation of $B_{\ab}$
where $B$ is the group $B(I;A,P)$ and $I = [a,c]$ is a compact interval with endpoints in $A$.
The idea is to use the isomorphism 
\[
\alpha \colon  B([0,\infty[\, ; A, P) \iso B(I; A, P),
\]
afforded by Proposition \ref{PropositionC1}, 
and the monomorphism  
\[
\nu_{\ab} \colon B([0,\infty[\, ; A, P)_{\ab} \mono \Z[(A \,\cap \,]0,\infty[)_\sim] \otimes P,
\] 
studied in section \ref{ssec:12.2New},
to derive information on the similarly defined homomorphism 
\[
\nu_* \colon B([a,c]; A, P)_{\ab} \mono \Z[(A \,\cap \,]a,c[)_\sim] \otimes P.
\]
The isomorphism $\alpha$ is induced 
by conjugation by an infinitary PL-homeo\-mor\-phism $\varphi$,
whence the square 
occurring in the statement of Lemma \ref{lem:Commutative-square-2}
may not be commutative.
This property, however, is used in the analysis carried out in Remark 15.3.
So the conclusions of the remark are in doubt.

One of these conclusions is Corollary D12,
the statement of which is identical with 
that of Corollary \ref{crl:Bab-for-P-cyclic-and-A-locally-cyclic}.
%
%
\subsection{Chapter E: minor changes in Sections \ref{sec:16} and \ref{sec:17}}
\label{ssec:Changes-E-minor-in-16-and-17}
In Chapter \ref{chap:E} there are minor changes 
and two important ones which deserve detailed comments.
I begin with minor changes.

 \subsubsection{Preamble to Chapter \ref{chap:E}} 
 This preamble does not figure in the memoir.

 \subsubsection{Section \ref{sec:16}}
 The proof of the original version of Lemma \ref{LemmaE2} contains a gap.
In the monograph,
 it is closed by replacing the old hypothesis by a stronger one.
One arrives thus at the actual Lemma \ref{LemmaE2} and its proof.
This modification has no impact on the statement of 
Theorem \ref{TheoremE3} and its main application,
Theorem \ref{TheoremE04}.
The proof of the ancestor of Theorem \ref{TheoremE3},
given in the memoir, omits details of some verifications;
here, they have been filled in.

In the memoir,
the proof of the uniqueness part of the analogue of Proposition \ref{PropositionE6} 
is rather sketchy.
Here the missing bits have been added.
Similarly,
the proofs of Propositions \ref{PropositionE8} and \ref{PropositionE9}
have been expanded so as to render easier their comprehension.

\subsubsection{Section \ref{sec:17}}
The proof Corollary \ref{CorollaryE12} is the same as in the memoir,
but I changed the statement of the corollary slightly:
in the memoir, 
the isomorphism types of the groups in the family 
$\{G([0,b];A,P) \mid b \in A_{>0} \}$
are shown to correspond to the orbits of the group $\Aff(IP \cdot A, \Aut_o(A))$ acting on $A$;
here, 
the isomorphism types match the orbits of the group $\Aut_o(A)$
or, equivalently, the orbits of the group $\Aut_o(A)/P$
acting on the group $A/(IP \cdot A)$.
This second formulation shows more clearly than the previous one
that the determination of the isomorphism types often boils down
to the geometric problem of finding the connected components of a finite graph;
this happens for instance, 
if $A = \Z[P]$ and $P$ is generated by finitely many positive integers.
 
In the memoir,
the statement of  Corollary \ref{CorollaryE13} is rather condensed.
Here a more leisurely wording is given 
and details have been added to the proof.

Corollary \ref{CorollaryE13} 
provides explicit descriptions of the outer automorphism group of 
the group $G(I;A,P)$ 
with $P$ a non-cyclic group  and various types of intervals $I$.
It allows one to obtain a description of the automorphism group of $G(I;A,P)$.
It is given in Corollary \ref{crl:Explicit-description-of-Aut(G)}
that is an addition to the monograph.

\subsubsection{Group $\Aut(A)$}
\index{Group Aut(A)@Group $\Aut(A)$!significance}%
In the course of Section \ref{sec:17} a new concept comes to the fore,
the group 
$\Aut(A) = \{s \in \R^\times \mid s \cdot A = A \}$,
called the \emph{automorphism group} of $A$, 
and its subgroup $\Aut_o(A) = \{s \in \R^\times_{>0} \mid s \cdot A = A \}$.
To help the reader in getting familiar this new concept,
some examples of automorphism groups are given in section
\ref{sssec:Automorphism-groups-some-examples-non-cyclic-P},
along with examples illustrating part (v) of Corollary \ref{CorollaryE13}.

%
\subsection{Chapter E: Theorem \ref{TheoremE04} and its proof}
\label{ssec:Changes-E-TheoremE04}
%
Theorem \ref{TheoremE04}
is a consequence of Theorem \ref{TheoremE3}
and the transitivity properties of the derived group of $B = B(I;A,P)$ 
and those of the derived group of $\bar{B} = B(\bar{I};\bar{A},\bar{P})$, 
established in Section \ref{sec:5}, in particular, in Remark \ref{remark:5.3}.

The derived group of $B$ and that of $\bar{B}$ must satisfy a technical condition,
\emph{they must contain strictly positive elements}.
In 1985, 
the existence of such elements was not known to Bieri and Strebel;
accordingly,
one reads at the bottom of page E-8 of the memoir: 
``We do not know whether $B'$ contains always a strictly positive element''
and so the existence of such elements is a hypothesis 
in Theorem E4 and in its corollaries.
In \cite{McRu05}, S. H. McCleary and M. Rubin then point out
that the existence of such elements is a consequence of the remaining hypotheses.
On page 125 they write: 
\begin{quotation}
From a more general result [\ldots],
Bieri and Strebel \cite[Theorem E4]{BiSt85}
deduce part (2) of the present Theorem 7.6  [\ldots],
provided each group $H_i$ contains a positive element. 
Here, besides deducing the entire First Order Reconstruction package 
(see Theorem 7.6), 
we find that
the positivity hypothesis can be discarded.
\end{quotation}
\index{McCleary, S. H.}%
\index{Rubin, M.}%
I noticed at the beginning of the year 2014
that the existence of strictly positive elements can also be established 
by a direct construction.

In this monograph,
this construction is used twice in the proof of Theorem \ref{TheoremE04}: 
firstly to show 
that the existence of strictly positive elements with bounded support need not be postulated 
and secondly to establish 
that the homeomorphism $\varphi$ maps the subset $A\, \cap\, \Int(I)$ 
onto $\bar{A}\, \cap\, \Int(\bar{I})$. 
The proof of the preparatory Theorem \ref{TheoremE3} 
is an adaptation of the proof of the Main Theorem 4 in \cite{McC78b}.
\index{McCleary, S. H.}%
%
\subsection{Chapter E: Supplement \ref{SupplementE11New} and its proof}
\label{ssec:Notes-SupplementE11New}
%
Theorem \ref{TheoremE10} considers an isomorphism $\alpha \colon G \iso \bar{G}$
where $G$ is a subgroup of $G(I;A,P)$ containing the derived group of $B(I;A,P)$
and where $\bar{G}$ has analogous properties.
The theorem shows that $\alpha$ is induced by a PL-homeomorphism $\varphi \colon \Int(I) \iso \Int(\bar{I})$
with slopes in a single coset $s \cdot P$ of $P$ and that $\bar{P} = P$.
The aim of the supplement to Theorem \ref{TheoremE10} is to relate the intervals $I$ and $\bar{I}$ under the additional hypothesis that $G = G(I;A,P)$
and that $I$ is either a line, a half line with endpoint in $A$ 
or a compact interval with endpoints in $A$.

In the memoir, 
the result corresponding to Supplement  \ref{SupplementE11New} claims 
that $\bar{I}$ is then of the same type as $I$.
The supplement in this monograph asserts less:
it assumes at the outset that $I$ and $\bar{I}$ have the same type
and claims
that $\bar{G} = G(\bar{I};\bar{A},P)$ 
and that $\varphi$ is a finitary PL-homeomorphism.

If $I$ is the line $\R$,
the assumption that $I$ and $\bar{I}$ have the same type can be dispensed with
(see Lemma \ref{lem:Isomorphisms-and-type-of-intervals-2}).
But, as pointed out in Remark \ref{remarks:SupplementE11New}  (ii),
I do not know
whether or not there exists, \eg{}a monomorphism 
$\beta_b \colon G([0,b]; A, P) \mono  G([0,\infty[\;; A, P)$ 
\emph{with $P$ not cyclic}
and so that $\im \beta_b$ contains the derived group of $B([0,\infty[; A, P)$. 
(Such monomorphisms exist if $P$ is cyclic; 
see section \ref{ssec:18.2}.)
%
\subsection{Chapter E: minor changes in Sections \ref{sec:18} and \ref{sec:19}}
\label{ssec:Changes-E-Sections18and19}
%
\subsubsection{Section \ref{sec:18}} 
\label{sssec:Changes-E-Section-18}
In the memoir, 
the proof of the analogue of Theorem \ref{TheoremE14} relies on an argument
that is now used in Part 2 of the proof of Supplement \ref{SupplementE11New}.
The proof in the monograph refers to this Part 2.
No proof of Lemma \ref{LemmaE15} figures in the memoir
and so I have added one for this monograph.
In addition,
I have corrected the formula for $\bar{u}_{2,b}$
and the preamble to Lemma E18.8.
The claim of Proposition \ref{prp:Preimages-yi} is not justified in the memoir.
Here a rather lengthy proof is provided 
and several rectangle diagrams have been added.

Corollary \ref{crl:Respecting-attractors} is new.
It is a simple,  but useful observation on increasing isomorphisms induced 
by conjugation by homeomorphisms
that is used in the proof of Lemma \ref{lemma:The-three-groups-are-not-isomorphic}
and implicitly in the proof of Lemma \ref{lem:Isomorphisms-special-case}.
Lemmata \ref{lemma:The-three-groups-are-not-isomorphic}
and \ref{lem:Piecing-together} are likewise new.
The proofs of these results justify a claim 
voiced at the very end of section 18.6 in the memoir.

Section \ref{sssec:18.8b} is a thoroughly rewritten 
and greatly expanded version of section 18.8 in the memoir.
It embodies ideas 
that go back to the paper \cite{BlWa07} by  C. Bleak and B. Wassink
and it promotes the statement made in the last two lines of section 18.8 
of the memoir to Theorem \ref{thm:Isomorphism-types-class-II}.
Remark \ref{remark:Analogues-of-thm-ismorphism-types-class-II}
 is a reminder that seems called for at this juncture,
but which is not part of the memoir.

Example 18.9 in the memoir has been discarded.
There exists by now far more general results which bring home the fact 
that the integral homology groups of $G([0,b]; A, P)$ 
do not allow one to determine 
whether, 
in case of non-cyclic slope groups,
 intervals of different lengths give rise to non-isomorphic groups;
 see \cite[Lemma 4.1]{Ste92}.
\index{Bleak, C.}%
\index{Wassink, B.}%
\index{Stein, M.}%
%
\subsubsection{Section \ref{sec:19}}
\label{sssec:Changes-E-Section-19}
This section has been reworked thoroughly.
 I have combined sections 19.2, 19.3 and 19.4 
 into the new section \ref{ssec:19.2New}.
 More important, however, are the additions to the proofs and remarks.

Remark \ref{remark:Homomorphism-gamma} figures in the first version,  
but with fewer details. 
 I changed the name $G(\R/ \Z \pfr; A, P)$, employed in the memoir, 
 into $T(\R/ \Z \pfr; A, P)$ 
 so as to align it with the notation used by Melanie Stein in \cite{Ste92}.
 
The long section Section \ref{sssec:19.2b} discusses the automorphism group 
and the outer automorphism group of $G = G([0, \infty[\, ;A,P)$
in the case where $P$ is a cyclic group generated by a real number $p > 1$.
By Theorem \ref{TheoremE04} the determination of $\Aut G$ 
amounts to finding the normalizer $\Autfr G$ of $G$ in $\Homeo(\,]0,\infty[\;)$.
In the memoir $\Autfr G$ is described in Lemma E18
by exhibiting subgroups that generate it.
Here,
I pass from the very beginning to the copy $\bar{G}$ of $G$ in  $G(\R;A,P)$, 
afforded by Lemma \ref{LemmaE16},
and characterize then $\bar{G}$ 
by conditions involving the homomorphism $\lambda$ and $\rho$; 
see Proposition \ref{prp:LemmaE18New}.
A consequence of this characterization is Corollary \ref{crl:PropositionE19};
it describes $\Outfr \bar{G}$ in the same way 
as does Proposition E19 in the memoir.

Section \ref{ssec:19.3New},
the last part of Section  \ref{sec:19}
discusses the automorphism group of the group $G =G([0,1]; A, P)$,
assuming that the number 1 lies in  $A$.
 As before, one passes to the realization $\bar{G} \subset G(\R;A,P)$ 
 and then restricts attention to the subgroup $\Autfr_{\PL} \bar{G}$ of $\Autfr \bar{G}$ 
 made up by the elements which are PL.
 Two results are obtained: a characterization of the elements of $\Autfr_{\PL} \bar{G}$ 
 and a description of $\Outfr_{\PL} \bar{G}$ as an extension of two better known groups; 
 see Corollary \ref{crl:PropositionE20New-part-II}.
 (The corollary corrects an error contained 
 in the the final assertion of Proposition E20 in the memoir.)
%
\subsection{Chapter E: omitted Section 20}
\label{ssec:Changes-E-Section20}
%
Section \ref{sec:19} investigates the group of automorphism of a group 
$G = G(I;A,P)$ with cyclic slope group $P$.
It shows, in particular, that every automorphism of such a group is induced 
by conjugation by a PL-homeomorphism if $I$ is a line or a half line,
but it does not answer the question 
whether this conclusion continues to be valid
if $I$ is a compact interval with endpoints in $A$.
The memoir finishes therefore with a section
that tries to shed light on this unanswered question.

Closer scrutiny of this attempt discloses 
that the exposition is uneven and the results themselves are premature.
The only group to which the obtained results seem to apply 
is Thompson's group $F$;
for this group, however, far more detailed and complete results 
have since been obtained by Matthew Brin in \cite{Bri96}.
\index{Brin, M. G.}%
Brin proves, in particular,
that every automorphism of a group $G$ sandwiched between $F'$ and $F$
is induced by conjugation by a PL-homeomorphism;
see Theorem 1 in \cite{Bri96}.
\index{Brin, M. G.}%
In addition, it is now known,
thanks to the work of M. Brin and F. Guzm{\'a}n, 
\index{Guzm{\'a}n, F.}%
that the automorphism groups of the generalized Thompson groups 
$G[p] = G([0,1]; \Z[1/p], \gp(p))$
behave quite differently if the integer $p$ is greater than 2:
there exists then automorphisms which are \emph{exotic}, 
\index{Exotic automorphisms!existence}%
\ie{}induced by conjugation by a homeomorphism that is not piecewise linear;
see \cite[p.\,285]{BrGu98} and the comment on item (III) on page 289.

In view of the stated facts I have decided not to include Section 20,
or a slight amendment of it, into this monograph.
%
%
\section{Related articles}
\label{sec:Related-articles}
%
The Bieri-Strebel memoir studies a collection of groups that depend on three parameters, 
the interval $I$ containing the supports,
the group $P$ containing the slopes
and the $\Z[P]$-submodule $A$ of $\R_{\add}$
where the breaks lie.
For some of the discussed problems, 
the memoir, and hence this monograph, is best viewed 
as a complement to the more recent literature,
providing a background to more detailed results
on groups defined by more special values of $I$, $A$ and $P$.

In this section,
I survey some articles 
that use findings of the memoir or are related to them.
%
\subsection{Chapter A}
\label{ssec:Related-results-A}
%
Groups of PL-homeomorphisms of the real line are, 
of course, 
examples of groups of automorphisms of a linearly ordered set.
In the literature on such groups one often finds the hypothesis
that the group act $\ell$-fold transitively
(in the sense of order preserving permutation groups) 
on an orbit $\Omega$.
The first example in \cite[§ 3]{Hig54a} furnishes an illustration with $\ell=2$, 
\index{Higman's simplicity result}%
while Corollary 3 and the Main Theorem 4 in \cite{McC78b} 
provide examples with $\ell = 3$.
\index{Higman, G.}%
\index{McCleary, S. H.}%

In the monograph, 
the hypotheses are expressed in terms of parameters $I$, $A$, $P$,
and, in some cases, by the position of the group inside the group $G(I;A,P)$.
There exists a basic result,
namely Theorem \ref{TheoremA}, 
\index{Theorem \ref{TheoremA}!discussion}%
which brings to light that this kind of hypotheses has useful consequences;
it asserts 
that $G(\R;A,P)$ contains a PL-homeomorphism mapping an interval $[0,b]$ with endpoint in $A$ 
onto another such interval $[0,b']$ if,
and only if,  $b'-b $ lies in the submodule $IP \cdot A$.
\index{Submodule IPA@Submodule $IP \cdot A$!significance}%

This result is remarkable in several ways:
first of all, it holds for all choices of $A$ and $P$,
and it involves the submodule $IP \cdot A$
which is familiar from Homology Theory of Groups.
\index{Homology Theory of Groups!submodule IPA@submodule $IP \cdot A$}%
It implies 
that the orbits of $G = G(I;A,P)$ are dense in $A \cap I$
for all choices of $I$, $A$ and $P$,
and that $G$ acts $\ell$-fold transitively on every such orbit 
for any given $\ell \in \N$,
provided merely that $I \neq \R$ or that $IP \cdot A = A$.
The quotient module $A/(IP \cdot A)$ is involved in other questions 
treated in the monograph:
\index{Quotient group A/IPA@Quotient group $A/(IP \cdot A)$!significance}%
finite generation of $G(I;A,P)$ (see Proposition \ref{PropositionB1}),
\index{Group G([0,infty[;A,P)@Group $G([0, \infty[\;;A,P)$!finite generation}%
finite generation of $B_{\ab}$ (see Proposition \ref{prp:LemmaC5} and Lemma \ref{LemmaD3})
\index{Subgroup B(I;A,P)@Subgroup $B(I;A,P)$!finite generation of Bab@finite generation of $B_{\ab}$}%
and finite presentability of $G(\R;A,P)$ (see Proposition \ref{PropositionD2}).
\index{Group G(R;A,P)@Group $G(\R;A,P)$!finite presentation}%

It also sheds light on the abelianization of the generalized Thompson group 
$T(\R/\Z \pfr;A,P)$
(see Remark \ref{remark:Homomorphism-gamma}). 
\index{Group T(R/Zpfr;A,P)@Group $T(\R/\Z \pfr;A,P)$!abelianization}%
\index{Thompson, R. J.}%
A related use occurs in the paper \cite{BrGu98};
\index{Brin, M. G.}%
\index{Guzm{\'a}n, F.}%
see section \ref{sssec:Notes-ChapterE-remark-in section-19.1} below for details.
%
\subsubsection{Choices of $A$ and $P$}
\label{sssec:Notes-Choices-of-A-and-P}
\index{Module A@Module $A$!select examples}%
\index{Group P@Group $P$!select examples}%
Theorem \ref{TheoremA} holds for all choices of $A$ and $P$.
In the literature on generalized Thompson groups typically very special choices are considered.
Most prominent among them is the case 
where $A$ is the subring $\Z[P]$ generated by the group $P$;
the quotient module $A/(IP \cdot A)$ is then a cyclic group.

As regards $P$, 
several authors choose $P$ to be a group generated by finitely many positive integers $p_i > 1$, 
in particular Ken Brown in \cite[Section 4]{Bro87a}, \index{Brown, K. S.}%
Matt Brin in \cite{Bri96}, \index{Brin, M. G.}%
then M. Brin and F. Guzm{\'a}n in \cite{BrGu98}, \index{Guzm{\'a}n, F.}%
Isabelle Liousse in \cite{Lio08} \index{Liousse, I.}%
and Dongping Zhuang in \cite{Zhu08}. \index{Zhuang, D.}%
The article \cite{Ste92} by Melanie Stein is an exception: \index{Stein, M.}%
in some parts the author uses the special choice mentioned before,
but in the proof of Lemma 4.1 and in Section 5 the parameters $A$ and $P$ can be arbitrary.
 The papers \cite{Cle95} and \cite{Cle00} \index{Cleary, S.}%
 deal with another kind of subgroup $P$:
in the first of them, 
$P$ is the cyclic group generated by the algebraic integer $p = 1 + \sqrt{2}$,
in the second $P$ is generated by the algebraic integer $(1 + \sqrt{5}\,)/2$.

From the point of view of an analyst
the group $G(\R;\R, \R^\times_{>0})$ and its subgroups $G(I;\R, \R^\times_{>0})$ are also of interest.
These groups have been studied by C. G. Chehata in \cite{Che52};
see section \ref{sssec:Notes-ChapterC-Chehata} below.
\index{Chehata, C. G.}%
The module $A$ is then, once more, 
the ring generated by the group $P = \R^\times_{>0}$.
%
\subsection{Chapter C}
\label{ssec:Related-results-C}
%
 Let $B$ and $G$ denote the groups $B(I;A,P)$ and $G(I;A,P)$.
 According to Corollary \ref{crl:Simplicity-of-B'},
 the derived group of $B$ is simple.
The quotient group $G/B$, on the other hand, 
is metabelian but not abelian if $I$ is a line or a half line,
and isomorphic to $P \times P$, 
hence a non-trivial abelian group, 
if $I$ is a compact interval with end points in $A$.
\footnote{See Corollaries \ref{CorollaryA2} and \ref{CorollaryA3} 
for more precise statements.} 
Moreover,
the group $G$ acts by the identity on $B_{\ab}$ whenever $I \neq \R$
(see Lemma \ref{LemmaC3New}).
The quotient group $G/B'$ is thus always a soluble group 
of derived length at most 3;
it is center-by-metabelian if $I$ is a half line 
and nilpotent of class at most 2 if $I$ is a compact interval.
\footnote{The group is actually abelian; 
see sections \ref{sssec:Notes-ChapterC-Brown-Stein}
and \ref{sssec:Notes-ChapterC-Gal-Gismatullin}.}

The quotient group $G/B$ can be described explicitly in terms of the parameters $A$ and $P$,
but the situation is more involved for $B_{\ab}$.
Fortunately, $B$ and hence $B_{\ab}$, do not depend on $I$ (by Proposition \ref{PropositionC1}),
and one can describe a quotient $K(A,P)$ of $B_{\ab}$ rather well
(see Proposition \ref{PropositionC4}).
\index{Subgroup K(A,P)@Subgroup $K(A,P)$!significance}%
Proposition \ref{prp:LemmaC5} then states 
that this quotient  is trivial if, and only if, $A = IP \cdot A$
and that it is finitely generated if, and only, if $A= IP \cdot A$, 
or if $P$ is finitely generated and $IP \cdot A$ has finite index in $A$.
This proposition furnishes thus a necessary condition for $B_{\ab}$ to be trivial 
(and hence $B$ to be simple).
There exists also a sufficient condition for the vanishing of $B_{\ab}$
which, however,  is far stronger that the stated necessary condition;
see Theorem \ref{TheoremC10}.
\index{Subgroup B(I;A,P)@Subgroup $B(I;A,P)$!vanishing of Bab@vanishing of $B_{\ab}$}
%
\subsubsection{The paper \cite{GaGi16} by \'{S}. R. Gal and J. Gismatullin}
\label{sssec:Notes-ChapterC-GaGi16}
\index{Gal, \'{S}. R.}%
\index{Gismatullin, J.}%
\index{Simplicity result for [B,B]@Simplicity result for $[B,B]$}%
Corollary \ref{crl:Simplicity-of-B'} asserts 
that the derived group $B'$ of $B = B(I;A,P)$ is simple
for every choice of the parameters $I$, $A$ and $P$.
Its proof is a variation on a proof which goes back to Higman's paper \cite{Hig54a};
it consists of four steps, 
the first three of which constitute the proof of Proposition 
\ref{prp:Higmans-Theorem1}.
This result just mentioned deals with a group $B \neq \{1\}$ 
satisfying the following commutativity property:
\begin{equation}
\label{eq:Higmans-condition-bis}
\left\{
\begin{minipage}{9.5cm}
For every ordered pair $(x,y) \in B^2$ and every $z \in B \smallsetminus \{1\}$ \\
there exists $u \in B$ such that the equation $[\act{u}x, \act{zu}y] = 1$ holds.
\end{minipage}
\right.
\end{equation}
In a first step, one deduces from the above property 
that, given $(x,y,z) \in B^3$ with $z \neq 1$,
the commutator $[x,y]$ is a product of two conjugates of $z$ 
and of two conjugates of $z^{-1}$.
It follows, first of all, 
that every non-trivial normal subgroup of $B$ contains $B'$.
Assume now that $B$ is not metabelian.
Then $B''$ contains the minimal normal subgroup $B'$, 
and so $B'' = B'$ is a minimal normal subgroup of $B$.
In a third step one then shows 
that the element $u$ mentioned in condition \ref{eq:Higmans-condition-bis}
can be chosen inside $B'$ if the elements $x$, $y$ and $z \neq 1$ lie in $B'$.
It follows that $B'$ is a non-abelian simple group.

The proof actually shows more:
if $B$ is not metabelian and $z \in B'\smallsetminus \{1\}$,
one can find, given $(x,y) \in B' \times B'$, 
elements $b_1$, \ldots, $b_4$ in $B'$
so that 
\[
[x,y] = \act{b_1} z \cdot \act{b_2} z^{-1} \cdot \act{b_3} z \cdot\act{b_4} z^{-1}.
\]
For every conjugacy class $\CC \neq \{1\}$, 
and each couple $(x, y) \in (B')^2$,
there exists therefore elements $z_1$, \ldots, $z_4$ in $\CC \cup\CC^{-1}$ 
so that $[x,y] = z_1 \cdots z_4$.
This finding and the fact that $B' = B''$ prompt the question 
whether the group $B'$ is uniformly simple in the sense of the following 
\begin{definitionNN}[\protect{\cite{GaGi16}}]
\label{definition:Uniform-simplicity}
\index{Uniform simplicity!definition|textbf}%
Let $H$ be a group and $n$ a positive integer.
Then $H$ is called \emph{$n$-uniformly simple} 
if, for each non-trivial conjugacy class $\CC \subset H$, 
every $h \in H$  is a product of at most $n$ elements in $\CC \cup \CC^{-1}$.
The group is called \emph{uniformly simple} 
if it is $n$-uniformly simple for some positive integer $n$.
\end{definitionNN}

The question raised in the above has been answered by \'{S}. R. Gal and J. Gismatullin:
\index{Gal, \'{S}. R.}%
\index{Gismatullin, J.}%
according to \cite[Remark 3.5]{GaGi16} 
the derived group $B'$ of every bounded group $B(I;A,P)$ is 6-uniformly simple (see \cite[Remark 3.5]{GaGi16}). 
In the proof of the stated result,
the authors make heavy use of a technique developed 
by D. Burago et al. in \cite{BIP08}.
\index{Burago, D.}%
\index{Ivanov, S.}%
\index{Polterovich, L.}%

\subsubsection{The paper \cite{Che52} by C. G. Chehata}
\label{sssec:Notes-ChapterC-Chehata}
\index{Chehata, C. G.}%
\index{Simplicity result for [B,B]@Simplicity result for $[B,B]$}%
In \cite{Che52}, 
the author studies groups of  ``bounded'' PL-homeomorphisms 
of an ordered field $(F, <)$.
More precisely, given an ordered field and a closed interval $I \subseteq F$, 
Chehata considers the group that could be called $B(I; F_{\add}, F^\times_{>0})$,
on the understanding that the definition of intervals in the field of real numbers 
is generalized in the obvious way to intervals of ordered fields.
His findings include results 
that are familiar from Chapters \ref{chap:A} and \ref{chap:C}:
Lemma 1 in Chapter I is an analogue of Theorem \ref{TheoremA}
\index{Theorem \ref{TheoremA}!analogues}%
and Lemma 9 in that chapter is the special case of Proposition \ref{PropositionC1} 
with $(A,P) =(\R_{\add}, \R^\times_{>0})$. 
The main result of Chapter II  in \cite{Che52} claims 
that the group $B(I; F_{\add}, F^\times_{>0})$ is simple 
for every ordered field  $(F, <)$.
\footnote{In the special case 
where $F$ is a subfield of $\R$ with the induced order relation,
this result follows also from Corollary \ref{crl:Simplicity-of-B'} and Theorem \ref{TheoremC10}.}
This result has been improved by \'{S}. R. Gal and J. Gismatullin;
they show 
that the group $B(I; F_{\add}, F^\times_{>0})$ is 6-uniformly simple
(see \cite[Remark 3.5]{GaGi16}).
\index{Gal, \'{S}. R.}%
\index{Gismatullin, J.}%
%
\subsubsection{The papers \cite{Bro87a} by Ken Brown 
and \cite{Ste92} by Melanie Stein}
\label{sssec:Notes-ChapterC-Brown-Stein}
\index{Brown, K. S.}%
\index{Stein, M.}%
%
Two articles published after 1985 deal with the quotient group $G/B'$.
In \cite{Bro87a},
Ken Brown considers a subgroup $F_n$ of the group $G(\R; \Z[P], P)$ 
where $P$ is a cyclic group generated by an integer $n \geq 2$.
In the set-up of this monograph,
the subgroup $F_n$ can be described like this.
According to Lemma \ref{LemmaE17} 
the group $G = G([0,1];\Z[P], P)$ with $P = \gp(n)$
can be realized as the subgroup $\bar{G}$ of $G(\R;\Z[P],P)$ 
made up of all PL-homeomorphisms $\bar{g}$ 
which are translations near $-\infty$ and near $+\infty$ with amplitudes 
that are (uncoupled, integer) multiples of $n-1$.
Let $h \colon \R \iso \R$ be the translation with amplitude 1 and set $F_n = \bar{G}\cdot \gp(h)$.
Then $\bar{G}$ has index $n-1$ in $F_n$.
The reasoning in section \ref{sssec:images-yi} shows next 
that $\bar{G}$ is generated by the translation $z_0$ with amplitude $n-1$ 
and the PL-homeomorphisms
\begin{equation*}
\label{eq:18.8bis}
z_i(t) = 
\begin{cases}
t &\text{if $t < i-1$}\\
p(t-i)+i &\text{if $i-1\leq t \leq i$}\\
t+(p-1) &\text{if $t > i$}
\end{cases},
\end{equation*}
the index $i$ varying over $\N \smallsetminus \{0\}$.  
The relations $\act{h} z_i = z_{i+1}$ hold for every index $i >0$;
they show that $F_n$ is an ascending HNN-extension with base group $\gp(\{z_i \mid i > 0 \})$
(isomorphic to $G$ 
\footnote{This follows, \eg{}from Lemma \ref{lemma:New18.3}.}) 
and stable letter $h$. 
Ken Brown states in the Remark on page 63 of \cite{Bro87a}
that $F_n/[B,B]$ is metabelian, but not abelian, whenever $n > 2$. 

In \cite{Ste92}, M. Stein states Lemma 4.1, a result due to Ken Brown.
The proof of this lemma is quite general; 
it shows, 
in particular,  
that $G' = B'$ for every slope group $P$, 
module $A$ and   group $G = G([0,b];A,P)$ with $b \in A$
\index{Subgroup B([a,c];A,P)@Subgroup $B([a,c];A,P)$!derived group}%
\index{Group G([a,c];A,P)@Group $G([a,c];A,P)$!derived group}%
(disregard the statement in the second line of the proof
which asserts 
that the exact sequence
\[
0 \to B_i \xrightarrow{\Incl} F_i \xrightarrow{\rho} P \times P \to 0
\]
splits).
For a compact interval $I$,
the group $G/B'$ is thus always abelian, not merely nilpotent of class at most 2.
A special case of this conclusion follows also from the exact sequence \eqref{eq:Exact-abelianized-sequence-P-cyclic},
displayed in Example \ref{example:P-cyclic-A-locally-cyclic}.
%
\subsubsection{An observation due to \'{S}. R. Gal and J. Gismatullin}
\label{sssec:Notes-ChapterC-Gal-Gismatullin}
\index{Gal, \'{S}. R.}%
\index{Gismatullin, J.}%
\'{S}. R. Gal and J. Gismatullin rediscovered recently the result, 
mentioned in the previous paragraph,
according to which the derived group of $G= G([0,b];A,P)$ 
coincides with the derived group of $B([0,b]; A, P)$.
\index{Subgroup B([a,c];A,P)@Subgroup $B([a,c];A,P)$!derived group}%
Their proof is short and goes like this.

Let $g_1$, $g_2$ be elements of $G$. 
The support of the commutator $[g_1, g_2]$ is contained in an interval of the form $[b_1, b_2]$ with $b_1$, $b_2$ in $A$ and and $0 < b_1 < b_2 < b$.
Choose $b_1' \in IP \cdot A$ with $0 < b'_1 < b_1$ 
and $b_2' \in b + IP \cdot A$ so that $b_2 < b_2' < b$.
Theorem \ref{TheoremA} allows one then to construct an element 
\index{Theorem \ref{TheoremE04}!consequences}%
$h \in G(\R;A,P)$
\index{Construction of PL-homeomorphisms!applications}%
that is the translation with amplitude $b'_1$ on $]-\infty, 0]$, 
the identity on $[b_1, b_2]$,
the translation with amplitude $b'_2- b$ on $[b, \infty]$ 
and which maps 
the interval $[0, b_1]$ onto the smaller interval $[b'_1, b_1]$ 
and the interval $[b_2, b]$ onto $[b_2, b'_2]$.
Then
\[
[g_1, g_2] = \act{h}[g_1, g_2] = [\act{h}g_1, \act{h} g_2] 
\in G([b'_1, b'_2]; A, P) \subset B([0,b];A,P).
\]
%
\subsection{Finiteness properties --- $I$ a line or half line}
\label{ssec:Notes-Finiteness-propertiers-line-halfline}
\index{Thompson's group F@Thompson's group $F$!properties}%
%
Thompson's group $F$ has striking properties:
it is finitely generated,
has a presentation with infinitely many generators of a particular kind,
and a representation in $G([0,1]; \Z[1/2], \gp(2))$ 
that is is dense in the space of all increasing homeomorphisms of the unit interval
(see section \ref{ssec:1.1} and Corollary \ref{CorollaryA5}).
One of the aims of the Bieri-Strebel memoir is to find out 
which of these properties persist and, if so, in what form they do
if one passes from the group $G([0,1]; \Z[1/2], \gp(2))$ 
to the more general groups $G(I;A,P)$.

It turns out that there is a marked contrast, 
for results as well as techniques leading to them,
between groups where $I$ is a line or half line, 
and groups where $I =[0,b]$ is a compact interval with end point $b \in A$.
The discussion of the finiteness properties 
will therefore be spread over two sections.
In this one, results for intervals of infinite length will be discussed.
%
\subsubsection{Finite generation}
\label{sssec:Notes-finite-generation-line-half-line}
%
We start out with \emph{necessary conditions for $G = G(I;A,P)$ to be finitely generated}.
Clearly, a group can only be finitely generated if it is countable;
and a countable group is finitely generated if, and only if,
it cannot be written as union of an infinite, properly ascending chain of subgroups.
These observations are exploited in the proof of Proposition \ref{PropositionB1},
a result which holds for all types of intervals.

It shows
that $G$ can only be finitely generated if $P$ is finitely generated, 
$A$ is a finitely generated $\Z[P]$-module,
$A/(IP \cdot A)$ is finite or $I = \R$,
and if the (finite) endpoints of $I$ lie in $A$.

I turn now to groups which have been shown to be \emph{finitely generated}.
All the positive results I am aware of start out 
with an infinite sets of generators $\GG$ 
and prove then that a finite subset $\GG_f$ of $\GG$ is already generating.
The methods leading to such a generating set $\GG$ differ widely 
for intervals of infinite length and for those of finite length.
If $I$ is a \emph{half line}, the group $G = G(I;A,P)$ is generated by the elements with one singularity
and there is a very simple presentation in terms of these generators 
(see section \ref{ssec:8.1}; 
\cf{}\cite[Corollary (2.6)]{BrSq85}).
The situation for a \emph{line} is similar: 
then $G$ is generated by the elements with at most one singularity
and there exists an explicit presentation with these generators, 
due to Brin and Squier
(see \cite[Corollary (2.8)]{BrSq85} or Section \ref{sec:7}).
\index{Brin, M. G.}%
\index{Squier, C. C.}%

Starting with the specified infinite generating sets,
the relations in the Brin-Squier presentations allow one to infer 
that the stated necessary conditions are also sufficient.
If $I$ is a line, the proof of this claim is simple,
but if $I$ is a half line it is more involved;
see the proofs of Theorems \ref{TheoremB2} and \ref{TheoremB4}.
%
\subsubsection{Finite presentation}
\label{sssec:Notes-finite-presentation-line-half-line}
The results for finite presentations obtained so far 
are less satisfactory than those for finite generation:
all that is known are  necessary, and far more demanding sufficient,  conditions.

This state of affairs is even true if $I$ is the real line.
If $G$ admits a finite presentation then so does its metabelian quotient $\Aff(A,P) \cong  A \rtimes P$
and $A/(IP \cdot A)$ is finite (see Proposition \ref{PropositionD2}).
Note that $\Aff(A,P)$ is \emph{finitely generated} if, and only if, 
$P$ is a finitely generated group and $A$ is a finitely generated $\Z[P]$-module.
\footnote{In view of Theorem \ref{TheoremB2},
the group $G(\R;A,P) $ is thus finitely generated if, and only if, 
its quotient  group $\Aff(A,P)$ is so.}
The stated necessary condition for a finite presentation of $G$ 
is thus stronger than the condition for finite generation.
The only known sufficient condition for finite presentability is rather technical;
see Proposition \ref{PropositionD4} for details. 
This sufficient conditions holds, in particular, 
if $A = \Z[P]$ and $P$ is freely generated by a finite set of (positive) integers
(see Example 1 in section \ref{sssec:13.3cNew}).

Suppose now $I$ is a half line.
There exists then a necessary condition for finite presentability 
that is stronger than that for finite generation:
 $P$ must be finitely generated,
$A$ must be  finitely generated $\Z[P]$-module,
$A/(IP \cdot A)$ must be finite
and, in addition, \emph{$A \rtimes P$ must admit a finite presentation}
(see  Proposition \ref{PropositionD5}).
No satisfactory sufficient condition for finite presentability is known.
There exists, however, a collection of finitely presented examples
where $A = \Z[P]$ and $P$ is freely generated by a set of positive integers 
(Theorem \ref{TheoremD6}).
In case $P$ is cyclic, 
the finite presentability of these groups has been proved 
earlier on by M. Brin and C. Squier;
see Theorem (2.9) in \cite{BrSq85}.
\index{Brin, M. G.}%
\index{Squier, C. C.}%
%
\subsection{Finiteness properties --- $I = [0,b]$ with $b \in A$}
\label{ssec:Notes-Finiteness-conditions-compact-interval}
%
The situation for groups  with $I =[0,b]$ differs widely from 
that for groups with $I$ the line or a half line.
Here are some key facts about these groups:
\begin{itemize}
\item the known necessary conditions for finite generation and those for finite presentation coincide;
\item all groups known to be finitely generated have $A = \Z[P]$ 
and $P$ is either generated by positive integers, or cyclic generated by a quadratic algebraic integer;
\item all groups known to be finitely generated admit a finite presentation 
and are of type $\FP_\infty$.
\end{itemize}

So far the approach propounded by the Bieri-Strebel memoir --- study groups $G(I;A,P)$ 
with parameters $(I,A,P)$ varying in some fairly wide classes --- 
has not been fruitful for the study of finiteness properties of groups 
with $I$ a compact interval (having endpoints in $A$).
Indeed, 
only one general fact about finiteness properties of this kind of groups 
seems to be known, 
namely Lemma \ref{LemmaB6};
it asserts that whether or not a group is \emph{finitely generated}
does not depend on the length $b$ of the interval $I=[0,b]$.
%
\subsubsection{The papers \cite{Bro87a} by Ken S. Brown and \cite{Ste92} by Melanie Stein}
\label{sssec:Notes-finiteness-properties-Brown-Stein}
\index{Brown, K. S.|(}%
\index{Stein, M.|(}%
\index{Finiteness properties of!G([a,c];A,P)@$G([a,c];A,P)$|(}%
%
I begin by describing one kind of groups studied in these papers.
Let $\PP = \{p_1, \ldots, p_k\}$ be a finite set of integers greater than 1, 
put $P = \gp(\PP)$, $A = \Z[P]$ and choose a number $\ell \in A_{>0}$.
Set
\begin{equation}
\label{eq:Steins-definition-of-generalized-F}
F(\ell, A, P) = G([0,\ell]; A, P).
\end{equation}

 Melanie Stein investigates these generalized $F$-groups in \cite{Ste92}. 
Among them are the groups $F_{n, r}$ studied by Ken Brown in \cite[Section 4]{Bro87a}
\footnote{Here $k = 1$, $r = \ell \in \N$ and $n = p_1$.}
and the groups 
\[
G[\PP] = G([0,1]; \Z[1/(p_1 \cdots p_k)], \gp(\PP)] \quad \text{with } \PP = \{p_1, \ldots, p_k\}
\]
examined in Sections 9 and 15 of the Bieri-Strebel memoir.
Stein determines first when two $F$-groups are isomorphic to each other
for obvious reasons. 
The homothety $\theta_p  \colon t \mapsto p \cdot t$ with $p \in P$
allow one to replace $\ell$  by $p \cdot \ell$; 
so $\ell$ can be assumed to be a natural number.
Next, there exists a PL-homeomorphism $f \colon [0,\ell] \iso [0, \ell']$ in $G(\R; A, P)$ 
if, and only, if $\ell' - \ell$ is a multiple of $d = \gcd\{p_j-1 \mid 1 \leq j \leq k\}$
(see \cite[pp.\,479--480]{Ste92}, 
or Theorem \ref{TheoremA} and Illustration \ref{illustration:4.3}).
\index{Theorem \ref{TheoremA}!consequences}%

It suffices thus to consider $\ell \in \{1,2 , \ldots, d\}$.
\footnote{The results in Chapter \ref{chap:E} allow one to go further 
and to determine which pairs of lengths lead to isomorphic groups;
for $k = 1$ all groups are pairwise isomorphic 
(see \cite[Proposition 4.1]{Bro87a} or Theorem \ref{TheoremE07}; 
for $k > 1$ Corollary \ref{CorollaryE12} has the answer. 
Notice that for $d > 1$ (and $k > 1$),
the groups with $\ell = 1$ and with $\ell = d$ are never isomorphic.}
Stein gives a second preliminary result.
By assumption, the group $P$ is generated by finitely many integers $p_i > 1$ 
and it is free abelian; the given integers, however, need not form a basis.
Her Proposition 1.1, due to Ken Brown, now guarantees 
that one can always find a new set of integers which is multiplicatively independent.
This result implies, in particular, 
that the requirement that $\PP$ is a basis of $P$ 
can be dispensed with in Proposition \ref{PropositionD1}
and allows one to simplify the proof of Theorem \ref{TheoremD6}.

I continue with a word 
on the proofs of finiteness properties of the groups $F(\ell; A, P)$.
They involve the construction of a useful infinite generating set.
In the proof given in section \ref{ssec:14.3}, this fact is evident.
In the papers by Brown and Stein, 
it is hidden in the verification that a certain poset is directed 
(see the first paragraph of \cite[p.\,481]{Ste92}).
This poset gives rise to a contractible simplicial complex $X$ 
acted on by the group $F(\ell,A,P)$.
By exhibiting and analyzing suitable contractible subcomplexes of $X$,
M. Stein is able to establish that, 
for every slope group $P$ generated by finitely many positive integers and $A = \Z[P]$,
each of the groups $F(\ell,A,P)$ with $\ell \in \{1, \ldots, d\}$ admits a finite presentation 
and is of type $\FP_\infty$ (see \cite[Theorem 2.5]{Ste92}).
\index{Group G([a,c];A,P)@Group $G([a,c];A,P)$!finite presentation}%
\index{Group G([a,c];A,P)@Group $G([a,c];A,P)$!type FPinfty@type $\FP_\infty$}%

I conclude my sketch with a word on Lemma 4.1 in \cite{Ste92};
it holds for arbitrary values of $A$ and $P$. 
The lemma shows that the homology groups $H_j (F(\ell, A, P))$ 
depend only on those of $B([0,\ell]; A, P)$ and of $P \times P$;
more precisely, 
$H_*(F(\ell, A, P))$ is shown to be isomorphic to  the tensor product
$
B_*([0,\ell]; A, P) \otimes H_*(P\times P).
$

In conjunction with Proposition \ref{PropositionC1} 
this finding implies
that the homology groups $H_*(-; \Z)$ do not allow one to prove 
that intervals of different lengths $\ell \neq\ell'$ 
may lead to non-isomorphic groups $F(\ell, A, P)$ and $F(\ell', A, P)$. 
So Corollary \ref{CorollaryE12} maintains its merits.
\index{Group G([a,c];A,P)@Group $G([a,c];A,P)$!isomorphisms}%
\index{Brown, K. S.|)}%
\index{Stein, M.|)}%
\index{Finiteness properties of!G([a,c];A,P)@$G([a,c];A,P)$|)}%

%
\subsubsection{The papers \cite{Cle95} and \cite{Cle00} by Sean Cleary}
\label{sssec:Notes-Finiteness-properties-Cleary}
\index{Cleary, S.|(}%
\index{Finiteness properties of!G([a,c];A,P)@$G([a,c];A,P)$|(}%
%
I return to the concept of $\PP$-regular subdivision,
introduced in section \ref{ssec:9.2}.
Suppose $\PP$ is a singleton consisting of an integer $p \geq 2$;
set $P = \gp(p)$ and $A = \Z[P]$.
Then every affine interpolation of $\PP$-regular subdivisions of $I = [0,1]$ with the same number of points 
produces an element $f \in  G(I; A, P)$, 
and every element $f \in F$ arises in this way.
The proof of this claim is rather simple;
see, \eg{}\cite[Proposition 4.4]{Bro87a}.
\footnote{If $P$ is generated by several positive integers, 
the proof of the analogous statement is a bit more involved,
but still fairly short, see section \ref{ssec:9.4}.}

In \cite{Cle95}, 
Sean Cleary introduces $\PP$-regular subdivisions 
for $p = \omega =  \sqrt{2\,} + 1$.
They involve the minimal polynomial of $\omega^{-1} = \sqrt{2\,} - 1$,
namely $X^2 + 2 X - 1$. 
This minimal polynomial implies 
that 1, 
the length of the unit interval, 
is the sum of the lengths of three intervals 
with lengths $\omega^{-1}$, $\omega^{-1}$ and $\omega^{-2}$.
So the unit interval is the union 
of two longer subintervals of length $\omega^{-1}$ 
and a shorter interval of length $\omega^{-2}$.
A novel feature of this type of subdivision is that the subintervals have two lengths
and so there are actually three regular subdivisions of  $[0,1]$, 
each with three subintervals.
Each of the three subintervals can be subdivided in the same manner, 
leading to subintervals with lengths
that satisfy one of the equations 
\[
\omega^{-1} = 2 \omega^{-2} + \omega^{-3}\quad \text{or}\quad 
\omega^{-2} = 2 \omega^{-3} + \omega^{-4}. 
\]
So far I have described regular subdivisions of level 1 and 2;
by iterating one arrives at regular subdivisions of arbitrary level $\ell$.

Set $G = G([0,1]; \Z[\sqrt{2}\,], \gp(\omega))$ and consider the statement 
that \emph{every PL-homeo\-morphism in $G$
is the affine interpolation of two $\{\omega\}$-regular subdivisions with the same number of points}.
The results in Chapter \ref{chap:A}, in particular Corollary \ref{CorollaryA2}, 
imply 
that every point $ b  \in [0,1] \cap \Z[\sqrt{2}\,]$ 
is a break of a suitable element in $G$.
Since the additive group of the ring $\Z[\sqrt{2}\,]$ is equal to $\Z \cdot 1 + \Z \cdot \sqrt{2}$,
the number $b$ is a $\Z$-linear combination of the form $m_0  + n_0 \cdot \sqrt{2}$. 
If $b$ is neither 0 nor 1, exactly one of the integers $m_0$, $n_0$ will be positive.
On the other hand, if $b$ is a point in an $\{\omega\}$-regular subdivision of $[0,1]$ 
then it will also be a point of a regular subdivision with intervals of only two lengths $\omega^{-k}$ and $\omega^{-k-1}$,
as is not hard to show (see \cite[p.\,938]{Cle95}).
So $b$ will have a representation of the form 
$b = m_k \omega^{-k} + n_k \omega^{-k-1}$
with non-negative integers $m_k$ and $n_k$.
These considerations lead to a first problem; 
it is dealt with in \cite[Lemma 2]{Cle95}.
Note that the analogous problem is easily solved
if $A = \Z[1/p]$ and $p \geq 2$ is an integer;
then every element  $b \in [0,1] \cap \Z[1/p]$ 
is a fraction of the form $m / p^k$ with $0 \leq m \leq p^k$
and so it is a point of the uniform subdivision of $[0,1]$ into $p^k$ intervals 
and this uniform subdivision is $\{p\}$-regular.

A further difficulty lies ahead: 
if $b$ is a point in an $\{\omega\}$-regular subdivision of $[0,1]$
with intervals of length $\omega^{-k}$ and $\omega^{-k-1}$, 
then the longer and shorter subintervals will appear in a certain order 
and so it may not be possible to realize a linear combination
$b = m_k \omega^{-k} + n_k \omega^{-k-1}$  
as subdivision induced on the initial segment $[0,b]$ by a regular subdivision of the unit interval
with subintervals of lengths $\omega^{-k}$ and $\omega^{-k-1}$.
This new problem is tackled on pages 941--951 in \cite{Cle95}.

The number $\omega = \sqrt{2} + 1$ is an example of  quadratic algebraic integer;
Sean Cleary states on pages 954--955 
that his techniques work also for other quadratic  algebraic integers.
Moreover, in \cite{Cle00} he gives details for this assertion 
in the case of the algebraic integer $(\sqrt{5} + 1)/2$. 
\index{Cleary, S.|)}%
\index{Finiteness properties of!G([a,c];A,P)@$G([a,c];A,P)$|)}
%
\subsection{Chapter E}
\label{ssec:Related-results-E}
%
The pivotal result of Chapter \ref{chap:E} is Theorem \ref{TheoremE04}.
It asserts that \emph{every isomorphism $\alpha \colon G \iso \bar{G}$ is induced 
by conjugation by a unique homeomorphism $\varphi \colon  \Int(I) \iso \Int(\bar{I})$  
provided $G$ is a subgroup of $G(I; A, P)$ that contains the derived group of $B(I;A,P)$ 
and $\bar{G}$ is a subgroup of $G(\bar{I}; \bar{A}, \bar{P})$ with the analogous property}.
There are no restrictions on the parameters $(I, A, P)$ and  $(\bar{I}, \bar{A}, \bar{P})$.

Let $G$ and $\bar{G}$ be as in the statement of Theorem \ref{TheoremE04}.
The fact that every isomorphism $\alpha \colon G \iso \bar{G}$ is induced by conjugation 
by a homeomorphism has some straightforward consequences.
\index{Theorem \ref{TheoremE04}!consequences}
It implies, first of all, that $B = B(I;A,P)$  is a characteristic subgroup of $G$ 
whenever $G$ contains $B$ (see Corollary \ref{CorollaryE5}).
The fact that every homeomorphism of an interval is either increasing or decreasing
has another useful consequence:
if $\varphi$ is increasing then $\alpha$ induces isomorphisms
\[
\alpha_\ell \colon \im (\lambda \colon G \to \Aff(A, P)) 
\iso \im (\bar{\lambda} \colon  \bar{G} \to \Aff(\bar{A}, \bar{P}))
\]
and a similar statement holds if $\varphi$ is decreasing (again Corollary \ref{CorollaryE5}).
The stated result implies, in particular, 
that two groups $G(I; A,P)$ and $G(\bar{I}; \bar{A}, \bar{P})$
can only be isomorphic if $I$ and $\bar{I}$ are of the same type:
either both are lines, or both are half lines, 
or both are compact intervals with endpoints in $A$ and $\bar{A}$,
respectively (see Proposition \ref{PropositionE6}). 

The mentioned consequences are starters. 
The focus of Chapter \ref{chap:E} revolves around the question 
whether $\varphi$ is actually a PL-homeomorphism. 
Here the algebraic nature of $P$ plays a decisive rôle.
%
\subsubsection{Groups with non-cyclic slope groups}
\label{sssec:Notes-Chapter-E-non-cyclic-slope-group}
%
If $P$ \emph{is not cyclic} then $\bar{P} = P$, 
and there exists a real number $s_\alpha$ such that $\bar{A} = s_\alpha \cdot A$ 
and $\varphi$ is a PL-homeomorphism with slopes in the coset $s_\alpha \cdot P$.
The PL-homeomorphism $\varphi$ need not be finitary, but every compact subinterval of $\Int(I)$ contains only finitely many breaks of $\varphi$ 
(see Theorem \ref{TheoremE10}).

Sharper results hold if $G$ is all of $G(I;A,P)$ and if, in addition,  $I$ and $\bar{I}$ are both lines, or both half lines, or both compact intervals with end points in $A$ and $\bar{A}$, respectively.
Then an isomorphism $\alpha \colon G \iso \bar{G}$ can only exist 
if $\bar{G} = G(\bar{I}; \bar{A}, \bar{P})$ 
and a homeomorphism inducing an isomorphism $\alpha \colon G \iso \bar{G}$  is necessarily piecewise linear with finitely many breaks (see Supplement \ref{SupplementE11New}).
\index{Theorem \ref{TheoremE10}!consequences}%
This result has three noteworthy consequences. 

(i) It allows one to determine 
when two groups $G(I; A, P)$ and $G(\bar{I}; A, P)$ 
with compact intervals $I = [0,b]$ and $\bar{I} = [0, \bar{b}]$ of different lengths are isomorphic
(consult Corollary \ref{CorollaryE12} 
and part (iii) of section 
\ref{sssec:Automorphism-groups-some-examples-non-cyclic-P}
for more details).

(ii) It implies that the outer automorphism group of $G(I;A,P)$ is abelian 
and actually a subgroup of  $\Aut(A)/P$; 
here $\Aut(A)$ denotes the group of all non-zero reals $s$ with $s \cdot A = A$.
\index{Outer automorphism group of!G(R;A,P)@$G(\R;A,P)$}%
\index{Outer automorphism group of!G([0,infty[;A,P)@$G([0, \infty[\;;A,P)$}%
\index{Outer automorphism group of!G([a,c];A,P)@$G([a,c];A,P)$}%
\index{Group Aut(A)@Group $\Aut(A)$!significance}%
This group contains always $-1$, 
but is is typically larger then $P \times \{1, -1\}$;
see part (i) in section 
\ref{sssec:Automorphism-groups-some-examples-non-cyclic-P} 
for some examples.

(iii) It permits one to prove that the automorphism group of $G(I;A,P)$ 
has a subgroup of index at most 2 which is (isomorphic to)
an explicitly describable subgroup of $G(I;A, \Aut_o(A)$
(Corollary \ref{crl:Explicit-description-of-Aut(G)} has more details).
%
\subsubsection{Groups with cyclic slope groups}
\label{sssec:Notes-Chapter-E-cyclic-slope-group}
%
If the slope group $P$ is cyclic,
the properties of isomorphisms and automorphisms are typically more intricate 
than those described in the previous section, and often counter-intuitive. 
There is, however, one exception:
if $I$ and $\bar{I}$ are both lines, or both half lines 
and if, in addition, $G = G(I;A,P)$,
the situation is much as in the case where $P$ is not cyclic
(details are given by Theorem \ref{TheoremE14}).

A first indication of the surprises that lie ahead when $I$ is compact
is the fact known, 
in the case of Thompson's group $F$ since the 1970s, 
that the group $G =G([0,b]; A, P)$ (with $b \in A$) 
can be embedded into $G([0, \infty[\,;A,P)$ 
and also  into $G(\R;A,P)$ in such a way 
that the image contains the subgroup $B([0, \infty[\,;A,P)$, 
respectively the subgroup $B(\R;A,P)$ 
(details are spelled out in sections
\ref{sssec:Embedding-mu1}  and  \ref{sssec:Embedding-mu2-mu1}).
\index{Group G([a,c];A,P)@Group $G([a,c];A,P)$!endomorphisms}%

More astonishing is the fact that $G$ contains, for every integer $n >1$, 
a subgroup of index $n$ which is isomorphic to $G$
and also subgroups of finite index which are not isomorphic to $G$.
\footnote{The answers to both parts of Question 4.5 in \cite{CST04} 
are thus in the affirmative.}
\index{Brown, K. S.}%
\index{Guba, V.}%
Details of these assertions are provided by section \ref{ssec:18.4},
in particular by Theorem \ref{thm:Isomorphism-types-class-II}, 
and also by the next section.
\index{Group G([a,c];A,P)@Group $G([a,c];A,P)$!subgroups of finite index}%

%
\subsubsection{The papers \cite{BlWa07} by Bleak-Wassink and \cite{BCR08} by Burillo et al.}
\label{sssec:Notes-ChapterE-Bleak-Wassink}
\index{Bleak, C.}%
\index{Wassink, B.}%
\index{Burillo, J.}%
\index{Cleary, S.}%
\index{R{\"o}ver, C. E.}%
%
The papers \cite{BlWa07} and \cite{BCR08} study finite index subgroups of Thompson's group  $F = G([0,1];\Z[1/2], \gp(2))$. 
The subgroup $B = B([0,1];\Z[1/2], \gp(2))$ coincides with $F'$ and it is simple; 
so every subgroup with finite index in $F$ lies above $B$.
We are thus in the situation considered in section \ref{sssec:18.8b}.
Let $\pi \colon F \to \Z^2$ be the homomorphism 
which maps $f \in F$ to  
\[
\left(\log_2(f'(0_+)), \log_2(f'(1_-))\right) \in \Z^2;
\]
this homomorphism is surjective.
Given integers $m \geq 1$ and $n \geq 1$, 
consider the preimage $K_{m,n}$ of $\Z m \times \Z n$ under $\pi$;
it has index $m \cdot n$ in $F$.
Both papers mentioned above establish 
that each such ``rectangular'' subgroup is isomorphic to $F$
(see \cite[Theorem1]{BlWa07} and \cite[Corollary 3.3]{BCR08}); 
the given proofs use different techniques 
and none of them is similar to the approach taken in section \ref{sssec:18.5}.
The cited results show, in addition, that the rectangular subgroups 
are the only subgroups of finite index which are isomorphic to $F$.
\index{Group G([a,c];A,P)@Group $G([a,c];A,P)$!subgroups of finite index}%
\index{Thompson's group F@Thompson's group $F$!subgroups of finite index}%
Moreover, 
C. Bleak and B. Wassink give an algorithm 
that allows one to decide whether two subgroups of finite index in $F$ are isomorphic to each other
(see \cite[Theorem 1.8]{BlWa07}). 
Some ingredients of this algorithm have been used in the proof of Theorem
\ref{thm:Isomorphism-types-class-II}.
%
\subsubsection{Automorphism group of $G([0,b]; A, P)$}
\label{sssec:Notes-ChapterE-Automorphism-group-of-G-P-cyclic}
%
Let $p>1$ be a generator of the cyclic group $P$.
A first basic question concerning $G = G([0,b]; A, P)$ is 
whether or not every auto-homeomorphism $\varphi \colon ]0,b[\, \iso \,]0,b[$  
inducing an automorphism $\alpha$ of $G$ is piecewise linear, 
possibly with infinitely many breaks.
The Bieri-Strebel memoir does not answer this question;
as a substitute, it introduces the subgroup
\begin{equation}
\label{eq:Aut_PL-Notes}
\Aut_{\PL}G = \{\alpha \in \Aut(G) \mid \alpha \text{ is induced by a PL-homeomorphism } \varphi \}
\end{equation}
and describes the subgroup $\Aut_{\PL}G/\Inn G $ of the outer automorphism group of $G$ 
(details are spelled out in  Corollary \ref{crl:PropositionE20New-part-II}).
This description shows, in particular, 
that $\Aut_{\PL} G /\Inn G $ contains a subgroup of index $|A :( IP \cdot A)|$ 
in the square of the generalized Thompson group $T(\R/\Z(p-1); A, P)$.
The outer automorphism group of $G([0,b];A, P)$ is thus far from being abelian,
in contrast to what happens if $P$ is not cyclic 
(see statement (ii) in section \ref{sssec:Notes-Chapter-E-non-cyclic-slope-group}).
\index{Outer automorphism group of!G([a,c];A,P)@$G([a,c];A,P)$}%
%
\subsubsection{The paper \cite{Bri96} by Matthew Brin}
\label{sssec:Notes-ChapterE-Brin96}
\index{Brin, M. G.}%
%
The best known group with $P$ infinite cyclic and $I$ compact
is Thompson's group $F = G([0,1];\Z[1/2], \gp(2))$.
Its automorphism group did not yield to the techniques of \cite{BiSt85},
and has only been determined a decade later by Matthew Brin in \cite{Bri96}. 
He states his main result in Theorem 1 on page 9;
this result deals with subgroups of $F$, 
a supergroup $\tilde{F}$ of $F$ with index 2,
and also with Thompson's group $T$ and one supergroup $\tilde{T}$ of $T$.
In the language of this monograph,
the findings for $F$ and its subgroup $B$ can be summarized as follows:
\begin{thmN}[see \protect{\cite[Theorem 1]{Bri96}}]
\label{thm:Main-result-Bri96-part-F}
\index{Group G([a,c];A,P)@Group $G([a,c];A,P)$!exotic automorphisms}%
\index{Group G([a,c];A,P)@Group $G([a,c];A,P)$!automorphism group}%
\index{Thompson's group F@Thompson's group $F$!automorphism group}%
Let  $F$ denote $G([0,1];\Z[1/2],\gp(2))$ and set $B = B([0,1];\Z[1/2],\gp(2))$. 
Then the following statements hold:
\begin{enumerate}[(i)]
\item every automorphism  of $F$ is induced by a PL-homeomorphism;
the group $\Aut F$ is thus isomorphic to the group $ \Autfr_{\PL} F$ 
discussed in section \ref{sssec:19.3a};
\item every automorphism of $B$ is piecewise linear and so $\Aut B$ is isomorphic 
to the group $G_\infty([0,1]; \Z[1/2], \gp(2))$
made up of all (finitary or infinitary) PL-homeomorphism of $[0,1]$ 
with slopes in $\gp(2)$, breaks in $\Z[1/2]$
and a set of breaks that does not accumulate in $]0,1[$.
\end{enumerate}
\end{thmN}
Notice that the subgroup $B$ is characteristic in $F$ (by Corollary \ref{CorollaryE5});
so the first assertion of (i) is actually a consequence of part (ii).

As mentioned before,
Matt Brin determines also the automorphism group of Thompson's group 
$T = T(\R/Z; \Z[1/2], \gp(2))$;
his result is very easy to state: 
\begin{thmN}[see \protect{\cite[Theorem 1]{Bri96}}]
\label{thm:Main-result-Bri96-part-T}
\index{Thompson's group T@Thompson's group $T$!automorphism group}%
Every automorphism of $T$ is an inner automorphism 
or the composition of an inner automorphism and the automorphism induced by reflection at the origin.
\end{thmN}
The outer automorphism group of $T$ is thus abelian of order 2, 
in sharp contrast to what happens with $F$ (\cf{}Corollary \ref{crl:PropositionE20New-part-II}).
\index{Brin, M. G.}%
\smallskip

The proofs of the stated results start out 
with variations on Theorem \ref{TheoremE04},
published by S. H. McCleary and M. Rubin in \cite{McRu96}. 
\index{McCleary, S. H.}%
\index{Rubin, M.}%
As these variations are also a stepping stone in \cite{BrGu98}, 
a follow-up of \cite{Bri96},
I continue with some words on the memoir \cite{McRu05},
a later version of \cite{McRu96}.

\subsubsection{The memoir \cite{McRu05} by Stephen H. McCleary and Matatyahu Rubin}
\label{sssec:Notes-McCleary-Rubin}
\index{McCleary, S. H.}%
\index{Rubin, M.}%
%
Recall the message of Theorem \ref{TheoremE04}:
each isomorphism $\alpha \colon G \iso \bar{G}$
of a subgroup $G \subseteq G(I;A,P)$ 
onto a subgroup $\bar{G} \subseteq G(\bar{I}; \bar{A}, \bar{P})$
is induced by conjugation by a unique homeomorphism $\varphi_\alpha \colon \Int(I) \iso \Int(\bar{I})$,
provided $G$ contains the derived group of $B(A;I,P)$ 
and $\bar{G}$ satisfies the analogous property.
Theorem \ref{TheoremE04} is a consequence of Theorem \ref{TheoremE3}.
That result deals with an isomorphism $\alpha \colon G \iso \bar{G}$
of a subgroup $G \subseteq \Homeo_o(J)$ onto a subgroup $\bar{G} \subseteq \Homeo_o(\bar{J})$;
here $J$ and $\bar{J}$ are open intervals of $\R$ 
and $G$, $\bar{G}$ are assumed to satisfy certain axioms, 
labelled \emph{Ax1}, \emph{Ax2}, \emph{Ax3} and \emph{Ax4}.

The ancestor of Theorem \ref{TheoremE3} 
is the Main Theorem 4 in S. McCleary's article \cite{McC78b}. 
\index{McCleary, S. H.}%
This result considers linear permutation groups $G_1$ and $G_2$ 
of (dense) chains $L_1$ and $L_2$,
respectively,
and shows that every isomorphism $\alpha \colon G_1 \iso G_2$ is induced by conjugation 
by a unique \emph{monotonic} (that is, order-preserving or order-reversing) automorphism 
$\varphi \colon L_1^{\cpl} \iso L_2^{\cpl}$ 
of the Dedekind completions of $L_1$ and of $L_2$,
provided the following requirements are satisfied:
\begin{enumerate}[1)]
\item $G_1$ and $G_2$ act 3-fold order-transitively on $L_1$, respectively on $L_2$; 
\item $G_1$ and $G_2$ contain both an element $g_i \neq \id$ of bounded support  with $t \leq g_i(t)$ for all $t \in L_i$.
\end{enumerate}

Suppose now that $G_1$ is the group $G(I_1;A_1,P_1)$ 
for certain parameters $I_1$, $A_1$ and $P_1$,
and that $G_2$ is the group $G(I_2;A_2,P_2)$ 
with parameters $I_2$, $A_2$ and $P_2$. 
Then each of the orbits $L_1$ of $G_1$ in $A_1 \cap \Int(I_1)$ is dense in $I_1$,
so $G$ acts faithfully on $L_1$ and the Dedekind completion  of $L_1$ is the interval $\Int(I_1)$.
Moreover, $G_1$ acts 3-fold order-transitively on $L_1$, unless $I_1 = \R$  
(see Corollary \ref{CorollaryA4}).
\index{Theorem \ref{TheoremE04}!consequences}%
\index{Group G(I;A,P)@Group $G(I;A,P)$!multiple transitivity}%
Similar statements hold for $G_2$ and $I_2$.
Requirement 1) holds therefore for  $(L_1, G_1)$ and $(L_2, G_2)$, 
unless one of the intervals $I_i$ is the line $\R$. 
Since $G_1$ and $G_2$ contain elements 
$g_1 \in G_1$ and $g_2 \in G_2$ that fulfill requirement 2),
 McCleary's Theorem applies, 
 except when $I_1 = \R$ or $I_2 = \R$.
 
 This \emph{Proviso} is one of the reasons 
 that incited Bieri and Strebel to look for a replacement of McCleary's Main Theorem 4
 that would allow them to establish the indicated result in full generality. 
 Their solution consisted in replacing 3-fold order-transitivity 
 by a version of approximate 6-fold transitivity,
 embodied in axiom \emph{Ax4}, 
 and in replacing the condition in statement 2) 
 by three conditions, dubbed axioms \emph{Ax1},  \emph{Ax2} and  \emph{Ax3}.
 Based on these hypotheses they proved Theorem \ref{TheoremE3},
 their substitute for the Main Theorem 4.
 The proof of the pivotal Theorem \ref{TheoremE04} amounted 
 then to a verification of axioms \emph{Ax1} through \emph{Ax4}.
 If $G_1$ is a subgroup of $G(I_1;A_2,P_1)$ containing $B_1 =B(I_1;A_1,P_1)$ 
 and if $G_2$ has analogous properties,
 this verification is  very easy thanks to the results of Chapter \ref{chap:A},
 but axiom \emph{Ax1} did not seem to follow from the remaining assumptions 
 when $G_1$  is taken to be the derived group of $B_1$ (and similarly for $G_2$).
 
 Here the results of the McCleary-Rubin memoir have been an eye-opener.
 The subject matter of this memoir and its proofs are intricate.
 One reason is that it deals not only 
 with groups of order-preserving automorphisms of linearly ordered sets,
 but also with groups of automorphisms of circularly ordered sets 
 and with groups of automorphisms of two further, related types of structures.
 In the sequel I shall mainly deal with \emph{linear permutation groups}.
 \footnote{Alias \emph{ordered permutation groups} $(L,G)$; 
 but note that in the present context it is not the group $G$ that is ordered, 
 but the orbit $L$.}
 I begin by stating Theorem 2 in \cite[p.\,2]{McRu05}
that is part of Theorem 7.6 in \cite{McRu05}:
 \begin{thmN}
 \label{thm:Linear-reconstruction-theorem}
 Let $(L_1,G_1)$ and $(L_2,G_2)$ be linear permutation groups 
 that are 2-interval-transitive and have each a non-identity bounded element.
 Then every isomorphism $\alpha \colon G_1 \iso G_2$ is induced by conjugation 
 by a unique monotonic bijection $\varphi \colon L_1^{\cpl} \iso L_2^{\cpl}$.
 \end{thmN}

Some of the terms occurring in the above statements deserve a detailed explanation.
As before, $L_i^{\cpl}$ denotes the Dedekind completion of the chain $L_i$
(which is assumed to be dense).
Interval transitivity is a weak form of transitivity. 
In the sequel another weak form will play a rôle; 
so I give both definitions.
\begin{definitionN}
\label{definition:Weak-forms-of-transitivity}
\index{Interval transitivity}%
Let $(L, G)$ be a linear permutation group and $n \geq1$ an integer.
\begin{itemize}
\item $(L,G)$ is called \emph{approximately $n$-fold order-transitive}  
if, for every sequence of points $a_1 < \cdots < a_n$ in $L$ 
and every sequence of intervals $J_1 < \cdots < J_n$ of $L$ with non-empty interiors,
there exists $g \in G$ such that $g(a_i) \in J_i$ for every $i$.

\item $(L,G)$ is called \emph{$n$-interval-transitive}  
if, for every couple of sequences 
\[
I_1 < \cdots < I_n \quad \text{and} \quad
J_1 < \cdots < J_n
\] 
of intervals of $L$ with non-empty interiors,
there exists $g \in G$ such that $g(I_i) \cap J_i$ is non-empty for $i \in \{1, \ldots, n\}$.
\end{itemize}
\end{definitionN}

Axiom \emph{Ax4} states a rather non-intuitive condition (see section \ref{ssec:16.1}),
but it clearly holds  if $G$ is approximately 6-fold order-transitive;
the proof of \cite[Proposition 12.1]{McRu05} shows then 
that approximate 6-fold order transitivity is implied by 12-interval transitivity
and Theorem 12.2 in \cite{McRu05}, finally, reveals that
12-interval transitivity is implied by 3-interval transitivity,
provided $G$ contains a non-identity element with bounded support.
It follows that Theorem \ref{TheoremE3} holds if axiom \emph{Ax4} is replaced 
by the requirement that $G$ acts 3-interval-transitively.
\index{Interval transitivity}%
The hypotheses of Theorem \ref{thm:Linear-reconstruction-theorem} 
are still weaker:
$G_1$ and $G_2$ are merely required to be 2-interval transitive and to satisfy axiom \emph{Ax2}.
\footnote{See also the discussion on page 15 of \cite{McRu05}.}

One of the benefits of the McCleary-Rubin memoir is that its deals not merely 
with isomorphisms of linear permutation groups,
but also with isomorphisms of circular permutation groups.
The following result reproduces Theorem 3  
(this result is also part of Theorem 7.20) in \cite{McRu05}.
\begin{thmN}
 \label{thm:Circular-reconstruction-theorem}
 Let $(C_1,G_1)$ and $(C_2,G_2)$ be circular permutation groups 
 which are 3-interval-transitive and have each a non-identity bounded element.
 Then every isomorphism $\alpha \colon G_1 \iso G_2$ is induced by conjugation 
 by a unique monocircular bijection $\varphi \colon C_1^{\cpl} \iso C_2^{\cpl}$.
 \end{thmN}
 
 The transitivity properties established in Chapter \ref{chap:A} allow one to deduce 
 from the above theorem the following companion to Theorem \ref{TheoremE04}.
 \index{Theorem \ref{TheoremE04}!consequences}
 \footnote{\cf{}Theorem 1.2.4 in \cite{BrGu98}}
 \begin{thmN}
\label{thm:Companion-TheoremE04}
For $i \in \{1,2\}$, let $P_i$ be a non-trivial slope group, 
$A_i$ a non-trivial $\Z[P_i]$-submodule of $\R_{\add}$
and $\pfr_i$ a positive element of $A_i$.
Then every isomorphism 
\[
\alpha \colon T(\R/\Z\pfr_1; A_1, P_1) \iso T(\R/\Z\pfr_2; A_2, P_2)
\]
of generalized $T$-groups
is induced by a homeomorphism $\varphi \colon \R/\Z \pfr_1 \iso  \R/\Z \pfr_2$ 
of a  circle of length $\pfr_1$ onto a circle of length $\pfr_2$.
The homeomorphism $\varphi$ is uniquely determined by $\alpha$ 
and it maps $A_1/\Z \pfr_1$ onto $A_2/\Z \pfr_2$.
\end{thmN}
%
\subsubsection{Remark \ref{remark:Homomorphism-gamma} and \cite{BrGu98}}
\label{sssec:Notes-ChapterE-remark-in section-19.1}
\index{Brin, M. G.}%
\index{Guzm{\'a}n, F.}%
A rudimentary form of Remark \ref{remark:Homomorphism-gamma} 
occurs already in the Bieri-Strebel memoir of 1985.
\index{Quotient group A/IPA@Quotient group $A/(IP \cdot A)$!applications}%
More precisely,
the epimorphism 
\[
\gamma \colon G_\infty(\R;A,P) \epi A / (IP \cdot A)
\]
is not mentioned there
but  $\bar{\gamma}_{\pfr} \colon T(\R/\Z \pfr;A,P) \epi A/(IP \cdot A)$ occurs;
it is used to prove that the group $T(\R/\Z \pfr;A,P) $ need not be simple,
in contrast to Thompson's group $T = T(\R/\Z;\Z[1/2], \gp(2))$.
\index{Group T(R/Zpfr;A,P)@Group $T(\R/\Z \pfr;A,P)$!abelianization}%
The group $G_\infty(\R;A,P)$  
and the homomorphism $\gamma$ 
are both used in \cite{BrGu98} in the special case 
where $P$ is the cyclic group generated by the integer $n \geq 2$
and $A  = \Z[P] = \Z[1/n]$. 
In addition, 
the group $T(\R/\Z \pfr; A, P)$ and the epimorphism $\bar{\gamma}_\pfr$
occur in \cite{BrGu98};
here $\pfr$ denotes an element of $IP \cdot A$.

The parallels between the notation employed in \cite[Section 1.1]{BrGu98}
and that used in this monograph are given below;
in so doing I denote by $P$ the cyclic group generated by $n \geq 2$ 
and by $A$ the $\Z[P]$-submodule $\Z[1/n]$ of $\R_{\add}$.  
Notice that the quotient group $A /(IP \cdot A) = \Z[1/n] / (n-1) \Z[1/n]$ 
is a cyclic group with $n-1$ elements
(\cf{}Illustration \ref{illustration:4.3}).
\begin{align}
&\Delta_n = \ker\left( \Z[1/n] \epi \Z_{n-1}\right) = IP \cdot A,
\label{eq:IPA}\\
&A_n(\R) = G_\infty(\R; A, P),
\label{eq:Group-An(R)}\\
&\rho_n \colon A_n(\R) \epi \Z_{n-1} \text{ corresponds to } 
\gamma \colon G_\infty(\R; A, P) \epi A/(IP \cdot A),
\label{eq:Homomorphism-rho-n)}\\
&B_n(\R) 
=  \ker\left( \rho_n \colon  A_n(\R) \epi \Z_{n-1} \right),
\label{eq:Group-Bn(R)}\\
&A_n(S_r)  = T(\R/\Z r;A,P) \text{ with } r \in \Delta_n = IP \cdot A,\\
&\rho_n \colon A_n(S_r) \epi \Z_{n-1} \text{ corresponds to } 
\bar{\gamma}_r \colon T(\R/\Z r; A, P) \epi A/(IP \cdot A).
\end{align}
Lemma 1.1.3 in \cite{BrGu98} states, \emph{inter alia}, 
that $\rho_n \colon A_n(\R) \to \Z_{n-1}$  is a well-defined epimorphism, 
and Lemma 1.1.5 says that $\rho_n$ induces,
by passage to the quotients, an epimorphism
$\rho \colon A_n(\R/\Z r)  \epi \Z_{n-1}$;
it is the analogue of the epimorphism 
\[
\bar{\gamma}_r \colon T(R/\Z r; A, P) \epi A/(IP \cdot A)
\]
defined in Remark \ref{remark:Homomorphism-gamma}
which goes back to \cite{BiSt85}.
%
\subsubsection{The paper \cite{BrGu98} by M. Brin 
and F. Guzm{\'a}n}
\label{sssec:Notes-ChapterE-BrGu98}
\index{Brin, M. G.}%
\index{Guzm{\'a}n, F.}%
In \cite{Bri96}, 
the automorphism groups of Thompson's groups $F$ and $T$,
and the automorphism groups of some related groups, are determined 
(\cf{}section \ref{sssec:Notes-ChapterE-Brin96} above).
Each automorphism of one of these groups turns out to be induced by a PL-homeomorphism.
This prompts one to ask whether this feature persists
if  $F$ and $T$ are replaced by generalized Thompson groups.

Matthew Brin and Ferdinand Guzm{\'a}n address this question in \cite{BrGu98}.
They consider, \emph{inter alia},
two subgroups of $G_n =G(\R; \Z[1/n], \gp(n))$ with $n > 2$ an integer,
previously studied by K. S. Brown in \cite{Bro87a},  \index{Brown, K. S.}%
and denote these groups by $F_{n, \infty}$ and $F_n$.
Both groups are defined in section 2.1 of \cite{BrGu98} 
by giving explicit sets of generators
and then characterized in Proposition 2.2.6 as follows:
\begin{characterizationN}
\begin{enumerate}[(i)]
\item  The group $F_{n,\infty}$ consists of all elements $g \in G_n$ 
 that are, near $-\infty$ and near $\infty$, translations with amplitudes in 
 $\Z (n-1)$. 
 (So the group $F_{n,\infty}$ is isomorphic to $G([0,1]; \Z[1/n], \gp(n) )$
 by Lemma \ref{LemmaE17}.)

\item The group $F_n$ consists of all elements $g \in G_n$ 
that are, near $-\infty$ and near $\infty$, 
translations with integer amplitudes $a_{-\infty}$ and $a_{+\infty}$ 
satisfying the congruence 
\[
a_{-\infty} \equiv a_{+\infty}\pmod{n-1}.
\]
\end{enumerate}
\end{characterizationN}

The results in \cite{BrGu98} are numerous,
but the precise rendering of some of the most staggering findings 
needs quite a bit of auxiliary definitions and notations.
In the sequel I therefore content myself with a summary,
taken from the introduction to \cite{BrGu98},
and turn then to some of the relations between the Brin-Guzm{\'a}n paper 
and this monograph.

Here is a very brief summary 
(see \cite[p.\,285]{BrGu98}):
\begin{quote}
In this paper, 
we study $\Aut(F_{n, \infty})$, $\Aut(F_n)$ and $\Aut(T_{n,n-1})$ 
\footnote{The group $T_{n,n-1}$ is called $T(\R/\Z(n-1); \Z[1/n],\gp(n))$ 
in  section \ref{ssec:19.1}.}
for $n > 2$.
We find that all three contain exotic elements 
\footnote{Recall that an automorphism $\alpha$ of $G(I;A,P)$ is called \emph{exotic}
if it is induced by a homeomorphism $\varphi_\alpha$
that is \emph{not} piecewise linear; \cf.{}\subsection{sssec:19.3b}.},
and that the complexity increases as $n$ increases.
\index{Exotic automorphisms!existence}%
\index{Group G([a,c];A,P)@Group $G([a,c];A,P)$!exotic automorphisms}%
We also find 
that the structures of $\Aut(F_n)$ and $\Aut(T_{n, n-1})$ are closely related,
while differing significantly from the structure of $\Aut(F_{n, \infty})$.
This difference increases as $n$ increases.
Our analysis is only partial in that we discover large subgroups of these groups containing exotic elements,
but we do not know whether the subgroups are proper. 
\end{quote}

Now to some connections between \cite{BrGu98} and this monograph.

(i) It has been pointed out in section 
\ref{sssec:Notes-ChapterE-remark-in section-19.1}
that the epimorphism 
\[
\gamma \colon G_\infty(\R;A,P) \epi A/(IP \cdot A)
\]
with  $P = \gp(n)$ and $A = \Z[1/n]$ occurs
in section 1.1 of \cite{BrGu98} 
under the name of $\rho_n$.
Let $K = K(\R;A,P)$ denote the kernel of $\gamma$.
Then $K$ is a \emph{characteristic subgroup} of $G$
(this follows from the idea of the proof of Proposition 
\cite[Proposition 1.4.1]{BrGu98} 
and from Theorem \ref{TheoremA}).
\index{Theorem \ref{TheoremA}!consequences}%
Similarly, 
one sees that $F_{n,\infty}$ is a characteristic subgroup of $F_n$.

The group $G = G(\R;A,P)$ is a rather well-known object,
but this is no longer so if one passes from $G$ to the group 
$G_\infty = G_\infty(\R; A, P)$ introduced in section \ref{ssec:19.1};
recall that this group consists of all \emph{finitary, as well as infinitary}, PL-homeomorphisms of $\R$ 
which map $A$ onto itself, have slopes in $P$, breaks in $A$,
and only finitely many breaks in every compact subinterval of $\R$.

(ii) Let $I$, $A$ and $P$ be arbitrary. 
In section \ref{ssec:17.3},
the investigation of outer auto\-morphism groups leads one
to the introduction of the ``automorphism'' group 
\[
\Aut(A) =\{s \in \R^\times \mid s \cdot A = A \}
\] 
of a module $A$ 
and to that of its subgroup of positive elements $\Aut_o(A)$.

If $P$ is \emph{not} cyclic,  
the outer automorphism group of $G =G(I;A,P)$ is an explicitly describable subgroup of $\Aut(A)/P$
(see Corollary \ref{CorollaryE13} for details).
\index{Group Aut(A)@Group $\Aut(A)$!significance}%
If $P$ is \emph{cyclic}, the integer 1 is in $A$ and $I$ is the unit interval,
the subgroup $\Out_{\PL} G $ of $\Out G $ is described in Corollary  \ref{crl:PropositionE20New-part-II};
here $\Out_{\PL} G $ is made up of those classes of automorphisms
that can be represented by automorphisms 
induced by PL-auto-homeomorphisms.
The description reveals that $\Out_{\PL}G $ is far from being abelian.

Let's now consider the special case investigated in \cite{BrGu98} in more detail.
One has $(A,P) = (\Z[1/n], \gp(n))$ with $n \geq 2$ an integer.
The group $F_{n, \infty}$, which is isomorphic to $G([0,1];A, P)$,
is a subgroup of the huge group $B_n(\R)$
(defined by equations \eqref{eq:Group-An(R)} and \eqref{eq:Group-Bn(R)}).
Corollary 3.2.4.1 in \cite{BrGu98}  
and 
Corollary \ref{crl:PropositionE20New-part-II} in this monograph
provide  the following descriptions of the groups 
$\Out_{\PL} F_{n,\infty} $ and $\Out_{\PL} B_n(\R)$: 
\begin{align}
&\Outfr_{\PL} F_{n,\infty}  
\;\text{ is an extension of a group $\bar{L}$ by $\Aut(A)/P$},
\label{eq:Notes-Out-F-n-infty}\\
&\Outfr_{\PL}B_n(\R)= \Aut(A)/ P.  
\label{eq:Notes-Out-Bn(R)}
\end{align}
In the above, 
$A = \Z[1/n]$ and $P =\gp(n)$,
while $\bar{L}$ denotes a subgroup of index $n-1$ in the square of the group
$T_{n, n-1} = T(\R/\Z(n-1);\Z[1/n],\gp(n))$.
With the help of $\bar{\gamma}_{n-1} \colon T_{n,n-1} \epi  \Z_{n-1}$
the subgroup $\bar{L}$ can be described like this:
\[
\bar{L} = \{(\bar{f}_1, \bar{f}_2) \in T_{n, n-1} \times  T_{n,n-1} 
\mid \bar{\gamma}(f_1) = \bar{\gamma} (f_2) \}
\]
(see Corollary \ref{crl:PropositionE20New-part-II}).
This description shows  
that $\Out_{\PL}F_{n,\infty}$ contains a copy of $T_{n, n-1}$
 and so it is far from being abelian.
 On comparing formulae \eqref{eq:Notes-Out-F-n-infty} 
 and \eqref{eq:Notes-Out-Bn(R)} one sees
 that $\Outfr_{\PL} B_n(\R)$ is far simpler than $\Outfr_{\PL} F_{n, \infty} $.
 Notice also that $\Outfr_{\PL} B_n(\R) $ has the form of the outer automorphism group of $G(\R;A,P)$ 
 in case $P$ is not cyclic (see part (iii) of Corollary \ref{CorollaryE13}).
 \smallskip

(iii) I conclude with a word on a third link
between \cite{BrGu98} and this monograph.
Proposition \ref{PropositionE9}  investigates a homeomorphism  $\varphi \colon \Int(I_1) \iso \Int(I_2)$ 
inducing an isomorphism 
\[
\alpha \colon G(I_1;A_1,P_1) \iso G(I_2;A_2,P_2). 
\]
It assumes that there exists a point $a_0 \in A \cap \Int(I_1)$ 
at which $\varphi$ has a one-sided derivative
and shows that $\varphi$ is piecewise linear 
with slopes in a single coset $s_\varphi \cdot P$ of $P$.
The proposition applies, in particular, 
to the auto-homeomorphisms inducing automorphisms in $\Aut_{\PL} G(I; A,P)$
and implies that $\Aut_{\PL} G(I; A,P)$ maps into the abelian group $\Aut(A)/P$.
\index{Group Aut(A)@Group $\Aut(A)$!significance}%
 
In  Lemma 3.2.1 of \cite{BrGu98}, 
Brin and Guzm{\'a}n consider a situation 
that leads to a similar conclusion as the previously mentioned proposition.
Suppose $f$ is an element of 
$G_\infty(\R; \R_{\add}, \R^\times_{>0})$ 
which maps $A = \Z[1/n]$ onto itself.
The lemma then shows that the slopes $s$ of $f$ have the property 
that $s^{\pm 1} \cdot A \subseteq A$, \ie{} that $ s \in \Aut(A) \cap  \R^\times_{>0} = \Aut_o(A)$.
The authors denote $\Aut_o(\Z[1/n])$ by $\langle \langle n \rangle \rangle$ 
and state that it is generated by the prime divisors of $n$; 
\cf{}Example (i)a) in section 
\ref{sssec:Automorphism-groups-some-examples-non-cyclic-P}.
%
\subsubsection{The paper \cite{Lio08} by Isabelle Liousse}
\label{sssec:Notes-ChapterE-Lio08}
\index{Liousse, I.|(}%
%
The Thompson-Stein groups referred to in the title of Isabelle Liousse's paper
are the groups $F([0,r]; A, P)$ and $T(\R/\Z \cdot r; A, P)$ with parameters
\begin{gather*}
r \in \N \smallsetminus \{0\},\quad 
n_1, \ldots,n_p \text{ multiplicatively  independent integers $> 1$},\\
P = \gp( n_1, \ldots, n_p) \quad\text{and}\quad  A = \Z[P] = \Z[1/(n_1 \cdots n_p)].
\end{gather*}
The author denotes by $T_{r, (m)}$  the generalized $T$-group with $p = 1$ and $n_1 = m$,
and by $T_{r,(n_i)}$  the generalized $T$-group with $n_i > 1$;
see section 1.1 in \cite{Lio08}.
A main theme of her paper is the determination 
of the rotation numbers of the PL-homeomorphisms
in the group $T(\R/\Z \cdot r;A, P)$, 
but her paper discusses also isomorphisms among two groups 
$T(\R/\Z \cdot r; A, P)$ and $T(\R/\Z \cdot r'; A', P')$, 
and the automorphism group of $T([0,r]; A, P)$;
these latter investigations are related to themes in the Bieri-Strebel memoir.
In the sequel, 
I comment on answers to the following two questions, 
posed on p.52 in \cite[section 1.2]{Lio08}:
\begin{enumerate}
\item[(2)] What are the automorphisms of $T_{r, (n_i)}$?
\item[(3)]Is it possible to classify the groups $T_{r, (n_i)}$ up to isomorphism?
up to quasi-isometry?
\end{enumerate}

A first answer to (3) deals with isomorphisms among the groups  $T_{r, (n_i)}$ 
with $r$ varying in  $\N \smallsetminus \{0\}$. 
The reasonings in section \ref{ssec:17.2} of this monograph
show 
\index{Theorem \ref{TheoremA}!consequences}%
that there exists a finitary PL-homeomorphism $\varphi \colon [0,r] \iso [0,r']$ 
having slopes in a single coset $s \cdot P$ of $P$ with $s \in \Aut_o(A)$  
and singularities in $A= \Z[1/(n_1 \cdots n_p)]$ if, and only if,
$s \cdot r \equiv r'$  \emph{modulo} $IP \cdot A = d \cdot A$;
here $d = \gcd\{n_1-1, \ldots, n_p-1\}$ (\cf{}Illustration \ref{illustration:4.3}).
This result  implies, in particular, 
that the groups $T_{r, (n_i)}$ and $T_{r', (n_i)}$ are isomorphic 
whenever $r \equiv r' \pmod{d}$.
\footnote{This consequence is part of the the so-called \emph{Bieri-Strebel criterion for Thompson-Stein groups} formulated on page 55 in \cite{Lio08}.
Note, however, that the first two sentences of the comments 
following the statement of questions (1) through (4) are misleading:
\cite{BiSt85} studies isomorphisms and automorphisms of the groups 
$F_{r,\Lambda, A}= G([0,r]; A, \Lambda)$, 
but not those of the groups $T_{r, \Lambda, A}$;
in addition, the preceding paragraph shows 
that groups $T_{r, (n_i)}$ and $T_{r', (n_i)}$  are isomorphic if $r'-r \in \Z \cdot d$,
but it indicates also that the converse may be false.}

The answer, given above,
formulates a condition on $(r,r')$ 
that implies that two groups  $T_{r, (n_i)}$ and $T_{r', (n_i)}$ are isomorphic.
I. Liousse gives also a necessary condition 
for groups with a cyclic slope group $P = \gp(m)$; 
its justification relies on properties of rotation numbers, 
formulated in \cite[Theorem 1]{Lio08}. 
Her Corollary 2 can be stated like this:
\begin{prpN}
\label{prp:Corollary2-Lio98}
For $m > 2$ and $(r, r') \in \{1, \ldots, m-2\}$ one has:
\begin{enumerate}[(i)]
\item The group $T_{r, (m)}$ is not isomorphic to $T_{m-1,(m)}$;
\item if $m$ and $r$ are odd and $r'$ is even, 
$T_{r,(m)}$ is not isomorphic to  $T_{r', (m)}$.
\end{enumerate}
\end{prpN}
Note that the group $T_{m-1, (m)}$ is known to contain \emph{exotic} automorphisms,
\ie{}automorphisms induced by homeomorphisms of the circle $\R/\Z (m-1)$ 
that are \emph{not} piecewise linear (see \cite{BrGu98}). 
\index{Exotic automorphisms!existence}%
So it is unlikely that results similar to the above proposition can be obtained
by the methods developed in \cite{BiSt85}.
Things are different if one looks for isomorphisms of generalized $T$-groups with non-cyclic slope group $P$; see Proposition \ref{prp:Corollary3-Lio98-improved} below.

A third answer, both to questions (2) and (3),
is provided by Corollary 3 in \cite{Lio08}.
This corollary has 4 parts; parts 3 and 4  deal with isomorphisms and automorphisms;
they can be restated thus:
\begin{prpN}
\label{prp:Corollary3-Lio98}
Let $T_{r, (n_i)}$ be a Thompson-Stein group with rank $p > 2$ and set $A = \Z[1/(n_1 \cdots n_p)]$.
Then
\begin{enumerate}
\item[3.] the group $T_{r, (n_i)}$has no exotic automorphisms;
\index{Exotic automorphisms!existence}%
more precisely, its automorphisms are realized by conjugations by maps in the larger group 
\[
T(\R/\Z \cdot r;A, \Aut(A)).
\index{Group Aut(A)@Group $\Aut(A)$!significance}%
\footnote{Here I imitate the notation of I. Liousse;
it is at variance with the convention 
that generalized $F$- and $T$-groups contain only increasing PL-homeomorphisms.}
\]
In particular, 
the outer automorphism group $\Out(T_{r, (n_i)})$ has order 2 provided 
that the integers $n_i$ are co-prime;
\footnote{This addition is not quite correct:
it holds if, and only if $P = \Aut_o(A)$;
by section \ref{sssec:Automorphism-groups-some-examples-non-cyclic-P} (i)a)
this latter condition is valid precisely if the group $P$,
generated by the integers $n_1$, \ldots, $n_p$,
coincides with the group generated by the prime numbers 
which divide the product $n_1 \cdots n_p$.}; 
\item[4.] for every  $r \in \N \smallsetminus \{0 \}$ two Thompson-Stein groups 
\[
T([0,r]; \Z[P], P= \gp(n_1, \ldots, n_p))
\quad \text{and} \quad T([0,r]; \Z[P'], P')
\]
with $P$ of rank $p > 2$ and $P'$ of rank  $q \geq 1$,
are isomorphic if, and only if $P = P'$.
\end{enumerate} 
\end{prpN}

In contrast to Proposition \ref{prp:Corollary2-Lio98},
the preceding Proposition \ref{prp:Corollary3-Lio98}
can also be derived by the techniques of this monograph
from a basic result of \cite{McRu05},
even in a stronger form:
\begin{prpN}
\label{prp:Corollary3-Lio98-improved}
\index{Group T(R/Zpfr;A,P)@Group $T(\R/\Z \pfr;A,P)$!automorphisms}%
\index{Group T(R/Zpfr;A,P)@Group $T(\R/\Z \pfr;A,P)$!isomorphisms}%
Assume  $P$ is \emph{not} cyclic, $A$ is a $\Z[P]$-module 
and $r$ is a positive element of $A$,
and that $P_1$ is a group, $A_1$ a $\Z[P_1]$-module and $r_1$ a positive element of $A_1$. 
Then
\begin{enumerate}[(i)]
\item the group $T(\R/\Z \cdot r; A, P)$ has no exotic automorphisms;
\index{Exotic automorphisms!existence}%
more precisely, 
its automorphisms are realized by conjugations by maps in the larger group 
\[
T(\R/\Z \cdot r;A, \Aut_o(A)) \circ \{ \id, t + \Z \cdot r \mapsto -t + \Z \cdot r\} ;
\]
\item the groups $T(\R/\Z \cdot r;  A, P)$ and $T(\R/\Z \cdot r_1; A_1, P_1)$ are isomorphic 
if, and only if, 
the groups $G([0,r]; A, P)$ and $G([0,r_1]; A_1, P_1)$ are so.
\end{enumerate} 
\end{prpN}

\begin{proof}
Let $\alpha$ be an isomorphism of $T(\R/\Z \cdot r;  A, P)$ 
onto $T(\R/\Z \cdot r_1; A_1, P_1)$.
Theorem \ref{thm:Companion-TheoremE04} then guarantees
that $\alpha$ is induced by conjugation by a homeomorphism 
$\varphi \colon \R/\Z \cdot r \iso \R/\Z \cdot r_1$.
Since $T(\R/\Z \cdot r_1;  A_1, P_1)$ contains the rotation
which sends the point $\varphi(0 + \Z \cdot r)$  to the point $0 + \Z \cdot r_1$,
we may and shall assume 
that $\varphi$ maps the point $x = 0 + \Z \cdot r$ of the circle 
$\R/\Z \cdot r$ to the point $x_1 =0 + \Z \cdot r_1$ in $\R/\Z \cdot r_1$.
It follows 
that $\varphi$ induces by conjugation an isomorphism of the stabilizer
$G = \St_{T(\R/\Z \cdot r; A. P)}(x)$  
onto the stabilizer $G_1 =\St_{T(\R/\Z \cdot r_1;  A_1, P_1)}(x_1)$.
Let $\tilde{\varphi}$ be the lift of $\varphi$ 
which maps the interval $[0, r]$ homeomorphically onto $[0,r_1]$.
Conjugation by $\tilde{\varphi}$ induces then an isomorphism $\tilde{\alpha}$ of  the group
$G([0,r]; A, P)$ onto the group $G([0,r_1]; A_1, P_1)$.
Since $P$ is \emph{not} cyclic Theorem \ref{TheoremE10}
\index{Theorem \ref{TheoremE10}!consequences}%
and Supplement \ref{SupplementE11New}
apply and show that $P = P_1$, 
that $\tilde{\varphi}$ is a finitary PL-homeomorphism with slopes in some coset $s \cdot P$ of $P$
and that $A_1 = s \cdot A$. 

The previous argument establishes one implication of assertion (ii).
The converse follows from the fact that an isomorphism $G([0,r]; A, P) \iso G([0,r_1]; A_1, P_1)$
is induced, by Theorem \ref{TheoremE10} and Supplement \ref{SupplementE11New},
\index{Supplement \ref{SupplementE11New}!consequences}%
by a finitary PL-homeomor\-phism $\tilde{\varphi} \colon [0, r] \iso [0,r_1]$ 
with slopes in a coset $s \cdot P$ and that $A_1 = s \cdot A$.
The homeomorphism $\tilde{\varphi}$ descends to a homeomorphism 
$\varphi \colon \R/\Z \cdot r \iso \R/\Z \cdot r_1$
which induces an isomorphism of $T(\R/\Z \cdot r;  A, P)$ onto $T(\R/\Z \cdot r_1; A_1, P_1)$.

Now to assertion (i).
The argument in the first paragraph shows 
that, given a homeomorphism $\varphi \colon \R/\Z \cdot r \iso \R/\Z \cdot r$,
there is a rotation $f \in T(\R/\Z \cdot r;  A, P)$,
 so that the composition $f \circ \varphi$ 
 is a finitary PL-auto-homeomorphism of $\R/\Z \cdot r$ with slopes in $\Aut(A)$.
As $\varphi$  has the same properties as $f \circ \varphi$  
the argument given in the first paragraph of the proof 
implies claim (i).
\end{proof}
\index{Liousse, I.|)}%
%

%% file: Z.Main.bbl
\providecommand{\bysame}{\leavevmode\hbox to3em{\hrulefill}\thinspace}
\providecommand{\MR}{\relax\ifhmode\unskip\space\fi MR }
\providecommand{\MRhref}[2]{%
  \href{http://www.ams.org/mathscinet-getitem?mr=#1}{#2}
}
\providecommand{\href}[2]{#2}
\begin{thebibliography}{MKS04}

\bibitem[AM69]{AtMa69}
M.~F. Atiyah and I.~G. Macdonald, \emph{Introduction to commutative algebra},
  Addison-Wesley Publishing Co., Reading, Mass.-London-Don Mills, Ont., 1969.
  \MR{0242802}

\bibitem[AW04]{AlWi04}
{\c{S}}aban Alaca and Kenneth~S. Williams, \emph{Introductory algebraic number
  theory}, Cambridge University Press, Cambridge, 2004. \MR{2031707
  (2005d:11152)}

\bibitem[BB76]{BaBi76}
Gilbert Baumslag and Robert Bieri, \emph{Constructable solvable groups}, Math.
  Z. \textbf{151} (1976), no.~3, 249--257. \MR{0422422 (54 \#10411)}

\bibitem[BCR08]{BCR08}
Jos{\'e} Burillo, Sean Cleary, and Claas~E. R{\"o}ver, \emph{Commensurations
  and subgroups of finite index of {T}hompson's group {$F$}}, Geom. Topol.
  \textbf{12} (2008), no.~3, 1701--1709. \MR{2421137 (2009g:20093)}

\bibitem[BG84]{BrGe84}
Kenneth~S. Brown and Ross Geoghegan, \emph{An infinite-dimensional torsion-free
  ${FP}_{\infty}$ group}, Invent. Math. \textbf{77} (1984), no.~2, 367--381.
  \MR{752825 (85m:20073)}

\bibitem[BG85]{BrGe85}
\bysame, \emph{Cohomology with free coefficients of the fundamental group of a
  graph of groups}, Comment. Math. Helv. \textbf{60} (1985), no.~1, 31--45.
  \MR{787660 (87b:20066)}

\bibitem[BG98]{BrGu98}
Matthew~G. Brin and Fernando Guzm{\'a}n, \emph{Automorphisms of generalized
  {T}hompson groups}, J. Algebra \textbf{203} (1998), no.~1, 285--348.
  \MR{1620674 (99d:20056)}

\bibitem[BIP08]{BIP08}
Dmitri Burago, Sergei Ivanov, and Leonid Polterovich,
  \emph{Conjugation-invariant norms on groups of geometric origin}, Groups of
  diffeomorphisms, Adv. Stud. Pure Math., vol.~52, Math. Soc. Japan, Tokyo,
  2008, pp.~221--250. \MR{2509711}

\bibitem[Bri96]{Bri96}
Matthew~G. Brin, \emph{The chameleon groups of {R}ichard {J}. {T}hompson:
  automorphisms and dynamics}, Inst. Hautes \'Etudes Sci. Publ. Math. (1996),
  no.~84, 5--33 (1997). \MR{1441005 (99e:57003)}

\bibitem[Bro87]{Bro87a}
Kenneth~S. Brown, \emph{Finiteness properties of groups}, Proceedings of the
  {N}orthwestern conference on cohomology of groups ({E}vanston, {I}ll., 1985),
  vol.~44, 1987, pp.~45--75. \MR{885095 (88m:20110)}

\bibitem[Bro94]{Bro94}
\bysame, \emph{Cohomology of groups}, Graduate Texts in Mathematics, vol.~87,
  Springer-Verlag, New York, 1994, Corrected reprint of the 1982 original.
  \MR{1324339 (96a:20072)}

\bibitem[BS78]{BiSt78}
Robert Bieri and Ralph Strebel, \emph{Almost finitely presented soluble
  groups}, Comment. Math. Helv. \textbf{53} (1978), no.~2, 258--278.
  \MR{MR0498863 (58 \#16890)}

\bibitem[BS80]{BiSt80}
\bysame, \emph{Valuations and finitely presented metabelian groups}, Proc.
  London Math. Soc. (3) \textbf{41} (1980), no.~3, 439--464. \MR{591649
  (81j:20080)}

\bibitem[BS81]{BiSt81a}
\bysame, \emph{A geometric invariant for modules over an abelian group}, J.
  Reine Angew. Math. \textbf{322} (1981), 170--189. \MR{603031 (82f:20017)}

\bibitem[BS85a]{BiSt85}
\bysame, \emph{On groups of {P}{L}-homeomorphisms of the real line}, preprint,
  1985.

\bibitem[BS85b]{BrSq85}
Matthew~G. Brin and Craig~C. Squier, \emph{Groups of piecewise linear
  homeomorphisms of the real line}, Invent. Math. \textbf{79} (1985), no.~3,
  485--498. \MR{782231 (86h:57033)}

\bibitem[BS14]{BiSt14}
Robert Bieri and Ralph Strebel, \emph{On groups of {P}{L}-homeomorphisms of the
  real line}, preprint, \texttt{arXiv:1411.2868v1}, 2014.

\bibitem[BW07]{BlWa07}
Collin Bleak and Bronlyn Wassink, \emph{Finite index subgroups of {R}.
  {T}hompson’s group {$F$}}, preprint, \texttt{arXiv:}\texttt{0711.1014v1, 7
  Nov}, 2007.

\bibitem[CFP96]{CFP96}
J.~W. Cannon, W.~J. Floyd, and W.~R. Parry, \emph{Introductory notes on
  {R}ichard {T}hompson's groups}, Enseign. Math. (2) \textbf{42} (1996),
  no.~3-4, 215--256. \MR{1426438 (98g:20058)}

\bibitem[Che52]{Che52}
C.~G. Chehata, \emph{An algebraically simple ordered group}, Proc. London Math.
  Soc. (3) \textbf{2} (1952), 183--197. \MR{0047031 (13,817b)}

\bibitem[Cle95]{Cle95}
Sean Cleary, \emph{Groups of piecewise-linear homeomorphisms with irrational
  slopes}, Rocky Mountain J. Math. \textbf{25} (1995), no.~3, 935--955.
  \MR{1357102 (97d:20040)}

\bibitem[Cle00]{Cle00}
\bysame, \emph{Regular subdivision in {$\Z[\frac{1+\sqrt 5}{2}]$}}, Illinois J.
  Math. \textbf{44} (2000), no.~3, 453--464. \MR{1772420 (2001h:20051)}

\bibitem[CST04]{CST04}
Sean Cleary, John Stallings, and Jennifer Taback, \emph{Thompson's group at 40
  years---preliminary problem list}, preprint,
  \texttt{http://www.aimath.org/WWN/thompsonsgroup/}
  \texttt{thompsonsgroup.pdf}, 2004.

\bibitem[Die69]{Die69}
J.~Dieudonn{\'e}, \emph{Foundations of modern analysis}, Academic Press, New
  York-London, 1969, Enlarged and corrected printing, Pure and Applied
  Mathematics, Vol. 10-I. \MR{0349288 (50 \#1782)}

\bibitem[DS78]{DySe78}
Jerzy Dydak and Jack Segal, \emph{Shape theory}, Lecture Notes in Mathematics,
  vol. 688, Springer, Berlin, 1978, An introduction. \MR{520227 (80h:54020)}

\bibitem[Dyd77]{Dyd77}
Jerzy Dydak, \emph{A simple proof that pointed {FANR}-spaces are regular
  fundamental retracts of {ANR}'s}, Bull. Acad. Polon. Sci. S\'er. Sci. Math.
  Astronom. Phys. \textbf{25} (1977), no.~1, 55--62. \MR{0442918 (56 \#1293)}

\bibitem[Eps70]{Eps70}
D.~B.~A. Epstein, \emph{The simplicity of certain groups of homeomorphisms},
  Compositio Math. \textbf{22} (1970), 165--173. \MR{0267589 (42 \#2491)}

\bibitem[FH79]{FrHe79}
P.~Freyd and A.~Heller, \emph{Splitting homotopy idempotents, {II}},
  mimeographed Notes, University of Pennsilvania, 1979.

\bibitem[FH93]{FrHe93}
Peter Freyd and Alex Heller, \emph{Splitting homotopy idempotents. {II}}, J.
  Pure Appl. Algebra \textbf{89} (1993), no.~1-2, 93--106. \MR{1239554
  (95h:55015)}

\bibitem[Geo08]{Geo08}
Ross Geoghegan, \emph{Topological methods in group theory}, Graduate Texts in
  Mathematics, vol. 243, Springer, New York, 2008. \MR{2365352 (2008j:57002)}

\bibitem[GG16]{GaGi16}
{\'{S}}.~R. Gal and Jakub Gismatullin, \emph{Uniform simplicity of proximal
  action}, preprint, \texttt{arXiv:1602.08740v1}, 2016.

\bibitem[Gla81]{Gla81}
A.~M.~W. Glass, \emph{Ordered permutation groups}, London Mathematical Society
  Lecture Note Series, vol.~55, Cambridge University Press, Cambridge, 1981.
  \MR{645351 (83j:06004)}

\bibitem[HH82]{HaHe82}
Harold~M. Hastings and Alex Heller, \emph{Homotopy idempotents on
  finite-dimensional complexes split}, Proc. Amer. Math. Soc. \textbf{85}
  (1982), no.~4, 619--622. \MR{660617}

\bibitem[Hig54]{Hig54a}
Graham Higman, \emph{On infinite simple permutation groups}, Publ. Math.
  Debrecen \textbf{3} (1954), 221--226 (1955). \MR{0072136 (17,234d)}

\bibitem[Hig74]{Hig74}
\bysame, \emph{Finitely presented infinite simple groups}, Department of Pure
  Mathematics, Department of Mathematics, I.A.S. Australian National
  University, Canberra, 1974, Notes on Pure Mathematics, No. 8 (1974).
  \MR{0376874 (51 \#13049)}

\bibitem[Hol65]{Hol65c}
Charles Holland, \emph{Transitive lattice-ordered permutation groups}, Math. Z.
  \textbf{87} (1965), 420--433. \MR{0178052 (31 \#2310)}

\bibitem[HS74]{HiSt74a}
Peter Hilton and Urs Stammbach, \emph{On torsion in the differentials of the
  {L}yndon-{H}ochschild-{S}erre spectral sequence}, J. Algebra \textbf{29}
  (1974), 349--367. \MR{0344311 (49 \#9050)}

\bibitem[HS97]{HiSt97}
P.~J. Hilton and U.~Stammbach, \emph{A course in homological algebra}, second
  ed., Graduate Texts in Mathematics, vol.~4, Springer-Verlag, New York, 1997.
  \MR{1438546 (97k:18001)}

\bibitem[Lio08]{Lio08}
Isabelle Liousse, \emph{Rotation numbers in {T}hompson-{S}tein groups and
  applications}, Geom. Dedicata \textbf{131} (2008), 49--71. \MR{2369191
  (2009b:37073)}

\bibitem[LR04]{LeRo04}
John~C. Lennox and Derek J.~S. Robinson, \emph{The {T}heory of {I}nfinite
  {S}oluble {G}roups}, Oxford Mathematical Monographs, The Clarendon Press
  Oxford University Press, Oxford, 2004. \MR{2093872 (2006b:20047)}

\bibitem[McC78]{McC78b}
Stephen~H. McCleary, \emph{Groups of homeomorphisms with manageable
  automorphism groups}, Comm. Algebra \textbf{6} (1978), no.~5, 497--528.
  \MR{484757 (81e:20045)}

\bibitem[MKS04]{MKS04}
Wilhelm Magnus, Abraham Karrass, and Donald Solitar, \emph{Combinatorial group
  theory}, second ed., Dover Publications Inc., Mineola, NY, 2004,
  Presentations of groups in terms of generators and relations. \MR{2109550
  (2005h:20052)}

\bibitem[MR96]{McRu96}
Stephen McCleary and Matatyahu Rubin, \emph{Locally moving groups and the
  reconstruction problems for chains and circles}, preprint, Bowling Green
  State University, Bowling Green, Ohio, 1996.

\bibitem[MR05]{McRu05}
\bysame, \emph{Locally moving groups and the reconstruction problems for chains
  and circles}, preprint, \texttt{arXiv:}\texttt{0510122v1, 6 Oct 2005}, 2005.

\bibitem[MT73]{McTh73}
Ralph McKenzie and Richard~J. Thompson, \emph{An elementary construction of
  unsolvable word problems in group theory}, Word problems: decision problems
  and the {B}urnside problem in group theory ({C}onf., {U}niv. {C}alifornia,
  {I}rvine, {C}alif. 1969; dedicated to {H}anna {N}eumann), Studies in Logic
  and the Foundations of Math., vol.~71, North-Holland, Amsterdam, 1973,
  pp.~457--478. \MR{0396769 (53 \#629)}

\bibitem[Pas77]{Pas77}
Donald~S. Passman, \emph{The algebraic structure of group rings}, Pure and
  Applied Mathematics, Wiley-Interscience [John Wiley \& Sons], New York, 1977.
  \MR{MR470211 (81d:16001)}

\bibitem[Rob76]{Rob76}
Derek J.~S. Robinson, \emph{The vanishing of certain homology and cohomology
  groups}, J. Pure Appl. Algebra \textbf{7} (1976), no.~2, 145--167.
  \MR{0404478 (53 \#8280)}

\bibitem[Rob96]{Rob96}
\bysame, \emph{A course in the theory of groups}, second ed., Graduate Texts in
  Mathematics, vol.~80, Springer-Verlag, New York, 1996. \MR{1357169
  (96f:20001)}

\bibitem[Ste92]{Ste92}
Melanie Stein, \emph{Groups of piecewise linear homeomorphisms}, Trans. Amer.
  Math. Soc. \textbf{332} (1992), no.~2, 477--514. \MR{1094555 (92k:20075)}

\bibitem[Str84]{Str84}
Ralph Strebel, \emph{Finitely presented soluble groups}, Group theory, Academic
  Press, London, 1984, pp.~257--314. \MR{780572 (86g:20050)}

\bibitem[Tho74]{Tho74}
Richard~J. Thompson, \emph{(no title)}, manuscript, written approximately in
  1974, 1974.

\bibitem[Tho80]{Tho80}
\bysame, \emph{Embeddings into finitely generated simple groups which preserve
  the word problem}, Word problems, {II} ({C}onf. on {D}ecision {P}roblems in
  {A}lgebra, {O}xford, 1976), Stud. Logic Foundations Math., vol.~95,
  North-Holland, Amsterdam, 1980, pp.~401--441. \MR{579955 (81k:20050)}

\bibitem[Whi63]{Whi63}
James~V. Whittaker, \emph{On isomorphic groups and homeomorphic spaces}, Ann.
  of Math. (2) \textbf{78} (1963), 74--91. \MR{0150750 (27 \#737)}

\bibitem[Zhu08]{Zhu08}
Dongping Zhuang, \emph{Irrational stable commutator length in finitely
  presented groups}, J. Mod. Dyn. \textbf{2} (2008), no.~3, 499--507.
  \MR{2417483 (2009k:20075)}

\end{thebibliography}
